\documentclass[reqno,10pt]{amsart}
\usepackage{amsfonts}
\usepackage{amssymb}
\usepackage{amsmath}
\usepackage{lineno}
\usepackage{a4wide}
\newtheorem{theorem}{Theorem} 
\newtheorem{definition}{Definition}
\newtheorem{lemma}{Lemma}
\newtheorem{proposition}{Proposition}
\newtheorem{example}{Example}

\begin{document}


\title{ \large{Regularity of the Boltzmann Equation\\ in Convex Domains}}
\author{{Y. Guo, C. Kim, D. Tonon, A. Trescases}}
\date{}
\maketitle

\begin{abstract}
A basic question about regularity of Boltzmann solutions in the presence of physical boundary conditions has been open due to characteristic nature of the boundary as well as the non-local mixing of the collision operator. Consider the Boltzmann equation in a strictly convex domain with the specular, bounce-back and diffuse boundary condition. With the aid of a distance function toward the grazing set, we construct weighted classical $C^{1}$ solutions away from the grazing set for all boundary conditions. For the diffuse boundary condition, we construct $W^{1,p}$ solutions for $1< p<2$ and weighted $%
W^{1,p}$ solutions for $2\leq p\leq  \infty $ as well. On the other hand, we show second derivatives do not exist up to the boundary in general by constructing counterexamples for all boundary conditions.
\end{abstract}

\footnotetext{%
Last update : \today}
 
\vspace{15pt}
  
 \tableofcontents

\section*{Introduction}
 
\vspace{8pt}
 
Boundary effects play an important role in the dynamics of Boltzmann
solutions of%
\begin{equation}
\partial _{t}F+v\cdot \nabla _{x}F=Q(F,F),\label{Boltzmann_F}
\end{equation}%
where $F(t,x,v)$ denotes the particle distribution at time $t,$ position $%
x\in \Omega \,$\ and velocity $v\in \mathbb{R}^{3}.$ Throughout this paper
the collision operator takes the form
\begin{equation}\label{collision_Q}
\begin{split}
Q(F_{1},F_{2})
&:=Q_{\text{gain}}(F_{1},F_{2})-Q_{\text{loss%
}}(F_{1},F_{2})\\
&=\int_{\mathbb{R}^{3}}\int_{\mathbb{S}^{2}}|v-u|^{\kappa
}q_{0}(\theta )\Big[F_{1}(u^\prime)F_{2}(v^\prime)-F_{1}(u)F_{2}(v)\Big]%
\mathrm{d}\omega \mathrm{d}u,
\end{split}
\end{equation}
where $u^\prime=u+[(v-u)\cdot \omega ]\omega ,\ v^\prime=v-[(v-u)\cdot
\omega ]\omega $ and $0\leq \kappa \leq 1$ (hard potential) and $0\leq
q_{0}(\theta )\leq C|\cos \theta |$ (angular cutoff) with $\cos \theta =%
\frac{v-u}{|v-u|}\cdot \omega $.

Despite extensive developments in the study of the Boltzmann equation, many basic
questions regarding solutions in a physical bounded domain, such as their regularity,
have remained largely open. This is partly due to the characteristic nature
of boundary conditions in the kinetic theory. In \cite{Guo10}, it is shown that
in convex domains, Boltzmann solutions are continuous away from the grazing
set. On the other hand, in \cite{Kim11}, it is shown that singularity
(discontinuity) does occur for Boltzmann solutions in a non-convex domain,
and such singularity propagates precisely along the characteristics emanating from the grazing set. The boundary of the phase space is
\[
\gamma:= \{(x,v) \in\partial\Omega \times\mathbb{R}^{3}\},
\]
where $n=n(x)$ the outward normal direction at $x\in \partial \Omega $. We decompose $\gamma$ as 
\begin{equation}\notag
\begin{split}
\gamma
_{-}&=\{(x,v)\in \partial \Omega \times \mathbb{R}^{3}: n({x})\cdot v<0\}, \ \ \  \ \ \ \ \ \ (\textit{the incoming set}),\\
\gamma
_{+}&=\{(x,v)\in \partial \Omega \times \mathbb{R}^{3}: n({x})\cdot v>0\}, \ \ \  \ \ \ \ \ \ (\textit{the outcoming set}),\\
\gamma _{0}&=\{(x,v)\in \partial \Omega \times \mathbb{R}^{3}:  n({x}%
) \cdot v=0\} , \ \ \  \ \ \ \ \ \ (\textit{the grazing set}).
\end{split}
\end{equation}

In general the boundary condition is imposed only for the incoming set $\gamma_{-}$ for general kinetic PDEs \cite{CIP, DV,Guo95,Guo10}.

Throughout this paper we assume that $\Omega$ is a bounded open subset of $\mathbb{R}^{3}$ and there exists $\xi : \mathbb{R}^{3} \rightarrow \mathbb{R}$ such that $\Omega = \{ x\in \mathbb{R}^{3}: \xi(x)<0\},$ and $\partial \Omega = \{x\in \mathbb{R}^{3}: \xi(x)=0\}$. Moreover for all $x\in \bar{\Omega}= \Omega \cup \partial\Omega \  (\text{therefore } \xi(x) \leq 0  )$ we assume the domain is \textit{strictly convex} :
 \begin{equation}
\sum_{i,j}\partial_{ij} \xi(x) \zeta_i \zeta_j \geq C_{\xi} |\zeta|^2  \ \ \ \text{for all \ } \zeta \in\mathbb{R}^{3}.
\label{convex} 
\end{equation}
We assume that $\nabla \xi(x) \neq 0$ when $|\xi(x)| \ll 1$ and we define the $\textit{outward normal}$ as $n(x) \equiv \frac{\nabla \xi(x)}{|\nabla \xi(x)|}.$

\vspace{8pt}

In this paper, we consider the following basic boundary conditions on $(x,v)\in \gamma _{-}$

\vspace{8pt}

\noindent\textit{(i) Diffuse boundary condition:} With $c_{\mu }\int_{n(x)\cdot u>0}\mu (u)\{n(x)\cdot
u\}\mathrm{d}u=1$,
\begin{equation}\notag\label{diffuseBC_F}
F(t,x,v)=c_{\mu }\mu (v)\int_{n(x)\cdot  u>0}F(t,x,u)\{n(x)\cdot u\}\mathrm{d}u.
\end{equation}%

\vspace{2pt}

\noindent\textit{(ii) Specular reflection boundary condition:} 
\[
F(t,x,v)= F(t,x,R_{x}v), \ \ \ \ \text{where} \ R_{x}v:= v-2n(x)(n(x)\cdot v).
\]
\vspace{2pt}

\noindent\textit{(iii) Bounce-back reflection boundary condition:} $$
F(t,x,v)= F(t,x,-v).$$

\vspace{8pt}

For $(x,v)\in\bar{\Omega} \times \mathbb{R}^{3}$ we define $t_{\mathbf{b}}(x,v)$ be \textit{the backward exit time} as
\begin{equation} 
t_{\mathbf{b}}(x,v)=\inf \{\tau >0:x-sv\notin \Omega \},  \label{exit}
\end{equation}%
and $x_{\mathbf{b}}(v)=x-t_{\mathbf{b}}v.$

The characteristics ODE of the Boltzmann equation (\ref{Boltzmann_F}) is 
\[
\frac{\mathrm{d}X(s)}{\mathrm{d}s}   = V(s) , \ \ \  \frac{\mathrm{d}V(s)}{\mathrm{d}s}=0.
\]

Before the trajectory hits the boundary, $t-s <t_{\mathbf{b}}(x,v)$, we have $[X(s;t,x,v),V(s;t,x,v)]=[x-(t-s)v,v]$ with the initial condition $[X(t;t,x,v),V(t;t,x,v)]=[x,v].$ On the other hand, when the trajectory hits the boundary we define the generalized characteristics as follows:

\begin{definition}[\cite{Guo10}] \label{cycles}
 
 Let $(x,v) \notin \gamma_{0}$ and $(t^{0}, x^{0}, v^{0}) = (t,x,v).$
 
 (i) Define the stochastic (diffuse) cycles as $(t^{1}, x^{1}, v^{1})= (t-t_{\mathbf{b}}(x,v), x-t_{\mathbf{b}}(x,v)v, v^{1})$ with $n(x^{1}) \cdot v^{1} >0$ and for $\ell \geq 1,$
 \begin{equation}
(t^{\ell+1} ,x^{\ell+1},v^{\ell+1}) = (t^{\ell}-t_{\mathbf{b}}(x^{\ell},v^{\ell}),x_{\mathbf{b}%
}(x^{\ell},v^{\ell}),v^{\ell+1}) \ \ \ \text{with} \ n(x^{\ell})\cdot v^{\ell}>0. \notag
\end{equation}

(ii) Define the specular cycles, $\ell\geq 1,$
\[
(t^{\ell+1}, x^{\ell+1},v^{\ell+1}) = (t^{\ell}-t_{\mathbf{b}}(x^{\ell}, v^{\ell}), x_{\mathbf{b}}(x^{\ell},v^{\ell}),   v^{\ell} - 2n(x^{\ell})(v^{\ell}\cdot n(x^{\ell}))).
\]

(iii) Define the bounce-back cycles, $\ell\geq 1,$
\begin{equation*}
(t^{\ell+1}, x^{\ell+1},v^{\ell+1}) = (t^{\ell}-t_{\mathbf{b}}(x^{\ell},
v^{\ell}), x_{\mathbf{b}}(x^{\ell},v^{\ell}), -v^{\ell}).
\end{equation*}
Then for $\ell\geq 1$
\begin{eqnarray*}
t^{\ell}= t^{1} -(\ell-1)t_{\mathbf{b}}(x^{1},v^{1}) , \ \ x^{\ell} = \frac{%
1-(-1)^{\ell}}{2} x^{1} + \frac{1+(-1)^{\ell}}{2} x^{2}, \ \ v^{\ell+1} =
(-1)^{\ell+1} v.
\end{eqnarray*}

(iv) We define the backward trajectory as
\begin{equation} \notag 
\begin{split}
X_{\mathbf{cl}}(s;t,x,v)   \ = \ \sum_{\ell} \mathbf{1}_{[t^{\ell+1},t^{%
\ell})}(s) \big\{ x^{\ell}- (t^{\ell}-s)v^{\ell}\big\},  \ \ 
V_{\mathbf{cl}}(s;t,x,v)   \ = \ \sum_{\ell} \mathbf{1}_{[t^{\ell+1},t^{%
\ell})}(s) v^{\ell}.
\end{split}%
\end{equation}

 \end{definition}
 
Note that if $G(t,x,v)$ solves $\partial_{t}G + v\cdot \nabla_{x}G=0$ with a boundary condition (either diffuse, specular, or bounce-back boundary condition) then 
\[
G(t,x,v) = G(s,X_{\mathbf{cl}}(s;t,x,v), V_{\mathbf{cl}}(s;t,x,v)),
\] 
where $[X_{\mathbf{cl}}(s), V_{\mathbf{cl}}(s)]$ is defined respectively(\cite{Guo10}).

In this paper we establish the first Sobolev regularity away from the grazing set $\gamma _{0}$ for Boltzmann solutions in convex domains. One of the crucial ingredient is the construction of a distance function towards the
grazing set $\gamma _{0}$ to achieve this goal. 

\begin{definition}[Kinetic Distance]\label{K_D}
For $(x,v)\in\bar{\Omega}\times \mathbb{R}^{3},$
\begin{equation}\notag
\begin{split}
\alpha(x,v) 
 :=  | v \cdot \nabla \xi(x) |^{2}- 2 \{ v \cdot \nabla^{2} \xi (x) \cdot v \} \xi(x).
\end{split}
\end{equation}
\end{definition}
%

Due to (\ref{convex}), the kinetic distance $\alpha(x,v)$ vanishes if and only if $(x,v)\in\gamma_0$. The important technique to treat $\alpha$ along the trajectory is based on the geometric lemma : 
\begin{lemma}[Velocity lemma, Lemma 1 of \cite{Guo10}]\label{velocity_lemma}  Along the backward trajectory we define  
\begin{equation}\notag
\alpha(s;t,x,v) := \alpha(X_{\mathbf{cl}}(s;t,x,v), V_{\mathbf{cl}}(s;t,x,v)).
\end{equation}
Then there exists $\mathcal{C}=\mathcal{C}({\xi})>0$ such that, for all $0 \leq s_{1} , s_{2} \leq t,$
\[
e^{-\mathcal{C}|v||s_{1}-s_{2}|} \alpha(s_{1};t,x,v) \ \leq \ \alpha(s_{2};t,x,v) \ \leq \ e^{ \mathcal{C}|v||s_{1}-s_{2}|} \alpha(s_{1};t,x,v).
\]
 
\end{lemma}
\begin{proof}
The proof is basically same as the proof of Lemma 1 of \cite{Guo10} but the definition of $\alpha$ is slightly different. By the explicit computation, we have
\begin{equation}\label{alpha_inv}
\begin{split}
v\cdot \nabla _{x}\alpha &  =2v\cdot \nabla \xi (x)[v\cdot \nabla ^{2}\xi
\cdot v]-2v\cdot \nabla \xi (x)[v\cdot \nabla ^{2}\xi \cdot v]-2v\{v\cdot
\nabla ^{3}\xi (x)\cdot v\}\xi (x) \\
& =-2v\{v\cdot \nabla ^{3}\xi (x)\cdot v\}\xi (x) = O_{\xi}(1)|v|^{3} |\xi(x)|\\
& = O_{\xi}(1) |v|\alpha (x,v),
\end{split}%
\end{equation}
where we used $\{v\cdot \nabla ^{2}\xi (x)\cdot
v\}\backsim |v|^{2}$ from (\ref{convex}). Therefore there exists $\mathcal{C}=\mathcal{C}_{\xi}>0$ such that
\[
 - \mathcal{C} |v| \alpha(x,v) \ \leq \ v\cdot \nabla_{x} \alpha(x,v) \ \leq  \ \mathcal{C} |v| \alpha(x,v).
\] 
Since $\frac{d}{d s} \alpha(X_{\mathbf{cl}}(s;t,x,v), V_{\mathbf{cl}}(s;t,x,v)) = v \cdot \nabla_{x} \alpha(X_{\mathbf{cl}}(s;t,x,v), V_{\mathbf{cl}}(s;t,x,v)),$ we conclude the lemma.
\end{proof}

This crucial invariant property of $%
\alpha $ under operator $v\cdot \nabla _{x}$ is the key for our analysis. On
the other hand, unless $\nabla ^{3}\xi \equiv0$ (for example the domain is a ball
or an ellipsoid), a growth factor $|v|$ creates a geometric effect which is
out of control for our analysis. We introduce a strong decay factor $%
e^{-\varpi\langle v\rangle t}$ with sufficiently large $\varpi>0$ to overcome such a geometric effect :%
\begin{equation}  \label{dist}
e^{-\varpi\langle v\rangle t}\alpha (x,v).
\end{equation}%
A direct computation yields
\begin{equation}\notag
\begin{split}
\{\partial_{t} + v\cdot \nabla_{x}\} [e^{-\varpi \langle v\rangle t}\alpha(x,v)] &= -\varpi \langle v\rangle e^{-\varpi \langle v\rangle t} \alpha(x,v) - e^{-\varpi \langle v\rangle t} 2v \{  v\cdot \nabla^{3}\xi(x)\cdot v\} \\
&\lesssim (-\varpi + O_{\xi}(1)) \langle v\rangle e^{-\varpi \langle v\rangle t} \alpha(x,v),
\end{split}
\end{equation}
with the geometric contribution $O_{\xi }(1)=\frac{2v\{v\cdot \nabla
^{3}\xi (x)\cdot v\}\xi }{\alpha \langle v\rangle }$ where we used the convexity of $\xi$ in (\ref{convex}). Throughout this paper we assume
\begin{equation}
\varpi>\max \frac{2v\{v\cdot \nabla ^{3}\xi (x)\cdot v\}\xi }{\alpha \langle
v\rangle }.  \label{llarge}
\end{equation}%
Remark that if $\xi $ is quadratic (for example if the domain is a ball or an
ellipsoid) then we are able to set $\varpi=0$ and $\{\partial _{t}+v\cdot \nabla
_{x}\}\alpha\equiv0.$


We denote $F= \sqrt{\mu }f$ ($f$ could be large) where $\mu
=e^{-\frac{|v|^{2}}{2}}$ is a global normalized Maxwellian. Then $f$ satisfies
\begin{equation}\label{boltzamnn_f}
 \partial_{t}f +v\cdot \nabla_{x} f   = \Gamma_{\text{gain } }(f,f) - \nu(\sqrt{\mu} f)f.
\end{equation}
Here
\begin{equation} \label{nu}
\begin{split}
\nu( \sqrt{\mu}f)(v)= \nu(F)(v)&:= \frac{1}{\sqrt{\mu(v)}}Q_{\text{loss}} (\sqrt{\mu}f,\sqrt{\mu}f)(v) \\
&=  \int_{\mathbb{R}^{3}} \int_{\mathbb{S}^{2}} B(v-u,\omega) \sqrt{\mu(u)}
f (u) \mathrm{d}\omega \mathrm{d}u,
\end{split}
\end{equation}
and the gain term of the nonlinear Boltzmann operator is given by
\begin{equation}\label{gain_Gamma}
\begin{split}
\Gamma_{\text{gain}}(f_{1},f_{2})(v) & := \frac{1}{\sqrt{\mu }} Q_{\text{gain}} (\sqrt{\mu}f_{1}, \sqrt{\mu}f_{2})(v)\\
&=
\int_{\mathbb{R}^{3}} \int_{\mathbb{S}^{2}} B(v-u,\omega) \sqrt{\mu(u)} f_{1}(u^{\prime}) f_{2}(v^{\prime}) \mathrm{d}\omega\mathrm{d}u.
\end{split}
\end{equation}
%
The corresponding boundary conditions for $f$ are followings:

\textit{(i) Diffuse boundary condition:}
\begin{equation}\label{diffuseBC}
f(t,x,v)=c_{\mu } \sqrt{\mu (v)}\int_{n(x)\cdot  u>0}f(t,x,u) \sqrt{\mu(u)}\{n(x)\cdot u\}\mathrm{d}u, \ \ \ \text{on} \ \gamma_{-}.
\end{equation}%

\textit{(ii) Specular reflection boundary condition:}
\begin{equation} \label{specularBC}
f(t,x,v)= f(t,x,R_{x}v), \ \ \ \text{on} \ \gamma_{-}. 
\end{equation}%

\textit{(iii) Bounce-back reflection boundary condition:}
\begin{equation} \label{bbBC}
f(t,x,v)= f(t,x,-v), \ \ \ \text{on} \ \gamma_{-}.
\end{equation}%


\vspace{12pt}
 
\noindent{\textbf{\large{ 1. Diffuse Reflection BC}}} 
 
 \vspace{8pt}
 
We denote $||\cdot ||_{p}$ the $L^{p}(\Omega \times \mathbb{R}^{3})$ norm,
while $|\cdot |_{\gamma ,p}$ is the $L^{p}(\partial \Omega \times \mathbb{R}%
^{3};\mathrm{d}\gamma )$ norm and $|\cdot |_{\gamma _{\pm },p}=|\cdot
\mathbf{1}_{\gamma _{\pm }}|_{\gamma ,p}$ where $\mathrm{d}\gamma
=|n(x)\cdot v|\mathrm{d}S_{x}\mathrm{d}v$ with the surface measure $\mathrm{d%
}S_{x}$ on $\partial\Omega$. Denote $\langle v\rangle =\sqrt{1+|v|^{2}}$. We define
\begin{equation}
\partial _{t}f(0)=\partial _{t}f_{0}\equiv -v\cdot \nabla
_{x}f_{0}  +\Gamma_{\text{gain}} (f_{0},f_{0}) -\nu(\sqrt{\mu}f_{0})f_{0}.  \label{f0}
\end{equation}
Throughout this paper we always assume
\[
F_{0} = \sqrt{\mu} f_{0} \geq 0.
\]

\begin{theorem}\label{Global_p}
Assume that $0 \leq \kappa \leq 1$ in (\ref{collision_Q}) and $f_{0}\in W^{1,p}(\Omega\times\mathbb{R}^3)$ and $ || \nabla_{x} f_{0}||_{p} + ||\nabla_{v}f_{0}||_{p}  + ||  e^{\theta|v|^{2}} f_0||_\infty  <+\infty $ for $0< \theta < \frac{1}{4}$ and any fixed $1< p<2,$
 and the compatibility condition on $(x,v) \in \gamma_-,$
\begin{equation}
f_0(x,v) = c_{\mu}\sqrt{\mu(v)} \int_{n(x)\cdot u>0} f_0(x,u) \sqrt{\mu(u)} \{n(x)\cdot u\} \mathrm{d} u,\label{compatibility_condition_1}
\end{equation}
then there exists $T=T(|| e^{\theta|v|^{2}} f_{0} ||_{\infty})>0$ such that $f \in L^\infty_{loc}([0, T] ;W^{1,p}(\Omega\times\mathbb{R}^3))$ solves the Boltzmann equation (\ref{boltzamnn_f}) with diffuse boundary condition (\ref{diffuseBC}), and for all $0\leq t\leq T$
\begin{equation}\label{global_p}
\begin{split}
&  ||\nabla _{x}f(t)||_{p}^p+||\nabla
_{v}f(t)||_{p}^p  + \int_0^t  \big[ \      |\nabla_x f(s)|_{\gamma,p}^p  + |\nabla_v f(s)|_{\gamma,p}^p \big] \mathrm{d} s 
\\
& \lesssim_t \ ||\nabla _{x}f_{0}||_{p}^p+||\nabla
_{v}f_{0}||_{p}^p + P(|| e^{\theta|v|^{2}} f_{0} ||_{\infty}),
\end{split}
\end{equation}
where $P$ is some polynomial.

Furthermore, if $F_{0}= \mu + \sqrt{\mu}g_{0}$ with $|| e^{\theta|v|^{2}}g_{0} ||_{\infty} \ll 1,$ then the theorem holds with $ \nabla_{x}g(t),$ and $\nabla_{v}g(t)$ for all $t\geq 0.$  
\end{theorem}

There can be no size restriction on initial data $F_{0} = \sqrt{\mu}f_{0}$. On the other hand, we also remark that from \cite{Guo10,EGKM}, the assumption $|| e^{\theta|v|^{2}} g_0||_{\infty}\ll 1$ for $F_{0}= \mu + \sqrt{\mu}g_{0}$ without a mass constraint $\iint_{\Omega\times\mathbb{R}^3} g_0 \sqrt{\mu} \mathrm{d} v \mathrm{d} x =0$ ensures a uniform-in-time bound as $\sup_{0\leq t\leq \infty}|| e^{\theta|v|^{2}} g(t)||_{\infty}\lesssim ||   e^{\theta|v|^{2}} g_0||_{\infty}$ (not a decay). In this case, the estimate (\ref{global_p}) is a global-in-$x$ estimate which includes the grazing set $\gamma _{0}$ and the constant grows exponentially with time. 

Moreover, we show that the estimate of (\ref{global_p}) in Theorem \ref{Global_p} for $p<2$
is indeed optimal even for the free transport equation $\partial
_{t}f+v\cdot \nabla _{x}f=0$ with the diffuse boundary condition (Lemma \ref{optimal}). In fact,
the boundary integral blows up at $p=2.$ 

We now illustrate main ideas of the proof of Theorem \ref{Global_p}. Clearly, both $t$ and $v$
derivatives behave nicely for the diffuse boundary condition as for $%
(x,v)\in \gamma _{-},$
\begin{eqnarray}
\partial _{t}f(t,x,v) &=&c_{\mu }\sqrt{\mu (v)}\int_{n(x)\cdot u>0}\partial _{t}f(t,x,u)\sqrt{\mu (u)}\{n({x})\cdot
u\}\mathrm{d}u,  \label{boundary_t} \\
\nabla _{v}f(t,x,v) &=&c_{\mu }\nabla _{v}\sqrt{\mu (v)}\int_{n(x)\cdot
u>0}f(t,x,u)\sqrt{\mu (u)}\{n({x})\cdot
u\}\mathrm{d}u.  \label{boundary_v}
\end{eqnarray}%
Let $\tau _{1}(x)$ and $\tau _{2}(x)$ be unit tangential vectors to $\partial\Omega$ satisfying $\tau
_{1}(x)\cdot n(x)=0=\tau _{2}(x)\cdot n(x)$ and $\tau _{1}(x)\times \tau
_{2}(x)=n(x)$. Define the orthonormal transformation from $\{n,\tau
_{1},\tau _{2}\}$ to the standard bases $\{\mathbf{e}_{1},\mathbf{e}_{2},%
\mathbf{e}_{3}\}$, i.e. $\mathcal{T}(x)n(x)=\mathbf{e}_{1},\ \mathcal{T}%
(x)\tau _{1}(x)=\mathbf{e}_{2},\ \mathcal{T}(x)\tau _{2}(x)=\mathbf{e}_{3},$
and $\mathcal{T}^{-1}=\mathcal{T}^{t}.$ Upon a change of variable: $%
u^{  \prime }=\mathcal{T}(x)u,$ we have%
\begin{equation*}
n(x)\cdot u=n(x)\cdot \mathcal{T}^{t}(x)u^{\prime }=n(x)^{t}%
\mathcal{T}^{t}(x)u^{ \prime }=[\mathcal{T}(x)n(x)]^{t}u^{
\prime }=\mathbf{e}_{1}\cdot u^{  \prime }=u_{1}^{  \prime },
\end{equation*}%
then
\begin{equation*}
c_{\mu }\sqrt{\mu (v)}\int_{n(x)\cdot u>0}f(t,x,u)\sqrt{%
\mu (u)}\{n(x)\cdot u\}\mathrm{d}u=c_{\mu }%
\sqrt{\mu (v)}\int_{u_{1}^{  \prime }>0}f(t,x,\mathcal{T}%
^{t}(x)u^{  \prime })\sqrt{\mu (u^{  \prime })}\{u_{1}^{
\prime }\}\mathrm{d}u^{  \prime },
\end{equation*}%
so that we can further take tangential derivatives $\partial _{\tau _{i}}$
as, for $(x,v)\in \gamma _{-},$
\begin{equation}
\begin{split}
&\partial _{\tau _{i}}f(t,x,v)\\
& =c_{\mu }\sqrt{\mu (v)}\int_{u_{1}^{
\prime }>0}\Big\{\partial _{\tau _{i}}f(t,x,\mathcal{T}^{t}(x)u^{
\prime })+\nabla _{v}f(t,x,\mathcal{T}^{t}(x)u^{  \prime })\frac{%
\partial \mathcal{T}^{t}(x)}{\partial \tau _{i}}u^{  \prime }\Big\}%
\sqrt{\mu (u^{  \prime })}\{u_{1}^{  \prime }\}\mathrm{d}u^{
\prime } \\
& =c_{\mu }\sqrt{\mu (v)}\int_{n(x)\cdot u >0}\partial _{\tau
_{i}}f(t,x,u)\sqrt{\mu (u)}\{n(x)\cdot u\}%
\mathrm{d}u\\
& \ \ +c_{\mu }\sqrt{\mu (v)}\int_{n(x)\cdot u>0}\nabla
_{v}f(t,x,u)\frac{\partial \mathcal{T}^{t}(x)}{\partial \tau _{i}}%
\mathcal{T}(x)u\sqrt{\mu (u)}\{n(x)\cdot u\}%
\mathrm{d}u.
\end{split}
\label{boundary_tau}
\end{equation}%

The difficulty is always the control of the normal spatial derivative of $%
\partial _{n}.$ From the general method of proving regularity in PDE with boundary conditions, it is
natural to use the Boltzmann equation to solve the normal derivative $%
\partial _{n}f$ inside the region$,$ in terms of $\partial _{t}f,$ $\nabla
_{v}f,$ and $\partial _{\tau }f$ $\ $as:
\begin{equation}
\partial _{n}f(t,x,v)=-\frac{1}{n(x)\cdot v}\bigg\{ \partial
_{t}f+\sum_{i=1}^{2}(v\cdot \tau _{i})\partial _{\tau _{i}}f   -\Gamma_{\text{gain}}
(f,f)  +  \nu(\sqrt{\mu}f)f\bigg\} ,  \label{fn}
\end{equation}%
at least near $\partial \Omega .$ Unfortunately, this standard approach
encounters a severe difficulty: $\frac{1}{n(x)\cdot v}$ $\notin L_{loc}^{1}$
in the velocity space (a $L^{\infty }$ bound is desirable for any $W^{1,p}$
estimate).

The first new ingredient of our approach is to use (\ref{fn}) \textit{not}
inside the domain, but at the boundary $\partial \Omega .$ Using special
feature of the diffuse boundary condition and (\ref{boundary_t}), (\ref%
{boundary_v}) and (\ref{boundary_tau}), we can express $\partial _{n}f$ at $%
(x,v)\in \gamma _{-}$ as
\begin{equation}
\begin{split}
&\partial _{n}f(t,x,v)\\
&= -\frac{1}{n(x)\cdot v}\bigg\{\ \sqrt{\mu (v)}%
\int_{n(x)\cdot u>0}\partial _{t}f(t,x,u)\sqrt{\mu
(u)}\{n(x)\cdot u\}\mathrm{d}u \\
& \ \ \ \ \ \ \ \ \ \ \ \ \ \ \ +\sum_{i=1}^{2}(v\cdot \tau _{i})\sqrt{\mu
(v)}\int_{n(x)\cdot u>0}\partial _{\tau _{i}}f(t,x,u)%
\sqrt{\mu (u)}\{n(x)\cdot u\}\mathrm{d}u \\
& \ \ \ \ \ \ \ \ \ \ \ \ \ \ \ +\sum_{i=1}^{2}(v\cdot \tau _{i})\sqrt{\mu
(v)}\int_{n(x)\cdot u>0}\nabla _{v}f(t,x,u)\frac{%
\partial \mathcal{T}^{t}(x)}{\partial \tau _{i}}\mathcal{T}(x)u
\sqrt{\mu (u)}\{n(x)\cdot u\}\mathrm{d}u \\
& \ \ \ \ \ \ \ \ \ \ \ \ \ \ \ -\Gamma_{\mathrm{gain}}  (f,f)  +\nu(\sqrt{\mu}f)f\ \ \ \bigg\},
\end{split}
\label{boundary_n}
\end{equation}

Due to the additional $u$ integral in (\ref{boundary_n}) and the
crucial factor $|n(x)\cdot u|$ in the measure $\mathrm{d}\gamma $ on the boundary $\gamma$, it is
clear that the singularity of $|\partial _{n}f|^{p}|n\cdot v|$ in (\ref%
{boundary_n}) is roughly of the order
\begin{equation*}
\frac{1}{\{n\cdot v\}^{p-1}},
\end{equation*}%
so that its $v-$integration is precisely finite if $1\leq $ $p<2$, and indeed its $v$ integration is \textit{uniformly} bounded
with respect to $x$.

However, in order to control $\partial _{t}f,\nabla _{v}f$ and $\partial
_{\tau }f$ $\ $for $p<2$, a new difficulty arises. It is well-known from \cite{Guo10,EGKM} that a crucial boundary estimate for
diffuse boundary takes the form of a $L^{2}-$contraction:
\begin{equation*}
\int_{\gamma _{-}}h^{2}\mathrm{d} \gamma \leq \int_{\gamma _{+}}h^{2}\mathrm{d}\gamma.
\end{equation*}%
Unfortunately, this is not expected to be valid for $p\neq 2,$ so it is
impossible to absorb the incoming part $\gamma _{\_}$ solely by the
outgoing part $\gamma _{+}$ part.

Our second new ingredient is to split the $\gamma _{+}$ integral into near
grazing set $\gamma _{+}^{\varepsilon }$ and the rest for $p\neq 2$ for our
boundary representation for derivatives (\ref{boundary_t}), (\ref{boundary_v}%
), (\ref{boundary_tau}), and (\ref{boundary_n}). For small $\varepsilon >0$ we define $\gamma _{+}^{\varepsilon }$, the set of almost grazing velocities or large
velocities
\begin{equation}
\gamma _{+}^{\varepsilon }=\{(x,v)\in \gamma _{+}:v\cdot n(x)<\varepsilon
\text{ or }|v|>1/\varepsilon \}.  \label{def:gamma_epsilon}
\end{equation}%
Denote $\partial= [\partial_t, \nabla_x, \nabla_v]$. We can roughly obtain
\begin{equation*}
\begin{split}
\int_{\gamma_-} |\partial f|^p &\lesssim \int_{\partial\Omega} \left( \int_{n\cdot v >0} |\partial f | \mu^{1/4} \{n\cdot v\} \mathrm{d} v\right)^p + \text{good terms},\\
&\lesssim \int_{\partial\Omega} \left( \int_{\{v:(x,v)\in\gamma_+^\varepsilon\}} |\partial f | \mu^{1/4} \{n\cdot v\}  \right)^p+
 \int_{\partial\Omega} \left( \int_{\{v:(x,v)\in\gamma_+\backslash\gamma_+^\varepsilon\}} |\partial f | \mu^{1/4} \{n\cdot v\}  \right)^p
 + \text{good terms},\\
 &\lesssim \sup_x \left(\int_{\{v: (x,v)\in\gamma_+^\varepsilon\}}\mu^{q/4}\{n\cdot v\}\mathrm{d} v\right)^{p/q} \int_{\gamma_+^\varepsilon} |\partial f|^p\mathrm{d} \gamma + \int_{\gamma_+ \backslash \gamma_+^\varepsilon} |\partial f|^p \mathrm{d} \gamma + \text{good terms}.
 \end{split}
\end{equation*}

It is important to realize that $\sup_x \left(\int_{\{v: (x,v)\in\gamma_+^\varepsilon\}}\mu^{q/4}\{n\cdot v\}\mathrm{d} v\right)^{p/q}$ has a small measure of order $\varepsilon ,$ for $p>1$, so that it can be
absorbed by the outgoing part $\int_{\gamma _{+}}.$ Fortunately, the
outgoing boundary integral $\int_{\gamma _{+}\setminus \gamma
_{+}^{\varepsilon }}$ can be further bounded by the integration in the bulk
and initial data by Lemma \ref{le:ukai} with a crucial time integration. On
the other hand, such a process produces a large constant in the Gronwall
estimates and leads to a growth in time. Of course, such approach breaks
down at $p=1$.

\begin{theorem}\label{weigh_W1p}
Assume the compatibility condition (\ref{compatibility_condition_1}) and $0< \kappa\leq 1$ and recall (\ref{f0}).

For any fixed $2\leq p< \infty$ and $\frac{p-2}{2p} < \beta < \frac{p-1}{2p},$ if
$|| \alpha ^{ \beta}\nabla_{ x,v}f_{0}||_p  + ||  e^{\theta|v|^2}f_0||_\infty < \infty $ for some $0<\theta< \frac{1}{4},$ then there exists $T=T(||  e^{\theta|v|^{2}}  f_{0}||_{\infty})>0$ such that $ e^{-\varpi \langle v\rangle t}\alpha^\beta \nabla_x f,   e^{-\varpi \langle v\rangle t}\alpha^\beta \nabla_v f \in L^\infty_{loc}([0,T];L^p(\Omega\times\mathbb{R}^3))$  and for all $0 \leq t\leq T,$
\begin{equation*}
\begin{split}
&  ||  e^{-\varpi \langle v\rangle t}\alpha ^\beta \nabla_{ x,v}f(t)||_{p}^p+  \int_0^t   |   e^{-\varpi \langle v\rangle t}\alpha ^\beta \nabla_{ x,v} f(s)|_{\gamma,p}^p   \mathrm{d} s\\
& \lesssim_t \ || \alpha^\beta
\nabla _{ x,v}f_{0}||_p^p +P( ||  e^{\theta|v|^2}f_0||_\infty ),
\end{split}
\end{equation*}
where $P$ is some polynomial.

If $||\alpha^{1/2} \nabla _{x,v}f_{0}||_{\infty}  +||  e^{\theta|v|^2}f_0||_\infty <+\infty$ for some $0<\theta< \frac{1}{4},$ then $ e^{-\varpi \langle v\rangle t} \alpha ^{1/2} \nabla_{ x,v} f  \in L^\infty ([0,T];L^\infty(\Omega\times\mathbb{R}^3))$ such that for all $0\leq t\leq T,$
\begin{equation*}
 || e^{-\varpi \langle v\rangle t} \alpha ^{1/2}\nabla_{ x,v} f(t)||_{\infty }   \ \lesssim
_{ t}\ || \alpha^{1/2}\nabla_{ x,v}  f_{0}||_{\infty }    + P(||  e^{\theta
|v|^{2}}f_{0}||_{\infty }).
\end{equation*}
If $\alpha^{1/2}\nabla f_0 \in C^0(\bar{\Omega}\times\mathbb{R}^3)$ and
\begin{equation}
v\cdot \nabla_x f_0   -\Gamma(f_0,f_0) = c_\mu \sqrt{\mu} \int_{n\cdot u >0} \big\{
u\cdot\nabla_x f_0   -\Gamma(f_0,f_0)
\big\} \sqrt{\mu} \{n\cdot u\} \mathrm{d} u,\label{compatibility_condition_2}
\end{equation}
is valid for $\gamma_- \cup \gamma_0,$ then $f\in C^{1}$ away from the grazing set $\gamma_{0}$. 

Furthermore, if $F_{0} = \mu + \sqrt{\mu} g_{0}$ with $|| e^{\theta |v|^{2}} g_{0} ||_{\infty} \ll 1,$ then the theorem holds with $\nabla_{x}g(t)$ and $\nabla_{v}g(t)$ for all $t\geq 0.$ 
\end{theorem}

There can be no size restriction on initial data $F_{0} = \sqrt{\mu} f_{0}$. On the other hand, we also remark that from \cite{Guo10,EGKM}, the assumption $|| e^{\theta|v|^{2}} g_0||_{\infty}\ll 1$ for $F_{0}= \mu + \sqrt{\mu}g_{0}$ without a mass constraint $\iint_{\Omega\times\mathbb{R}^3} g_0 \sqrt{\mu} \mathrm{d} v \mathrm{d} x =0$ ensures a uniform-in-time bound as $\sup_{0\leq t\leq \infty}|| e^{\theta|v|^{2}} g(t)||_{\infty}\lesssim ||   e^{\theta|v|^{2}} g_0||_{\infty}$ (not a decay). 

 We remark for $%
\varpi\neq 0$, $\partial f(t) \sim e^{ \varpi\langle v\rangle t} $ so
that in terms of solution $f(t),$ such an estimate not only creates an
exponential growth in time, but also creates less integrability in velocity. Furthermore, when $\varpi\neq 0,$ we crucially need a strong weight
function $e^{\theta |v|^{2}}$ to balance such a factor $e^{-\varpi\langle v\rangle
t}$, which produces a super exponential growth $e^{t^{2}}$ in time. We suspect that it is
impossible to obtain a uniform in time estimate especially when $\varpi\neq 0.$
The distance function $\alpha$ plays an important role in the
study of regularity in convex domains for Vlasov equations (\cite{Guo95,HV}%
), which can be controlled along the characteristics via the geometric Velocity
lemma (Lemma \ref{velocity_lemma}). However, such an approach has not been successful in the study of
Boltzmann equation due to the non-local nature of the Boltzmann collision
operator, which mixes up different velocities so that their distance towards
$\gamma _{0}$ can not be controlled. In addition to the key boundary
representation, we establish a delicate estimate for interaction of $%
 e^{-\varpi \langle v\rangle t} \alpha(x,v)$ and the collision kernel $ e^{-\varpi \langle v\rangle t} \alpha(x,v) ^{\beta  }\Gamma_{\mathrm{gain}}(\frac{\partial f}{  e^{-\varpi \langle v\rangle t} \alpha ^{\beta  }},f)$ in (\ref{Kd}) for $\beta <\frac{p-1}{2p}$. An additional
requirement $\beta >\frac{p-2}{2p}$ is needed to control the boundary
singularity in (\ref{d_lambda_boundary}). These estimates are sufficient to treat
the case for $\beta <1/2,$ but unfortunately these fail for the case $\beta =1/2,$
which accounts for the important $C^{1}$ estimate. In order to establish the
$C^{1}$ estimate, we employ the Lagrangian view point, estimating along the stochastic cycles \cite{Guo10,EGKM} in Definition \ref{cycles}.

Our fourth new ingredient is the dynamical non-local to local estimates (Lemma \ref{lemma_nonlocal}). Even
though $e^{-\varpi \langle v\rangle t} \sqrt{\alpha } \Gamma_{\mathrm{gain}}(\frac{\partial f}{ e^{-\varpi \langle v\rangle t} \sqrt{\alpha }}, f)$ is impossible to estimate
directly due to severe singularity of $\frac{1}{ e^{-\varpi \langle v\rangle t} \sqrt{\alpha(x,v)} }$ in the velocity
space, along the characteristics,  $\frac{1}{ e^{-\varpi \langle v\rangle (t-s)} \sqrt{\alpha(x-(t-s)v,v)} }$ is
integrable in time for a convex domain. Therefore the integral
\begin{equation*}
\int_{t-t_{\mathbf{b}}(x,v)}^{t}  e^{-\varpi \langle v\rangle( t-s)} \sqrt{\alpha(x,v)} 
\Gamma_{\mathrm{gain}}(\frac{\partial f}{ e^{-\varpi \langle v\rangle( t-s)} \sqrt{\alpha(x-(t-s)v,v) }},f) \mathrm{d}s
\end{equation*}%
can be controlled by first integrating over time, and we can close the
desired estimate.

\vspace{15pt}
 
\noindent{\textbf{\large{ 2. Dynamical non-local to local estimates}}}

 \vspace{0pt }

\begin{lemma}
\label{lemma_nonlocal} Let $(t,x,v)\in \lbrack 0,\infty )\times \bar{\Omega}%
\times \mathbb{R}^{3}$ and $\frac{1}{2}<\beta <\frac{3}{2}$ and $0<\kappa \leq 1$ and $r\in
\mathbb{R}$ and $Z(s,x,v) \geq 0$.

(1) Let $X_{\mathbf{cl}}(s;t,x,v)= x-(t-s)v$ on $s\in [t-t_{\mathbf{b}}(x,v),t].$ 

For any $\varepsilon >0$, there exist $l \gg_{\xi}1$ such that
\begin{equation}
\begin{split}
& \int^{t}_{t-t_{\mathbf{b}}(x,v)}\int_{\mathbb{R}^{3}}e^{-l\langle v\rangle
(t-s)}\frac{e^{-\theta |v-u|^{2}}}{|v-u|^{2-\kappa }[\alpha ( X_{\mathbf{cl}}(s;t,x,v), {u})]^{\beta }}\frac{\langle u\rangle ^{r}}{\langle v\rangle
^{r}}Z(s,x,v)\mathrm{d}u\mathrm{d}s \\
& \lesssim  \ \ \min \left\{ \frac{\varepsilon ^{\frac{3}{2}%
-\beta }}{|v|^{2}\{\alpha (x,v)\}^{\beta -1}},\frac{\{\alpha (x,v)\}^{\frac{1
}{4}-\frac{\beta }{2}}|t_{Z}|^{\frac{3}{2}-\beta }}{|v|^{2\beta -1}}\right\}
\sup_{s\in \lbrack t-t_{\mathbf{b}}(x,v),t]}\{e^{-l\langle v\rangle
(t-s)}Z(s,x,v)\} \\
& \ \ \ \ \ \ \ \ +\frac{C_{\varepsilon}}{ \{\alpha (x,v)\}^{\beta -1/2}}%
\int^{t}_{t-t_{\mathbf{b}}(x,v)}e^{- \frac{l}{2} \langle v\rangle (t-s)}Z(s,x,v)\mathrm{d}%
s,
\end{split}
\label{nonlocal}
\end{equation}%
where $t_{Z}=\sup \{s:Z(s,x,v)\neq 0\}.$
   
   \vspace{4pt}
   
(2)    Let $[X_{\mathbf{cl}}(s;t,x,v), V_{\mathbf{cl%
}}(s;t,x,v)]$ be the specular backward trajectory or the bounce-back
trajectory in Definition \ref{cycles}.

For any $\varepsilon>0$, there exist $l \gg_{\xi}1$ such that
\begin{equation}  \label{specular_nonlocal}
\begin{split}
& \int_{0}^{t} \int_{\mathbb{R}^{3}} e^{- l\langle v\rangle (t-s)} \frac{%
e^{-\theta|V_{\mathbf{cl}}(s;t,x,v)-u|^{2}}}{|V_{\mathbf{cl}}(s;t,x,v)-u|^{2-\kappa}}
\frac{\langle u\rangle^{r}}{\langle v\rangle^{r}} \frac{Z(s,x,v)}{ \big[ %
\alpha(X_{\mathbf{cl}}(s;t,x,v),u) \big]^{\beta}} \mathrm{d}u\mathrm{d}s \\
& \lesssim \ \frac{ O(\varepsilon)}{\langle v\rangle \big[\alpha(x,v)\big]^{\beta-1/2} }
\sup_{0 \leq s\leq t} \big\{ e^{-  \frac{l}{2}  \langle v\rangle
(t-s)} Z(s,x,v) \big\}.
\end{split}%
\end{equation} 
 \end{lemma}

The control of $\int_{u} \frac{e^{-\theta|v-u|^{2}}}{|v-u|^{2-\kappa}} \frac{1}{\alpha(u) ^{\beta }} $ is addressed throughout such so-called dynamical non-local to local estimates. We discover that
the non-local $u$ integration does not destroy the local property, upon a
crucial time integration along the characteristics. The proof of such
non-local to local estimates are a combination of analytical and geometrical arguments. The
first part is a precise estimate of $u$ integration which is bounded via $%
\frac{1}{|v|^{2\beta -1}|\xi (x-(t-s)v)|^{\beta -1/2}}.$ In this part of the
proof we make use of a series of change of variables to obtain the precise
power. The second part is to relate $\frac{1}{|\xi (x-(t-s)v)|^{\beta -1/2}}$
back to $\frac{1}{\alpha }.$ Clearly, 
\begin{equation*}
\frac{1}{|\xi (x-(t-s)v)|^{2}}\backsim \frac{1}{\alpha }\backsim \frac{1}{|v\cdot
\nabla \xi (x-(t-s)v|^{2}+ |\xi (x-(t-s)v) ||v|^{2}}.
\end{equation*}%
for $|\xi (X_{\mathbf{cl}}(s)) ||v|^{2}$ is larger than $|v\cdot \nabla \xi (X_{\mathbf{cl}}(s))|.$ On the other hand,
when $|v\cdot \nabla \xi (X_{\mathbf{cl}}(s)) |$ dominates, this can only be achieved through a
crucial use of time integration and geometric Velocity lemma (Lemma \ref{velocity_lemma}), by connecting 
\begin{equation*}
\mathrm{d}t\backsim \frac{\mathrm{d}\xi }{|v\cdot \nabla \xi |},
\end{equation*}%
and recover $\alpha $ as in the bound of $\xi-$integration through the
geometric Velocity Lemma (Lemma \ref{velocity_lemma}).

The more striking feature is that not only our estimates retain the local
structure for $\alpha ,$ but they \textit{gain} $\sqrt{\alpha }$ order of
regularity. Such a precise gain of regularity is exactly enough to balance
out the singularity in $\alpha $ appeared in $\partial X_{\mathbf{cl}%
}(s;t,x,v)$ and $\partial V_{\mathbf{cl}}(s;t,x,v)$ in both the specular and
bounce-back cycles. In order to squeeze out a small constant for $|v|\gg1,$
we need to use the decay of $e^{-l\langle v\rangle (t-s)}.$ This requires a
precise regrouping of the cycles according to the time scale of 
\begin{equation*}
t|v|\sim 1.
\end{equation*}%
Within such an important time scale, $V_{\mathbf{cl}}(s;t,x,v)$ stays\textit{%
\ }almost \textit{invariant} due to the Velocity Lemma(Lemma \ref{velocity_lemma}). We then are able to
obtain precise estimate for the number of bounces within $t|v|\backsim 1$
and extract smallness from $e^{-l\langle v\rangle (t-s)}$ for $t-s\geq \frac{%
1}{|v|}.$ On the other hand, for $t-s\leq \frac{1}{|v|},$ the smallness
comes from Lemma \ref{lemma_nonlocal}.

\vspace{15pt}
 
\noindent{\textbf{\large{ 3. Specular Reflection BC}}} 
 
 \vspace{3pt }

Recall the specular reflection boundary condition in (\ref{specularBC}) and the specular cycles in Definition \ref{cycles}. Our main
theorem is as follow.

\begin{theorem}\label{main_specular} 
Assume $F_{0} = \sqrt{\mu}f_{0} \geq 0$ and $f_{0} \in   W^{1,\infty}(\Omega\times \mathbb{R}^{3})$ and $0< \kappa\leq 1$ for $1<\beta< \frac{3}{2},  \ 0<  \theta < \frac{1}{4},$ and $b\in\mathbb{R},$
$$
\big|\big|    \frac{      \alpha^{\beta-\frac{1}{2} }  }{\langle v\rangle^{b }} \partial_{x} f_{0}    \big|\big|_{\infty} + \big|\big|  
  \frac{  |v|^{2} {\alpha}^{ \beta-1}   }{\langle v\rangle^{b }} \partial_{v} f_{0}    \big|\big|_{\infty}
   +  ||  e^{\theta|v|^{2}} f_{0}||_{\infty} <\infty, 
$$
and the compatibility condition  
\begin{equation}\label{compatibility_specular}
f_{0}(x,v) = f_{0}(x, R_{x}v) \ \   \text{on}  \ (x,v) \in\gamma_{-}.
\end{equation}
Then for all $0\leq t\leq T$ with $T=T( || e^{\theta|v|^{2}}f_{0} ||_{\infty})>0$ 
\begin{equation}\label{specular_main_estimate}
\begin{split}
&|| e^{-\varpi \langle v\rangle t} \frac{  \alpha^{\beta}    }{\langle v\rangle^{b+1}} \partial_{x}f(t) ||_{\infty} + || e^{-\varpi \langle v\rangle t}  \frac{| v |   \alpha^{\beta- \frac{1}{2}}   }{\langle v\rangle^{b }}  \partial_{v}f(t) ||_{\infty}  \\
\lesssim_{\xi,t}   & \ 
  \big|\big|   \frac{   \alpha^{\beta-\frac{1}{2} }      }{\langle v\rangle^{b }} \partial_{x} f_{0}    \big|\big|_{\infty} + \big|\big| \frac{  |v|^{2} {\alpha}^{ \beta-1}      }{\langle v\rangle^{b }}\partial_{v} f_{0}    \big|\big|_{\infty} + P (||   \partial_{t}f_{0}||_{\infty})
  + P (||  e^{\theta|v|^{2}} f_{0}||_{\infty}).
\end{split}
\end{equation}

Moreover, if $\Omega$ is real analytic ($\xi$ is real analytic on $\mathbb{R}^{3}$) and $F_{0} = \mu + \sqrt{\mu} g_{0} \geq 0$ with $||e^{\theta|v|^{2}} g_{0}||_{\infty} \ll 1$ then this theorem holds for the arbitrarily large time $t \geq 0$.

Furthermore, if $ f_{0} \in C^{1}$ and 
 \begin{equation}\label{compatibility_specular_t}
v\cdot \nabla_{x}f_{0}(x,v) = R_{x}v \cdot \nabla_{x}f_{0}(x, R_{x}v) \ \ \text{on}  \ (x,v) \in\gamma_{-}.
\end{equation}
then $f\in C^{1}$ away from the grazing set $\gamma_{0}.$

\end{theorem}

There can be no size restriction on initial data $F_{0} = \sqrt{\mu}f_{0}$. We remark from the local existence theorem, $T>0.$ The analyticity is a
crucial assumption to ensure global stability in \cite{Guo10}. We also remark that
the specular theorem is drastically different from the diffusive theorem: in
addition to the loss of moments, there is a loss of regularity of $\alpha $
with respect to the initial data. This makes it impossible to use the
continuity argument to choose small time interval to close the estimates. We
need to use large $\varpi$ in $e^{-\varpi \langle v\rangle t}$ to extract a small constant to close, which requires
extra precise estimates. We note that in 3D case, $\beta >1/2,$ due to the
failure of the proof of the non-local to local estimates for the critical $%
\beta =1/2$(Lemma \ref{lemma_nonlocal}). On the other hand, in 2D, due to boundedness of $\partial
_{v_{3}}f$ from $x_{3}-$invariance, we are able to estimate $\partial
_{v}\Gamma _{\text{gain }}$ for the critical case $\beta =1/2$ by Lemma \ref{2D_Gamma}.

In additional to the dynamical non-local to local estimate, the second
important ingredient for the specular reflection BC is the following crucial estimate for the derivatives
of specular cycles $[X_{\mathbf{cl}}(s;t,x,v)$,$V_{\mathbf{cl}}(s;t,x,v)].$

 \begin{theorem}\label{theorem_Dxv}
There exists $C=C(\Omega)>0$ such that for all $(s;t,x,v)\in \mathbb{R}\times \mathbb{R}\times \bar{\Omega}\times \mathbb{R}^{3}$ with $s\neq t^{\ell}$ for $\ell = 1,2,\cdots, \ell_{*}$ 
\begin{equation}\label{lemma_Dxv}
\begin{split}
|\partial _{x}X_{\mathbf{cl}}(s;t,x,v)| & \ \lesssim \ e^{C|v|(t-s)}\frac{|v|}{%
\sqrt{\alpha (x,v)}} , \\
|\partial _{v}X_{\mathbf{cl}}(s;t,x,v)| &  \ \lesssim  \ e^{C|v|(t-s)}\frac{1}{|v|}%
, \\
|\partial _{x}V_{\mathbf{cl}}(s;t,x,v)| & \ \lesssim \ e^{C|v|(t-s)}\frac{|v|^{3}%
}{\alpha (x,v)} , \\
|\partial _{v}V_{\mathbf{cl}}(s;t,x,v)| &  \ \lesssim \ e^{C|v|(t-s)}\frac{|v|}{%
\sqrt{\alpha (x,v)}} .
\end{split}
\end{equation}
\end{theorem}

Our estimates are optimal in terms of the order of $\frac{1}{\alpha },$ and $%
e^{C|v|(t-s)}$ relates to the $|v|$ growth in the Velocity lemma (Lemma \ref{velocity_lemma}). We remark
that these precise orders of singularity, play a critical role for our
design of the anisotropic norms in Theorem \ref{main_specular}. In fact, if $\,|\partial _{x}X_{\mathbf{cl}%
}(s;t,x,v)|\backsim \frac{1}{\alpha},$ it would have been too singular for
the half power gain of $\alpha$ from the dynamical non-local to local
estimates (Lemma \ref{lemma_nonlocal}), and our method should fail. Moreover, it is also crucial to have
precise $|v|$ growth in both $|\partial _{x}X_{\mathbf{cl}%
}(s;t,x,v)|$ and $|\partial _{x}V_{\mathbf{cl}}(s;t,x,v)|$ to be controlled by $%
e^{-\varpi \langle v\rangle t}$.

We remark that $|\partial _{x}X_{\mathbf{cl}}(s;t,x,v)|\backsim \frac{1}{%
\sqrt{\alpha }}$ is unexpected, even after one bounce we would have $\partial _{x} x^{1}\backsim \frac{1}{\sqrt{\alpha }}$ and it is natural to
expect $\partial _{x}X_{\mathbf{cl}}(s;t,x,v)$ picks up additional power
of $\frac{1}{\sqrt{\alpha }}$ in the accumulation of $\frac{1}{\sqrt{\alpha }%
}~$\ number of bounces. However, via direct computations in 2D disk, we
discover that even though 
\begin{equation*}
\partial _{x}t^{\ell} \backsim \frac{1}{\alpha },\text{ and }\partial
_{x}x^{\ell} \backsim \frac{1}{\alpha },
\end{equation*}%
but surprisingly
\begin{equation*}
\partial _{x}X_{\mathbf{cl}}(s;t,x,v)=\partial
_{x}[x ^{\ell}-(t^{\ell}-s)v^{\ell}]\backsim \frac{1}{\sqrt{\alpha }} \ !
\end{equation*}%
Clearly, certain cancellations take place in the disk, which is difficult to
even expect for general domains.

The proof of our theorem is split into 10 steps, and it is the most
delicate proof throughout this paper. We first remark that, due to the
`discontinuous behaviors' of the normal component of $v\cdot n$ at each
specular reflection, it is impossible to apply the standard techniques for
ODE to estimate $|\partial X_{\mathbf{cl}}(s;t,x,v)|$ and $|\partial V_{%
\mathbf{cl}}(s;t,x,v)|.$ We have to develop different strategies to overcome
several analytical difficulties to finally complete the proof.

\vspace{4pt}

\textit{Topological obstruction and moving frames. } It turns out that we
only need to consider the most delicate case in which all the bounces are
almost grazing and staying near the boundary for $\mathbf{r}^{\ell}=\frac{|v^{\ell}\cdot n|%
}{|v^{\ell}|}\ll1.$ It is important for us to introduce the spherical
co-ordinate system to cover the whole cycle and transform it into the ODE (\ref{ODE_ell}). Unfortunately, due to the `hair-ball' theorem in Topology, such a change of
coordinate system (or any change of coordinates) can not be smooth
everywhere in the 2D surface $\partial \Omega .$  In the case of a ball, all
the trajectories are confined in a plane, so that one may choose a single
chart to cover the whole trajectories. However, in other convex domains
except the ball case, with large $t$, the specular trajectories are extremely
complicated, which can reach almost every point on $\partial\Omega.$ Hence, choosing a single
chart is all but impossible. On the other hand, a `sudden' change of a chart
may create new order of singularity of $\alpha $ from the matrix $\mathcal{P}
$ as in (\ref{diagonal_matrix}), which will ruin the
estimates. It is therefore important to design a `continuous' changes of
charts associated with the almost grazing bounces. Given $n(x),$ we need to
construct another globally defined, orthogonal, and continuous vector field. This
would have been impossible if we were to seek it only in the physical space, in light of
the `hair-ball' theorem. The key observation is that, we need continuity not
from just $\partial \Omega ,$ but from the phase space $\partial \Omega
\times \mathbb{R}^{3}.$ In fact, for almost grazing bounces, the velocity
field $v$ is almost perpendicular to $n(x),$ which provides a natural
choice for construction of the desired moving frames. These continuous moving frames
cost manageable errors for each bounce, which are controlled by the next
method.

\vspace{4pt}

\textit{Matrix Method for normal parts of }$\partial X_{\mathbf{cl}}(s)$ and 
$\partial V_{\mathbf{cl}}(s)$\textit{. }With such a well-defined moving
charts, via the chain rule, one can represent $\partial X_{\mathbf{cl}%
}(s;t,x,v)$ and $\partial V_{\mathbf{cl}}(s;t,x,v)$ via a multiplication of
Jacobian matrices $(t^{\ell},x^{\ell},v^{\ell})\rightarrow (t^{\ell-1},x^{\ell-1},v^{\ell-1})$
in the spherical coordinate system. The `matrix method' refers to the study of
each discrete Jacobian matrix and precise estimates of their multiplication ($%
\frac{1}{\sqrt{\alpha }}$ of them!). One important step is to bound such a
matrix by $J(\mathbf{r}^{\ell})$ in (\ref{l_to_l+1}) which can be diagonalized as $%
J(\mathbf{r}^{\ell})=\mathcal{P}^{-1}\Lambda \mathcal{P}$, with a diagonal matrix $%
\Lambda $.  Based on the crucial cancellation property (\textbf{\ref%
{time_gap}}), we can extract a crucial second order of $\mathbf{r}^{\ell} \ll 1$ appeared in $%
J(\mathbf{r}^{\ell}).$ Therefore,  over the interval $t|v|\backsim 1,$ we are able to
estimate $\Pi _{\ell=1}^{\frac{1}{\sqrt{\alpha }}}J( \mathbf{r}^{\ell})\sim \frac{1}{%
\sqrt{\alpha }}$. Together with $\frac{1}{\sqrt{\alpha }}$ from the initial
bounce, we expect $\frac{1}{\alpha }-$singularity for both $\partial X_{\mathbf{cl}%
}(s;t,x,v)$ and $\partial V_{\mathbf{cl}}(s;t,x,v)\,$\ as in (\ref{middle}).
Even though such estimate is too singular for our purpose we can improve it. Upon a closed
inspection,  
\begin{equation*} 
\begin{split} 
|\partial _{x}\mathbf{X}_{\perp }(s;t,x,v)|\lesssim \frac{1}{\sqrt{\alpha }},
\\
\end{split}
\end{equation*}
for the normal component of $X_{\mathbf{cl}}(s).$ This is based on the fact $v_{\perp }^{\ell}\sim \sqrt{\alpha }$ via the
Velocity lemma (Lemma 1, \cite{Guo10}). Unfortunately, the tangential part $\partial _{x}\mathbf{X}%
_{||}(s;t,x,v)\backsim \frac{1}{\alpha }$ is still too singular.

\vspace{4pt}

\textit{ODE Method for tangential parts of }$\partial X_{\mathbf{cl}}(s)$
and $\partial V_{\mathbf{cl}}(s)$\textit{.} To improve such an estimate, we
observe that given the estimates for the normal parts [$\mathbf{X}_{\perp
}(s;t,x,v),$ $\mathbf{V}_{\perp }(s;t,x,v)],$ the sub-system of ODE for $[%
\mathbf{X}_{||}(s;t,x,v),\mathbf{V}_{||}(s;t,x,v)],$ enjoys much better
property. In fact, at each specular reflection, $[\mathbf{X}_{||}(s;t,x,v),%
\mathbf{V}_{||}(s;t,x,v)]$ are continuous, unlike the the normal velocity $\mathbf{V}%
_{\perp }(s;t,x,v).$ Upon integrating over time as $\mathbf{V}_{\perp
}(s;t,x,v)=\mathbf{\dot{X}}_{\perp }(s;t,x,v)$ (position $\mathbf{X}_{\perp
}(s;t,x,v)$ is still continuous at specular reflection)$,$ we are able to
derive an integral equations of $[\mathbf{X}_{||}(s;t,x,v),\mathbf{V}%
_{||}(s;t,x,v)]$ without broken into small discontinuous pieces (\ref{Dode}) at each specular reflection$.$ In other
words, we can use the standard ODE theory to estimate these tangential
parts. Our ODE method refers such ODE (Gronwall) estimates (\ref%
{ODE_onegroup}) which lead to the final conclusion\ of the theorem.

With such crucial estimates, we are able to design anisotropic norms in terms
of singularity of $\frac{1}{\alpha }$.
Thanks to $\int_{0}^{t} \int_{u} \frac{e^{-C_{\theta} |v-u|^{2}}}{|v-u|^{2-\kappa}} \frac{1}{\alpha(X(s),u) ^{\beta }} \lesssim \alpha
^{ -\beta +  {1}/{2} }$ and $\int_{0}^{t} \int_{u} \frac{e^{-C_{\theta} |v-u|^{2}}}{|v-u|^{2-\kappa}} \frac{1}{\alpha(X(s),u) ^{\beta -1/2}} \lesssim \alpha
^{ -\beta +  1 }$ from the dynamical non-local to local estimates for $%
\beta >1,$ we have exact cancellations of the power of $\alpha $ in the
coefficients on the right hand side, and we are able to close the estimates. For $|v|$
either small or large, more careful analysis is needed. In particular,  it
is important to use the weight function of $e^{-\varpi\langle v\rangle t}$ in (\ref{dist}) to
control both the growth in Theorem \ref{theorem_Dxv} as well as $|v|$ in front of $\partial
_{x}X_{\mathbf{cl}}$ and $\partial _{v}V_{\mathbf{cl}}$ to control singularity of $|v|$ in (\ref{lemma_Dxv}).

\vspace{15pt}
 
\noindent{\textbf{\large{ 4. Bounce-back Reflection BC}}} 
 
 \vspace{3pt }

We recall the bounce-back reflection boundary condition (\ref{bbBC}) and the bounce-back cycles in Definition \ref{cycles}. Our main theorem is
 
\begin{theorem}\label{main_bb} 
Assume $f_{0} \in W^{1,\infty}(\Omega\times \mathbb{R}^{3})$ and $0< \kappa \leq 1$ for $0 <\theta < \frac{1}{4} ,$
\[
|| \langle v\rangle \partial_{x} f_{0} ||_{\infty} + || \partial_{v} f_{0} ||_{\infty} + ||  e^{\theta|v|^{2}} \partial_{t}f_{0}||_{\infty} +  ||  e^{\theta|v|^{2}}  f_{0}||_{\infty}<+\infty,
\]
and the compatibility conditions 
\begin{equation}\label{compatibility_bb}
f_{0}(x,v) = f_{0}(x,-v), \ \ v\cdot \nabla_{x} f_{0}  (x,v) = -v\cdot \nabla_{x} f_{0}(x,-v)  \ \ \text{on} \  \gamma_{-}  .
\end{equation}
Then there is $T=T(||   e^{\theta|v|^{2}} f_{0}||_{\infty})>0$ so that for all $0 \leq t\leq T$
 \begin{equation}\label{bb_main_estimate}
 \begin{split}
& || e^{-\varpi \langle v\rangle t} \frac{\alpha}{\langle v\rangle^{2}}\partial_{x}f(t)   ||_{\infty}+|| e^{-\varpi \langle v\rangle t} \frac{|v| \alpha^{1/2} }{\langle v\rangle^{2}} \partial_{v} f(t)||_{\infty} + || e^{ \theta|v|^{2}} \partial_{t}f(t)||_{\infty }   \\
\lesssim_{\xi,t}& \  || \langle v\rangle \partial_{x} f_{0} ||_{\infty} + || \partial_{v} f_{0}||_{\infty}+ 
P(  ||  e^{\theta|v|^{2}} \partial_{t} f_{0} ||_{\infty}) + P(||    e^{\theta |v|^{2}} f_{0} ||_{\infty}),
\end{split}
\end{equation}
for some polynomial $P$.

Moreover, if $f_{0} \in C^{1}$ then $f\in C^{1}$ away from the grazing set $\gamma_{0}.$ Furthermore, if $F_{0} = \mu + \sqrt{\mu} g_{0} \geq 0$ with $|| e^{\theta |v|^{2}} g_{0} ||_{\infty} \ll 1$ then this theorem holds for the arbitrarily large time $t \geq 0.$

\end{theorem}

There can be no size restriction on initial data $F_{0} = \sqrt{\mu} f_{0}$. We remark that the bounce-back case enjoys explicit
expressions of $\partial X_{\mathbf{cl}}(s;t,x,v)$ and $\partial V_{%
\mathbf{cl}}(s;t,x,v).$ Since $%
\partial _{x}t^{\ell}\sim \frac{1}{\alpha }$ and $\partial
_{x}x^{\ell}\sim \frac{1}{\sqrt{\alpha }}$, a new difficulty arises in the estimate 
\begin{equation*}
\partial _{x}X_{\mathbf{cl}}(s;t,x,v)\sim \frac{1}{\alpha },
\end{equation*}%
which is too singular to control by our non-local to local estimates (Lemma \ref{lemma_nonlocal}). Roughly speaking, the new difficulty is
exactly the opposite to the specular case : $\partial x^{\ell}$ and $\partial
v^{\ell}$ are in desired form but not $\partial _{x}X_{\mathbf{cl}}(s;t,x,v)!$
The crucial observation is the following:

\begin{lemma}\label{change_time_bb}
In the sense of distribution,
\begin{equation}\notag
\begin{split}
&\partial_{\mathbf{e}} \Big[   \int^{t^{j}}_{     t^{j+1}  } f (\tau, x^{j}-(t^{j}- {\tau})  v^{j}, v^{j}) \mathrm{d} {\tau}  \Big]\\
=& \int^{t^{j}}_{   t^{j+1} }  \big[ \partial_{\mathbf{e}} t^{j}, \partial_{\mathbf{e}} x^{j} + \tau \partial_{\mathbf{e}} v^{j}, \partial_{\mathbf{e}}v^{j}  \big] \cdot \nabla_{t,x,v} f  (\tau, x^{j} -(t^{j} - \tau) v^{j}, v^{j}) \mathrm{d}\tau\\
&+    \lim_{\tau \downarrow t^{j+1}} \big[  \partial_{\mathbf{e}} t^{j} -  \partial_{\mathbf{e}} t^{j+1} ] f  ( \tau, x^{j}-(t^{j}-\tau) v^{j}, v^{j}) \big]  .
\end{split}
\end{equation}
\end{lemma}
The key idea is to make a change of variable to transform
\begin{equation*}
\begin{split}
\\
\partial _{x}X_{\mathbf{cl}}(s;t,x,v)\backsim v^{\ell} \partial _{x}t^{\ell}+\partial
_{x}x^{\ell},
\\
\end{split}
\end{equation*}%
while $\partial _{x}t^{\ell}$ captures the worst singulairty of $\frac{1}{%
\alpha }.$ Fortunately, $\partial _{x}t^{\ell}$ is paired with $\partial
_{t}f,$ which is bounded, from the time-invariance of the problem and we
are able to close the estimate.

\vspace{15pt}
 
\noindent{\textbf{\large{ 5. Non-existence of $\protect\nabla ^{2}f$ up to the boundary}}} 
 
 \vspace{3pt } 

 In the appendix, we demonstrate that, our estimates can not be valid for
higher order derivatives. Otherwise, if $\partial ^{2}f$ exists up to the
boundary, we observe that from taking second derivatives of the Boltzmann
equation: 
\begin{equation*}
v_{n}\partial _{n}^{2}f=-\partial _{tn}f-(\partial _{n}v_{n})\partial
_{n}f-\sum_{i=1}^{2}\partial _{n}(v_{\tau _{i}})\partial _{\tau
_{i}}f-\sum_{i=1}^{2}v_{\tau _{i}\partial _{n\tau _{i}}}f-\nu (F)\partial
_{n}f+\partial _{n}K(f)+\partial _{n}\Gamma _{\text{gain}}(f,f).
\end{equation*}%
If $|\partial _{n}f|\geq \frac{1}{\sqrt{\alpha }}$ and $\partial_{n}K( f) \sim K(\partial _{n}f)$ then at the boundary we
have 
\begin{equation*}
|\partial _{n}f|\geq \frac{1}{|v_{n}|}\notin L_{loc}^{1}(\mathbb{R}^{1}),
\end{equation*}%
so that $\partial_{n}K( f)$ is not defined. Since $|\partial _{n}f|$ is
expected to behave at least as bad as $\frac{1}{\sqrt{\alpha }}$ for all
diffusive, specular and bounce-back cases, we are able to identify initial
conditions such that $|\partial _{n}f|\geq \frac{1}{|v_{n}|}$ for some
future time. 


 

\section{\large{Preliminary}} 
   
 \vspace{4pt}
 Recall $\Gamma_{\mathrm{gain}}$ and $\nu$ in    
Due to the Grad estimate \cite{gl}    
\begin{equation}\label{grad}
\begin{split}
&\Gamma_{\mathrm{gain}}( \sqrt{\mu}, g) + \Gamma_{\mathrm{gain}}(g, \sqrt{\mu} )  = \int_{\mathbb{R}^{3}} \mathbf{k}_{2}(v,u) g(u) \mathrm{d}u,\\&\nu(\sqrt{\mu}g)  = \int_{\mathbb{R}^{3}} \mathbf{k}_{1}(v,u) g(u) \mathrm{d}u ,
\end{split}
\end{equation}
where
 \begin{equation}\label{k_estimate}
\begin{split}
 \mathbf{k}_{1}(u,v)=  & \ |u-v|^{\kappa }   e^{- \frac{|v|^{2} +|u|^{2}}{2} }\int_{
\mathbb{S}^{2}}q_{0}(\frac{v-u}{|v-u|}\cdot \omega )\mathrm{d}\omega , 
\\
\mathbf{k}_{2}(u,v) =& \ \frac{2}{|u-v|^{2}}e^{-\frac{1}{8}|u-v|^{2}-\frac{1}{8}
\frac{(|u|^{2}-|v|^{2})^{2}}{|u-v|^{2}}} \\
&\times \int_{w\cdot (u-v)=0}q_{0}\Big( \frac{u-v}{\sqrt{ |u-v|^{2 } + |w|^{2}}}  \cdot \frac{u-v}{|u-v|}  \Big) e^{-|w+\varsigma |^{2}}(|w|^{2}+|u-v|^{2})^{ \frac{\kappa}{2} }\mathrm{d}
w,
\end{split}
\end{equation}
where $\varsigma:= \big( \frac{v+u}{2} \cdot \frac{w}{|w|}\big) \frac{w}{|w|}.$ See page 315 of \cite{Guo03} for details.

\begin{lemma}
\label{lemma_K} For $0\leq \kappa \leq 1,$
\begin{equation}\notag
\begin{split}
|\mathbf{k}_{1}(u,v)| + |\mathbf{k}_{2}(u,v)| &\ \lesssim \ \{|v-u|^{\kappa }+|v-u|^{-2+\kappa }\}e^{-\frac{1}{8}%
|v-u|^{2}-\frac{1}{8}\frac{(|v|^{2}-|u|^{2})^{2}}{|v-u|^{2}}}  \\
& \ \lesssim \ \frac{e^{-\frac{1}{10}%
|v-u|^{2}-\frac{1}{10}\frac{(|v|^{2}-|u|^{2})^{2}}{|v-u|^{2}}} }{|v-u|^{2-\kappa }} .
\label{eq:grad_estimate}
\end{split}
\end{equation}%
For $\varrho>0$ and $-2\varrho<  \theta <2\varrho$ and $\zeta\in\mathbb{R}$, we have for $0<\kappa \leq 1$,
\begin{equation}\notag
\int_{\mathbb{R}^{3}}
 \{|v-u|^{\kappa }+|v-u|^{-2+\kappa }\}e^{-\varrho%
|v-u|^{2}-\varrho\frac{(|v|^{2}-|u|^{2})^{2}}{|v-u|^{2}}}
  \frac{\langle
v\rangle ^{\zeta }e^{\theta |v|^{2}}}{\langle u\rangle ^{\zeta }e^{\theta
|u|^{2}}}\mathrm{d}u \ \lesssim \ \langle v\rangle ^{-1}.  \label{int_k}
\end{equation}%
\end{lemma}

\begin{proof}
The proof is based on \cite{Guo10}. Note that
\[
 \frac{\langle v\rangle^\zeta e^{\theta|v|^2}}{\langle u \rangle^\zeta e^{\theta|u|^2}}
 \ \lesssim \  [1+|v-u|^2]^{\frac{\zeta}{2}} e^{-\theta(|u|^2 -|v|^2)}.
\]
Set $v-u = \eta$ and $u =v-\eta$ in the integration. Now we compute the total exponent of the integrand as
\begin{equation*}
\begin{split}
&-\varrho|\eta|^2 -\varrho \frac{||\eta|^2-2 v\cdot \eta|^2}{|\eta|^2} -\theta\{|v-\eta|^2 -|v|^2\}
=  - 2\varrho|\eta|^2 + 4\varrho \{v\cdot \eta\} - 4\varrho \frac{|v\cdot \eta|^2}{|\eta|^2} -\theta\{|\eta|^2 -2v\cdot \eta\}\\
& =\left(-\theta-2\varrho\right) |\eta|^2 + \left( 4\varrho+ 2\theta\right)v\cdot \eta -4\varrho\frac{|v\cdot \eta|^2}{|\eta|^2}.
\end{split}
\end{equation*}
Since $-2\varrho <\theta< 2\varrho,$ the discriminant of the above quadratic form of $|\eta|$ and $\frac{v\cdot \eta}{|\eta|}$ is negative: $\left(4\varrho+ 2\theta\right)^2 +16 \varrho(-\theta-2\varrho)   = 4\theta^2 -16 \varrho^2<0$. We thus have
\begin{equation*}
-\varrho|\eta|^2 -\varrho \frac{||\eta|^2-2 v\cdot \eta|^2}{|\eta|^2} -\theta\{|v-\eta|^2 -|v|^2\} \
\lesssim_{\varrho, \theta} \ - \Big\{\frac{|\eta|^2}{2}+|v\cdot \eta|\Big\} .
\end{equation*}
Hence, for $0\leq \kappa \leq 1$ the integration is bounded by
\begin{equation*}
\begin{split}
\int_{\mathbb{R}^3} \Big\{|\eta|^\kappa + |\eta|^{-2+\kappa}\Big\}\langle\eta\rangle  ^{\zeta } e^{- {C_{\varrho,\theta}}   |\eta|^2 }  \ \lesssim_{\varrho, \theta, \kappa} \ 1.
\end{split}
\end{equation*}%
Therefore in order to prove Lemma \ref{lemma_K} it suffices to consider the case $|v|\geq 1$. We make another change of variables $\eta_\parallel = \left\{\eta\cdot \frac{v}{|v|}\right\}\frac{v}{|v|}$ and $\eta_\perp = \eta-\eta_\parallel$, so that $|v\cdot\eta|=|v||\eta_\parallel|$ and $|v-u|\geq |\eta_\perp|.$ We can absorb $\langle \eta \rangle^\zeta, \ |\eta|\langle \eta \rangle^\zeta$ by $e^{-C_{\varrho,\theta}|\eta|^2}$, and bound the integration by, for $0 < \kappa \leq 1,$ \begin{equation*}
\begin{split}
& \ \ \int_{\mathbb{R}^3} \big\{ 1+|\eta|^{-2+ \kappa}\big\} e^{-C_{\varrho,\theta}\left\{\frac{|\eta|^2}{2}+|v\cdot \eta|\right\}} \mathrm{d} \eta \ \leq \ \int_{\mathbb{R}^3}  \big\{ 1+|\eta|^{-2+ \kappa }\big\} e^{-\frac{C_{\varrho,\theta}}{2}|\eta|^2}     e^{-C_{\varrho,\theta}|v\cdot \eta|} \mathrm{d} \eta\\
&\leq \int_{\mathbb{R}^2} \{1+|\eta_\perp|^{-2+\kappa  }\} e^{-\frac{C_{\varrho,\theta}}{2}|\eta_{\perp}|^2} \left\{\int_{\mathbb{R}}  e^{-C_{\varrho,\theta} |v|\times |\eta_\parallel|} \mathrm{d} |\eta_\parallel|\right\} \mathrm{d} \eta_\perp\\
& \lesssim \langle v\rangle^{-1 } \int_{\mathbb{R}^2} \{1+|\eta_\perp|^{-2+\kappa  }\} e^{-\frac{C_{\varrho,\theta}}{2}|\eta_{\perp}|^2} \left\{ \int_{0}^\infty  e^{-C_{\varrho,\theta}  y } \mathrm{d} y\right\} \mathrm{d} \eta_\perp, \ \ (y=|v||\eta_\parallel|),\\
& \lesssim \langle v\rangle^{-1 } .
\end{split}
\end{equation*}
\end{proof} We define
\begin{equation}\label{v_gamma}
\Gamma_{\mathrm{gain},v}(g_{1},g_{2})(v):=\int_{\mathbb{R}^{3}}\int_{\mathbb{S}^{2}} B(v-u,\omega) \nabla_v(\sqrt{\mu})(u) g_{1}(u^{\prime})g_{2}(v^{\prime})  \mathrm{d} \omega \mathrm{d} u,
\end{equation}
where $u^{\prime} =u-[(u-v)\cdot \omega]\omega$ and $v^{\prime} = v+ [(u-v)\cdot \omega]\omega$.

\begin{lemma}
\label{lemma_operator}
(i) For $0 < \theta < \frac{1}{4}$ and $0 \leq \kappa\leq 1$, there exists $C_{\theta}>0$ such that, for $(i,j)=(1,2)$ or $(2,1),$
\begin{equation}\label{Gamma_k}
\left\vert \Gamma_{\mathrm{gain}} (g_{1},g_{2})(v)\right\vert   \ \lesssim_{\theta} \ ||   e^{\theta
|v|^{2}}g_{i}||_{\infty } \int_{%
\mathbb{R}^{3}} \frac{e^{-C_{\theta} |v-u|^{2} }}{|v-u|^{2-\kappa}}|g_{j}(u)|\mathrm{d}u,
\end{equation}
and
\begin{equation}\notag
\begin{split}
\big\vert
\Gamma_{\mathrm{gain}}  (g_1,g_2)(v)\big\vert  & \ \lesssim_{\theta} \ \langle v\rangle^{\kappa} e^{-
 {\theta}   |v|^{2}}||  e^{ {\theta} |v|^2}g_1 ||_\infty ||  e^{ {\theta} |v|^2}g_2 ||_\infty,\\
  \left\vert
\Gamma_{\mathrm{gain},v} (g_1,g_2)(v)\right\vert & \ \lesssim_{\theta} \  \langle v\rangle^{\kappa}e^{-%
 \theta |v|^{2}}||  e^{\theta|v|^2}g_1||_\infty|| e^{\theta|v|^2}g_2||_\infty,\\
 |\nu(  \sqrt{\mu} g_{1}) g_{2}(v)|    & \ \lesssim_{ \theta} \    
 ||   e^{\theta|v|^{2}}g_{2}||_{\infty} \int_{\mathbb{R}^{3}}  \frac{e^{-C_{\theta} |v-u|^{2} }}{|v-u|^{2-\kappa}} |g_{1}(u)|\mathrm{d}u   .
 \end{split}
\end{equation}

\noindent(ii) For $p\in[1,\infty)$ and $0 < \theta < \frac{1}{4},$ and for $(i,j)=(1,2)$ or $(i,j)=(2,1),$
\begin{equation}\notag
\begin{split}
\Vert  \Gamma_{\mathrm{gain}}(g_{1} ,g_{2} )   \Vert _{p}      \ 
&
\lesssim_{\theta,p} \ || e^{\theta|v|^{2}}g_{i}  ||_{\infty} \Vert g_{j} \Vert _{p},\\ 
 \Vert  \nu(\sqrt{\mu} g_{1} )g_{2}     \Vert _{p} \ 
&
\lesssim_{\theta,p} \ || e^{\theta|v|^{2}}g_{2}  ||_{\infty} \Vert g_{1} \Vert _{p},\\ 
\big|\iint_{\Omega \times \mathbb{R}^{3}}\Gamma_{\mathrm{gain}} (g_{1},g_{2})g_{3}%
\mathrm{d}v\mathrm{d}x\big| \ &
 \lesssim_{\theta,p} \
|| e^{\theta|v|^{2}} g_{i}||_{\infty} \Vert  g_{j}\Vert _{p}\Vert  g_{3}\Vert _{q}, \\
\big|\iint_{\Omega \times \mathbb{R}^{3}}  \nu(\sqrt{\mu} g_{1})  g_{2} g_{3}%
\mathrm{d}v\mathrm{d}x\big| \ &
 \lesssim_{\theta,p} \
|| e^{\theta|v|^{2}} g_{2}||_{\infty} \Vert  g_{1}\Vert _{p}\Vert  g_{3}\Vert _{q}.
\end{split}
\end{equation}

\noindent(iii) For $p\in[1,\infty)$ and $0 < \theta < \frac{1}{4},$ and for $(i,j)=(1,2)$ or $(i,j)=(2,1),$
\begin{equation}\notag
\begin{split}
  \Vert \nabla _{v}[\Gamma_{\mathrm{gain}}(g_{1},g_{2}) ] \Vert
_{p}  
\ &\lesssim_{\theta,p} \  \sum_{(i,j)} || e^{\theta|v|^{2}}g_{i}  ||_{\infty}   \Vert \nabla _{v}g_{j}\Vert _{p} ,\\
 \Vert  \nu(\sqrt{\mu} \nabla_{v} g_{1} )g_{2}    \Vert _{p}
\ &\lesssim_{\theta,p} \   || e^{\theta|v|^{2}}g_{2}  ||_{\infty}   \Vert \nabla _{v}g_{1}\Vert _{p},\\
  \big|\iint_{\Omega\times\mathbb{R}^{3}}\nabla_v\Gamma_{\mathrm{gain}}(g_{1},g_{2})g_{3}
\mathrm{d} v \mathrm{d} x \big| \ &    \lesssim_{\theta,p} \   \sum_{(i,j)} 
|| e^{\theta|v|^{2}}g_{i} ||_{\infty} || \nabla_{v}g_{j} ||_{p} || g_{3}||_{q},\\
  \big|\iint_{\Omega\times\mathbb{R}^{3}}  \nu(\sqrt{\mu} \nabla_{v}g_{1}) g_{2}g_{3}
\mathrm{d} v \mathrm{d} x \big| \ &    \lesssim_{\theta,p} \   
|| e^{\theta|v|^{2}}g_{2} ||_{\infty} || \nabla_{v}g_{1} ||_{p} || g_{3}||_{q}.
\end{split}
\end{equation}

\noindent(iv) Let $[Y,W]=[Y(x,v), W(x,v)]\in \Omega\times \mathbb{R}^{3}$. For $0<\theta < \frac{1}{4}$ and $\partial_{\mathbf{e}}$ with $\mathbf{e} \in \{x,v\},$
%
\begin{equation}\notag
\begin{split} 
    &  |\partial_{\mathbf{e}} \Gamma_{\mathrm{gain}} (g,g)(Y , W ) | \\
\lesssim & \  |\partial_{\mathbf{e}} Y|||  e^{\theta|v|^{2}} g||_{\infty} \int_{\mathbb{R}^{3}} \frac{e^{- {C_{\theta}|u-W|^{2}} }}{|u-W|^{2-\kappa}} |\nabla_{x}g(Y,u)| \mathrm{d}u  \\
&  + \ 
 |\partial_{\mathbf{e}} W|||  e^{\theta|v|^{2}} g||_{\infty} \int_{\mathbb{R}^{3}}  \frac{e^{- {C_{\theta}|u-W|^{2}} }}{|u-W|^{2-\kappa}} |\nabla_{v}g(Y,u)| \mathrm{d}u + \langle v\rangle^{\kappa} e^{-  {\theta} |v|^{2}} |\partial_{\mathbf{e}} W|  ||  e^{\theta|v|^{2}}  g||_{\infty}^{2}. 
 \end{split}
\end{equation}

%
%
\end{lemma}

\begin{proof}
\textit{(i)} First we show (\ref{Gamma_k}) for $(i,j)=(1,2)$. Clearly
\[
|\Gamma_{\mathrm{gain}} (g_{1}, g_{2})| \ \lesssim \ |\Gamma_{\mathrm{gain}} (e^{-\theta |v|^{2}} , |g_{2}|)| \times || e^{\theta|v|^{2}} g_{1} ||_{\infty}.
\] 
Then we follow the Grad estimate, page 315 of \cite{Guo03}, to bound $|\Gamma_{\mathrm{gain}} (e^{-\theta |v|^{2}} , |g_{2}|)|$ by $\int_{\mathbb{R}^{2}}\mathbf{k}_{2}(v,u)|g_{2}(u)|\mathrm{d}u$ with different exponent of $\mathbf{k}_{2}(v,u).$ Then we use Lemma \ref{lemma_K} to conclude (\ref{Gamma_k}).

For the second estimate we use (\ref{gain_Gamma})
\begin{equation}\notag
\begin{split}
|\Gamma_{\mathrm{gain}}(g_{1}, g_{2})(v)|& \lesssim \Gamma_{\mathrm{gain}}(e^{-\theta|v|^{2}}, e^{-\theta|v|^{2}}) \times || e^{\theta|v|^{2}} g_{1} ||_{\infty}|| e^{\theta|v|^{2}} g_{2} ||_{\infty} \\
&= e^{-\theta|v|^{2}}\iint B(v-u,\omega) \sqrt{\mu(u)}  e^{-\theta|u|^{2}} \mathrm{d}\omega \mathrm{d}u\times || e^{\theta|v|^{2}} g_{1} ||_{\infty}|| e^{\theta|v|^{2}} g_{2} ||_{\infty}\\
& \lesssim \langle v\rangle^{\kappa} e^{-\theta|v|^{2}}|| e^{\theta|v|^{2}} g_{1} ||_{\infty}|| e^{\theta|v|^{2}} g_{2} ||_{\infty},
\end{split}
\end{equation}
 where we have used $|u^{\prime}|^{2} + |v^{\prime}|^{2} =|u |^{2} + |v |^{2}.$ The third estimate follows similarly with $\nabla_{u}(\sqrt{\mu})(u) \lesssim \mu(u)^{1/2-\delta}$ for any $\delta>0.$ The forth estimate follows from
 \[
 e^{-\theta|v|^{2}}  \nu(\sqrt{\mu} g_{1})(v) \lesssim    \int_{\mathbb{R}^{3}} |v-u|^{\kappa}  e^{-\theta|v|^{2}} \sqrt{\mu(u)} |g_{1}(u)|  \mathrm{d}u
 \lesssim  \int_{\mathbb{R}^{3}}   \frac{e^{-C_{\theta}|v-u|^{2}  }}{|v-u|^{2-\kappa}}  |g_{1}(u)|  \mathrm{d}u
 ,
 \]
 and $e^{-\theta|v|^{2}}|v-u|^{\kappa} \sqrt{\mu(u)}\lesssim  \frac{e^{-C_{\theta}|v-u|^{2}  }}{|v-u|^{2-\kappa}}.$

\vspace{4pt}

\noindent\textit{(ii)} First two estimates are a direct consequence of \textit{(i)}:
\begin{equation}\notag
\begin{split}
|| \Gamma_{\mathrm{gain}} (g_{1}, g_{2}) ||_{p} & \lesssim || e^{\theta |v|^{2}} g_{i} ||_{\infty} \Big|\Big| \Big( \int_{u } \frac{e^{-C_{\theta} |v-u|^{2}  }}{|v-u|^{2-\kappa}}\Big)^{1/q}   \Big(  \int_{u}  \frac{e^{-C_{\theta} |v-u|^{2}  }}{|v-u|^{2-\kappa}} | g_{j} (u)|^{p}   \Big)^{1/p}   \Big|\Big|_{L^{p}_{v}}\\
&\lesssim  || e^{\theta |v|^{2}} g_{i} ||_{\infty}  \Big( \int_{u } \frac{e^{-C_{\theta} | u|^{2}  }}{| u|^{2-\kappa}}\Big)^{1/q}  
\Big( \int_{u}  |g_{j}(u)|^{p}\int_{v}  \frac{e^{-C_{\theta} | v-u|^{2}  }}{| v-u|^{2-\kappa}} \Big)^{1/p}  \\
& \lesssim  || e^{\theta |v|^{2}} g_{i} ||_{\infty}
 \Big( \int_{u } \frac{e^{-C_{\theta} | u|^{2}  }}{| u|^{2-\kappa}}\Big) || g_{j}||_{p}\\ & \lesssim || e^{\theta |v|^{2}} g_{i} ||_{\infty}
   || g_{j}||_{p}.
\end{split}
\end{equation}
Using the forth estimate of \textit{(i)}, the same proof holds for $|| \nu(\sqrt{\mu} g_{1}) g_{2} ||_{p} \lesssim_{\theta,p} || e^{\theta|v|^{2}} g_{2}||_{\infty} || g_{1} ||_{p}.$ 

For the third estimate we use (\ref{Gamma_k}) to bound as
\begin{equation}\notag
\begin{split}
  & || e^{\theta |v|^{2}} g_{i} ||_{\infty} \iiint_{\Omega \times \mathbb{R}^{3} \times \mathbb{R}^{3}}  \frac{e^{-C_{\theta} |v-u|^{2}}}{|v-u|^{2-\kappa}} | g_{j} (x,u)| |g_{3}(x,v)| \mathrm{d}u \mathrm{d}v  \mathrm{d}x \\
 \lesssim & \ \Big( \iiint  \frac{e^{-C_{\theta} |v-u|^{2}}}{|v-u|^{2-\kappa}} | g_{j} (x,u)|^{p}\Big)^{1/p}  \Big( \iiint  \frac{e^{-C_{\theta} |v-u|^{2}}}{|v-u|^{2-\kappa}} | g_{3} (x,u)|^{q}\Big)^{1/q} \\
 \lesssim & \ || g_{j} ||_{p} || g_{3} ||_{q}.
\end{split}
\end{equation}
The same proof holds with exchanging $i$ and $j$. Using the forth estimate of \textit{(i)}, the same proof holds for the forth estimate.

\vspace{4pt}

\noindent\textit{(iii)} We compute the velocity derivative of $\Gamma_{\mathrm{gain}}$ after the change of variable $u:=v-u$:
\begin{equation}\notag\label{Gammav}
\begin{split}
 \nabla_v \Gamma_{\mathrm{gain}}(g_{1},g_{2})
 & \ = \ 
 \nabla_{v} \Big[\int_{\mathbb{R}^{3}}  \int_{\mathbb{S}^{2}} 
 B(u,\omega) \sqrt{\mu(u)}  g_{1} (u-(u\cdot \omega) \omega) g_{2} (v+ (u\cdot \omega) \omega)
 \mathrm{d}\omega\mathrm{d}u\Big]
 \\
& \ = \ \Gamma_{\mathrm{gain}}(g_1,\nabla_v g_2) + \Gamma_{\mathrm{gain}}(\nabla_v g_1, g_2) + \Gamma_{\mathrm{gain},v} (g_1,g_2).\end{split}
\end{equation} 
The two first terms are estimated directly by \textit{(ii)}. For the $\Gamma_{\mathrm{gain},v}$ we use the fact $|\nabla_v(\sqrt{\mu})(v-u)|\le C {\mu(v-u)}^{1/4}$ and then apply \textit{(ii)}. The other estimates are direct consequence of the previous estimates.

\textit{(iv)} It suffices to show the following computation: For $0<\theta< \frac{1}{4},$
\begin{equation}\label{nabla_Gamma}
\begin{split}
&|\partial_{\mathbf{e}} \Gamma_{\mathrm{gain}} (g,g)(Y , W ) |\\
&= \ \ \Big| \partial_{\mathbf{e}}  \int_{\mathbb{S}^{2}}\int_{\mathbb{R}^{3}} |u|^{\kappa} q_{0} \Big(\frac{u}{|u|} \cdot \omega  \Big) e^{-\frac{|u+W|^{2}}{4}}  g(Y,W+[ u \cdot \omega] \omega) g(Y, W + u -(u\cdot \omega) \omega)
\mathrm{d}\omega \mathrm{d}u\Big|\\
 &= \ \ |\Gamma_{\mathrm{gain}}( \partial_{\mathbf{e}}  Y\cdot \nabla_{x}g,g)(Y,W)| + |\Gamma_{\mathrm{gain}}(g, \partial_{\mathbf{e}}  Y\cdot \nabla_{x}g )(Y,W)| \\
 & \ \  + |\Gamma_{\mathrm{gain}}( \partial_{\mathbf{e}}  W\cdot \nabla_{v}g,g)(Y,W) |+ |\Gamma_{\mathrm{gain}}(g, \partial_{\mathbf{e}}  W\cdot \nabla_{v}g )(Y,W)|\\
 & \ \ +\Big|\int_{\mathbb{S}^{2}}\int_{\mathbb{R}^{3}} |u|^{\kappa} q_{0} \Big(\frac{u}{|u|} \cdot \omega  \Big)
 (\frac{-1}{2}) (u+W)\cdot  \partial_{\mathbf{e}} W
  \sqrt{\mu(u+W)}\\
  & \ \ \ \ \ \ \ \ \    \ \ \ \ \ \ \ \ \ \ \   \ \ \ \ \ \ \ \ \ \ \ \ \ \ \ \ \ \ \ \ \ \  \ \times  g(Y,W+[ u \cdot \omega] \omega) g(Y, W + u -(u\cdot \omega) \omega)
\mathrm{d}\omega \mathrm{d}u\Big|\\
&  \lesssim | \partial_{\mathbf{e}}  Y|||  e^{\theta|v|^{2}} g||_{\infty} \int_{\mathbb{R}^{3}}  \frac{e^{-C_{\theta}|u-W|^{2}  }}{|u-W|^{2-\kappa}} |\partial_{x}g(Y,u)| \mathrm{d}u \\
& \ \ +
 | \partial_{\mathbf{e}}  W|||  e^{\theta|v|^{2}} g||_{\infty} \int_{\mathbb{R}^{3}}  \frac{e^{-C_{\theta}|u-W|^{2}  }}{|u-W|^{2-\kappa}} |\partial_{v}g(Y,u)| \mathrm{d}u + |\partial_{\mathbf{e}} W| \langle v\rangle^{\kappa}e^{-\theta|v|^{2}} ||   e^{\theta|v|^{2}} g||_{\infty}^{2},
\end{split}
\end{equation}
 where we have used the change of variables $u-V\mapsto u$.

\end{proof}

\begin{lemma}[Local Existence]\label{local_existence}
For $0\leq \theta< 1/4 $, if $||  e^{\theta|v|^{2}} f_{0 } ||_{\infty} < + \infty$ then there exists $T>0$ depending on $||  e^{\theta|v|^{2}} f_{0 } ||_{\infty}$ such that there exists unique $F=\mu + \sqrt{\mu}f$ solves the Boltzmann equation (\ref{Boltzmann_F}) in $[0,T]$ and satisfies the initial condition and boundary conditions (\ref{diffuseBC}), (\ref{specularBC}), (\ref{bbBC}) respectively and $F(t,x,v) \geq 0$ on $[0,T]\times \bar{\Omega} \times\mathbb{R}^{3}$ and $f$ satisfies, for some $0<\theta^{\prime} < \theta,$
\begin{equation}\label{bounded}
\sup_{0 \leq t\leq T}||   e^{\theta^\prime|v|^{2}}f(t) ||_{\infty} \lesssim P(||   e^{\theta|v|^{2}} f_{0 } ||_{\infty}),
\end{equation}
for some polynomial $P$. Moreover if $f_{0}$ is continuous and satisfies the compatibility conditions (\ref{compatibility_condition_1}), (\ref{compatibility_specular}), (\ref{compatibility_bb}) respectively then $f$ is continuous away from the grazing set $\gamma_{0}$.  

Moreover, for $0\leq \bar{\theta}  < \frac{1}{4}$, if $||   e^{ \bar{\theta}  |v|^{2}} \partial_{t}f_{0 } ||_{\infty} \equiv \big|\big|  e^{\bar{\theta} |v|^{2}} \frac{-v\cdot \nabla_{x}F_{0} + Q(F_{0},F_{0})  }{\sqrt{\mu}} \big|\big|_{\infty} < + \infty$ and compatibility conditions (\ref{compatibility_condition_1}), (\ref{compatibility_specular}), (\ref{compatibility_bb}) respectively, then 
\begin{equation}\label{bounded_t}
\sup_{0 \leq t\leq T^{*}}|| e^{\bar{\theta} |v|^{2}} \partial_{t}f(t) ||_{\infty} \lesssim P( ||   e^{ \bar{\theta} |v|^{2}} \partial_{t}f_{0 } ||_{\infty}) +  {P}( ||   e^{\theta|v|^{2}} f_{0 } ||_{\infty}),
\end{equation}
for some polynomial $P$. 

Furthermore for the diffuse and bounce-back boundary conditions if $F_{0} = \mu + \sqrt{\mu} g_{0}$ with $||  e^{\theta|v|^{2}} g_{0 } ||_{\infty}\ll 1$ for $0 < \theta < \frac{1}{4}$ then the results hold for all $t\geq0$. For the specular reflection boundary condition, if $\xi$ is real analytic ($\xi$ is real analytic), and if $||  e^{\theta|v|^{2}} g_{0 } ||_{\infty}\ll 1$ for $0 < \theta < \frac{1}{4}$ then the results hold for all $t\geq0$. . 
\end{lemma}
\begin{proof}
 We use the positive preserving iteration of \cite{Guo10,Kim11}
\begin{equation}\label{positive_iteration}
\partial_{t} F^{m+1}  + v\cdot \nabla_{x} F^{m+1} + \nu(F^{m}) F^{m+1} = Q_{\mathrm{gain}}(F^{m},F^{m}), \ \ \ F^{m+1}|_{t=0}=F_{0}\geq 0,
\end{equation}
which is equivalent to, with $F^{m}:= \sqrt{\mu}f^{m},$ 
\begin{equation}\label{fm}
\partial_{t} f^{m+1} + v\cdot \nabla_{x} f^{m+1} + \nu(F^{m}) f^{m+1}  = \Gamma_{\mathrm{gain}}(f^{m},f^{m}), \ \ \ f^{m+1}|_{t=0}=f_{0}.
\end{equation}
The starting of this iteration is $F^{0}\equiv F_{0}\geq 0, \ f^{0} \equiv f_{0}$ and let $F^{-m}\equiv F^{0} , \ f^{-m}\equiv f^{0}$ for all $m\in\mathbb{N}$. 

Along the trajectory we have the Duhamel formula (ignoring the boundary condition):
\begin{equation}\notag
\begin{split}
f^{m+1}(t,x,v)& = 
e^{- \int^{t}_{0}\nu(\sqrt{\mu} f^{m})(s,  X_{\mathbf{cl}}(s), V_{\mathbf{cl}}(s)) \mathrm{d}s  } f_{0} ( X_{\mathbf{cl}}(0),  V_{\mathbf{cl}}(0)) \\
& \ + \int^{t}_{0} e^{- \int^{t}_{s}\nu(\sqrt{\mu} f^{m})(\tau,  X_{\mathbf{cl}}(\tau),  V_{\mathbf{cl}}(\tau))}  \Gamma_{\mathrm{gain}}(f^{m},f^{m})(s,  X_{\mathbf{cl}}(s),  V_{\mathbf{cl}}(s))  \mathrm{d}s.
\end{split}
\end{equation}
The local existence theorem without boundary is standard:
\begin{equation}\notag
\begin{split}
&|e^{(\theta-t) |v|^{2}}  f^{m+1}(t,x,v) |\\
 \lesssim & \ |e^{(\theta-t) |v|^{2}}  f_{0} |  +  \int^{t}_{0}  | \Gamma_{\mathrm{gain}} (f^{m},f^{m})(s,  X_{\mathbf{cl}}(s),  V_{\mathbf{cl}}(s)) | \mathrm{d}s\\
 \lesssim& \ ||e^{ \theta  |v|^{2}}  f_{0} ||_{\infty} + e^{(\theta-t) |v|^{2}}\int^{t}_{0} \iint_{\mathbb{R}^{3} \times \mathbb{S}^{2}}
B(V_{\mathbf{cl}}(s)-u,\omega) \sqrt{\mu(u)} |f^{m}(s, X_{\mathbf{cl}}(s),u^{\prime})|  |f^{m}(s, X_{\mathbf{cl}}(s),v^{\prime})| \\
\lesssim & \ ||e^{ \theta  |v|^{2}}  f_{0} ||_{\infty} + \big(\sup_{0 \leq s\leq t} || e^{(\theta-s)|v|^{2}} f^{m}(s)||_{\infty}\big)^{2}
\int_{0}^{t}\iint B(v-u,\omega) \sqrt{\mu(u)} e^{(\theta-t) |v|^{2}} e^{-(\theta-s) |u^{\prime}|^{2}} e^{-(\theta-s) |v^{\prime}|^{2}}\\
\lesssim & \ ||e^{ \theta  |v|^{2}}  f_{0} ||_{\infty} + \big(\sup_{0 \leq s\leq t} || e^{(\theta-s)|v|^{2}} f^{m}(s)||_{\infty}\big)^{2}
\int_{0}^{t} e^{-(t-s)|v|^{2}} \int_{u} |v-u|^{\kappa} \sqrt{\mu(u)} \\
\lesssim & \ ||e^{ \theta  |v|^{2}}  f_{0} ||_{\infty} + \big(\sup_{0 \leq s\leq t} || e^{(\theta-s)|v|^{2}} f^{m}(s)||_{\infty}\big)^{2}
\int_{0}^{t} e^{-(t-s)|v|^{2}} \langle v\rangle \{ \mathbf{1}_{|v|>N} + \mathbf{1}_{|v|\leq N}\} \mathrm{d}s\\
\lesssim & \ ||e^{ \theta  |v|^{2}}  f_{0} ||_{\infty} + \big(\sup_{0 \leq s\leq t} || e^{(\theta-s)|v|^{2}} f^{m}(s)||_{\infty}\big)^{2}
\int_{0}^{t} e^{-(t-s)|v|^{2}} \langle v\rangle \{ \mathbf{1}_{|v|>N} + \mathbf{1}_{|v|\leq N}\} \mathrm{d}s\\
\lesssim & \ ||e^{ \theta  |v|^{2}}  f_{0} ||_{\infty} + \big(\sup_{0 \leq s\leq t} || e^{(\theta-s)|v|^{2}} f^{m}(s)||_{\infty}\big)^{2}
\Big\{ \frac{1}{N^{2}} + Nt\Big\}.
\end{split}
\end{equation}

Now we choose sufficiently large $N\gg 1$ and then small $0< T\ll \theta$ to obtain the uniform-in-$m$ estimate
\begin{equation}\label{local_estimate}
\sup_{0 \leq t \leq T}|| e^{\theta^{\prime}|v|^{2} } f^{m+1}(t)||_{\infty} \lesssim ||e^{ \theta  |v|^{2}}  f_{0} ||_{\infty}, 
\end{equation}
for some $0< \theta^{\prime} < \theta.$

With the boundary condition the Duhamel form is evolved with boundary conditions accordingly:

\textit{(i)} Diffuse reflection boundary condition, on $(x,v)\in\gamma_{-},$
\begin{equation} \label{diffuse_BC_m}
\begin{split}
f^{m +1}(t,x,v)  \ = \ c_{\mu} \sqrt{\mu(v)} \int_{n(x)\cdot u>0} f^{m}(t,x,u) \sqrt{\mu(u)} \{n(x)\cdot u\} \mathrm{d}u. 
\end{split}
\end{equation}

\textit{(ii)} Specular reflection boundary condition, on $(x,v)\in\gamma_{-},$
\begin{equation}\label{specular_BC_m}
f^{m+1}(t,x,v)  \ = \ f^{m}(t,x,R_{x} v),
\end{equation}
where $R_{x}v= v-2n(x) (n(x)\cdot v).$ 

\textit{(iii)} Bounce-back reflection boundary condition, on $(x,v)\in\gamma_{-},$ 
\begin{equation}\label{bb_BC_m}
f^{m+1}(t,x,v)   \ = \ f^{m}(t,x,-v).
\end{equation}
We follow the proof of \cite{Guo10,Kim11} to have same estimate of (\ref{local_estimate}).
\end{proof}

 \vspace{8pt}

\section{\large{Traces and the In-flow Problems}}

\vspace{4pt}


Recall the almost grazing set $\gamma _{+}^{\varepsilon }$ defined in (\ref%
{def:gamma_epsilon}). We first estimate the outgoing trace on $\gamma
_{+}\setminus \gamma _{+}^{\varepsilon }$. We remark that for the outgoing part, our estimate is global in time without cut-off, in contrast to the general trace
theorem.

\begin{lemma}
\label{le:ukai} Assume that $\varphi =\varphi (v)$ is $L_{loc}^{\infty }(%
\mathbb{R}^{3})$. For any small parameter $\varepsilon >0$, there exists a
constant $C_{\varepsilon ,T,\Omega }>0$ such that for any $h$ in $%
L^{1}([0,T],L^{1}(\Omega \times \mathbb{R}^{3}))$ with $\partial _{t}h+v\cdot
\nabla _{x}h+\varphi h$ is in $L^{1}([0,T],L^{1}(\Omega \times \mathbb{R}%
^{3}))$, we have for all $0\leq t\leq T,$
\begin{equation*}
\int_{0}^{t}\int_{\gamma _{+}\setminus \gamma _{+}^{\varepsilon }}|h|\mathrm{%
d}\gamma \mathrm{d}s\leq C_{\varepsilon ,T,\Omega }\left[ \
||h_{0}||_{1}+\int_{0}^{t}\big\{\Vert h(s)\Vert _{1}+\big{\Vert} \lbrack \partial
_{t}+v\cdot \nabla _{x}+\varphi ]h(s)\big{\Vert} _{1}\big\}\mathrm{d}s\ \right] .
\end{equation*}%
Furthermore, for any $(s,x,v)$ in $[0,T]\times \Omega \times \mathbb{R}^{3}$
the function $h(s+s^{\prime },x+s^{\prime }v,v)$ is absolutely continuous in
$s^{\prime }$ in the interval $[-\min \{t_{\mathbf{b}}(x,v),s\},\min \{t_{%
\mathbf{b}}(x,-v),T-s\}]$.
\end{lemma}

\begin{proof}
With a proper change of variables (e.g. Page 247 in \cite{CIP}) we have
\begin{equation}\label{integration}
\begin{split}
&  \int_0^T \iint_{\Omega\times\mathbb{R}^3} h(t,x,v\mathrm{d}) \mathrm{d} v\mathrm{d} x \mathrm{d} t\\
& = \int^0_{-\min\{T,t_{\mathbf{b}}(x,v)\}} \iint_{\Omega\times\mathbb{R}^3} h(T+s,x+sv,v) \mathrm{d} v\mathrm{d} x\mathrm{d} s \\
& \  + \int_0^{\min\{T,t_{\mathbf{b}}(x,-v)\}} \iint_{\Omega\times\mathbb{R}^3} h(0+s,x+sv,v) \mathrm{d} v \mathrm{d} x \mathrm{d} s \\
&  \ + \int_0^T \int_{\gamma_+} \int^0_{-\min\{t,t_{\mathbf{b}}(x,v)\}}
h(t+s,x+sv,v) \mathrm{d} s \mathrm{d} \gamma
\mathrm{d} t \\
& \ + \int_0^T \int_{\gamma_-} \int_0^{\min\{T-t,t_{\mathbf{b}}(x,-v)\}} h(t+s,x+sv,v) \mathrm{d} s \mathrm{d} \gamma\mathrm{d} t. 
\end{split}
\end{equation}
For $(t,x,v)\in [0,T]\times\gamma_+$ and $0 \leq s\leq \min \{t,t_{\mathbf{b}}(x,v)\},$
\begin{equation*}
\begin{split}
h(t,x,v)  = h(t-s,x-sv,v) e^{-\varphi(v)s}  + \int^0_{-s} e^{\varphi(v)\tau} [\partial_t  h+ v\cdot \nabla_x h + \varphi(v) h](t+\tau, x+\tau v,v) \mathrm{d} \tau.
\end{split}
\end{equation*}
Now for $(t,x,v) \in [\varepsilon_1,T]\times \gamma_+\setminus \gamma_+^\varepsilon,$
we integrate over $\int_{\varepsilon_1}^T \int_{\gamma_+\setminus\gamma_+^\varepsilon}
\int^0_{\min\{t,t_{\mathbf{b}}(x,v)\}}$ to get
\begin{equation}\notag
\begin{split}
&\min\{\varepsilon_1 , \varepsilon^3\} \times \int^T_{\varepsilon_1} \int_{\gamma_+\setminus\gamma^\varepsilon_+} | h(t,x,v) | \mathrm{d} \gamma \mathrm{d} t \\
& \lesssim  \ \min_{[\varepsilon_1,T]\times [\gamma_+\setminus\gamma_+^\varepsilon]}\{t,t_{\mathbf{b}}(x,v)\} \times \int^T_{\varepsilon_1} \int_{\gamma_+\setminus\gamma^\varepsilon_+} |h(t,x,v) |\mathrm{d} \gamma \mathrm{d} t\\
& \lesssim \int_0^T  \int_{\gamma_+} \int^0_{- \min \{t,t_{\mathbf{b}}(x,v)\}}|h(t+s,x+sv,v)|   \mathrm{d} s  \mathrm{d} \gamma \mathrm{d} t \\
&  \ \ + T \int_0^T \int_{\gamma_+} \int^0_{-\min\{t,t_{\mathbf{b}}(x,v)\}}  |\partial_t h + v\cdot \nabla_x h + \varphi h|(t+ \tau, x+\tau v,v)\mathrm{d} \tau \mathrm{d} \gamma\mathrm{d} t\\
 & \lesssim  \int_0^T ||h(t)||_1\mathrm{d} t + \int^T_0 || [\partial_t  + v\cdot \nabla_x   + \varphi]h(t)  ||_1\mathrm{d} t,
\end{split}
\end{equation}
where we have used the integration identity (\ref{integration}), and (40) of \cite{Guo10} to obtain
$t_{\mathbf{b}}(x,v) \geq C_{\Omega}  {|n(x)\cdot v|}/ {|v|^2} \geq C_{\Omega}  \varepsilon^3$ for $(x,v) \in \gamma_+\setminus \gamma_+^\varepsilon$. Now we choose $\varepsilon_1 = \varepsilon_1(\Omega,\varepsilon)$ as
\[
\varepsilon_1 \leq C_{\Omega} \varepsilon^3 \leq \inf_{(x,v)\in \gamma_+\setminus \gamma_+^\varepsilon} t_{\mathbf{b}}(x,v).
\]
We only need to show, for $\varepsilon_1 \leq  C_{\Omega} \varepsilon^3,$
\begin{equation}\notag
\begin{split}
\int^{\varepsilon_1}_0 \int_{\gamma_+\setminus \gamma_+^\varepsilon} |h(t,x,v)| \mathrm{d} \gamma \mathrm{d} t \lesssim_{\Omega, \varepsilon,\varepsilon_1} || h_0||_1 + \int_0^{\varepsilon_1} ||[\partial_t + v\cdot \nabla_x +\varphi] h(t) ||_1\mathrm{d} t.
\end{split}
\end{equation}
Because of our choice $\varepsilon$ and $\varepsilon_1$,  $t_{\mathbf{b}}(x,v) >t$ for all $(t,x,v) \in [0,\varepsilon_1]\times \gamma_+\setminus\gamma_+^\varepsilon.$ Then
\[
|h(t,x,v)| \lesssim |h_0(x-tv,v)|+ \int_0^t \Big{|}[\partial_t + v\cdot \nabla_x + \varphi(v)]h(s,x-(t-s)v,v)\Big{|}\mathrm{d} s,
\]
where the second contribution is bounded, from (\ref{integration}), by
\begin{equation}\notag
\begin{split}
&\int^{\varepsilon_1}_0 \int_{\gamma_+\setminus \gamma_+^\varepsilon} \int_0^t \Big{|}[\partial_t + v\cdot \nabla_x + \varphi(v)]h(s,x-(t-s)v,v)\Big{|}\mathrm{d} s \mathrm{d} \gamma \mathrm{d} t \\
\lesssim& \int_0^{\varepsilon_1} || [\partial_t + v\cdot \nabla_x + \varphi(v)]h(t)||_1 \mathrm{d} t.
\end{split}
\end{equation}

Consider the initial datum contribution of $|h_0(x-tv,v)|$: We may assume $\partial_{x_3}\xi(x_0)\neq 0$. By the implicit function theorem $\partial\Omega$ can be represented locally by the graph $\eta=\eta(x_1,x_2)$ satisfying $\xi(x_1,x_2,\eta(x_1,x_2))=0$ and $(\partial_{x_1}\eta(x_1,x_2),\partial_{x_2}\eta(x_1,x_2))= (-\partial_{x_1}\xi / \partial_{x_3}\xi,-\partial_{x_2}\xi / \partial_{x_3}\xi )$ at $(x_1,x_2,\eta(x_1,x_2)).$ We define the change of variables
\begin{equation}\notag
\begin{split}
  (x,t)\in \partial\Omega \cap\{x\sim  x_0\} \times [0,\varepsilon_1] \mapsto y = x-tv \in \bar{\Omega},
\end{split}
\end{equation}
where $
\left|\frac{\partial y}{\partial(x,t)}\right|  = -v_1\frac{\partial_{x_1}\xi}{\partial_{x_3}\xi} - v_2 \frac{\partial_{x_2}\xi}{\partial_{x_3}\xi}- v_3.
$

Therefore
\begin{equation*}
\begin{split}
|n(x)\cdot v|\mathrm{d} S_x \mathrm{d} t &= (n(x)\cdot v) \Big[1+ \big(\frac{\partial_{x_1}\xi}{\partial_{x_3}\xi}\big)^2
+ \big(\frac{\partial_{x_2}\xi}{\partial_{x_3}\xi}\big)^2 \Big]^{1/2} \mathrm{d} x_1 \mathrm{d} x_2 \mathrm{d} t\\
& = \left[ -v_1\frac{\partial_{x_1}\xi}{\partial_{x_3}\xi} - v_2 \frac{\partial_{x_2}\xi}{\partial_{x_3}\xi}- v_3 \right]\mathrm{d} x_1 \mathrm{d} x_2 \mathrm{d} t = \mathrm{d} y,
\end{split}
\end{equation*}
and
$\int_0^{\varepsilon_1} \int_{\gamma_+\setminus \gamma_+^\varepsilon \cap \{x\sim x_0\}}
 |h_0(x-tv,v)| \mathrm{d} \gamma \mathrm{d} t \ \lesssim_{\varepsilon,\varepsilon_1, x_0}
\ \iint_{\Omega\times\mathbb{R}^3} |h_0(y,v)| \mathrm{d} y \mathrm{d} v.$ Since $\partial\Omega$
is compact we can choose a finite covers of $\partial\Omega$ and repeat the same argument for each piece to conclude \begin{equation}\notag
\int_0^{\varepsilon_1} \int_{\gamma_+\setminus \gamma_+^\varepsilon } |h_0(x-tv,v)| \mathrm{d} \gamma \mathrm{d} t \ \lesssim_{\Omega, \varepsilon,\varepsilon_1} \ \iint_{\Omega\times\mathbb{R}^3} |h_0(y,v)| \mathrm{d} y \mathrm{d} v.
\end{equation}\end{proof}

\begin{lemma}[Green's Identity]
\label{Green} For $p\in[1,\infty)$ assume that $f,\partial_t f + v\cdot
\nabla_x f \in L^p ([0,T];L^p(\Omega\times\mathbb{R}^3))$ and $%
f_{\gamma_-} \in L^p ([0,T];L^p(\gamma))$. Then $f  \in C^0( [0,T];
L^p(\Omega\times\mathbb{R}^3))$ and $f_{\gamma_+} \in
L^p ([0,T];L^p(\gamma))$ and for almost every $t\in [0,T]$ :
\begin{eqnarray*}
|| f(t) ||_{p}^p + \int_0^t |f |_{\gamma_+,p}^p = || f(0)||_p^p + \int_0^t
|f |_{\gamma_-,p}^p + \int_0^t \iint_{\Omega\times\mathbb{R}^3 }
\{\partial_t f + v\cdot \nabla_x f \} |f|^{p-2}f .
\end{eqnarray*}
\end{lemma}



See \cite{Guo10} for the proof. Now we state and prove following
propositions for the in-flow problems:
\begin{equation}
\{\partial _{t}+v\cdot \nabla _{x}+\nu \}f=H,\ \ \ f(0,x,v)=f_{0}(x,v),\
\ \ f(t,x,v)|_{\gamma _{-}}=g(t,x,v), \label{eq:transport}
\end{equation}%
where $\nu(t,x,v)\geq 0.$ For notational simplicity, we define

\begin{eqnarray}
\partial _{t}f_{0} &\equiv &-v\cdot \nabla _{x}f_{0}-\nu f_{0}+H(0,x,v),
\label{finitial} \\
\nabla _{x}g &\equiv &\frac{n}{n\cdot v}\Big\{-\partial
_{t}g-\sum_{i=1}^{2}(v\cdot \tau _{i})\partial _{\tau _{i}}g-\nu g+H\Big\}%
+\sum_{i=1}^{2}\tau _{i}\partial _{\tau _{i}}g. \label{gboundary}
\end{eqnarray}%
We remark that $\partial _{t}f_{0}$ is obtained from formally solving (\ref%
{eq:transport}), and $(\ref{gboundary})$ leads to the usual tangential
derivatives of $\partial _{\tau_i }g$, while defines new `normal derivative' $%
\partial _{n}g$ from the equation (\ref{eq:transport}).

\begin{proposition}
\label{theo:trace} Assume a compatibility condition
\begin{equation}
f_{0}(x,v)=g(0,x,v)\ \ \ \text{for}\ \ (x,v)\in \gamma _{-}.
\label{compatibility_inflow}
\end{equation}%
For any fixed $p\in[1,\infty)$, assume 
\begin{eqnarray*}
\nabla _{x}f_{0},\nabla _{v}f_{0},-v\cdot \nabla _{x}f_{0}-\nu
f_{0}+H(0,x,v) &\in &L^{p}(\Omega \times \mathbb{R}^{3}), \\
\langle v\rangle g,\partial _{t}g,\nabla _{v}g,\partial_{\tau _{i}}g,\frac{1}{n(x)\cdot v}%
\{-\partial _{t}g-\sum_{i}(v\cdot \tau _{i})\partial _{\tau _{i}}g-\nu
(v)g+H\} &\in &L^{p}([0,T]\times \gamma _{-}),
\end{eqnarray*}%
and, assume $1/p+1/q=1$ there exist $TC_{T}\sim O(T)$
and $\varepsilon \ll 1$ such that for all $t\in \lbrack 0,T]$
\begin{equation*}
\left\vert \iint_{\Omega \times \mathbb{R}^{3}}\partial H(t)h(t)\mathrm{d}x%
\mathrm{d}v\right\vert \leq C_{T} ||h(t)||_{q}, \ \ \ \nu \in L^{p},
\end{equation*}%
Then for sufficiently small $T>0$ there exists a unique solution $f$ to (\ref%
{eq:transport}) such that $f,\partial _{t}f,\nabla _{x}f,\nabla _{v}f\in
C^0([0,T];L^{p}(\Omega \times \mathbb{R}^{3}))$ and the traces
satisfy
\begin{equation}
\begin{split}
\partial _{t}f|_{\gamma _{-}}& =\partial _{t}g,\ \nabla _{v}f|_{\gamma
_{-}}=\nabla _{v}g,\ \nabla _{x}f|_{\gamma _{-}}=\nabla _{x}g, \ \ \text{on} \ \ \gamma_-,\\
\nabla _{x}f(0,x,v)& =\nabla _{x}f_{0},\ \nabla _{v}f(0,x,v)=\nabla
_{v}f_{0},\ \partial _{t}f(0,x,v)=\partial _{t}f_{0}, \ \  \text{in} \ \ \Omega\times\mathbb{R}^3,
\end{split}
\label{eq:nabla_bc}
\end{equation}%
where $\partial _{t}f_{0}$ and $\nabla _{x}g$ are given by (\ref{finitial})
and (\ref{gboundary}). Moreover
\begin{eqnarray}
 ||\partial _{t}f(t)||_{p}^{p}+\int_{0}^{t}|\partial _{t}f|_{\gamma
_{+},p}^{p}  &\leq&||\partial _{t}f_{0}||_{p}^{p}+\int_{0}^{t}|\partial _{t}g|_{\gamma
_{-},p}^{p}+p\int_{0}^{t}\iint_{\Omega \times \mathbb{R}^{3}} |\partial
_{t}H ||\partial _{t}f|^{p-2} ,  \label{global_t} \\
 ||\nabla _{x}f(t)||_{p}^{p}+\int_{0}^{t}|\nabla _{x}f|_{\gamma
_{+},p}^{p}&\leq&||\nabla _{x}f_{0}||_{p}^{p}+\int_{0}^{t}|\nabla _{x}g|_{\gamma
_{-},p}^{p}+p\int_{0}^{t}\iint_{\Omega \times \mathbb{R}^{3}} |\nabla
_{x}H ||\nabla _{x}f|^{p-1} ,    \label{global_x}   \\
 ||\nabla _{v}f(t)||_{p}^{p}+\int_{0}^{t}|\nabla _{v}f|_{\gamma
_{+},p}^{p}  &\leq&||\nabla _{v}f_{0}||_{p}^{p}+\int_{0}^{t}|\nabla _{v}g|_{\gamma
_{-},p}^{p}\notag\\
&&+p\int_{0}^{t}\iint_{\Omega \times \mathbb{R}^{3}}|\{\nabla
_{v}H-\nabla _{x}f-\nabla _{v}\nu f\}| |\nabla _{v}f|^{p-1} .  \label{global_v} 
\end{eqnarray}
\end{proposition}

\begin{proof}
We apply the trace theorem to the derivatives of $f$ by explicit computations. Denote $\nu(s)=\nu(s,x-(t-s)v,v).$
First we assume $f_0,g$
and $H$ have compact support in $v\in\mathbb{R}^3$. We integrate the equation (\ref{eq:transport}) along the backward trajectories. If the initial condition is reached before hitting the boundary (case $t< t_{\mathbf{b}}$), we have
\begin{equation*} f(t,x,v) = e^{- \int^{t}_{0}  \nu }f_{0}(x-tv,v) + \int_{0}^{t} e^{-  \int^{s}_{0}  \nu } H(t-s,x-vs,v) \mathrm{d} s.
\end{equation*}
If the boundary is first reached (case $t>t_{\mathbf{b}}$), we have
\begin{equation*} f(t,x,v) = e^{-    \int^{t_{\mathbf{b}}}_{0}  \nu }g(t-t_{\mathbf{b}},x_{\mathbf{b}},v) + \int_{0}^{t_{\mathbf{b}}} e^{-  \int^{s}_{0}  \nu } H(t-s,x-vs,v) \mathrm{d} s.
\end{equation*}
Let us rewrite it
\begin{equation}\label{eq:explicit_f} 
\begin{split} f(t,x,v) =& \mathbf{1}_{\{t\leq t_{\mathbf{b}}\}}e^{-  \int^{t}_{0}  \nu }f_{0}(x-tv,v) + \mathbf{1}_{\{t>t_{\mathbf{b}}\}}e^{-  \int^{t_{\mathbf{b}}}_{0}  \nu }g(t-t_{\mathbf{b}},x_{\mathbf{b}},v)\\
&+ \int_{0}^{\min(t,t_{\mathbf{b}})} e^{-  \int^{s}_{0}  \nu } H(t-s,x-vs,v) \mathrm{d} s. 
\end{split}
\end{equation}
We take derivative of $f$ with respect to time, space and velocity for $t\neq t_{\bf b}$. Recall the following derivatives of $x_{\mathbf{b}}$ and $t_{\mathbf{b}}$ (see lemma 2 in \cite{Guo10}) :
\begin{equation}\label{xb}
\begin{split}
&\nabla_{x} t_{\mathbf{b}}  =\frac{n(x_{\mathbf{b}})}{v\cdot n(x_{\mathbf{b}})},   \ \  \nabla_{v}
t_{\mathbf{b}} = -\frac{t_{\mathbf{b}} n(x_{\mathbf{b}})}{v\cdot n(x_{\mathbf{b}})}, \\
&\nabla_{x} x_{\mathbf{b}} =I-\frac{n(x_{\mathbf{b}})}{v\cdot n(x_{\mathbf{b}})} \otimes v,  \ \
  \nabla_{v} x_{\mathbf{b}}= - t_{\mathbf{b}} I + \frac{t_{\mathbf{b}} n(x_{\mathbf{b}})}{v\cdot n(x_{\mathbf{b}})} \otimes v.
\end{split}
\end{equation}
Since $g$ is defined on a surface, we cannot define its space gradient. We then
use directly the gradient in space of $g(x_\mathbf{b})$.
Regarding $g(t-t_{\mathbf{b}},x_\mathbf{b}(x,v),v)$ as function on $[0,T]\times \bar{\Omega}\times\mathbb{R}^3$
 we obtain from (\ref{xb})
\begin{equation*}
\begin{split}
\nabla_x[ g(t-t_{\mathbf{b}},x_{\mathbf{b}},v)] &= -\nabla_x t_{\mathbf{b}} \partial_t g + \nabla_x x_{\mathbf{b}} \nabla_\tau g
 = -\frac{n(x_{\mathbf{b}})}{v\cdot n(x_{\mathbf{b}})} \partial_t g + \left(I-\frac{n\otimes v}{n\cdot v}\right) \nabla_\tau g\\
 &= \tau_1 \partial_{\tau_1} g + \tau_2 \partial_{\tau_2}g- \frac{n(x_{\mathbf{b}})}{v\cdot n(x_{\mathbf{b}})} \left\{\partial_t g +
v\cdot \tau_1 \partial_{\tau_1}g + v\cdot \tau_2 \partial_{\tau_2}g\right\},\\
 \nabla_v [g(t-t_{\mathbf{b}},x_{\mathbf{b}},v)] &= -t_{\mathbf{b}} \nabla_x[g(t-t_{\mathbf{b}},x_{\mathbf{b}},v)]+ \nabla_v g,
\end{split}
\end{equation*}
where $\tau_1(x)$ and $\tau_2(x)$ are unit vectors satisfying $\tau_1(x) \cdot n(x) = 0=\tau_2(x) \cdot n(x)$ and $\tau_1(x) \times \tau_2(x) =n(x).$

Therefore by direct computation for $t\neq t_{\mathbf{b}}$, we deduce
\begin{eqnarray}
&&\partial _{t}f(t,x,v) \mathbf{1}_{\{t\neq t_{\mathbf{b}}\}} \label{compute_f_t}\\
 &&= -\mathbf{1}_{\{t< t_{\mathbf{b}}\}}e^{-  \int^{t}_{0}  \nu }[\nu f_{0} + v\cdot\nabla_{x}f_{0}-H_{|t=0}]( x-tv,v)
  +  \mathbf{1}_{\{t>t_{\mathbf{b}}\}}e^{-  \int^{t_{\mathbf{b}}}_{0}  \nu }\partial_{t}g(t-t_{\mathbf{b}},x_{\mathbf{b}},v)\nonumber\\
& & \ +  \int_{0}^{\min(t,t_{\mathbf{b}})} e^{-  \int^{s}_{0}  \nu }\partial_{t}H(t-s,x-vs,v) \mathrm{d} s ,\notag\\
&&\nabla _{x}f(t,x,v)\mathbf{1}_{\{t\neq t_{\mathbf{b}}\}} \label{compute_f_x} \\
 &&= \mathbf{1}_{\{t< t_{\mathbf{b}}\}}e^{-  \int^{t}_{0}  \nu }\nabla _{x}f_{0}( x-tv,v)\nonumber\\
&& \ + \mathbf{1}_{\{t>t_{\mathbf{b}}\}}e^{-  \int^{t_{\mathbf{b}}}_{0}  \nu } \Big\{ \sum_{i=1}^2
 \tau_i \partial_{\tau_i} g - \frac{n(x_{\mathbf{b}})}{v\cdot n(x_{\mathbf{b}})} \big\{\partial_t g +  \sum_{i=1}^2 (v\cdot \tau_i) \partial_{\tau_i}g +\nu g -H \big\}
\Big\}(t-t_{\mathbf{b}},x_\mathbf{b},v)\notag \\
&& \ + \int_{0}^{\min(t,t_{\mathbf{b}})} e^{- \int^{s}_{0}  \nu } \nabla_{x}H(t-s,x-vs,v) \mathrm{d} s ,\notag\\
&&\nabla _{v}f(t,x,v) \mathbf{1}_{\{t\neq t_{\mathbf{b}}\}}\label{compute_f_v}\\
 &&=  \mathbf{1}_{\{t< t_{\mathbf{b}}\}}e^{-  \int^{t}_{0}  \nu }[-t\nabla_{x}f_{0} + \nabla_{v}f_{0} -t \nabla_v\nu(v) f_{0} ]( x-tv,v)\nonumber\\
&& \ -
\mathbf{1}_{\{t>t_{\mathbf{b}}\}}t_{\mathbf{b}} e^{-  \int^{t_{\mathbf{b}}}_{0}  \nu } \Big\{ \sum_{i=1}^2
 \tau_i \partial_{\tau_i} g - \frac{n(x_{\mathbf{b}})}{v\cdot n(x_{\mathbf{b}})} \big\{\partial_t g +  \sum_{i=1}^2 (v\cdot \tau_i) \partial_{\tau_i}g +\nu g -H \big\}
\Big\}(t-t_{\mathbf{b}},x_\mathbf{b},v)
\notag\\
&& \ + \mathbf{1}_{\{t>t_{\mathbf{b}}\}}e^{-  \int^{t_{\mathbf{b}}}_{0}  \nu }\Big\{  \nabla_v g(t-t_{\mathbf{b}},x_{\mathbf{b}},v) -t_{\mathbf{b}} \nabla_v \nu(v) g(t-t_{\mathbf{b}},x_{\mathbf{b}},v)
\Big\}\notag\\
&& \ + \:\int_{0}^{\min(t,t_{\mathbf{b}})} e^{-  \int^{s}_{0}  \nu }\{\nabla_{v}H -s\nabla_{x}H - s\nabla \nu H \}(t-s,x-vs,v) \mathrm{d} s .\notag
\end{eqnarray}

First we show that $\partial f \mathbf{1}_{\{t>t_{\mathbf{b}}\}}\in L^p$ and $\partial f \mathbf{1}_{\{t<t_{\mathbf{b}}\}} \in L^{p}$ separately. Now we take $L^{p}$ norms above with the changes of variables in Lemma 2.1 of
 \cite{Guo95} and using Jensen's inequality in $[0,t]$. More precisely, for $\phi \in L^1$ with $\phi\geq0,$
\begin{equation}\label{change1}
\begin{split}
&\iint_{\Omega\times\mathbb{R}^3} \mathbf{1}_{\{x-tv \in\Omega\}} \phi(x-tv,v) \\
& = \int_{\mathbb{R}^3}\left[\int_{\Omega}\mathbf{1}_{\{x-tv\in\Omega\}} \phi(x-tv,v) \mathrm{d} x\right] \mathrm{d} v \leq \iint_{\Omega\times\mathbb{R}^3} \phi(x,v) ,\\
&\iint_{\{\Omega\times\mathbb{R}^3 \}\cap B((x_0,v_0);\delta)}\mathbf{1}_{\{t\geq t_{\mathbf{b}}\}} \phi(t-t_{\mathbf{b}}(x,v),x_{\mathbf{b}}(x,v),v)\\
& \leq \int_0^t \int_{\partial\Omega\times\mathbb{R}^3} \phi(s,x,v) |n(x)\cdot v| \mathrm{d} S_x \mathrm{d} v \mathrm{d} s,
\end{split}
\end{equation}
where for the second inequality we have used the change of variables for fixed $t,v,$
\begin{equation}
\label{change2}
 x  \mapsto (t-t_{\mathbf{b}}(x,v),x_{\mathbf{b}}(x,v) ).
\end{equation}
In fact, without the loss of generality we may
 assume $\partial_{x_3} \xi(x_{\mathbf{b}}(x,v))\neq 0$ for $(x,v) \in B((x_0,v_0);\delta)$
 so that $x_{\mathbf{b}}(x,v)= (x_{\mathbf{b},1},x_{\mathbf{b},2}, \eta(x_{\mathbf{b},1},x_{\mathbf{b},2}))$.
Using (\ref{xb}), we compute the Jacobian
\begin{eqnarray*}
\text{det}\left(\begin{array}{ccc} -\nabla_x t_{\mathbf{b}}  \\
-\nabla_x x_{\mathbf{b},1} \\
-\nabla_x x_{\mathbf{b},2}
 \end{array}\right)= \text{det} \left(\begin{array}{ccc} -(v\cdot n)^{-1}n   \\
-\nabla_x x_{\mathbf{b},1}  \\
-\nabla_x x_{\mathbf{b},2}  \\
 \end{array}\right) = \left|-v_1 \frac{\partial_{x_1}\xi}{\partial_{x_3}\xi} - v_2 \frac{\partial_{x_2}\xi}{\partial_{x_3}\xi } + v_3 \right|^{-1}.
\end{eqnarray*}
Therefore
$
\mathrm{d} x \mathrm{d} v = \left|-v_1 \frac{\partial_{x_1}\xi}{\partial_{x_3}\xi} - v_2 \frac{\partial_{x_2}\xi}{\partial_{x_3}\xi } + v_3 \right| \mathrm{d} x_1 \mathrm{d} x_2 \mathrm{d} v \mathrm{d} t = |n\cdot v| \mathrm{d} S_x \mathrm{d} v \mathrm{d} t= \mathrm{d}\gamma \mathrm{d}t.
$
Using these changes of variables,  we obtain
\begin{equation}\notag 
\|f(t) \mathbf{1}_{\{t\neq t_{\mathbf{b}}\}}\|_{p}  \leq  \  \|f_{0}\|_{p} + \left[\int_{0}^{t}\int_{\gamma_{-}} |g|^{p} \mathrm{d} \gamma \mathrm{d} s\right]^{1/p} + t^{(p-1)/p} \left[\int_{0}^{t} \|H\|_{p}^{p} \mathrm{d} s\right]^{1/p},\label{estimate_f}\\
\end{equation}
and
\begin{equation}\notag
\begin{split}
\|\partial_{t}f(t)\mathbf{1}_{\{t\neq t_{\mathbf{b}}\}}\|_{p}  \leq& \ \|v\cdot  \nabla_{x}f_{0} + \nu f_{0} - H(0,\cdot,\cdot)\|_{p}\\
 & +\left[\int_{0}^{t}\int_{\gamma_{-}} |\partial_{t}g|^{p} \mathrm{d} \gamma \mathrm{d} s\right]^{1/p} + t^{(p-1)/p} \left[\int_{0}^{t} \|\partial_{t}H\|_{p}^{p} \mathrm{d} s\right]^{1/p},
\end{split}
\end{equation}
and
\begin{equation}\notag
\begin{split}
&\|\nabla_{x}f(t)\mathbf{1}_{\{t\neq t_{\mathbf{b}}\}}\|_{p} \\
 \leq& \ \|\nabla_{x}f_{0}\|_{p}+ t^{(p-1)/p} \left[\int_{0}^{t} \|\nabla_{x}H(s) \|_{p}^{p} \mathrm{d} s\right]^{1/p}\\
   &+ \left[\int_{0}^{t}\int_{\gamma_{-}}  \Big|
\Big\{ \sum_{i=1}^2
 \tau_i \partial_{\tau_i} g - \frac{n(x_{\mathbf{b}})}{v\cdot n(x_{\mathbf{b}})} \big\{\partial_t g +  \sum_{i=1}^2 (v\cdot \tau_i) \partial_{\tau_i}g +\nu g -H \big\}
\Big\}(t-t_{\mathbf{b}},x_\mathbf{b},v)
  \Big|^p \mathrm{d} \gamma \mathrm{d} s\right]^{1/p} ,
\end{split}
\end{equation}
and
\begin{equation}\notag
\begin{split}
&\|\nabla_{v}f(t)\mathbf{1}_{\{t\neq t_{\mathbf{b}}\}}\|_{p} \\
\leq& \ t \|\nabla_{x}f_{0}\|_{p} + \|\nabla_{v}f_{0}\|_{p} + C \|f_{0}\|_{p} + Ct \left[\int_{0}^{t}\int_{\gamma_{-}} |g|^{p}\mathrm{d} \gamma \mathrm{d} s\right]^{1/p}\\
&+ t \left[\int_{0}^{t}\int_{\gamma_{-}}   \Big|
\Big\{ \sum_{i=1}^2
 \tau_i \partial_{\tau_i} g - \frac{n(x_{\mathbf{b}})}{v\cdot n(x_{\mathbf{b}})} \big\{\partial_t g +  \sum_{i=1}^2 (v\cdot \tau_i) \partial_{\tau_i}g +\nu g -H \big\}
\Big\}(t-t_{\mathbf{b}},x_\mathbf{b},v)
  \Big|^p  \mathrm{d} \gamma \mathrm{d} s\right]^{1/p} \\
&+ \left[\int_{0}^{t}\int_{\gamma_{-}} |\nabla_{v}g|^{p} \mathrm{d} \gamma \mathrm{d} s\right]^{1/p} + t\left[\int_{0}^{t}\int_{\gamma_{-}} |\langle v\rangle g|^{p} \mathrm{d} \gamma \mathrm{d} s\right]^{1/p}+ t^{( p-1)/p} \left[\int_{0}^{t} \|\nabla_{x}H\|_{p}^{p} \mathrm{d} s\right]^{1/p}\\
&+ t^{(p-1)/p}\left[\int_{0}^{t} \|\nabla_{v}H\|_{p}^{p} \mathrm{d} s\right]^{1/p} + C t^{(p-1)/p} \left[\int_{0}^{t} \|H\|_{p}^{p} \mathrm{d} s\right]^{1/p} .
\end{split}
\end{equation}

From our hypothesis and assumption on $f_0, g$ and $H$ to have compact supports, these terms are bounded, therefore
\begin{equation*}
\partial f \mathbf{1}_{\{t\neq t_{\mathbf{b}}\}} \equiv \big[\partial_t f \mathbf{1}_{\{t\neq t_{\mathbf{b}}\}},\nabla_x f\mathbf{1}_{\{t\neq t_{\mathbf{b}}\}},\nabla_v f\mathbf{1}_{\{t\neq t_{\mathbf{b}}\}}\big] \ \in \ L^{\infty}([0,T];
L^p(\Omega\times\mathbb{R}^3)).\label{piecewise_df}
\end{equation*}

On the other hand, thanks to the compatibility condition, we need to show $f$ has the same trace on the set
\begin{equation}
\mathcal{M}  \equiv \{t=t_{\mathbf{b}}(x,v)\} \equiv \{(t_{\mathbf{b}}(x,v),x,v) \in [0,T]\times {\Omega}\times\mathbb{R}^3\}.\label{M_sing}
\end{equation}
We claim the following fact: \textit{Let $\phi(t,x,v) \in C^\infty_c((0,T)\times\Omega\times\mathbb{R}^3)$ and we have
\begin{equation}\label{piecewise}
\int_0^T \iint_{\Omega\times\mathbb{R}^3} f \partial \phi = - \int_0^T \iint_{\Omega\times\mathbb{R}^3} \partial f \mathbf{1}_{\{t\neq t_{\mathbf{b}}\}} \phi, 
\end{equation}
so that $f \in W^{1,p}$ with weak derivatives given by $\partial f \mathbf{1}_{\{t\neq t_{\mathbf{b}}\}}.$}

\textit{Proof of claim.} We first fix the test function $\phi(t,x,v).$ There exists $\delta=\delta_\phi>0$ such that $\phi\equiv 0$ for $t\geq \frac{1}{\delta},$ or $\mathrm{dist}(x,\partial\Omega)< \delta,$ or $|v|\geq \frac{1}{\delta}.$ Let $\phi(t,x,v)\neq 0$ and $(t,x,v)\in\mathcal{M}$. By (\ref{M_sing}) and (\ref{exit}), $t=t_{\mathbf{b}}(x,v), x_{\mathbf{b}}=x-t_{\mathbf{b}} v,$ and $|x-x_{\mathbf{b}}|=t_{\mathbf{b}} |v|,$ and
\[
\mathrm{dist}(x,\Omega) \leq |x-x_{\mathbf{b}}| = t_{\mathbf{b}} |v|.
\]
Since $t_{\mathbf{b}} \leq \frac{1}{\delta}$, this implies that
\[
|v|\geq \frac{\delta}{t_{\mathbf{b}}} \geq \delta^2.
\]
Otherwise $\mathrm{dist}(x,\partial\Omega)\leq \delta$ so that $\phi(t,x,v)=0.$ Furthermore, by the Velocity lemma and this lower bound of $|v|$, we conclude that there exists $\delta^\prime(\delta,\Omega)>0$ such that
\begin{equation*}
\begin{split}
|v\cdot n(x_{\mathbf{b}})|^2 &\gtrsim_\Omega |v\cdot \nabla_x \xi(x_{\mathbf{b}})|^2 = \alpha(t-t_{\mathbf{b}};t,x,v)\\
& \geq e^{-C_\Omega \langle v\rangle t_{\mathbf{b}}} \alpha(t;t,x,v) \geq e^{-C_\Omega \langle v\rangle t_{\mathbf{b}}} C_\xi |v|^2 |\xi(x)|\\
& \geq e^{-{C_\Omega}{\delta^{-2}} } C_\xi \delta^4 \min_{\mathrm{dist}(x,\partial\Omega)\geq \delta} |\xi(x)|  \\
&= 2\delta^\prime(\delta,\Omega)>0.
\end{split}
\end{equation*}
In particular, this lower bound and a direct computation of (\ref{xb}) imply that $\{\phi\neq 0\}\cap \mathcal{M}$ is a smooth 6D hypersurface.

We next take $C^1$ approximation of $f^l_0, \ H^l,$ and $g^l$ (by partition of unity and localization) such that
\[
|| f^l_0 -f_0 ||_{W^{1,p}}\rightarrow 0, \ \ || g^l-g ||_{W^{1,p}([0,T]\times\gamma_-\backslash\gamma_-^{\delta^\prime})} \rightarrow 0 , \  \ || H^l -H ||_{W^{1,p}([0,T]\times\Omega\times\mathbb{R}^3)} \rightarrow 0.
\]
This implies, from the trace theorem, that
\[
f^l_0(x,v) \rightarrow f_0(x,v) \ \ \ \text{and} \ \ \ g^l(0,x,v) \rightarrow g(0,x,v) \ \ \ \text{in} \ \ L^1(\gamma_-\backslash \gamma_-^{\delta^\prime}).
\]
We define accordingly, for $(t,x,v)\in [0,T]\times\Omega\times\mathbb{R}^3,$
\begin{equation}\label{fl}
\begin{split}
f^l(t,x,v) &= \mathbf{1}_{\{t< t_{\mathbf{b}}\}} e^{-  \int^{t}_{0}  \nu } f_0^l(x-tv,v) + \mathbf{1}_{\{t>t_{\mathbf{b}}\}} e^{-  \int^{t_{\mathbf{b}}}_{0}  \nu } g^l(t-t_{\mathbf{b}},x_{\mathbf{b}},v)\\
& \ \  + \int^{\min\{t,t_{\mathbf{b}}\}}_0 e^{-  \int^{s}_{0}  \nu } H^l(t-s,x-sv,v) \mathrm{d}s,
\end{split}
\end{equation}
and $f^l_\pm(t,x,v)\equiv   \mathbf{1}_{\{t\gtrless t_{\mathbf{b}}\}}f^l$. Therefore for all $(x,v)\in\gamma_-$,
\begin{equation*}
\begin{split}
f^l_+(s,x+sv,v) -f^l_- (s,x+sv,v) &= e^{-\int^{s}_{0} \nu} g^l(0,x,v) - e^{ -\int^{s}_{0} \nu} f^l_0(x,v).
\end{split}
\end{equation*}
Since $\{\phi \neq 0\}\cap \mathcal{M}$ is a smooth hypersurface, we apply the Gauss theorem to $f^l$ to obtain
\begin{equation}\label{gauss}
\begin{split}
\iiint \partial_{\mathbf{e}} \phi f^l  \mathrm{d}x\mathrm{d}v\mathrm{d}t &= \iint [f^l_+ - f^l_-] \phi \ \mathbf{e}\cdot \mathbf{n}_{\mathcal{M}} \mathrm{d}\mathcal{M} \\ 
& \ \ - \left\{ \iiint_{t>t_{\mathbf{b}}} \phi \ \partial_{\mathbf{e}}f_{+}^l \mathrm{d}x\mathrm{d}v\mathrm{d}t
 +
 \iiint_{t<t_{\mathbf{b}}} \phi \ \partial_{\mathbf{e}}f_{-}^l \mathrm{d}x\mathrm{d}v\mathrm{d}t
\right\},
\end{split}
\end{equation}
where $\partial_\mathbf{e}= [\partial_t,\nabla_x,\nabla_v]=[\partial_t, \partial_{x_1}, \partial_{x_2}, \partial_{x_3}, \partial_{v_1},\partial_{v_2},\partial_{v_3}]$ and $$\mathbf{n}_\mathcal{M}=\frac{1}{\sqrt{1+ |\nabla_{x} t_{\mathbf{b}}|^{2} + |\nabla_{v} t_{\mathbf{b}}|   }} ( 1, - \nabla_{x} t_{\mathbf{b}}, -\nabla_{v} t_{\mathbf{b}} )\in\mathbb{R}^7.$$ We have used $(s,x+sv,v)$ and $(x,v)\in\gamma_-$ as our parametrization for the manifold $\mathcal{M}\cap \{\phi\neq0\}$, so that $n(x_{\mathbf{b}}(x,v))\cdot v\geq 2\delta^\prime$ is equivalent to $n(x)\cdot v\geq 2\delta^\prime$. Therefore the above hypersurface integration over $\{t\neq t_{\mathbf{b}}\}$ is bounded by
\begin{equation*}
\begin{split}
&\lesssim_{\phi, \delta} \int_0^{\frac{1}{\delta}} \int_{n(x)\cdot v\geq 2\delta^\prime} |f^l_+(s,x+sv,v)-f^l_-(s,x+sv,v)| \mathrm{d}S_x \mathrm{d}v\mathrm{d}s\\
&\lesssim_{\phi,\delta} \int_{n(x)\cdot v\geq 2\delta^\prime} | g^l(0,x,v)- f^l_0(s,v)| \mathrm{d}S_x \mathrm{d}v \rightarrow 0 , \ \ \ \text{as} \ l\rightarrow \infty,
\end{split}
\end{equation*}
since the compatibility condition $f_0(x,v)= g(0,x,v)$ for $(x,v)\in\gamma_-$. Clearly, taking difference of (\ref{fl}) and (\ref{eq:explicit_f}), we deduce $f^l \rightarrow f$ strongly in $L^p{(\{\phi\neq0\})}$ due to the first estimate of (\ref{estimate_f}). Furthermore, due to (\ref{estimate_f}), we have a uniform-in-$l$ bound of $f^l_{\pm}$ in $W^{1,p}(\{t\gtrless t_{\mathbf{b}}, \ \phi\neq0\})$ such that, up to subsequence,
$$
\partial_{\mathbf{e}} f^l_{+} \rightharpoonup\partial_{\mathbf{e}} f \mathbf{1}_{\{t>t_{\mathbf{b}}\}}, \ \ \partial_{\mathbf{e}} f^l_- \rightharpoonup \partial_{\mathbf{e}} f \mathbf{1}_{\{t<t_{\mathbf{b}}\}}, \  \ \ \text{weakly in} \ L^p(\{\phi\neq0\}).
$$
Finally we conclude the claim (\ref{piecewise}) by letting $l\rightarrow \infty$ in \eqref{gauss}.

Now notice that from its explicit form \eqref{eq:explicit_f}, and since all
the data are compactly supported in velocity, $f$ is itself
compactly supported in velocity. Recall $\partial=[\partial_t, \nabla_x,\nabla_v]$.
From this and the $L^p$ bounds above,
we conclude
\begin{equation}\label{partialf}
\{\partial_t+v\cdot\nabla_x + \nu \}\partial f=\partial H-\partial v\cdot \nabla_x f-\partial \nu f \in L^p.
\end{equation}
 By the trace theorem (Lemma \ref{le:ukai}), traces of $\partial_t f,\nabla_x f, \nabla_v f$ exist.
To evaluate these traces, we take derivatives along characteristics.
Letting $t\rightarrow t_{\mathbf{b}}$ and $t\rightarrow 0$, we deduce  (\ref{eq:nabla_bc}). From the Green's identity, Lemma \ref{Green}, we have (\ref{global_t}), (\ref{global_x}) and (\ref{global_v}), and therefore we conclude $\partial f \in C^0([0,T]; L^p ).$

In order to remove the compact support assumption we employ the cut-off function $\chi$ used in (\ref{dist}). Define $f^m=\chi(|v|/m) f$ then $f^m$ satisfies
\begin{eqnarray}
\{\partial_t + v\cdot\nabla_x +  \chi(|v|/m)  \nu \} f^m = \chi(|v|/m)H,\label{cutoff1}\\
f^m(0,x,v) = \chi(|v|/m) f_0, \ \ f^m|_{\gamma_-} = \chi(|v|/m) g\notag\label{cutoff2}.
\end{eqnarray}
Note that $\nabla_v [\chi(|v|/m)g] =\chi(|v|/m) \nabla_v g + g\nabla_v \chi(|v|/m) $ and $\chi(|v|/m) f_0(x,v) = \chi(|v|/m)g(0,x,v)$ for $(x,v)\in\gamma_-$.
Apply previous result to compute the traces of the derivatives of $f^m$. It is standard (using Green's identity) to show that $\partial_t f^m, \nabla_x f^m$ and $\nabla_v f^m$ are Cauchy and we can pass a limit.
\end{proof}

We now study weighted $W^{1,p}$ estimate. Recall (\ref{dist}). We first
define an effective collision frequency:
\begin{equation}
\begin{split}
\nu _{\varpi,\beta }(t,x,v)=  \nu(v)+  \varpi \langle v\rangle -\beta \alpha^{ -1}[v\cdot \nabla_{x} \alpha]
,
\end{split}
\label{nula}
\end{equation}%
and
\begin{equation}\label{df}
 [\partial_{t} + v\cdot \nabla_{x} + \nu_{\varpi, \beta}] (e^{-\varpi \langle v\rangle t} \alpha^{\beta} f) = e^{-\varpi \langle v\rangle t} \alpha^{\beta} [\partial_{t}f + v\cdot \nabla_{x}f + \nu f]  .
\end{equation}
Due to (\ref{alpha_inv}) and $\varpi \gg1$, $\nu _{\varpi,\beta}(t,x,v)\sim   \beta \langle v\rangle $.

\begin{proposition}
\label{inflowW1p} Let $f$ be a solution of (\ref{eq:transport}). Assume (\ref%
{compatibility_inflow}) and $\langle v\rangle g\in L^{\infty }([0,T]\times
\gamma _{-}),$ and $\nu, \langle v\rangle H\in L^{\infty }([0,T]\times \Omega \times
\mathbb{R}^{3})$. For any fixed $p\in[2,\infty]$, assume
\begin{eqnarray*}
 e^{-\varpi \langle v\rangle t} \alpha^{\beta}\partial _{t}g,\  e^{-\varpi \langle v\rangle t} \alpha^{\beta}\nabla _{\tau }g
&\in &L^{\infty }([0,T];L^{p}(\gamma _{-})), \\
 e^{-\varpi \langle v\rangle t} \alpha^{\beta}\big\{|\nabla _{\tau }g|+\frac{1}{n(x)\cdot v}\big(%
|\partial _{t}g|+\langle v\rangle |\nabla _{\tau }g|+|H|\big)\big\} &\in
&L^{\infty }([0,T];L^{p}(\gamma _{-})), \\
 e^{-\varpi \langle v\rangle t} \alpha^{\beta}\big|-v\cdot \nabla _{x}f_{0}-\nu  f_{0}+H_{0}\big|
&\in &L^{p}({\Omega }\times \mathbb{R}^{3}),
\end{eqnarray*}%
and assume $1/p+1/q=1$ there exist $TC_{T}=O(T)$
and $\varepsilon \ll 1$ such that for all $t\in \lbrack 0,T]$
\begin{equation*}
\left\vert \iint_{\Omega \times \mathbb{R}^{3}} e^{-\varpi \langle v\rangle t} \alpha^{\beta} \partial
H(t)h(t)\right\vert \leq C_{T}\big\{||h(t)||_{q}+\varepsilon ||\nu
_{l,\beta }^{1/q}h(t)||_{q}\big\}.
\end{equation*}%
Then $f(t,x,v)$ satisfies
\begin{equation*}
||f(t)||_{\infty }\leq ||f_{0}||_{\infty }+\sup_{0\leq s\leq
t}||g(s)||_{\infty }+\Big{|}\Big{|}\int_{0}^{t}H(s)ds\Big{|}\Big{|}_{\infty }.
\end{equation*}%
Recall $\partial =[\partial _{t},\nabla _{x},\nabla _{v}],$ then%
\begin{eqnarray*}
\\
\{\partial _{t}+v\cdot \nabla _{x}+\nu _{\varpi, \beta}\}[ e^{-\varpi \langle v\rangle t} \alpha^{\beta}\partial f] &=& e^{-\varpi \langle v\rangle t} \alpha^{\beta}
\big[  -\partial v\cdot \nabla_{x} f - \partial \nu f  + \partial H \big]
 , \\
 e^{-\varpi \langle v\rangle t} \alpha^{\beta} \partial f|_{t=0} &=& e^{-\varpi \langle v\rangle t} \alpha^{\beta}\partial f_0, \ \ \
 e^{-\varpi \langle v\rangle t} \alpha^{\beta}\partial f|_{\gamma _{-}}  \ = \  e^{-\varpi \langle v\rangle t} \alpha^{\beta}[\partial g|_{\gamma _{-}}],
 \\
\end{eqnarray*}%
where $[\partial g|_{\gamma _{-}}]$ is given in (\ref{eq:nabla_bc}).
Moreover, recalling (\ref{finitial}) and (\ref{gboundary}), we have for $%
2\leq p<\infty, $%
\begin{equation}\label{green_alpha}
\begin{split}
& \int_{\Omega \times \mathbb{R}^{3}}| e^{-\varpi \langle v\rangle t} \alpha^{\beta}\partial
f(t)|^{p}+\int_{0}^{t}\int_{\Omega \times \mathbb{R}^{3}}\nu _{\varpi,\beta }|%
 e^{-\varpi \langle v\rangle t} \alpha^{\beta}\partial f|^{p}+\int_{0}^{t}\int_{\gamma _{+}}| e^{-\varpi \langle v\rangle t} \alpha^{\beta}\partial f|^{p}   \\
&    \lesssim \int_{\Omega \times \mathbb{R}^{3}}| e^{-\varpi \langle v\rangle t} \alpha^{\beta}\partial
f_{0}|^{p}+\int_{0}^{t}\int_{\gamma _{-}}| e^{-\varpi \langle v\rangle t} \alpha^{\beta}\partial
g|^{p}  \\
& \  +\int_{0}^{t}\int_{\Omega \times \mathbb{R}^{3}}| e^{-\varpi \langle v\rangle t} \alpha^{\beta} \partial H - e^{-\varpi \langle v\rangle t} \alpha^{\beta}\partial v \cdot   \nabla_x f -\partial \nu  e^{-\varpi \langle v\rangle t} \alpha^{\beta} f|| e^{-\varpi \langle v\rangle t} \alpha^{\beta} \partial f|^{p-1} , \\ 
& || e^{-\varpi \langle v\rangle t} \alpha^{\beta}\partial f(t)||_{\infty }  \\
&  \lesssim  || e^{-\varpi \langle v\rangle t} \alpha^{\beta}\partial f_{0}||_{\infty }+|| e^{-\varpi \langle v\rangle t} \alpha^{\beta}\partial g||_{\infty } \\
& +\int_{0}^{t}|| e^{-\varpi \langle v\rangle t} \alpha^{\beta} \partial H - \partial v \cdot  e^{-\varpi \langle v\rangle t} \alpha^{\beta}\nabla_x f -\partial \nu  e^{-\varpi \langle v\rangle t} \alpha^{\beta} f||_{\infty } ,\text{ }\ \ \ \ \text{for }p=\infty .   
\end{split}
\end{equation}
\end{proposition}

\begin{proof}
First we assume $f_0, g$ and $H$ have compact supports in $\{v\in\mathbb{R}^3 : |v|<m\}$. We estimate $\partial f$ in the bulk. From the velocity lemma (Lemma \ref{velocity_lemma}), we have
\begin{equation}\notag
\begin{split}
&\sup_{t\leq
t_{\mathbf{b}}}\frac{ e^{-\varpi \langle v\rangle t} \alpha^{\beta}(x,v)}{%
  \alpha^{\beta}(x-tv,v)} \leq e^{
C_{m,\beta}   t}, \ \
\sup_{t\geq t_{\mathbf{b}}}%
\frac{ e^{-\varpi \langle v\rangle t} \alpha^{\beta}}{%
 e^{-\varpi \langle v\rangle (t-t_{\mathbf{b}})} \alpha^{\beta}(x_{\mathbf{b}},v)} \leq e^{ C_{m,\beta}
t_{\mathbf{b}}},\\   \
&
\sup_{\max\{t-t_{\mathbf{b}},0\}\leq s\leq t}\frac{ e^{-\varpi \langle v\rangle t} \alpha^{\beta}}{%
 e^{-\varpi \langle v\rangle (t-s)} \alpha(x-sv,v)^{\beta} } \leq e^{ C_{m,\beta}   s}.
\end{split}
\end{equation}
Multiply $ e^{-\varpi \langle v\rangle t} \alpha^{\beta}$ by the above direct computations
and use the above inequalities to get
\begin{equation}\label{dpartialf}
\begin{split}
& e^{-\varpi \langle v\rangle t} \alpha^{\beta} |%
\partial
_{t}f(t,x,v)|\\ &\lesssim \
e^{C_{m, \beta }t}e^{-\int^{t}_{0}\nu }   \alpha^{\beta} \big|[\nu
f_{0} +
v\cdot\nabla_{x}f_{0}-H|_{t=0}]( x-tv,v)\big| \mathbf{1}_{\{t<t_{\mathbf{b}}\}}\\
&  \ +
e^{C_{m,\beta }t_{\mathbf{b}}}e^{-\int^{t_{\mathbf{b}}}_{0}\nu} e^{-\varpi \langle v\rangle (t-t_{\mathbf{b}})} \alpha^{\beta} \partial_{t}\big|g(t-t_{\mathbf{b}},x_{
\mathbf{b}},v)\big| \mathbf{1}_{\{t>t_{\mathbf{b}}\}}\\
 & \ +
\int_{0}^{\min(t,t_{\mathbf{b}})}
e^{C_{m,\beta  }s}e^{-\int^{s}_{0}\nu} e^{-\varpi \langle v\rangle (t-s)} \alpha^{\beta}
\big|\partial_{t}  H(t-s,x-vs,v)
\big|\mathrm{d} s ,\\
& e^{-\varpi \langle v\rangle t} \alpha^{\beta} |\nabla
_{x}f(t,x,v)|\\   &\lesssim \
e^{C_{m,\beta }t}e^{-\int^{t}_{0}\nu}
 \alpha^{\beta} \big|\nabla _{x}f_{0}( x-tv,v)\big|\mathbf{1}_{\{t<t_{\mathbf{b}}\}}+
e^{C_{m,\beta }t_{\mathbf{b}}}e^{-\int^{t_{\mathbf{b}}}_{0}\nu} \sum_{i=1}^2
 \tau_i
 e^{-\varpi \langle v\rangle( t-t_{\mathbf{b}})} \alpha^{\beta}
|\partial_{\tau_i} g(t-t_{\mathbf{b}},x_\mathbf{b},v)|\mathbf{1}_{\{t>t_{\mathbf{b}}\}}
\\
& \ +e^{C_{m,\beta }t_{\mathbf{b}}}e^{- \int^{t_{\mathbf{b}}}_{0}\nu}    n(x_{\mathbf{b}}) \frac{
e^{-\varpi \langle v\rangle (t-t_{\mathbf{b}} )} \alpha^{\beta}
(
x_{\mathbf{b}},v)}{
|v\cdot n(x_{\mathbf{b}})|} \Big| \Big\{\partial_t g +  \sum_{i=1}^2 (v\cdot
\tau_i)
\partial_{\tau_i}g +\nu g -H \Big\}
 (t-t_{\mathbf{b}},x_\mathbf{b}%
,v)\Big|\mathbf{1}_{\{t>t_{\mathbf{b}}\}}\\
& \ + \int_{0}^{\min(t,t_{\mathbf{b}})} e^{C_{m, \beta}s} e^{-\int^{s}_{0}\nu}
 e^{-\varpi \langle v\rangle (t-s)} \alpha^{\beta} \big|
\nabla_{x}H(t-s,x-vs,v)\big| \mathrm{d} s
,\\ 
& e^{-\varpi \langle v\rangle t} \alpha^{\beta} |\nabla
_{v}f(t,x,v)| \\  &\lesssim \
e^{C_{m,\beta }t}e^{-\int^{t}_{0}\nu}  \alpha^{\beta}\big|[-t\nabla_{x}f_{0}
+
\nabla_{v}f_{0} -t \nabla_v\nu(v) f_{0} ]( x-tv,v)\big|\mathbf{1}_{\{t<t_{\mathbf{b}}\}}\\
& \ +
e^{C_{m,\beta }t_{\mathbf{b}}}
e^{-\int^{t_{\mathbf{b}}}_{0}\nu}   \sum_{i=1}^2
 \tau_i  e^{-\varpi \langle v\rangle (t-t_{\mathbf{b}})} \alpha^{\beta}
|\partial_{\tau_i} g
(t-t_{\mathbf{b}},x_{\mathbf{b}},v)|\mathbf{1}_{\{t >t_{\mathbf{b}}\}}
\\
 & \ + e^{C_{m,\beta }t_{\mathbf{b}}} e^{-\int^{t_{\mathbf{b}}}_{0}
\nu} n(x_{\mathbf{b}})\frac{
 e^{-\varpi \langle v\rangle( t-t_{\mathbf{b}})} \alpha^{\beta} }{|v\cdot n(x_{\mathbf{b}})|}
\Big|\Big\{\partial_t g +
\sum_{i=1}^2 (v\cdot \tau_i) \partial_{\tau_i}g
+\nu g -H \Big\}
 (t-t_{\mathbf{b}},x_%
\mathbf{b},v)\Big|\mathbf{1}_{\{t>t_{\mathbf{b}}\}}
\\
 & \ +
e^{C_{m,\beta }t_{\mathbf{b}}} e^{-\int^{t_{\mathbf{b}}}_{0} \nu}
 e^{-\varpi \langle v\rangle (t-t_{\mathbf{b}})} \alpha^{\beta}  \big\{ | \nabla_v
g(t-t_{\mathbf{b}},x_{\mathbf{b}},v)| +|t_{\mathbf{b}} \nabla_v \nu(v)| |
g(t-t_{\mathbf{b}},x_{\mathbf{b}},v) |
\big\}\mathbf{1}_{\{t>t_{\mathbf{b}}\}}\\
& \ + \int_{0}^{\min(t,t_{\mathbf{b}})}
e^{C_{m,\beta }s}
e^{-\int^{s}_{0}\nu} e^{-\varpi \langle v\rangle (t-s)} \alpha^{\beta}\big|\{\nabla_{v}H -s\nabla_{x}H -
s\nabla \nu H
\}(t-s,x-vs,v) \big|\mathrm{d} s .
\end{split}
\end{equation}
Following (\ref{change1}) and (\ref{change2}) of Proposition 1 and using the condition of Proposition 2, we deduce
\begin{equation}\notag
\begin{split}
|| e^{-\varpi \langle v\rangle t} \alpha^{\beta}\partial_t f(t)||_p &\lesssim_{t,m,\beta } ||
  \alpha^{\beta} [v\cdot \nabla_x f_0 + \nu f_0 -H(0,\cdot,\cdot)]||_p + \left[
\int_0^t ||
 e^{-\varpi \langle v\rangle s} \alpha^{\beta}  \partial_t g(s) ||_{\gamma,p}^p\mathrm{d} s\right]^{1/p}\\
& \ + \left[
\int_0^t ||  e^{-\varpi \langle v\rangle s} \alpha^{\beta} \partial_t H(s)  ||_{p}^p\mathrm{d} s\right]
^{1/p},\\
||
 e^{-\varpi \langle v\rangle t} \alpha^{\beta} \nabla_x f(t)||_p & \lesssim_{t,m,\beta } ||
  \alpha^{\beta}
\nabla_x f_0 ||_p + \sum_{i=1}^2 \left[\int_0^t ||
 e^{-\varpi \langle v\rangle s} \alpha^{\beta}
\partial_{\tau_i} g(s) ||_{\gamma,p}^p\mathrm{d} s\right]^{1/p} \\
& \ + \left[
\int_0^t \Big{|}\Big{|} \frac{ e^{-\varpi \langle v\rangle t} \alpha^{\beta}}{v\cdot n}
\{\partial_t g + \sum
(v\cdot \tau_i) \partial_{\tau_i}g+ \nu g -H\} \Big{|}
\Big{|} _{\gamma,p}^p
\mathrm{d} s\right]^{1/p}\\
 & \ +\left[\int_0^t ||  e^{-\varpi \langle v\rangle s} \alpha^{\beta}
\nabla_x H(s)
||_{p}^p\mathrm{d} s\right]^{1/p},  
\end{split}
\end{equation}
\begin{equation}\notag
\begin{split}
|| e^{-\varpi \langle v\rangle t} \alpha^{\beta}  \nabla_v
f(t)||_p &
\lesssim_{t,m,\beta }||  \alpha^{\beta}  \nabla_v f_0 ||_p +
\sum_{i=1}^2 \left[
\int_0^t ||  e^{-\varpi \langle v\rangle s} \alpha^{\beta}  \partial_{\tau_i} g(s)
||_{\gamma,p}^p\mathrm{d} s\right]
^{1/p}+ \sup_{ 0 \leq s\leq t}|| \langle
v\rangle  g(s) ||_\infty \\
 & \ +
\left[\int_0^t \Big{|}\Big{|} \frac{%
 e^{-\varpi \langle v\rangle t} \alpha^{\beta}}{v\cdot n} \{\partial_t g +
\sum (v\cdot \tau_i)
\partial_{\tau_i}g+ \nu g -H\} \Big{|}\Big{|}
_{\gamma,p}^p \mathrm{d} s\right]
^{1/p}\\
& \ + \left[\int_0^t || e^{-\varpi \langle v\rangle s} \alpha^{\beta} \nabla_v
g (s) ||_p^p \mathrm{d} s\right]
^{1/p} \\
  & \ +\left[\int_0^t || e^{-\varpi \langle v\rangle s} \alpha^{\beta}
\nabla_v H(s)  ||_{p}^p +  ||
 e^{-\varpi \langle v\rangle s} \alpha^{\beta}  \nabla_x H(s)  ||_{p}^p \mathrm{d} s
\right]^{1/p}
   + \sup_{ 0
\leq s\leq t} ||\langle v \rangle H(s)||_\infty.
\end{split}
\end{equation}
By the hypothesis of
Proposition 2 and assumption on $f_0,g$ and $H$ to have compact support, the right hand sides are bounded and hence $ e^{-\varpi \langle v\rangle t} \alpha^{\beta} \partial_t f ,  e^{-\varpi \langle v\rangle t} \alpha^{\beta} \nabla_x f,$ and $ e^{-\varpi \langle v\rangle t} \alpha^{\beta} \nabla_v f$ are in $L^\infty([0,T]; L^p(\Omega\times\mathbb{R}^3))$.

Since $f_0, g$ and $H$ are compactly supported on
$\{v\in\mathbb{R}^3 :|v|\leq m\}$, the derivatives $ e^{-\varpi \langle v\rangle t} \alpha^{\beta}  \partial_t
f,$  $e^{-\varpi \langle v\rangle t} \alpha^{\beta}
\nabla_x f$ and $ e^{-\varpi \langle v\rangle t} \alpha^{\beta}  \nabla_v f$
are compactly
supported on $\{v\in\mathbb{R}^3 : |v|\leq m\}$ and hence from (\ref{df})
and (\ref{partialf})
\begin{equation}\notag
\{\partial_t + v\cdot\nabla_x + \nu_{\varpi,\beta}\}
[ e^{-\varpi \langle v\rangle t} \alpha^{\beta}
\partial f]=  e^{-\varpi \langle v\rangle t} \alpha^{\beta}\partial H-
\partial
v\cdot
 e^{-\varpi \langle v\rangle t} \alpha^{\beta}\nabla_x f
-\partial \nu(v)  e^{-\varpi \langle v\rangle t} \alpha^{\beta} f.
\end{equation}

Moreover, from the general definition of traces, by choosing a
test function
multiplied by $ e^{-\varpi \langle v\rangle t} \alpha^{\beta}$, we deduce $ e^{-\varpi \langle v\rangle t} \alpha^{\beta}
\partial f$
has the
same trace as $ e^{-\varpi \langle v\rangle t} \alpha^{\beta} [\partial f|_\gamma]$.

Now we can apply Lemma
\ref{Green} to have (\ref{green_alpha}) which does not depend on the
velocity cut-off. Therefore for the general case, we
use (\ref{cutoff1}) and pass a limit to conclude the proof.
\end{proof}

 \vspace{12pt}

\section{\large{Dynamical Non-local to Local Estimate} }
 
\vspace{4pt}

The main purpose of this section is to prove Lemma \ref{lemma_nonlocal} and its variants (Lemma \ref{lemma_nonlocal_u}).

We first prove the dynamical non-local to local estimates for the stochastic (diffuse) cycles:

\begin{proof}[\textbf{Proof of (1) of Lemma \protect\ref{lemma_nonlocal}}]
Since $\frac{\langle u\rangle^{r}}{\langle v\rangle^{r}} \lesssim \{ 1+
|v-u|^{2}\}^{\frac{r}{2}}$ and $\langle V_{\mathbf{cl}}(s)-u\rangle^{r}
e^{-\theta |V_{\mathbf{cl}}(s)-u|^{2}}\lesssim e^{-C_{\theta,r} |V_{\mathbf{%
cl}}(s)-u|^{2}}$, it suffices to consider $r=0$ case. We prove (\ref{nonlocal}).

\noindent\textit{Step 1.} We show
that
\begin{equation}  \label{int_xi}
\int_{\mathbb{R}^{3}}\frac{e^{-\theta |v-u|^{2}}}{|v-u|^{2-\kappa }[\alpha
(X_{\mathbf{cl}} (s;t,x,v),u  )]^{\beta}} \mathrm{d}u \ \lesssim  \ \frac{1}{%
|v|^{2\beta-1} |\xi(X_{\mathbf{cl}} (s;t,x,v))  |^{\beta-\frac{1}{2}}} .
\end{equation}
For fixed $s\in \lbrack 0,t_{\mathbf{b}}(x,v))$ and therefore fixed $X_{%
\mathbf{cl}}(s) = x-(t_{\mathbf{b}}(x,v) -s)v\in \bar{\Omega}$.

Firstly, we consider the case of $|\xi(x)|\leq \delta_{\Omega}\ll 1$. From the
assumption, we have $\nabla \xi(x) \neq 0$ and therefore there is uniquely
determined unit vector $n(X_{\mathbf{cl}}(s))= \frac{\nabla \xi(X_{\mathbf{cl%
}}(s))}{|\nabla \xi(X_{\mathbf{cl}}(s))|}$. We choose two unit vector $%
\tau_{1}$ and $\tau_{2}$ so that $\{ \tau _{1},\tau _{2},n(X_{\mathbf{cl}%
}(s))\}$ is an orthonormal basis of $\mathbb{R}^{3}$.

We decompose the velocity variables $u \in \mathbb{R}^{3}$ as
\begin{equation*}
u=u_{n}n(X_{\mathbf{cl}}(s))+u_{\tau }\cdot \tau =u_{n}n(X_{\mathbf{cl}%
}(s))+\sum_{i=1}^{2}u_{\tau ,i}\tau _{i}.
\end{equation*}
We note that $u_{\tau }\in \mathbb{R}^{2}$ and $u_{n}\in \mathbb{R}$ are
completely free coordinates. Therefore using the Fubini's theorem we can
rearrange the order of integration freely. Now we split, for $0\leq s\leq t_{%
\mathbf{b}}(x,v),$
\begin{equation}
\begin{split}
& \int_{\mathbb{R}^{3}}\frac{ e^{-\theta|v-u|^{2}}}{|v-u|^{2-\kappa }}\frac{1%
}{[\alpha (   X_{\mathbf{cl}}(s;t,x,v),u   )]^{\beta} }\mathrm{d}u \\
\lesssim & \int_{\mathbb{R}^{2}} \int_{\mathbb{R}}\frac{e^{- \theta|v-u|^{2}}%
}{|v-u|^{2-\kappa }\big[ |u_{n}|^{2}+|\xi (X_{\mathbf{cl}}(s))||u|^{2} \big]%
^{\beta} } \mathrm{d}u_{n}\mathrm{d}u_{\tau } \\
=& \  {\int_{|u|\geq 5|v|}} \ + \  {%
\int_{|u|\leq \frac{|v|}{2}}}  \ + \  {\int_{\frac{%
|v|}{2}\leq |u|\leq 5|v|}} =  \mathbf{(I)} +  \mathbf{(II)} +  \mathbf{(III)}.
\end{split}
\notag
\end{equation}

For the first term $\mathbf{(I)}$ we use, for $|u|\geq 5|v|$ (therefore $%
|v|\leq \frac{|u|}{5}$),
\begin{equation*}
|u-v|^{2}=\frac{|u-v|^{2}}{2}+\frac{|u-v|^{2}}{2}\geq \frac{\frac{ |u|^{2}}{2%
}-|v|^{2}}{2}+\frac{ \frac{ |u|^{2}}{2}-|v|^{2}}{2}\geq \frac{23}{4} |v|^{2}+%
\frac{23}{100} |u|^{2} \gtrsim |v|^{2}+|u|^{2},
\end{equation*}%
and we use $\big[ |u_{n}|^{2} + |\xi| |u|^{2} \big]^{\beta} \geq \big[%
|u_{n}|^{2} + 25 |\xi| |v|^{2}\big]^{\beta} \gtrsim \big[|u_{n}|^{2} +
|\xi||v|^{2}\big]^{\beta}$ for $|u| \geq 5|v|$ to have
\begin{equation}
\mathbf{(I)}\lesssim  e^{- {C}
|v|^{2}}\int_{\mathbb{R}^{2}} \mathrm{d}u_{\tau} \frac{  e^{-C|u_{\tau}|^{2}}}{ |v_{\tau }-u_{\tau }|^{2-\kappa }}  \int_{\mathbb{R}}  \mathrm{d}u_{n}
\frac{e^{- {C} |u_{n}|^{2}}}{ \big[%
|u_{n}|^{2}+|\xi ||v|^{2}\big]^{\beta}}     .
\notag
\end{equation}
Since $\frac{1}{| v_{\tau} - u_{\tau} |^{2-\kappa}} \in L^{1}_{\mathrm{loc}%
}(\{ u_{\tau} \in \mathbb{R}^{2}\})$ for $\kappa >0$ we first integrate over $u_{\tau}$ is finite. Then 
\begin{equation}
\begin{split}
\mathbf{(I)}\lesssim & \ e^{- {C} |v|^{2}}\int_{\mathbb{R}}\frac{e^{- C
|u_{n}|^{2}}}{\big[ |u_{n}|^{2}+|\xi ||v|^{2} \big]^{\beta}}
\mathrm{d}u_{n} \\
\lesssim & \ e^{- {C} |v|^{2}} \Big\{\int_{ 10}^{\infty}\frac{e^{- C
|u_{n}|^{2}}}{|u_{n}|^{2\beta} \mathbf{1}_{\{ |u_{n}| \geq 10 \}} } \mathrm{d%
}|u_{n}| + \int_{0}^{10} \frac{ \mathrm{d}|u_{n}| }{\big[|u_{n}|^{2} + |\xi|
|v|^{2}\big]^{\beta}}\Big\}  \\
\lesssim & \ \big(1+\int_{0}^{10}\frac{\mathrm{d}|u_{n}|}{ \big[%
|u_{n}|^{2}+|\xi ||v|^{2} \big]^{\beta}}\big) {e^{-C|v|^{2}}} \lesssim \ {%
e^{- C |v|^{2}}} \big(1+\int_{0}^{10} \frac{\mathrm{d}[|\xi |^{\frac{1}{2}%
}|v|\tan \theta ]}{|\xi |^{\beta}|v|^{2\beta}(1+\tan ^{2}\theta )^{\beta}}%
\big) \\
\lesssim & \ {e^{- C |v|^{2}}} \Big(1+ \frac{1 }{|v|^{2\beta -1}}\frac{1}{%
|\xi |^{ \beta-{1}/{2}}} \int_{0}^{\pi/2} (\cos\theta)^{2 \beta-2} \mathrm{d}
\theta \Big) \lesssim e^{-C|v|^{2}}\Big(1+\frac{1 }{|v|^{2\beta -1}}\frac{1}{%
|\xi |^{ \beta-{1}/{2}}}\Big) \\
\lesssim & \ \frac{e^{-C_{\theta}|v|^{2}} }{|v|^{2\beta -1}}\frac{1}{|\xi  (X_{\mathbf{cl}}(s;t,x,v))
|^{\beta-1/2}} ,
\end{split}
\notag
\end{equation}%
where we have used a change of variables: $|u_{n}|=|\xi |^{\frac{1}{2}%
}|v|\tan \theta$ and $\mathrm{d}|u_{n}| = |\xi|^{\frac{1}{2}} |v| \sec^{2}
\theta \mathrm{d} \theta$ and $(\cos\theta)^{2\beta -2} \in L^{1}_{\mathrm{%
loc}}(\{ \theta \in [0, \frac{\pi}{2}] \})$ for $\beta> \frac{1}{2} $.

For the second term $\mathbf{(II)}$, we use $|v-u|\geq |v|-|u| \geq |v|-%
\frac{|v|}{2} \geq \frac{|v|}{2}$ from $|u|\leq \frac{|v|}{2}$, and apply
the change of variables $u\mapsto |v|u$ to have
\begin{equation}
\begin{split}
\mathbf{(II)}& \ \lesssim \ \frac{1}{|v|^{2-\kappa}} {\int_{|u_{n}|+|u_{\tau
}|\leq \frac{|v|}{2}}\frac{e^{-C|v|^{2}} \mathrm{d}u_{n}\mathrm{d}u_{\tau }}{
\big[|u_{n}|^{2}+|\xi ||u_{\tau }|^{2}\big]^{\beta}}} \\
& \ = \ \frac{1}{|v|^{2-\kappa}} \int_{|v|( |u_{n}| + |u_{\tau}|) \leq \frac{%
|v|}{2}} \frac{e^{-C|v|^{2}} |v| \mathrm{d} u_{n} |v|^{2} \mathrm{d}u_{\tau}
}{ \big[ |v|^{2}|u_{n}|^{2} + |\xi | |v|^{2} |u_{\tau}|^{2} \big]^{\beta} }
\\
& \ \lesssim \ \frac{e^{-C|v|^{2}}}{|v|^{2\beta -\kappa -1 }}%
\int_{|u_{\tau}| \leq \frac{1}{2}} \int_{ |u_{n }|\leq \frac{1}{2}}\frac{1%
}{ \big[|u_{n}|^{2}+|\xi ||u_{\tau }|^{2}\big]^{\beta}} \mathrm{d}u_{n}
\mathrm{d}u_{\tau}.
\end{split}
\notag
\end{equation}

Now we apply the change of variables $|u_{n}| = |\xi|^{\frac{1}{2}}
|u_{\tau}| \tan\theta$ for $\theta \in [0,\frac{\pi}{2}]$ with $\mathrm{d}%
u_{n} = |\xi|^{\frac{1}{2}} |u_{\tau}| \sec^{2 } \theta \mathrm{d}\theta$ to
have
\begin{equation}
\begin{split}
\mathbf{(II)}& \ \lesssim \ \frac{e^{-C|v|^{2}}}{|v|^{2\beta -\kappa -1}}
\int_{|u_{\tau}| \leq \frac{1}{2}} \mathrm{d} u_{\tau} \int_{0}^{\frac{\pi}{2%
}} \frac{|\xi|^{\frac{1}{2}} |u_{\tau}| \sec^{2}\theta \mathrm{d}\theta }{ %
\big[ |\xi| |u_{\tau}|^{2} \tan^{2}\theta + |\xi| |u_{\tau}|^{2}\big]^{\beta}%
} \\
& \ \lesssim \ \frac{e^{-C |v|^{2}}}{|v|^{2\beta -\kappa-1} |\xi|^{\beta-1/2}%
} \int_{|u_{\tau}| \leq \frac{1}{2}} \frac{\mathrm{d} u_{\tau}}{%
|u_{\tau}|^{2\beta-1}} \int_{0}^{\pi/2} (\cos\theta)^{2\beta -2}\mathrm{d}%
\theta \\
& \ \lesssim \frac{e^{-C |v|^{2}}}{|v|^{2\beta -\kappa-1} |\xi|^{\beta-1/2}},
\end{split}
\notag
\end{equation}
where we have used $\frac{1}{|u_{\tau}|^{2\beta-1}} \in L^{1}_{\text{loc}%
}(\{ u_{\tau} \in \mathbb{R}^{2}\})$ for $\beta< \frac{3}{2}$ and $%
(\cos\theta)^{2\beta -2} \in L^{1}_{\text{loc}}(\{ \theta \in [0,\frac{\pi}{2%
}] \})$ for $\beta > \frac{1}{2}.$

For the last term $\mathbf{(III)}$, we use the lower bound of $|u|$ ($%
|u|\geq \frac{|v|}{2}$) to have $\big[  |u_{n}|^{2} +|\xi||u|^{2} \big]%
^{\beta} \geq \big[  |u_{n}|^{2} + |\xi| \frac{|v|^{2}}{4} \big]^{\beta}
\gtrsim \big[  |u_{n}|^{2} + |\xi| |v|^{2} \big]^{\beta}$ and
\begin{equation}
\begin{split}
\int_{\frac{|v|}{2}\leq |u|\leq 5|v|} & \lesssim {\int_{ 0\leq |u_{\tau
}|\leq 5{|v|}}\frac{e^{-\frac{C}{2}|v_{\tau }-u_{\tau }|^{2}}}{|v_{\tau
}-u_{\tau }|^{2-\kappa }}}\mathrm{d}u_{\tau }\int_{0 }^{5|v|}\frac{1}{\big[%
|u_{n}|^{2}+|\xi ||v|^{2}\big]^{\beta}}\mathrm{d}u_{n} \\
& \lesssim \int_{0}^{5|v|}\frac{1}{\big[|u_{n}|^{2}+|\xi ||v|^{2}\big]%
^{\beta}}\mathrm{d}u_{n},
\end{split}
\notag
\end{equation}
where we have used $\frac{1}{|u_{\tau }|^{2-\kappa }}\in L^{1}_{\mathrm{loc}%
}(\mathbb{R}^{2})$ for $\kappa >0.$ We apply a change of variables: $%
|u_{n}|=|\xi |^{1/2}|v |\tan \theta $ for $\theta \in \lbrack 0,\pi /2]$
with $\mathrm{d}|u_{n}| = |\xi|^{\frac{1}{2}} |v|\sec^{2} \theta \mathrm{d}%
\theta$. Hence
\begin{equation}
\begin{split}
\mathbf{(III)}& \lesssim \int_{0}^{5|v|}\frac{1}{\big[|u_{n}|^{2}+|\xi
||v|^{2}\big]^{\beta}}\mathrm{d}u_{n}=\int^{ \frac{\pi}{2}}_{0}\frac{
(\cos\theta)^{2\beta -2} }{|\xi |^{\beta- \frac{1}{2}}|v|^{2\beta -1} }
\mathrm{d}\theta \lesssim \frac{1}{|v|^{2\beta -1}} \frac{1}{|v|^{ 2\beta -1}%
},
\end{split}
\notag
\end{equation}
where we used $(\cos\theta)^{2\beta -2}\in L^{1}_{\text{loc}}(\{ \theta \in
[0,\frac{\pi}{2}] \})$ for $\beta> \frac{1}{2}$. Overall, we combine the
estimates of $\mathbf{(I)},\mathbf{(II)}$ and $\mathbf{(III)}$ to conclude (%
\ref{int_xi}).

Secondly, we consider the case of $|\xi(x)|>\delta_{\Omega}.$ Then we can
choose any orthonormal basis, for example standard basis $\{\tau_{1},
\tau_{2}, n\} = (e_{1},e_{2},e_{3})$, to decompose the velocity variables $u
\in \mathbb{R}^{3}$ as $u= u_{1}e_{1} + u_{2} e_{2} + u_{3} e_{3}:=
u_{\tau,1} e_{1} + u_{\tau,2} e_{2} + u_{n} e_{3}$. Then
\begin{equation}
\begin{split}
\alpha(X_{\mathbf{cl}}(s),u)& = |u\cdot \nabla \xi( X_{\mathbf{cl}}(s)
)|^{2} - 2\xi(X_{\mathbf{cl}}(s)) \{u\cdot \nabla^{2} \xi(X_{\mathbf{cl}%
}(s))\cdot u\} \\
&\geq 2|\xi(X_{\mathbf{cl}}(s))| \{u\cdot \nabla^{2} \xi(X_{\mathbf{cl}%
}(s))\cdot u\} \\
& =\delta_{\Omega} C_{\xi} |u|^{2} + |\xi(X_{\mathbf{cl}}(s))| \{u\cdot
\nabla^{2} \xi(   X_{\mathbf{cl}}(s )  )\cdot u\} \\
&\gtrsim |u_{n}|^{2} + |\xi(   X_{\mathbf{cl}}(s;t,x,v) )||u|^{2}.
\end{split}
\notag
\end{equation}
Then we follow all the proof with the same decomposition for $v:= v_{\tau,1}
e_{1} + v_{\tau,2} e_{2} + v_{n} e_{3}$ as well to conclude (\ref{int_xi})
for $|\xi(x)|> \delta_{\Omega}.$

\vspace{4pt}

\noindent\textit{Step 2. \ } In this step we establish (\ref{sigma}) and (\ref{COV_xi_s}).

We first assume $v\cdot \nabla\xi(x)\geq 0$ and $x\in \partial \Omega.$ There exist $\sigma _{1},\sigma _{2}>0$ such that
\begin{equation}  \label{sigma}
\begin{split}
& |v\cdot\nabla \xi ( x-(t_{\mathbf{b}}(x,v)-s) v ) |\gtrsim \sqrt{\alpha (
x-(t_{\mathbf{b}}(x,v)-s) v,v)}\  \\
& \ \ \ \ \ \ \ \ \ \ \ \ \ \ \ \ \ \ \ \ \ \ \ \ \ \ \ \ \ \ \ \ \ \ \ \ \
\ \ \ \ \ \ \ \ \ \ \ \ \ \ \ \ \ \ \ \ \ \text{for all }s\in \lbrack
0,\sigma _{1}]\cup \lbrack t_{\mathbf{b}}(x,v)-\sigma _{2},t_{\mathbf{b}%
}(x,v)], \\
\end{split}%
\end{equation}
and $|v|\sqrt{-\xi ( x-(t_{\mathbf{b}}(x,v)-s) v )}\gtrsim \sqrt{\alpha (
x-(t_{\mathbf{b}}(x,v)-s) v ,v)} \ \   \text{for all }s\in \lbrack \sigma
_{1},t_{\mathbf{b}}(x,v)-\sigma _{2}].$ The mapping $s \mapsto \xi(x-(t_{%
\mathbf{b}}(x,v) -s )v)$ is one-to-one and onto on $s \in [0,\sigma_{1}]$ or
on $s \in [t_{\mathbf{b}}(x,v) -\sigma_{2}, t_{\mathbf{b}}(x,v)]$. Moreover
this mapping $s \mapsto \xi(x-(t_{\mathbf{b}}(x,v) -s )v)$ is diffeomorphism
and we have a change of variables on $s \in [0,\sigma_{1}] \ \text{ or } \
s\in [t_{\mathbf{b}}(x,v) - \sigma_{2}, t_{\mathbf{b}}(x,v)]$.
\begin{equation}  \label{COV_xi_s}
\mathrm{d}s = \frac{\mathrm{d} |\xi|}{ | \nabla \xi (x-(t_{\mathbf{b}} (x,v)
-s )v ) \cdot v |} \lesssim \frac{\mathrm{d}|\xi|}{\sqrt{\alpha (x-(t_{%
\mathbf{b}} (x,v) -s )v )}}.
\end{equation}

Firstly we prove (\ref{sigma}). Recall the definition of $\alpha$ in Definition \ref{K_D}. It suffices to show
\begin{equation}
\begin{split}
& |v\cdot \nabla \xi ( x-(t_{\mathbf{b}}(x,v)-s) v) |\geq |v|\sqrt{-\xi (
x-(t_{\mathbf{b}}(x,v)-s) v)}, \ \ s\in \lbrack 0,\sigma _{1}]\cup \lbrack
t_{\mathbf{b}}(x,v)-\sigma _{2},t_{\mathbf{b}}(x,v)], \\
& |v\cdot \nabla \xi ( x-(t_{\mathbf{b}}(x,v)-s) v) |\leq |v|\sqrt{-\xi (
x-(t_{\mathbf{b}}(x,v)-s) v)}, \ \ s\in \lbrack \sigma _{1},t_{\mathbf{b}%
}(x,v)-\sigma _{2}].\notag
\end{split}
\end{equation}

If $v=0$ or $v\cdot \nabla\xi(x)=0$ then (\ref{sigma}) holds clearly.
Therefore we may assume $v\neq 0$ and $v\cdot \nabla \xi(x)>0$. Due to the Velocity lemma, $v\cdot \frac{ \nabla \xi (x)}{|\nabla \xi (x)|}>0$ and $%
v\cdot \frac{\nabla\xi (x_{\mathbf{b}}(x,v) )}{|\nabla \xi (x_{\mathbf{b}%
}(x,v))|}<0$. By the mean value theorem we choose $t^{\ast }\in (0,t_{\mathbf{b}}(x,v))$ solving $v\cdot \nabla \xi
(x-(t_{\mathbf{b}}(x,v)-t^{*})v)=0$. Moreover due to the convexity of $\xi$ we have
%
\begin{equation*}
\frac{d}{ds}\Big(v\cdot \nabla \xi (x-(t_{\mathbf{b}}(x,v)-s)v)\Big)=v\cdot
\nabla ^{2}\xi (X_{\mathbf{cl}}(s))\cdot v\geq C_{\xi }|v|^{2},
\end{equation*}%
and therefore $t^{\ast }\in (0,t_{\mathbf{b}}(x,v))$ is uniquely determined. Clearly we have $v\cdot \nabla \xi (x-(t_{%
\mathbf{b}}(x,v)-s)v)\geq 0$ for $s\in \lbrack t^{\ast },t_{\mathbf{b}%
}(x,v)] $ and $v\cdot \nabla \xi (x-(t_{\mathbf{b}}(x,v)-s)v)\leq 0$ for $%
s\in \lbrack 0,t^{\ast }]$.

Define $\Phi (s)=\big\{|v\cdot \nabla \xi (x-(t_{\mathbf{b}%
}(x,v)-s)v)|^{2}+|v|^{2}\xi (x-(t_{\mathbf{b}}(x,v)-s)v)\big\}.$ Since $2%
\big(v\cdot \nabla ^{2}\xi (x-(t_{\mathbf{b}}(x,v)-s)v)\cdot v\big)%
+|v|^{2}>0 $ we have
\begin{equation*}
\frac{d}{ds}\Phi (s)=\big(v\cdot \nabla \xi (x-(t_{\mathbf{b}}(x,v)-s)v)\big)%
\Big\{2\big(v\cdot \nabla ^{2}\xi (x-(t_{\mathbf{b}}(x,v)-s)v)\cdot v\big)%
+|v|^{2}\Big\},
\end{equation*}
is strictly negative for $s\in \lbrack 0,t^{\ast }]$ and is strictly
positive for $s\in \lbrack t^{\ast },t_{\mathbf{b}}(x,v)]$. Note that $\Phi
(0)>0$ and $\Phi(t_{ \mathbf{b}}(x,v))>0$ from $v\cdot \frac{\nabla \xi(x)}{%
|\nabla \xi(x)|} >0$ and $v\cdot \frac{\nabla \xi(x_{\mathbf{b}}(x,v))}{%
|\nabla \xi(x_{\mathbf{b}}(x,v))|} <0$. Note that $\Phi$ is continuous
function on the interval $[0,t_{\mathbf{b}}(x,v)]$ so that it has a minimum.
If $\min_{[0,t_{\mathbf{b}}(x,v)]}\Phi (s)\leq 0,$ there exist $\sigma
_{1},\sigma _{2}>0$ satisfying
\begin{equation}
\begin{split}
\Phi (t_{\mathbf{b}}(x,v)+\sigma _{1})&=\Phi (t_{\mathbf{b}%
}(x,v))+\int_{0}^{\sigma _{1}}\frac{d}{ds}\Phi (s)\mathrm{d}s=0, \\
\Phi (t_{\mathbf{b}}(x,v)-\sigma _{2})&=\Phi (t_{\mathbf{b}}(x,v))-\int_{t_{%
\mathbf{b}}(x,v) -\sigma _{2}}^{t_{\mathbf{b}}(x,v)}\frac{d}{ds}\Phi (s)%
\mathrm{d}s=0,
\end{split}
\notag
\end{equation}%
then $\sigma _{1}\leq t^{\ast }$ and $t_{\mathbf{b}}(x,v)-\sigma _{2}\geq
t^{\ast }$ and there is no other $s\in [0,t_{\mathbf{b}}(x,v)]$ satisfying $%
\Phi (s)=0$. Moreover we have $\Phi (s)\leq 0$ for $s\in \lbrack \sigma
_{1},t_{\mathbf{b}}(x,v)-\sigma _{2}]$. If $\min_{[0,t_{\mathbf{b}%
}(x,v)]}\Phi (s)>0,$ there does not exist such $\sigma _{1}$ and $\sigma
_{2} $ then we let $\sigma _{1}=t^{\ast }$ and $\sigma _{2}=t_{\mathbf{b}%
}(x,v)-t^{\ast }$. This proves (\ref{sigma}).

Secondly we prove (\ref{COV_xi_s}). By the proof of ({\ref{sigma}}) and the fact
\begin{equation*}
\frac{d |\xi|}{ds} = -\frac{d}{ds} \xi ( x-(t_{\mathbf{b}}(x,v) -s)v) = -
v\cdot \nabla_{x} \xi(x-(t_{\mathbf{b}}(x,v) -s)v ),
\end{equation*}
and the inverse function theorem we prove (\ref{COV_xi_s}).

\vspace{8pt}

\noindent\textit{Step 3. \ } For small $0 < \tilde{\delta}\ll 1,$ we define
\begin{equation}  \label{sigma_delta}
\tilde{\sigma}_{1} := \min \Big\{\sigma_{1}, \tilde{\delta} \frac{\sqrt{%
\alpha(x,v)}}{|v|^{2}}\Big\} , \ \ \ \tilde{\sigma}_{2} := \min \Big\{%
\sigma_{2}, \tilde{\delta} \frac{\sqrt{\alpha(x,v)}}{|v|^{2}}\Big\}.
\end{equation}
Then both of (\ref{sigma}) and (\ref{COV_xi_s}) hold on $s\in [0,\tilde{
\sigma}_{1}]\cup [t_{\mathbf{b}}(x,v)-\tilde{ \sigma}_{2}, t_{\mathbf{b}%
}(x,v)]$ without constant changing. Moreover, if $s\in [0,\tilde{ \sigma}%
_{1}]\cup [t_{\mathbf{b}}(x,v)-\tilde{ \sigma}_{2}, t_{\mathbf{b}}(x,v)]$
then by the Velocity lemma
\begin{equation}  \label{max_xi}
\max \{|\xi|\} := \max_{s\in [0, \tilde{\sigma}_{1}] \cup [t_{\mathbf{b}%
}(x,v) - \tilde{\sigma}_{2}, t_{\mathbf{b}}(x,v) ]} |\xi(X_{\mathbf{cl}%
}(s))| \ \lesssim \ \tilde{\delta} \frac{\alpha(x,v)}{|v|^{2}}.
\end{equation}
On $s\in [ \tilde{ \sigma}_{1}, t_{\mathbf{b}}(x,v) -\tilde{\sigma}_{2} ]$
we have the following estimate with $\tilde{\delta}-$dependent constant:
\begin{equation}  \label{lower_delta}
\begin{split}
&|v| \sqrt{-\xi(x-(t_{\mathbf{b}}(x,v)-s)v)} \  \ \gtrsim_{\xi, \tilde{\delta}} \ \  \sqrt{\alpha(x-(t_{\mathbf{b}}(x,v)-s)v,v)}.
\end{split}%
\end{equation}

The proof of (\ref{max_xi}) is due to, for $s\in \lbrack 0,\tilde{\sigma}%
_{1}],$
\begin{equation}\label{tz}
\begin{split}
|\xi (x-(t_{\mathbf{b}}(x,v)-s)v)| &\leq  \int_{0}^{s}|v\cdot \nabla \xi
(x-(t_{\mathbf{b}}(x,v)-\tau )v)|\mathrm{d}\tau \\
&\lesssim  \sqrt{\alpha (x,v)}|s|  \lesssim  \min \left\{ \sqrt{\alpha }t_{Z},\frac{\tilde{\delta}\alpha }{%
|v|^{2}}\right\} \equiv B,
\end{split}
\end{equation}

where we have used $\alpha (X_{\mathbf{cl}}(\tau ),V_{\mathbf{cl}}(\tau
))\lesssim _{\xi }\alpha (x,v)$ from the Velocity lemma (Lemma \ref%
{velocity_lemma}). The proof for $s\in \lbrack t_{\mathbf{b}}(x,v)-\tilde{%
\sigma}_{2},t_{\mathbf{b}}(x,v)]$ is exactly same.

Now we prove (\ref{lower_delta}). Recall that $t^{*} \in[0, t_{\mathbf{b}%
}(x,v)]$ in the previous step: $v\cdot \nabla \xi(x-(t_{\mathbf{b}%
}(x,v)-t^{*})v)=0$. Clearly $|\xi(X_{\mathbf{cl}}(s))|$ is an increasing
function on $s\in [0, t^{*}]$ and a decreasing function on $s\in [t^{*}, t_{%
\mathbf{b}}(x,v)]$. This is due to the convexity of $\xi$:
\begin{equation*}
\frac{d^{2}}{ds^{2}} [-\xi(s-(t_{\mathbf{b}}(x,v)-s)v)] = v\cdot
\nabla\xi(x-(t_{\mathbf{b}}(x,v)-s)v) \cdot v \gtrsim_{\xi} |v|^{2},
\end{equation*}
and $v\cdot \nabla \xi(x)>0$ and $v\cdot \nabla \xi(x_{\mathbf{b}}(x,v))<0.$


Therefore
\begin{equation}
\begin{split}
-\xi (x-(t_{\mathbf{b}}(x,v)-s)v)& =-\xi (x)-\int_{t_{\mathbf{b}%
}(x,v)}^{s}v\cdot \nabla \xi (x-(t_{\mathbf{b}}(x,v)-\tau )v)\mathrm{d}\tau
\\
& =\int_{s}^{t_{\mathbf{b}}(x,v)}v\cdot \nabla \xi (x-(t_{\mathbf{b}%
}(x,v)-\tau )v)\mathrm{d}\tau \\
& \geq (t_{\mathbf{b}}(x,v)-s)(v\cdot \nabla \xi (x-(t_{\mathbf{b}%
}(x,v)-s)v)) \\
& \geq \tilde{\sigma}_{2}|v\cdot \nabla \xi (x-\tilde{\sigma}_{2}v)|\ \ \
\text{for}\ \ s\in \lbrack t^{\ast },t_{\mathbf{b}}(x,v)-\tilde{\sigma}_{2}],
\end{split}
\notag
\end{equation}%
\begin{equation}
\begin{split}
-\xi (x-(t_{\mathbf{b}}(x,v)-s)v)& =-\xi (x_{\mathbf{b}}(x,v))-\int_{0}^{s}v%
\cdot \nabla \xi (x-(t_{\mathbf{b}}(x,v)-\tau )v)\mathrm{d}\tau \\
& \geq s|v\cdot \nabla \xi (x-(t_{\mathbf{b}}-s)v)| \\
& \geq \tilde{\sigma}_{1}|v\cdot \nabla \xi (x_{\mathbf{b}}(x,v)+\tilde{%
\sigma}_{1}v)|\ \ \ \text{for}\ \ s\in \lbrack 0,t^{\ast }].
\end{split}
\notag
\end{equation}%
%
%
%

Hence, for $s \in [\tilde{\sigma}_{1}, t_{\mathbf{b}}(x,v)-\tilde{\sigma}%
_{2}],$
\begin{equation}
\begin{split}
|\xi(x -(t_{\mathbf{b}} -s )v )| &\geq \min \Big\{ |\xi(x - \tilde{\sigma}%
_{2} v )|, | \xi( x_{\mathbf{b}}(x,v) + \tilde{\sigma}_{1} v ) |\Big\} \\
&\geq \min \Big\{ \tilde{\sigma}_{2} |v \cdot \nabla \xi(x -\tilde{\sigma}%
_{2} v)|, \tilde{\sigma}_{1} | v \cdot \nabla \xi(x_{\mathbf{b}} (x,v)+
\tilde{\sigma}_{1} v)|\Big\}.
\end{split}
\notag
\end{equation}
From the definition of $\tilde{\sigma}_{1}$ and $\tilde{\sigma}_{2}$ in (\ref%
{sigma_delta}) we have
\begin{equation}  \label{xi_lower1_tilde}\notag
|v |^{2} |\xi(x -(t_{\mathbf{b}}(x,v) -s) v)| \geq \tilde{\delta} \sqrt{%
\alpha(x , v )} \min \Big\{ |v \cdot \nabla \xi(x -\tilde{\sigma}_{2} v)|, |
v \cdot \nabla \xi(x_{\mathbf{b}} (x,v)+ \tilde{\sigma}_{1} v)|\Big\} .
\end{equation}
Without loss of generality we may assume $|v \cdot \nabla \xi(x -\tilde{%
\sigma}_{2} v)|=\min \big\{ 
|v \cdot \nabla \xi(x -\tilde{\sigma}_{2} v)|, 
| v \cdot \nabla \xi(x_{\mathbf{b}} (x,v)+ \tilde{\sigma}_{1} v)|\big\}.$
Then by the Velocity lemma we have $\sqrt{\alpha(x,v)}\gtrsim_{\xi}
|v||\xi(x-\tilde{\sigma}_{2} v)|^{1/2}$. Then we choose $s= t_{\mathbf{b}%
}(x,v) - \tilde{\sigma}_{2}$ to have $|v|^{2} |\xi(x-\tilde{\sigma}_{2}
v)|\geq \tilde{\delta} |v| |\xi(x-\tilde{\sigma}_{2} v)|^{1/2} \times
|v\cdot \nabla \xi(x-\tilde{\sigma}_{2}v)| $ and
\begin{equation}
\begin{split}
|v| |\xi(x-\tilde{\sigma}_{2} v)|^{1/2}\gtrsim \tilde{\delta} \times |v\cdot
\nabla \xi(x-\tilde{\sigma}_{2}v)|.
\end{split}
\notag
\end{equation}
The left hand side is the lower bound of $|v|^{2}|\xi(x-(t_{\mathbf{b}%
}(x,v)-s)v)|$ for $s\in [\tilde{\sigma}_{1}, t_{\mathbf{b}}(x,v)-\tilde{%
\sigma}_{2}]$ and the right hand side is bounded below by the Velocity
lemma: $e^{-\mathcal{C}|v|t_{\mathbf{b}}(x,v)} \alpha(x,v)\gtrsim_{\xi}
\alpha(x,v)$. Therefore we conclude (\ref{lower_delta}).

\vspace{8pt}

\noindent \textit{Step 4. \ } We prove (\ref{nonlocal}). From (\ref%
{int_xi})
\begin{equation}
\begin{split}
& \int_{0}^{t_{\mathbf{b}}(x,v)}\int_{\mathbb{R}^{3}}e^{-l\langle v\rangle
(t-s)}\frac{e^{-\theta |v-u|^{2}}}{|v-u|^{\kappa }\alpha (x-(t_{\mathbf{b}%
}(x,v)-s)v,u)}Z(s,v)\mathrm{d}u\mathrm{d}s \\
& \lesssim {\int_{0}^{t_{\mathbf{b}}(x,v)}e^{-l\langle v\rangle (t-s)}\frac{1%
}{|v|^{2\beta -1}|\xi |^{\beta -\frac{1}{2}}}Z(s,v)\mathrm{d}s}.
\end{split}
\notag
\end{equation}%
According to (\ref{sigma_delta}) we split the time integration as
\begin{equation}\notag
{\int_{0}^{t_{\mathbf{b}}(x,v)}e^{-l\langle v\rangle (t-s)}\frac{1}{%
|v|^{2\beta -1}|\xi |^{\beta -\frac{1}{2}}}Z(s,v)\mathrm{d}s}=\underbrace{%
\int_{0}^{\tilde{\sigma}_{1}}+\int_{t_{\mathbf{b}}(x,v)-\tilde{\sigma}%
_{2}}^{t_{\mathbf{b}}(x,v)}}_{\mathbf{(IV)}}+\underbrace{\int_{\tilde{\sigma}%
_{1}}^{t_{\mathbf{b}}(x,v)-\tilde{\sigma}_{2}}}_{\mathbf{(V)}}.
\end{equation}%
For the first two terms $(\mathbf{IV})$, we use the mapping of (\ref%
{COV_xi_s})
\begin{equation*}
s\in \lbrack 0,\tilde{\sigma}_{1}]\cup \lbrack t_{\mathbf{b}}(x,v)-\tilde{%
\sigma}_{2},t_{\mathbf{b}}(x,v)]\mapsto |\xi (x-(t_{\mathbf{b}%
}(x,v)-s)v)|\in \big[0,B\big),
\end{equation*}%
where the range of $|\xi |$ has been bounded in (\ref{max_xi}), and $B$ is
given by (\ref{tz}). By the change of variables of (\ref{COV_xi_s})
\begin{equation}
\begin{split}
\mathbf{(IV)}& \lesssim \ \sup_{0\leq s\leq t_{\mathbf{b}}(x,v)}\{e^{-l%
\langle v\rangle (t-s)}Z(s,v)\}\frac{1}{|v|^{2\beta -1}}\int_{0}^{C\tilde{%
\delta}\frac{\alpha (x,v)}{|v|^{2}}}\frac{1}{|\xi |^{\beta -1/2}}\frac{%
\mathrm{d}|\xi |}{\sqrt{\alpha (x,v)}} \\
& \lesssim \ \sup_{0\leq s\leq t_{\mathbf{b}}(x,v)}\{e^{-l\langle v\rangle
(t-s)}Z(s,v)\}\frac{1}{|v|^{2\beta -1}}\frac{1}{\sqrt{\alpha (x,v)}}\Big[%
|\xi |^{-\beta +\frac{3}{2}}\Big]_{|\xi |=0}^{|\xi |=B},
\end{split}
\notag
\end{equation}%
where we have used $\beta <\frac{3}{2}$. The lemma follows with $B$ given by
(\ref{tz}).

For $\mathbf{(V)}$ we use $\sqrt{\alpha(X_{\mathbf{cl}}(s))}\lesssim_{\xi,%
\tilde{\delta}} |v|\sqrt{-\xi(X_{\mathbf{cl}}(s))}$ for $s\in [ \tilde{\sigma%
}_{1}, t_{\mathbf{b}}(x,v)-\tilde{\sigma}_{2}]$, from (\ref{sigma}), to have
\begin{equation*}
\frac{1}{|v|^{2\beta -1} |\xi|^{\beta-\frac{1}{2}}} = \frac{1}{\big(|v|
\sqrt{-\xi} \ \big)^{2(\beta -\frac{1}{2})}} \lesssim \frac{1}{
[\alpha(x,v)]^{\beta-\frac{1}{2}}}.
\end{equation*}
Finally
\begin{equation*}
\mathbf{(V)} \lesssim \frac{1}{[\alpha(x,v)]^{\beta- 1/2}} \int_{0}^{t_{%
\mathbf{b}}(x,v)} e^{- l \langle v\rangle (t-s)} Z(s,v) \mathrm{d}s \lesssim
\frac{O(l^{-1}) }{\langle v\rangle[\alpha(x,v)]^{\beta- 1/2}} \sup_{0 \leq
s\leq t} \{ e^{-l\langle v\rangle (t-s)} Z(s,x,v)\} .
\end{equation*}

Now we assume $x\notin \partial \Omega .$ We find $\bar{x}\in \partial
\Omega $ and $\bar{t}_{\mathbf{b}}$ so that
\begin{equation*}
x-(t_{\mathbf{b}}(x,v)-s)v=\bar{x}-(\bar{t}_{\mathbf{b}} -s)v.
\end{equation*}%
Therefore, by the \textit{Step 1} and the fact $\bar{x}\in \partial \Omega $%
, we have
\begin{equation*}
\begin{split}
& \int_{0}^{\bar{t}_{\mathbf{b}}} \int_{\mathbb{R}^{3}} e^{-l \langle
v\rangle (t-s)} \frac{e^{-\theta |v-u|^{2}}}{|v-u|^{2-\kappa} \big[\alpha(%
\bar{x} -(\bar{t}_{\mathbf{b}} -s )v,u )\big]^{\beta}} Z(s,v) \mathrm{d}u
\mathrm{d}s \\
& \lesssim \ \ \ \int_{0}^{\bar{t}_{\mathbf{b}}} e^{-l \langle v\rangle
(t-s)} \frac{ e^{-C|v-u|^{2}}}{|v|^{2\beta -1} |\xi|^{\beta-\frac{1}{2}}}
Z(s,v)\mathrm{d}s.
\end{split}%
\end{equation*}%
We then deduce our lemma since $\alpha (\bar{x},v)\backsim \alpha (x,v)$ via
the Velocity Lemma with the fact $\bar{t}_{\mathbf{b}}|v|\lesssim_{\Omega}
1. $ \end{proof}

\begin{proof}[\textbf{Proof of (2) Lemma \protect\ref{lemma_nonlocal}}]
 It suffices to consider $r=0$ case. 
For the specular cycles and the bounce-back cycles it is important to control the \textit{number of bounces}, 
\[
\ell_{*}(s)= \ell_{*}(s;t,x,v)\in\mathbb{N} \ \ \ \text{if} \ \ \   t^{\ell_{*}+1} \leq s< t^{\ell_{*}}.
\]

Here we prove $t_{\mathbf{b}}(x,v) \simeq \alpha(x,v)^{1/2}/|v|^{2}.$ Recall (40) of \cite{Guo10} or (2.4) of \cite{EGKM}: if $\Omega$ is bounded and $\partial\Omega$(i.e. $\xi$) is smooth then for $(x,v) \in\gamma_{-}$
\begin{equation}\label{40}
t_{\mathbf{b}}(x,v)  \gtrsim_{\Omega} \frac{ \sqrt{\alpha(x,v)}}{|v|^{2}}.
\end{equation}
It suffices to prove $t_{\mathbf{b}}(x,v)  \gtrsim_{\Omega} \frac{ |n(x)\cdot v|}{|v|^{2}}.$ For $x\in\partial\Omega$ there exists $0 < \delta \ll 1$ such that
\[
\sup_{\substack{y \in\partial\Omega \\ |x-y|< \delta}} \frac{|(x-y)\cdot n(x)|}{|x-y|^{2}} \lesssim \max_{\substack{y \in\partial\Omega \\ |x-y|< \delta}} |\nabla^{2} \xi(x)|.
\]
If $|x-y|\geq \delta$ then $\frac{|(x-y)\cdot n(x)|}{|x-y|^{2}} \leq \delta^{-2} |(x-y)\cdot n(x)| \lesssim_{\delta,\Omega}1.$ By the compactness of $\Omega$ and $\partial\Omega$ we have $|(x-y)\cdot n(x)|\lesssim |x-y|^{2}$ for all $x,y \in\partial\Omega.$ Taking the inner product of $x-x_{\mathbf{b}}(x,v)=t_{\mathbf{b}}(x,v)v$ with $n(x)$ we have
\[
t_{\mathbf{b}}(x,v)| v\cdot n(x)|= |(x-x_{\mathbf{b}}(x,v))\cdot n(x)|\lesssim |x-x_{\mathbf{b} }(x,v)|^{2} = C_{\Omega} |v|^{2} |t_{\mathbf{b}}(x,v)|^{2},
\]
and this proves (\ref{40}).

If $\Omega$ is convex (\ref{convex}) then for $(x,v)\in\gamma_{-}$
\begin{equation}\label{41}
t_{\mathbf{b}}(x,v)\lesssim_{\xi} \frac{\sqrt{\alpha(x,v)}}{|v|^{2}}.
\end{equation}
It suffices to show $t_{\mathbf{b}}(x,v)\lesssim_{\xi} \frac{|n(x)\cdot v|}{|v|^{2}}.$ Since $\xi(x)=0=\xi(x-t_{\mathbf{b}}(x,v)v)$ for $(x,v)\in\gamma_{-}$, we have
\begin{equation}\notag
\begin{split}
0 &= \xi(x-t_{\mathbf{b}}(x,v)v) = \xi(x) + \int^{t_{\mathbf{b}}(x,v)}_{0} [-v\cdot \nabla_{x} \xi(x-sv)] \mathrm{d}s\\
&= [- v\cdot \nabla_{x}\xi(x)] t_{\mathbf{b}}(x,v) + \int^{t_{\mathbf{b}}(x,v)}_{0} \int^{s}_{0} \{v\cdot \nabla_{x}^{2} \xi(x-\tau v)\cdot v\}  \mathrm{d}\tau \mathrm{d}s.
\end{split}
\end{equation}
By the convexity of $\xi$ in (\ref{convex}) we have $[v\cdot \nabla_{x}\xi(x)] t_{\mathbf{b}}(x,v)   \ \geq \ \frac{(t_{\mathbf{b}}(x,v))^{2}}{2} C_{\xi} |v|^{2},$
and therefore this proves (\ref{41}).

An important consequence of Velocity lemma is that for the specular cycles
$$\alpha(X_{\mathbf{cl}}(s;t,x,v), V_{\mathbf{cl}}(s;t,x,v)  ) \gtrsim e^{-\mathcal{C}|v||t-s|} \alpha(x,v),$$
and therefore for the specular cycles
\begin{equation}\label{number_control}
\begin{split}
\ell_{*}(s;t,x,v) &\leq \frac{|t-s|}{ \min_{0 \leq \ell \leq \ell_{*}(s;t,x,v)}  |t^{\ell}-t^{\ell+1}| } \lesssim
 \frac{|t-s|}{ \min_{0 \leq \ell \leq \ell_{*}(s;t,x,v)}  \frac{ \sqrt{\alpha(x^{\ell}, v^{\ell})}}{|v^{\ell}|^{2}} } \\
 &\lesssim \frac{|t-s||v|^{2}}{\sqrt{\alpha(x,v)}}e^{\mathcal{C}|v|(t-s)}.
\end{split}
\end{equation}
Remark that for the bounce-back cycles we do not have the growth term $e^{\mathcal{C}|v|(t-s)}$. This is because of the fact $\alpha(X_{\mathbf{cl}}(s),V_{\mathbf{cl}}(s))$ is either $\alpha(x^{1},v^{1})$ or $\alpha(x^{2},v^{2}),$ and the fact $|t-t^{2}| \leq 2|t^{1}-t^{2}| \lesssim \frac{C_{\Omega}}{|v|}$ for the bounded domain.

We consider the specular BC case first. For fixed $(x,v)$ we use the following notation $\alpha(s):= \alpha(s;t,x,v) := \alpha(X_{\mathbf{cl}}(s;t,x,v), V_{\mathbf{cl%
}}(s;t,x,v)).$

Firstly we consider the estimate (\ref{specular_nonlocal}) for $|v| <
\delta. $ Using (\ref{number_control}),
\begin{equation}
\begin{split}
&\mathbf{1}_{\big\{ |v| \leq \delta \big\}}\int_{0}^{t} \int_{\mathbb{R}%
^{3}} e^{-l \langle v\rangle (t-s)} \frac{e^{-\theta |V_{\mathbf{cl}} (s)
-u|^{2}}}{|V_{\mathbf{cl}} (s) -u|^{2-\kappa}} \frac{Z(s,x,v)}{\big[ %
\alpha(X_{\mathbf{cl}}(s;t,x,v),u) \big]^{\beta}} \mathrm{d}u \mathrm{d}s \\
&\lesssim \sum_{\ell=0}^{\ell_{*}(0;t,x,v)} \int^{t^{\ell}}_{t^{\ell+1}}
\int_{\mathbb{R}^{3}} e^{-l \langle v\rangle (t-s)} \frac{e^{-\theta |
v^{\ell}-u |^{2}}}{|v^{\ell} -u|^{2-\kappa}} \frac{Z(s,x,v)}{\big[ %
\alpha(x^{\ell} -(t^{\ell} -s) v^{\ell},u) \big]^{\beta}} \mathrm{d}u
\mathrm{d}s \\
&\lesssim \frac{t|v|^{2} e^{\mathcal{C}t\delta}}{\big[ \alpha(x,v)\big]%
^{1/2}} \sup_{\ell} \Big\{ \frac{O(\tilde{\delta})e^{ \mathcal{C}t\delta} }{%
|v|^{2} [\alpha(x,v)]^{ \beta-1} } \sup_{t^{\ell+1} \leq s \leq t^{\ell}} \{
e^{-l \langle v\rangle(t-s) } Z(s,x,v)\} \\
& \ \ \ \ \ \ \ \ \ \ \ \ \ \ \ \ \ \ \ \ \ \ \ \ \ \ \ +\frac{C_{\tilde{%
\delta}}e^{ \mathcal{C}t\delta} }{[\alpha(x,v)]^{\beta- 1/2}} \int_{
t^{\ell+1}}^{t^{\ell}} e^{-l \langle v\rangle (t-s)} Z(s,x,v) \mathrm{d}s %
\Big\},
\end{split}
\notag
\end{equation}
where we have used (\ref{nonlocal}). By (\ref{40}) and (\ref{41}) and the Velocity lemma (Lemma \ref%
{velocity_lemma}) we have $|t^{\ell} -t^{\ell+1}| \lesssim_{\xi} \frac{\sqrt{%
\alpha(x,v)}}{|v|^{2}} e^{ \mathcal{C}t|v|}\lesssim_{\xi,\delta} \frac{%
\sqrt{\alpha(x,v)}}{|v|^{2}} e^ {\mathcal{C}t \delta}$ and hence we deduce (%
\ref{specular_nonlocal}) for $|v|< \delta$ by
\begin{equation*}
\mathbf{1}_{\big\{|v| < \delta \big\}}\int \cdots \ \lesssim_{\xi} \ \frac{O(%
\tilde{\delta} + l^{-1})te^{2\mathcal{C}t\delta}}{\big[\alpha(x,v)\big]%
^{\beta-\frac{1}{2}}} \sup_{ 0 \leq s\leq t} \big\{ e^{-l \langle v \rangle
(t-s)} Z(s,x,v)\big\}.
\end{equation*}

Now we consider $|v|\geq \delta.$ We split the time interval as
\begin{equation}  \label{time_splitting}
[0,t] \ = \ [ t- \frac{1}{|v|}, t] \ \cup \ \bigcup_{j=1}^{[t|v|]+1} [t-(j+1)%
\frac{1}{|v|}, t- j \frac{1}{|v|}].
\end{equation}

Consider the first time section $[ t- \frac{1}{|v|}, t]$. Due to (\ref{number_control}) we bound
\begin{equation*}
\sup_{ s\in [ t- \frac{1}{|v|}, t] }\ell_{*}(s;t,x,v) \lesssim_{\xi} \frac{%
\frac{1}{|v|} |v|^{2} e^{\mathcal{C}\frac{1}{|v|}|v| }}{[\alpha(x,v)]^{1/2}}
\lesssim \frac{|v| e^{\mathcal{C}}}{[\alpha(x,v)]^{1/2}},
\end{equation*}
and for $s\in [ t- \frac{1}{|v|}, t], \ $ $e^{-\mathcal{C}} \alpha(x,v) \lesssim \alpha(X_{\mathbf{cl}}(s;t,x,v),V_{%
\mathbf{cl}}(s;t,x,v)) \lesssim e^{\mathcal{C} } \alpha(x,v),$
and $
|t^{\ell}-t^{\ell+1}| \lesssim \frac{[\alpha(x,v)]^{1/2} e^{\mathcal{C}}}{%
|v|^{2}}$ due to the Velocity lemma.
Then we use (\ref{nonlocal}) to have
\begin{equation}
\begin{split}
& \int_{t-1/|v|}^{t} \int_{\mathbb{R}^{3}} e^{-l \langle v\rangle (t-s)}
\frac{e^{-\theta |V_{\mathbf{cl}} (s) -u|^{2}}}{|V_{\mathbf{cl}} (s)
-u|^{2-\kappa}} \frac{Z(s,x,v)}{\big[ \alpha(X_{\mathbf{cl}}(s;t,x,v),u) %
\big]^{\beta}} \mathrm{d}u \mathrm{d}s \\
&\lesssim \sum_{\ell=0}^{\ell_{*}(0;t,x,v)} \int^{t^{\ell}}_{t^{\ell+1}}
\int_{\mathbb{R}^{3}} e^{-l \langle v\rangle (t-s)} \frac{e^{-\theta |
v^{\ell}-u |^{2}}}{|v^{\ell} -u|^{2-\kappa}} \frac{Z(s,x,v)}{\big[ %
\alpha(x^{\ell} -(t^{\ell} -s) v^{\ell},u) \big]^{\beta}} \mathrm{d}u
\mathrm{d}s \\
&\lesssim \frac{C_{\xi}|v|}{[\alpha(x,v)]^{1/2}} \sup_{\ell} \frac{O(\tilde{%
\delta})e^{C_{\xi}} }{|v|^{2} [\alpha(x,v)]^{ \beta-1} } \sup_{t^{\ell+1}
\leq s \leq t^{\ell}} \{ e^{-l \langle v\rangle(t-s) } Z(s,x,v)\} \\
& \ \ + \sum_{\ell =0}^{\ell_{*}(0,t,x,v)} \frac{C_{\tilde{\delta}%
}e^{C_{\xi}} }{[\alpha(x,v)]^{\beta- 1/2}} \int_{ t^{\ell+1}}^{t^{\ell}}
e^{-l \langle v\rangle (t-s)} Z(s,x,v) \mathrm{d}s \\
& \lesssim \frac{O(\tilde{\delta})}{ |v| [\alpha(x,v)]^{\beta-1/2}} \sup_{0
\leq s\leq t} \{ e^{-Cl \langle v\rangle (t-s)} Z(s,x,v)\} + \frac{C_{\tilde{%
\delta}, \xi}}{[\alpha(x,v)]^{\beta-1/2}} \int^{t}_{0} e^{-Cl \langle
v\rangle (t-s)} Z(s,x,v) \mathrm{d}s \\
& \lesssim \Big(\frac{O(\tilde{\delta})}{ |v| [\alpha(x,v)]^{\beta-1/2}} +
\frac{C_{\tilde{\delta},\xi }}{ l \langle v\rangle [\alpha(x,v)]^{\beta-1/2}}
\Big) \sup_{0 \leq s\leq t} \{ e^{-Cl \langle v\rangle (t-s)} Z(s,x,v)\}.
\end{split}
\notag
\end{equation}

Now we consider time sections $[t-(j+1)\frac{1}{|v|}, t- j \frac{1}{|v|}]$
for $j\geq 1.$ Assume that
\begin{equation*}
[t-(j+1) \frac{1}{|v|}, t-j \frac{1}{|v|}] \subset [t^{\ell_{j+1}-1},
t^{\ell_{j+1}}] \cup \cdots \cup [t^{\ell_{j}+1}, t^{\ell_{j}} ],
\end{equation*}
and $[t-(j+1) \frac{1}{|v|}, t-j \frac{1}{|v|}] \cap [t^{\ell_{j+1}-2},
t^{\ell_{j+1}-1}]= \emptyset$ and $[t-(j+1) \frac{1}{|v|}, t-j \frac{1}{|v|}%
] \cap [t^{\ell_{j } }, t^{\ell_{j }-1}]= \emptyset$.

Note that for all $s \in[t-(j+1)\frac{1}{|v|}, t- j \frac{1}{|v|}] $
\begin{equation*}
e^{-C_{\xi } j} \alpha(t) \lesssim \alpha(s) \lesssim e^{C_{\xi } j}
\alpha(t),
\end{equation*}
and
\begin{equation*}
\ell_{j+1} - \ell_{j} \lesssim \frac{(j+1) \frac{1}{|v|} - j \frac{1}{|v|}}{%
\frac{\sqrt{\alpha( t-j \frac{1}{|v|} )}}{|v|^{2}}} \lesssim \frac{|v|}{%
\sqrt{\alpha(t)}} e^{C_{\xi} j},
\end{equation*}
and for $\ell \in [\ell_{j+1}-1, \ell_{j} ]$
\begin{equation*}
|t^{\ell}-t^{\ell+1}| \lesssim \frac{\sqrt{\alpha( t-j \frac{1}{|v|} )}}{%
|v|^{2}} \lesssim \frac{\sqrt{\alpha(t)}}{|v|^{2}} e^{C_{\xi} j}.
\end{equation*}
From (\ref{nonlocal}), for all $\ell \in [\ell_{j+1}-1, \ell_{j} ]$
\begin{equation}
\begin{split}
& \int^{t^{\ell+1}}_{t^{\ell}} \int_{\mathbb{R}^{3}} e^{-l \langle v\rangle
(t-s)} \frac{e^{-\theta |V_{\mathbf{cl}} (s) -u|^{2}}}{|V_{\mathbf{cl}} (s)
-u|^{2-\kappa}} \frac{Z(s,x,v)}{\big[ \alpha(X_{\mathbf{cl}}(s;t,x,v),u) %
\big]^{\beta}} \mathrm{d}u \mathrm{d}s \\
&\lesssim \frac{ \tilde{\delta}}{|v|^{2}\alpha(t-j/|v|)^{\beta-1 }}
\sup_{[t^{\ell},t^{\ell+1}]} \{ e^{-l \langle v\rangle (t-s)} Z \} + \frac{C_{\tilde{\delta}}
}{\alpha(t-j/|v|)^{\beta-1/2} } \int^{t^{\ell}}_{t^{\ell+1}} e^{-l \langle
v\rangle (t-s)}Z,
\end{split}
\notag
\end{equation}
is bounded by
\begin{equation}
\begin{split}
&\frac{\tilde{\delta} e^{C_{\xi}j}}{ |v|^{2} \alpha(t)^{\beta-1}} e^{-\frac{l%
}{2} j } \sup_{[t^{\ell},t^{\ell+1}]} \{ e^{-\frac{l}{2} \langle v\rangle
(t-s)} Z \} + \frac{e^{ C_{\xi} j}}{\alpha(t)^{\beta-1/2}} e^{-\frac{l}{2}j}
\int^{t^{\ell}}_{t^{\ell+1}} e^{-\frac{l}{2} \langle v\rangle (t-s)}Z \\
&\lesssim \frac{e^{-C l j} }{ |v|^{2} \alpha(t)^{\beta-1}} \sup_{0 \leq s
\leq t} \{ e^{-\frac{l}{2} \langle v\rangle (t-s)} Z(s) \} ,
\end{split}
\notag
\end{equation}
where we have used the fact $t-s \geq j \frac{1}{|v|}$ for $s \in[t-(j+1)%
\frac{1}{|v|}, t- j \frac{1}{|v|}]$.

Therefore
\begin{equation}
\begin{split}
& \int^{ t-j \frac{1}{|v|} }_{ t-(j+1) \frac{1}{|v|} } \int_{\mathbb{R}^{3}}
e^{-l \langle v\rangle (t-s)} \frac{e^{-\theta |V_{\mathbf{cl}} (s) -u|^{2}}%
}{|V_{\mathbf{cl}} (s) -u|^{2-\kappa}} \frac{Z(s,x,v)}{\big[ \alpha(X_{%
\mathbf{cl}}(s;t,x,v),u) \big]^{\beta}} \mathrm{d}u \mathrm{d}s \\
&\lesssim |\ell_{j+1}-\ell_{j}| \sup_{\ell_{j+1} \leq \ell \leq \ell_{j} }
\int^{t^{\ell+1}}_{t^{\ell}} \cdots \\
& \lesssim \frac{|v|}{\sqrt{\alpha(t)}} e^{C_{\xi} j} \times \frac{e^{- lj /4}}{%
|v|^{2} \alpha(t)^{\beta-1}} \sup_{ 0 \leq s \leq t} \{ e^{-\frac{l}{2}
\langle v\rangle (t-s)} Z(s,x,v) \} \\
& \lesssim \frac{e^{-   lj/4}}{|v| \alpha(t)^{\beta-1/2}} \sup_{ 0
\leq s \leq t} \{ e^{-\frac{l}{2} \langle v\rangle (t-s)} Z(s,x,v) \}.
\end{split}
\notag
\end{equation}

Now we sum up all contributions of $[t-(j+1)\frac{1}{|v|}, t- j \frac{1}{|v|}%
]$ for $j\geq 1:$
\begin{equation}
\begin{split}
\sum_{j=1}^{t|v|} \int^{t-j/|v|}_{t-(j+1)/|v|}& \leq \sum_{j=1}^{t|v|} \frac{%
e^{-  lj /8 }}{|v| \alpha(t)^{\beta-1/2}} \sup_{ 0 \leq s \leq t} \{
e^{-\frac{l}{2} \langle v\rangle (t-s)} Z(s,x,v) \} \\
& \lesssim \frac{ e^{- l/8}}{|v| [\alpha(x,v)]^{\beta-1/2}} \sup_{ 0
\leq s \leq t} \{ e^{-\frac{l}{2} \langle v\rangle (t-s)} Z(s,x,v) \}.
\end{split}
\notag
\end{equation}
where we used $\sum_{j=1}^{t|v|} e^{-  lj/8}= e^{-  l/16}
\sum_{j=2}^{t|v|} e^{-C^{\prime} lj} \leq e^{-  l/16}. $

These prove (\ref{specular_nonlocal}). For the bounce-back case we set $%
\mathcal{C}=0$ and we have same conclusion.
\end{proof}

\begin{lemma}\label{lemma_nonlocal_u}  Let $(t,x,v)\in \lbrack 0,\infty )\times \bar{\Omega}%
\times \mathbb{R}^{3}$ and $Z\geq 0.$

(1) For $0< \varepsilon \ll 1$ and $\frac{1}{2}<\beta <1$ and $0<\kappa \leq 1$, we have
\begin{equation}
\begin{split}
& \int_{0}^{t_{\mathbf{b}}(x,v)}\int_{\mathbb{R}^{3}}e^{-l\langle v\rangle
(t-s)}\frac{e^{-\theta |v-u|^{2}}}{|v-u|^{2-\kappa }[\alpha (x-(t_{\mathbf{b}%
}(x,v)-s)v, {u})]^{\beta }}\frac{|v|}{| {u}|}\frac{\langle u\rangle
^{r}}{\langle v\rangle ^{r}}Z(s,x,v)\mathrm{d}u\mathrm{d}s \\
& \lesssim _{\theta ,r}\ \ \frac{O(\varepsilon)}{|v|^{2}[\alpha
(x,v)]^{\beta -1}}\sup_{s\in \lbrack 0,t_{\mathbf{b}}(x,v)]}\{e^{-l\langle
v\rangle (t-s)}Z(s,x,v)\} \\
& \ \ \ \ \ \ \ \ +\frac{C_{ \varepsilon} }{[\alpha (x,v)]^{\beta -1/2}}%
\int_{0}^{t_{\mathbf{b}}(x,v)}e^{-Cl\langle v\rangle (t-s)}Z(s,x,v)\mathrm{d}%
s.
\end{split}
\label{nonlocal_u}
\end{equation}


(2) Let $[X_{\mathbf{cl}}(s;t,x,v), V_{\mathbf{cl%
}}(s;t,x,v)]$ be either the specular backward trajectory or the bounce-back backward trajectory in Definition \ref{cycles}. For $0 < \varepsilon \ll 1$ and $\frac{1}{2}< \beta < 1$ and $0< \kappa\leq 1$ and $r\in\mathbb{R}$, there exists $l \gg_{\xi} 1$ and $C=C_{l,\beta, \xi, r}>0$ such that
\begin{equation}  \label{specular_nonlocal_u}
\begin{split}
& \int_{0}^{t} \int_{\mathbb{R}^{3}} e^{- l\langle v\rangle (t-s)} \frac{%
e^{-\theta|V_{\mathbf{cl}}(s)-u|^{2}}}{|V_{\mathbf{cl}}(s)-u|^{2-\kappa}}
\frac{|v|}{|u|}\frac{\langle u\rangle^{r}}{\langle v\rangle^{r}} \frac{%
Z(s,x,v)}{ \big[ \alpha(X_{\mathbf{cl}}(s;t,x,v),u) \big]^{\beta}} \mathrm{d}%
u\mathrm{d}s \\
& \lesssim_{\xi,r} \frac{O_{\beta, \kappa, r}(\varepsilon) }{\langle v\rangle \big[\alpha(x,v)\big]^{\beta-1/2} }
\sup_{0 \leq s\leq t} \big\{ e^{- C_{l, \beta, \xi,r} \langle v\rangle
(t-s)} Z(s,x,v) \big\}.
\end{split}%
\end{equation}
\end{lemma}

Now we prove a variant of (1) of Lemma \ref{lemma_nonlocal} with extra $\frac{|v|}{|u|}.$

\begin{proof}[\textbf{Proof of Lemma \ref{lemma_nonlocal_u}}]
 We prove (\ref{nonlocal_u}). Due to \textit{Step 2}
and \textit{Step 3} in the proof of \textit{(1)} of Lemma \ref{lemma_nonlocal}, it suffices to show
\begin{equation}  \label{int_xi_u}
\int_{\mathbb{R}^{3}} \frac{|v| e^{-\theta |v-u|^{2}} \mathrm{d}u }{%
|v-u|^{2-\kappa} |u| \big[\alpha(x-(t_{\mathbf{b}}(x,v)-s)v,u)\big]^{\beta}}
\lesssim \frac{1}{|v|^{2\beta-1} |\xi(x-(t_{\mathbf{b}}(x,v)-s) v)|^{\beta-%
\frac{1}{2}}}.
\end{equation}
As \textit{Step 1} in the proof of \textit{(1)}, for fixed $s$ and $x-(t_{%
\mathbf{b}}(x,v)-s)v$, we decompose $u=u_{\tau,1}\tau_{1} +
u_{\tau,2}\tau_{2} + u_{n} n$ where $\{ \tau _{1},\tau _{2},n\}$ is the
orthonormal basis that we chose in the proof of \textit{(1)}.

Now we split as
\begin{equation}
\begin{split}
&\int_{\mathbb{R}^{3}} \frac{|v| e^{-\theta |v-u|^{2}} \mathrm{d}u }{%
|v-u|^{2-\kappa} |u| \big[\alpha(x-(t_{\mathbf{b}}(x,v)-s)v,u)\big]^{\beta}}\\
&\lesssim_{\xi}  \int_{\mathbb{R}^{2}} \int_{\mathbb{R}}\frac{|v|
e^{-\theta|v-u|^{2}}\mathrm{d}u_{n}\mathrm{d}u_{\tau }}{|v-u|^{2-\kappa }|u|%
\big\{ |u_{n}|^{2}+|\xi (X_{\mathbf{cl}}(s))||u|^{2}\big\}^{\beta }} \\
 &= \int_{|u|\geq \frac{|v|}{5}}+\int_{|u|\leq \frac{|v|}{5}}.
\end{split}
\notag
\end{equation}%
For the first term, we have $\frac{|v| }{|u|}\leq 5 $ so we reduce it to the
previous case (\ref{int_xi})
\begin{equation*}
\int_{|u|\geq \frac{|v|}{5}}\lesssim \int_{\mathbb{R}^{3}}\frac{
e^{-\theta|v-u|^{2}}\mathrm{d}u_{n}\mathrm{d}u_{\tau }}{|v-u|^{2-\kappa }%
\big\{|u_{n}|^{2}+|\xi (X_{\mathbf{cl}}(s))||u|^{2}\big\}^{\beta }},
\end{equation*}%
which is bounded by $\frac{1}{|v|^{2\beta -1}|\xi |^{\beta -1/2}}.$

Now we consider the case of $|u|\leq \frac{|v|}{5}$. For fixed $0<
\kappa\leq 1$
\begin{equation*}
\frac{|v|}{|v-u|^{2-\kappa }} \ \lesssim \ \frac{|v| }{|v|^{2-\kappa }} \
\lesssim \ {|v|^{-1+ \kappa}},
\end{equation*}
and we have, from $|v-u|^{2} = \frac{|v-u|^{2}}{2} + \frac{|v-u|^{2}}{2}
\geq \frac{4^{2}}{2\cdot 5^{2}}|v|^{2} + \frac{4^{2}}{2}|u|^{2},$
\begin{equation*}
e^{-\theta|v-u|^{2}} \leq e^{-C_{\theta}|v|^{2}} e^{-C_{\theta}|u|^{2}}.
\end{equation*}

We split $\int_{\mathbb{R}^{3}}\mathrm{d}u=\int_{|u_{n}|\geq |\xi
|^{1/2}|u_{\tau }|}+\int_{|u_{n}|\leq |\xi |^{1/2}|u_{\tau }|}$ to have
(Note $\frac{1}{2}< \beta <1 $)
\begin{equation}
\begin{split}
\int_{|u_{n}|\geq |\xi |^{1/2}|u_{\tau }|}& \lesssim \frac{e^{-C|v|^{2}}}{
|v|^{1-\kappa} } \int_{\mathbb{R}^{2}}\frac{e^{-C_{\theta}|u_{\tau }|^{2}}}{%
|u_{\tau }|} \int^{|v|/5}_{ |\xi |^{1/2}|u_{\tau }|}|u_{n}|^{-2\beta
}e^{-C_{\theta}|u_{n}|^{2}}\mathrm{d}|u_{n}|\mathrm{d}u_{\tau } \\
& \lesssim \frac{e^{-C|v|^{2}}}{ |v|^{1-\kappa} } \int_{|u_{\tau}|\leq \frac{%
|v|}{5}}\frac{e^{-C_{\theta}|u_{\tau }|^{2}}}{|u_{\tau }|} \int_{|\xi
|^{1/2}|u_{\tau }|}^{\frac{|v|}{5}}\frac{\mathrm{d}u_{n}}{|u_{n}|^{2\beta }}%
 \mathrm{d}u_{\tau } \\
& \lesssim \frac{e^{-C|v|^{2}}}{ |v|^{1-\kappa} } \Big\{ |v|
+|v|^{-2\beta+1} \int_{|u_{\tau}| \leq \frac{|v|}{5}} \frac{ \mathrm{d}%
u_{\tau}}{|u_{\tau}|} + \frac{1}{|\xi|^{\beta -\frac{1}{2}}}
\int_{|u_{\tau}| \leq \frac{|v|}{5}} |u_{\tau}|^{-2\beta}\mathrm{d}u_{\tau} %
\Big\} \\
& \lesssim e^{-C|v|^{2}} |v|^{\kappa}\big\{ 1+ \frac{1}{|v|^{2\beta -1}}(1+
\frac{1}{|\xi|^{ \beta-\frac{1}{2}}}) \big\} \\
& \lesssim \frac{e^{-C|v|^{2}}}{|v|^{2\beta -1} |\xi|^{\beta-\frac{1}{2}}} ,
\\
\int_{|u_{n}|\leq |\xi |^{1/2}|u_{\tau }|}& \lesssim \frac{e^{-C|v|^{2}}}{
|v|^{1-\kappa} } \int_{|u_{\tau}| \leq \frac{|v|}{5} }\frac{e^{-
C_{\theta}|u_{\tau }|^{2}}}{|\xi |^{\beta }|u_{\tau }|^{2\beta }|u_{\tau }|}%
\int_{|u_{n}|\leq |\xi |^{1/2}|u_{\tau }|}\mathrm{d}u_{n}\mathrm{d}u_{\tau }
\\
& \lesssim \frac{e^{-C|v|^{2}}}{ |v|^{1-\kappa} } \int_{ |u_{\tau}| \leq
\frac{|v|}{5}}\frac{e^{-C_{\theta}|u_{\tau }|^{2}}}{|\xi |^{\beta
-1/2}|u_{\tau }|^{2\beta }}\mathrm{d}u_{\tau } \\
&\lesssim \frac{ |v|^{-2\beta +\kappa+1} e^{-C|v|^{2}}}{|\xi |^{\beta -1/2}}
\lesssim \frac{e^{-C|v|^{2}}}{|v|^{2\beta -1} |\xi|^{\beta-\frac{1}{2}}} .
\end{split}
\notag
\end{equation}%
Therefore, combining the cases of $|u|\leq \frac{|v|}{5}$ and $|u|\geq \frac{%
|v|}{5},$ we conclude (\ref{int_xi_u}).

The proof of (\ref{specular_nonlocal_u}), (2) of Lemma \ref{lemma_nonlocal_u} is a
direct consequence of (\ref{int_xi_u}) and the proof of (\ref%
{specular_nonlocal}).
\end{proof}

\section{\large{Diffuse Reflection BC}}

 \vspace{4pt}


\noindent
{   \large{{{4.1 $ \  {W^{1,p} (1< p< 2)}$ Estimate   }}}} 
 \vspace{6pt}

Consider the iteration (\ref{positive_iteration}) with (\ref{diffuse_BC_m}) and with $f^{0}\equiv f_{0}$, and with the compatibility condition for the
initial datum (\ref{compatibility_condition_1}). Remark that the normalized
Maxwellian is $\mu (v)=e^{-\frac{|v|^{2}}{2}}$. From Lemma \ref{local_existence}, we
have a uniform bound (\ref{bounded}) for $0< T \ll 1$. We apply Proposition 1 for $m=1,2,...$ with 
\begin{equation}\notag
\nu= \nu(\sqrt{\mu}f^{m})\geq 0 ,\ \  H= \Gamma_{\mathrm{gain}} (f^{m},f^{m}),\ \ g=c_{\mu }\sqrt{\mu (v)}\int_{n\cdot
u>0}f^{m}(t,x,u)\sqrt{\mu (u)}\{n(x)\cdot
u\}\text{d}u.  \label{hgm}
\end{equation}
For $\partial_{\mathbf{e}}=[ \partial_{x},\partial _{v}]$, $\partial
f^{m}$ satisfies
\begin{equation}
\{\partial _{t}+v\cdot \nabla _{x}+\nu (\sqrt{\mu}f^{m})\}\partial f^{m+1}=\mathcal{G}%
^{m},\ \ \ \partial f^{m+1}(0,x,v)=\partial f_{0}(x,v),  \label{seq_w_p}
\end{equation}
where
\begin{equation}\label{G_m}
\begin{split}
 \mathcal{G}^{m}& \ = \ -[\partial v]\cdot \nabla _{x}f^{m+1}-\partial [\nu
(\sqrt{\mu}f^{m}) ]f^{m+1}+\partial \lbrack  \Gamma_{\mathrm{gain}} (f^{m},f^{m})],    \\
 |\mathcal{G}^{m}|& \ \lesssim \ 
 | \nabla_{x} f^{m+1}| +  e^{-\frac{\theta}{2}|v|^{2}} || e^{\theta|v|^{2}} f_{0} ||_{\infty}^{2} + P( ||e^{\theta|v|^{2}} f_{0} ||_{\infty}) \int_{\mathbb{R}^{3}} \frac{e^{-C_{\theta} |v-u|^{2}  }}{|v-u|^{2-\kappa}} |\partial f^{m} (u)| \mathrm{d}u
 ,  
 \end{split}
\end{equation}%
where we have used \textit{(iv)} of Lemma \ref{lemma_operator} and (\ref{bounded}) of Lemma \ref{local_existence}.
%

On $(x,v)\in \gamma _{-}$, from (\ref{hgm}), (\ref{gboundary}) and (\ref{boundary_tau}), the
boundary condition is bounded by
\begin{equation} \label{seq_w_p_boundary}
\begin{split}
|\partial f^{m+1}(t,x,v)| \lesssim & \  \sqrt{\mu (v)}\left( 1+\frac{%
\langle v\rangle }{|n(x)\cdot v|}\right) \int_{n(x)\cdot u>0}|\partial f^{m}(t,x,u)|\langle u\rangle \sqrt{\mu }%
\{n(x)\cdot u\}\mathrm{d}u  \\
& +\frac{1}{|n(x)\cdot v|}\Big\{\nu (\sqrt{\mu}f^{m})|f^{m+1}| +|\Gamma_{\mathrm{gain}}
(f^{m},f^{m})|\Big\}  \\
\lesssim & \  \sqrt{\mu (v)}\left( 1+\frac{\langle v\rangle }{|n(x)\cdot
v|}\right) \int_{n(x)\cdot u>0}|\partial f^{m}(t,x,u)|\mu ^{1/4}\{n(x)\cdot u\}\mathrm{d}u  \\
&   +   \frac{   e^{-\frac{\theta}{2} |v|^{2}  }}{|n(x)\cdot v|}   P(||   e^{\theta|v|^{2}}  f_{0}||_{\infty }  ).  
\end{split}
\end{equation}%
 Set $\partial f^0 = [\partial_t f^0, \nabla_x f^0, \nabla_v f^0]=[0,0,0]$.

Now we are ready to prove Theorem \ref{Global_p}:

\begin{proof}[\textbf{Proof of Thoerem \ref{Global_p}}]
We claim that  for $1\leq p<2$, if $0<T \ll 1$ (therefore, (\ref{bounded}) and (\ref{bounded_t}) for $0  < \theta < \frac{1}{4}$ from Lemma \ref{local_existence}), and the
compatibility condition (\ref{compatibility_condition_1}) then uniformly-in-$%
m$,%
\begin{equation}
\sup_{0\leq t\leq T_{\ast }}||\partial f^{m}||_{p}^{p}+\int_{0}^{T_{\ast
}}|\partial f^{m}|_{\gamma ,p}^{p} \lesssim _{\Omega ,T_{\ast }}||\partial
f_{0}||_{p}^{p}+P(|| e^{\theta|v|^{2}}f_{0}||_{\infty }),
\label{global_bounded}
\end{equation}
 for some polynomial $P.$

Recall that the time derivative of the initial datum is defined as $%
\partial_t f_0\equiv -v\cdot\nabla_x f_0 -\nu(\sqrt{\mu} f_{0}) f_{0}+\Gamma_{\mathrm{gain}}(f_0,f_0)$. We remark that the sequence (\ref{positive_iteration}) is the one used in Lemma \ref{local_existence} and shown to be Cauchy in $L^\infty$. Therefore the limit function $f$ is a solution of the Boltzmann equation with the diffuse boundary condition. On the other hand, due to the weak lower semi-continuity for $L^p$ in the case of $p>1$, once we
have (\ref{global_bounded}) then we pass a limit $\partial f^m \rightharpoonup
\partial f$ weakly in $\sup_{t\in [0,T_*]}||\cdot ||_p^p$ and $\partial f^m|_{\gamma}
  \rightharpoonup \partial f|_{\gamma}$ in $\int_0^{T_*}|\cdot|^p_{\gamma,p}$ to conclude that $\partial f$ satisfies the same estimate of (\ref{global_bounded}).
 Repeat the same
procedure for $[T_*,2T_*], [2T_*,3T_*],\cdots,$ to conclude Theorem \ref%
{Global_p}.

We prove the claim (\ref{global_bounded}) by induction. From Proposition \ref{theo:trace}, $\partial f^1$ exists.
Because of our choice $\partial f^0$ the estimate (\ref{global_bounded}) is
valid for $m=1$. Now assume that $ \partial f^i$ exists and (\ref{%
global_bounded}) is valid for all $i=1,2,\cdots,m $. Applying Proposition
\ref{theo:trace} to show that $\partial f^{m+1}$ exists and to get $(\ref{global_t}), (\ref{global_x}), \text{and }(\ref{global_v})$, we have
\begin{equation}\label{abstract_green_n}
\begin{split}
&
\sup_{0 \leq s\leq t}|| \partial f^{m+1}(s)||_{p}^p + \int_0^t |  \partial
f^{m+1}|^p_{\gamma_+,p}   \\
 \lesssim& \ \  ||   \partial f_0||_{p}^p + \int_0^t
|  \partial f^{m+1}|^p_{\gamma_-,p}  + \int_{0}^{t} \iint_{\Omega\times \mathbb{R}^{3}} |\mathcal{G}^{m}| |\partial f^{m+1}|^{p-1} \\
\lesssim & \ \ 
||   \partial f_0||_{p}^p + \int_0^t
|  \partial f^{m+1}|^p_{\gamma_-,p}  + P(|| e^{\theta|v|^{2}} f_{0} ||_{\infty}  )\Big\{  \int_{0}^{t} || \partial f^{m+1} (s)||_{p}^{p} +   \int_{0}^{t} || \partial f^{m} (s)||_{p}^{p}\Big\},
\end{%
split}
\end{equation}
where we have used (\ref{G_m}) and Lemma \ref{lemma_operator} and H\"older inequality.
%
%
%
%
%

Now we consider the boundary contributions. We use (\ref{seq_w_p_boundary}) to obtain

\begin{equation}\nonumber
\begin{split}
 \int_0^t
\int_{\gamma_-} |  \partial f^{m +1}(s)| ^p 
\lesssim_p&
\sup_{x\in\partial\Omega}
\left(\int_{\gamma_-} \sqrt{\mu(v)}^p
\left(|n\cdot v| + \frac{\langle v\rangle^p}{|n\cdot v|^{p-1}}\right)\mathrm{d}
v\right)  \\
&  \ \ \ \ \  \times \int_0^t \int_{\partial\Omega} \left[ \int_{u\cdot n(x)>0} |\partial f^m(s,x,u)|\mu^{1/4}(u)\{n\cdot u\}\mathrm{d} u\right]^p\mathrm{d} S_x\mathrm{d} s  \notag \\
&  \ +\sup_{x\in\partial\Omega}
\left(\int_{\gamma_-} \langle
v\rangle^{-p \beta  }|n\cdot v|^{1-p}\mathrm{d} v\right)\times t || e^{\theta|v|^{2}} f_0||_\infty^p\notag\\
 \lesssim_p  & \int_0^t \int_{\partial\Omega} \left[ \int_{u\cdot n(x)>0} |\partial f^m(s,x,u)|\mu^{1/4}(u)\{n\cdot u\}\mathrm{d} u\right]^p\mathrm{d} S_x\mathrm{d} s   +  t  P(|| e^{\theta|v|^{2}}  f_0||_\infty ).\label{global_grazing}
\end{split}
\end{equation}
Now we focus on $  \int_0^t \int_{\partial\Omega} \left[ \int_{u\cdot n(x)>0} |\partial f^m(s,x,u)|\mu^{1/4}(u)\{n\cdot u\}\mathrm{d} u\right]^p\mathrm{d} S_x\mathrm{d} s  $. Recall (\ref{def:gamma_epsilon}). We split the $\{u \in\mathbb{R}^3 : n(x)\cdot u>0\}$ as
\begin{equation}\label{split}
\begin{split}
\int_0^t \int_{\partial\Omega} \left[\int_{n\cdot u>0} |\partial f^m| \mu^{1/4}\{n\cdot u\} \mathrm{d} u\right]^p  \lesssim_p \int_0^t \int_{\partial\Omega} \left[\int_{ (x,u)\in\gamma_+ \backslash \gamma_+^\varepsilon}  \mathrm{d} u\right]^p + \int_0^t \int_{\partial\Omega} \left[\int_{ (x,u)\in \gamma_+^\varepsilon} \mathrm{d} u\right]^p.
\end{split}
\end{equation}
We use H\"{o}lder's inequality to bound
\begin{equation}\notag
\begin{split}
\left[\int_{(x,u)\in\gamma_+^\varepsilon}   \mathrm{d} u\right]^{p} \leq \left[\int_{(x,u)\in\gamma_+^\varepsilon} \mu^{\frac{p}{4(p-1)}} \{n\cdot u\}\mathrm{d} u\right]^{p-1} \left[\int_{(x,u)\in\gamma_+^\varepsilon} |\partial f^m(s,x,u)|^p\{n(x)\cdot u\} \mathrm{d} u\right],
\end{split}
\end{equation}
to bound the second term of (\ref{split})
\begin{equation}\label{global_grazing}
\int_0^t \int_{\partial\Omega} \left[\int_{(x,u)\in\gamma_+^\varepsilon} \mathrm{d} u\right]^p \lesssim_p \varepsilon \int_0^t |\partial f^m(s)|_{\gamma_+,p}^p \mathrm{d} s.
\end{equation}
For the first term (non-grazing part) of (\ref{split}) we use H\"{o}lder's inequality and Lemma \ref{le:ukai} and Lemma \ref{lemma_K} and Lemma \ref{lemma_operator}
for $f^m$ to estimate
%
\begin{equation}\label{global_non_grazing}
\begin{split}
 & \int_0^t \int_{\partial\Omega} \left[\int_{(x,u)\in\gamma_+\backslash \gamma_+^\varepsilon}\mathrm{d} u\right]^p \\
&\lesssim_\varepsilon  || \partial f_0||_p^p+ \int_0^t || \partial
f^{m}(s)||_p^p \mathrm{d} s+ \int_0^t \iint_{\Omega\times\mathbb{R}^3}
\big|\mathcal{G}^{m} \big| |\partial
f^{m}|^{p-1}   \\ 
& \lesssim_\varepsilon  ||
\partial f_0||_p^p+P( ||e^{\theta|v|^{2}} f_0||_\infty)   \sum_{i=m,m-1}\int_0^t
||\partial f^{i}(s)||_p^p  + t P(|| e^{\theta|v|^{2}}f_0||_\infty^{p} ).
\end{split}
\end{equation}

Putting together the estimates (\ref{abstract_green_n}), (\ref{global_grazing}), (\ref{%
global_non_grazing}), and choosing sufficiently small $0< \varepsilon \ll 1,0< T\ll 1,$, we deduce that
\begin{equation}\notag
\begin{split}
&\sup_{0 \leq t\leq T } ||\partial f^{m+1}(t)||_p^p + \int_0^{T } |\partial f^{m+1}|_{\gamma_+,p}^p  \\
 &\lesssim_{T ,\Omega}   || \partial f_0 ||_p^p +  {P} (||  e^{\theta|v|^{2}} f_0||_\infty) 
 + \frac{1}{8} \max_{i=m,m-1}\Big\{\sup_{0\leq t \leq T_*}||\partial f^i(t)||_{p}^p + \int_0^{T_*}|\partial f^i|_{\gamma_+,p}^p  \Big\}.\label{iteratem}
\end{split}
\end{equation}

To conclude the proof we use the following fact from \cite{EGKM} : Suppose $a_i\geq 0, D\geq 0$ and $A_i= \max\{a_i,a_{i-1},\cdots, a_{i-(k-1)}\}$ for fixed $k \in\mathbb{N}$.
\begin{equation}
 {If} \ \ a_{m+1} \leq \frac{1}{8}A_m+D \ \ \text{then} \ \ A_m \ \leq \ \frac{1}{8} A_0 + \left(\frac{8}{7}\right)^2 D, \ \ \text{for} \ \frac{m}{k} \gg 1.\label{aAD}
\end{equation}
{\it Proof of (\ref{aAD}):}  In fact, we can iterate for $m, m-1,...$ to get
\begin{eqnarray*}
a_m &\leq& \frac{1}{8} \max\{\frac{1}{8}A_{m-2}+D,A_{m-2}\}+D \leq \frac{1}{8}A_{m-2} + (1+\frac{1}{8})D\\
&\leq& \frac{1}{8} \max\{\frac{1}{8}A_{m-3}+D, A_{m-3}\}+ (1+\frac{1}{8})D
 \leq \frac{1}{8} A_{m-3} + (1+\frac{1}{8}+\frac{1}{8^2})D\\
&\leq& \frac{1}{8} A_{m-k} + \frac{8}{7}D.
\end{eqnarray*}
Similarly $a_{m-i} \leq \frac{1}{8}A_{m-k} + \frac{8}{7}D$ for all $i=0,1,\cdots,k-1.$ Therefore if $1 \ll m/k \in \mathbb{N}$,
\begin{eqnarray*}
A_m &=&  \max\{a_m,a_{m-1},\cdots, a_{m-(k-1)}\}
 \leq   \frac{1}{8}A_{m-k}+ \frac{8}{7}D\\
 &\leq& \frac{1}{8^2} A_{m-2k} + \frac{8}{7}(1+\frac{1}{8})D \leq \frac{1}{8^3}A_{m-3k} + \frac{8}{7}(1+\frac{1}{8}+ \frac{1}{8^2})D\\
 &\leq& \left(\frac{1}{8} \right)^{ \left[\frac{m}{k}\right]} A_{m-\left[\frac{m}{k}\right]k} + \left(\frac{8}{7}\right)^2 D \leq  \left(\frac{1}{8}\right)^{\frac{m}{k}}A_0 + \left(\frac{8}{7}\right)^2 D \ \leq  \ \frac{1}{8}A_0 + \left(\frac{8}{7}\right)^2 D.
\end{eqnarray*}
This completes the proof of (\ref{aAD}).

In (\ref{aAD}), setting $k=2$ and
\begin{equation*}
\begin{split}
a_i  = \sup_{0 \leq t \leq T_*} ||\partial f^{i}(t)||_p^p + \int_0^t |\partial f^{i}|_{\gamma_+,p}^p  ,\ \ \ \ \
D = C_{T_*,\Omega} \big\{  || \partial f_0 ||_p^p + {P}( || \langle v\rangle^\beta f_0||_\infty )\big\},
\end{split}
\end{equation*}
and applying (\ref{aAD}), we complete the proof of the lemma.\end{proof}

The following result indicates that Theorem 1 is optimal :

\begin{lemma}
\label{optimal} Let $\Omega=B(0;1)$ with $B(0;1)= \{ x\in \mathbb{R}^3 :
|x|<1\}$. There exists an initial datum $f_0(x,v)\in C^\infty$ with $f_0
\subset\subset B(0;1)\times B(0;1)$ so that the solution $f$ to
\begin{equation}
\begin{split}
\partial_t f + v\cdot \nabla_x f &=0, \ \ \ f|_{t=0}=f_0, \\
f(t,x,v)|_{\gamma_-} &= c_\mu \sqrt{\mu(v)} \int_{n(x)\cdot u >0}
f(t,x,u) \sqrt{\mu(u)} \{n(x)\cdot u\} \mathrm{d}
u,  \label{free_transport}
\end{split}
\end{equation}
satisfies
\begin{eqnarray*}
\int_0^{1} \int_{\gamma_-} |\nabla_x f(s,x,v)|^2 \mathrm{d} \gamma \mathrm{d}
s \ = \ + \infty,
\end{eqnarray*}
so that the estimate (\ref{global_p}) of Theorem 1 fails for $p=2$.
\end{lemma}

\begin{proof}
We prove by contradiction. Suppose $\int_0^{1} \int_{\gamma_-} |\partial f(s,x,v)|^2 \mathrm{d} \gamma \mathrm{d} s< +\infty.$ Then
$$
\partial_n f(t,x,v) = \frac{1}{n\cdot v} \Big\{-\partial_t f - (\tau_1 \cdot v ) \partial_{\tau_1}f - (\tau_2 \cdot v) \partial_{\tau_2}f\Big\}, \ \ \ for \ (x,v) \in \gamma_-.
$$
We use the boundary condition to define:
\begin{eqnarray*}
\partial_t f(t,x,v)|_{\gamma_-} &=& c_\mu \sqrt{\mu(v)}A(t,x)\equiv c_\mu\sqrt{\mu} \int_{n\cdot u >0} \partial_t f \sqrt{\mu} \{n\cdot u\} \mathrm{d} u,\\
\partial_{\tau_i} f(t,x,v)|_{\gamma_-}&=& c_\mu \sqrt{\mu(v)}B_i(t,x)\\
 &\equiv& c_\mu\sqrt{\mu} \int_{n\cdot u >0} \partial_{\tau_i} f \sqrt{\mu} \{n\cdot u\} \mathrm{d} u
 +c_\mu \sqrt{\mu} \int_{n\cdot u >0} \nabla_v f \frac{\partial \mathcal{T}}{\partial\tau_i} \mathcal{T}^{-1} u \sqrt{\mu} \{n\cdot u\} \mathrm{d} u.
\end{eqnarray*}
We make a change of variables
 $v_n =  v \cdot n(x) , \ v_{\tau_1} = v\cdot \tau_1(x), \ v_{\tau_2} = v\cdot \tau_2(x)$ to compute
\begin{equation*}
\begin{split}
&  \int_{\partial\Omega} \mathrm{d} S_x \int_0^\infty \mathrm{d} v_{n} \iint_{\mathbb{R}^2} \mathrm{d} v_{\tau_1} \mathrm{d} v_{\tau_2} \\
& \ \ \ \ \ \times  \frac{\mu(v)}{v_\perp}\Big\{
(A)^2 + (v_{\tau_1})^2 (B_1)^2 + (v_{\tau_2})^2 (B_2)^2 +2 v_{\tau_1} AB_1 + 2v_{\tau_2} AB_2 + 2v_{\tau_1} v_{\tau_2} B_1B_2
\Big\}\\
& = \int_0^\infty \mathrm{d} v_{n} \frac{e^{-\frac{|v_n|^2}{2}}}{v_n} \int_{\partial\Omega} \mathrm{d} S_x \big\{
(A)^2 + 2\pi (B_1)^2 + 2\pi (B_2)^2\big\}.
\end{split}
\end{equation*}
Note that the integration over $\partial\Omega$ is a function of $t$ only (independent of $v$). Since $\int_{0}^{\infty} \frac{ \mathrm{d} v_n}{v_n} =\infty,$ we conclude that $A=B_1=B_2\equiv0$ for $(t,x)\in[0,\infty)\times\partial\Omega$. In particular from $A(t,x)=0$ we have for all $t\geq 0$
\begin{eqnarray}
  \int_{n(x)\cdot u >0} f(t,x,u ) \sqrt{\mu(u)} \{n(x) \cdot u\} \mathrm{d} u
 =   \int_{n(x)\cdot u >0} f (0,x,u ) \sqrt{\mu(u)} \{n(x) \cdot u\} \mathrm{d} u.\label{zeroBC}
\end{eqnarray}
We now choose the initial datum to vanish near $\partial\Omega$ :
\begin{eqnarray*}
f_0(x,v) = \phi(|x|)\phi(|v|),
\end{eqnarray*}
where $\phi\in C^{\infty}([0,\infty))$ and $\phi\geq 0$ and $ \mathrm{supp} \phi \subset\subset  [0,1)$ and $\phi\equiv 1$ on $[0,\frac{1}{2}].$ Clearly $$c_\mu \sqrt{\mu(v)} \int_{n(x)\cdot u >0} f_0(x,u) \sqrt{\mu(u)} \{n(x)\cdot u\} \mathrm{d} u=0.$$ Hence $f(t,x,v)\geq 0$ from $f_0\geq 0$ and the zero inflow boundary condition from (\ref{zeroBC}) and the above equality. Moreover following the backward trajectory to the initial plane for $t\in [\frac{1}{8},\frac{1}{4}]$ and $(x,v)\in\gamma_+$ and $|v- \frac{x}{|x|}|< \frac{1}{64},$ and $|v| \in [\frac{1}{8},\frac{1}{2}],$
\begin{eqnarray*}
f(t,x,v) = f_0(x-tv,v) =1,
\end{eqnarray*}
which contradicts to $ c_\mu\sqrt{\mu(v)} \int_{n\cdot u >0} f(t,x,u) \sqrt{\mu(u)} \{n(x)\cdot u\} \mathrm{d} u
 = 0$ for $(t,x,v) \in [0,\infty)\times\gamma_-$ from (\ref{zeroBC}).
 \end{proof}

\vspace{15pt}

\noindent{ \large{4.2. Weighted $W^{1,p} \ (2\leq p< \infty )$ Estimate}}

\vspace{8pt}

We now establish the weighted $W^{1,p}$ estimate for $2\leq p<\infty $ with
the same iteration (\ref{positive_iteration}). From Lemma \ref{local_existence} for $0< \theta < \frac{1}{4}$, we have a uniform bound (\ref{bounded}) and (\ref{bounded_t}).
Recall the notation $\partial_{\mathbf{e}} =[ \nabla _{x},\nabla _{v}]$. Then $ e^{-\varpi \langle v\rangle t}[ \alpha(x,v)] 
^{\beta }\partial f^{m}$ satisfies
\begin{equation}
\label{abstract_w_p}
\begin{split}
[\partial _{t}+v\cdot \nabla _{x}+\nu _{\varpi,\beta } + \nu(\sqrt{\mu}f^{m})]   ( e^{-\varpi \langle v\rangle t} \alpha(x,v) 
^{\beta }\partial f^{m+1} ) &=   e^{-\varpi \langle v\rangle t}  \alpha(x,v) 
^{\beta }
  \mathcal{G}^{m},\\ 
 \alpha(x,v)
^{\beta }\partial
f^{m+1}(0,x,v)&=    \alpha(x,v)  
^{\beta }\partial f_{0}(x,v).
\end{split}
\end{equation}

Here $\nu _{\varpi,\beta }$ is defined in (\ref{nula}) and $\mathcal{G}^{m}$ is defined in (\ref{G}).
Recall from (\ref{G_m})
\begin{equation}\notag
\begin{split}
 &e^{- \varpi \langle v\rangle t} \alpha(x,v)^{\beta} |\mathcal{G}^{m}| \\
\lesssim   & \ 
e^{-\varpi \langle v\rangle t} \alpha(x,v)^{\beta}\Big\{
|\nabla_{x} f^{m+1}| + P(|| e^{\theta|v|^{2}} f_{0} ||_{\infty} ) \Big[e^{-\frac{\theta}{2}|v|^{2}}  +   \int_{\mathbb{R}^{3}} 
\frac{e^{-C_{\theta} |v-u|^{2}}}{|v-u|^{2-\kappa}} | \partial f^{m}(u)|  \mathrm{d}u \Big]
\Big\}.
\end{split}
\label{N}
\end{equation}%

For $(x,v)\in \gamma ,$ from (\ref{seq_w_p_boundary}), the
boundary condition is bounded for $\beta <\frac{p-1}{2p}$ by
\begin{equation}
\begin{split}
&e^{-\varpi \langle v\rangle t} [\alpha(x,v)]^{\beta}|\partial f^{m+1}(t,x,v)|\\
& \lesssim \  e^{-\varpi \langle v\rangle t} [\alpha(x,v)]^{\beta}\sqrt{\mu(v) }\Big(1+\frac{\langle v\rangle }{|n(x)\cdot v|}\Big)%
\int_{n\cdot u>0}|\partial f^{m}(t,x,u)|\langle
u\rangle \sqrt{\mu }\{n\cdot u\}\mathrm{d}u \\
& \ \ +\frac{ e^{-\varpi \langle v\rangle t} [\alpha(x,v)]^{\beta}}{|n(x)\cdot v|}e^{-\frac{\theta }{4}%
|v|^{2}}    {P}(  ||  e^{\theta |v|^{2}}f_0||_{\infty }). \\
\end{split}
\label{dmboundary}
\end{equation}
Set $f^{0}=f_{0}$ and $\partial f^{0}=[\partial _{t}f^{0},\nabla
_{x}f^{0},\nabla _{v}f^{0}]=[0,0,0]$. The main estimate is the following:
 
\begin{proof}[\textbf{Proof of Lemma \ref{weigh_W1p}}]
Fix $p\geq 2, \frac{p-2}{2p} < \beta < \frac{p-1}{2p}$ and $\varpi\gg_{\Omega}1$. We claim that there exists $\
0< T_*\ll 1$ such that we have the following uniformly-in-$m$%
,
\begin{equation}
\begin{split}
 \sup_{0 \leq t \leq T_*} ||  e^{-\varpi \langle v\rangle t} \alpha^{\beta} \partial f^m(t)||_p^p +
\int_0^{T_*} |  e^{-\varpi \langle v\rangle s} \alpha^{\beta}  \partial f^m  |_{\gamma,p}^p 
 \lesssim_{\Omega,T_*} \   {P} ( || 
e^{\theta|v|^2} f_0||_\infty )  ||    \alpha^{\beta}\partial f_0||_p^p,  \label{estimate_W_1_p} 
\end{split}
\end{equation} 
where ${P}$ is some polynomial.

Once we have (\ref{estimate_W_1_p}) then we pass to the limit, $ e^{-\varpi \langle v\rangle t} \alpha^{\beta} \partial f^m
\rightharpoonup  e^{-\varpi \langle v\rangle t} \alpha^{\beta} \partial f $ weakly with norms $\sup_{t \in
[0,T_*]}|| \cdot ||_p^p$ and $  e^{-\varpi \langle v\rangle t} \alpha^{\beta} \partial f ^m|_{\gamma} \rightharpoonup   e^{-\varpi \langle v\rangle t} \alpha^{\beta} \partial f |_\gamma$ in $\int_0^{T_*} |\cdot|_{\gamma,p}^p$ and $%
  e^{-\varpi \langle v\rangle t} \alpha^{\beta} \partial f$ satisfies (\ref{estimate_W_1_p}). Repeat the
same procedure for $[T_*,2T_*], [2T_*,3T_*],\cdots,$ up to the local existence time interval $[0,T^{*}]$ in Lemma \ref{local_existence} to conclude Theorem \ref%
{weigh_W1p}.

We prove (\ref{estimate_W_1_p}) by induction. From Proposition \ref{inflowW1p}, $\partial f^1$ exits. More
precisely we construct $\partial_t f^1, \nabla_x f^1$ first and then $
\nabla_v f^1$. Because of our choice of $\partial f^0,$ the estimate (\ref{estimate_W_1_p}) is valid for $m=1$. Now assume that $ \partial f^i$ exists
and (\ref{estimate_W_1_p}) is valid for all $i=1,2,\cdots,m $. Applying the weighted inflow estimate (Proposition \ref{inflowW1p}) we deduce that $\partial f^{m+1}$ exists.  From the Green's identity (Lemma
\ref{Green}) we have
\begin{equation}\label{abstract_green}
\begin{split}
&\sup_{0 \leq s\leq t}|| e^{-\varpi \langle v\rangle s} \alpha^{\beta} \partial f^{m+1}(s)||_{p}^p +
\int_0^t | e^{-\varpi \langle v\rangle s} \alpha^{\beta}  \partial f^{m+1}|^p_{\gamma_+,p} \\
&+  \int^{t}_{0} || \langle v\rangle^{1/p} e^{-\varpi \langle v\rangle s} \alpha^{\beta} \partial f^{m+1}  ||_{p}^{p}  \\
\lesssim& \ 
 ||    \alpha^{\beta}  \partial f_0||_{p} + t P ( || e^{\theta|v|^{2}} f_{0} ||_{\infty})+ \int_0^t | e^{-\varpi \langle v\rangle s} \alpha^{\beta}  \partial
f^{m+1}|^p_{\gamma_-,p} 
\\
& + (t+ \varepsilon)  \sup_{0 \leq s\leq t}||
 e^{-\varpi \langle v\rangle s} \alpha^{\beta}   \partial  f^{m+1}(s) ||_p^p +    \int_0^t \iint_{\Omega\times\mathbb{R}^3}
[ e^{-\varpi \langle v\rangle s} \alpha^{\beta}]^{p }|\mathcal{G}_{m}| |\partial f^{m+1}|^{p-1}\\
\lesssim & \  ||    \alpha^{\beta}  \partial f_0||_{p} + t P ( || e^{\theta|v|^{2}} f_{0} ||_{\infty})+ \int_0^t | e^{-\varpi \langle v\rangle s} \alpha^{\beta}  \partial
f^{m+1}|^p_{\gamma_-,p}   \\
& + (t+ \varepsilon)  \sup_{0 \leq s\leq t}||
 e^{-\varpi \langle v\rangle s} \alpha^{\beta}   \partial  f^{m+1}(s) ||_p^p  \\
&  + P(|| e^{\theta|v|^{2}} f_{0} ||_{\infty}) \int_{0}^{t} \iint_{\Omega\times \mathbb{R}^{3}} [ e^{- \varpi \langle v\rangle s}  \alpha^{\beta}]^{p  } |\partial f^{m+1}|^{p-1 }
 \int_{\mathbb{R}^{3}}  \frac{e^{-C_{\theta} |v-u|^{2} }}{|v-u|^{2-\kappa}}    |\partial f^{m}(u)|.
\end{split}
\end{equation}

\noindent { \textit{Step 1. Estimate for the nonlocal term}}: The
key estimate is the following: For $0<\beta  < \frac{ p-1 }{2p},$  $0 <
\theta < \frac{1}{4},$ and some $C_{\varpi, \beta ,p}>0$,
\begin{equation}
\sup_{x\in \Omega  }\int_{\mathbb{R}^3}
 \frac{e^{-C_\theta|v-u|^{2}}}{|v-u|^{2-\kappa}}
\frac{ [e^{- \frac{\varpi}{\beta} \langle v\rangle s} \alpha(x,v)  ] ^\frac{ \beta  p }{p-1}%
}{ [e^{- \frac{\varpi}{\beta} \langle u\rangle s} \alpha(x,u) ]^{%
\frac{ \beta  p }{p-1}}}\mathrm{d} u  \ \lesssim_{\Omega,\theta} \ \langle
v\rangle^{\frac{ \beta  p}{p-1}}e^{C_{\varpi, \beta ,p}s^{2}}. \label{Kd}
\end{equation}
%
%
%
%
%
First we assume $|\xi(x)| < \delta_{\Omega}$ so that $n(x):=\frac{\nabla\xi(x)}{|\nabla\xi(x)|}$ is well-defined. We decompose $u_n= u\cdot
n(x)= u\cdot \frac{\nabla \xi(x)}{|\nabla \xi(x)|}$ and $u_{\tau} =u- u_n n(x).
$
For $0 \leq\kappa\leq 1$ is bounded by
\begin{eqnarray*}
&&|v|^{\frac{ \beta  p }{p-1}}\int_{\mathbb{R}^3}   
  \frac{1}{|v-u|^{2- \kappa}} 
e^{-C_\theta{|v-u|^2}}
\frac{e^{-%
\frac{ \varpi   p}{ p-1}\langle v\rangle t}}{e^{-\frac{ \varpi    p }{%
 p-1 }\langle u\rangle t}}
\frac{1}{ |u \cdot \nabla \xi(x)|^{\frac{
\beta  p }{p-1}}}
\mathrm{d} u\\
&& \lesssim_\Omega \ |v |^{\frac{ \beta  p }{%
p-1}} \int_{\mathbb{R}^3}  |v-u|^{-2+\kappa} 
e^{-\frac{%
C_\theta |v-u|^2}{2}} e^{ \frac{\varpi    p }{  p-1 }t|v-u|} {|u_{n}|^{%
\frac{-\beta   p }{p-1}}} \mathrm{d} u\\
&& \lesssim_\Omega \ |v |^{\frac{
\beta p }{p-1}} e^{C_{\varpi,\beta  ,p}t^2} \int_{\mathbb{R}^2} \mathrm{d} u_{\tau}
\int_{\mathbb{R}} \mathrm{d} u_n
 |v-u|^{-2+\kappa} 
 e^{-\frac{%
C_\theta|v-u|^2}{4}} |u_n |^{-\frac{\beta   p }{p-1}}\\
&&
\lesssim_{\Omega} \ C_\kappa|v |^{\frac{ \beta  p }{p-1}} e^{C_{ \varpi, \beta
,p}t^2}. 
\end{eqnarray*}
where we have used
\begin{equation}
 e^{ \frac{ \varpi\beta  p }{  p-1 }t|v-u|}\lesssim e^{C_{\varpi,\beta  ,p}t^2}\times e^{-\frac{C_\theta
|v-u|^2}{4}},\label{exponent}
\end{equation}
for some $C_{\varpi,\beta  ,p}>0$%
. Furthermore we split the last integration as $\int_{ {|u_n|}/{2} \leq
|v_n- u_n|} + \int_{ {|u_n |}/{2} \geq |v_{n}- u_n| }$.
Both of them are bounded by
\begin{eqnarray*}
C\left[\int \frac{e^{-\frac{%
C_\theta|v_n-u_n|^2}{8}}}{|u_n |^{\frac{ \beta  p}{p-1} }}
\mathrm{d} u_{n}  + \int \frac{e^{-\frac{C_\theta|v_n-u_n|^2}{8}}}{%
|v_n-u_n |^{\frac{ \beta p}{p-1} }} \mathrm{d} u_{n}\right] \lesssim
\langle v_{n}\rangle^{-\frac{ \beta p}{p-1} }+1.
\end{eqnarray*}%
If $|\xi(x)|\geq \delta_{\Omega}$ then 
\[
\alpha(x,v) \geq 2|\xi(x)| \{ v\cdot \nabla^{2}\xi(x) \cdot v \} \gtrsim \delta_{\Omega} |v|^{2} \gtrsim \delta_{\Omega} |v_{3}|^{2},
\]
where $v=(v_{1},v_{2},v_{3})$ is the standard coordinate. We set $v_{3}=v_{n}$ and $v_{\tau}=(v_{1},v_{2})$ and follow the exactly same proof.
Therefore we conclude (\ref{Kd}).

Therefore
\begin{eqnarray*}
& &    e^{-\varpi \langle v\rangle s} \alpha^\beta    \Big| \int_{\mathbb{R}^3} \frac{e^{-C_{\theta}|v-u|^{2} } }{|v-u|^{2-\kappa}}  \partial f^m(u)\mathrm{d} u\Big|
\\
&\lesssim_\theta 
 &   \Big(\int_{\mathbb{R}^3} \frac{e^{-C_{\theta}|v-u|^{2} } }{|v-u|^{2-\kappa}} \frac{ [e^{- \frac{\varpi}{\beta} \langle u\rangle s} \alpha]^{ \beta {q} }}{  [e^{- \frac{\varpi}{\beta} \langle u\rangle s} \alpha]^{ \beta
{q} }} \mathrm{d} u\Big)^{\frac{1}{q}}
\Big(
\int_{\mathbb{R}^3} \frac{e^{- C_{\theta}|v-u|^{2} } }{|v-u|^{2-\kappa}} |  [e^{-\varpi \langle u\rangle s} \alpha]^\beta   \partial f^m   (u)|^p \mathrm{d} u
\Big)^\frac{1}{p%
}
 \\
&\lesssim_\theta&  
 \langle v \rangle^{\beta}
e^{Cs^2}\Big(\int_{\mathbb{R}^3} \frac{e^{- C_{\theta}|v-u|^{2} } }{|v-u|^{2-\kappa}}
|    e^{-\varpi \langle u\rangle s} \alpha ^\beta    \partial f^m  (u)|^p \mathrm{d} u
\Big)^\frac{1}{p} ,
\end{eqnarray*}
where at the last line we used $\frac{p-2}{2p}< \beta <%
\frac{p-1}{2p}$ so that $\langle v \rangle^\beta   \leq \langle v \rangle.$%

Finally we use H\"older estimate to bound the last term (nonlocal term) of (\ref{abstract_green}) by
\begin{equation} \label{abstract_dgamma}
\begin{split}
&  C t e^{C_{\varpi , \beta
,p}t^2}  {P}(  ||  e^{\theta |v|^2} f_{0}  ||_\infty  )
\sup_{0 \leq s \leq t}  \iint_{\Omega\times\mathbb{R}^3} | e^{-\varpi \langle v\rangle s} \alpha^
\beta  \partial f^m|^p 
\\
& \ \ + (\delta+ \varepsilon)  {P}( || e^{\theta|v|^2} f_{0}  ||_\infty ) 
\max_{i=m,m+1} \int_0^t
\iint_{\Omega\times\mathbb{R}^3} \langle v \rangle | e^{-\varpi \langle v\rangle s} \alpha^ \beta  \partial
f^i|^p.
\end{split}
\end{equation} 
\\
\\  \textit{Step 2. Boundary Estimate: } Recall (\ref{def:gamma_epsilon}). We use (\ref{dmboundary}) to estimate the contribution of $\gamma_-  $
\begin{equation}\label{d_lambda_boundary}
\begin{split}
& \int_0^t \int_{\gamma_-} |   e^{-\varpi \langle v\rangle s} \alpha(x,v) ^{\beta} 
  \partial f^{m+1}(s,x,v)|^p\\
&\lesssim_p \int_0^t \int_{\gamma_-}     [e^{-\varpi \langle v\rangle s} \alpha(x,v)^{\beta}]
^{  p} \sqrt{\mu}^p \Big(1+ \frac{\langle v\rangle }{|n(x)\cdot v|}\Big)^p \left[ \int_{n(x)\cdot u>0} |\partial f^m(s,x,u)| \mu^{1/4} \{n \cdot u\} \mathrm{d} u\right]^p 
\\
& \ \ + {P} (||  e^{\theta|v|^2}f_0||_\infty)\int_0^t \int_{\gamma_-} \frac{  [e^{-\varpi \langle v\rangle s} \alpha(x,v)^{\beta}]^{ p}}{|n(x)\cdot v|^p} e^{-\frac{\theta p}{4}|v|^2}\mathrm{d} \gamma\mathrm{d} s.
\end{split}
\end{equation}
Using $e^{-\varpi \langle v\rangle s} \alpha(x,v) \leq e^{-\frac{\varpi \langle v\rangle }{2}s} |\nabla_x \xi(x)\cdot v|^{2}$ for $x\in\partial\Omega$, the last term is bounded by
\begin{equation}\notag
C_{\Omega}  {P} ({|| e^{\theta|v|^2}f_0||_\infty }) \int_0^t \int_{\partial\Omega} \int_{\mathbb{R}^3}
|n(x)\cdot v|^{\beta p -p +1} e^{-\frac{\theta p}{4}|v|^2}
\mathrm{d} v \mathrm{d} S_x \mathrm{d} s \lesssim_{\Omega, p, \zeta} t {P} (|| e^{\theta|v|^2} f_0||_\infty),
\end{equation}
 for $\beta >\frac{p-2}{2p}$ so that $2\beta p -p +1 >-1.$

For the first term in (\ref{d_lambda_boundary}) we split as
\begin{equation*}
\left[\int_{n(x)\cdot u>0} \cdots \mathrm{d} u\right]^p \lesssim_p \left[\int_{(x,u)\in\gamma_+^\varepsilon}\cdots \mathrm{d} u\right]^p + \left[\int_{(x,u)\in\gamma_+ \backslash\gamma_+^\varepsilon}\cdots \mathrm{d} u\right]^p.
\end{equation*}
The $\gamma_+^\varepsilon$ contribution (grazing part) of (\ref{d_lambda_boundary}) is bounded by
\begin{equation}\notag
\begin{split}
&C_p\int_0^t \int_{\gamma_-}
  [e^{-\varpi \langle v\rangle s} \alpha(x,v)^{\beta}]^{       p} \sqrt{\mu}^{p} \Big(|n\cdot v|+ \frac{\langle v\rangle^p}{|n\cdot v|^{p-1}}\Big) \\
  & \ \ \ \times 
\left|
\int_{(x,u)\in \gamma_+^\varepsilon}
   e^{-\varpi \langle u\rangle s} \alpha(x,u) ^{\beta   }  \partial f^{m }  \{n  \cdot u\}^{1/p} \frac{\{n \cdot
u\}^{1/q}  {\mu}^{1/4}}{   e^{-\varpi \langle u\rangle s} \alpha(x,u) ^{\beta   }} \mathrm{d} u
\right|^p \mathrm{d} v \mathrm{d} S_x \mathrm{d} s
\\
&\lesssim_{\Omega,p } \int_0^t \int_{ \gamma_-}   [e^{-\varpi \langle v\rangle s} \alpha( v)^{\beta}]^{ p}\Big(|n\cdot v|+ \frac{\langle v\rangle^p}{|n\cdot v|^{p-1}}\Big)
\sqrt{\mu }^{p} \\
& \ \ \ \times\left[\int_{(x,u)\in\gamma_+}   [e^{-\varpi \langle v\rangle s} \alpha( u)^{\beta}]^{   p} |\partial f ^{m }|^p \{n
\cdot u\} \mathrm{d} u\right]
 \left[\int_{(x,u)\in
\gamma_+^\varepsilon}    [e^{-\varpi \langle u\rangle s} \alpha( u)^{\beta}]^{-{   q } } \mu^{q/4} \{n  \cdot
u\}
\mathrm{d} u\right]^{p/q}  \mathrm{d} v \mathrm{d} S_x \mathrm{d}
s,\\
& \lesssim_{\Omega,p,\varpi,\beta  } \varepsilon^a
e^{C_{\varpi, \beta ,p}t^2} \int_0^t  |   e^{-\varpi \langle v\rangle s} \alpha ^{ \beta  } \partial f^{m }(s)
|_{\gamma_+,p}^p  \mathrm{d} s,
\end{split}
\end{equation}
where we used  $  [e^{-\varpi \langle v\rangle s} \alpha(x,v)]\leq  |\nabla \xi (x)\cdot v|^{2}\lesssim_\Omega |n(x)\cdot v|^{2}$ and, for $\beta > \frac{p-2}{2p}$ ($2\beta p -p +1>-1$),
\begin{equation*}
\begin{split}
  &[e^{-\varpi \langle v\rangle s} \alpha(x,v)^{\beta}]^{ p} \Big(|n\cdot v| + \frac{\langle v\rangle^p}{|n\cdot v|^{p-1}}\Big)\sqrt{\mu}^p\\
   &\lesssim_{\Omega} \Big(|n(x)\cdot v|^{1+2\beta p} + \langle v\rangle^p |n(x)\cdot v|^{ 2\beta p- p+1}\Big) \sqrt{\mu(v)}^{p} \in L^1(\{v\in\mathbb{R}^3\}),
\end{split}
\end{equation*}
and, here, $a>0$ is determined via
, with $\frac{p-1}{p}=\frac{1}{q}$,
\begin{equation*}
\begin{split}
&\int_{ \gamma_+^\varepsilon}  [e^{- \frac{\varpi}{\beta} \langle u\rangle s} \alpha(x,u)]^{-\frac{ \beta p }{p-1}} \mu^{%
\frac{p}{4(p-1)}} \{n  \cdot u\} \mathrm{d} u\\
& \lesssim_{\Omega} \int_{  \gamma_+^\varepsilon}
\Big[e^{-\frac{\frac{\varpi}{\beta}\langle u \rangle s}{2}}
|u\cdot \nabla \xi(x)|\Big]^{-\frac{ \beta  p }{p-1}} e^{-\frac{p}{4(p-1)}%
|u|^2} 
|n  \cdot u|\mathrm{d} u\\
& \lesssim_{\Omega} \int_{
\gamma_+^\varepsilon}|u\cdot n|^{1-\frac{ \beta  p }{p-1}} e^{\frac{\varpi
   }{2(p-1)}\langle u \rangle s} e^{-\frac{p}{4(p-1)}|u|^2}
\mathrm{d}
u  \\
&\lesssim_{\Omega} e^{C_{\varpi, \beta ,p}s^2}\int_{
\gamma_+^\varepsilon} |u\cdot n|^{1-\frac{ \beta p }{p-1}} e^{-\frac{p}{%
8(p-1)}|u|^2} \mathrm{d} u   \\
& \lesssim_{\Omega,p} \varepsilon^a e^{C_{\varpi,\beta
,p}t^2},
\end{split}
\end{equation*}
\noindent for some $a>0$ since $1-\frac{ 2\beta p }{p-1}>-1$.

On the other
hand, for the non-grazing contribution $\gamma_+\backslash \gamma_+^\varepsilon$, we use a similar estimate to get
\begin{equation}\notag
\begin{split}
&\int_0^t \int_{\gamma_-}
  [e^{-\varpi \langle v\rangle s} \alpha(x,v)^{\beta}]^{  p} \sqrt{\mu}^p \Big(1+\frac{\langle v\rangle }{|n(x)\cdot v|}\Big)^p \left[
\int_{\gamma_+\backslash \gamma_{+}^\varepsilon}
|\partial f^m(s,x,u)| \mu(u)^{1/4} \{n(x)\cdot u\} \mathrm{d} u
\right]^p
\mathrm{d} \gamma \mathrm{d} s\\
&\lesssim_{\Omega} \int_0^t
\int_{\partial\Omega}
\int_{\mathbb{R}^3}
  [e^{-\varpi \langle v\rangle s} \alpha(x,v)^{\beta}]^{  p} 
\Big(|n\cdot v|+ \frac{\langle v\rangle ^p}{|n\cdot v|^{p-1}}\Big)\sqrt{\mu}^p\\
&  \ \ \  \times
\left[\int_{\gamma_+\backslash \gamma_+^\varepsilon}   e^{-\varpi \langle v\rangle s} \alpha(x,v)^\beta |\partial f^m(s,x,u)|\{n\cdot u\}^{1/p} \frac{\{n\cdot u\}^{1/q}\mu(u)^{1/4}}{  [e^{-\varpi \langle u\rangle s} \alpha(x,u)]^\beta} \mathrm{d} u\right]^p
\mathrm{d} v
\mathrm{d} S_x
 \mathrm{d} s\\
&\lesssim_{\Omega} \int_0^t \int_{\gamma_-}
  [e^{-\varpi \langle v\rangle s} \alpha(x,v)^{\beta}]^{ p} \Big(|n\cdot v|+\frac{\langle v\rangle^p}{|n\cdot v|^{p-1}}\Big)\sqrt{\mu}^p\\
  &  \ \ \  \times\left[\int_{\gamma_+\backslash \gamma_+^\varepsilon}     [e^{-\varpi \langle u\rangle s} \alpha(x,u)^{\beta}]^{ p} |\partial f^m|^p \{n\cdot u\}\mathrm{d} u\right]\left[\int_{\gamma_+}    [e^{-\varpi \langle u\rangle s} \alpha(x,u)^{\beta}] ^{- q} \mu^{q/4}\{n\cdot u\} \mathrm{d} u\right]^{p/q}
\mathrm{d} v \mathrm{d} S_x\mathrm{d} s\\
& \lesssim_\Omega e^{C_{\varpi,\beta,p}t^2} \int_0^t \int_{\gamma_+ \backslash \gamma_+^\varepsilon}    [e^{-\varpi \langle u\rangle s} \alpha(x,u)^{\beta}]^{ p} |\partial f^m(s)|^p \mathrm{d} \gamma \mathrm{d} s,
\end{split}
\end{equation}
where we used $\frac{p-2}{2p}< \beta < \frac{p-1}{2p}$ and
\begin{equation}\notag
\begin{split}
&\int_{\gamma_+}    [e^{-\varpi \langle u\rangle s} \alpha(x,u)^{\beta}]^{-  q} \mu(u)^{q/4} \{n(x)\cdot u\} \mathrm{d} u\\
& = \int_{\gamma_+}    [e^{-\varpi \langle u\rangle s} \alpha(x,u)^{\beta}]^{-\frac{  p}{p-1}} \mu(u)^{\frac{p}{4(p-1)}} \{n\cdot u\} \mathrm{d} u \lesssim_{\Omega,p} e^{C_{\varpi,\beta,p}t^2}.
\end{split}
\end{equation}
By Lemma \ref{le:ukai} and (\ref{%
abstract_w_p}), and (\ref{abstract_dgamma}) the non-grazing part is further bounded by
\begin{equation}%
\label{ab_w_non_graz}
\begin{split}
&\int_0^t \int_{\gamma_+\setminus
\gamma_+^\varepsilon} \\
&\lesssim_\varepsilon
 \int_0^t||  \alpha^\beta
\partial f_0||_p^p+\int_0^t ||  e^{-\varpi \langle v\rangle s}\alpha^\beta    \partial f^{m}||_p^p+
\int_0^t \iint_{\Omega\times\mathbb{R}^3} | \mathcal{G}^{m}  | [ e^{-\varpi \langle v\rangle s}\alpha^\beta]^{p} |\partial f^m|^{p-1}  \notag\\
&\lesssim  
\int_0^t||  \alpha^\beta   \partial f_0||_p^p+ \int_0^t || e^{-\varpi \langle v\rangle s}\alpha^\beta
\partial f^{m}||_p^p   +  t \sup_{0 \leq s\leq t}
||   e^{-\varpi \langle v\rangle s}\alpha^\beta   \partial f^m(s)||_p^p + (1+t)  {P}(  ||  e^{\theta|v|^2} f_0||_\infty )\\
& \ \ +  C t e^{C_{\varpi , \beta
,p}t^2}  {P}(  ||  e^{\theta |v|^2} f_{0}  ||_\infty  )
\sup_{0 \leq s \leq t}  \iint_{\Omega\times\mathbb{R}^3} | e^{-\varpi \langle v\rangle s} \alpha^
\beta  \partial f^m|^p 
\\
& \ \ + (\delta+ \varepsilon)  {P}( || e^{\theta|v|^2} f_{0}  ||_\infty ) 
\max_{i=m,m+1} \int_0^t
\iint_{\Omega\times\mathbb{R}^3} \langle v \rangle | e^{-\varpi \langle v\rangle s} \alpha^ \beta  \partial
f^i|^p.
\end{split}
\end{equation}%

\noindent In summary, the boundary contribution of (\ref{abstract_green}) is
controlled by, for all $0 \leq t\leq T ,$
\begin{equation}
\begin{split}%
\label{abstract_boundary}
 &\int_0^t |    e^{-\varpi \langle v\rangle s} \alpha ^\beta   \partial f^m
(s)|_{\gamma_-,p}^p \mathrm{d} s\\
 \lesssim& \  \int_0^{T }||      \alpha(v)^\beta   \partial
f_0||_p^p+\varepsilon^a \int_0^{T } |    e^{-\varpi \langle v\rangle s} \alpha  ^\beta  \partial
f^m|_{\gamma_+,p}^p\\
& +  T  \max_{i=m-1,m} \sup_{0 \leq t \leq
T_*} ||     e^{-\varpi \langle v\rangle t} \alpha  ^\beta   \partial f^i (t)||_p^p +   {P}(||  
e^{\theta|v|^2}f_0||_\infty) 
\notag\\ 
& +  C t e^{C_{\varpi , \beta
,p}t^2}  {P}(  ||  e^{\theta |v|^2} f_{0}  ||_\infty  )
\sup_{0 \leq s \leq t}  \iint_{\Omega\times\mathbb{R}^3} | e^{-\varpi \langle v\rangle s} \alpha^
\beta  \partial f^m|^p 
\\
& + (\delta+ \varepsilon)  {P}( || e^{\theta|v|^2} f_{0}  ||_\infty ) 
\max_{i=m,m+1} \int_0^t
\iint_{\Omega\times\mathbb{R}^3} \langle v \rangle | e^{-\varpi \langle v\rangle s} \alpha^ \beta  \partial
f^i|^p.
\end{split}
\end{equation}

\noindent Finally we collect the terms to deduce 
\begin{equation*}
\begin{split}
&\sup_{0 \leq t \leq T }||    e^{-\varpi \langle v\rangle t} \alpha ^\beta   \partial
f^{m+1} (t)||_p^p  + \int_0^{T } ||\langle v\rangle ^{1/p}
   e^{-\varpi \langle v\rangle s} \alpha ^\beta    \partial f^{m+1}||_p^p \\
   &  + \int_0^{T } |    e^{-\varpi \langle v\rangle s} \alpha ^\beta  \partial
f^{m+1}|_{\gamma_+,p}^p \mathrm{d}s  \\
\leq & \ C_{T ,\Omega}\big\{||
  \alpha  ^{ \beta   } \partial f_0||_p^p +  {P} (||  e^{\theta|v|^2}f_0||_\infty)
\big\}+\big\{\varepsilon +\delta  +T 
e^{C_{\varpi, \beta ,p}(T )^2}\big\}  {P}(||  e^{\theta|v|^2}f_0||_\infty )  \\
&\times \max_{i=m,m-1} \Big\{
\sup_{0 \leq t \leq T }|| \alpha  ^\beta    \partial f^{i}(t)||_p^p  +
\int_0^{T } |    e^{-\varpi \langle v\rangle s} \alpha  ^\beta   \partial f^{i}|_{\gamma_+,p}^p 
+ \int_0^{T } || \langle v\rangle ^{1/p}     e^{-\varpi \langle v\rangle t} \alpha  ^\beta   \partial f^{i}||_p^p
\Big\}.
\end{split}
\end{equation*}
Recall $C_{\varpi,\beta  ,p}$ from (\ref{Kd}). Choose $0<T \ll
1,$ and $0<\varepsilon\ll 1, \ 0<\delta \ll 1
$ and hence
\begin{equation*}
\begin{split}
& \sup_{0 \leq t \leq T }|| e^{-\varpi \langle v\rangle t} \alpha  ^\beta
\partial f^{m+1}(t)||_p^p  + \int_0^{T } |    e^{-\varpi \langle v\rangle t} \alpha  ^\beta    \partial
f^{m+1}|_{\gamma_+,p}^p    \\
\leq  & \ C_{T ,\Omega}\big\{   ||
 \alpha ^{\beta    } \partial f_0||_p^p +  {P} (|| e^{\theta|v|^2}f_0||_\infty)\}
 \\
& 
 +  \frac{1}{8}\max_{i=m,m-1} \Big\{
 \sup_{0 \leq t \leq
T }||   e^{-\varpi \langle v\rangle t} \alpha ^\beta   \partial f^{i}(t)||_p^p  + \int_0^{T } |
  e^{-\varpi \langle v\rangle t} \alpha ^\beta  \partial f^{i}|_{\gamma_+,p}^p  
\Big\}.
\end{split}
\end{equation*}
Set 
\begin{eqnarray*}
a_i&=& \sup_{0 \leq t \leq
T_*}||    e^{-\varpi \langle v\rangle t} \alpha  ^\beta  \partial f^{m+1}(t)||_p^p  + \int_0^{T } |
   e^{-\varpi \langle v\rangle t} \alpha ^\beta    \partial f^{m+1}|_{\gamma_+,p}^p  ,\\
D&=&C_{T ,\Omega}\big\{||   \alpha  ^{ \beta } \partial
f_0||_p^p + {P} (||  e^{\theta|v|^2}f_0||_\infty ) \big\}.
\end{eqnarray*}
Apply (\ref{aAD}) with $k=2$ to complete the proof.\end{proof}

\vspace{4pt}

\noindent{ \large{4.3.  Weighted $\mathbf{C^1}$ {Estimate}}}

\vspace{8pt}

We start with the same iterative sequences (\ref{abstract_w_p}) with $%
\beta =\frac{1}{2}$. 
For $(x,v)\in \gamma ,$ note that $\sqrt{\alpha(x,v)}=|n(x)\cdot v|.$ Recall $\mathcal{G}$ in (\ref{G_m}). We define
\begin{equation}\label{N}
\mathcal{N}^{m} (t,x,v): = e^{- \varpi \langle v\rangle t} \sqrt{\alpha(x,v)} \mathcal{G}^{m}(t,x,v).
\end{equation}
From (\ref{dmboundary})
with $\beta =\frac{1}{2},$ we have, for $(x,v)\in\gamma_-,$
\begin{equation}\label{boundaryC1}
\begin{split}
& e^{-\varpi \langle v\rangle t}| \sqrt{ \alpha(x,v)}\partial f^{m+1}(t,x,v)|\\
& \lesssim \ \langle v\rangle c_{\mu }%
\sqrt{\mu (v)}\int_{n(x)\cdot u>0}   e^{-\varpi \langle u\rangle t} \sqrt{ \alpha(x,u)} |\partial
f^{m}(t,x,u)|e^{ {\varpi\langle u\rangle } t}\langle
u\rangle \sqrt{\mu (u)}\mathrm{d}u \\
& \ \ +e^{- {\frac{\theta}{4} } |v|^{2}} {P}(||  e^{\theta |v|^{2}}f_0||_{\infty
} ).
\end{split}%
\end{equation}%
%
%
%

Recall the stochastic cycles in Definition \ref{cycles} : For $(t,x,v)$ with $%
(x,v)\notin \gamma_0$ and let $(t^0,x^0,v^0)=(t,x,v)$. For $v^\ell \cdot
n(x^{\ell+1})>0$ we define the $(\ell+1)-$component of the back-time cycle as
\begin{equation}\notag
(t^{\ell+1},x^{\ell+1},v^{\ell+1}) = (t^k-t_{\mathbf{b}}(x^\ell,v^\ell),x_{\mathbf{b}%
}(x^\ell,v^\ell),v^{\ell+1}).  \label{stochastic}
\end{equation}

\begin{lemma}
\label{iteration} If $t^{1}<0$ then
\begin{equation}
| e^{-\varpi \langle v\rangle t} {\alpha(x,v)}^{1/2}  \partial f^{m+1}(t,x,v)|\lesssim \ ||    \alpha(x,v)^{1/2} \partial
f_{0}||_{\infty }+\int_{0}^{t}|\mathcal{N}^{m}(s,x-(t-s)v,v)|\mathrm{d}s.
\label{t1<0}
\end{equation}

If $t^1 > 0$ then
\begin{equation}
\begin{split}  \label{t1geq0}
& | e^{-\varpi \langle v\rangle t} \alpha(x,v)^{1/2} \partial f^{m+1}(t,x,v)| \\
\lesssim& \ \int_{t^{1}}^t |\mathcal{N}^m(s,x-(t-s)v,v)| \mathrm{d} s + e^{-%
\frac{\theta}{4}|v|^2} {P}(||  e^{\theta |v|^2} f_0||_{\infty}) \\
&+ \frac{1}{{w}(v)} \int_{\prod_{j=1}^{\ell-1}\mathcal{V}_j} \sum_{i=1}^{\ell-1}
1_{\{t^{\ell+1} < 0 < t^{\ell}\}} \ |   { \alpha }^{1/2} \partial f^{m+1-i}(0,x^{i}-t^{i}
v^{i},v^{i}) | \ \mathrm{d} \Sigma_i^{\ell-1} \\
&+ \frac{1}{w(v)} \int_{\prod_{j=1}^{\ell-1}\mathcal{V}_j} \sum_{i=1}^{\ell-1}
1_{\{t^{i+1} < 0 < t^i\}} \int_{0} ^{t^i} |\mathcal{N}%
^{m-i}(s,x^i-(t^i-s)v^i,v^i)| \ \mathrm{d} s \ \mathrm{d} \Sigma_i^{\ell-1} \\
&+ \frac{1}{w(v)} \int_{\prod_{j=1}^{\ell-1}\mathcal{V}_j} \sum_{i=1}^{\ell-1}
1_{\{t^{i+1} < 0 \}} \int_{t^{i+1}}^{t^i} |\mathcal{N}%
^{m-i}(s,x^i-(t^i-s)v^i,v^i)| \ \mathrm{d} s \ \mathrm{d}\Sigma_i^{\ell-1} \\
&+ \frac{1}{w(v)}\int_{\prod_{j=1}^{\ell-1}\mathcal{V}_j} \sum_{i=2}^{\ell-1}
1_{\{t^{i-1} < 0 \}} e^{-\frac{\theta}{4}|v^{i-1}|^2}   {P}(||  e^{\theta |v|^2}
f_0||_{\infty})  \ \mathrm{d%
}\Sigma_{i-1}^{\ell-1} \\
&+ \frac{1}{w(v)} \int_{\prod_{j=1}^{\ell-1}\mathcal{V}_j} \mathbf{1}%
_{\{t^{\ell}>0 \}} | e^{-\varpi \langle v^{\ell-1}  \rangle t^{\ell}} \alpha(x^{\ell}, v^{\ell-1})^{1/2}\partial f^{ m+1 -\ell}(t^{\ell},x^{\ell}, v^{\ell-1})| \
\mathrm{d}\Sigma_{\ell-1}^{\ell-1},
\end{split}%
\end{equation}
where $\mathcal{V}_j = \{v^j \in \mathbb{R}^3 : n(x^j)\cdot v^j >0\}$ and
\begin{equation}\notag
{w}(v)=\frac{c_\mu}{\langle v\rangle\sqrt{\mu(v)}},  \label{w}
\end{equation}
and
\begin{equation*}
\mathrm{d}\Sigma_i^{\ell-1} 
 =  \Big\{\Pi_{j=i+1}^{\ell-1}
\mu(v^j)c_\mu|n(x^j)\cdot v^j| \mathrm{d} v^j \Big\}
 \Big\{ {w}(v^i)e^{ {%
\varpi\langle v^i \rangle} t^i} \langle v^i \rangle^2 c_\mu \mu(v^i) \mathrm{d} v^i %
\Big\}
 \Big\{ \Pi_{j=1}^{i-1} e^{ {\varpi\langle v^j \rangle} t^j} \langle
v^j \rangle^2 c_\mu \mu(v^j) \mathrm{d} v^j \Big\}.
\end{equation*}
\end{lemma}
Remark that $\mathrm{d}\Sigma _{i}^{\ell-1}$ is not a probability measure!
\begin{proof}
For $t^1<0$ we use (\ref{abstract_w_p}) with $\beta=1$ to obtain
\begin{equation*}
 e^{-\varpi \langle v\rangle t} \alpha(x,v)^{1/2}  \partial f^{m+1}(t,x,v) \lesssim   \alpha(x-tv,v)^{1/2} \partial f_0(x-tv,v) + \int_0^t e^{-\nu_{\varpi,1}(v)(t-s)} \mathcal{N}^m(s,x-(t-s)v,v) \mathrm{d} s.
\end{equation*}
Consider the case of $t^1>0$. We prove by the induction on $\ell$, the number of iterations. First for $\ell=1$, along the characteristics, for $t^1 > 0$, we have
\begin{eqnarray*}
&&e^{-\varpi \langle v\rangle t} \alpha^{1/2} \partial f^{m+1}(t,x,v)\\
&\lesssim& e^{-\nu_{\varpi,1}(t-t_1)}  e^{-\varpi \langle v\rangle t^{1}} \alpha^{1/2} \partial f^{m+1}(t^1,x^1,v)  + \int_{t^{1}}^t
e^{ -\nu_{\varpi,1}(t-s) } \mathcal{N}^m(s,x-(t-s)v,v) \mathrm{d} s.
\end{eqnarray*}
Now we apply (\ref{boundaryC1}) to the first term above to further estimate
\begin{equation}
\begin{split}\label{1}
 &e^{-\varpi \langle v\rangle t} \alpha^{1/2} |\partial f^{m+1}(t,x,v)|\\
  &\lesssim \ e^{-\nu_{\varpi,1}(v)(t-t^1)} e^{-\frac{\theta}{4}|v|^2}P( || e^{\theta|v|^2}f_0||_{\infty}  )  + \int_{t^1}^t e^{-\nu_{\varpi,1}(v)(t-s)} |\mathcal{N}^m (s,x-(t-s)v,v)| \mathrm{d} s \\
& \ \ + e^{-\nu_{\varpi,1}(v)(t-t^1)} \langle v\rangle c_\mu \sqrt{\mu(v)} \int_{\mathcal{V}_{1}} e^{-\varpi \langle v^{1} \rangle t^{1}} \alpha^{1/2} | \partial  f^m(t^1,x^1,v^1)| e^{ {\varpi\langle v^1 \rangle} t^1} \langle v^1 \rangle \sqrt{\mu(v^1)} \mathrm{d} v^1\\ 
& \lesssim \ e^{-\frac{\theta}{4}|v|^2} P(|| e^{\theta|v|^2}f_0||_{\infty}) + \int_{t^1}^t   |\mathcal{N}^m (s,x-(t-s)v,v)| \\
& \ \ +\frac{c_\mu}{w(v)} \int_{\mathcal{V}_1}  e^{-\varpi \langle v^{1} \rangle t^{1}} \alpha^{1/2} | \partial f^m(t^1,x^1,v^1)| e^{ {\varpi\langle v^1 \rangle} t^1}w(v^1) \langle v^1 \rangle^2 \mu(v^1) \mathrm{d} v^1 ,
\end{split}
\end{equation}
where $w(v)$ is defined in (\ref{w}). Now we continue to express $\partial f^m(t^1,x^1,v^1)$ via backward trajectory to get
\begin{equation*}
\begin{split}\label{m}
&  e^{-\varpi \langle v^{1} \rangle t^{1}} \alpha(x^{1},v^{1})^{1/2}  | \partial f^m(t^1,x^1,v^1)|\\
& \leq   \    \mathbf{1}_{\{t^2 < 0 < t^1 \}}\Big\{ \alpha^{1/2} | \partial f^m (0,x^1-t^1v^1,v^1) |
+ \int_0^{t^1} |\mathcal{N}^{m-1}(s,x^1-(t^1-s)v^1,v^1)| \mathrm{d} s
\Big\} \\
 & \ \  + \mathbf{1}_{\{t^2 >0\}} \Big\{  e^{- \varpi \langle v^{1} \rangle t^{2}}
 \alpha^{1/2} |\partial f^m(t^2,x^2,v^1)| + \int_{t^2}^{t^1} |\mathcal{N}^{m-1}(s,x^1-(t^1-s)v^1,v^1) |\mathrm{d} s
\Big\}.
\end{split}
\end{equation*}
Therefore we conclude from (\ref{1}) that
\begin{eqnarray*}
\begin{split}
 & e^{-\varpi \langle v \rangle t } \alpha(x ,v )^{1/2}  |\partial f^{m+1}(t,x,v)|\\
\lesssim &\int_{t^1}^t   |\mathcal{N}^m (s,x-(t-s)v,v)| \mathrm{d} s+ e^{-\frac{\theta}{4}|v|^2} P(||      e^{\theta|v|^2}f_{0}||_{\infty})  \\
&+  \frac{1}{w(v)} \int_{\mathcal{V}_1} \mathbf{1}_{\{t^2  < 0 < t^1\}} \alpha(x^{1}-t^{1}v^{1},v^{1})^{1/2} | \partial f_{0} (x^1-t^1 v^1,v^1) | e^{ {\varpi \langle v^1 \rangle } t^1} w(v^1) \langle v^1 \rangle^2 c_\mu \mu(v^1) \mathrm{d} v^1\\
&+  \frac{1}{w(v)} \int_{\mathcal{V}_1} \mathbf{1}_{\{t^2 < 0 < t^1\}} \int_0^{t^1} |\mathcal{N}^{m-1}(s,x^1-(t^1-s)v^1,v^1)| \mathrm{d} s e^{ {\varpi \langle v^1 \rangle } t^1} w(v^1) \langle v^1 \rangle^2 c_\mu\mu(v^1) \mathrm{d} v^1\\
&+  \frac{1}{w(v)} \int_{\mathcal{V}_1} \mathbf{1}_{\{t^2 >0\}} \int_{t^2}^{t^1} |\mathcal{N}^{m-1}(s,x^1-(t^1-s)v^1,v^1)| \mathrm{d} s e^{ {\varpi \langle v^1 \rangle } t^1} w(v^1) \langle v^1 \rangle^2 c_\mu\mu(v^1) \mathrm{d} v^1\\
&+  \frac{1}{w(v)} \int_{\mathcal{V}_1} \mathbf{1}_{\{t^2 > 0 \}}  e^{-\varpi \langle v^{1} \rangle t^{2}} \alpha(x^{2},v^{1})^{1/2}  | \partial  f^m(t^2,x^2,v^1) | e^{ {\varpi \langle v^1 \rangle } t^1} w(v^1) \langle v^1 \rangle^2 c_\mu \mu(v^1) \mathrm{d} v^1,
\end{split}
\end{eqnarray*}
and it equals (\ref{t1geq0}) for $\ell=2.$

Assume (\ref{t1geq0}) is valid for $\ell \in \mathbb{N}$. We use (\ref{boundaryC1}) and express the last term of (\ref{t1geq0}) as
\begin{equation}
\begin{split}\label{1_m}
& \mathbf{1}_{\{t^{\ell}>0\}}   e^{- \varpi \langle v^{\ell-1} \rangle t^{\ell}} \alpha(x^{\ell}, v^{\ell-1}) |\partial f^{m+1-k}(t^{\ell},x^{\ell},v^{\ell-1})|\\
& \lesssim \langle v^{\ell-1} \rangle c_\mu \sqrt{\mu(v^{\ell-1})}
\int_{\mathcal{V}_{\ell}} \mathbf{1}_{\{t^{\ell}>0\}} e^{-\varpi \langle v^{\ell} \rangle t^{\ell}} \alpha^{1/2}| \partial  f^{m+1-(k+1)}(t^{\ell},x^{\ell},v^{\ell})| e^{{\varpi \langle v^{\ell} \rangle }t^{\ell}} \langle v^{\ell} \rangle \sqrt{\mu(v^{\ell})} \mathrm{d} v^{\ell}\\
&  \ + e^{-\frac{\theta}{4}|v_{k-1}|^2} P( ||  e^{\theta|v|^2} f_0||_{\infty} ).
\end{split}
\end{equation}
Then we decompose $\mathbf{1}_{\{t^{\ell}>0\}} e^{-\varpi \langle v^{\ell}  \rangle t^{\ell}} \alpha^{1/2}| \partial f^{m+1-(\ell+1)}(t^{\ell},x^{\ell},v^{\ell})|= \mathbf{1}_{\{t^{\ell+1}< 0 < t^{\ell}\}} +\mathbf{1}_{\{t^{\ell+1}>0\}},$
where the first part hits the initial plane as
\begin{equation}\label{3}
\begin{split}
& \mathbf{1}_{\{t^{\ell+1} < 0 < t^{\ell}\}} e^{-\varpi \langle v^{\ell} \rangle t^{\ell}}\alpha^{1/2} | \partial  f^{m+1-(\ell+1)}(t^{\ell},x^{\ell},v^{\ell})|\\
 \lesssim & \ \alpha^{1/2} | \partial  f_0(x^{\ell}-t^{\ell}v^{\ell},v^{\ell})| + \int_0^{t^{\ell}} | \mathcal{N}^{m+1-(\ell+2)}(s,x^{\ell}-(t^{\ell}-s)v^{\ell},v^{\ell}) | \mathrm{d} s,
\end{split}
\end{equation}
and the second part hits at the boundary as
\begin{equation}\label{4}
\begin{split}
&\mathbf{1}_{\{t^{\ell+1}>0\}} e^{-\varpi \langle v\rangle t} \alpha^{1/2} | \partial  f^{m+1-(\ell+1)}(t^{\ell},x^{\ell},v^{\ell})|\\
\lesssim & \ e^{-\varpi \langle v^{\ell} \rangle t^{\ell+1}}\alpha^{1/2} | \partial  f^{m+1-(\ell+1) }(t^{\ell+1},x^{\ell+1},v^\ell)|\\
& + \int_{t^{\ell+1}}^{t^{\ell}} | \mathcal{N}^{ m+1-(\ell+2)} (s,x^{\ell}-(t^{\ell}-s)v^{\ell},v^{\ell}) | \mathrm{d} s.
 \end{split}
\end{equation}
To summarize, from (\ref{1_m}) upon integrating over $\prod_{j=1}^{\ell-1}\mathcal{V}_j$, we obtain a bound for the last term of (\ref{t1geq0}) as
\begin{equation*}
\begin{split}
&  \frac{1}{w(v)} \int_{\prod_{j=1}^{\ell-1}\mathcal{V}_j}
\mathbf{1}_{\{t^{\ell} >0\}} | e^{-\varpi \langle v^{\ell-1}  \rangle t^{\ell}} \alpha^{1/2} \partial f^{m+1-\ell}(t^{\ell},x^{\ell},v^{\ell-1})| \mathrm{d} \Sigma_{\ell-1}^{\ell-1}\\
&\lesssim P(||  e^{\theta|v|^2} f_0||_{\infty})  \frac{1}{w(v)} \int_{\prod_{j=1}^{\ell-1}\mathcal{V}_j} \mathbf{1}_{\{t^{\ell}>0\}} e^{-\frac{\theta}{4}|v^{\ell-1}|^2} \mathrm{d} \Sigma_{\ell-1}^{\ell-1}\\
&  \ +  \frac{1}{w(v)}\int_{\prod_{j=1}^{\ell}\mathcal{V}_j} \mathbf{1}_{\{t^{\ell} >0\}}  e^{-\varpi \langle v^{\ell} \rangle t^{\ell}} \sqrt{\alpha} |\partial f^{m+1-(\ell+1)}(t^{\ell},x^{\ell},v^{\ell})| \mathrm{d} \Sigma_\ell^\ell,
\end{split}
\end{equation*}
where by (\ref{3}) and (\ref{4}), the last term is bounded by
\begin{equation*}
\begin{split}
& \frac{1}{w(v)}\int_{\prod_{j=1}^{\ell}\mathcal{V}_j}
\langle v^{\ell-1}\rangle c_{\mu}\sqrt{\mu(v^{\ell-1})} \sqrt{\mu(v^{\ell})} \langle v^{\ell}\rangle e^{ {\varpi\langle v^{\ell}\rangle} t^{\ell}} \mathrm{d} v^{\ell}\\
&  \ \ \ \times
\ \prod_{j=1}^{\ell-2}\Big\{   e^{ {\varpi \langle v^{j} \rangle} t^{j}} \langle v^{j}\rangle^2 c_{\mu}  \mu(v^{j}) \mathrm{d} v^{j}\Big\} \Big\{w(v^{\ell-1}) e^{ {\varpi \langle v^{\ell-1}\rangle } t^{\ell-1}}\langle v^{\ell-1} \rangle^2\mu(v^{\ell-1})\mathrm{d} v^{\ell-1}\Big\}
\\
& \times\bigg\{ \mathbf{1}_{\{t^{\ell+1} < 0 < t^{\ell}\}}\Big[
\alpha^{1/2} |\partial f(0,x^{\ell}-t^{\ell}v^{\ell},v^{\ell})|  +   \int_0^{t^{\ell}}  | \mathcal{N}^{m-\ell-2}(s,x^{\ell}-(t^{\ell}-s)v^{\ell},v^{\ell})| \mathrm{d} s
\Big]\\
&  \ \ \ \ + \mathbf{1}_{\{t^{\ell+1}>0\}} \Big[
e^{-\varpi \langle v^{\ell} \rangle t^{\ell+1}}  \alpha^{1/2}|\partial f^{m-\ell-1}(t^{\ell+1},x^{\ell+1},v^\ell)| + \int_{t^{\ell+1}}^{t^{\ell}} |\mathcal{N}^{m-\ell-2}(s,x^{\ell}-(t^{\ell}-s)v^{\ell},v^{\ell}) |\mathrm{d} s
\Big]\bigg\}.
\end{split}
\end{equation*}
Now we use (\ref{w}) to conclude Lemma \ref{iteration}.
\end{proof}

\begin{lemma}
\label{tk} There exists $\ell_0( \varepsilon)>0$ such that for $\ell\geq \ell_0$ and
for all $(t,x,v) \in [0,1] \times \bar{\Omega} \times \mathbb{R}^3,$ we
have
\begin{eqnarray*}
\int_{\prod_{j=1}^{\ell-1}\mathcal{V}_j} \mathbf{1}_{\{t^{\ell}(t,x,v,v^{1},\cdots,
v^{\ell-1}) > 0\}} \mathrm{d}\Sigma_{\ell-1}^{\ell-1} & \lesssim_{\Omega} & \left(%
\frac{1}{2}\right)^{- {\ell}/5}.
\end{eqnarray*}
\end{lemma}

\begin{proof}
The proof is based on Lemma 23 of \cite{Guo10}. We note that, for some fixed constant $C_0>0,$
\begin{eqnarray*}
 d\Sigma_{\ell-1}^{\ell-1}
&\leq&
w(v^{\ell-1}) e^{ {\varpi \langle v^{\ell-1} \rangle } t^{\ell-1}} \langle v^{\ell-1} \rangle^2 c_{\mu}\mu(v^{\ell-1})
  \Pi_{j=1}^{\ell-2} e^{{\varpi \langle v^j \rangle}t^j} \langle v^j \rangle^2 c_{\mu}\mu(v^j) \mathrm{d} v^1 \dots
\mathrm{d} v^{\ell-1}\\
&\leq&\Pi_{j=1}^{\ell-1} \{C^\prime e^{C^\prime t^2} \mu(v^j)^\frac{1}{4}\} \ \mathrm{d} v^1 \dots
\mathrm{d} v^{\ell-1}
 \leq  \{ C_0\}^{\ell} \Pi_{j=1}^{\ell-1}\mu(v^j)^{\frac{1}{4}} \mathrm{d} v^j.
\end{eqnarray*}
Choose $\delta=\delta(C_0)>0$ small and define
\begin{eqnarray*}
\mathcal{V}_{j}^{\delta} \equiv \{v^j \in \mathcal{V}_j : v^j \cdot n(x^j)
\geq \delta , \ |v^j| \leq \delta^{-1}\},
\end{eqnarray*}
where we have
$\int_{\mathcal{V}_j \backslash \mathcal{V}_j^\delta} C_0 \mu(v^j)^{\frac{1}{4}} \lesssim
 \delta$ for some $C_0>0.$ Choose sufficiently small $\delta>0.$

On the other hand if $v^j \in \mathcal{V}_j^{\delta}$ then by Lemma 6 of \cite{Guo10}, $(t^j -t^{j+1})  \ \geq \  {\delta^3}/{C_\Omega}.$
Therefore if $t^{\ell}  \geq 0$ then there can be at most $\left\{\left[\frac{C_\Omega}{%
\delta^3}\right]+1\right\}$ numbers of $v^m \in \mathcal{V}_m^\delta$ for $1\leq m
\leq \ell-1$. Equivalently there are at least $\ell-2- \left[\frac{C_\Omega}{%
\delta^3}\right]$ numbers of $v^{m_i} \in \mathcal{V}_{m_i} \backslash \mathcal{V}_{m_i}^\delta$. Hence from $\{C_0\}^\ell = \{C_0\}^m \times \{C_0\}^{\ell-1-m}$, we have

\begin{eqnarray*}
&&\int_{\prod_{j=1}^{\ell-1}\mathcal{V}_j} 1_{\{t^{\ell}(t,x,v,v^{1},\cdots, v^{\ell-1})
> 0\}} d\Sigma_{\ell-1}^{\ell-1} \\
&\leq& \sum_{m=1}^{\left[\frac{C_\Omega}{\delta^3}\right]+1} \int_{\Big\{\begin{array}{ccc}{\text{there are exactly } m \text{ of } v_{m_i} \in \mathcal{V}_{m_i}^\delta} \\ {\small \text{ and } \ \ell-1-m \
of \ v_{m_i} \in \mathcal{V}_{m_i} \backslash \mathcal{V}_{m_i}^{\delta}}\end{array}\Big\}}
  \prod_{j=1}^{\ell-1} C_0 \mu(v^j)^{1/4} \mathrm{d} v^j \\
&\leq& \sum_{m=1}^{\left[\frac{C_\Omega}{\delta^3}\right]+1} \left(%
\begin{array}{ccc}
\ell-1  \\
m
\end{array}%
\right) \left\{ \int_{\mathcal{V}}C_0 \mu(v)^{1/4}\mathrm{d} v\right\}^{m} \left\{ \int_{%
\mathcal{V}\backslash \mathcal{V}^\delta}C_0 \mu(v)^{1/4}\mathrm{d} v\right\}^{\ell-1-m} \\
&\leq& \left(\left[\frac{C_\Omega}{\delta^3}\right]+1\right) \{\ell-1\}^{\left[%
\frac{C_\Omega}{\delta^3}\right]+1} \{ \delta\}^{\ell-2-\left[\frac{%
C_\Omega}{\delta^3}\right]} \left\{ \int_{\mathcal{V}} C_0\mu(v)^{1/4}
dv\right\}^{\left[\frac{C_\Omega}{\delta^3}\right]+1} \lesssim  \frac{\ell}{N} \{Ck\}^{\frac{\ell}{N}} \left(\frac{\ell}{N}\right)^{-\frac{N\ell}{10}} \\&\leq& \{CN\}^{\frac{\ell}{N}} \left(\frac{\ell}{N}\right)^{\frac{\ell}{N}} \left(\frac{\ell}{N}\right)^{-\frac{\ell}{N} \frac{N^2}{20}}
 \leq  \left(\frac{\ell}{N}\right)^{\frac{\ell}{N} \left(- \frac{N^2}{20} + 3\right) } \leq \left(\frac{1}{\ell/N}\right)^{\frac{-\frac{N^2}{20} + 3}{N}\ell} \leq \left(\frac{1}{2}\right)^{-\ell}
,
\end{eqnarray*}
where we have chosen $\ell=N \times \left(\left[\frac{C_\Omega}{\delta^3}\right]%
+1\right)$ and $N=\left(\left[\frac{C_\Omega}{\delta^3}\right]%
+1\right)   \gg C>1$.
\end{proof}


Now we are ready to prove the weighted $C^1$ part of the main theorem:
 \begin{proof}[\textbf{Proof of weighted $C^1$ part in Theorem \ref{weigh_W1p}}]

First we show $W^{1,\infty}$ estimate.
Recall that we use the same sequences (\ref{abstract_w_p}) with $\beta=\frac{1}{2}$ used for the weighted $W^{1,p}$ estimate $(2\leq p <\infty)$. We estimate along the stochastic cycles with (\ref{t1<0}) and (\ref{t1geq0}).  For $t^1<0$, the backward trajectory first hits $t=0$. From Lemma \ref{iteration} and Lemma \ref{lemma_operator} for (\ref{N}), we deduce
\begin{equation}\notag\label{est_c1_t1leq0}
\begin{split}
& \sup_{0 \leq t\leq T }|| \mathbf{1}_{\{t_1<0\}}  e^{-\varpi \langle v \rangle t}  \alpha^{1/2}\partial f^{m+1}(t) ||_\infty   \\
&\lesssim ||  \alpha^{1/2} \partial f_0 ||_{\infty} + {P} (||  e^{\theta|v|^2} f_0||_{\infty})  +   T  \sup_{0\leq t\leq T } ||  e^{-\varpi \langle v\rangle t} \alpha^{1/2}\partial f^{m+1}(t)||_{\infty}  \\
 & +
 \underbrace{ \int^{t}_{t^{1}}
\int_{\mathbb{R}^{3}} e^{-\varpi \langle v\rangle (t-s)}
 \frac{e^{-C_{\theta} |v-u|^{2}}}{|v-u|^{2-\kappa}}  \frac{  \alpha(x,v)^{1/2} }{  \alpha(x,u)^{1/2} }
\mathrm{d}u
 \mathrm{d}s}\times  {P} (|| e^{\theta|v|^2} f_0||_{\infty}) \max_{m} \sup_{0 \leq t \leq T_{*}} ||  e^{-\varpi \langle v \rangle t}  \alpha^{1/2}\partial f^{m+1}(t) ||_\infty,
\end{split}
 \end{equation}
where we have used (\ref{exponent}). Note that, for any $  \beta>\frac{1}{2}$,
\begin{equation}\label{avoid_onehalf}
\frac{1}{\alpha(x,u)^{1/2}} \lesssim \frac{1}{\alpha(x,u)^{\beta}} + 1
\end{equation}

We apply (\ref{nonlocal}) to bound the underbrace term as, for $1\geq \beta>\frac{1}{2}$, 
\begin{equation}\label{nonlocal_diffuse}
\begin{split}
&\Big\{\mathbf{1}_{|v|\lesssim 1} \frac{ \alpha(x,v)^{\frac{1}{2} + \frac{3}{4} -\frac{\beta}{2} } t_{Z}^{\frac{3}{2}-\beta} }{ |v|^{2\beta-1}}  
+ \mathbf{1}_{|v|\gtrsim 1} \frac{\varepsilon^{\frac{3}{2} -\beta  } \alpha(x,v)^{\frac{1}{2}}}{ |v|^{2} \alpha(x,v)^{\beta-1} }\Big\}    
+ \frac{ \alpha(x,v)^{\frac{1}{2}}}{  \varpi\langle  v\rangle \varepsilon^{2} \alpha(x,v)^{ \beta-\frac{1}{2}}}\\
&  \lesssim \ t^{\frac{3}{2}-\beta} + \varepsilon^{\frac{3}{2}-\beta}  + \frac{1}{\varepsilon^{2} \varpi},
\end{split}
\end{equation}
where we used $\alpha(x,v) \lesssim |v|^{2}.$
 
If $t^{1} (t,x,v) \geq 0$, the backward trajectory first hits the boundary, then from (\ref{t1geq0}) we have the following
line-by-line estimate
\begin{equation*}
\begin{split}
&| \mathbf{1}_{\{  t_{1} >0 \}}e^{-\varpi \langle v\rangle t} \alpha^{1/2}  \partial f^{m+1}(t,x,v)|\\
\leq& \ 
 \underbrace{ \int_{t^{1}}^{t} \int_{\mathbb{R}^{3}}e^{-\varpi \langle v\rangle (t-s)}  \frac{e^{-C_{\theta}|V_{\mathbf{cl}}(s)-u|^{2}  }}{|V_{\mathbf{cl}}(s)-u|^{2\kappa}} \frac{\alpha(X_{\mathbf{cl}}(s), V_{\mathbf{cl}}(s))^{\frac{1}{2}}}{ \alpha(X_{\mathbf{cl}}(s), u)^{\frac{1}{2}}} \mathrm{d}u \mathrm{d}s}
|| e^{-\varpi \langle v\rangle s}\alpha^{1/2} \partial f^{m}(s)||_{\infty} 
  \\
&   +  {P}( ||   e^{\theta|v|^{2}} f_{0}||_{\infty}  )+ \ell(C e^{Ct^{2}})^{\ell} \max_{1 \leq i \leq \ell-1} || \alpha^{\frac{1}{2}} \partial f_{0}^{m+1-i} ||_{\infty}
\\
&+  \ell(C e^{Ct^{2}})^{\ell} \langle v\rangle \sqrt{\mu(v)} \times
\underbrace{\max_{i}
\int_{0}^{t^{i}} \int_{\mathbb{R}^{3}} 
e^{-\varpi \langle v^{i} \rangle (t-s)}    \frac{e^{-C_{\theta}|V_{\mathbf{cl}}(s)-u|^{2}  }}{|V_{\mathbf{cl}}(s)-u|^{2\kappa}}
\frac{\alpha(X_{\mathbf{cl}}(s), V_{\mathbf{cl}}(s))^{\frac{1}{2}}}{\alpha(X_{\mathbf{cl}}(s),u)^{\frac{1}{2}}} \mathrm{d}u \mathrm{d}s}\\
& \ \ \times  \max_{1 \leq i \leq k-1} \sup_{ 0\leq s\leq t} ||  e^{-\varpi \langle v\rangle t} \alpha^{1/2} \partial  f^{m+1-i}(t)||_{\infty}
\\
&+   \ell(C e^{Ct^{2}})^{\ell}  {P}( || e^{\theta|v|^{2}} f_{0}||_{\infty}) +    \left(\frac{1}{2}\right)^{-\frac{\ell%
}{5}}\sup_{0 \leq s\leq t} ||  e^{-\varpi \langle v\rangle s} \alpha^{1/2}  \partial f^{m+1- \ell }(s)||_\infty,
\end{%
split}
\end{equation*}
where we have used (\ref{abstract_w_p}), Lemma \ref{tk}, and Lemma \ref{lemma_operator} for (\ref{N}) and (\ref{exponent}). For the underbraced terms we apply (\ref{avoid_onehalf}) and (\ref{nonlocal_diffuse}). Therefore
\begin{equation}\notag
\begin{split}
&| \mathbf{1}_{\{  t_{1} >0 \}}e^{-\varpi \langle v\rangle t} \alpha^{1/2}  \partial f^{m+1}(t,x,v)|\\
\lesssim & \ C_{\ell } C^{C\ell t^{2}} \big\{  t^{\frac{3}{2}-\beta} + \varepsilon^{\frac{3}{2}- \beta} + \frac{1}{\varepsilon^{2} \varpi}  \big\} \times \max_{0 \leq i \leq m}
\sup_{0 \leq s\leq t} || e^{-\varpi \langle v\rangle s} \alpha^{1/2} \partial f^{i}(s)||_{\infty}\\
& +C_{\ell } C^{C\ell t^{2}} \max_{0 \leq i \leq m} || \alpha^{1/2} \partial f_{0}^{i}||_{\infty} 
+ \left(\frac{1}{2}\right)^{-\frac{\ell}{5}}   \max_{0 \leq i \leq m} \sup_{0 \leq s\leq t} || e^{-\varpi \langle v\rangle s} \alpha^{1/2} \partial f^{i}(s)||_{\infty}.
\end{split}
\end{equation}

We choose a large $\ell$ and then small $t$ and then small $\varepsilon$ and then finally large $\varpi$ to conclude 
%
%
%
%
%
%
\begin{eqnarray*}
 \sup_{0 \leq t \leq  T_*} || e^{-\varpi \langle v\rangle t} \alpha^{1/2}
\partial f^{m+1}(t)||_{\infty}  & \leq& 
\frac{1}{8} \max_{m-\ell \leq i \leq m }
\sup_{0 \leq t \leq  T_*} || e^{-\varpi \langle v\rangle t} \alpha^{1/2}
 \partial f^{i }(t)||_{\infty} \\
 &&+||
 \alpha^{1/2}  \partial f_0||_\infty +   {P}( ||    e^{\theta|v|^{2}}  \partial f_0||_\infty) .
\end{eqnarray*}
Set $D=  ||
 \alpha^{1/2}  \partial f_0||_\infty +    {P}( ||    e^{\theta|v|^{2}}  \partial f_0||_\infty)  ,$
\begin{eqnarray*}
a_i = \sup_{0\leq t \leq  T_*} || e^{-\varpi \langle v\rangle t} \alpha^{1/2}
\partial f^{i}(t)||_\infty, \ \ A_i = \max\{ a_i, a_{i-1},\cdots,
a_{i-(\ell-1)}\},
\end{eqnarray*}
then we have $a_{m+1} \leq \frac{1}{8}A_m +
D.$ Use (\ref{aAD}) to conclude
\begin{eqnarray*}
\sup_{0 \leq t \leq T_*}
 || e^{-\varpi \langle v\rangle t} \alpha^{1/2}\partial f (t)||_{\infty} \lesssim  ||   \alpha^{1/2}  \partial f_0
||_{\infty} + {P}( || e^{\theta|v|^2} f_0||_{\infty}).
\end{eqnarray*}
The
existence and uniqueness and the estimate in Theorem \ref{weigh_W1p} are
clear for short time $T_*>0$. We follow the same procedure for $t\in
[T_*,2T_*]$ to conclude
$$
\sup_{T_* \leq t \leq 2T_*} || e^{-\varpi \langle v\rangle t} \alpha^{1/2}\partial f (t)||_{\infty}
\lesssim_{\Omega,T_*} || e^{-\varpi \langle v\rangle T_{*}} \partial f(T_*)||_\infty + {P} (||  e^{\theta|v|^2}
f_0||_{\infty}) .
$$

Then we conclude the weighted $W^{1,\infty}$ part of Theorem \ref{weigh_W1p} following the same procedure for $[T_*,2T_*], [2T_*,3T_*],\cdots.$

Now
we consider the continuity of $ e^{-\varpi \langle v\rangle t}  \alpha^{1/2}\partial f$. Remark that for each step $%
 e^{-\varpi \langle v\rangle t}  \alpha^{1/2} \partial f^m$ satisfies the condition of Proposition 2.
Therefore we conclude $ e^{-\varpi \langle v\rangle t}  \alpha^{1/2} \partial f^m \in C^1([0,T_*]\times\bar{\Omega}%
\times\mathbb{R}^3)$. Now we follow $W^{1,\infty}$ estimate part for $ e^{-\varpi \langle v\rangle t}  \alpha^{1/2}
[\partial f^{m+1}-\partial f^m]$ to show that $e^{-\varpi\langle v\rangle t} \alpha^{1/2}\partial f^m$ is Cauchy
in $L^\infty.$ Then $ e^{-\varpi \langle v\rangle t}  \alpha^{1/2} \partial f^{m} \rightarrow e^{-\varpi \langle v\rangle t}  \alpha^{1/2} \partial f$
strongly in $L^\infty$ so that $ e^{-\varpi \langle v\rangle t}  \alpha^{1/2}\partial f \in C^0([0,T_*]\times\bar{%
\Omega}\times\mathbb{R}^3)$.
\end{proof}

  \vspace{4pt}

\section{\large{Specular Reflection BC}}

\vspace{4pt}

We denote the standard spherical coordinate $\mathbf{x}_{\parallel} = \mathbf{x}_{\parallel} (\omega) =(\mathbf{x}_{\parallel,1}, \mathbf{x}_{\parallel,2})$ for $\omega \in\mathbb{S}^{2}$
\[
\omega = (\cos \mathbf{x}_{\parallel,1}(\omega) \sin \mathbf{x}_{\parallel,2}(\omega), \sin  \mathbf{x}_{\parallel,1}(\omega) \sin \mathbf{x}_{\parallel,2}(\omega), \cos \mathbf{x}_{\parallel,2}(\omega)),
\]
where $\mathbf{x}_{\parallel,1}(\omega) \in [0,2\pi)$ is the azimuth and $ \mathbf{x}_{\parallel,2} (\omega)\in [0,\pi)$ is the inclination. 

We define an orthonormal basis of $\mathbb{R}^{3}$
\begin{equation}\notag
\begin{split}
 {\hat{ {r}}}(\omega) & \ : = \ (\cos  \mathbf{x}_{\parallel,1}(\omega) \sin \mathbf{x}_{\parallel,2}(\omega), \sin \mathbf{x}_{\parallel,1}(\omega) \sin \mathbf{x}_{\parallel,2}(\omega), \cos \mathbf{x}_{\parallel,2}(\omega)),\\
 {\hat{{\phi}}}(\omega) & \ : = \ (\cos  \mathbf{x}_{\parallel,1}(\omega) \cos \mathbf{x}_{\parallel,2}(\omega), \sin \mathbf{x}_{\parallel,1}(\omega) \cos \mathbf{x}_{\parallel,2}(\omega), -\sin \mathbf{x}_{\parallel,2}(\omega)),\\
 {\hat{{\theta}}} (\omega) & \ := \ (-\sin  \mathbf{x}_{\parallel,1}(\omega), \cos  \mathbf{x}_{\parallel,1}(\omega),0).
\end{split}
\end{equation}
Moreover, ${\hat{ {r}}}\times  {\hat{\phi}} = {\hat{\theta}}, \     {\hat{\phi}}\times  {\hat{\theta}} =  {\hat{r}}, \  {\hat{\theta}}\times  {\hat{r}} = {\hat{\phi}},$ and
\begin{equation}\label{Dr}
 \partial_{ \mathbf{x}_{\parallel,1}}  { \hat{r}} = \sin \mathbf{x}_{\parallel,2} \  {\hat{\theta}},  \ \ \ \partial_{\mathbf{x}_{\parallel,2}}  {\hat{r}} =  {\hat{\phi}},
\end{equation}
where $ \partial_{ \mathbf{x}_{\parallel,1}}  { \hat{r}}$ does not vanish (non-degenerate) away from $\mathbf{x}_{\parallel,2} =0$ or $\pi.$

\begin{lemma}\label{chart_lemma}
Assume $\mathbf{0}=(0,0,0)\in\Omega$ and $\Omega$ is convex  (\ref{convex}). Fix
\[
\mathbf{p}=(z, w)\in \partial\Omega\times \mathbb{S}^{2} \ \text{ with } \ n(z)\cdot w=0.
\]
We define the north pole $\mathcal{N}_{\mathbf{p}} \in\partial\Omega$ and the south pole $\mathcal{S}_{\mathbf{p}} \in\partial\Omega$ as
\[
\mathcal{N}_{\mathbf{p}} := |\mathcal{N}_{\mathbf{p}}| ( \frac{z}{|z|}\times w) \in\partial\Omega, \ \ \ \mathcal{S}_{\mathbf{p}} := - |\mathcal{S}_{\mathbf{p}}| ( \frac{z}{|z|}\times w) \in\partial\Omega,
\]
where $\partial_{\mathbf{x}_{\parallel,1}} \hat{r}$ is degenerate. We define the straight-line $\mathcal{L}_{\mathbf{p}}$ passing both poles
\[
\mathcal{L}_{\mathbf{p}} := \{ \tau \mathcal{N}_{\mathbf{p}} + (1-\tau) \mathcal{S}_{\mathbf{p}} : \tau\in \mathbb{R}\}.
\]

(i) There exists a smooth map
\begin{equation}\label{x_parallel}
\begin{split}
 \mathbf{\eta}_{\mathbf{p}} \ : \ \ &  \ \ \ [0,2\pi)\times (0,\pi) \ \ \ \ \rightarrow \ \ \partial\Omega \backslash\{\mathcal{N}_{\mathbf{p}} , \mathcal{S}_{\mathbf{p}}  \}, \\
  &\mathbf{x}_{\parallel_{\mathbf{p}}}:= (\mathbf{x}_{\parallel_{\mathbf{p}},1  }, \mathbf{x}_{\parallel_{\mathbf{p}},2  })     \mapsto \ \ \ \  \mathbf{\eta}_{\mathbf{p}}( \mathbf{x}_{\parallel_{\mathbf{p}}}) ,
\end{split}
\end{equation}
which is one-to-one and onto. Here on $[0,2\pi)\times (0,\pi)$ we have $\partial_{i} \mathbf{\eta}_{\mathbf{p}}:= \frac{\partial \mathbf{\eta}_{\mathbf{p}}}{\partial \mathbf{x}_{\parallel_{\mathbf{p}} ,i  }}\neq 0$ and
 \begin{equation}\label{nondegenerate_eta}
 \frac{\partial   \mathbf{\eta}_{ \mathbf{p}} }{\partial\mathbf{x}_{\parallel_{ \mathbf{p}},1}}  ( \mathbf{x}_{\parallel_{\mathbf{p}} }  ) \times \frac{\partial   \mathbf{\eta}_{ \mathbf{p}} }{\partial\mathbf{x}_{\parallel_{\mathbf{p}},2}}  ( \mathbf{x}_{\parallel_{\mathbf{p}} }  )\neq 0.
 \end{equation}
We define
$$\mathbf{n}_{ \mathbf{p}}:= n\circ\mathbf{\eta}_{ \mathbf{p}} : [ 0 ,2\pi) \times (0,\pi) \rightarrow  \mathbb{S}^{2}.$$

\vspace{8pt}

(ii) We define the $\mathbf{p}-$spherical coordinate:

\vspace{4pt}

For $\delta>0, \ \delta_{1}>0 , \ C>0,$ we have a smooth one-to-one and onto map

\vspace{4pt}

\begin{equation}\notag
\begin{split} 
\Phi_{\mathbf{p}} :  [0,C\delta) \times [0, 2\pi) \times (\delta_{1},\pi-\delta_{1}) \times \mathbb{R}\times \mathbb{R}^{2} & \   \rightarrow \  \{x\in \bar{\Omega} : |\xi(x)|< \delta\} \backslash B_{C\delta_{1}}(\mathcal{L}_{\mathbf{p}})\times \mathbb{R}^{3},\\
(\mathbf{x}_{\perp_{\mathbf{p}}}, \mathbf{x}_{\parallel_{\mathbf{p}},1 }, \mathbf{x}_{\parallel_{\mathbf{p}},2 }, \mathbf{v}_{\perp_{\mathbf{p}}}, \mathbf{v}_{\parallel_{\mathbf{p}},1},  \mathbf{v}_{\parallel_{\mathbf{p}},2}) & \ \mapsto \ 
\Phi_{\mathbf{p}} (\mathbf{x}_{\perp_{\mathbf{p}}}, \mathbf{x}_{\parallel_{\mathbf{p}},1 }, \mathbf{x}_{\parallel_{\mathbf{p}},2 }, \mathbf{v}_{\perp_{\mathbf{p}}}, \mathbf{v}_{\parallel_{\mathbf{p}},1},  \mathbf{v}_{\parallel_{\mathbf{p}},2}),
\\
\end{split}
\end{equation}
\vspace{4pt}

where $B_{C\delta_{1}}(\mathcal{L}_{\mathbf{p}}): =\{ x\in \mathbb{R}^{3} : |x - y |< C\delta_{1} \text{ for some } y \in \mathcal{L}_{\mathbf{p}}  \}.$

Explicitly,
\begin{equation}\label{polar}
\Phi_{\mathbf{p}} ( \mathbf{x}_{\perp_{\mathbf{p}}},\mathbf{x}_{\parallel_{\mathbf{p}}} ,  \mathbf{v}_{\perp_{\mathbf{p}}},\mathbf{v}_{\parallel_{\mathbf{p}}}):=
\left[\begin{array}{c}
\mathbf{x}_{\perp_{\mathbf{p}}}[- \mathbf{n}_{\mathbf{p}}   (\mathbf{x}_{\parallel_{\mathbf{p}}})     ]  + \mathbf{\eta}_{ \mathbf{p}}(\mathbf{x}_{\parallel_{ \mathbf{p}}}) \\
  \mathbf{v}_{\perp_{ \mathbf{p}}} [-    \mathbf{n}_{\mathbf{p}}   (\mathbf{x}_{\parallel_{ \mathbf{p}}}) ]  + \mathbf{v}_{\parallel_{ \mathbf{p}}} \cdot \nabla \mathbf{\eta}_{ \mathbf{p}} (\mathbf{x}_{\parallel_{ \mathbf{p}}}) + \mathbf{x}_{\perp_{ \mathbf{p}}} \mathbf{v}_{\parallel_{ \mathbf{p}}} \cdot \nabla [-\mathbf{n}_{ \mathbf{p}} (\mathbf{x}_{\parallel_{ \mathbf{p}}})]
  \end{array} \right],
\end{equation}
where $\nabla \mathbf{\eta}_{\mathbf{p}}= (\partial_{1} \mathbf{\eta}_{\mathbf{p}}, \partial_{2} \mathbf{\eta}_{\mathbf{p}})= ( \frac{\partial \mathbf{\eta}_{\mathbf{p}}   }{\partial \mathbf{x}_{\parallel_{   \mathbf{p } ,1  }   }} ,\frac{\partial \mathbf{\eta}_{\mathbf{p}}   }{\partial \mathbf{x}_{\parallel_{   \mathbf{p } ,2  }   }}  )$ and $\nabla \mathbf{n}_{\mathbf{p}}= (\partial_{1} \mathbf{n}_{\mathbf{p}}, \partial_{2} \mathbf{n}_{\mathbf{p}}  )= (\frac{\partial \mathbf{ n}_{\mathbf{p}}   }{\partial \mathbf{x}_{\parallel_{   \mathbf{p } ,1  }   }},\frac{\partial \mathbf{n}_{\mathbf{p}}   }{\partial \mathbf{x}_{\parallel_{   \mathbf{p } ,2  }   }}  ).$

We fix an inverse map
\begin{equation}\notag
\begin{split}
\Phi_{\mathbf{p}}^{-1 }   :    \{x\in \bar{\Omega} : |\xi(x)|< \delta\}\backslash  B_{C\delta^{\prime}}(\mathcal{L}_{\mathbf{p}})    \times \mathbb{R}^{3}    & \ \rightarrow \  [0,C\delta) \times  [0, 2\pi) \times (\delta_{1},\pi-\delta_{1}) \times \mathbb{R}\times \mathbb{R}^{2}.
\end{split}
\end{equation}
In general this choice is not unique but once we fix the range as above then an inverse map is uniquely determined.

We denote, for $(x,v) \in   \{x\in \bar{\Omega} : |\xi(x)|< \delta\}\backslash  B_{C\delta^{\prime}}(\mathcal{L}_{\mathbf{p}})    \times \mathbb{R}^{3}$
\[
(\mathbf{x}_{\perp_{\mathbf{p}} }, \mathbf{x}_{\parallel_{\mathbf{p}},1}, \mathbf{x}_{\parallel_{\mathbf{p}},2}, \mathbf{v}_{\perp_{\mathbf{p}}}, \mathbf{v}_{\parallel_{\mathbf{p}},1}, \mathbf{v}_{\parallel_{\mathbf{p}},2})=  \Phi_{\mathbf{p}}^{-1} (x,v).
\]

 \vspace{8pt}

(iii) For $|\xi(X_{\mathbf{cl}}(s;t,x,v)) |< \delta$ and $|X_{\mathbf{cl}}(s;t,x,v)-\mathcal{L}_{\mathbf{p}}|> C\delta_{1}$ we define
\begin{equation}\notag
\begin{split}
 (\mathbf{X}_{\mathbf{p}}(s;t,x,v), \mathbf{V}_{\mathbf{p}}(s;t,x,v))& := \Phi^{-1}_{\mathbf{p}} ( {X}_{\mathbf{cl}}(s;t,x,v), {V}_{\mathbf{cl}}(s;t,x,v))\\
& := (\mathbf{x}_{\perp_{\mathbf{p}}}(s;t,x,v), \mathbf{x}_{\parallel_{\mathbf{p}}}(s;t,x,v),\mathbf{v}_{\perp_{\mathbf{p}}}(s;t,x,v), \mathbf{v}_{\parallel_{\mathbf{p}}}(s;t,x,v) )
.
\end{split}
\end{equation}
Then $|v|\simeq |\mathbf{V}_{\mathbf{p}}| $ and
\begin{equation}\label{ODE_ell}
\begin{split}
 &\dot{\mathbf{x}} _{\perp_{\mathbf{p}}}(s;t,x,v) = \mathbf{v}_{\perp_{ \mathbf{p}}} (s;t,x,v),\\
 &\dot{\mathbf{x}} _{\parallel_{ \mathbf{p}}}(s;t,x,v) = \mathbf{v}_{\parallel_{ \mathbf{p}}} (s;t,x,v), \\
 &\dot{\mathbf{v}} _{\perp_{ \mathbf{p}}}(s;t,x,v) = F_{\perp_{ \mathbf{p}}}(\mathbf{x}_{ \mathbf{p}}(s;t,x,v), \mathbf{v}_{ \mathbf{p}}(s;t,x,v)),\\
 &\dot{\mathbf{v}} _{\parallel_{ \mathbf{p}}}(s;t,x,v) = F_{\parallel_{ \mathbf{p}}}(\mathbf{x}_{ \mathbf{p}}(s;t,x,v), \mathbf{v}_{ \mathbf{p}}(s;t,x,v)).
\end{split}
\end{equation}
Here
\begin{equation}\label{F_perp}
\begin{split}
F_{\perp_{ \mathbf{p}}}&=F_{\perp_{\mathbf{p}}}(\mathbf{x}_{\perp_{ \mathbf{p}}}, \mathbf{x}_{\parallel_{ \mathbf{p}}},   \mathbf{v}_{\parallel_{ \mathbf{p}}})\\
&=  \sum_{j,k=1}^{2} \mathbf{v}_{\parallel_{ \mathbf{p}},k} \mathbf{v}_{\parallel_{ \mathbf{p}},j} \ \partial_{j} \partial_{k} \mathbf{\eta}_{ \mathbf{p}}(\mathbf{x}_{\parallel_{ \mathbf{p}}}) \cdot \mathbf{n}_{ \mathbf{p}}(\mathbf{x}_{\parallel_{ \mathbf{p}}})  
  - \mathbf{x}_{\perp_{ \mathbf{p}}} \sum_{k=1}^{2} \mathbf{v}_{\parallel_{ \mathbf{p}},k} (\mathbf{v}_{\parallel_{ \mathbf{p}}}\cdot \nabla) \partial_{k} \mathbf{n}_{ \mathbf{p}}(\mathbf{x}_{\parallel_{ \mathbf{p}}}) \cdot \mathbf{n}_{ \mathbf{p}}(\mathbf{x}_{\parallel_{ \mathbf{p}}}),
\end{split}
\end{equation}
where
$$\sum_{j,k=1}^{2} \mathbf{v}_{\parallel_{ \mathbf{p}},k} \mathbf{v}_{\parallel_{ \mathbf{p}},j} \ \partial_{j} \partial_{k} \mathbf{\eta}_{ \mathbf{p}}(\mathbf{x}_{\parallel_{ \mathbf{p}}}) \cdot \mathbf{n}_{ \mathbf{p}}(\mathbf{x}_{\parallel_{ \mathbf{p}}})\lesssim_{\xi}- |\mathbf{v}_{\parallel}|^{2},$$
and
\begin{equation}\label{F||}
\begin{split}
F_{\parallel_{ \mathbf{p}}}&=F_{\parallel_{ \mathbf{p}}}(\mathbf{x}_{\perp_{ \mathbf{p}}}, \mathbf{x}_{\parallel_{ \mathbf{p}}}, \mathbf{v}_{\perp_{ \mathbf{p}}}, \mathbf{v}_{\parallel_{ \mathbf{p}}})\\
& =  \sum_{i} G_{ \mathbf{p},ij} (\mathbf{x}_{\perp_{ \mathbf{p}}}, \mathbf{x}_{\parallel_{ \mathbf{p}}})\frac{(-1)^{i }}{\mathbf{n}_{ \mathbf{p}}(\mathbf{x}_{\parallel_{ \mathbf{p}}}) \cdot (\partial_{1} \mathbf{\eta}_{ \mathbf{p}}(\mathbf{x}_{\parallel_{ \mathbf{p}}}) \times \partial_{2} \mathbf{\eta}_{ \mathbf{p}}(\mathbf{x}_{\parallel_{ \mathbf{p}}}))} \\
& \ \ \  \times \big\{
2 \mathbf{v}_{\perp_{ \mathbf{p}}} \mathbf{v}_{\parallel_{ \mathbf{p}}} \cdot \nabla \mathbf{n}_{ \mathbf{p}} (\mathbf{x}_{\parallel_{ \mathbf{p}}}) - \mathbf{v}_{\parallel_{ \mathbf{p}}} \cdot  \nabla^{2} \mathbf{\eta}_{ \mathbf{p}}(\mathbf{x}_{\parallel_{ \mathbf{p}}}) \cdot \mathbf{v}_{\parallel_{ \mathbf{p}}} + \mathbf{x}_{\perp_{ \mathbf{p}}} \mathbf{v}_{\parallel_{ \mathbf{p}}} \cdot \nabla^{2} \mathbf{n}_{ \mathbf{p}}(\mathbf{x}_{\parallel_{ \mathbf{p}}}) \cdot \mathbf{v}_{\parallel_{ \mathbf{p}}}
\big\} \\
& \ \ \  \   \cdot \big\{ \mathbf{n}_{ \mathbf{p}} (\mathbf{x}_{\parallel_{ \mathbf{p}}}) \times \partial_{i+1} \mathbf{\eta}_{ \mathbf{p}}(\mathbf{x}_{\parallel_{ \mathbf{p}}})  \big\},
\end{split}
\end{equation}
where a smooth bounded function $G_{ \mathbf{p},ij} (\mathbf{x}_{\perp_{ \mathbf{p}}}, \mathbf{x}_{\parallel_{\mathbf{p}}})$ is specified in (\ref{G}).


\vspace{8pt}

(iv) Let $\mathbf{q} = (y,u) \in \partial\Omega \times \mathbb{S}^{2}$ with $n(y)\cdot u=0$ and $\mathbf{p}\sim \mathbf{q}$ and
\[
  {\Phi}_{\mathbf{p}} ( \mathbf{x}_{\perp_{  \mathbf{p}   }}, \mathbf{x}_{\parallel_{ \mathbf{p}}}, \mathbf{v}_{\perp_{\mathbf{p}}}, \mathbf{v}_{\parallel_{ \mathbf{p}}}) =(x,v)
= {\Phi}_{\mathbf{q}} ( \mathbf{x}_{\perp_{ \mathbf{q}   }}, \mathbf{x}_{\parallel_{ \mathbf{q}}}, \mathbf{v}_{\perp_{ \mathbf{q}}}, \mathbf{v}_{\parallel_{ \mathbf{q}}}) .
\]
Then
\begin{equation}\label{chart_changing}
\begin{split}
  \frac{\partial (\mathbf{x}_{\perp_{   \mathbf{p} }}, \mathbf{x}_{\parallel_{   \mathbf{p} }}, \mathbf{v}_{\perp_{   \mathbf{p} }}, \mathbf{v}_{\parallel_{   \mathbf{p} }})}{\partial ( \mathbf{x}_{\perp_{ \mathbf{q} }}, \mathbf{x}_{\parallel_{ \mathbf{q} }}, \mathbf{v}_{\perp_{ \mathbf{q} }}, \mathbf{v}_{\parallel_{ \mathbf{q} }})} = \nabla \Phi_{ \mathbf{q} }^{-1} \nabla \Phi_{  \mathbf{p} } 
=   
\mathbf{Id}_{6,6} + O_{\xi}(|\mathbf{p}-\mathbf{q}|)
\left[
\begin{array}{ccc|ccc}
0 & 0 & 0 &  &  &  \\
0 & 1 & 1 &  & \mathbf{0}_{3,3} & \\
0 & 1 & 1 &  & &   \\ \hline
0 & 0 & 0 & 0 & 0 & 0 \\
0 & |v| &  |v|  & 0 & 1 & 1 \\
0 &  |v|  &  |v|  & 0 & 1 & 1
\end{array}
\right].
\end{split}
\end{equation}

\end{lemma}

\begin{proof}[Proof of (i) in Lemma \ref{chart_lemma}] Denote
\begin{eqnarray*}
&\frac{z}{|z|}= {\hat{r}}(\frac{z}{|z|}):= (\cos \mathbf{x}_{\parallel,1}  (\frac{z}{|z|})\sin   \mathbf{x}_{\parallel,2} (\frac{z}{|z|} ), \sin  \mathbf{x}_{\parallel,1} (\frac{z}{|z|}) \sin  \mathbf{x}_{\parallel,2} ( \frac{z}{|z|}), \cos  \mathbf{x}_{\parallel,2} ( \frac{z}{|z|}) ),&\\
&w=  w_{ \mathbf{x}_{\parallel,2}  } \mathbf{\hat{\phi}}  (\frac{z}{|z|})+ w_{ \mathbf{x}_{\parallel,1}   }  \mathbf{\hat{\theta}}(\frac{z}{|z|}): =
(w \cdot  \mathbf{\hat{\phi}}(\frac{z}{|z|}) ) \mathbf{\hat{\phi}}(\frac{z}{|z|})+ (w\cdot  \mathbf{\hat{\theta}}(\frac{z}{|z|}))   \mathbf{\hat{\theta}}(\frac{z}{|z|})
,&\\
&\frac{z}{|z|}\times w = w_{ \mathbf{x}_{\parallel,2}  } \mathbf{\hat{\theta}}(\frac{z}{|z|}) - w_{  \mathbf{x}_{\parallel,1} } \mathbf{\hat{\phi}}(\frac{z}{|z|}).&
\end{eqnarray*}
We define the rotational matrix which maps $\{ \frac{z}{|z|}, w, \frac{z}{|z|}\times w \} \mapsto \{ \mathbf{e}_{1}, \mathbf{e}_{2}, \mathbf{e}_{3}\}$:
\begin{equation}\notag
\begin{split}
 \mathcal{O}_{\mathbf{p}} = \left[\begin{array}{c} \frac{z}{|z|} \\ w \\ \frac{z}{|z|} \times w  \end{array} \right]_{3\times 3}
=  \left[\begin{array}{c}   {\hat{r}}(\frac{z}{|z|})  \\ w_{ \mathbf{x}_{\parallel_{\mathbf{p}},2 }} \mathbf{\hat{\phi}}(\frac{z}{|z|}) + w_{ \mathbf{x}_{\parallel_{\mathbf{p}},1 } } \mathbf{\hat{\theta}}(\frac{z}{|z|}) \\
- w_{ { \mathbf{x}_{\parallel_{\mathbf{p}},1 } }  } \mathbf{\hat{\phi}}(\frac{z}{|z|}) + w_{   \mathbf{x}_{\parallel_{\mathbf{p}},2 }  } \mathbf{\hat{\theta}}(\frac{z}{|z|})
 \end{array}\right]_{3\times 3}.
\end{split}
\end{equation}

For $x \in\partial\Omega$ with $x\neq \mathcal{N}_{\mathbf{p}}$ and $x\neq \mathcal{S}_{\mathbf{p}}$ we define
\begin{equation}\notag
( \mathbf{x}_{\parallel_{\mathbf{p}},1}  , \mathbf{x}_{\parallel_{ \mathbf{p}},2}) \in [0, 2\pi) \times (0,\pi), \ \   \text{such that}  \ \  {\hat{r}}(\mathbf{x}_{\parallel_{ \mathbf{p}},1},  \mathbf{x}_{\parallel_{ \mathbf{p}},2}) = \mathcal{O}_{ \mathbf{p}}\Big( \frac{x}{|x|}\Big).
\end{equation}

Now we define $ {R}_{ \mathbf{p}} : [ 0 ,2\pi) \times [0,\pi) \rightarrow (0,\infty)$ such that
\begin{equation}\label{radius}
\xi(   {R}_{ \mathbf{p}} ( \mathbf{x}_{\parallel_{ \mathbf{p}},1},\mathbf{x}_{\parallel_{ \mathbf{p}},2} ) \mathcal{O}_{ \mathbf{p}}^{-1}  {\hat{r}} ( \mathbf{x}_{\parallel_{ \mathbf{p}},1},\mathbf{x}_{\parallel_{ \mathbf{p}},2} )  )=0.
\end{equation}
 We also define $\mathbf{\eta}_{ \mathbf{p}} : [ 0 ,2\pi) \times [0,\pi) \rightarrow \partial\Omega $ such that
\begin{equation}\notag
\begin{split}
 \mathbf{\eta}_{ \mathbf{p}}  ( \mathbf{x}_{\parallel_{ \mathbf{p}},1},\mathbf{x}_{\parallel_{ \mathbf{p}},2} )= R_{ \mathbf{p}}  ( \mathbf{x}_{\parallel_{ \mathbf{p}},1},\mathbf{x}_{\parallel_{ \mathbf{p}},2} ) \mathcal{O}_{ \mathbf{p}}^{-1} {\hat{r}} ( \mathbf{x}_{\parallel_{ \mathbf{p}},1},\mathbf{x}_{\parallel_{ \mathbf{p}},2} ).
\end{split}
\end{equation}

Directly, from (\ref{Dr}) and (\ref{radius}), with fixed $\mathbf{p}= (z,w),$
\begin{equation}\notag
\begin{split}
\frac{\partial R_{ \mathbf{p}}}{\partial\mathbf{x}_{\parallel_{ \mathbf{p}},1} } ( \mathbf{x}_{\parallel_{ \mathbf{p}},1},\mathbf{x}_{\parallel_{ \mathbf{p}},2} )  &=  \frac{- \sin(\mathbf{x}_{\parallel_{ \mathbf{p}},2} )  R_{\mathbf{p}}  \nabla\xi(\mathbf{\eta}_{\mathbf{p}} ( \mathbf{x}_{\parallel_{ \mathbf{p}},1},\mathbf{x}_{\parallel_{ \mathbf{p}},2} ))\cdot   \mathcal{O}_{ \mathbf{p}}^{-1} \mathbf{\hat{\theta}} ( \mathbf{x}_{\parallel_{ \mathbf{p}},1},\mathbf{x}_{\parallel_{ \mathbf{p}},2} ) }{   \nabla\xi(\mathbf{\eta}_{\mathbf{p}}  ( \mathbf{x}_{\parallel_{ \mathbf{p}},1},\mathbf{x}_{\parallel_{ \mathbf{p}},2} ))  \cdot \mathcal{O}_{ \mathbf{p}}^{-1}  {\hat{r}} ( \mathbf{x}_{\parallel_{ \mathbf{p}},1},\mathbf{x}_{\parallel_{ \mathbf{p}},2} ) }\\
&=
\frac{- \sin( \mathbf{x}_{\parallel_{ \mathbf{p}},2})  [ R_{ \mathbf{p}}  ( \mathbf{x}_{\parallel_{ \mathbf{p}},1},\mathbf{x}_{\parallel_{ \mathbf{p}},2} ) ]^{2}
\nabla\xi(\mathbf{\eta}_{ \mathbf{p}}  ( \mathbf{x}_{\parallel_{ \mathbf{p}},1},\mathbf{x}_{\parallel_{ \mathbf{p}},2} ))
\cdot
\mathcal{O}_{ \mathbf{p}}^{-1} \mathbf{\hat{\theta}} ( \mathbf{x}_{\parallel_{ \mathbf{p}},1},\mathbf{x}_{\parallel_{ \mathbf{p}},2} )}{\nabla \xi (\mathbf{\eta}_{ \mathbf{p}} ( \mathbf{x}_{\parallel_{ \mathbf{p}},1},\mathbf{x}_{\parallel_{ \mathbf{p}},2} )) \cdot \mathbf{\eta}_{ \mathbf{p}} ( \mathbf{x}_{\parallel_{ \mathbf{p}},1},\mathbf{x}_{\parallel_{ \mathbf{p}},2} )   },\\
\frac{\partial  R_{ \mathbf{p}}}{\partial \mathbf{x}_{\parallel_{ \mathbf{p}},2}} ( \mathbf{x}_{\parallel_{ \mathbf{p}},1},\mathbf{x}_{\parallel_{ \mathbf{p}},2} )&=  \frac{-  R_{ \mathbf{p}}
( \mathbf{x}_{\parallel_{ \mathbf{p}},1},\mathbf{x}_{\parallel_{ \mathbf{p}},2} )
\nabla \xi( \mathbf{\eta}_{ \mathbf{p}}  ( \mathbf{x}_{\parallel_{ \mathbf{p}},1},\mathbf{x}_{\parallel_{ \mathbf{p}},2} )) \cdot
\mathcal{O}_{ \mathbf{p}}^{-1} \mathbf{\hat{\phi}} ( \mathbf{x}_{\parallel_{ \mathbf{p}},1},\mathbf{x}_{\parallel_{ \mathbf{p}},2} ) }{\nabla \xi ( \mathbf{\eta}_{ \mathbf{p}}   ( \mathbf{x}_{\parallel_{ \mathbf{p}},1},\mathbf{x}_{\parallel_{\mathbf{p}},2} )  ) \cdot \mathcal{O}_{\mathbf{p}}^{-1}  {\hat{r}} ( \mathbf{x}_{\parallel_{\mathbf{p}},1},\mathbf{x}_{\parallel_{\mathbf{p}},2} )   }\\
&=
\frac{-  [R_{\mathbf{p}}  ( \mathbf{x}_{\parallel_{\mathbf{p}},1},\mathbf{x}_{\parallel_{\mathbf{p}},2} )]^{2}
\nabla \xi( \mathbf{\eta}_{\mathbf{p}}  ( \mathbf{x}_{\parallel_{\mathbf{p}},1},\mathbf{x}_{\parallel_{\mathbf{p}},2} )) \cdot
\mathcal{O}_{\mathbf{p}}^{-1} \mathbf{\hat{\phi}} ( \mathbf{x}_{\parallel_{\mathbf{p}},1},\mathbf{x}_{\parallel_{\mathbf{p}},2} ) }{ \nabla \xi ( \mathbf{\eta}_{\mathbf{p}}   ( \mathbf{x}_{\parallel_{\mathbf{p}},1},\mathbf{x}_{\parallel_{\mathbf{p}},2} )   ) \cdot  \mathbf{\eta}_{\mathbf{p}} ( \mathbf{x}_{\parallel_{\mathbf{p}},1},\mathbf{x}_{\parallel_{\mathbf{p}},2} )  },
\end{split}
\end{equation}
where $ \nabla \xi ( \mathbf{\eta}_{\mathbf{p}}   ( \mathbf{x}_{\parallel_{\mathbf{p}},1},\mathbf{x}_{\parallel_{\mathbf{p}},2} )   ) \cdot  \mathbf{\eta}_{\mathbf{p}} ( \mathbf{x}_{\parallel_{\mathbf{p}},1},\mathbf{x}_{\parallel_{\mathbf{p}},2} )\neq 0$. 

And by (\ref{Dr})
\begin{equation}\notag
\begin{split}
\frac{\partial   \mathbf{\eta}_{ \mathbf{p}} }{\partial\mathbf{x}_{\parallel_{ \mathbf{p}},1}}  ( \mathbf{x}_{\parallel_{ \mathbf{p}},1},\mathbf{x}_{\parallel_{ \mathbf{p}},2} )& = \frac{\partial R_{ \mathbf{p}} }{\partial \mathbf{x}_{\parallel_{ \mathbf{p}},1}} \mathcal{O}_{ \mathbf{p}}^{-1}  {\hat{r}} + \sin(\mathbf{x}_{\parallel_{ \mathbf{p}},2})R_{\mathbf{p}} \mathcal{O}^{-1}_{ \mathbf{p}} \mathbf{\hat{\theta}},\\
\frac{\partial    \mathbf{\eta}_{ \mathbf{p}}}{\partial \mathbf{x}_{\parallel_{\mathbf{p}},2}}  ( \mathbf{x}_{\parallel_{\mathbf{p}},1},\mathbf{x}_{\parallel_{ \mathbf{p}},2} ) & = \frac{\partial R_{ \mathbf{p}}}{\partial \mathbf{x}_{\parallel_{ \mathbf{p}},2}} \mathcal{O}_{ \mathbf{p}}^{-1}  {\hat{r}} +
R_{ \mathbf{p}} \mathcal{O}^{-1}_{ \mathbf{p}} \mathbf{\hat{\phi}}.
\end{split}
\end{equation}
Directly we check a non-degenerate condition (\ref{nondegenerate_eta})
\begin{equation}\notag
\begin{split}
 &\frac{\partial   \mathbf{\eta}_{\mathbf{p}} }{\partial\mathbf{x}_{\parallel_{\mathbf{p}},1}}  ( \mathbf{x}_{\parallel_{\mathbf{p}} }  ) \times \frac{\partial   \mathbf{\eta}_{\mathbf{p}} }{\partial\mathbf{x}_{\parallel_{\mathbf{p}},2}}  ( \mathbf{x}_{\parallel_{\mathbf{p}} }  )\\
&=  R_{\mathbf{p}} \frac{\partial R_{\mathbf{p}}}{\partial \mathbf{x}_{\parallel_{\mathbf{p}},1}} \mathcal{O}_{\mathbf{p}}^{-1} \mathbf{\hat{\theta}}
+ \sin (\mathbf{x}_{\parallel_{\mathbf{p}},2}) R_{p} \frac{\partial R_{\mathbf{p}}}{\partial \mathbf{x}_{\parallel_{\mathbf{p}},2}} \mathcal{O}_{\mathbf{p}}^{-1} \mathbf{\hat{\phi}}
- \sin (\mathbf{x}_{\parallel_{\mathbf{p}},2}) R_{\mathbf{p}}^{2} \mathcal{O}_{\mathbf{p}}^{-1}  {\hat{r}}\neq 0.
\end{split}
\end{equation}

\vspace{8pt}

\textit{Proof of (ii) of Lemma \ref{chart_lemma}.} 
We fix $\mathbf{p}=(z,w)$ and drop $\mathbf{p}-$index (for the chart) in this step. Define
\begin{equation}\label{Phi}
\Phi_{1} : [0, \infty) \times [0,2\pi)\times (0,\pi)  \rightarrow  \bar{\Omega}\backslash \mathcal{L}_{\mathbf{p}}, \ \ \   \Phi_{ 1}(\mathbf{x}_{\perp  }, \mathbf{x}_{\parallel   }) =   \mathbf{\eta  }(\mathbf{x}_{\parallel   }) + \mathbf{x}_{\perp    } [-\mathbf{n} (\mathbf{x}_{\parallel  })] .
\end{equation}
Note that this mapping is surjective: For any $x\in\bar{\Omega}\backslash \mathcal{L}_{\mathbf{p}},$ there exists (could be several) $\mathbf{x}_{\parallel  }\in [0, 2\pi) \times (0,\pi)$ satisfying $|x- \eta (\mathbf{x}_{\parallel })|^{2}=\min_{ \mathbf{y}_{\parallel}\in   \mathbb{S}^{2}  }|x- \mathbf{\eta}  (\mathbf{y}_{\parallel})|^{2}$ ($\mathbb{S}^{2}$ is compact).  Then $(x- \mathbf{\eta}(\mathbf{x}_{\parallel}^{*}))\cdot \frac{\partial \mathbf{\eta}}{\partial \mathbf{x}_{\parallel,i}}(\mathbf{x}_{\parallel}^{*})=0$ for $i=1,2$. Since $\nabla \mathbf{\eta}(\mathbf{x}_{\parallel}) \neq 0$ from (\ref{nondegenerate_eta}) and $\xi( \mathbf{\eta}(\mathbf{x}_{\parallel}))=0$, we have
$0\equiv \nabla \xi(\mathbf{\eta}(\mathbf{x}_{\parallel})) \cdot \frac{\partial \mathbf{\eta}}{\partial \mathbf{x}_{\parallel,i}} (\mathbf{x}_{\parallel})$. Due to (\ref{nondegenerate_eta}), we conclude $\frac{x-\mathbf{\eta}(\mathbf{x}_{\parallel}^{*})}{| x-\mathbf{\eta}(\mathbf{x}_{\parallel}^{*})   |} = [-\mathbf{n}(\mathbf{x}_{\parallel}^{*})]$ and
$x= \mathbf{\eta}(\mathbf{x}_{\parallel}^{*}) + (x-\mathbf{\eta}(\mathbf{x}_{\parallel}^{*}) ) = \mathbf{\eta}(\mathbf{x}_{\parallel}^{*})  +|x-\mathbf{\eta}(\mathbf{x}_{\parallel}^{*}) | [- \mathbf{n}(\mathbf{x}_{\parallel}^{*})].$

Since $\mathbf{\eta}$ and $\xi$ (therefore $n$ and $\mathbf{n}$) are smooth, the $\Phi_{1}$ is smooth. The Jacobian matrix is
\begin{equation}\label{jac_Phi1}
\frac{\partial \Phi_{1}(\mathbf{x}_{\perp}, \mathbf{x}_{\parallel})}{\partial (\mathbf{x}_{\perp}, \mathbf{x}_{\parallel})}
= \bigg[\begin{array}{ccc} && \\  -\mathbf{n}(\mathbf{x}_{\parallel}) & \substack{ \frac{\partial \mathbf{\eta}}{\partial \mathbf{x}_{\parallel,1}} (\mathbf{x}_{\parallel}) \\ + \mathbf{x}_{\perp} [- \frac{\partial \mathbf{n}}{\partial \mathbf{x}_{\parallel,1}} (\mathbf{x}_{\parallel}) ] } &  \substack{ \frac{\partial \mathbf{\eta}}{\partial \mathbf{x}_{\parallel,2}} (\mathbf{x}_{\parallel})  \\  + \mathbf{x}_{\perp} [ -\frac{\partial \mathbf{n}}{\partial \mathbf{x}_{\parallel,2}} (\mathbf{x}_{\parallel}) ] }\\ &&\end{array}\bigg]_{3\times 3},
\end{equation}
where $\bigg[-\mathbf{n}(\mathbf{x}_{\parallel})\bigg],
\bigg[\substack{
\frac{\partial \mathbf{\eta}}{\partial \mathbf{x}_{\parallel,1}} (\mathbf{x}_{\parallel}) \\
 + \mathbf{x}_{\perp} [-\frac{\partial \mathbf{n}}{\partial \mathbf{x}_{\parallel,1}} (\mathbf{x}_{\parallel}) ]
}
\bigg]
,\bigg[ \substack{\frac{\partial \mathbf{\eta}}{\partial \mathbf{x}_{\parallel,2}} (\mathbf{x}_{\parallel}) \\ + \mathbf{x}_{\perp} [-\frac{\partial \mathbf{n}}{\partial \mathbf{x}_{\parallel,2}} (\mathbf{x}_{\parallel})  ]}   \bigg]$ are column vectors in $\mathbb{R}^{3}$. By the basic linear algebra, the Jacobian (a determinant of the Jacobian matrix) equals
\begin{equation}\notag
\begin{split}
-\mathbf{n}  \cdot (  \frac{\partial \eta}{\partial \mathbf{x}_{\parallel,1}}
 \times   \frac{\partial \eta}{\partial \mathbf{x}_{\parallel,2}}  ) + \mathbf{x}_{\perp}\mathbf{ n}\cdot (   \frac{\partial  \mathbf{n}  }{\partial \mathbf{x}_{\parallel,1}}      \times  \frac{\partial \eta}{\partial \mathbf{x}_{\parallel,2}}  ) - \mathbf{x}_{\perp} \mathbf{n}\cdot  (   \frac{\partial  \mathbf{n}  }{\partial \mathbf{x}_{\parallel,2}}      \times  \frac{\partial \eta}{\partial \mathbf{x}_{\parallel,1}}  ) -|\mathbf{x}_{\perp}|^{2} \mathbf{n} \cdot 
 (  \frac{\partial \mathbf{n}}{ \partial \mathbf{x}_{\parallel,1}}   \times   \frac{\partial \mathbf{n}}{ \partial \mathbf{x}_{\parallel,2}}).\\
\end{split}
\end{equation}
We use the facts $\nabla\mathbf{\eta}(\mathbf{x}_{\parallel}) \neq 0$ and $\xi(\mathbf{\eta}(\mathbf{x}_{\parallel}))=0$ and
$$0\equiv \nabla \xi(\mathbf{\eta}(\mathbf{x}_{\parallel})) \cdot \frac{\partial \mathbf{\eta}}{\partial \mathbf{x}_{\parallel,i}} (\mathbf{x}_{\parallel})  = |\nabla \xi(\mathbf{\eta}(\mathbf{x}_{\parallel}))| \left(\mathbf{n} (\mathbf{x}_{\parallel})  \cdot \frac{\partial \mathbf{\eta}}{\partial \mathbf{x}_{\parallel,i}} (\mathbf{x}_{\parallel})
\right),$$ and therefore
\[
 -\mathbf{n}(\mathbf{x}_{\parallel}) \cdot \left(  \frac{\partial \mathbf{\eta}}{\partial \mathbf{x}_{\parallel,1}} (\mathbf{x}_{\parallel})\times   \frac{\partial \mathbf{\eta}}{\partial \mathbf{x}_{\parallel,2}} (\mathbf{x}_{\parallel}) \right) \neq 0, \ \ \ \text{for all } \mathbf{x}_{\parallel} \in  [0,2\pi)\times (0,\pi),
\]
to conclude that there exists small $\delta>0$ such that if $|\mathbf{x}_{\perp}| \leq \delta$ and $\mathbf{x}_{\parallel} \in [0, 2\pi ) \times (0,\pi)$ then
\[
\text{det} \left(\frac{\partial \Phi_{1}(\mathbf{x}_{\perp}, \mathbf{x}_{\parallel})}{\partial (\mathbf{x}_{\perp}, \mathbf{x}_{\parallel})} \right)
 =- \mathbf{n}(\mathbf{x}_{\parallel}) \cdot \left(  \frac{\partial \mathbf{\eta}}{\partial \mathbf{x}_{\parallel,1}} (\mathbf{x}_{\parallel})\times   \frac{\partial \mathbf{\eta}}{\partial \mathbf{x}_{\parallel,2}} (\mathbf{x}_{\parallel}) \right) + O_{\xi}(|\mathbf{x}_{\perp}|) \neq 0.
\]
We use the inverse function theorem and we choose an inverse map $$\Phi_{1}^{-1} : \Phi_{1}([0,\delta) \times [0,2\pi) \times (0,\pi))\rightarrow [0,\delta) \times [0,2\pi) \times (0,\pi).$$
Note that in general there are infinitely many inverse maps.

If $x\in  \Phi_{1}([0,\delta) \times [0,2\pi) \times (0,\pi))$ then
\[
 \Phi_{1}^{-1}(x):=(\mathbf{x}_{\perp}, \mathbf{x}_{\parallel}) \ \ \ \text{and} \ \ \ x= \mathbf{\eta}(\mathbf{x}_{\parallel}) + \mathbf{x}_{\perp} [-\mathbf{n}(\mathbf{x}_{\parallel})].
\]
Since $\Phi_1$ is surjective onto $\bar{\Omega}\backslash \mathcal{L}_{\mathbf{p}}$, for $x\in\bar{\Omega}\backslash \mathcal{L}_{\mathbf{p}}$ and $\mathbf{x}_{\perp}\geq 0,$
\begin{equation}\notag
\begin{split}
\xi(x) &=\xi ( \mathbf{\eta}(\mathbf{x}_{\parallel}) + \mathbf{x}_{\perp}[- \mathbf{n}(\mathbf{x}_{\parallel})])\\
& = \xi(\mathbf{\eta}(\mathbf{x}_{\parallel})) + \int^{\mathbf{x}_{\perp}}_{0} \frac{d}{d s} \xi( \mathbf{\eta}(\mathbf{x}_{\parallel}) + s[-\mathbf{n}(\mathbf{x}_{\parallel})]) \mathrm{d}s\\
& = \int^{\mathbf{x}_{\perp}}_{0} [- \mathbf{n}(\mathbf{x}_{\parallel})]\cdot \nabla \xi(\mathbf{\eta}(\mathbf{x}_{\parallel} ) + s [-\mathbf{n}(\mathbf{x}_{\parallel})])  \mathrm{d}s \\
&= \int^{\mathbf{x}_{\perp}}_{0} \Big\{ [-\mathbf{n}(\mathbf{x}_{\parallel})]\cdot \nabla \xi(\mathbf{\eta}(\mathbf{x}_{\parallel})) + \int^{s}_{0} \mathbf{n}(\mathbf{x}_{\parallel})  \cdot \nabla^{2} \xi (\mathbf{\eta}(\mathbf{x}_{\parallel})  + \tau [- \mathbf{n}(\mathbf{x}_{\parallel})] ) \cdot \mathbf{n}(\mathbf{x}_{\parallel}) \mathrm{d}\tau  \Big\}\mathrm{d}s,
\end{split}
\end{equation}
and by the convexity of $\xi$ we have the following equivalent relation:

For all $x\in \bar{\Omega}$ there exists (not uniquely) $(\mathbf{x}_{\perp}, \mathbf{x}_{\parallel}) \in [0,\infty) \times[0,2\pi)\times (0,\pi)$ satisfying $x=\mathbf{x}_{\perp} [-\mathbf{n}(\mathbf{x}_{\parallel})] + \mathbf{\eta}(\mathbf{x}_{\parallel})$. Then for all $(\mathbf{x}_{\perp}, \mathbf{x}_{\parallel})$ with $x= \mathbf{x}_{\perp}[- \mathbf{n}(\mathbf{x}_{\parallel})] + \mathbf{\eta}(\mathbf{x}_{\parallel})$ we have
\begin{equation}\label{xi_xperp}
\begin{split}
|\nabla \xi(\mathbf{\eta}(\mathbf{x_{\parallel}}))| |\mathbf{x}_{\perp}| - \frac{1}{C_{\xi}} \frac{|\mathbf{x}_{\perp}|^{2}}{2}
 \ &\leq \ |\xi(x)|
 = |\xi (\mathbf{\eta}(\mathbf{x}_{\parallel}) + \mathbf{x}_{\perp}[- \mathbf{n}(\mathbf{x}_{\parallel})]) |\\
 \ &\leq  \ |\nabla \xi(\mathbf{\eta}(\mathbf{x}_{\parallel}))| | \mathbf{x}_{\perp}  | -   C_{\xi} \frac{|\mathbf{x}_{\perp}|^{2}}{2}.
 \end{split}
\end{equation}
Therefore there exists $0<C_{1} \ll 1$ such that if $|\xi(x)| \leq C_{1} \delta$ then $|\mathbf{x}_{\perp}| < \delta$ and hence there exists unique $(\mathbf{x}_{\perp}, \mathbf{x}_{\parallel})$ and all the above computations hold.

Next we define
\[
\Phi( \mathbf{x}_{\perp}, \mathbf{x}_{\parallel}, \mathbf{v}_{\perp}, \mathbf{v}_{\parallel}) =\left(\begin{array}{cc}{\mathbf{x}_{\perp} [-\mathbf{n}(\mathbf{x}_{\parallel})] + \mathbf{\eta}(\mathbf{x}_{\parallel})}  \\ { \mathbf{v}_{\perp} [-\mathbf{n}(\mathbf{x}_{\parallel})] + \mathbf{v}_{\parallel} \cdot \nabla_{ \mathbf{x}_{\parallel}}\mathbf{ \eta}(\mathbf{x}_{\parallel})- \mathbf{x}_{\perp} \mathbf{v}_{\parallel} \cdot \nabla_{\mathbf{x}_{\parallel}} \mathbf{n}(\mathbf{x}_{\parallel})} \end{array} \right).
\]
The Jacobian matrix is
\begin{equation}\label{jac_Phi}
\frac{\partial \Phi( \mathbf{x}_{\perp}, \mathbf{x}_{\parallel}, \mathbf{v}_{\perp}, \mathbf{v}_{\parallel})}{\partial ( \mathbf{x}_{\perp}, \mathbf{x}_{\parallel}, \mathbf{v}_{\perp}, \mathbf{v}_{\parallel})} = \left[\begin{array}{ccc|ccc}
 &
 \frac{\partial \Phi_{1}(\mathbf{x}_{\perp}, \mathbf{x}_{\parallel})}{\partial (\mathbf{x}_{\perp}, \mathbf{x}_{\parallel})}
 &
  & &\mathbf{0}_{3,3} & \\ \hline
-\mathbf{v}_{\parallel}  \cdot \nabla_{\mathbf{x}_{\parallel}} \mathbf{n}(\mathbf{x}_{\parallel})&
 \substack{ -\mathbf{v}_{\perp} \frac{ \partial \mathbf{n}}{\partial \mathbf{x}_{\parallel,1}}(\mathbf{x}_{\parallel}) \\
 + \mathbf{v}_{\parallel} \cdot \nabla_{\mathbf{x}_{\parallel}} \frac{\partial \mathbf{\eta}}{\partial \mathbf{x}_{\parallel,1}} (\mathbf{x}_{\parallel})  \\
  - \mathbf{x}_{\perp} \mathbf{v}_{\parallel} \cdot \nabla_{\mathbf{x}_{\parallel}} \frac{\partial \mathbf{n}}{\partial \mathbf{x}_{\parallel,1}}(\mathbf{x}_{\parallel})  }&
\substack{ -\mathbf{v}_{\perp} \frac{\partial \mathbf{n}}{\partial \mathbf{x}_{\parallel,2}}(\mathbf{x}_{\parallel}) \\ + \mathbf{v}_{\parallel} \cdot \nabla_{\mathbf{x}_{\parallel}} \frac{\partial \mathbf{\eta}}{\partial \mathbf{x}_{\parallel,2}} (\mathbf{x}_{\parallel})  \\ - \mathbf{x}_{\perp} \mathbf{v}_{\parallel} \cdot \nabla_{\mathbf{x}_{\parallel}} \frac{\partial \mathbf{n}}{\partial \mathbf{x}_{\parallel,2}}(\mathbf{x}_{\parallel})  }
&
&
 \frac{\partial \Phi_{1}(\mathbf{x}_{\perp}, \mathbf{x}_{\parallel})}{\partial (\mathbf{x}_{\perp}, \mathbf{x}_{\parallel})}
&
\end{array}\right].
\end{equation}
The Jacobian (a determinant of the Jacobian matrix) equals
\[
\text{det} \left(  \frac{\partial \Phi( \mathbf{x}_{\perp}, \mathbf{x}_{\parallel}, \mathbf{v}_{\perp}, \mathbf{v}_{\parallel})}{\partial ( \mathbf{x}_{\perp}, \mathbf{x}_{\parallel}, \mathbf{v}_{\perp}, \mathbf{v}_{\parallel})} \right) = \left( \text{det} \left( \frac{\partial \Phi_{1}(\mathbf{x}_{\perp}, \mathbf{x}_{\parallel})}{\partial (\mathbf{x}_{\perp}, \mathbf{x}_{\parallel})} \right)\right)^{2} \neq 0,
\]
for $|\xi(x)| \leq \delta$ (and therefore $|\mathbf{x}_{\perp}| \leq C\delta$) and $\mathbf{x}_{\parallel,{2}} \in (0,\pi)$. By the inverse function theorem we have the inverse mapping $\Phi^{-1}$.

\vspace{8pt}

\textit{Proof of (iii) of Lemma \ref{chart_lemma}.} From $\dot{v}=0$ and the second equation of (\ref{polar}) equals
\begin{equation}\label{dotv0}
\begin{split}
0  = & \dot{\mathbf{v}}_{\perp} (s) [-\mathbf{n}(\mathbf{x}_{\parallel}(s))] -2 \mathbf{v}_{\perp} (s)\mathbf{v}_{\parallel} \cdot \nabla \mathbf{n}(\mathbf{x}_{\parallel}(s))+ \dot{\mathbf{v}}_{\parallel} (s)\cdot \nabla \mathbf{\eta}(\mathbf{x}_{\parallel}(s))\\
& + \mathbf{v}_{\parallel} \cdot \nabla^{2} \mathbf{\eta}(\mathbf{x}_{\parallel} ) \cdot \mathbf{v}_{\parallel}   - \mathbf{x}_{\perp} \dot{\mathbf{v}}_{\parallel} \cdot \nabla \mathbf{n}(\mathbf{x}_{\parallel}) + \mathbf{x}_{\perp} \mathbf{v}_{\parallel} \cdot \nabla^{2} \mathbf{n}(\mathbf{x}_{\parallel}) \cdot \mathbf{v}_{\parallel}.
\end{split}
\end{equation}
We take the inner product with $\mathbf{n}(\mathbf{x}_{\parallel}(s))$ to the above equation to have
\begin{equation}\label{dot_v_perp}
\dot{\mathbf{v}}_{\perp}(s) =   [ \mathbf{v}_{\parallel} \cdot \nabla^{2} \mathbf{ \eta}(\mathbf{x}_{\parallel}) \cdot \mathbf{v}_{\parallel}] \cdot \mathbf{n}(\mathbf{x}_{\parallel})+ \mathbf{x}_{\perp} [ \mathbf{v}_{\parallel} \cdot \nabla^{2} \mathbf{n}(\mathbf{x}_{\parallel})\cdot \mathbf{v}_{\parallel}] \cdot \mathbf{n}( \mathbf{x}_{\parallel}) := F_{\perp} ( \mathbf{v}_{\perp}, \mathbf{v}_{\parallel}, \mathbf{x}_{\parallel} ),
 \end{equation}
where we have used the fact $\nabla \mathbf{n} \perp \mathbf{n}$ and $\nabla \mathbf{\eta} \perp \mathbf{n}$.

Since $0=\xi(\mathbf{\eta}(\mathbf{x}_{\parallel}))$ we take $\mathbf{x}_{\parallel,i}$ and $\mathbf{x}_{\parallel,j}$ derivatives to have
\[
0=\partial_{\mathbf{x}_{\parallel,j}} \big[  \sum_{k} \partial_{k}\xi \partial_{\mathbf{x}_{\parallel},i} \mathbf{\eta}_{k} \big]
 = \sum_{k,m} \partial_{k}\partial_{m} \xi \partial_{\mathbf{x}_{\parallel, j}} \mathbf{\eta}_{m} \partial_{\mathbf{x}_{\parallel,i}} \mathbf{\eta}_{k} + \sum_{k}
 \partial_{k} \xi \partial_{\mathbf{x}_{\parallel,i}}\partial_{\mathbf{x}_{\parallel,j}} \mathbf{\eta}_{k}   ,
\]
we have from the convexity (\ref{convex})
\[
\big[\mathbf{v}_{\parallel} \cdot \nabla^{2} \mathbf{\eta} \cdot \mathbf{v}_{\parallel}\big] \cdot \mathbf{n} =
\sum_{i,j,k}\frac{ \mathbf{v}_{\parallel,{i}} \partial_{k}\xi \partial_{i  } \partial_{j} \mathbf{\eta}_{k} \mathbf{v}_{\parallel,j}  }{| \nabla \xi|} = - \sum_{i,j,k,m} \frac{  \{ \mathbf{v}_{\parallel,i} \partial_{i} \mathbf{\eta}_{m} \} \partial_{k} \partial_{m} \xi \{  \partial_{j} \mathbf{\eta}_{m} \mathbf{v}_{\parallel,j} \} }{|\nabla \xi|}\lesssim_{\xi} -|\mathbf{v}_{\parallel}|^{2}.
\]

Define $a_{ij}(\mathbf{x}_{\parallel})$ via
\begin{equation}\notag
\left[\begin{array}{cc} a_{11} & a_{12} \\ a_{21} & a_{22} \end{array} \right] = \left[\begin{array}{cc}  \partial_{1} \mathbf{n} \cdot \partial_{1} \mathbf{n} & \partial_{1 } \mathbf{n} \cdot \partial _{2} \mathbf{n} \\ \partial_{2}   \mathbf{n}  \cdot \partial_{1}   \mathbf{n}  & \partial_{2 }   \mathbf{n}  \cdot \partial_{2}   \mathbf{n}  \end{array}\right] \left[\begin{array}{cc}  \partial_{1} \mathbf{\eta} \cdot \partial_{1}   \mathbf{\eta}  & \partial_{1 }  \mathbf{\eta}  \cdot \partial _{2}   \mathbf{\eta}  \\ \partial_{2}   \mathbf{\eta}  \cdot \partial_{1}   \mathbf{\eta} & \partial_{2 }   \mathbf{\eta} \cdot \partial_{2}  \mathbf{\eta}  \end{array}\right]^{-1},
\end{equation}
where $\text{det} (\partial_{i} \mathbf{\eta} \cdot \partial_{j} \mathbf{\eta}) = |\partial_{1} \mathbf{\eta}\times \partial_{2} \mathbf{\eta}|^{2} \neq 0$ due to (\ref{nondegenerate_eta}). Then $\nabla \mathbf{n}$ is generated by $\nabla \mathbf{\eta}$ :
\[
-\partial_{i} \mathbf{n} (\mathbf{x}_{\parallel}) = \sum_{k} a_{ik}(\mathbf{x}_{\parallel}) \partial_{k} \mathbf{\eta}(\mathbf{x}_{\parallel}).
\]
We take the inner product (\ref{dotv0}) with $(-1)^{i+1} ( \mathbf{n}(\mathbf{x}_{\parallel}) \times \partial_{i} \mathbf{n}(\mathbf{x}_{\parallel}))$ to have
\begin{equation}\notag
\begin{split}
&\sum_{k} ( \delta_{ki} + \mathbf{x}_{\perp} a_{ki}) \dot{\mathbf{v}}_{\parallel,k} \\
&=\frac{(-1)^{i+1}}{-\mathbf{n}(\mathbf{x}_{\parallel}) \cdot (\partial_{1} \mathbf{\eta}(\mathbf{x}_{\parallel}) \times \partial_{2} \mathbf{\eta}(\mathbf{x}_{\parallel}))} \\
& \    \times \Big\{-
2 \mathbf{v}_{\perp} \mathbf{v}_{\parallel} \cdot \nabla \mathbf{n} (\mathbf{x}_{\parallel}) + \mathbf{v}_{\parallel} \cdot  \nabla^{2} \mathbf{\eta}(\mathbf{x}_{\parallel}) \cdot \mathbf{v}_{\parallel} - \mathbf{x}_{\perp} \mathbf{v}_{\parallel} \cdot \nabla^{2} \mathbf{n}(\mathbf{x}_{\parallel}) \cdot \mathbf{v}_{\parallel}
\Big\} \cdot ( -\mathbf{n} (\mathbf{x}_{\parallel}) \times \partial_{i+1} \mathbf{\eta}(\mathbf{x}_{\parallel})),
\end{split}
\end{equation}
where we used the notational convention for $\partial_{i+1} \mathbf{\eta}$, the index $i+1 \ \text{mod } 2$ . For $|\xi(x)| \ll 1$(and therefore $|\mathbf{x}_{\perp}|\ll 1$) the matrix $\delta_{ki} + \mathbf{x}_{\perp} a_{ki}$ is invertible: there exists the inverse matrix $G_{ij}$ such that $\sum_{i} (\delta_{ki} + \mathbf{x}_{\perp} a_{ki}(\mathbf{x}_{\parallel})) G_{ij} (\mathbf{x}_{\perp}, \mathbf{x}_{\parallel}) = \delta_{kj}.$ Therefore we have
\begin{equation}\label{dot_v_||}
\begin{split}
\dot{\mathbf{v}}_{\parallel,j} &= \sum_{i} G_{ij} (\mathbf{x}_{\perp}, \mathbf{x}_{\parallel})\frac{(-1)^{i+1}}{-\mathbf{n}(\mathbf{x}_{\parallel}) \cdot (\partial_{1} \mathbf{\eta}(\mathbf{x}_{\parallel}) \times \partial_{2} \mathbf{\eta}(\mathbf{x}_{\parallel}))}\\
& \ \ \ \times  \Big\{-
2 \mathbf{v}_{\perp} \mathbf{v}_{\parallel} \cdot \nabla \mathbf{n} (\mathbf{x}_{\parallel}) + \mathbf{v}_{\parallel} \cdot  \nabla^{2}\mathbf{ \eta}(\mathbf{x}_{\parallel}) \cdot \mathbf{v}_{\parallel} - \mathbf{x}_{\perp} \mathbf{v}_{\parallel} \cdot \nabla^{2} \mathbf{n}(\mathbf{x}_{\parallel}) \cdot \mathbf{v}_{\parallel}
\Big\}\\
& \ \ \ \ \  \cdot (-\mathbf{n} (\mathbf{x}_{\parallel}) \times \partial_{i+1} \mathbf{\eta}(\mathbf{x}_{\parallel}))\\
&:= F_{\parallel,j}(\mathbf{x}_{\perp}, \mathbf{x}_{\parallel}, \mathbf{v}_{\perp}, \mathbf{v}_{\parallel}).
\end{split}
\end{equation}
Here
\begin{equation}\label{G}
\begin{split}
&\left[\begin{array}{cc}G_{11} & G_{12} \\ G_{21} & G_{22} \end{array} \right] \\
& = \frac{1}{1 + \mathbf{x}_{\perp} (a_{11} + a_{22}) + (\mathbf{x}_{\perp})^{2} (a_{11} a_{22} - a_{12} a_{21})} \left[\begin{array}{cc}  1+ \mathbf{x}_{\perp} a_{22} & -\mathbf{x}_{\perp} a_{12} \\ - \mathbf{x}_{\perp} a_{21} & 1+ \mathbf{x}_{\perp} a_{11}\end{array} \right],
\\
& \left[\begin{array}{cc} a_{11} & a_{12} \\ a_{21} & a_{22} \end{array} \right] \\ &= \frac{1}{|\partial_{1} \mathbf{\eta}|^{2} |\partial_{2} \mathbf{\eta}|^{2} - (\partial_{1}\mathbf{ \eta} \cdot \partial_{2} \mathbf{\eta})^{2}}
 \\
& \ \ \times
\left[\begin{array}{cc}
|\partial_{1} \mathbf{n}|^{2} |\partial_{2} \mathbf{\eta}|^{2} -(\partial_{1}\mathbf{n} \cdot \partial_{2} \mathbf{n}) (\partial_{1} \mathbf{\eta} \cdot \partial_{2} \mathbf{\eta})  &- |\partial_{1}\mathbf{n}|^{2} (\partial_{1} \mathbf{\eta} \cdot \partial_{2} \mathbf{\eta}) + (\partial_{1} \mathbf{n} \cdot \partial_{2} \mathbf{n}) |\partial_{1}\mathbf{ \eta}|^{2} \\
(\partial_{1} \mathbf{n} \cdot \partial_{2} \mathbf{n}) |\partial_{2} \mathbf{\eta}|^{2} - |\partial_{2} \mathbf{n}|^{2} (\partial_{1} \mathbf{\eta} \cdot \partial_{2}\mathbf{ \eta}) & - (\partial_{1} \mathbf{n} \cdot \partial_{2} \mathbf{n}) (\partial_{1}\mathbf{ \eta}\cdot \partial_{2} \mathbf{\eta}) + |\partial_{2} \mathbf{n}|^{2}| \partial_{1} \mathbf{\eta}|^{2}
 \end{array} \right].
\end{split}
\end{equation}

\vspace{8pt}

\textit{Proof of (iv) of Lemma \ref{chart_lemma}.} Let $\mathbf{q}=(y,u) \in \partial\Omega \times \mathbb{S}^{2}$ with $n(y)\cdot u=0$ and $\mathbf{p}\sim \mathbf{q}.$ First we claim
\begin{equation}\label{pq}
\begin{split}
\mathbf{x}_{\perp_{\mathbf{p}}}& = \mathbf{x}_{\perp_{\mathbf{q}}},\\
\mathbf{\eta}_{\mathbf{p}}  (\mathbf{x}_{\parallel_{\mathbf{p}}})& = \mathbf{\eta}_{\mathbf{q}} (\mathbf{x}_{\parallel_{\mathbf{q}}}),\\
\mathbf{v}_{\perp_{\mathbf{p}}} &= \mathbf{v}_{\perp_{\mathbf{q}}},\\
\mathbf{v}_{\parallel_{\mathbf{p}}} \cdot \nabla \mathbf{\eta}_{\mathbf{p}} (\mathbf{x}_{\parallel_{\mathbf{p}}}) -  \mathbf{x}_{\perp_{\mathbf{p}}} \mathbf{v}_{\parallel_{\mathbf{p}}} \cdot \nabla \mathbf{n}_{\mathbf{p}} (\mathbf{x}_{\parallel_{\mathbf{p}}}) & =
\mathbf{v}_{\parallel_{\mathbf{q}}} \cdot \nabla \mathbf{\eta}_{\mathbf{q}} (\mathbf{x}_{\parallel_{\mathbf{q}}}) -  \mathbf{x}_{\perp_{\mathbf{q}}} \mathbf{v}_{\parallel_{\mathbf{q}}} \cdot \nabla \mathbf{n}_{\mathbf{q}} (\mathbf{x}_{\parallel_{\mathbf{q}}}).
\end{split}
\end{equation}

Once we show the first two equalities then the third and fourth equalities are clearly valid because $\mathbf{n}_{\mathbf{p}} \perp \mathbf{v}_{\parallel}  \cdot \nabla_{\mathbf{x}_{\parallel}} \eta_{\mathbf{p}}$ and $\mathbf{n}_{\mathbf{p}} \perp \mathbf{v}_{\parallel}  \cdot \nabla_{\mathbf{x}_{\parallel}}  \mathbf{n}_{\mathbf{p}}$ for all $\mathbf{v}_{\parallel} \in \mathbb{R}^{2}.$ (Since $\xi(\eta_{\mathbf{p}})=0$ we have $\mathbf{v}_{\parallel,i} \partial_{\mathbf{x}_{\parallel,i}} [\xi(\eta_{\mathbf{p}}(\mathbf{x}_{\parallel,1} , \mathbf{x}_{\parallel,2}))]= 
\mathbf{v}_{\parallel,i} \partial_{\mathbf{x}_{\parallel,i}} \eta_{\mathbf{p}} \cdot \nabla_{\mathbf{x}_{\parallel}} \xi =0,$ and since $\mathbf{n}_{\mathbf{p}} \cdot \mathbf{n}_{\mathbf{p}} =1$ we have $\mathbf{n}_{\mathbf{p}} \cdot [\mathbf{v}_{\parallel} \cdot \nabla_{\mathbf{x}_{\parallel}} \mathbf{n}_{\mathbf{p}} ]=0)$

Now we prove the first two equalities of (\ref{pq}) and it suffices to prove the second one. And it suffices to show that for $x\in \bar{\Omega}$ with $|\xi(x)| \ll 1$ there exists a unique $x^{*} \in \partial\Omega \cap B(x,\delta)$ for some $0<\delta \ll 1$ such that 
\begin{equation}\label{x*}
|x-x^{*}|^{2} = \min_{y \in\partial\Omega, y \sim x} |x-y|^{2}.
\end{equation}
By the definition of $(\ref{Phi})$ the uniqueness of such $x^{*}$ in (\ref{x*}) implies $\eta_{\mathbf{p}}(\mathbf{x}_{\parallel_{\mathbf{p}}})= x^{*} = \eta_{\mathbf{q}}(\mathbf{x}_{\parallel_{\mathbf{q}}}).$

The existence of such $x^{*} \in \partial \Omega$ is clear from the compactness of $\partial\Omega$. Without loss of generality (up to rotation) we may assume $\partial_{x_{3}}\xi(y) \neq 0$ for $y\sim x^{*}$ and $\partial_{x_{1}}\xi(x^{*})=0= \partial_{x_{2}}\xi(x^{*}).$ Then we can find the graph $a : (x_{1},x_{2}) \mapsto \mathbb{R}$ but $\xi(x_{1},x_{2},a(x_{1}, x_{2}))=0$ when $x^{*} = (x^{*}_{1}, x^{*}_{2}, a(x^{*}_{1}, x^{*}_{2}))\in \partial\Omega.$ By the implicit function theorem,
\begin{equation}\notag 
\partial_{x_{1}} a= - \partial_{x_{1}} \xi / \partial_{x_{3}} \xi, \ \ \ 
\partial_{x_{2}} a = - \partial_{x_{2}} \xi / \partial_{x_{3}} \xi,
\end{equation}
and $\partial_{x_{1}}a(x^{*}_{1}, x^{*}_{2})=0= \partial_{x_{2}} a(x^{*}_{1}, x^{*}_{2}).$ 

Clearly $x^{*}= (x^{*}_{1}, x^{*}_{2}, a(x^{*}_{1}, x^{*}_{2}))$ satisfies $\big{|} (x_{1},x_{2},x_{3})- (x_{1}^{*}, x_{2}^{*}, a(x_{1}^{*}, x_{2}^{*}))\big{|}\ll 1$ and 
\begin{equation}\notag
\begin{split}
\frac{\partial}{\partial x_{i}^{*}} \big{|} (x_{1},x_{2},x_{3})- (x_{1}^{*}, x_{2}^{*}, a(x_{1}^{*}, x_{2}^{*}))\big{|}^{2}
=
 - \Big\{
 (x_{i} - x_{i}^{*}) + (x_{3} - a(x_{1}^{*}, x_{2}^{*})) \frac{\partial a}{\partial x_{i}^{*}}(x_{1}^{*}, x_{2}^{*})
 \Big\}
=0, \ \ \text{for}  \ i=1,2,
\end{split}
\end{equation}
if and only if (\ref{x*}) holds. We take $x_{i}^{*}-$derivative to get
\begin{equation}\notag
1 +  \Big( \frac{\partial a}{\partial x_{i}^{*}}  (x^{*}_{1}, x^{*}_{2}) \Big)^{2}
- (x_{3}- a(x^{*}_{1}, x^{*}_{2})   ) \frac{\partial^{2} a}{\partial x_{i}^{*} \partial x_{i}^{*}}(x^{*}_{1}, x^{*}_{2})
= 1 -  (x_{3}- a(x^{*}_{1}, x^{*}_{2})   ) \frac{\partial^{2} a}{\partial x_{i}^{*} \partial x_{i}^{*}}(x^{*}_{1}, x^{*}_{2}) \neq 0,
\end{equation}
for $|x_{3} - a(x_{1}^{*}, x_{2}^{*})|    \ll_{\xi}1.$ Using the inverse function theorem we have a uniquely determined $x^{*} : \{ y \in \bar{\Omega} : y \sim x \} \rightarrow \partial\Omega \cap B(x,\delta)$. This proves our claim (uniqueness of $x^{*}$ in (\ref{x*})) and therefore we prove (\ref{pq}).

From the second equality of (\ref{pq}) and (\ref{x_parallel})
\begin{equation}\notag
\begin{split}
 {\hat{r}}(\mathbf{x}_{\parallel_{\mathbf{q}},1},  \mathbf{x}_{\parallel_{\mathbf{q}},2}) = \mathcal{O}_{\mathbf{q}} \mathcal{O}^{-1}_{\mathbf{p}}  {\hat{r}}(\mathbf{x}_{\parallel_{\mathbf{p}},1},  \mathbf{x}_{\parallel_{\mathbf{p}},2}).
\end{split}
\end{equation}
Therefore for $i=1,2,$
\begin{equation}\notag
\begin{split}
&\sum_{j=1,2}\frac{\partial {\hat{r}}}{\partial \mathbf{x}_{\parallel_{\mathbf{q}},j}} (\mathbf{x}_{\parallel_{\mathbf{q}}})  \frac{\partial \mathbf{x}_{\parallel_{\mathbf{q}},j}}{\partial \mathbf{x}_{\parallel_{\mathbf{p}},i}} = \mathcal{O}_{\mathbf{q}} \mathcal{O}_{\mathbf{p}}^{-1} \frac{\partial   {\hat{r}}}{\partial \mathbf{x}_{\parallel_{\mathbf{p}},i}}(\mathbf{x}_{\parallel_{\mathbf{p}}}) ,
\end{split}
\end{equation}
and from (\ref{Dr})
\begin{equation}\notag
\begin{split}
&\left[\begin{array}{c|c|c} & & \\
 -\sin (\mathbf{x}_{\parallel_{\mathbf{q}},2})   {\hat{r}}(\mathbf{x}_{\parallel_{\mathbf{q}}})
&  \sin(\mathbf{x}_{\parallel_{\mathbf{q}},2}) \mathbf{\hat{\theta}}(\mathbf{x}_{\parallel_{\mathbf{q}}})
& \mathbf{\hat{\phi}}(\mathbf{x}_{\parallel_{\mathbf{q}}})
\\
& &
\end{array} \right]
\left[\begin{array}{c|c}
0 &  \mathbf{0}_{1,2}  \\ \hline
\mathbf{0}_{2,1} & \frac{\partial \mathbf{x}_{\parallel_{\mathbf{q}}}}{\partial \mathbf{x}_{\parallel_{\mathbf{p}}}}
 \end{array} \right]_{3\times 3}\\
 &=
 \mathcal{O}_{\mathbf{q}} \mathcal{O}_{\mathbf{p}}^{-1}
  \left[\begin{array}{c|c|c} & & \\
\mathbf{0}_{3,1}
& \sin(\mathbf{x}_{\parallel_{\mathbf{p}},2}) \mathbf{\hat{\theta}}(\mathbf{x}_{\parallel_{\mathbf{p}}})  & \mathbf{\hat{\phi}}(\mathbf{x}_{\parallel_{\mathbf{p}}}) \\
& &
\end{array} \right],
\end{split}
\end{equation}
where we used $  \mathbf{\hat{\theta}} \times \mathbf{\hat{\phi}} = -   {\hat{r}}.$

For $\mathbf{x}_{\parallel_{\mathbf{p}},2}, \mathbf{x}_{\parallel_{\mathbf{q}},2} \notin\{ 0, \pi\}$,
\begin{equation}\notag
\begin{split}
 \left[\begin{array}{ccc}
 0 & 0 & 0\\
0& \frac{\partial \mathbf{x}_{\parallel_{\mathbf{q}},1}}{\partial \mathbf{x}_{\parallel_{\mathbf{p}},1}} & \frac{\partial \mathbf{x}_{\parallel_{\mathbf{q}},1}}{\partial \mathbf{x}_{\parallel_{\mathbf{p}},2}}    \\
0& \frac{\partial \mathbf{x}_{\parallel_{\mathbf{q}},2}}{\partial \mathbf{x}_{\parallel_{ \mathbf{p}},1}} & \frac{\partial \mathbf{x}_{\parallel_{\mathbf{q}},2}}{\partial \mathbf{x}_{\parallel_{\mathbf{p}},2}}
  \end{array} \right]
 &=  \left[\begin{array}{c}
   \frac{-1}{\sin (\mathbf{x}_{\parallel_{\mathbf{q}},2})}
 {\hat{r}}(\mathbf{x}_{\parallel_{\mathbf{q}}})^{T}\\
 \frac{1}{\sin (\mathbf{x}_{\parallel_{\mathbf{q}},2})}
\mathbf{\hat{\theta}}(\mathbf{x}_{\parallel_{\mathbf{q}}})^{T}\\
\mathbf{\hat{\phi}}(\mathbf{x}_{\parallel_{\mathbf{q}}})^{T}
 \end{array} \right]
 \mathcal{O}_{\mathbf{q}} \mathcal{O}_{\mathbf{p}}^{-1}
 \left[\begin{array}{c|c|c} & & \\
\mathbf{0}_{3,1}
& \sin(\mathbf{x}_{\parallel_{\mathbf{p}},2}) \mathbf{\hat{\theta}}(\mathbf{x}_{\parallel_{\mathbf{p}}})  & \mathbf{\hat{\phi}}(\mathbf{x}_{\parallel_{\mathbf{p}}}) \\
& &
\end{array} \right] .
\end{split}
\end{equation}

Here $\mathcal{O}_{ \mathbf{q}} = \mathcal{O}_{\mathbf{p}} + O_{\xi}(|\mathbf{p}-\mathbf{q}|) ,$ and $\sin(\mathbf{x}_{\parallel_{\mathbf{p}},2}) \mathbf{\hat{\theta}}(\mathbf{x}_{\parallel_{\mathbf{p}}})=  \sin(\mathbf{x}_{\parallel_{\mathbf{q}},2}) \mathbf{\hat{\theta}}(\mathbf{x}_{\parallel_{\mathbf{q}}})   + O_{\xi}(|\mathbf{p}-\mathbf{q}|)$ and $ \mathbf{\hat{\phi}}(\mathbf{x}_{\parallel_{\mathbf{p}}}) =  \mathbf{\hat{\phi}}(\mathbf{x}_{\parallel_{\mathbf{q}}}) +  O_{\xi}(|\mathbf{p}-\mathbf{q}|)$.

Therefore for $\mathbf{x}_{\parallel_{\mathbf{p}},2}, \mathbf{x}_{\parallel_{\mathbf{q}},2} \notin\{ 0, \pi\}$
\begin{equation}\label{Xx}
\bigg[ \frac{\partial \mathbf{x}_{\parallel_{\mathbf{ q}}}}{\partial \mathbf{x}_{\parallel_{\mathbf{p}}}}  \bigg]_{2\times 2} \ \lesssim \ \mathbf{Id}_{2,2} + O_{\xi}(|\mathbf{p}-\mathbf{q}|).
\end{equation}

From the third equality of (\ref{pq})
\begin{equation}\notag
\begin{split}
&\left[\begin{array}{c|c|c}
- \mathbf{n}_{\mathbf{q}} (\mathbf{x}_{\parallel_{\mathbf{q}}}) &
\substack{ \partial_{1} \mathbf{\eta}_{\mathbf{q}}(\mathbf{x}_{\parallel_{\mathbf{q}}})  \\ - \mathbf{x}_{\perp_{\mathbf{q}}} \partial_{1} \mathbf{n}_{\mathbf{q}}(\mathbf{x}_{\parallel_{\mathbf{q}}}) } &
\substack{\partial_{2} \mathbf{\eta}_{\mathbf{q}}  (\mathbf{x}_{\parallel_{\mathbf{q}}})\\  - \mathbf{x}_{\perp_{\mathbf{q}}} \partial_{2} \mathbf{n}_{\mathbf{q}}(\mathbf{x}_{\parallel_{\mathbf{q}}}) }
\end{array} \right]
\left[\begin{array}{c|c}
0 & \mathbf{0}_{1,2}\\ \hline
\mathbf{0}_{2,1} & \frac{\partial \mathbf{v}_{\parallel_{\mathbf{q}}}}{\partial \mathbf{x}_{\parallel_{\mathbf{p}}}}
\end{array}
\right] = \Big[\begin{array}{c|c|c}\mathbf{0}_{3,1} & Z_{1} & Z_{2}  \end{array}\Big],
\end{split}
\end{equation}
and
\begin{equation}\notag
\begin{split}
Z_{i} &=  \sum_{j=1}^{2} \mathbf{v}_{\parallel_{\mathbf{p}},j} \sum_{m=1}^{2}
\Big(\partial_{m} \partial_{j} \mathbf{\eta}_{\mathbf{p}} - \mathbf{x}_{\perp_{\mathbf{p}}} \partial_{m}\partial_{j} \mathbf{n}_{\mathbf{p}}  \Big) \Big(  \delta_{mi} - \frac{\partial \mathbf{x}_{\parallel_{\mathbf{q}},m}}{\partial \mathbf{x}_{\parallel_{\mathbf{p}},i} } \Big)
  \lesssim  O_{\xi}(1) |v| |\mathbf{p}-\mathbf{q}|,
\end{split}
\end{equation}
where we have used (\ref{Xx}).

Therefore 
\begin{equation}\notag
\begin{split}
 \left[
\begin{array}{c|c}
0 & \mathbf{0}_{1,2}\\ \hline
\mathbf{0}_{2,1} & \frac{\partial \mathbf{v}_{\parallel_{\mathbf{q}}}}{\partial \mathbf{x}_{\parallel_{\mathbf{p}}}}
\end{array}
\right]  &= \frac{1}{[-\mathbf{n}_{\mathbf{q}}] \cdot \big([\partial_{1} \mathbf{\eta}_{\mathbf{q}} - \mathbf{x}_{\perp_{\mathbf{q}}} \partial_{1} \mathbf{n}_{\mathbf{q}} ] \times
[\partial_{2} \mathbf{\eta}_{\mathbf{q}} - \mathbf{x}_{\perp_{\mathbf{q}}} \partial_{2} \mathbf{n}_{\mathbf{q}}]
\big)}   \\
& \ \ \ \  \times \left[
\begin{array}{c}
(\partial_{1} \mathbf{\eta}_{\mathbf{q}} - \mathbf{x}_{\perp_{\mathbf{q}}} \partial_{1} \mathbf{n}_{\mathbf{q}}) \times( \partial_{2} \mathbf{\eta}_{\mathbf{q}}- \mathbf{x}_{\perp_{\mathbf{q}}} \partial_{2} \mathbf{n}_{\mathbf{q}}) \\
(\partial_{2} \mathbf{\eta}_{\mathbf{q}} - \mathbf{x}_{\perp_{\mathbf{q}} }   \partial_{2} \mathbf{n}_{\mathbf{q}}) \times (- \mathbf{n}_{\mathbf{q}})
 \\
(-\mathbf{n}_{\mathbf{q}}) \times ( \partial_{1} \mathbf{\eta}_{\mathbf{q}} - \mathbf{x}_{\perp_{\mathbf{q}}} \partial_{1} \mathbf{n}_{\mathbf{q}})
\end{array}
\right]  \Big[\begin{array}{c|c|c}\mathbf{0}_{3,1} & Z_{1} & Z_{2}  \end{array}\Big],
\end{split}
\end{equation}
and hence from the above estimate of ${Z_{i}}$ we have
\begin{equation}\notag
\begin{split}
 \bigg[
  \frac{\partial \mathbf{v}_{\parallel_{\mathbf{q}}}}{\partial \mathbf{x}_{\parallel_{\mathbf{p}}}}
\bigg]_{2\times 2} &\lesssim_{\xi} | {v} ||\mathbf{p}-\mathbf{q}|.
\end{split}
\end{equation}

Again from the third equality of (\ref{pq})
\begin{equation}\notag
\begin{split}
&\left[\begin{array}{c|c|c}
-\mathbf{n}_{\mathbf{q}} (\mathbf{x}_{\parallel_{\mathbf{q}}} ) & \substack{  \partial_{1} \mathbf{\eta}_{\mathbf{q}}(\mathbf{x}_{\parallel_{\mathbf{q}}}) \\ - \mathbf{x}_{\perp_{\mathbf{q}}} \partial_{1} \mathbf{n}_{\mathbf{q}}(\mathbf{x}_{\parallel_{\mathbf{q}}})    } & \substack{ \partial_{2}  \mathbf{\eta}_{\mathbf{q}}(\mathbf{x}_{\parallel_{\mathbf{q}}})   \\  - \mathbf{x}_{\perp_{\mathbf{q}}} \partial_{2} \mathbf{n}_{\mathbf{q}} (\mathbf{x}_{\parallel_{\mathbf{q}}}) }
\end{array} \right]
\left[\begin{array}{c|cc}
1& 0  \\  \hline
0 & \frac{\partial \mathbf{v}_{\parallel_{\mathbf{q}} }  }{\partial \mathbf{v}_{\parallel_{\mathbf{p}} }}   
 \end{array} \right]\\
& =
 \left[\begin{array}{c|c|c}
-\mathbf{n}_{\mathbf{p}} (\mathbf{x}_{\parallel_{\mathbf{p}}}) &  \substack{ \partial_{1} \mathbf{\eta}_{\mathbf{p}} (\mathbf{x}_{\parallel_{\mathbf{p}}}) \\ - \mathbf{x}_{\perp_{\mathbf{p}}} \partial_{1} \mathbf{n}_{\mathbf{p}}  (\mathbf{x}_{\parallel_{\mathbf{p}}}) }
& \substack{ \partial_{2} \mathbf{\eta}_{\mathbf{p}} (\mathbf{x}_{\parallel_{\mathbf{p}}}) \\ - \mathbf{x}_{\perp_{\mathbf{p}}} \partial_{2} \mathbf{n}_{\mathbf{p}}  (\mathbf{x}_{\parallel_{\mathbf{p}}}) }
 \end{array}  \right].
 \end{split}
\end{equation}
Since
\begin{equation}\notag
\begin{split}
& \left[\begin{array}{c|c|c}
-\mathbf{n}_{\mathbf{p}} (\mathbf{x}_{\parallel_{\mathbf{p}}}) &  \substack{ \partial_{1} \mathbf{\eta}_{\mathbf{p}} (\mathbf{x}_{\parallel_{\mathbf{p}}}) \\ - \mathbf{x}_{\perp_{\mathbf{p}}} \partial_{1} \mathbf{n}_{\mathbf{p}}  (\mathbf{x}_{\parallel_{\mathbf{p}}}) }
& \substack{ \partial_{2} \mathbf{\eta}_{\mathbf{p}} (\mathbf{x}_{\parallel_{\mathbf{p}}}) \\ - \mathbf{x}_{\perp_{\mathbf{p}}} \partial_{2} \mathbf{n}_{\mathbf{p}}  (\mathbf{x}_{\parallel_{\mathbf{p}}}) }
 \end{array}  \right]\\
&  = \left[\begin{array}{c|c|c}
-\mathbf{n}_{\mathbf{q}} (\mathbf{x}_{\parallel_{\mathbf{q}}} ) & \substack{  \partial_{1} \mathbf{\eta}_{\mathbf{q}}(\mathbf{x}_{\parallel_{\mathbf{q}}}) \\ - \mathbf{x}_{\perp_{\mathbf{q}}} \partial_{1} \mathbf{n}_{\mathbf{q}}(\mathbf{x}_{\parallel_{\mathbf{q}}})    } & \substack{ \partial_{2}  \mathbf{\eta}_{\mathbf{q}}(\mathbf{x}_{\parallel_{\mathbf{q}}})   \\  - \mathbf{x}_{\perp_{\mathbf{q}}} \partial_{2} \mathbf{n}_{\mathbf{q}} (\mathbf{x}_{\parallel_{\mathbf{q}}}) }
\end{array} \right] + O_{\xi}(|\mathbf{p}-\mathbf{q}|),
\end{split}
\end{equation}
we have
\[
 \left[\begin{array}{cc}
  \frac{\partial \mathbf{v}_{\parallel_{\mathbf{q}},1}  }{\partial \mathbf{v}_{\parallel_{\mathbf{p}},1}} & \frac{\partial \mathbf{v}_{\parallel_{\mathbf{q}},1}  }{\partial \mathbf{v}_{\parallel_{\mathbf{p}},2}} \\
  \frac{\partial \mathbf{v}_{\parallel_{\mathbf{q}},2}  }{\partial \mathbf{v}_{\parallel_{\mathbf{p}},1}} & \frac{\partial \mathbf{v}_{\parallel_{\mathbf{q}},2}  }{\partial \mathbf{v}_{\parallel_{\mathbf{p}},2}}
 \end{array} \right]= \mathbf{Id}_{2,2} + O_{\xi}(|\mathbf{p}-\mathbf{q}|).
\]

\end{proof}

\vspace{4pt}

 We are ready to prove Theorem \ref{theorem_Dxv}:

\begin{proof}[\textbf{Proof of Theorem \ref{theorem_Dxv}}]
First we consider the case of $t<t_{\mathbf{b}}(x,v).$ In this case
 \[
(X_{\mathbf{cl}}(s;t,x,v), V_{\mathbf{cl}}(s;t,x,v) ) =(x-(t-s)v,v).
\]
Directly
\begin{equation}\notag\label{Dxv_interior}
\begin{split}
 \frac{\partial(  {X}_{\mathbf{cl}}(s;t,x,v), V_{\mathbf{cl}}(s,t,x,v))}{\partial (t,x,v)} 
 & = \left[\begin{array}{cccccc}   -v & \mathbf{Id}_{3, 3} & -(t-s)\mathbf{Id}_{3,3} \\ \mathbf{0}_{3,1} & \mathbf{0}_{3,3} & \mathbf{Id}_{3,3} \end{array}\right]_{6\times 7}\\
&:= \left[\begin{array}{c|ccc|ccc}
-v_{1} & 1 & 0 & 0  & -(t-s) & 0 & 0 \\
-v_{2} & 0 & 1 & 0 & 0 & -(t-s) & 0 \\
-v_{3} & 0 & 0 & 1 & 0 & 0 & -(t-s) \\ \hline
0 & 0 & 0 & 0 & 1 & 0 & 0 \\
0 & 0 & 0 & 0 & 0 & 1 & 0 \\
0 & 0 & 0 & 0 & 0 & 0 & 1 \end{array}
   \right]
,
\end{split}
\end{equation}
 where $\mathbf{Id}_{m,m}$ is the $m$ by $m$ identity matrix and $\mathbf{0}_{m,n}$ is the $m$ by $n$ zero matrix.

 \vspace{8pt}

Now we consider the case of $t\geq t_{\mathbf{b}}(x,v)$. We split our proof into 10 steps.

 \vspace{8pt}

\noindent{\textit{Step 1. Moving frames and grouping with respect to the scaling $t|v|=L_{\xi}$, with fixed $0< L_{\xi}\ll 1.$   }  }

\vspace{4pt}

Fix $(t,x,v)\in [0,\infty)\times \bar{\Omega}\times \mathbb{R}^{3}.$ Also we fix small constant $\delta=\delta_{\xi}>0$ which depends on the domain. We define, at the boundary,
\begin{equation}
\mathbf{r}^{\ell} : = \frac{|\mathbf{v}_{\perp}^{\ell}|}{|v^{\ell}|} = \frac{|v\cdot n(x^{\ell})|}{|v|} = \frac{|V_{\mathbf{cl}}(t^{\ell};t,x,v)\cdot n(X_{\mathbf{cl}} (t^{\ell};t,x,v))|}{|v|}.\label{r}
\end{equation}

Bounces $\ell$(and $(t^{\ell}, x^{\ell},v^{\ell})$) are categorized as \textit{Type I} or \textit{Type II}:
\begin{equation}\label{type_r}
\begin{split}
 \text{a bounce } \ell \text{ is \textit{Type I (almost grazing)} if and only if } & \mathbf{r}^{\ell} \leq   \sqrt{\delta}, \\
 \text{a bounce } \ell \text{ is \textit{Type II (non-grazing)} if and only if }  & \mathbf{r}^{\ell} >  \sqrt{\delta} .
\end{split}
\end{equation}
%
%
%
%
%

Let $s_{*} \in [t^{\ell+1}, t^{\ell}]$ such that $|\xi(X_{\mathbf{cl}}(s_{*};t^{\ell},x^{\ell},v^{\ell})) |= \max_{t^{\ell+1} \leq \tau \leq t^{\ell}}
|\xi(X_{\mathbf{cl}}(\tau;t^{\ell},x^{\ell},v^{\ell}))|.$ Since $\frac{d^{2}}{ds^{2}} \xi(X_{\mathbf{cl}}(s;t^{\ell},x^{\ell},v^{\ell}))=\frac{d^{2}}{ds^{2}} \xi(x^{\ell}-(t^{\ell}-s)v^{\ell}) = v^{\ell}\cdot \nabla_{x}^{2} \xi(x^{\ell}-(t^{\ell}-s)v^{\ell})\cdot
v^{\ell} >0$ there exists a unique $s$ solving $\frac{d}{ds} \xi(X_{\mathbf{cl}}(s;t^{\ell},x^{\ell},v^{\ell})) =  v^{\ell}\cdot
\nabla_{x}\xi(X_{\mathbf{cl}}(s;t^{\ell},x^{\ell},v^{\ell}))=0$ which is $s_{*}$. Note that $v^{\ell}\cdot \nabla \xi(x^{\ell}-(t^{\ell}-s) v^{\ell})$ is monotone in either one of the interval $(t^{\ell+1}, s_{*})$ or $(s_{*}, t^{\ell})$. Without of generality we may assume $|t^{\ell}-s_{*}| \geq \frac{1}{2}|t^{\ell+1}-t^{\ell}|.$ Then
\begin{equation}\notag
\begin{split}
|\xi(X_{\mathbf{cl}}(s_{*} ;t^{\ell}, x^{\ell},v^{\ell})  )| & = \Big| \int^{t^{\ell}} _{ s_{*}} v^{\ell} \cdot \nabla \xi (x^{\ell} - (t^{\ell}-s)v^{\ell}, v^{\ell}) \Big|= \Big| \int^{t^{\ell}}_{s_{*}} \int^{t^{\ell}}_{s } v^{\ell} \cdot \nabla^{2} \xi(x^{\ell} - (t^{\ell}-\tau)v^{\ell}, v^{\ell}) \cdot v^{\ell}
\Big|\\
& \simeq_{\xi} \frac{|v^{\ell}|^{2}|t^{\ell}-s_{*}|^{2}}{2} \simeq_{\xi}   \Big(\sup_{s\in [t^{\ell+1}, t^{\ell}]}\frac{|v^{\ell} \cdot n(X_{\mathbf{cl}}(s))|}{|v^{\ell}|}\Big)^{2},
\end{split}
\end{equation}
where we used (\ref{40}) and (\ref{41}) and the Velocity lemma (Lemma \ref{velocity_lemma}).

Therefore if a bounce $\ell$ is \textit{Type I} then $\max_{t^{\ell+1} \leq \tau \leq t^{\ell}} |\xi(X_{\mathbf{cl}}(\tau;t,x,v))| \leq C\delta$. If a bounce $\ell$ is \textit{Type II} then $|\xi(X_{\mathbf{cl}}(\tau;t,x,v))| > C \delta \text{ for some } \tau \in [t^{\ell+1},t^{\ell}]$.

\vspace{4pt}

Now we assign a coordinate chart for each bounce $\ell$ (moving frames).

For \textit{Type I} bounce $\ell$ in (\ref{type_r}), we assign $ {\mathbf{p}}^{\ell}\in \partial\Omega\times \mathbb{S}^{2}$ and $ {\mathbf{p}}^{\ell}-$spherical coordinates in Lemma \ref{chart_lemma} and (\ref{polar}): we choose $ {\mathbf{p}}^{\ell}:=(z^{\ell}, w^{\ell})$ on $\partial\Omega\times \mathbb{S}^{2}$ with $n(z^{\ell} ) \cdot w^{\ell}=0$ such that $z^{\ell}$ and $w^{\ell}$ do not depends on $(t,x,v)$ and
\begin{equation}\label{pl}
|z^{\ell} - x^{\ell}| < \mathbf{r}^{\ell}, \ \ \ \Big|w^{\ell} - \frac{v^{\ell} - {(v^{\ell}\cdot n(z^{\ell}))}n(z^{\ell})}{|  v^{\ell} - {(v^{\ell}\cdot n(z^{\ell}))}n(z^{\ell})  |} \Big| < \mathbf{r}^{\ell}.
\end{equation}
Note that, by the definition of \textit{Type I }$ bounce, |v^{\ell} - (v^{\ell} \cdot n(z^{\ell}) n(z^{\ell}))|^{2}=|v|^{2} - |\mathbf{v}_{\perp}^{\ell}|^{2}\gtrsim |v|^{2}(1-\delta) \gtrsim_{\delta} |v|^{2}$ and hence $w^{\ell}$ is well-defined.

Moreover
\begin{equation}\label{L}
|X_{\mathbf{cl}}(s;t,x,v)- \mathcal{L}_{\mathbf{p}^{\ell}}| \gtrsim C_\delta>0,
\end{equation}
for $|v||t^{\ell}-s| \leq  \frac{1}{100} \min_{x\in\partial\Omega} |x|.$ This is due to the fact that the projection of $V_{\mathbf{cl}}(s)$ on the plane passing $z^{\ell}$ and perpendicular to $n(z^{\ell}) \times w^{\ell}$ is at most $|v|$ but the distance from $z^{\ell}$ to the origin(the projection of poles $\mathcal{N}_{\mathbf{p}^{\ell}}$ and $\mathcal{S}_{\mathbf{p}^{\ell}}$) has lower bound $\frac{1}{10} \min_{x\in \partial\Omega}|x|,$ $s\sim t^{\ell}.$

For \textit{Type II} bounce $\ell$$(t^{\ell},x^{\ell},v^{\ell})$, we choose $\mathbf{p}^{\ell}= (z^{\ell}, w^{\ell})$ with $|z^{\ell}-x^{\ell}|\leq \sqrt{\delta}$ but we choose arbitrary $w^{\ell} \in\mathbb{S}^{2}$ satisfying $n(z^{\ell})\cdot w^{\ell}=0$. We choose $\mathbf{p}^{\ell}-$spherical coordinate in Lemma \ref{chart_lemma} and (\ref{polar}) with this $\mathbf{p}^{\ell}.$ Note that unlike \textit{Type I}, this $\mathbf{p}^{\ell}-$spherical coordinate might not be defined for $s\in [t^{\ell+1}, t^{\ell}]$ but only defined near the boundary.

Whenever the moving frame is defined (for all $\tau \in (t^{\ell+1}, t^{\ell}]$ when $\ell$ is \textit{Type I}, and $\tau \sim t^{\ell}$ when $\ell$ is \textit{Type II}) we denote
\[
(\mathbf{X}_{\mathbf{\ell}}(\tau), \mathbf{V}_{\mathbf{\ell}}(\tau))=
( \mathbf{x}_{\perp_\mathbf{\ell}}(\tau), \mathbf{x}_{\parallel_\mathbf{\ell}}(\tau), \mathbf{v}_{\perp_\mathbf{\ell}}(\tau), \mathbf{v}_{\parallel_{\mathbf{\ell}}}(\tau)) 
: = \Phi^{-1}_{\mathbf{p}^{\ell}}(X_{\mathbf{cl}}(\tau), V_{\mathbf{cl}}(\tau)).
\]
Especially at the boundary we denote
\[
(\mathbf{x}_{\perp_{\ell}}^{\ell},\mathbf{x}_{\parallel_{\ell}}^{\ell} , \mathbf{v}_{\perp_{\ell}}^{\ell},\mathbf{v}_{\parallel_{\ell}}^{\ell} )
:= \lim_{\tau \uparrow t^{\ell}} (\mathbf{X}_{\mathbf{\ell}}(\tau), \mathbf{V}_{\mathbf{\ell}}(\tau))
, \ \ \ \text{with  } \mathbf{x}_{\perp_{\ell}}^{\ell}=0, \ \mathbf{v}_{\perp_{\ell}}^{\ell}\geq 0.
\]
Then we define
\[
( \mathbf{x}_{\perp_{\ell}}^{\ell+1}, \mathbf{x}_{\parallel_{\ell}}^{\ell+1},  \mathbf{v}_{\parallel_{\ell}}^{\ell+1} ) = 
\lim_{\tau \downarrow t^{\ell+1}} (   \mathbf{x}_{\perp_{\ell}}(\tau), \mathbf{x}_{\parallel_{\ell}}(\tau),  \mathbf{v}_{\parallel_{\ell}}(\tau) ),
\]
and
\begin{equation}\label{v_perp}
  \mathbf{v}_{\perp_{\ell}}^{\ell+1}:= - \lim_{\tau \downarrow t^{\ell+1}} \mathbf{v}_{\perp_{\ell}}(\tau).
\end{equation}


Now we regroup the indices of the specular cycles, without order changing, as
\begin{equation}\notag
\begin{split}
\{0,1,2,\cdots,   \ell_{*}-1, \ell_{*}\}  =\{0\} \cup  \mathcal{G}_{1} \cup \mathcal{G}_{2} \cup \cdots \cup\mathcal{G}_{[\frac{|t-s||v|}{L_{\xi}}]} \cup \mathcal{G}_{[\frac{|t-s||v|}{L_{\xi}}]+1} 
,
\end{split}
\end{equation}
where $\big[ a \big]\in\mathbb{N}$ is the greatest integer less than or equal to $a$. Each group is
\begin{equation}\label{group}
\begin{split}
\mathcal{G}_{1} &= \{ 1, \cdots, \ell_{1}-1, \ell_{1}\},\\
 \mathcal{G}_{2} &= \{ \ell_{1}, \ell_{1}+1,\cdots , \ell_{2}-1,\ell_{2}\},    \\
 & \ \   \vdots\\
\mathcal{G}_{[\frac{|t-s||v|}{L_{\xi}}]} &= \{\ell_{[\frac{|t-s||v|}{L_{\xi}}]-1}  ,\ell_{[ \frac{|t-s||v|}{{L_{\xi}}}]-1 }+1,\cdots,\ell_{[ \frac{|t-s||v|}{L_{\xi}}]}-1,\ell_{[ \frac{|t-s||v|}{L_{\xi}}]}  \},\\
 \mathcal{G}_{[ \frac{|t-s||v|}{L_{\xi}}]+1}  &= \{\ell_{[ \frac{|t-s||v|}{L_{\xi}}] }  ,\ell_{[ \frac{|t-s||v|}{L_{\xi}}] } +1,\cdots,\ell_{*}
  \},
\end{split}
\end{equation}
where $\ell_{1} = \inf\{ \ell \in\mathbb{N} : |v|\times |t^{0} - t^{\ell_{1}}| \geq L_{\xi} \}$ and inductively
\begin{equation}\label{L_xi}
\ell_{i} = \inf\{ \ell \in\mathbb{N} :    |v|\times |t^{\ell_{i}} - t^{\ell_{i+1}}| \geq L_{\xi}\},
\end{equation}
and we have denoted $\ell_{*} = \ell_{[ \frac{|t-s||v|}{L_{\xi}}]+1}$.

By the chain rule, with the assigned $\mathbf{p}^{\ell}-$spherical coordinate (moving frame), we have for fixed $0 \leq s \leq t$ and $s \in (t^{\ell_{*}+1}, t^{\ell_{*}})$
\begin{equation}\label{chain}
\begin{split}
&\frac{\partial (s,X_{\mathbf{cl}}(s;t,x,v),V_{\mathbf{cl}}(s;t,x,v))}{\partial (t,x,v)}\\
 = &\underbrace{\frac{\partial (s,X_{\mathbf{cl}}(s), V_{\mathbf{cl}}(s))}{\partial (t^{\ell_{*}},0, \mathbf{x}_{\parallel_{\ell_{*}}}^{\ell_{*}}, \mathbf{v}_{\perp_{\ell_{*}}}^{\ell_{*}}, \mathbf{v}_{\parallel_{\ell_{*}}}^{\ell_{*}})}  }_{\text{from the last bounce to the }s-\text{plane}} \\
\times& \underbrace{\prod_{i=1}^{ [\frac{|t-s||v|}{L_{*}}]} \underbrace{\frac{\partial (t^{\ell_{i+1}},0, \mathbf{x}_{\parallel_{\ell_{i+1}}}^{\ell_{i+1}}, \mathbf{v}_{\perp_{\ell_{i+1}}}^{\ell_{i+1}}, \mathbf{v}_{\parallel_{\ell_{i+1}}}^{\ell_{i+1}})}{\partial (t^{\ell_{i+1}-1},0, \mathbf{x}_{\parallel_{{\ell_{i+1}-1}} }^{\ell_{i+1}-1}, \mathbf{v}_{\perp_{{\ell_{i+1}-1}}}^{\ell_{i+1}-1}, \mathbf{v}_{\parallel_{{\ell_{i+1}-1}}}^{\ell_{i+1}-1})}
 \times \cdots \times
 \frac{\partial (t^{\ell_{i}+1},0, \mathbf{x}_{\parallel_{{\ell_{i} +1}}}^{\ell_{i} +1}, \mathbf{v}_{\perp_{{\ell_{i} +1}}}^{\ell_{i}+1}, \mathbf{v}_{\parallel_{{\ell_{i} +1}}}^{\ell_{i}+1})}{\partial (t^{\ell_{i} },0, \mathbf{x}_{\parallel_{ {\ell_{i} }}}^{\ell_{i}  }, \mathbf{v}_{\perp_{ {\ell_{i} }}}^{\ell_{i} }, \mathbf{v}_{\parallel_{ {\ell_{i} }}}^{\ell_{i} })}  }_{i-\text{th intermediate group}} }_{\text{whole intermediate groups}}\\
 \times & \underbrace{ \frac{\partial (t^{1}, 0 , \mathbf{x}_{\parallel_{1}}^{1}, \mathbf{v}_{\perp_{1}}^{1}, \mathbf{v}_{\parallel_{1}}^{1})}{\partial (t,x,v)}}_{\text{from the } t-\text{plane to the first bounce} }.
\end{split}
\end{equation}

\vspace{4pt}

\noindent\textit{Step 2. From the last bounce $\ell_{*}$ to the $s-$plane}

We choose $s^{ \ell_{*}}\in (\frac{t^{\ell_{*}}+s}{2},t^{\ell_{*}})\subset(s,t^{\ell_{*}})$ such that $|v||t^{\ell_{*}} - s^{\ell_{*}}|\ll 1$ and the $\ell_{*}-$spherical coordinate $(\mathbf{X}_{\ell_{*}}(s^{\ell_{*}}), \mathbf{V}_{\ell_{*}}(s^{\ell_{*}}))$ is well-defined regardless of {type}s of $\ell_{*}$ in (\ref{type_r}). Notice that $s^{\ell_{*}}$ is independent of $t^{\ell_{*}}$ and $s$ so that $\frac{\partial s^{\ell_{*}}}{\partial t^{\ell_{*}}}=0= \frac{\partial s^{\ell_{*}} }{\partial s  }.$

By the chain rule,
 \begin{equation}\notag
 \begin{split}
&\frac{\partial (s,X_{\mathbf{cl}}(s), V_{\mathbf{cl}}(s))}
{\partial (t^{\ell_{*}},0, \mathbf{x}_{\parallel_{\ell_{*}}}^{\ell_{*}}, \mathbf{v}_{\perp_{\ell_{*}}}^{\ell_{*}}, \mathbf{v}_{\parallel_{\ell_{*}}}^{\ell_{*}})}\\
&= \frac{\partial (s,X_{\mathbf{cl}}(s), V_{\mathbf{cl}}(s))}
{\partial (s^{\ell_{*}}, \mathbf{X}_{\ell_{*}}(s^{\ell_{*}}),   \mathbf{V}_{\ell_{*}}(s^{\ell_{*}})  )}
   \frac{\partial (s^{\ell_{*}}, \mathbf{x}_{\perp_{\ell_{*}}}(s^{\ell_{*}}),  \mathbf{x}_{\parallel_{\ell_{*}}}(s^{\ell_{*}}),   \mathbf{v}_{\perp_{\ell_{*}}}(s^{\ell_{*}}), \mathbf{v}_{\parallel_{\ell_{*}}}(s^{\ell_{*}})
    )}{\partial (t^{\ell_{*}},0, \mathbf{x}_{\parallel_{\ell_{*}}}^{\ell_{*}}, \mathbf{v}_{\perp_{\ell_{*}}}^{\ell_{*}}, \mathbf{v}_{\parallel_{\ell_{*}}}^{\ell_{*}})}
\\
& = \frac{\partial (s,X_{\mathbf{cl}}(s), V_{\mathbf{cl}}(s))}{\partial (s^{\ell_{*}},  {X}_{ \mathbf{cl} }(s^{\ell_{*}}),   {V}_{ \mathbf{cl}   }(s^{\ell_{*}})  )}
\frac{\partial (s^{\ell_{*}},  {X}_{  \mathbf{cl}   } (s^{\ell_{*}}),   {V}_{ \mathbf{cl}   }(s^{\ell_{*}})  )}{\partial (s^{\ell_{*}}, \mathbf{X}_{ {\ell_{*}}}(s^{\ell_{*}}),   \mathbf{V}_{ {\ell_{*}}}(s^{\ell_{*}})  )}
   \frac{\partial (s^{\ell_{*}}, \mathbf{x}_{\perp_{\ell_{*}}}(s^{\ell_{*}}),  \mathbf{x}_{\parallel_{\ell_{*}}}(s^{\ell_{*}}),   \mathbf{v}_{\perp_{\ell_{*}}}(s^{\ell_{*}}), \mathbf{v}_{\parallel}(s^{\ell_{*}})
    )}{\partial (t^{\ell_{*}},0, \mathbf{x}_{\parallel_{\ell_{*}}}^{\ell_{*}}, \mathbf{v}_{\perp_{\ell_{*}}}^{\ell_{*}}, \mathbf{v}_{\parallel_{\ell_{*}}}^{\ell_{*}})}.
 \end{split}
 \end{equation}
Firstly, we claim
\begin{equation}\label{ss1}
\frac{\partial (s,X_{   \mathbf{cl}}(s), V_{ \mathbf{cl}  }(s))}{\partial (s^{\ell_{*}}, \mathbf{X}_{ {\ell_{*}}}(s^{\ell_{*}}),   \mathbf{V}_{ {\ell_{*}}}(s^{\ell_{*}})  )}
=
\left[\begin{array}{ccc}
0 & \mathbf{0}_{1,3} & \mathbf{0}_{1,3}\\
-V_{\mathbf{cl}}(s^{\ell_{*}})& O_{\xi}(1) (1+|v||s^{\ell_{*}}-s|) & O_{\xi}(1)  |s^{\ell_{*}}-s|\\
\mathbf{0}_{3,1} & O_{\xi}(1)  |v| & O_{\xi}(1)
\end{array} \right].
\end{equation}
Since
\[
X_{\mathbf{cl}}(s) = X_{\mathbf{cl}}(s^{\ell_{*}}) - (s^{\ell_{*}}-s) V_{\mathbf{cl}}(s^{\ell_{*}}), \ \ V_{\mathbf{cl}}(s) = V_{\mathbf{cl}}(s^{\ell_{*}}),
\]
and $s^{\ell_{*}}$ is independent of $s$, we have
\begin{equation}\notag
\begin{split}
\frac{\partial (s,X_{\mathbf{cl}}(s), V_{\mathbf{cl}}(s))}{\partial (s^{\ell_{*}},  {X}_{\mathbf{cl}}(s^{\ell_{*}}),   {V}_{\mathbf{cl}}(s^{\ell_{*}})  )}&= \left[\begin{array}{ccc}
0 & \mathbf{0}_{1,3} & \mathbf{0}_{1,3} \\
-V_{\mathbf{cl}}(s^{\ell_{*}}) & \mathbf{Id}_{3,3} & -(s^{\ell_{*}}-s) \mathbf{Id}_{3,3}\\
\mathbf{0}_{3,1} & \mathbf{0}_{3,3} & \mathbf{Id}_{3,3}
\end{array} \right].
\end{split}
\end{equation}
Due to Lemma \ref{chart_lemma},
\begin{equation}\notag
\begin{split}
&\frac{\partial (s^{\ell_{*}},  {X}_{\mathbf{cl}}(s^{\ell_{*}}),   {V}_{\mathbf{cl}}(s^{\ell_{*}})  )}{\partial (s^{\ell_{*}}, \mathbf{X}_{ {\ell_{*}}}(s^{\ell_{*}}),   \mathbf{V}_{ {\ell_{*}}}(s^{\ell_{*}})  )}\\
&=  {\left[ \begin{array}{c|ccc|ccc}
1 &   & \mathbf{0}_{1,3} &  &  & \mathbf{0}_{1,3} &   \\ \hline
\mathbf{0}_{3,1}& -\mathbf{n}_{{\ell_{*}}}  &
   \substack{  \partial_{ 1} \mathbf{\eta}_{{\ell_{*}}}  \\
- \mathbf{x}_{\perp_{{\ell_{*}}}}  \partial_{ 1} \mathbf{n}_{{\ell_{*}}}  }&
 \substack{   \partial_{2} \mathbf{\eta}_{{\ell_{*}}}\\
-\mathbf{x}_{\perp_{{\ell_{*}}}}  \partial_{2} \mathbf{n}_{{\ell_{*}}}   }
 & & \mathbf{0}_{3,3} & \\ \hline
\mathbf{0}_{3,1}&
-\mathbf{v}_{\parallel_{{\ell_{*}}}} \cdot \nabla_{ \mathbf{x}_{\parallel_{{\ell_{*}}}} } \mathbf{n}_{{\ell_{*}}}
&  \substack{ \mathbf{v}_{\parallel_{{\ell_{*}}}} \cdot \nabla  \partial_{  1  } \mathbf{\eta}_{{\ell_{*}}} \\
-  \mathbf{v}_{\perp_{{\ell_{*}}}}\partial_{1}  \mathbf{n}_{{\ell_{*}}}   \\
-  \mathbf{x}_{\perp_{{\ell_{*}}}}  \mathbf{v}_{\parallel_{{\ell_{*}}}}\cdot \nabla \partial_{ 1} \mathbf{n}_{{\ell_{*}}}   }
 & \substack{ \mathbf{v}_{\parallel_{{\ell_{*}}}} \cdot \nabla \partial_{2} \mathbf{\eta}_{{\ell_{*}}} \\
- \mathbf{v}_{\perp_{{\ell_{*}}}} \partial_{ 2 } \mathbf{n}{\ell_{*}} \\
-  \mathbf{x}_{\perp_{{\ell_{*}}}}  \mathbf{v}_{\parallel_{{\ell_{*}}}}\cdot \nabla  \partial_{2} \mathbf{n}_{{\ell_{*}}}  }
    & -\mathbf{n}_{{\ell_{*}}}  
    & \substack{\partial_{1} \mathbf{\eta}_{{\ell_{*}}} \\
   - \mathbf{x}_{\perp_{{\ell_{*}}}}\partial_{1} \mathbf{n}_{{\ell_{*}}} }
     & \substack{\partial_{2} \mathbf{\eta}_{{\ell_{*}}}  \\ - \mathbf{x}_{\perp_{{\ell_{*}}}} \partial_{2} \mathbf{n}}_{{\ell_{*}}}
  \end{array}\right],}
\end{split}
\end{equation}
 where all entries are evaluated at $(\mathbf{X}_{ {\ell_{*}}}(s^{\ell_{*}}), \mathbf{V}_{ {\ell_{*}}}(s^{\ell_{*}})).$ The multiplication of above two matrices gives (\ref{ss1}).

Secondly, we claim that whenever $\mathbf{p}^{\ell}-$spherical coordinate is defined for all $\tau \in [s^{\ell}, t^{\ell}]$
\begin{equation}\label{s1tstar}
\begin{split}
&\frac{\partial (s^{\ell }, \mathbf{x}_{\perp_{{\ell }}}(s^{\ell }),  \mathbf{x}_{\parallel_{{\ell }}}(s^{\ell }),   \mathbf{v}_{\perp_{{\ell }}}(s^{\ell }), \mathbf{v}_{\parallel_{{\ell }}}(s^{\ell })
    )}{\partial (t^{\ell },0, \mathbf{x}_{\parallel_{{\ell }}}^{\ell }, \mathbf{v}_{\perp_{{\ell }}}^{\ell }, \mathbf{v}_{\parallel_{{\ell }}}^{\ell })}\\
=& \tiny{ \left[\begin{array}{c|cc|cc}
0 & 0& \mathbf{0}_{1,2} & 0 &   \mathbf{0}_{1,2}  \\ \hline
 -\mathbf{v}_{\perp}(s^{\ell}) & 0 & O_{\xi}(1) |v|^{2} |t^{\ell}-s^{\ell}|^{2} & O_{\xi}(1)   |t^{\ell}-s^{\ell}|  &  O_{\xi}(1)     |v| |t^{\ell}-s^{\ell}|^{2 }   \\
-\mathbf{v}_{\parallel}(s^{\ell}) & \mathbf{0}_{2,1}  &\mathbf{Id}_{2,2}+ O_{\xi}(1) |v|^{2}|t^{\ell}-s^{\ell}|^{2} &  O_{\xi}(1)   |v| |t^{\ell}-s^{\ell}|^{2} &
O_{\xi}(1)|t^{\ell}-s^{\ell}| ( \mathbf{Id}_{2,2}+  |v||t^{\ell}-s^{\ell}|)   \\ \hline
O_{\xi}(1)|v|^{2} & 0  &O_{\xi}(1) |v|^{2}|t^{\ell}-s^{\ell}| &1+O_{\xi}(1)|v||t^{\ell}-s^{\ell}| &    O_{\xi}(1)|v||t^{\ell}-s^{\ell}| \\
O_{\xi}(1)|v|^{2} & \mathbf{0}_{2,1}  &O_{\xi}(1) |v|^{2}|t^{\ell}-s^{\ell}| & O_{\xi}(1)|v||t^{\ell}-s^{\ell}| & \mathbf{Id}_{2,2}+  O_{\xi}(1)|v||t^{\ell}-s^{\ell}|
\end{array}\right] .}
\end{split}
\end{equation}
In this step we just need (\ref{s1tstar}) for $\ell=\ell_{*}$ but we need (\ref{s1tstar}) for general $\ell$ in Step 8.

Clearly the first raw is identically zero since $s^{\ell}$ is chosen to be independent of $(t^{\ell}, \mathbf{x}_{\parallel_{\ell}}^{\ell}, \mathbf{v}_{\perp_{\ell}}^{\ell}, \mathbf{v}_{\parallel_{\ell}}^{\ell})$. The first column (temporal derivatives) holds due to the fact that the characteristics ODE (\ref{ODE_ell}) is autonomous. More explicitly,
\begin{equation}\notag
\begin{split}
&\frac{\partial}{\partial t^{\ell}}  (\mathbf{X}_{ \ell}(s^{\ell}; t^{\ell}, x^{\ell}, v^{\ell}), \mathbf{V}_{ \ell}(s^{\ell}; t^{\ell}, x^{\ell}, v^{\ell}))\\ &=  \frac{\partial}{\partial t^{\ell}}  (\mathbf{X}_{ \ell}(s^{\ell}-t^{\ell}; 0, x^{\ell}, v^{\ell}), \mathbf{V}_{ \ell}(s^{\ell}-t^{\ell}; 0, x^{\ell}, v^{\ell})) \\
& = -\frac{\partial}{\partial s^{\ell}} (\mathbf{X}_{ \ell}(s^{\ell}; t^{\ell}, x^{\ell}, v^{\ell}), \mathbf{V}_{ \ell}(s^{\ell}; t^{\ell}, x^{\ell}, v^{\ell}))
\\
&=  -(\mathbf{V}_{ \ell  }(s^{\ell}; t^{\ell}, x^{\ell}, v^{\ell}), F(\mathbf{X}_{ \ell}(s^{\ell}; t^{\ell}, x^{\ell}, v^{\ell}), \mathbf{V}_{ \ell}(s^{\ell}; t^{\ell}, x^{\ell}, v^{\ell}))\\
& = (- \mathbf{v}_{\perp}(s^{\ell}), - \mathbf{v}_{\parallel}(s^{\ell}), O_{\xi}(1)|v|^{2}, O_{\xi}(1)|v|^{2}).
\end{split}
\end{equation}

Now we prove the the remainder. Firstly we claim that if the $\mathbf{p}^{ \ell}-$spherical coordinate is well-defined for $ t^{\ell+1}< \tau < t^{\ell}$ ($\tau$ is independent of $t^{\ell}$) then
\[
[\mathbf{X}_{ \ell}(\tau;t,x,v), \mathbf{V}_{ \ell}(\tau;t,x,v)] \equiv [\mathbf{X}_{ \ell}(\tau;t^{\ell} ,0, \mathbf{x}_{\parallel_{\ell}}^{\ell} , \mathbf{v}_{\perp_{\ell}}^{\ell}, \mathbf{v}_{\parallel_{\ell}}^{\ell} ), \mathbf{V}_{ \ell} (\tau;t^{\ell} ,0, \mathbf{x}_{\parallel_{\ell}}^{\ell} , \mathbf{v}_{\perp_{\ell}}^{\ell}, \mathbf{v}_{\parallel_{\ell}}^{\ell} )]
\]
and
\begin{equation}\label{Dxv_free}
\begin{split}
|\partial_{\mathbf{x}^{\ell}_{\parallel_{\ell}}   } \mathbf{X}_{ \ell   }(\tau)| & \ \lesssim \ e^{C_{\xi} |v|   |    \tau - t^{\ell}   |   } \ \lesssim \ 1,\\
|\partial_{\mathbf{v}_{ {\ell}}^{\ell}  } \mathbf{X}_{ \ell}(\tau)| & \ \lesssim \ | \tau - t^{\ell} |e^{C_{\xi} |v|| \tau - t^{\ell}  |} \ \lesssim \ 
  | \tau - t^{\ell}  | ,\\
|\partial_{\mathbf{x}^{\ell}_{\parallel_{\ell}}  } \mathbf{V}_{ \ell}(\tau)| &  \ \lesssim \ |v| \times |v|| \tau - t^{\ell}  |e^{C_{\xi}|v||  \tau - t^{\ell} |} \ \lesssim \ |v|^{2}| \tau - t^{\ell} |  ,\\
|\partial_{\mathbf{v}^{\ell}_{ {\ell}}  } \mathbf{V}_{ \ell}(\tau)| & \ \lesssim \ e^{C_{\xi}|v|| \tau - t^{\ell}  |} \ \lesssim \ 1,
\end{split}
\end{equation}
where $\partial_{\mathbf{v}_{ {\ell}}^{\ell}  } = [\partial_{\mathbf{v}_{\perp_{\ell}}^{\ell}  }, \partial_{\mathbf{v}_{\parallel_{\ell}}^{\ell}  }].$

If the $\mathbf{p}^{\ell}-$spherical coordinate is well-defined for $t^{\ell+1} < \tau < s < t^{\ell}$ then
\begin{equation}\notag
\begin{split}
&[\mathbf{X}_{\ell}(\tau;t,x,v), \mathbf{V}_{\ell}(\tau; t,x,v)  ] \\
\equiv & \ [ \mathbf{X}_{\ell}(\tau; s, \mathbf{X}_{\ell} (\tau;t,x,v),\mathbf{V}_{\ell} (\tau;t,x,v) ),  
\mathbf{V}_{\ell}(\tau; s, \mathbf{X}_{\ell} (\tau;t,x,v),\mathbf{V}_{\ell} (\tau;t,x,v) )  ]
,
\end{split}
\end{equation}
and 
\begin{equation}\label{Dxv_free_s}
\begin{split}
|\partial_{ \mathbf{X}_{\ell}(s)  } \mathbf{X}_{ \ell   }(\tau)| & \ \lesssim \ e^{C_{\xi} |v|   |    \tau - s  |   } \ \lesssim \ 1,\\
|\partial_{ \mathbf{V}_{\ell}(s) } \mathbf{X}_{ \ell}(\tau)| & \ \lesssim \ | \tau - s |e^{C_{\xi} |v|| \tau - s  |} \ \lesssim \ 
  | \tau - s  | ,\\
|\partial_{ \mathbf{X}_{\ell}(s)  }  \mathbf{V}_{ \ell}(\tau)| &  \ \lesssim \ |v| \times |v|| \tau - s  |e^{C_{\xi}|v||  \tau - s |} \ \lesssim \ |v|^{2}| \tau - s |  ,\\
|\partial_ { \mathbf{V}_{\ell}(s) }  \mathbf{V}_{ \ell}(\tau)| & \ \lesssim \ e^{C_{\xi}|v|| \tau - s  |} \ \lesssim \ 1.
\end{split}
\end{equation}

\noindent\textit{Proof of }(\ref{Dxv_free}) and (\ref{Dxv_free_s}). From (\ref{F_perp}) and (\ref{F||}), $\dot{\mathbf{x}}_{\parallel_{ \ell}} = \mathbf{v}_{\parallel_{ \ell}}, \ \dot{\mathbf{x}}_{\perp_{ \ell}} = \mathbf{v}_{\perp_{ \ell}}$ and $\dot{\mathbf{v}}_{\perp_{ \ell}} =  {F}_{\perp_{ \ell}}$ and $\dot{\mathbf{v}}_{\parallel_{ \ell}} =  {F}_{\parallel_{ \ell}}$. Denote $\partial = [  \frac{\partial }{ \partial {\mathbf{x}_{\parallel_{ \ell}}^{\ell} } },  \frac{\partial }{\partial { \mathbf{v}_{\perp_{ \ell}}^{\ell} } },  \frac{\partial }{ \partial { \mathbf{v}_{\parallel_{ \ell}}^{\ell} }} ].$ From (\ref{F_perp}) and (\ref{F||}),
\begin{equation}\label{D_F}
\begin{split}
|\partial F_{\perp}| & \ \lesssim \   | {v} |^{2} \{ |\partial \mathbf{x}_{\perp}|  +  | \partial\mathbf{x}_{\parallel}| \}  +  |v| |\partial \mathbf{v}_{\parallel}|    ,\\
|\partial F_{\parallel}| & \ \lesssim \  |v|^{2} \{  |\partial \mathbf{x}_{\perp}| + |\partial \mathbf{x}_{\parallel}| \}   
+ |v| \{ |\partial \mathbf{v}_{\perp}| + |\partial \mathbf{v}_{\parallel}|\}.
\end{split}
\end{equation}
Now we use a single (rough) bound of $|\partial F_{\perp}|  + |\partial F_{\parallel}| \lesssim  |v|^{2} \{  |\partial \mathbf{x}_{\perp}| + |\partial \mathbf{x}_{\parallel}| \}   
+ |v| \{ |\partial \mathbf{v}_{\perp}| + |\partial \mathbf{v}_{\parallel}|\}$ to have
\begin{equation}\notag
\begin{split}
\frac{d}{d\tau} \{  |\partial  \mathbf{v}_{\perp_{ \ell   } }(\tau)|  +   |\partial  \mathbf{v}_{\parallel_{ \ell   } }(\tau)|   \} 
&\lesssim \  |\partial F_{\perp_{\ell}}(\tau)|   +  |\partial F_{\parallel_{\ell}}(\tau)|\\
&\lesssim \ |v|^{2} \big\{     |\partial \mathbf{x}_{\perp_{\ell   }}(\tau)|   + |\partial \mathbf{x}_{\parallel_{\ell   }}(\tau)| \big\}+ |v|\big\{    |\partial \mathbf{v}_{\perp_{\ell   }}(\tau)| + |\partial \mathbf{v}_{\parallel _{\ell   }}(\tau)|  \big\}.
\end{split}
\end{equation}
Combining with $ \frac{d}{d\tau}[\mathbf{x}_{\perp_{\ell}}(\tau), \mathbf{x}_{\parallel_{\ell}}(\tau) ]= [\mathbf{v}_{\perp_{\ell}}(\tau), \mathbf{v}_{\parallel_{\ell}}(\tau) ]$,

\vspace{4pt}

\begin{equation}\notag
\begin{split}
&\frac{d}{d\tau}\left[\begin{array}{ccccc}|\partial \mathbf{x}_{\perp_{\ell   }}(\tau)| + |\partial \mathbf{x}_{\parallel_{\ell   }}(\tau)| \\ |\partial \mathbf{v}_{\perp_{\ell   }}(\tau)| + |\partial \mathbf{v}_{\parallel_{\ell   }}(\tau)|   \end{array}\right]  \ \lesssim_{\xi}  \ \left[\begin{array}{cc} 0 & 1 \\ |v|^{2} & |v|\end{array}\right]  \left[\begin{array}{ccccc}|\partial \mathbf{x}_{\perp_{\ell   }}(\tau)| + |\partial \mathbf{x}_{\parallel_{\ell   }}(\tau)| \\ |\partial \mathbf{v}_{\perp_{\ell   }}(\tau)| + |\partial \mathbf{v}_{\parallel_{\ell   }}(\tau)|   \end{array}\right].
 \end{split}
\end{equation}

\vspace{4pt}

We diagonalize the matrix as
\begin{equation*}
\left[
\begin{array}{cc}
0 & 1    \\
|v|^{2} & |v|
\end{array}%
\right] =\left[
\begin{array}{cc }
1 & 1    \\
\frac{1+\sqrt{5}}{2}|v| & \frac{1-\sqrt{5}}{2}|v|
\end{array}%
\right] \left[
\begin{array}{cc }
\frac{1+\sqrt{5}}{2}|v| & 0   \\
0 & \frac{1-\sqrt{5}}{2}|v|
\end{array}%
\right] \left[
\begin{array}{cc }
-\frac{1-\sqrt{5}}{2\sqrt{5}} & \frac{1}{|v|\sqrt{5}}   \\
\frac{1+\sqrt{5}}{2\sqrt{5}} & \frac{-1}{|v|\sqrt{5}}
\end{array}%
\right] := PDP^{-1}.
\end{equation*}%
Now
\begin{equation*}
\begin{split}
 \left[
\begin{array}{cc}
|\partial \mathbf{x}_{\parallel }(\tau)|+|\partial \mathbf{x}_{\perp }(\tau)|   \\
|\partial \mathbf{v}_{\parallel }(\tau)|+|\partial \mathbf{v}_{\perp }(\tau)|
\end{array}%
\right]  \leq& P e^{C_{\xi}  |\tau -t^{\ell} |D   }
  P^{-1}    \left[
\begin{array}{c }
|\partial \mathbf{x}_{\parallel }(t^{\ell})|+|\partial \mathbf{x}_{\perp }(t^{\ell})|    \\
|\partial \mathbf{v}_{\parallel }(t^{\ell})|+|\partial \mathbf{v}_{\perp }(t^{\ell})|
\end{array}%
\right],
\end{split}
\end{equation*}
which is further bounded as, by matrix multiplication, 
\begin{equation*}\notag
\begin{split}
& \leq    \tiny{\left[
\begin{array}{cc }
-\frac{1-\sqrt{5}}{2\sqrt{5}}e^{C_{\xi}\frac{1+\sqrt{5}}{2}|v|| \tau - t^{\ell}|}+\frac{1+%
\sqrt{5}}{2\sqrt{5}}e^{C_{\xi}\frac{1-\sqrt{5}}{2}|v|| \tau -t^{\ell}   |} & \frac{1}{\sqrt{%
5}|v|}e^{\frac{C_{\xi }}{2}|v||  \tau -t^{\ell}   |}\Big\{e^{\frac{C_{\xi}\sqrt{5}}{2}%
|v||   \tau -t^{\ell}   |}-e^{C_{\xi}\frac{-\sqrt{5}}{2}|v||  \tau -t^{\ell}   |}\Big\}    \\
\frac{|v|}{\sqrt{5}}e^{C_{\xi}\frac{|v|}{2}|  \tau -t^{\ell}   |}\Big\{e^{C_{\xi}\frac{%
\sqrt{5}}{2}|v||  \tau -t^{\ell}   |}-e^{-C_{\xi}\frac{\sqrt{5}}{2}|v|| \tau -t^{\ell}   |}\Big\} & \frac{1+%
\sqrt{5}}{2\sqrt{5}}e^{C_{\xi}\frac{1+\sqrt{5}}{2}|v|| \tau -t^{\ell}   |}-\frac{1-\sqrt{5}%
}{2\sqrt{5}}e^{C_{\xi }\frac{1-\sqrt{5}}{2}|v||  \tau -t^{\ell}   |}
\end{array}%
\right]}\\
& \ \ \ \times  \left[
\begin{array}{cc }
|\partial \mathbf{x}_{\parallel }(t^{\ell})|+|\partial \mathbf{x}_{\perp }(t^{\ell})|    \\
|\partial \mathbf{v}_{\parallel }(t^{\ell})|+|\partial \mathbf{v}_{\perp }(t^{\ell})|
\end{array}%
\right]  \\
&\leq  \left[
\begin{array}{c }
e^{C_{\xi }|v|| \tau -t^{\ell}   |}\big\{|\partial \mathbf{x}_{\parallel }(t^{\ell})|+|\partial \mathbf{x}_{\perp
}(  t^{\ell})|\big\}+|  \tau -t^{\ell}  |e^{C_{\xi }|v|| \tau -t^{\ell}   |}\big\{|\partial \mathbf{v}_{\parallel }(t^{\ell})|+|\partial
\mathbf{v}_{\perp }(t^{\ell})|\big\}    \\
|v|^{2}| \tau -t^{\ell}  |e^{C_{\xi }|v||  \tau -t^{\ell} |}\big\{|\partial \mathbf{x}_{\parallel }(t^{\ell})|+|\partial
\mathbf{x}_{\perp }(t^{\ell})|\big\}+ e^{C_{\xi }|v|| \tau -t^{\ell}   |}\big\{|\partial \mathbf{v}_{\parallel
}(t^{\ell})|+|\partial \mathbf{v}_{\perp }(t^{\ell} )|\big\}
\end{array}%
\right].
\end{split}
\end{equation*}%
Since $|v||\tau - t^{\ell}| \lesssim_{\xi} 1$, this proves our claim (\ref{Dxv_free}). The proof of (\ref{Dxv_free_s}) is exactly same but we use $\partial = [ \partial_{\mathbf{X}_{\ell}}(s), \partial_{\mathbf{V}_{\ell}}(s) ]$ to conclude the proof.

From the characteristics ODE, (\ref{ODE_ell}) in the $\mathbf{p}^{\ell}-$spherical coordinate,
\begin{eqnarray*}
\mathbf{x}_{\perp_{\ell }} (s^{\ell}) &=&\mathbf{v}_{\perp_{ \ell}
}^{\ell}(s^{\ell}-t^{\ell})+\int_{t^{\ell}}^{s^{\ell}}\int_{t^{\ell}}^{\tau }  {F}_{\perp_{\ell}}
(s^{\prime };t^{\ell},\mathbf{x}_{\parallel_{\ell} }^{\ell},\mathbf{v}_{\parallel _{\ell}}^{\ell},\mathbf{v}_{\perp _{\ell}}^{\ell})%
\mathrm{d}s^{\prime }\mathrm{d}\tau , \\
\mathbf{x}_{\parallel _{ \ell}}(s^{\ell}) &=&\mathbf{x}_{\parallel _{ \ell}}^{\ell}+\mathbf{v}_{\parallel_{ \ell}
}^{\ell}(s^{\ell}-t^{\ell})+\int_{t^{\ell}}^{s^{\ell}}\int_{t^{\ell}}^{\tau } {F}_{\parallel
}(s^{\prime };t^{\ell},\mathbf{x}_{\parallel _{ \ell}}^{\ell},\mathbf{v}_{\parallel_{ \ell} }^{\ell},\mathbf{v}_{\perp _{ \ell} }^{\ell})%
\mathrm{d}s^{\prime }\mathrm{d}\tau , \\
\mathbf{v}_{\perp_{ \ell} }(s^{\ell}) &=&\mathbf{v}_{\perp_{ \ell} }^{\ell}+\int_{t^{\ell}}^{s^{\ell}} {F}_{\perp_{ \ell} }(\tau
;t^{n},\mathbf{x}_{\parallel _{ \ell}}^{\ell},\mathbf{v}_{\parallel _{ \ell}}^{\ell},\mathbf{v}_{\perp  \ell}^{\ell})\mathrm{d}\tau  , \\
\mathbf{v}_{\parallel_{ \ell} }(s^{\ell}) &=&\mathbf{v}_{\parallel _{ \ell}}^{\ell}+\int_{t^{\ell}}^{s^{\ell}}  {F}_{\parallel_{ \ell}
}(\tau ;t^{\ell},\mathbf{x}_{\parallel _{ \ell}}^{\ell},\mathbf{v}_{\parallel_{ \ell} }^{\ell},\mathbf{v}_{\perp_{ \ell} }^{\ell})\mathrm{d}%
\tau.
\end{eqnarray*}%
Plugging (\ref{Dxv_free}) into (\ref{F_perp}) and (\ref{F||}) and collecting terms, we deduce for $|v||s^{\ell}-t^{\ell}|\lesssim 1$
\begin{eqnarray*}
\frac{\partial \mathbf{x}_{\perp_{\ell  } }(s^{\ell})}{\partial \mathbf{x}_{\parallel_{\ell} }^{\ell}} &\leq
&C_{\Omega }|v|^{2}|s^{\ell}-t^{\ell}|^{2}(1+|v||s^{\ell}-t^{\ell}|)e^{|v||s^{\ell}-t^{\ell}|}\lesssim
_{\Omega } |v|^{2}|s^{\ell}-t^{\ell}|^{2}, \\
\frac{\partial \mathbf{x}_{\perp_{\ell} }(s^{\ell})}{\partial \mathbf{v}_{\perp_{\ell} }^{\ell}} &\leq & |s^{\ell}-t^{\ell}|+ C_{\Omega} |v||s^{\ell}-t^{\ell}|^{2} (1+|v ||s^{\ell}-t^{\ell}|)e^{|v ||s^{\ell}-t^{\ell}|}%
\lesssim _{\Omega }|s^{\ell}-t^{\ell}|,\\
\frac{\partial \mathbf{x}_{\perp_{\ell} }(s^{\ell})}{\partial \mathbf{v}_{\parallel_{\ell} }^{\ell}} &\leq
&C_{\Omega }|v||s^{\ell}-t^{\ell}|^{2}(1+|v||s^{\ell}-t^{\ell}|) e^{|v ||s^{\ell}-t^{\ell}|} \lesssim_{\Omega} |v||s^{\ell}-t^{\ell}| ,\\
\frac{\partial \mathbf{x}_{\parallel_{\ell} }(s^{\ell})}{\partial \mathbf{x}_{\parallel_{\ell} }^{\ell}} &\leq
&\mathbf{Id}_{2,2}+C_{\Omega }|v |^{2}|s^{\ell}-t^{\ell}|^{2}(1+|v||s^{\ell}-t^{\ell}|)e^{{%
|v||s^{\ell}-t^{\ell}|}} \leq  \mathbf{Id}_{2,2} +O_{\Omega}(1)|v|^{2}|s^{\ell}-t^{\ell}|^{2}
, \\
\frac{\partial \mathbf{x}_{\parallel_{\ell} }(s^{\ell})}{\partial \mathbf{v}_{\perp_{\ell} }^{\ell}} &\leq
&C_{\Omega }|s^{\ell}-t^{\ell}|^{2} |v|(1+ |v||s^{\ell}-t^{\ell}|)e^{|v ||s^{\ell}-t^{\ell}|}\lesssim _{\Omega } |v||s^{\ell}-t^{\ell}|^{2}
, \\
\frac{\partial \mathbf{x}_{\parallel_{\ell} }(s^{\ell})}{\partial \mathbf{v}_{\parallel_{\ell} }^{\ell}} &\leq
&|s^{\ell}-t^{\ell}|\Big\{\mathbf{Id}_{2,2}+C_{\Omega
}|v||s^{\ell}-t^{\ell}|(1+|v||s^{\ell}-t^{\ell}|)e^{|v||s^{\ell}-t^{\ell}|}\Big\}\\
&\leq&
|s^{\ell}-t^{\ell}| \mathbf{Id}_{2,2} + O_{\Omega}(1)|v||s^{\ell}-t^{\ell}|^{2} ,\\
\frac{\partial \mathbf{v}_{\perp_{\ell} }(s^{\ell})}{\partial \mathbf{x}_{\parallel_{\ell} }^{\ell}} &\leq
&C_{\Omega } |s^{\ell}-t^{\ell}| |v|^{2}(1+ |v||s^{\ell}-t^{\ell}|)   e^{|v||s^{\ell}-t^{\ell}|}\lesssim _{\Omega }    |v|^{2} |s^{\ell}-t^{\ell}| , \\
\frac{\partial \mathbf{v}_{\perp_{\ell} }(s^{\ell})}{\partial \mathbf{v}_{\perp_{\ell} }^{\ell}} &\leq
&1+C_{\Omega }|s^{\ell}-t^{\ell}||v| (1+ |v||s^{\ell}-^{\ell}|)e^{|v||s^{\ell}-t^{\ell}|}\leq 1+ O_{\Omega}(1) |v||s^{\ell}-t^{\ell}|,
\\
\frac{\partial \mathbf{v}_{\perp_{\ell} }(s^{\ell})}{\partial \mathbf{v}_{\parallel_{\ell} }^{\ell}} &\leq
&C_{\Omega } |s^{\ell}-t^{\ell}| |v| (1+ |v||s^{\ell}-t^{\ell}|) e^{|v||s^{\ell}-t^{\ell}|}\lesssim
_{\Omega } |v|  |s^{\ell}-t^{\ell}|,\\
\frac{\partial \mathbf{v}_{\parallel_{\ell} }(s^{\ell})}{\partial \mathbf{x}_{\parallel_{\ell} }^{\ell}} &\leq
&C_{\Omega }|s^{\ell}-t^{\ell}| |v|^{2}(1+ |v||s^{\ell}-t^{\ell}|)
e^{|v||s^{\ell}-t^{\ell}|}\lesssim _{\Omega } |v|^{2} |s^{\ell}-t^{\ell}| ,
\\
\frac{\partial \mathbf{v}_{\parallel _{\ell}}(s^{\ell})}{\partial \mathbf{v}_{\perp_{\ell} }^{\ell}} &\leq
&C_{\Omega }|s^{\ell}-t^{\ell}||v|(1+|v||s^{\ell}-t^{\ell}|)e^{|v||s^{\ell}-t^{\ell}|}\lesssim
_{\Omega } |v|  |s^{\ell}-t^{\ell}|, \\
\frac{\partial \mathbf{v}_{\parallel_{\ell} }(s^{\ell})}{\partial \mathbf{v}_{\parallel_{\ell} }^{\ell}} &\leq
&\mathbf{Id}_{2,2}+C_{\Omega
}|s^{\ell}-t^{\ell}||v|(1+ |v||s^{\ell}-t^{\ell}|)e^{|v||  s^{\ell}-t^{\ell}|}\leq \mathbf{Id}_{2,2} + O_{\Omega}(1)|v||s^{\ell}-t^{\ell}| , \\
\end{eqnarray*}
and this proves the claim (\ref{s1tstar}).

\vspace{8pt}

\noindent\textit{Step 3. From $t-$plane to the first bounce}

\vspace{4pt}

We choose $s^{1} \in (t^{1}, \frac{t^{1}+t}{2})\subset (t^{1},t)$ such that $|v||t^{1}-s^{1}|\ll 1$ and the polar coordinate $(\mathbf{X}_{1}(s^{1}), \mathbf{V}_{1}(s^{1}))$ is well-defined. More precisely we choose $0 < \Delta  $ such that $|v||t-\Delta -t^{1}| \ll 1$ and define
\begin{equation}
s^{1}:= t - \Delta.\label{Delta_step3}
\end{equation}

Then, by the chain rule,
 \begin{equation}\notag
 \begin{split}
& \frac{\partial (t^{1}, 0 , \mathbf{x}_{\parallel_{1}}^{1} , \mathbf{v}_{\perp_{1}}^{1} , \mathbf{v}_{\parallel_{1}}^{1}  )}{\partial (t,x,v)}\\
 & = \frac{\partial (t^{1}, 0 , \mathbf{x}_{\parallel_{1}}^{1}, \mathbf{v}_{\perp_{1}}^{1}, \mathbf{v}_{\parallel_{1}}^{1})}{\partial (s^{1}, X_{\mathbf{cl}}(s^{1} ), V_{\mathbf{cl}}(s^{1} ))}
\frac{\partial (s^{1}, X_{\mathbf{cl}}(s^{1}), V_{\mathbf{cl}}(s^{1}))}{\partial (t,x,v)}
\\
&= \frac{\partial (t^{1}, 0 , \mathbf{x}_{\parallel_{1}}^{1}, \mathbf{v}_{\perp_{1}}^{1}, \mathbf{v}_{\parallel_{1}}^{1})}{ \partial (s^{1}, \mathbf{x}_{\perp_{1}}(s^{1}), \mathbf{x}_{\parallel_{1}}(s^{1}), \mathbf{v}_{\perp_{1}}(s^{1}), \mathbf{v}_{\parallel_{1}}(s^{1}) )}
\frac{\partial (s^{1}, \mathbf{X}_{1}(s^{1}),  \mathbf{V}_{ 1}(s^{1})    )}{\partial (s^{1}, X_{ \mathbf{cl}}(s^{1} ), V_{ \mathbf{cl}}(s^{1} ))}
\frac{\partial (s^{1}, X_{\mathbf{cl}}(s^{1}), V_{ \mathbf{cl}}(s^{1}))}{\partial (t,x,v)}.
 \end{split}
 \end{equation}
We fix $\mathbf{p}^{1}-$spherical coordinate and drop the index of the chart.

 Firstly, we claim
 \begin{equation}\label{t1s1}
 \begin{split}
& \frac{\partial (t^{1}, 0 , \mathbf{x}_{\parallel}^{1}, \mathbf{v}_{\perp}^{1}, \mathbf{v}_{\parallel}^{1})}{ \partial (s^{1}, \mathbf{x}_{\perp}(s^{1}), \mathbf{x}_{\parallel}(s^{1}), \mathbf{v}_{\perp}(s^{1}), \mathbf{v}_{\parallel}(s^{1}) )}\\
\lesssim_{\Omega} & {\left[\begin{array}{c|cc|cc}
1 &  \frac{1}{|\mathbf{v}_{\perp}^{1}|}  &   \frac{|v|^{2}|s^{1}-t^{1}|^{2}}{|\mathbf{v}_{\perp}^{1}|} & \frac{|s^{1}-t^{1}|}{|\mathbf{v}_{\perp}^{1}|} & \frac{|v||s^{1}-t^{1}|^{2}}{|\mathbf{v}_{\perp}^{1}|}  \\ \hline
0 & 0 & \mathbf{0}_{1,2} & 0 & \mathbf{0}_{1,2} \\
\mathbf{0}_{2,1} &     \frac{|v|}{|\mathbf{v}_{\perp}^{1}|}  + |v|^{2}|s^{1}-t^{1}|^{2} & \mathbf{Id}_{2,2} + |v||s^{1}-t^{1}| &   \frac{|s^{1}-t^{1}||v|}{|\mathbf{v}_{\perp}^{1}|} + |s^{1}-t^{1}|^{2}|v| & |s^{1}-t^{1}|\\ \hline
0& \frac{|v|^{2}}{|\mathbf{v}_{\perp}^{1}|} + |v|^{2}|s^{1}-t^{1}|  & \frac{|v|^{2}}{|\mathbf{v}^{1}_{\perp}|} + |v|^{2}|s^{1}-t^{1}| & 1+ |v||s^{1}-t^{1}| & |v||s^{1}-t^{1}| \\
\mathbf{0}_{2,1} & \frac{|v|^{2}}{|\mathbf{v}_{\perp}^{1}|} + |v|^{2}|s^{1}-t^{1}| &   |v|^{2}|s^{1}-t^{1}|   & 1+ |v||s^{1}-t^{1}|  &  \mathbf{Id}_{2}+ |v| |s^{1}-t^{1}| \end{array}\right]}.
 \end{split}
 \end{equation}

The $t^{1}$ is determined via $\mathbf{x}_{\perp}(t^{1})=0$, i.e.
\begin{equation}\label{equation_t1}
0= \mathbf{x}_{\perp}(s^{1}) - \mathbf{v}_{\perp}(s^{1}) (s^{1}-t^{1}) + \int^{s^{1}}_{t^{1}} \int^{s^{1}}_{s} F_{\perp}(\mathbf{X}_{\mathbf{cl}}(\tau), \mathbf{V}_{\mathbf{cl}}(\tau)) \mathrm{d}\tau \mathrm{d}s,
\end{equation}
where $\mathbf{X}_{\mathbf{cl}}(\tau)  = \mathbf{X}_{\mathbf{cl}}(\tau; s^{1}, \mathbf{X}_{\mathbf{cl}}(s^{1};t,x,v) , \mathbf{V}_{\mathbf{cl}}(s^{1};t,x,v) ), \mathbf{V}_{\mathbf{cl}}(\tau)   =\mathbf{V}_{\mathbf{cl}} (\tau; s^{1}, \mathbf{X}_{\mathbf{cl}}(s^{1};t,x,v) , \mathbf{V}_{\mathbf{cl}}(s^{1};t,x,v) ).$

Recall that, from (\ref{dot_v_perp}) and (\ref{dot_v_||}) and (\ref{v_perp}),
\begin{equation}\notag
\begin{split}
\mathbf{v}_{\perp}^{1} &= - \lim_{s \downarrow t^{1}} \mathbf{v}_{\perp}(s) = - \mathbf{v}_{\perp}(s^{1}) + \int^{s^{1}}_{t^{1}} F_{\perp}(\mathbf{X}_{\mathbf{cl}}(\tau), \mathbf{V}_{\mathbf{cl}}(\tau)) \mathrm{d}\tau,\\
\mathbf{x}_{\parallel}^{1} & = \mathbf{x}_{\parallel}(s^{1}) - (s^{1}-t^{1}) \mathbf{v}_{\parallel}(s^{1}) + \int^{s^{1}}_{t^{1}} \int^{s^{1}}_{\tau} F_{\parallel}(\mathbf{X}_{\mathbf{cl}}(\tau), \mathbf{V}_{\mathbf{cl}}(\tau)) \mathrm{d}\tau \mathrm{d}s^{1},\\
\mathbf{v}_{\parallel}^{1} &= \mathbf{v}_{\parallel}(s^{1}) - \int^{s^{1}}_{t^{1}} F_{\parallel}(\mathbf{X}_{\mathbf{cl}}(\tau), \mathbf{V}_{\mathbf{cl}}(\tau)) \mathrm{d}\tau.
\end{split}
\end{equation}
Note that since ODE is autonomous we have $\frac{\partial t^{1}}{\partial s} =1 , \ \frac{\partial (\mathbf{x}^{1} , \mathbf{v}^{1})}{\partial s^{1}}=0$. From the fact $|s^{1}-t^{1}| \lesssim_{\xi}\min\{  \frac{| \mathbf{v}_{\perp}^{1}|}{|v|^{2}} , t \}$ and (\ref{40}) and (\ref{41}),
and (\ref{Dxv_free_s}) and (\ref{D_F}) to have
\begin{equation}\notag
\begin{split}
\frac{\partial t^{1}}{\partial \mathbf{x}_{\perp }(s^{1})}& =\frac{1}{%
 \mathbf{v}_{\perp }^{1} }\left\{
1+\int^{s^{1}}_{t^{1}}\int^{s^{1}}_{s}\frac{\partial }{\partial \mathbf{x}_{\perp
}(s^{1})}F_{\perp }( \mathbf{X}_{\mathbf{cl}}(\tau), \mathbf{V}_{\mathbf{cl}}(\tau)   )\mathrm{d}\tau \mathrm{d}s\right\}\\
&  \lesssim_{\xi }  \frac{1}{|\mathbf{v}_{\perp}^{1}|} \left\{1 + \int^{s^{1}}_{t^{1}} \int^{s^{1}}_{s} \big[1+ |v| (s^{1}-\tau) \big] |v|^{2} e^{C_{\xi}|v|(s^{1}-\tau)}
 \mathrm{d}\tau\mathrm{d}s\right\} \\
 & \lesssim_{\xi} \frac{1}{|\mathbf{v}_{\perp}^{1}|} \Big\{   1+ \big[ 1+ |v||s^{1}-t^{1}|\big] |v|^{2}|s^{1}-t^{1}|^{2} e^{C_{\xi}|v||s^{1}-t^{1}|}  \Big\}  \lesssim_{\xi,t} \frac{1}{|\mathbf{v}_{\perp}^{1}|},\\ 
\frac{\partial t^{1}}{\partial \mathbf{x}_{\parallel }(s^{1})}& =\frac{1}{%
 \mathbf{v}^{1}_{\perp } }\int^{s^{1}}_{t^{1}}%
\int^{s^{1}}_{s}\frac{\partial }{\partial \mathbf{x}_{\parallel }(s^{1})}F_{\perp
}(    \mathbf{X}_{\mathbf{cl}}(\tau), \mathbf{V}_{\mathbf{cl}}(\tau)     )\mathrm{d}\tau \mathrm{d}s\\
& \lesssim_{\xi} \frac{1}{|\mathbf{v}_{\perp}^{1}|} \left\{ \int_{t^{1}}^{s^{1}} \int^{s^{1}}_{s} \big[1+ |v| (s^{1}-\tau) \big] |v|^{2} e^{C_{\xi}|v|(s^{1}-\tau)}  \mathrm{d}\tau \mathrm{d}s    \right\} \\
& \lesssim_{\xi} \frac{1}{|\mathbf{v}_{\perp}^{1}|}   \big[ 1+ |v||s^{1}-t^{1}|\big] |v|^{2}|s^{1}-t^{1}|^{2} e^{C_{\xi}|v||s^{1}-t^{1}|}   \lesssim_{\xi,t}
\frac{|v|^{2}|s^{1}-t^{1}|^{2}}{|\mathbf{v}_{\perp}^{1}|}
,\\
\frac{\partial t^{1}}{\partial \mathbf{v}_{\perp }(s^{1})}& =\frac{1}{\mathbf{v}_{\perp }^{1}}%
\left\{ (t^{1}-s^{1})+\int^{s^{1}}_{t^{1}}\int^{s^{1}}_{s}\frac{\partial }{%
\partial \mathbf{v}_{\perp }(s^{1})}F_{\perp }( \mathbf{X}_{\mathbf{cl}}(\tau), \mathbf{V}_{\mathbf{cl}}(\tau)  )\mathrm{d}\tau \mathrm{d}%
s\right\} \\
& \lesssim_{\xi} \frac{|s^{1}-t^{1}|}{|\mathbf{v}_{\perp}^{1}|} + \frac{1}{|\mathbf{v}_{\perp}^{1}|} \int^{s^{1}}_{t^{1}} \int^{s^{1}}_{s}   |v|   \big[1+|v||s^{1}-\tau|\big]   e^{C_{\xi}|v| (s^{1}-\tau)  } \mathrm{d}\tau \mathrm{d} s    \\
& \lesssim_{\xi} \frac{|s^{1}-t^{1}|}{|\mathbf{v}_{\perp}^{1}|} \Big\{ 1+ |v| |s^{1}-t^{1}| e^{C_{\xi}|v||s^{1}-t^{1}|}  \Big\} \lesssim_{\xi,t}  \frac{|s^{1}-t^{1}|}{|\mathbf{v}_{\perp}^{1}|}
,\\
\frac{\partial t^{1}}{\partial \mathbf{v}_{\parallel }(s_{1})}& =\frac{1}{
 \mathbf{v}_{\perp }^{1} }\int^{s_{1}}_{t^{1}}%
\int^{s_{1}}_{s}\frac{\partial }{\partial \mathbf{v}_{\parallel }(s_{1})}F_{\perp
}( \mathbf{X}_{\mathbf{cl}}(\tau), \mathbf{V}_{\mathbf{cl}}(\tau) )\mathrm{d}\tau \mathrm{d}s\\
&  \lesssim_{\xi}  \frac{1}{|\mathbf{v}_{\perp}^{1}|} \int^{s^{1}}_{t^{1}} \int^{s^{1}}_{s}   |v| \big[1+|v||s^{1}-\tau|\big]  e^{C_{\xi}|v| (s^{1}-\tau)  } \mathrm{d}\tau \mathrm{d} s   \\
&   \lesssim_{\xi} \frac{|s^{1}-t^{1}|}{|\mathbf{v}_{\perp}^{1}|}  |v|  |s^{1}-t^{1}|  e^{C_{\xi}|v||s^{1}-t^{1}|} \\& \lesssim_{\xi,t}
\frac{|v||s^{1}-t^{1}|^{2}   }{|\mathbf{v}_{\perp}^{1}|}.
\end{split}%
\end{equation}%
Together with the above estimates and (\ref{Dxv_free_s}) and (\ref{D_F}),
\begin{equation}\notag
\begin{split}
\frac{\partial \mathbf{x}^{1}_{\parallel}}{\partial \mathbf{x}_{\perp}(s^{1})} & = \frac{\partial t^{1}}{\partial \mathbf{x}_{\perp}(s^{1})} \mathbf{v}_{\parallel}^{1} + \int^{s^{1}}_{t^{1}} \int^{s^{1}}_{s} \frac{\partial}{\partial \mathbf{x}_{\perp}(s^{1})} F_{\parallel}(\mathbf{X}_{\mathbf{cl}}(\tau), \mathbf{V}_{\mathbf{cl}}(\tau)) \mathrm{d}\tau \mathrm{d}s\\
& \lesssim_{\xi}  \frac{|v|}{|\mathbf{v}_{\perp}^{1}|} +     \big[ 1+ |v||s^{1}-t^{1}|\big] |v|^{2}|s^{1}-t^{1}|^{2} e^{C_{\xi}|v||s^{1}-t^{1}|} \lesssim_{\xi,t}  \frac{|v|}{|\mathbf{v}_{\perp}^{1}|} + |v|^{2}|s^{1}-t^{1}|^{2}   ,\\
\frac{\partial \mathbf{x}_{\parallel}^{1}}{\partial  \mathbf{x}_{\parallel}(s^{1})} &
 = \mathbf{Id}_{2,2}  + \mathbf{v}_{\parallel}^{1} \frac{\partial t^{1}}{\partial \mathbf{x}_{\parallel}(s^{1})} + \int^{s^{1}}_{t^{1}} \int^{s^{1}}_{s} \frac{\partial}{\partial \mathbf{x}_{\parallel}(s^{1})} F_{\parallel}(  \mathbf{X}_{\mathbf{cl}}(\tau), \mathbf{V}_{\mathbf{cl}}(\tau) ) \mathrm{d}\tau \mathrm{d}s\\
&   \lesssim_{\xi,t} \mathbf{Id}_{2,2} +   |v|^{2}|s^{1}-t^{1}|^{2} \lesssim_{\xi,t} \mathbf{Id}_{2,2} + |v||s^{1}-t^{1}| ,\\
\frac{\partial \mathbf{x}_{\parallel}^{1}}{\partial  \mathbf{v}_{\perp}(s^{1})} & = \frac{\partial t^{1}}{\partial \mathbf{v}_{\perp}(s^{1})} \mathbf{v}_{\parallel}^{1} + \int^{s^{1}}_{t^{1}} \int^{s^{1}}_{s} \frac{\partial }{\partial \mathbf{v}_{\perp}(s^{1})} F_{\parallel}(  \mathbf{X}_{\mathbf{cl}}(\tau), \mathbf{V}_{\mathbf{cl}}(\tau)  ) \mathrm{d}\tau \mathrm{d}s\\
&   \lesssim_{\xi, t}  \frac{|s^{1}-t^{1}||v|}{|\mathbf{v}_{\perp}^{1}|} + |s^{1}-t^{1}|^{2} |v|
, 
\end{split}
\end{equation}
\begin{equation}\notag
\begin{split}
\frac{\partial \mathbf{x}_{\parallel}^{1}}{\partial \mathbf{v}_{\parallel}(s^{1})} & = -(s^{1}-t^{1}) \mathbf{Id}_{2,2} + \mathbf{v}_{\parallel}^{1} \frac{\partial t^{1}}{\partial \mathbf{v}_{\parallel}(s^{1})} + \int^{s^{1}}_{t^{1}} \int^{s^{1}}_{s} \frac{\partial}{\partial \mathbf{v}_{\parallel}(s^{1})} F_{\parallel} ( \mathbf{X}_{\mathbf{cl}}(\tau), \mathbf{V}_{\mathbf{cl}}(\tau)   ) \mathrm{d}\tau \mathrm{d}s  \\
&\lesssim_{\xi}- (s^{1}-t^{1}) \mathbf{Id}_{2,2} + |v| \frac{|s^{1}-t^{1}|}{|\mathbf{v}_{\perp}^{1}|} |v||s^{1}-t^{1}| + |v||s^{1}-t^{1}|^{2} \big[1+ |v||s^{1}-t^{1}|\big] \\
& \lesssim_{\xi,t} |s^{1}-t^{1}| \Big(1+ \frac{|v|^{2} |s^{1}-t^{1}|}{|\mathbf{v}_{\perp}^{1}|}\Big)   \lesssim_{\xi,t} |s^{1}-t^{1}|
.
\end{split}
\end{equation}
Moreover by (\ref{Dxv_free_s}) and (\ref{D_F})
\begin{equation}\notag
\begin{split}
\frac{\partial \mathbf{v}_{\perp }^{1}}{\partial \mathbf{x}_{\perp }(s^{1})} &=\frac{%
-F_{\perp }( x^{1},v)}{\mathbf{v}_{\perp }^{1}}-\frac{F_{\perp }(x^{1},v)}{\mathbf{v}_{\perp }^{1}}%
\int_{t^{1}}^{s^{1}}\int^{s^{1}}_{s}\frac{\partial F_{\perp }(  \mathbf{X}_{\mathbf{cl}}(\tau), \mathbf{V}_{\mathbf{cl}}(\tau)    )  }{\partial \mathbf{x}_{\perp
}(s^{1})}    \mathrm{d}\tau \mathrm{d}s+\int^{s^{1}}_{t^{1}}%
\frac{\partial F_{\perp }(   \mathbf{X}_{\mathbf{cl}}(\tau), \mathbf{V}_{\mathbf{cl}}(\tau) )}{\partial \mathbf{x}_{\perp }(s^{1})}\mathrm{d}\tau
\\
& \lesssim   \frac{F_{\perp}(x^{1},v)}{|\mathbf{v}_{\perp}^{1}|} + \Big( |v| + \frac{F_{\perp}(x^{1},v)}{|\mathbf{v}_{\perp}^{1}|}|v||s^{1}-t^{1}| \Big) \big[1+|v| |s^{1}-t^{1}|\big] |v||s^{1}-t^{1}| e^{C_{\xi}|v||s^{1}-t^{1}|}\\
&\lesssim \frac{|v|^{2}}{|\mathbf{v}_{\perp}^{1}|} + |v|^{2}|s^{1}-t^{1}|
,\\
\frac{\partial \mathbf{v}_{\perp }^{1}}{\partial \mathbf{x}_{\parallel }(s^{1})} &=\frac{%
-F_{\perp }(x^{1},v)}{\mathbf{v}_{\perp }^{1}}+\int^{s^{1}}_{t^{1}}\frac{\partial }{%
\partial \mathbf{x}_{\parallel }(s^{1})}F_{\perp }(  \mathbf{X}_{\mathbf{cl}}(\tau), \mathbf{V}_{\mathbf{cl}}(\tau))\mathrm{d}\tau\\
& \lesssim
\frac{|F_{\perp}(x^{1},v)|}{|\mathbf{v}_{\perp}^{1}|} +   |v|^{2}|s^{1}-t^{1}| \big(1+ |v||s^{1}-t^{1}|\big) e^{C_{\xi}|v||s^{1}-t^{1}|} \lesssim \frac{|v|^{2}}{|\mathbf{v}_{\perp}^{1}|} + |v|^{2}|s^{1}-t^{1}|
 ,\\
\frac{\partial \mathbf{v}_{\perp}^{1}}{\partial \mathbf{v}_{\perp}(s^{1})} & = -1 + \frac{(s^{1}-t^{1})F_{\perp}(x^{1},v)}{\mathbf{v}_{\perp}^{1}} -\frac{F_{\perp}(x^{1},v)}{\mathbf{v}_{\perp}^{1}} \int^{s^{1}}_{t^{1}} \int^{s^{1}}_{s} \frac{\partial F_{\perp}(\mathbf{X}_{\mathbf{cl}}(\tau), \mathbf{V}_{\mathbf{cl}}(\tau)) }{\partial \mathbf{v}_{\perp}(s^{1})}  -\int_{s_{1}}^{t^{1}}\frac{\partial F_{\perp
}( \mathbf{X}_{\mathbf{cl}}(s), \mathbf{V}_{\mathbf{cl}}(s)    ) }{\partial \mathbf{v}_{\perp }(s_{1})}\mathrm{d}s \\
&\lesssim -1 + \frac{|s^{1}-t^{1}| |F_{\perp}(x^{1},v)| }{|\mathbf{v}_{\perp}^{1}|} \Big\{ 1+ |v| |s^{1}-t^{1}|  e^{C_{\xi}|v||s^{1}-t^{1}|}\Big\}
+ |v|  |s^{1}-t^{1}| e^{C_{\xi}|v||s^{1}-t^{1}|}\\
&\lesssim_{\xi}  1+ |v||s^{1}-t^{1}|,\\
\frac{\partial \mathbf{v}_{\perp }^{1}}{\partial \mathbf{v}_{\parallel }(s^{1})} &=\frac{%
-F_{\perp }(x^{1},v)}{\mathbf{v}_{\perp }^{1}}\int_{s^{1}}^{t^{1}}\int_{s^{1}}^{s}\frac{%
\partial F_{\perp }(   \mathbf{X}_{\mathbf{cl}}(\tau), \mathbf{V}_{\mathbf{cl}}(\tau) ) }{\partial \mathbf{v}_{\parallel }(s^{1})}\mathrm{d}\tau
\mathrm{d}s-\int_{s^{1}}^{t^{1}}\frac{\partial  F_{\perp }( \mathbf{X}_{\mathbf{cl}}(\tau), \mathbf{V}_{\mathbf{cl}}(\tau))}{\partial \mathbf{v}_{\parallel
}(s^{1})}\mathrm{d}\tau \\
& \lesssim  \frac{|v|^{2}}{|\mathbf{v}_{\perp}^{1}|}|s^{1}-t^{1}|^{2}|v| + |s^{1}-t^{1}||v| \lesssim |v||s^{1}-t^{1}| \Big(1+ \frac{|s^{1}-t^{1}||v|^{2}}{|\mathbf{v}_{\perp}^{1}|}\Big)\lesssim |v||s^{1}-t^{1}|,
\end{split}
\end{equation}
and
\begin{equation}\notag
\begin{split}
\frac{\partial \mathbf{v}_{\parallel}^{1}}{\partial \mathbf{x}_{\perp}(s^{1})} & = \frac{\partial t^{1}}{\partial \mathbf{x}_{\perp}(s^{1})} F_{\parallel}(x^{1},v) - \int^{s^{1}}_{t^{1}} \frac{\partial}{\partial \mathbf{x}_{\perp}(s^{1})} F_{\parallel} (\mathbf{X}_{\mathbf{cl}}(\tau), \mathbf{V}_{\mathbf{cl}}(\tau)) \mathrm{d}\tau\\
& \lesssim_{\xi} \frac{|F_{\parallel}(x^{1},v)|}{|\mathbf{v}_{\perp}^{1}|} \big\{  1+ [1+|v||s^{1}-t^{1}|] |v|^{2} |s^{1}-t^{1}|^{2} e^{C_{\xi}|v||s^{1}-t^{1}|}  \big\}
+  |v|^{2} |s^{1}-t^{1}| [1+|v||s^{1}-t^{1}|] e^{C_{\xi}|v||s^{1}-t^{1}|} \\
& \lesssim_{\xi,t} \frac{|v|^{2}}{|\mathbf{v}_{\perp}^{1}|} +   |v|^{2}|s^{1}-t^{1}|
,\\
\frac{\partial \mathbf{v}_{\parallel}^{1}}{\partial \mathbf{x}_{\parallel}(s^{1})} & = \frac{\partial t^{1}}{\partial \mathbf{x}_{\parallel}(s^{1})} F_{\parallel}(x^{1},v) -\int^{s^{1}}_{t^{1}} \frac{\partial}{\partial \mathbf{x}_{\parallel}(s^{1})} F_{\parallel}(  \mathbf{X}_{\mathbf{cl}}(\tau), \mathbf{V}_{\mathbf{cl}}(\tau)  ) \mathrm{d}\tau\\
& \lesssim \frac{|F_{\parallel}(x^{1},v)|}{|\mathbf{v}_{\perp}^{1}|}\big[ 1+|v||s^{1}-t^{1}| \big] |v|^{2}|s^{1}-t^{1}|^{2} e^{C_{\xi}|v||s^{1}-t^{1}|} +   |v|^{2}|s^{1}-t^{1}| \big[1+|v||s^{1}-t^{1}|\big] e^{C_{\xi}|v||s^{1}-t^{1}|}\\
&\lesssim_{\xi,t}  |v|^{2}|s^{1}-t^{1}|
,
\end{split}
\end{equation}
\begin{equation}\notag
\begin{split}
\frac{\partial \mathbf{v}_{\parallel}^{1}}{\partial \mathbf{v}_{\perp}(s^{1})} & = \frac{\partial t^{1}}{\partial \mathbf{v}_{\perp}(s^{1})} F_{\parallel} (x^{1},v) - \int^{s^{1}}_{t^{1}} \frac{\partial}{\partial \mathbf{v}_{\perp}(s^{1})} F_{\parallel} (\mathbf{X}_{\mathbf{cl}}(\tau), \mathbf{V}_{\mathbf{cl}}(\tau)) \mathrm{d}\tau\\
&\lesssim \frac{|s^{1}-t^{1}||v|^{2}}{|\mathbf{v}_{\perp}^{1}|}  + |v| |s^{1}-t^{1}|  \lesssim 1+|v||s^{1}-t^{1}|
,\\
\frac{\partial \mathbf{v}_{\parallel}^{1}}{\partial \mathbf{v}_{\parallel} (s^{1})} & = \mathbf{Id}_{2,2} + \frac{\partial t^{1}}{\partial \mathbf{v}_{\parallel}(s^{1})} F_{\parallel}(x^{1},v) - \int^{s^{1}}_{t^{1}} \frac{\partial }{\partial \mathbf{v}_{\parallel}(s^{1})} F_{\parallel} (\mathbf{X}_{\mathbf{cl}}(\tau), \mathbf{V}_{\mathbf{cl}}(\tau) \mathrm{d}\tau\\
& \lesssim \mathbf{Id}_{2,2}  + \frac{|v|^{3}|s^{1}-t^{1}|^{2}}{|\mathbf{v}_{\perp}^{1}|} + |s^{1}-t^{1}||v| \big[1+|v||s^{1}-t^{1}|\big] \lesssim_{\xi,t}
\mathbf{Id}_{2,2}   + |v||s^{1}-t^{1}|
.
\end{split}
\end{equation}

Secondly, we claim
\begin{equation}\label{s1t}
\begin{split}
 &\frac{\partial (s^{1}, \mathbf{X}_{1}(s^{1}),  \mathbf{V}_{1}(s^{1})    )} {\partial (t,x,v)}
 =\frac{\partial (s^{1}, \mathbf{X}_{1}(s^{1}),  \mathbf{V}_{1}(s^{1})    )}{\partial (s^{1}, X_{\mathbf{cl}}(s^{1} ), V_{\mathbf{cl}}(s^{1} ))}
\frac{\partial (s^{1}, X_{\mathbf{cl}}(s^{1}), V_{\mathbf{cl}}(s^{1}))}{\partial (t,x,v)}
\\
&=
 \left[\begin{array}{c|c|ccccccccccccccc}
1 & \mathbf{0}_{1,3} & \mathbf{0}_{1,3}\\ \hline
 & \frac{(\partial_{1} \mathbf{\eta} \times \partial_{2}  \mathbf{\eta})^{T}}{  \mathbf{n}   \cdot (\partial_{1} \mathbf{\eta} \times \partial_{2} \mathbf{\eta})} + O_{\xi}(|v||t^{1}-s^{1}|) &  O_{\xi} ( |t-s^{1}|)\\
\mathbf{0}_{3,1} & \frac{(\partial_{2} \mathbf{\eta} \times   \mathbf{n} )^{T}}{  \mathbf{n} \cdot (\partial_{1}  \mathbf{\eta} \times \partial_{2}  \mathbf{\eta})} +  O_{\xi}(|v||t^{1}-s^{1}|) &    O_{\xi}( |t-s^{1}|)\\
 & \frac{(    \mathbf{n}  \times  \partial_{2} \mathbf{\eta})^{T}}{  \mathbf{n} \cdot (\partial_{1} \mathbf{\eta} \times \partial_{2}  \mathbf{\eta})} +  O_{\xi}(|v||t^{1}-s^{1}|)& O_{\xi}( |t-s^{1}|)\\ \hline
  & O_{\xi}(  |v|) &  \frac{(\partial_{1}  \mathbf{\eta} \times \partial_{2}  \mathbf{\eta})^{T}}{   \mathbf{n} \cdot (\partial_{1}  \mathbf{\eta} \times \partial_{2} \mathbf{\eta})} +  O_{\xi}(|v| |t-s^{1}| ) \\
\mathbf{0}_{3,1}&O_{\xi}(  |v|) &  \frac{(\partial_{2}  \mathbf{\eta} \times   \mathbf{n} )^{T}}{   \mathbf{n} \cdot (\partial_{1}  \mathbf{\eta} \times \partial_{2} \mathbf{\eta})} +  O_{\xi}(|v| |t-s^{1}| )\\
 &O_{\xi}( |v|) &  \frac{( \mathbf{n}   \times \partial_{1} \mathbf{\eta})^{T}}{   \mathbf{n} \cdot (\partial_{1}  \mathbf{\eta} \times \partial_{2} \mathbf{\eta})} + O_{\xi}(|v| |t-s^{1}| )
\end{array} \right],
\end{split}
\end{equation}
where the entries are evaluated at $(\mathbf{X}_{1}(s^{1}), \mathbf{V}_{1}(s^{1}))$. Note that $|v||t^{1}-s^{1}|\lesssim_{\xi } 1$.

Clearly
\begin{equation}\notag
\begin{split}
 \left[\begin{array}{ccccccc}  {\partial s^{1}}/{\partial (s^{1}, X_{\mathbf{cl}}(s^{1} ), V_{\mathbf{cl}}(s^{1} ))} \\
\partial \mathbf{X}_{\mathbf{cl}}(s^{1}) / {\partial (s^{1}, X_{\mathbf{cl}}(s^{1} ), V_{\mathbf{cl}}(s^{1} ))} \\
\partial \mathbf{V}_{\mathbf{cl}}(s^{1}) / {\partial (s^{1}, X_{\mathbf{cl}}(s^{1} ), V_{\mathbf{cl}}(s^{1} ))}
\end{array}\right] =
\left[\begin{array}{c|c}1 & \mathbf{0}_{1,6}   \\ \hline
\mathbf{0}_{6,1} &   \frac{\partial (\mathbf{X}_{\mathbf{cl}}(s^{1}), \mathbf{V}_{\mathbf{cl}}(s^{1})   )}{ \partial (X_{\mathbf{cl}} (s^{1}), V_{\mathbf{cl}}(s^{1})   )}  \\
  \end{array}\right].
\end{split}
\end{equation}
Now we consider the right lower 6 by 6 submatrix.
Recall, from (\ref{jac_Phi})
\begin{equation}\notag
\begin{split}
& \frac{\partial ( {X}_{\mathbf{cl}}(s^{1}),  {V}_{\mathbf{cl}}(s^{1}))}{\partial (\mathbf{X}_{\mathbf{cl}}(s^{1}), \mathbf{V}_{\mathbf{cl}}(s^{1}))}
= \frac{\partial \Phi (  \mathbf{X}_{\mathbf{cl}}(s^{1}), \mathbf{V}_{\mathbf{cl}}(s)  )}{\partial ( \mathbf{X}_{\mathbf{cl}}(s^{1}), \mathbf{V}_{\mathbf{cl}}(s) )}
 : =  \left[\begin{array}{c|c} A & \mathbf{0}_{3,3} \\ \hline B & A   \end{array}\right]  + \mathbf{x}_{\perp} \left[\begin{array}{c|c}   \mathbf{0}_{3,3}  & \mathbf{0}_{3,3} \\ \hline D &  \mathbf{0}_{3,3}   \end{array}\right].
\end{split}
\end{equation}
Note that, from (\ref{jac_Phi1}) and (\ref{nondegenerate_eta}),
\begin{equation}\notag
\begin{split}
\text{det}(A)&=\text{det} \Big[ \begin{array}{ccc} [- \mathbf{n}(\mathbf{x}_{\parallel})] & \partial_{\mathbf{x}_{\parallel,1}} \mathbf{\eta}(\mathbf{x}_{\parallel}) & \partial_{\mathbf{x}_{\parallel,2}} \mathbf{\eta}(\mathbf{x}_{\parallel})  \end{array}  \Big]  =
[-\mathbf{n}(\mathbf{x}_{\parallel}) ] \cdot \big( \partial_{\mathbf{x}_{\parallel,1}} \mathbf{\eta}(\mathbf{x}_{\parallel}) \times \partial_{\mathbf{x}_{\parallel,2}} \mathbf{\eta}(x_{\parallel}) \big) \neq 0,\\
A^{-1} &= \frac{1}{[-\mathbf{n}]\cdot (\partial_{\mathbf{x}_{\parallel,1}} \mathbf{\eta} \times \partial_{\mathbf{x}_{\parallel,2}} \mathbf{\eta})}  \bigg[ ( \partial_{\mathbf{x}_{\parallel,1}} \mathbf{\eta} \times \partial_{\mathbf{x}_{\parallel,2}} \mathbf{\eta} )^{T},  ( \partial_{\mathbf{x}_{\parallel,2}} \mathbf{\eta}  \times [-\mathbf{n}] )^{T},   ([-\mathbf{n}]\times \partial_{\mathbf{x}_{\parallel,1}} \mathbf{\eta} )^{T}\bigg].
\end{split}
\end{equation}
From basic linear algebra
\begin{equation}\notag
\begin{split}
&\text{det}\left( \frac{\partial ( {X}_{\mathbf{cl}}(s^{1}),  {V}_{\mathbf{cl}}(s^{1}))}{\partial (\mathbf{X}_{\mathbf{cl}}(s^{1}), \mathbf{V}_{\mathbf{cl}}(s^{1}))}  \right)= \text{det} \left[\begin{array}{c|c} A  & \mathbf{0}_{3,3} \\ \hline B+ \mathbf{x}_{\perp} D & A  \end{array}   \right]
= \{\det(A )\}^{2}    = \big\{[-\mathbf{n}] \cdot (\partial_{1} \mathbf{\eta} \times \partial_{2} \mathbf{\eta})\big\}^{2} ,
\end{split}
\end{equation}
and $\left( \frac{\partial ( {X}_{\mathbf{cl}}(s^{1}),  {V}_{\mathbf{cl}}(s^{1}))}{\partial (\mathbf{X}_{\mathbf{cl}}(s^{1}), \mathbf{V}_{\mathbf{cl}}(s^{1}))}  \right)$ is invertible. By the basic linear algebra
\begin{equation}\label{XVchange}
 \begin{split}
 &\frac{\partial (\mathbf{X}_{\mathbf{cl}}(s^{1}), \mathbf{V}_{\mathbf{cl}}(s^{1})   )}{   \partial (  X_{\mathbf{cl}} (s^{1}), V_{\mathbf{cl}}(s^{1}) )} =
\left[ \frac{   \partial (  X_{\mathbf{cl}} (s^{1}), V_{\mathbf{cl}}(s^{1}) )}{\partial (\mathbf{X}_{\mathbf{cl}}(s^{1}), \mathbf{V}_{\mathbf{cl}}(s^{1})   )}\right]^{-1} = \left[ \begin{array}{c|c} A   & \mathbf{0}_{3,3} \\ \hline B+ \mathbf{x}_{\perp} D & A  \end{array}\right]^{-1} \\
 &= \left[\begin{array}{c|c} A ^{-1} & \mathbf{0}_{3,3} \\ \hline - A ^{-1} (B+\mathbf{x}_{\perp} D)  A ^{-1} &  A ^{-1}   \end{array}\right] = \left[\begin{array}{cc} A^{-1}(\mathbf{x}_{\parallel})  & \mathbf{0}_{3,3} \\ |v|  + O_{\xi}(\mathbf{x}_{\perp}) & A^{-1}(\mathbf{x}_{\parallel})      \end{array}\right],
\end{split}
\end{equation}
and we obtain
 \begin{equation}\notag
 \begin{split}
 \frac{\partial (s^{1}, \mathbf{X}_{\mathbf{cl}}(s^{1}) ,   \mathbf{V}_{\mathbf{cl}}(s^{1}) )}{\partial (s^{1}, X_{\mathbf{cl}}(s^{1}) ,   {V}_{\mathbf{cl}}(s^{1}) )}  =  \left[\begin{array}{c|c|c }1 & \mathbf{0}_{1,3} & \mathbf{0}_{1,3} \\ \hline
 0 &  \frac{(\partial_{1} \mathbf{\eta} \times \partial_{2} \mathbf{\eta})^{T}}{ [-\mathbf{n}] \cdot (\partial_{1} \mathbf{\eta} \times \partial_{2} \mathbf{\eta})}   &  \\
 0&     \frac{(\partial_{2} \mathbf{\eta} \times [-\mathbf{n}])^{T}}{ [-\mathbf{n}] \cdot (\partial_{1} \mathbf{\eta} \times \partial_{2} \mathbf{\eta})}    &  \mathbf{0}_{3,3} \\
 0 &  \frac{( [-\mathbf{n}] \times \partial_{1} \mathbf{\eta})^{T}}{ [-\mathbf{n}] \cdot (\partial_{1} \mathbf{\eta} \times \partial_{2} \mathbf{\eta})}   \\ \hline
 0& O_{\xi}(1)(  |v|) &  \frac{(\partial_{1} \mathbf{\eta} \times \partial_{2} \mathbf{\eta})^{T}}{ [-\mathbf{n}] \cdot (\partial_{1} \mathbf{\eta} \times \partial_{2} \mathbf{\eta})}     \\
 0  &  O_{\xi}(1)(|v|)&   \frac{(\partial_{2} \mathbf{\eta} \times [-\mathbf{n}])^{T}}{ [-\mathbf{n}] \cdot (\partial_{1} \mathbf{\eta} \times \partial_{2} \mathbf{\eta})}    \\
 0 & O_{\xi}(1)(|v|) &  \frac{( [-\mathbf{n}] \times \partial_{1} \mathbf{\eta})^{T}}{ [-\mathbf{n}] \cdot (\partial_{1} \mathbf{\eta} \times \partial_{2} \mathbf{\eta})}     \end{array}\right].
 \end{split}
 \end{equation}

From $X_{\mathbf{cl}}(s_{1};t,x,v) = x-(t-s_{1})v= x-\Delta \times v$ and $V_{\mathbf{cl}}(s_{1};t,x,v)=v,$
\begin{equation}\notag
\begin{split}
&\frac{\partial (s_{1}, X_{\mathbf{cl}}(s_{1}), V_{\mathbf{cl}}(s_{1}))}{\partial (t,x,v)}
=
\left[\begin{array}{ccc} 1 & \mathbf{0}_{1, 3} & \mathbf{0}_{1, 3} \\  \mathbf{0}_{3,1} & \mathbf{Id}_{3,3} & -(t-s^{1})\mathbf{Id}_{3,3} \\ \mathbf{0}_{3, 1} & \mathbf{0}_{3, 3} & \mathbf{Id}_{3,3}  \end{array}\right].
\end{split}
\end{equation} Finally we multiply above two matrices and use $|\mathbf{x}_{\perp}(s^{1})|\lesssim  |v||t^{1}-s^{1}|$ to conclude the second claim (\ref{s1t}).

\vspace{8pt}

\noindent \textit{Step 4. Estimate of  ${\partial (t^{\ell+1}, 0 , \mathbf{x}_{\parallel_{{\ell+1}}}^{\ell+1}, \mathbf{v}_{\perp_{{\ell+1}}}^{\ell+1}, \mathbf{v}_{\parallel_{{\ell+1}}}^{\ell+1})}/{\partial (t^{\ell}, 0 , \mathbf{x}_{\parallel_{{\ell}}}^{\ell}, \mathbf{v}_{\perp_{{\ell}}}^{\ell}, \mathbf{v}_{\parallel_{{\ell}}}^{\ell})}$}

\vspace{4pt}

Recall $
\mathbf{r}^{\ell}$ from (\ref{r}). We show that there exists $M =M_{\xi,t}\gg 1$, which is only depending on $\Omega$, such that for all $\ell \in \mathbb{N}$ and $0 \leq t^{\ell+1} \leq t^{\ell} \leq t$ and $v\in \mathbb{R}^{3},$
\begin{equation}\label{l_to_l+1}
\begin{split}
J^{\ell+1}_{\ell}&:= \frac{\partial (t^{\ell+1},0,\mathbf{x}_{\parallel_{\ell+1}
}^{\ell+1}, \mathbf{v}_{\perp_{\ell+1} }^{\ell+1},\mathbf{v}_{\parallel_{\ell+1} }^{\ell+1})}{\partial (t^{\ell},0,\mathbf{x}_{\parallel_{\ell}
}^{\ell},\mathbf{v}_{\perp_{\ell} }^{\ell},\mathbf{v}_{\parallel_{\ell} }^{\ell})}  \\
& \leq \left[\begin{array}{c|ccc|ccc}
 1& 0 & \frac{M}{|v|} \mathbf{r}^{\ell+1}&  \frac{M}{|v|} \mathbf{r}^{\ell+1}& \frac{M}{|v |^{2}} &  \frac{M}{|v|^{2}} \mathbf{r}^{\ell+1} &  \frac{M}{|v|^{2}}  \mathbf{r}^{\ell+1}\\ \hline
 0 & 0 & 0 & 0 & 0 & 0 & 0\\
 0 & 0 & 1 + M \mathbf{r}^{\ell+1} & M \mathbf{r}^{\ell+1} & \frac{M}{|v|}    & \frac{M}{|v|} \mathbf{r}^{\ell+1}   &  \frac{M}{|v|} \mathbf{r}^{\ell+1}  \\
  0 & 0 &  M \mathbf{r}^{\ell+1} & 1 +M \mathbf{r}^{\ell+1} & \frac{M}{|v|}    & \frac{M}{|v|} \mathbf{r}^{\ell+1}   &  \frac{M}{|v|} \mathbf{r}^{\ell+1}  \\ \hline
  0 & 0 & M|v|( \mathbf{r}^{\ell+1})^{2} & M|v|( \mathbf{r}^{\ell+1})^{2} & 1+M\mathbf{r}^{\ell+1} & M (\mathbf{r}^{\ell+1})^{2} & M (\mathbf{r}^{\ell+1})^{2}\\
  0 & 0 & M|v|\mathbf{r}^{\ell+1} & M|v| \mathbf{r}^{\ell+1} & M  & 1+ M\mathbf{r}^{\ell+1} & M \mathbf{r}^{\ell+1}\\
     0 & 0 & M|v|\mathbf{r}^{\ell+1} & M|v| \mathbf{r}^{\ell+1} & M &  M\mathbf{r}^{\ell+1} & 1+M \mathbf{r}^{\ell+1}
 \end{array}\right] \\
 &:=     \underbrace{  J(\mathbf{r}^{\ell+1})   }_{\text{  Definition of } J(\mathbf{r}^{\ell+1})}.
\end{split}
\end{equation}

We also denote the Jacobian matrix within a single $\mathbf{p}^{\ell}-$ spherical coordinate:
\[
\tilde{J}_{\ell}^{\ell+1} : =  \frac{\partial (t^{\ell+1},0,\mathbf{x}_{\parallel_{\ell }
}^{\ell+1}, \mathbf{v}_{\perp_{\ell } }^{\ell+1},\mathbf{v}_{\parallel_{\ell } }^{\ell+1})}{\partial (t^{\ell},0,\mathbf{x}_{\parallel_{\ell}
}^{\ell},\mathbf{v}_{\perp_{\ell} }^{\ell},\mathbf{v}_{\parallel_{\ell} }^{\ell})}.
\]

Note this bound (\ref{l_to_l+1}) holds for both \textit{Type I} and \textit{Type II} in (\ref{type_r}). We split the proof for each \textit{Type}:

\vspace{4pt}

\noindent\textit{Proof of (\ref{l_to_l+1}) when $\mathbf{r}^{\ell}< \sqrt{\delta}$ and $\mathbf{r}^{\ell+1}< \sqrt{\delta}$: }  Note that $\mathbf{p}^{\ell}-$spherical coordinate is well-defined of all $\tau\in [t^{\ell+1}, t^{\ell}]$. Due to the chart changing
\begin{equation*}
 { \frac{\partial (t^{\ell+1},0,\mathbf{x}_{\parallel_{\ell+1}
}^{\ell+1}, \mathbf{v}_{\perp_{\ell+1} }^{\ell+1},\mathbf{v}_{\parallel_{\ell+1} }^{\ell+1})}{\partial (t^{\ell},0,\mathbf{x}_{\parallel_{\ell}
}^{\ell},\mathbf{v}_{\perp_{\ell} }^{\ell},\mathbf{v}_{\parallel_{\ell} }^{\ell})} } =
  \left[\begin{array}{c|c} 1 & \mathbf{0}_{1,6}\\  \hline \mathbf{0}_{6,1} & \nabla \Phi_{\mathbf{p}^{\ell}}^{-1} \nabla \Phi_{\mathbf{p}^{\ell+1}} \end{array} \right]
 \underbrace{  \frac{\partial (t^{\ell+1},0,\mathbf{x}_{\parallel_{\ell }
}^{\ell+1}, \mathbf{v}_{\perp_{\ell } }^{\ell+1},\mathbf{v}_{\parallel_{\ell } }^{\ell+1})}{\partial (t^{\ell},0,\mathbf{x}_{\parallel_{\ell}
}^{\ell},\mathbf{v}_{\perp_{\ell} }^{\ell},\mathbf{v}_{\parallel_{\ell} }^{\ell})}}_{=\tilde{J}_{\ell}^{\ell+1}}.
\end{equation*}
Note that
\begin{equation}\notag
\begin{split}
|\mathbf{p}^{\ell}-\mathbf{p}^{\ell+1}| \leq& \  |z^{\ell} - z^{\ell+1}| + | u^{\ell}- u^{\ell+1}|\\
 \lesssim& \ |z^{\ell} - x^{\ell }| +  |x^{\ell} - x^{\ell+1}|  +  |z^{\ell+1} - x^{\ell+1}| \\
 &+ \Big|u^{\ell} - \frac{v^{\ell} - (v^{\ell} \cdot n(z^{\ell}) )n(z^{\ell})}{|{v^{\ell} - (v^{\ell} \cdot n(z^{\ell})) n(z^{\ell})}|}\Big|
+ \Big|u^{\ell+1} - \frac{v^{\ell+1} - (v^{\ell+1} \cdot n(z^{\ell+1}) )n(z^{\ell+1})}{|{v^{\ell+1} - (v^{\ell+1} \cdot n(z^{\ell+1})) n(z^{\ell+1})}|}\Big|\\
&+ \frac{|\mathbf{v}_{\perp}^{\ell} | + |\mathbf{v}_{\perp}^{\ell+1}| + |v||x^{\ell} - z^{\ell}| + |v||x^{\ell+1} - z^{\ell+1}| }{{|{v^{\ell} - (v^{\ell} \cdot n(z^{\ell}) )n(z^{\ell} )}|}}\\
\lesssim_{\xi} &  \ \mathbf{r}^{\ell}.
\end{split}
\end{equation}
where we have used $\mathbf{r}^{\ell} \leq C \sqrt{\delta}$ (therefore$|v^{\ell} - (v^{\ell} \cdot n(z^{\ell})) n(z^{\ell}) | \gtrsim (1-\sqrt{\delta})|v|$) and $(\ref{pl})$ and (\ref{40}).

In order to show (\ref{l_to_l+1}) it suffices to show that ${\tilde{J}_{\ell}^{\ell+1}}$ is bounded as (\ref{l_to_l+1}):
\begin{equation}\label{tilde_J}
\tilde{J}_{\ell}^{\ell+1}  \ \leq  \  J(\mathbf{r}^{\ell+1}).
\end{equation}

This is due to the following matrix multiplication
\begin{equation}\notag
\begin{split}
  &\left[\begin{array}{c|c} 1 & \mathbf{0}_{1,6}\\  \hline \mathbf{0}_{6,1} & \nabla \Phi_{\mathbf{p}^{\ell}}^{-1} \nabla \Phi_{\mathbf{p}^{\ell+1}} \end{array} \right]\tilde{J}_{\ell}^{\ell+1}\\
  &\leq
 \left[
\begin{array}{c|ccc|ccc}
1 & & \mathbf{0}_{1,3}  &  &&  \mathbf{0}_{1,3} & \\ \hline
&1 & 0 & 0 &  &  &  \\
  \mathbf{0}_{3,1} &0 & 1+ C\mathbf{r}^{\ell+1} &  C\mathbf{r}^{\ell+1} &  & \mathbf{0}_{3,3} & \\
&0 &  C\mathbf{r}^{\ell+1} &  1+C\mathbf{r}^{\ell+1} &  & &   \\ \hline
&0 & 0 & 0 & 1 & 0 & 0 \\
  \mathbf{0}_{3,1} &0 &C\mathbf{r}^{\ell+1} |v| & C\mathbf{r}^{\ell+1} |v|  & 0 &  1+C\mathbf{r}^{\ell+1} &  C\mathbf{r}^{\ell+1} \\
&0 &  C\mathbf{r}^{\ell+1}|v|  &  C\mathbf{r}^{\ell+1}|v|  & 0 &  C\mathbf{r}^{\ell+1} & 1+C\mathbf{r}^{\ell+1}
\end{array}
\right]J(\mathbf{r}^{\ell+1})\\
&\leq J( C \mathbf{r}^{\ell+1}),
\end{split}
\end{equation}
where we used (\ref{chart_changing}).

Now we prove the claim (\ref{tilde_J}). We fix the $\mathbf{p}^{\ell}-$spherical coordinate and drop the index $\ell$ for the chart.
%

If $\mathbf{v}_{\perp}^{\ell}=0$ then $t^{\ell+1}=t^{\ell}$. Otherwise if $\mathbf{v}_{\perp}^{\ell} \neq 0$ then $t^{\ell+1}$ is determined through
\begin{equation}
0=\mathbf{v}_{\perp}^{\ell}(t^{\ell+1}-t^{\ell})+\int^{t^{\ell}  }_{ t^{\ell+1}}\int^{t^{\ell}}_{s}F_{\perp }(
\mathbf{X}_{\ell}(\tau; t^{\ell}, x^{\ell},v^{\ell}), \mathbf{V}_{\ell}(\tau;t^{\ell}, x^{\ell},v^{\ell} )
)\mathrm{d}\tau
\mathrm{d}s.\label{define}
\end{equation}
Since the ODE for $[\mathbf{X}_{\ell}(\tau;t,x,v), \mathbf{V}_{\ell}(\tau;t,x,v)]$ is autonomous,
\begin{equation}\notag
0=\mathbf{v}_{\perp}^{\ell}(t^{\ell+1}-t^{\ell})+\int^{t^{\ell}- t^{\ell+1}}_{0}   \int^{0}_{t^{\ell+1}-t^{\ell} + s} F_{\perp }(
\mathbf{X}_{\ell}(\tau; 0, x^{\ell}, v^{\ell}), \mathbf{V}_{\ell}(\tau; 0, x^{\ell}, v^{\ell} )
)\mathrm{d}\tau
\mathrm{d}s. 
\end{equation}
We take $t^{\ell}-$derivative to have
\begin{equation}\notag
\begin{split}
0&=  \frac{\partial (t^{\ell+1} - t^{\ell})}{\partial t^{\ell}}  \Big\{ \mathbf{v}_{\perp}^{\ell}  -   \int_{0}^{ t^{\ell} - t^{\ell+1} } 
F_{\perp} (\mathbf{X}_{\ell} (t^{\ell+1} -t^{\ell} +s; 0, x^{\ell} ,v^{\ell} ),\mathbf{V}_{\ell} (t^{\ell+1} -t^{\ell} +s; 0, x^{\ell} ,v^{\ell} ) ) \mathrm{d}s \Big\}\\
&= \frac{\partial (t^{\ell+1} - t^{\ell})}{\partial t^{\ell}} \Big\{  \mathbf{v}_{\perp}^{\ell} - \int_{t^{\ell+1}}^{t^{\ell}}  F_{\perp} (\mathbf{X}_{\ell} (  s; t^{\ell}  , x^{\ell} ,v^{\ell}  ), \mathbf{V}_{\ell}( s; t^{\ell}, x^{\ell} ,v^{\ell} )  ) \mathrm{d}s  \Big\}  \\
& =  \frac{\partial (t^{\ell+1} - t^{\ell})}{\partial t^{\ell}}  (- \mathbf{v}_{\perp}^{\ell+1}),
\end{split}
\end{equation}
where we used the definition
\begin{equation}\label{v_ell}
\mathbf{v}_{\perp}^{\ell+1} = - \lim_{s\downarrow t^{\ell+1}} \mathbf{v}_{\perp}(s) = -\mathbf{v}_{\perp}^{\ell} + \int^{t^{\ell}}_{t^{\ell+1}}  F_{\perp}(\mathbf{X}_{\mathbf{cl}}(\tau;t,x,v),\mathbf{V}_{\mathbf{cl}}(\tau;t,x,v) ) \mathrm{d}\tau.
\end{equation}
Therefore we conclude $$\frac{\partial t^{\ell+1}}{\partial t^{\ell}}=1.$$ Then combining with
\begin{eqnarray*} 
\mathbf{x}_{\parallel}^{\ell+1} &=& \mathbf{x}_{\parallel}^{\ell} + \int^{t^{\ell+1} -t^{\ell}}_{0} \mathbf{v}_{\parallel} (s;0, x^{\ell}, v^{\ell}) \mathrm{d}s,\\
\mathbf{v} ^{\ell+1} &=& \mathbf{v} ^{\ell} + \int^{t^{\ell+1} -t^{\ell}}_{0} F  (s;0,x^{\ell}, v^{\ell}) \mathrm{d}s,
\end{eqnarray*}
we conclude
\[
\frac{\partial \mathbf{x}_{\parallel}^{\ell+1}}{\partial t^{\ell}}  = \frac{\partial \mathbf{v}_{\parallel}^{\ell+1}}{\partial t^{\ell}} = \frac{\partial \mathbf{v}_{\perp}^{\ell+1}}{\partial t^{\ell}}= 0.
\]

Now we use $|t^{\ell}-t^{\ell+1}| \lesssim_{\xi,t} \min \{\frac{|\mathbf{v}_{\perp}^{\ell+1}|}{|v|^{2}},1\}$, from (\ref{41}), and (\ref{Dxv_free}) and (\ref{D_F}) to have 
\begin{equation}\label{tx}
\begin{split}
\frac{\partial t^{\ell+1}}{\partial \mathbf{x}_{\parallel }^{\ell}}& =\frac{1}{\mathbf{v}_{\perp
}^{\ell+1}}\int_{t^{\ell}}^{t^{\ell+1}}\int_{t^{\ell}}^{s}\frac{\partial }{\partial
\mathbf{x}_{\parallel }^{\ell}}F_{\perp }(   \mathbf{X}_{\ell}(\tau),   \mathbf{V}_{\ell}(\tau))\mathrm{d}\tau \mathrm{d}s\\
&\lesssim  \frac{|t^{\ell}-t^{\ell+1}|^{2}}{|\mathbf{v}_{\perp}^{\ell+1}|} |v^{\ell}|^{2} [1+ |v^{\ell}||t^{\ell}-t^{\ell+1}|] \lesssim \frac{|v^{\ell}|^{2} |t^{\ell}-t^{\ell+1}|^{2}}{|\mathbf{v}_{\perp}^{\ell+1}|} \\
&\lesssim_{\xi,t}  |t^{\ell}-t^{\ell+1}| \lesssim_{\xi,t} \frac{1}{|v|} \frac{|\mathbf{v}_{\perp}^{\ell+1}|}{|v|}
, \\
\frac{\partial t^{\ell+1}}{\partial \mathbf{v}_{\perp }^{\ell}}& =\frac{1}{\mathbf{v}_{\perp }^{\ell+1}}%
\left\{ (t^{\ell+1}-t^{\ell})+\int_{t^{\ell}}^{t^{\ell+1}}\int_{t^{\ell}}^{s}\frac{\partial
}{\partial \mathbf{v}_{\perp }^{\ell}}F_{\perp }(   \mathbf{X}_{\ell}(\tau), \mathbf{V}_{\ell}(\tau) )\mathrm{d}\tau \mathrm{d}%
s\right\} \\
&\lesssim \{1+ |v^{\ell}|  |t^{\ell}-t^{\ell+1}| \}\frac{|t^{\ell}-t^{\ell+1}|}{|\mathbf{v}_{\perp}^{\ell+1}|} \lesssim_{\xi,t}  \frac{|t^{\ell} -t^{\ell+1}|}{|\mathbf{v}_{\perp}^{\ell+1}|} \lesssim_{\xi,t} \frac{1}{|v|^{2}}, \\
\frac{\partial t^{\ell+1}}{\partial \mathbf{v}_{\parallel }^{\ell}}& =\frac{1}{\mathbf{v}_{\perp
}^{\ell+1}}\int^{t^{\ell}}_{t^{\ell+1}}\int^{t^{\ell}}_{s}\frac{\partial }{\partial
\mathbf{v}_{\parallel }^{\ell}}F_{\perp }(  \mathbf{X}_{\ell}(\tau),   \mathbf{V}_{\ell}(\tau))\mathrm{d}\tau \mathrm{d}s\lesssim \frac{|v^{\ell}|  |t^{\ell}-t^{\ell+1}|^{2}}{|\mathbf{v}_{\perp}^{\ell+1}|} \lesssim_{\xi,t} \frac{1}{|v|^{2}} \frac{|\mathbf{v}_{\perp}^{\ell+1}|}{|v|}.
\end{split} 
\end{equation}

We use (\ref{tx}) and (\ref{Dxv_free}) and (\ref{D_F}) and (\ref{41}) to have 
\begin{equation}\notag
\begin{split}
\frac{\partial \mathbf{x}_{\parallel }^{\ell+1}}{\partial \mathbf{x}_{\parallel }^{\ell}}&
=\mathbf{Id}_{2,2}+\frac{\mathbf{v}_{\parallel }^{\ell+1}}{\mathbf{v}_{\perp }^{\ell+1}}%
\int_{t^{\ell}}^{t^{\ell+1}}\int_{t^{\ell}}^{s}\frac{\partial }{\partial \mathbf{x}_{\parallel
}^{\ell}}F_{\perp }(\mathbf{X}_{\mathbf{cl}}(\tau), \mathbf{V}_{\mathbf{cl}}(\tau) )\mathrm{d}\tau \mathrm{d}s+\int_{t^{\ell}}^{t^{\ell+1}}%
\int_{t^{\ell}}^{s}\frac{\partial }{\partial \mathbf{x}_{\parallel }^{\ell}}F_{\parallel
}( \mathbf{X}_{\mathbf{cl}}(\tau), \mathbf{V}_{\mathbf{cl}}(\tau)   )\mathrm{d}\tau \mathrm{d}s\\
&\lesssim \mathbf{Id}_{2,2} + \Big( 1+ \frac{|v^{\ell}|}{|\mathbf{v}_{\perp}^{\ell+1}|}\Big)  |t^{\ell+1}-t^{\ell}|^{2} |v^{\ell}|^{2}
\lesssim \mathbf{Id}_{2,2} + \frac{|\mathbf{v}_{\perp}^{\ell}|}{|v|},\\
\frac{\partial \mathbf{x}_{\parallel }^{\ell+1}}{\partial \mathbf{v}_{\perp }^{\ell}}& =\Big\{\frac{%
t^{\ell+1}-t^{\ell}}{\mathbf{v}_{\perp }^{\ell+1}}+\frac{1}{\mathbf{v}_{\perp }^{\ell+1}}%
\int_{t^{\ell}}^{t^{\ell+1}}\int_{t^{\ell}}^{s}\frac{\partial }{\partial \mathbf{v}_{\perp
}^{\ell}}F_{\perp }( \mathbf{X}_{\mathbf{cl}}(\tau), \mathbf{V}_{\mathbf{cl}}(\tau) )\mathrm{d}\tau \mathrm{d}s\Big\}\mathbf{v}_{\parallel
}^{\ell+1}\\
& \ \ \ \ \ \ +\int_{t^{\ell}}^{t^{\ell+1}}\int_{t^{\ell}}^{s}\frac{\partial }{\partial
\mathbf{v}_{\perp }^{\ell}}F_{\parallel }( \mathbf{X}_{\mathbf{cl}}(\tau), \mathbf{V}_{\mathbf{cl}}(\tau) )\mathrm{d}\tau \mathrm{d}s\\
&\lesssim |t^{\ell}-t^{\ell+1}| \frac{|v^{\ell}|}{|\mathbf{v}_{\perp}^{\ell+1}|} + |t^{\ell}-t^{\ell+1}|^{2} \frac{|v^{\ell}|^{2}}{|\mathbf{v}_{\perp}^{\ell+1}|}+ |t^{\ell}-t^{\ell+1}|^{2} |v^{\ell}| \lesssim  \frac{1}{|v|}  ,  \\
\frac{\partial \mathbf{x}_{\parallel }^{\ell+1}}{\partial \mathbf{v}_{\parallel }^{\ell}}&
=(t^{\ell+1}-t^{\ell})+\frac{\mathbf{v}_{\parallel }^{\ell+1}}{\mathbf{v}_{\perp }^{\ell+1}}%
\int_{t^{\ell}}^{t^{\ell+1}}\int_{t^{\ell}}^{s}\frac{\partial }{\partial \mathbf{v}_{\parallel
}^{\ell}}F_{\perp }( \mathbf{X}_{\mathbf{cl}}(\tau), \mathbf{V}_{\mathbf{cl}}(\tau) )\mathrm{d}\tau \mathrm{d}s \\
& \ \ \ \ \ \  +\int_{t^{\ell}}^{t^{\ell+1}}%
\int_{t^{\ell}}^{s}\frac{\partial }{\partial \mathbf{v}_{\parallel }^{\ell}}F_{\perp }(  \mathbf{X}_{\mathbf{cl}}(\tau), \mathbf{V}_{\mathbf{cl}}(\tau)
)\mathrm{d}\tau \mathrm{d}s\\
&\lesssim  |t^{\ell}-t^{\ell+1}| + \Big(1+ \frac{|v^{\ell}|}{|\mathbf{v}_{\perp}^{\ell+1}|}\Big) |t^{\ell}-t^{\ell+1}|^{2} |v^{\ell}| \big[ 1+ |v^{\ell}||t^{\ell}-t^{\ell+1}|\big] \\
&\lesssim  |t^{\ell}-t^{\ell+1}| +  \Big(1+ \frac{|v^{\ell}|}{|\mathbf{v}_{\perp}^{\ell+1}|}\Big)    |t^{\ell}-t^{\ell+1}|^{2} |v^{\ell}| \lesssim \frac{1}{|v|} \frac{|\mathbf{v}_{\perp}^{\ell+1}|}{|v|}.
\end{split}
\end{equation}

Now we move to $D \mathbf{v}_{\perp}^{\ell+1}$ estimates. First we claim the crucial estimate of  $t^{\ell}-t^{\ell+1}$:
\begin{equation}
(t^{\ell}-t^{\ell+1})F_{\perp }(\mathbf{x}^{\ell+1}, \mathbf{v}^{\ell})=2\mathbf{v}_{\perp }^{\ell+1}+O_{\xi
}(1)|t^{\ell}-t^{\ell+1}|^{2}|v^{\ell}|^{3}.
\label{time_gap}
\end{equation}
As (\ref{define}), we use the fact $\mathbf{x}_{\perp}^{\ell} = 0 = \mathbf{x}_{\perp}^{\ell+1}$ and the definition $\mathbf{v}_{\perp }^{\ell+1}=-\lim_{s\downarrow t^{\ell+1}} \mathbf{v}_{\perp }(s)$ and 
\begin{equation}\notag
\begin{split}
\dot{\mathbf{v}}_{\perp}(s) & \ = \ F_{\perp}(\mathbf{X}_{\ell}(s;t^{\ell},\mathbf{x}^{\ell}, \mathbf{v}^{\ell}),\mathbf{V}_{\ell}(s;t^{\ell},\mathbf{x}^{\ell}, \mathbf{v}^{\ell})) \\
& \  =  \ F_{\perp}(\mathbf{X}_{\ell}(s;t^{\ell+1},\mathbf{x}^{\ell+1}, \mathbf{v}^{\ell+1}),\mathbf{V}_{\ell}(s;t^{\ell+1},\mathbf{x}^{\ell+1}, \mathbf{v}^{\ell+1})),
\end{split}
\end{equation} 
to conclude the similar identity of (\ref{define})
\begin{equation}\label{define_l+1}
0=-\mathbf{v}_{\perp
}^{\ell+1}(t^{\ell}-t^{\ell+1})+\int_{t^{\ell+1}}^{t^{\ell}}\int_{t^{\ell+1}}^{s}   F_{\perp}(\mathbf{X}_{\ell}(\tau;t^{\ell+1},\mathbf{x}^{\ell+1}, \mathbf{v}^{\ell+1}),\mathbf{V}_{\ell}(\tau;t^{\ell+1},\mathbf{x}^{\ell+1}, \mathbf{v}^{\ell+1}))\mathrm{d}%
\tau \mathrm{d}s.
\end{equation}
For $t^{\ell+1}<\tau< t^{\ell},$ we have
\begin{equation}\notag 
\begin{split}
&F_{\perp} ( \mathbf{X}_{\mathbf{cl}}(\tau; t^{\ell}, \mathbf{x}^{\ell}, \mathbf{v}^{\ell}), \mathbf{V}_{\mathbf{cl}}(\tau; t^{\ell}, \mathbf{x}^{\ell}, \mathbf{v}^{\ell})) \\
& = F_{\perp} (\mathbf{x}^{\ell}, \mathbf{v}^{\ell}) + \int^{\tau}_{t^{\ell}} \frac{\partial}{\partial \tau} F_{\perp} (  \mathbf{X}_{\mathbf{cl}}(\tau;t^{\ell },\mathbf{x}^{\ell }, \mathbf{v}^{\ell }),\mathbf{V}_{\mathbf{cl}}(\tau;t^{\ell },\mathbf{x}^{\ell }, \mathbf{v}^{\ell})   ) \mathrm{d}\tau, 
\end{split}
\end{equation}
and
\begin{equation}\notag
\begin{split}
& F_{\perp} ( \mathbf{X}_{\mathbf{cl}}(\tau; t^{\ell+1}, \mathbf{x}^{\ell+1}, \mathbf{v}^{\ell+1}), \mathbf{V}_{\mathbf{cl}}(\tau; t^{\ell+1}, \mathbf{x}^{\ell+1}, \mathbf{v}^{\ell+1})) \\
& = F_{\perp} (\mathbf{x}^{\ell+1}, \mathbf{v}^{\ell}) + \int^{\tau}_{t^{\ell+1}} \frac{\partial}{\partial \tau}   F_{\perp} (  \mathbf{X}_{\mathbf{cl}}(\tau;t^{\ell },\mathbf{x}^{\ell }, \mathbf{v}^{\ell }),\mathbf{V}_{\mathbf{cl}}(\tau;t^{\ell },\mathbf{x}^{\ell }, \mathbf{v}^{\ell})   )  \mathrm{d}\tau.
\end{split} 
\end{equation}
Therefore 
\begin{equation}\notag
\begin{split}
F_{\perp }(\mathbf{X}_{\mathbf{cl}}(\tau;t^{\ell}, \mathbf{x}^{\ell} , \mathbf{v}^{\ell}) ,\mathbf{V}_{\mathbf{cl}}(\tau;t^{\ell}, \mathbf{x}^{\ell} , \mathbf{v}^{\ell})   )&=F_{\perp }(\mathbf{x}^{\ell+1}, \mathbf{v}^{\ell} )+O_{\xi
}(1)|t^{\ell+1}-t^{\ell}||v |^{3}.
\end{split}
\end{equation}
Plugging this into (\ref{define_l+1}) we have
\begin{equation*}
0=-\mathbf{v}_{\perp }^{\ell+1}(t^{\ell}-t^{\ell+1})+\frac{1}{2}(t^{\ell}-t^{\ell+1})^{2}F_{\perp
}( \mathbf{x}^{\ell+1} , \mathbf{v}^{\ell} )+O_{\xi }(1)|t^{\ell}-t^{\ell+1}|^{3}|v |^{3},
\end{equation*} 
and this proves our claim (\ref{time_gap}).

Using (\ref{time_gap}), we can find an extra cancellation in
terms of order of $t^{\ell}-t^{\ell+1}$ to get
\begin{equation}\notag
\begin{split}
\frac{\partial \mathbf{v}_{\perp }^{\ell+1}}{\partial \mathbf{x}_{\parallel }^{\ell}}& =\frac{%
-F_{\perp }(\mathbf{x}^{\ell+1}, \mathbf{v}^{\ell})}{\mathbf{v}_{\perp }^{\ell+1}}\int^{t^{\ell}}_{t^{\ell+1}}\int^{t^{\ell}}_{s}%
\frac{\partial }{\partial \mathbf{x}_{\parallel }^{\ell}}F_{\perp }( \mathbf{X}_{\mathbf{cl}}(\tau), \mathbf{V}_{\mathbf{cl}}(\tau) )\mathrm{d}\tau
\mathrm{d}s+\int^{t^{\ell}}_{t^{\ell+1}}\frac{\partial }{\partial \mathbf{x}_{\parallel
}^{\ell}}F_{\perp }(\mathbf{X}_{\mathbf{cl}}(\tau), \mathbf{V}_{\mathbf{cl}}(\tau) )\mathrm{d}\tau  \\
& =\Big\{\frac{(t^{\ell}-t^{\ell+1})F_{\perp }(\mathbf{x}^{k+1}, \mathbf{v}^{\ell})}{-2\mathbf{v}_{\perp }^{\ell+1}}+1%
\Big\}(t^{\ell}-t^{\ell+1})\frac{\partial }{\partial \mathbf{x}_{\parallel }^{\ell}}F_{\perp
}(\mathbf{x}^{\ell}, \mathbf{v}^{\ell})  \\
&   \ \ \ \ \ +O_{\xi }(1)\Big\{|t^{\ell}-t^{\ell+1}|^{2}|v^{\ell}|^{3}+\frac{%
|t^{\ell}-t^{\ell+1}|^{3}|v^{\ell}|^{3}}{|\mathbf{v}_{\perp }^{\ell+1}|}|\frac{\partial }{%
\partial \mathbf{x}_{\parallel }^{\ell}}F_{\perp }(\mathbf{x}^{\ell}, \mathbf{v}^{\ell})|\Big\}\\
&\lesssim_{\xi} \Big\{   -1 + O_{\xi}(1) \frac{|t^{\ell}-t^{\ell+1}|^{2 }   |v^{\ell}|^{3}}{|\mathbf{v}_{\perp}^{\ell+1}|} + 1\Big\} |t^{\ell}-t^{\ell+1}|  |v^{\ell}|^{2} +|t^{\ell}-t^{\ell+1}|^{2}|v^{\ell}|^{3} \Big\{ 1+ \frac{|t^{\ell}-t^{\ell+1}||v^{\ell}|^{2}}{|\mathbf{v}_{\perp}^{\ell+1}|}  \Big\}
\\
&\lesssim _{\xi}  |t^{\ell}-t^{\ell+1}|^{2} |v^{\ell}|^{3} \Big(1+ \frac{|t^{\ell}-t^{\ell+1}||v^{\ell}|^{2}}{|\mathbf{v}_{\perp}^{\ell+1}|} \Big) \lesssim_{\xi}   |t^{\ell}-t^{\ell+1}|^{2} |v^{\ell}|^{3} , \\
&\lesssim_{\xi,t}  \frac{|\mathbf{v}_{\perp}^{\ell+1}|^{2}}{|v^{\ell}|}  ,
\end{split}
\end{equation}
\begin{equation}\notag
\begin{split}
\frac{\partial \mathbf{v}_{\perp }^{\ell+1}}{\partial \mathbf{v}_{\perp }^{\ell}}& =-1-\frac{%
\partial t^{\ell+1}}{\partial \mathbf{v}_{\perp }^{\ell}}F_{\perp
}(\mathbf{x}^{\ell+1}, \mathbf{v}^{\ell})+\int^{t^{\ell}}_{t^{\ell+1}}\frac{\partial }{\partial \mathbf{v}_{\perp }^{\ell}}%
F_{\perp }( \mathbf{X}_{\mathbf{cl}}(\tau), \mathbf{V}_{\mathbf{cl}}(\tau) )\mathrm{d}\tau \\
&= -1+\frac{F_{\perp }(\mathbf{x}^{\ell+1}, \mathbf{v}^{\ell})}{\mathbf{v}_{\perp }^{\ell+1}}(t^{\ell}-t^{\ell+1})-\frac{%
F_{\perp } (\mathbf{x}^{\ell+1}, \mathbf{v}^{\ell})}{\mathbf{v}_{\perp }^{\ell+1}}\int_{t^{\ell}}^{t^{\ell+1}}\int_{t^{\ell}}^{s}%
\frac{\partial }{\partial \mathbf{v}_{\perp }^{\ell}}F_{\perp }(\mathbf{X}_{\mathbf{cl}}(\tau), \mathbf{V}_{\mathbf{cl}}(\tau))\mathrm{d}\tau
\mathrm{d}s \\
& \ \ \ \ \ \ \ \ +\int^{t^{\ell}}_{t^{\ell+1}}\frac{\partial }{\partial \mathbf{v}_{\perp }^{\ell}}%
F_{\perp }( \mathbf{X}_{\mathbf{cl}}(\tau) , \mathbf{V}_{\mathbf{cl}}(\tau) )\mathrm{d}\tau \\
&=- 1+2+O_{\xi }(1)\frac{|t^{\ell}-t^{\ell+1}|^{2}|v^{\ell}|^{3}}{\mathbf{v}_{\perp }^{\ell+1}} \\
& \ \ \ \ \ -
\frac{F_{\perp } (\mathbf{x}^{\ell+1}, \mathbf{v}^{\ell})}{\mathbf{v}_{\perp }^{\ell+1}}\frac{(t^{\ell}-t^{\ell+1})^{2}}{2}%
\Big\{ \lim_{s \uparrow t^{\ell}}  \frac{\partial }{\partial \mathbf{v}_{\perp }^{\ell}}   F_{\perp }( \mathbf{X}_{\mathbf{cl}}(\tau), \mathbf{V}_{\mathbf{cl}}(\tau)  )+O_{\xi
}(1)|t^{\ell}-t^{\ell+1}||v^{\ell}|^{2}\Big\} \\
 & \ \ \ \ \ +(t^{\ell}-t^{\ell+1})\Big\{  \lim_{s \uparrow t^{\ell}}\frac{\partial }{\partial \mathbf{v}_{\perp }^{\ell}}F_{\perp
}( \mathbf{X}_{\mathbf{cl}}(\tau), \mathbf{V}_{\mathbf{cl}}(\tau)   )+O_{\xi }(1)|t^{\ell}-t^{\ell+1}||v^{\ell}|^{2}\Big\} \\
&  =1 + O_{\xi}(1) \Big\{\frac{|t^{\ell}-t^{\ell+1}|^{2}|v^{\ell}|^{3}}{|\mathbf{v}_{\perp}^{\ell+1}|}+ \frac{|t^{\ell}-t^{\ell+1}|^{3}}{|\mathbf{v}_{\perp}^{\ell+1}|}|v^{\ell}|^{3}  \Big| \lim_{s \uparrow t^{\ell}}\frac{\partial }{\partial \mathbf{v}_{\perp }^{\ell}}F_{\perp
}( \mathbf{X}_{\mathbf{cl}}(\tau), \mathbf{V}_{\mathbf{cl}}(\tau)   )\Big| + |t^{\ell}-t^{\ell+1}|^{2} |v^{\ell}|^{2}\Big\}\\
&\lesssim 1 + |t^{\ell}-t^{\ell+1}|^{2} |v^{\ell}|^{2} \Big\{ 1+ \frac{|v^{\ell}|}{|\mathbf{v}_{\perp}^{\ell+1}|} + \frac{|t^{\ell}-t^{\ell+1}||v^{\ell}|^{2}  }{|\mathbf{v}_{\perp}^{\ell+1}|}  \Big\} \lesssim_{\xi,t}  1+ \frac{|\mathbf{v}_{\perp}^{\ell+1}|}{|v^{\ell}|} ,\\
\frac{\partial \mathbf{v}_{\perp }^{\ell+1}}{\partial \mathbf{v}_{\parallel }^{\ell}}& =\frac{%
-F_{\perp }(\mathbf{x}^{\ell+1},  v^{\ell})}{ \mathbf{v}_{\perp }^{\ell+1}}\int^{t^{\ell}}_{t^{\ell+1}}\int^{t^{\ell}}_{s}%
\frac{\partial }{\partial  \mathbf{v}_{\parallel }^{\ell}}F_{\perp }( \mathbf{X}_{\mathbf{cl}}(\tau), \mathbf{V}_{\mathbf{cl}}(\tau) )\mathrm{d}\tau
\mathrm{d}s-\int^{t^{\ell}}_{t^{\ell+1}}\frac{\partial }{\partial  \mathbf{v}_{\parallel
}^{\ell}}F_{\perp }(  \mathbf{X}_{\mathbf{cl}}(\tau), \mathbf{V}_{\mathbf{cl}}(\tau) )\mathrm{d}\tau \\
& =\Big\{\frac{(t^{\ell}-t^{\ell+1})F_{\perp }(\mathbf{x}^{\ell+1}, \mathbf{v}^{\ell})}{-2 \mathbf{v}_{\perp }^{\ell+1}}+1%
\Big\}(t^{\ell}-t^{\ell+1})\frac{\partial }{\partial  \mathbf{v}_{\parallel }^{\ell}}F_{\perp
}(\mathbf{x}^{\ell}, \mathbf{v}^{\ell})  \\
& \ \ \ \  +O_{\xi }(1) |t^{\ell}-t^{\ell+1}|^{2} |v^{\ell}|^{2}
\Big\{\frac{|F_{\perp }(\mathbf{x}^{\ell+1}, \mathbf{v}^{\ell})|  |t^{\ell}-t^{\ell+1}| }{| \mathbf{v}_{\perp }^{\ell+1}|}%
+1\Big\}\\
& \lesssim_{\xi} |t^{\ell}-t^{\ell+1}|^{2}|v^{\ell}|^{2} \Big\{1+ \frac{|t^{\ell}-t^{\ell+1}||v^{\ell}|^{2}}{| \mathbf{v}_{\perp}^{\ell+1}|}\Big\} \lesssim_{ \xi,t}   \frac{| \mathbf{v}_{\perp}^{\ell+1}|^{2}}{|v^{\ell}|^{2}}.
\end{split}
\end{equation}
By $(\ref{Dxv_free})$ and (\ref{41}),
\begin{equation}\notag
\begin{split} 
\frac{\partial  \mathbf{v}_{\parallel }^{\ell+1}}{\partial \mathbf{x}_{\parallel }^{\ell}}& =\frac{%
F_{\parallel }(\mathbf{x}^{\ell+1}, \mathbf{v}^{\ell})}{\mathbf{v}_{\perp }^{\ell+1}}\int_{t^{\ell}}^{t^{\ell+1}}%
\int_{t^{\ell}}^{s}\frac{\partial }{\partial \mathbf{x}_{\parallel }^{\ell}}F_{\perp }( \mathbf{X}_{\mathbf{cl}}(\tau), \mathbf{V}_{\mathbf{cl}}(\tau)
)\mathrm{d}\tau \mathrm{d}s\\
& \ \ +\int_{t^{\ell}}^{t^{\ell+1}}\frac{\partial }{\partial
\mathbf{x}_{\parallel }^{\ell}}F_{\parallel }(  \mathbf{X}_{\mathbf{cl}}(\tau), \mathbf{V}_{\mathbf{cl}}(\tau) )\mathrm{d}\tau \lesssim  |t^{\ell}-t^{\ell+1}||v^{\ell}|^{2} \Big\{1+ \frac{|t^{\ell}-t^{\ell+1}||v^{\ell}|^{2}}{|\mathbf{v}_{\perp}^{\ell+1}|}\Big\} \\
&\lesssim_{\xi }  |t^{\ell}-t^{\ell+1}| |v^{\ell}|^{2} \lesssim_{\xi,t}   | \mathbf{v}_{\perp}^{\ell+1}|   ,\\
\frac{\partial  \mathbf{v}_{\parallel }^{\ell+1}}{\partial  \mathbf{v}_{\perp }^{\ell}}& =\frac{%
-(t^{\ell}-t^{\ell+1})F_{\parallel }(\mathbf{x}^{\ell+1}, \mathbf{v}^{\ell})}{ \mathbf{v}_{\perp }^{\ell+1}}+\frac{%
F_{\parallel }(\mathbf{x}^{\ell+1}, \mathbf{v}^{\ell})}{ \mathbf{v}_{\perp }^{\ell+1}}\int_{t^{\ell}}^{t^{\ell+1}}%
\int_{t^{\ell}}^{s}\frac{\partial }{\partial  \mathbf{v}_{\perp }^{\ell}}F_{\perp }(  \mathbf{X}_{\mathbf{cl}}(\tau), \mathbf{V}_{\mathbf{cl}}(\tau) )%
\mathrm{d}\tau \mathrm{d}s\\
& \ \ +\int_{t^{\ell}}^{t^{\ell+1}}\frac{\partial }{\partial
 \mathbf{v}_{\perp }^{\ell}}F_{\parallel }( \mathbf{X}_{\mathbf{cl}}(\tau), \mathbf{V}_{\mathbf{cl}}(\tau) )\mathrm{d}\tau\\
& \lesssim_{\xi,t}   \Big(1+ \frac{|v^{\ell}|}{| \mathbf{v}_{\perp}^{\ell+1}|}\Big) \min\{|v^{\ell}|, \frac{|\mathbf{v}_{\perp}^{\ell+1}|}{|v^{\ell}|}\} \lesssim_{\xi,t}   1+\frac{| \mathbf{v}_{\perp}^{\ell+1}|}{|v^{\ell}|} 
,\\
\frac{\partial  \mathbf{v}_{\parallel }^{\ell+1}}{\partial  \mathbf{v}_{\parallel }^{\ell}}&
=\mathbf{Id}_{2,2}+\frac{F_{\parallel }(\mathbf{x}^{\ell+1}, \mathbf{v}^{\ell})}{ \mathbf{v}_{\perp }^{\ell+1}}%
\int_{t^{\ell}}^{t^{\ell+1}}\int_{t^{\ell}}^{s}\frac{\partial }{\partial  \mathbf{v}_{\parallel
}^{\ell}}F_{\perp }(\mathbf{X}_{\mathbf{cl}}(\tau), \mathbf{V}_{\mathbf{cl}}(\tau) )\mathrm{d}\tau \mathrm{d}s +\int_{t^{\ell}}^{t^{\ell+1}}%
\frac{\partial }{\partial  \mathbf{v}_{\parallel }^{\ell}}F_{\parallel }(\mathbf{X}_{\mathbf{cl}}(\tau), \mathbf{V}_{\mathbf{cl}}(\tau) )\mathrm{d}%
\tau \\
& \lesssim_{\xi} \mathbf{Id}_{2,2} + |t^{\ell}-t^{\ell+1}||v^{\ell}|\big\{1+ \frac{|v^{\ell}|^{2} |t^{\ell}-t^{\ell+1}|}{|\mathbf{v}_{\perp}^{\ell+1}|}\big\}  \lesssim_{\xi,t} \mathbf{Id}_{2,2} +  \frac{| \mathbf{v}_{\perp}^{\ell+1}|}{|v^{\ell}|} .
\end{split}
\notag
\end{equation}
These estimates prove the claim (\ref{tilde_J}).

 \vspace{4pt}

\noindent\textit{Proof of (\ref{l_to_l+1}) for either $\mathbf{r}^{\ell}\geq \sqrt{\delta}$ or $\mathbf{r}^{\ell+1}\geq \sqrt{\delta}$: } 
Without loss of generality we assume $\mathbf{r}^{\ell} > C \sqrt{\delta}$ in (\ref{type_r}). 
Recall that we chose a $\mathbf{p}^{\ell}-$spherical coordinate as $\mathbf{p}^{\ell}= (z^{\ell}, w^{\ell})$ with $|z^{\ell}-x^{\ell}| \leq \sqrt{\delta}$ and any $w^{\ell} \in \mathbb{S}^{2}$ with $n(z^{\ell})\cdot w^{\ell} =0.$

Fix $\ell$. Let us choose fixed numbers $\Delta_{1}, \Delta_{2}>0$ such that $|v|\Delta_{1}\ll 1$ and $|v||t^{\ell+1}- (t^{\ell}-\Delta_{1}-\Delta_{2})|\ll 1$ so that
$$
s^{\ell}\equiv t^{\ell}-\Delta_{1}, \ \  s^{\ell+1}\equiv s^{\ell}-\Delta_{2} = t^{\ell} - \Delta_{1} -\Delta_{2},
$$
 satisfying $|v||t^{\ell+1}-s^{\ell+1} |= |v||t^{\ell+1}-(t^{\ell}-\Delta_{1}-\Delta_{2})|\ll 1$ and $|v||t^{\ell}-s^{\ell} | =|v||\Delta_{1}| \ll 1$ so that the spherical coordinates are well-defined for $s\in [t^{\ell+1},s ^{\ell+1}]$ and $ s\in[s ^{\ell}, t^{\ell}]$. 
 
 Notice that
 \[
 \frac{\partial t^{\ell+1}}{\partial s^{\ell+1} }
 = \frac{\partial( s^{\ell+1} + \Delta_{1} + \Delta_{2} -t_{\mathbf{b}}(x^{\ell},v^{\ell}))}{\partial s^{\ell+1}}
 =1,
 \ \ \frac{\partial s^{\ell+1}}{\partial s^{\ell}}  = \frac{\partial ( s^{\ell} -\Delta_{1})}{\partial s ^{\ell}}=1,
 \ \ \frac{\partial s ^{\ell}}{\partial t^{\ell}} = \frac{\partial (t^{\ell}-\Delta_{1})}{\partial t^{\ell}}=1.
 \]

By the chain rule,
\begin{equation*}
\begin{split}
& \frac{\partial ( t^{\ell+1}, 0, \mathbf{x}_{\parallel_{\ell+1}}^{\ell+1}, \mathbf{v}_{\perp_{\ell+1}}^{\ell+1}, \mathbf{v}_{\parallel_{\ell+1}}^{\ell+1}  )}{\partial (t^{\ell }, 0, \mathbf{x}_{\parallel_{\ell}}^{\ell }, \mathbf{v}_{\perp_{\ell}}^{\ell }, \mathbf{v}_{\parallel_{\ell}}^{\ell } )} \\
=&
 \frac{\partial ( t^{\ell+1}, 0, \mathbf{x}_{\parallel_{\ell+1}}^{\ell+1}, \mathbf{v}_{\perp_{\ell+1}}^{\ell+1}, \mathbf{v}_{\parallel_{\ell+1}}^{\ell+1}  )}
 {\partial( s^{\ell+1} , \mathbf{x}_{\perp_{\ell+1}}(s ^{\ell+1}), \mathbf{x}_{\parallel_{\ell+1}}(s ^{\ell+1}), \mathbf{v}_{\perp_{\ell+1}}(s ^{\ell+1}), \mathbf{v}_{\parallel_{\ell+1}}(s ^{\ell+1}) )}
\frac{\partial (s ^{\ell+1}, \mathbf{X}_{ {\mathbf{p}^{\ell+1}}}(s ^{\ell+1}), \mathbf{V}_{{\mathbf{p}^{\ell+1}}}(s ^{\ell+1}))}
{\partial (s ^{\ell+1}, X_{\mathbf{cl}}(s ^{\ell+1}),V_{\mathbf{cl}}(s ^{\ell+1}))}
\\
 &\times\frac{\partial (s^{\ell+1},  {X}_{\mathbf{cl}}(s ^{{\ell+1}}),  {V}_{\mathbf{cl}}(s ^{{\ell+1}}))}
{ \partial (s^{{\ell }},  {X}_{\mathbf{cl}}(s  ^{{\ell }}),  {V}_{\mathbf{cl}}(s ^{{\ell }}))}
 \frac{\partial (   s ^{{\ell }}, {X}_{\mathbf{cl}}(s ^{{\ell }}), {V}_{\mathbf{cl}}(s ^{{\ell }}))}
 {\partial (   s ^{{\ell }}, \mathbf{X}_{\mathbf{p}^{\ell}}(s^{{\ell }}), \mathbf{V}_{\mathbf{p}^{\ell}}(s ^{{\ell }}))}
 \frac{\partial(s ^{{\ell }}, \mathbf{x}_{\perp_{\ell}}(s ^{{\ell }}), \mathbf{x}_{\parallel_{\ell}}(s ^{{\ell }}), \mathbf{v}_{\perp_{\ell}}(s ^{{\ell }}), \mathbf{v}_{\parallel_{\ell}}(s ^{{\ell }}) )}
{\partial (t^{\ell }, 0, \mathbf{x}_{\parallel_{\ell}}^{\ell }, \mathbf{v}_{\perp_{\ell}}^{\ell }, \mathbf{v}_{\parallel_{\ell}}^{\ell } )}.
\end{split}
\end{equation*}

We can express that $t^{\ell+1} = t^{\ell} - t_{\mathbf{b}}(x^{\ell}, v^{\ell}) = s^{\ell+1} + \Delta_{1} + \Delta_{2}- t_{\mathbf{b}}(x^{\ell}, v^{\ell}).$ Let us regard $t^{\ell+1}$ as $t^{1}$ and $s^{\ell+1}$ as $s^{1}$ and $\Delta_{1} + \Delta_{2}$ as $\Delta$ in (\ref{Delta_step3}). Then we use (\ref{t1s1}) and (\ref{41}) to have
$$
  \frac{\partial ( t^{\ell+1}, 0, \mathbf{x}_{\parallel}^{\ell+1}, \mathbf{v}_{\perp}^{\ell+1}, \mathbf{v}_{\parallel}^{\ell+1}  )}{\partial( s^{\ell+1}, \mathbf{x}_{\perp}(s^{\ell+1} ), \mathbf{x}_{\parallel}(s^{\ell+1} ), \mathbf{v}_{\perp}(s^{\ell+1} ), \mathbf{v}_{\parallel}(s^{\ell+1} ) )}
\leq   \left[\begin{array}{c|c|c}
1 &  O_{\delta,\xi}(1)\frac{1}{|v|}  & O_{\delta,\xi}(1) \frac{1}{|v|^{2}} \\ \hline
0 & \mathbf{0}_{1,3} & \mathbf{0}_{1,3} \\
\mathbf{0}_{2,1} & O_{\delta,\xi}(1) & O_{\delta,\xi}(1) \frac{1}{|v|} \\ \hline
\mathbf{0}_{3,1} & O_{\delta,\xi}(1)|v| & O_{\delta,\xi}(1)
\end{array}\right].
  $$ 
From (\ref{XVchange})
\begin{equation}\notag
\begin{split}
 \frac{\partial (s^{\ell+1}, \mathbf{X}_{\mathbf{p}^{\ell+1}}(s ^{\ell+1}) ,   \mathbf{V}_{\mathbf{p}^{\ell+1}}(s ^{\ell+1}) )}{\partial (s ^{\ell+1}, X_{\mathbf{cl}}(s ^{\ell+1}) ,   {V}_{\mathbf{cl}}(s ^{\ell+1}) )}
 \lesssim_{\xi} \left[\begin{array}{c|c|c}
1 & \mathbf{0}_{1,3}  & \mathbf{0}_{1,3}\\ \hline
\mathbf{0}_{3,1} & O_{\xi}(1) & \mathbf{0}_{3,3}\\ \hline
\mathbf{0}_{3,1} & O_{\xi}(1)|v| & O_{\xi}(1)
\end{array} \right],
\end{split}
\end{equation}
and from $s^{\ell+1} = s^{\ell} -\Delta_{2}, \ X_{\mathbf{cl}}(s^{\ell+1} ) = X_{\mathbf{cl}}(s^{\ell} ) - (s^{\ell+1} -s^{\ell}) V_{\mathbf{cl}}(s^{\ell} ), \ V_{\mathbf{cl}}(s^{\ell +1} )=V_{\mathbf{cl}}(s^{\ell }),$
\begin{equation}\notag
\begin{split}
&\frac{\partial (s^{\ell+1}, X_{\mathbf{cl}}(s^{\ell+1}), V_{\mathbf{cl}}(s^{\ell+1} ))}{\partial (s^{\ell} ,X_{\mathbf{cl}}(s^{\ell} ), V_{\mathbf{cl}}(s^{\ell} ) )}
\lesssim_{\xi}  \left[\begin{array}{ccc}
1 & \mathbf{0}_{1,3} & \mathbf{0}_{1,3} \\
\mathbf{0}_{3,1}& \mathbf{Id}_{3,3} & |s_{1}-s_{2}|\mathbf{Id}_{3,3} \\
\mathbf{0}_{3,1} & \mathbf{0}_{3,3} & \mathbf{Id}_{3,3}
\end{array}\right],
\end{split}
\end{equation}
and from (\ref{jac_Phi})
\begin{equation}\notag
\begin{split}
& \frac{\partial (s^{\ell} , X_{\mathbf{cl}}(s^{\ell} ), V_{\mathbf{cl}}(s^{\ell} )  )}{\partial (s^{\ell} , \mathbf{X}_{\mathbf{p}^{\ell}}(s^{\ell} ), \mathbf{V}_{\mathbf{p}^{\ell}}(s^{\ell} )  )}
\lesssim_{\xi} \left[ \begin{array}{ccc}
1 &   \mathbf{0}_{1,3} &   \mathbf{0}_{1,3}   \\
\mathbf{0}_{3,1}&   O_{\xi}(1)
 &   \mathbf{0}_{3,3}   \\
\mathbf{0}_{3,1}   & |v|    & O_{\xi}(1)      \end{array}\right].
\end{split}
  \end{equation}
Recall (\ref{s1tstar}) to have
$$
\frac{\partial (s^{\ell} , \mathbf{x}_{\perp_{\ell}}(s^{\ell}),  \mathbf{x}_{\parallel_{\ell}}(s^{\ell}),   \mathbf{v}_{\perp_{\ell}}(s^{\ell}), \mathbf{v}_{\parallel_{\ell}}(s^{\ell})
    )}{\partial (t^{\ell },0, \mathbf{x}_{\parallel_{\ell}}^{\ell }, \mathbf{v}_{\perp_{\ell}}^{\ell }, \mathbf{v}_{\parallel_{\ell}}^{\ell })}
    \lesssim_{\xi}
     \left[\begin{array}{c|cc|c}
1 & 0 & \mathbf{0}_{1,2} & \mathbf{0}_{1,3}\\ \hline
O_{\xi}(1)|v|  & 0 & O_{\xi}(1) & O_{\xi}(1)|t^{\ell}-s_{1}| \\ \hline
O_{\xi}(1)|v|^{2} & 0 & O_{\xi}(1)|v| & O_{\xi}(1)
\end{array}\right].$$

By direct matrix multiplication
\begin{equation}\notag
\begin{split}
 \frac{\partial ( t^{\ell+1}, 0, \mathbf{x}_{\parallel_{{\ell+1}}}^{\ell+1}, \mathbf{v}_{\perp_{{\ell+1}}}^{\ell+1}, \mathbf{v}_{\parallel_{{\ell+1}}}^{\ell+1}  )}{\partial (t^{\ell }, 0, \mathbf{x}_{\parallel_{{\ell }}}^{\ell }, \mathbf{v}_{\perp_{{\ell }}}^{\ell }, \mathbf{v}_{\parallel_{{\ell }}}^{\ell } )}
   \lesssim_{ t,\xi} \left[\begin{array}{c|cc|c}1 & 0 & \frac{1}{|v|} & \frac{1}{|v|^{2}} \\ \hline
0 & 0 & \mathbf{0}_{1,2} & \mathbf{0}_{1,3}\\
\mathbf{0}_{2,1}   & \mathbf{0}_{2,1} & 1 & \frac{1}{|v|}\\ \hline
 \mathbf{0}_{3,1} & \mathbf{0}_{3,1} & |v| & 1
\end{array}\right].
\end{split}
\end{equation}
Note that for \textit{Type II} we have $\mathbf{r}^{\ell+1} \gtrsim \sqrt{\delta}$ so that from (\ref{l_to_l+1})
\[
J(\mathbf{r}^{\ell+1}) \gtrsim \left[\begin{array}{c|cc|c}
1 & 0 & \frac{M}{|v|}\sqrt{\delta} &  \frac{M}{|v|^{2}} \min\{1, \sqrt{\delta}\} \\ \hline
0 & 0 & \mathbf{0}_{1,2} & \mathbf{0}_{1,3}\\
\mathbf{0}_{2,1} & \mathbf{0}_{2,1} & M \sqrt{\delta} & \frac{M}{|v|} \min \{  1, \sqrt{\delta} \} \\ \hline
\mathbf{0}_{3,1} & \mathbf{0}_{3,1} & M|v| \min \{ \delta, \sqrt{\delta}\} & M \min\{\delta,\sqrt{\delta}\}
\end{array} \right] \gtrsim_{\delta,t,\xi}  \frac{\partial ( t^{\ell+1}, 0, \mathbf{x}_{\parallel_{{\ell+1}}}^{\ell+1}, \mathbf{v}_{\perp_{{\ell+1}}}^{\ell+1}, \mathbf{v}_{\parallel_{{\ell+1}}}^{\ell+1}  )}{\partial (t^{\ell }, 0, \mathbf{x}_{\parallel_{{\ell }}}^{\ell }, \mathbf{v}_{\perp_{{\ell }}}^{\ell }, \mathbf{v}_{\parallel_{{\ell }}}^{\ell } )}
.\]
This proves our claim (\ref{l_to_l+1}) for \textit{Type II}.

 \vspace{8pt}

\noindent{\textit{Step 5. Eigenvalues and diagonalization of of (\ref{l_to_l+1})}}

\vspace{4pt}

By a basic linear algebra (row and column operations), the characteristic polynomial of (\ref{l_to_l+1}) 
 equals, with $\mathbf{r}= \mathbf{r}^{\ell+1},$
\begin{equation}\notag
\begin{split}
&\text{det}\left[\begin{array}{ccccccc} 1- \lambda & 0 & \frac{M}{|v|} \mathbf{r} & \frac{M}{|v|}\mathbf{r} & \frac{M}{|v|^{2}} & \frac{M}{|v|^{2}} \mathbf{r} & \frac{M}{|v|^{2}} \mathbf{r} \\
0 & -\lambda & 0 & 0 & 0 & 0 & 0 \\
0 & 0 & 1+M\mathbf{r} -\lambda  & M \mathbf{r} & \frac{M}{|v|} & \frac{M}{|v|} \mathbf{r} &  \frac{M}{|v|} \mathbf{r}\\
0 & 0 &  M \mathbf{r}& 1+M\mathbf{r} -\lambda  & \frac{M}{|v|} & \frac{M}{|v|} \mathbf{r} &  \frac{M}{|v|} \mathbf{r}\\
0 & 0 & M|v| \mathbf{r}^{2} & M |v| \mathbf{r}^{2} & 1+ M\mathbf{r} -\lambda & M \mathbf{r}^{2} & M \mathbf{r}^{2} \\
0 & 0 & M|v| \mathbf{r}  & M |v| \mathbf{r}  & M   & 1+ M\mathbf{r} -\lambda  & M \mathbf{r} \\
0 & 0 & M|v| \mathbf{r}  & M |v| \mathbf{r}  & M     & M \mathbf{r} & 1+ M\mathbf{r} -\lambda\\
\end{array}\right]\\
&= - \lambda( \lambda-1)^{5} [\lambda- (1+5M\mathbf{r})].
\end{split}
\end{equation}
Therefore eigenvalues are
\begin{equation}\label{eigenvalue}
\begin{split}
\lambda_{0} &=0, \ \ \lambda_{1}= \lambda_{2} = \lambda_{3} = \lambda_{4} = \lambda_{5} = 1,\\
\lambda_{7}&= 1+ 5M \mathbf{r}^{\ell+1} = 1+ 5M \frac{|\mathbf{v}_{\perp}^{\ell+1}|}{|v^{\ell+1}|}.
\end{split}
\end{equation}
Corresponding eigenvectors are
\begin{eqnarray*}
\left(\begin{array}{ccccccc}0 \\ 1 \\ 0 \\ 0 \\ 0 \\ 0 \\ 0  \end{array}\right), \left(\begin{array}{ccccccc}1 \\ 0 \\ 0 \\ 0 \\ 0 \\ 0 \\0   \end{array}\right),
 \left(\begin{array}{ccccccc}0 \\ 0 \\ 1 \\ -1 \\ 0 \\ 0 \\0   \end{array}\right)
 ,
 \left(\begin{array}{ccccccc}0 \\ 0 \\ 1 \\ 0 \\ -|v|\mathbf{r} \\ 0 \\0   \end{array}\right),
 \left(\begin{array}{ccccccc}0 \\ 0 \\ 1 \\ 0 \\ 0 \\ -|v| \\0   \end{array}\right),
 \left(\begin{array}{ccccccc}0 \\ 0 \\ 1 \\ 0 \\ 0 \\ 0 \\-|v|   \end{array}\right),
  \left(\begin{array}{ccccccc}1 \\ 0 \\ |v| \\ |v| \\ |v|^{2} \mathbf{r} \\ |v|^{2} \\|v|^{2}   \end{array}\right).
\end{eqnarray*}
Write $P=P(\mathbf{r}^{\ell}) $ as a block matrix of above column eigenvectors. Then
\begin{equation}  \label{diagonal_matrix}
\mathcal{P}= \left[\begin{array}{ccccccc}
0 & 1 & 0 & 0 & 0 & 0 & 1\\
1 & 0 & 0 & 0 & 0 & 0 & 0\\
0 & 0 & 1 & 1 & 1 & 1 & |v|\\
0 & 0 &-1 & 0 & 0 & 0 & |v|\\
0 & 0 & 0 &-|v|\mathbf{r} & 0 & 0 & |v|^{2} \mathbf{r}\\
0& 0& 0 & 0 & -|v| & 0 & |v|^{2}\\
0 & 0 & 0 & 0 & 0 & -|v| & |v|^{2}
 \end{array}\right], \  
\mathcal{P}^{-1}  =  \left[\begin{array}{ccccccc}
0 & 1 & 0 & 0 & 0 & 0 & 0 \\
1 & 0 & \frac{-1}{5|v|} & \frac{-1}{5|v|} & \frac{-1}{5|v|^{2} \mathbf{r}} & \frac{-1}{5|v|^{2}} & \frac{-1}{5|v|^{2}}\\
0 & 0 & \frac{1}{5} & \frac{-4}{5} & \frac{1}{5|v| \mathbf{r}} & \frac{1}{5|v|} & \frac{1}{5|v|} \\
0 & 0 & \frac{1}{5} & \frac{1}{5} & \frac{-4}{5|v| \mathbf{r}} & \frac{1}{5|v|} & \frac{1}{5|v|} \\
0 & 0 & \frac{1}{5} & \frac{1}{5} & \frac{1}{5|v| \mathbf{r}} & \frac{-4}{5|v|} & \frac{1}{5|v|} \\
0 & 0 & \frac{1}{5} & \frac{1}{5} & \frac{1}{5|v| \mathbf{r}} & \frac{1}{5|v|} & \frac{-4}{5|v|} \\
0 & 0 & \frac{1}{5|v|} & \frac{1}{5|v|} & \frac{1}{5|v|^{2} \mathbf{r}} & \frac{1}{5|v|^{2}} & \frac{1}{5|v|^{2}} \\
\end{array}\right].
\end{equation}
Therefore
\begin{equation}\notag
  {J}(\mathbf{r}) = \mathcal{P}(\mathbf{r}) \Lambda(\mathbf{r}) \mathcal{P}^{-1} (\mathbf{r}) ,
\end{equation}
and 
\[
 \Lambda(\mathbf{r}):= \text{diag}\Big[ 0, 1, 1, 1, 1, 1, 1+5M\mathbf{r}\Big],
 \]
where the notation $\text{diag}[a_{1},\cdots , a_{m}]$ is a $m\times m-$matrix with $a_{ii}=a_{i}$ and $a_{ij}=0$ for all $i\neq j.$

\vspace{8pt}

 \noindent{\textit{Step 6. The $i-$th intermediate group}}

\vspace{4pt}

 We claim that, for $i=1,2,\cdots, [\frac{|t-s| |v|}{L_{\xi}}]$,
\begin{equation}\label{one_group}
\begin{split}
&J^{\ell_{i+1}}_{\ell_{i+1}-1} \times \cdots \times J_{\ell_{i}}^{\ell_{i}+1}\\
 =& \ {\frac{\partial (t^{\ell_{i+1}},0, \mathbf{x}_{\parallel_{{\ell_{i+1}}}}^{\ell_{i+1}}, \mathbf{v}_{\perp_{{\ell_{i+1}}}}^{\ell_{i+1}}, \mathbf{v}_{\parallel_{{\ell_{i+1}}}}^{\ell_{i+1}})}{\partial (t^{\ell_{i+1}-1},0, \mathbf{x}_{\parallel_{{\ell_{i+1}}-1}}^{\ell_{i+1}-1}, \mathbf{v}_{\perp_{{\ell_{i+1}}-1}}^{\ell_{i+1}-1}, \mathbf{v}_{\parallel_{{\ell_{i+1}}-1}}^{\ell_{i+1}-1})}
 \times \cdots \times
 \frac{\partial (t^{\ell_{i}+1},0, \mathbf{x}_{\parallel_{{\ell_{i}+1}}}^{\ell_{i} +1}, \mathbf{v}_{\perp_{{\ell_{i}+1}}}^{\ell_{i}+1}, \mathbf{v}_{\parallel_{{\ell_{i}+1}}}^{\ell_{i}+1})}{\partial (t^{\ell_{i} },0, \mathbf{x}_{\parallel_{{\ell_{i}  }}}^{\ell_{i}  }, \mathbf{v}_{\perp_{{\ell_{i}  }}}^{\ell_{i} }, \mathbf{v}_{\parallel_{{\ell_{i}  }}}^{\ell_{i} })}  }\\
 \leq & \ 
 \mathcal{P}(\mathbf{r}_{i}) (\Lambda(\mathbf{r}_{i}))^{\frac{C_{\xi}}{\mathbf{r}_{i}}}  \mathcal{P}^{-1}( \mathbf{r}_{i}).
\end{split}
\end{equation}

By the definition of the group, $L_{\xi}\leq |v||t^{\ell_{i}}-t^{\ell_{i+1}}| \leq C_{1}<+\infty$ for all $i$. By the Velocity lemma(Lemma \ref{velocity_lemma}),
  \[
\frac{1}{\mathcal{C}_{1}} e^{-\frac{\mathcal{C}}{2}C_{1}} \mathbf{r}^{\ell_{i}}  \leq  \mathbf{r}^{\ell_{i+1}} \equiv \frac{|\mathbf{v}_{\perp}^{\ell_{i+1}}|}{|v|},   \mathbf{r}^{\ell_{i+1}-1} \equiv \frac{|\mathbf{v}_{\perp}^{\ell_{i+1}-1}|}{|v|}, \cdots ,   \mathbf{r}^{\ell_{i }+1} \equiv \frac{|\mathbf{v}_{\perp}^{\ell_{i}+1}|}{|v|},   \mathbf{r}^{\ell_{i}} \equiv \frac{|\mathbf{v}_{\perp}^{\ell_{i}}|}{|v|}
  \leq \mathcal{C}_{1} e^{\frac{\mathcal{C}}{2}C_{1}} \mathbf{r}^{\ell_{i}},
  \]
and define
\[
\mathbf{r}_{i}\equiv \mathcal{C}_{1} e^{\frac{\mathcal{C}}{2}C_{1}} \mathbf{r}^{\ell_{i}}.
\]
Then we have
\begin{equation}\label{bound_ri}
\frac{1}{(\mathcal{C}_{1})^{2}} e^{-\mathcal{C} C_{1}} \mathbf{r}_{i}  \ \leq  \ \mathbf{r}^{j} \ \leq \ \mathbf{r}_{i} \ \ \ \text{for all} \ \  \ell_{i+1}\leq j \leq \ell_{i}.
\end{equation}
From (\ref{l_to_l+1}), we have a uniform bound for all $\ell_{i+1}\leq j \leq \ell_{i}$
\[
J^{j+1}_{j}  \lesssim  {J}(\mathbf{r}_{i})= \mathcal{P}(\mathbf{r}_{i}) \Lambda(\mathbf{r}_{i}) \mathcal{P}^{-1}(\mathbf{r}_{i}).
\]
Therefore
\[
J^{\ell_{i+1}}_{\ell_{i+1}-1} \times \cdots \times J^{\ell_{i}+1}_{\ell_{i}} \leq \mathcal{P}(\mathbf{r}_{i}) [ \Lambda(\mathbf{r}_{i})]^{|\ell_{i+1}-\ell_{i}|} \mathcal{P}^{-1}(\mathbf{r}_{i}).
\]

Now we only left to prove $|\ell_{i+1}-\ell_{i}| \lesssim_{\Omega} \frac{1}{\mathbf{r}_{i}}$: For any $\ell_{i+1} \leq j \leq \ell_{i}$, we have $\xi(x^{j})=0=\xi(x^{j+1})=\xi(x^{j}-(t^{j}-t^{j+1})v^{j}).$ We expand $\xi(x^{j}-(t^{j}-t^{j+1})v^{j})$ in time to have
\begin{equation}\notag
\begin{split}
\xi(x^{j+1}) &=  \xi(x^{j}) + \int^{t^{j+1}}_{t^{j}} \frac{d}{ds} \xi(X_{\mathbf{cl}}(s)) \mathrm{d}s\\
&= \xi(x^{j}) + (v^{j}\cdot \nabla \xi(x^{j}) ) (t^{j+1}-t^{j}) + \int^{t^{j+1}}_{t^{j}} \int^{s}_{t^{j}} \frac{d^{2}}{d\tau^{2}} \xi(X_{\mathbf{cl}}(\tau))  \mathrm{d}\tau\mathrm{d}s,\\
\end{split}
\end{equation}
and  
\begin{equation}\notag
0 = (v^{j}\cdot \nabla \xi(x^{j})) (t^{j+1}-t^{j}) + \frac{(t^{j}-t^{j+1})^{2}}{2} (v^{j} \cdot \nabla^{2}\xi(X_{\mathbf{cl}}(\tau_{*})) \cdot v^{j}), \ \ \ \text{for some } \tau_{*} \in [t^{j+1},t^{j}].
\end{equation}
Therefore
\begin{equation}\notag
\begin{split}
\frac{v^{j}\cdot \nabla \xi(x^{j})}{|v|}& = (t^{j}-t^{j+1})|v| \frac{v^{j} \cdot \nabla^{2}\xi(X_{\mathbf{cl}}(\tau_{*})) \cdot v^{j}}{2|v|^{2}}.
\end{split}
\end{equation}
From the convexity (\ref{convex}), there exists $C_{2}\gg 1$
\begin{equation}\label{time_angle}
 \frac{1}{C_{2}}  |t^{j}-t^{j+1}||v|\leq|\mathbf{r}^{j}|=\frac{|\mathbf{v}_{\perp}^{j}|}{|v|}=\frac{|v^{j}\cdot \nabla \xi(x^{j})|}{|v|} \leq C_{2}|t^{j}-t^{j+1}||v|.
\end{equation}
Therefore we have a lower bound of $|v||t^{j}-t^{j+1}|$: $
|v||t^{j}-t^{j+1}| \geq \frac{1}{C_{2}} |\mathbf{r}^{j}| \geq  \frac{1}{(\mathcal{C}_{1})^{2}C_{2}} e^{-\mathcal{C}C_{1}} \mathbf{r}_{i},$
where we have used (\ref{bound_ri}).
Finally, using the definition of one group($1 \leq  |v||t^{\ell_{i}}-t^{\ell_{i+1}}| \leq C_{1}$), we have the following upper bound of the number of bounces in this one group($i-$th intermediate group)
\[
|\ell_{i}-\ell_{i+1}|  \leq \frac{|v||t^{\ell_{i}}-t^{\ell_{i+1}}|}{ \min_{\ell_{i} \leq j \leq \ell_{i+1}} |v||t^{j}-t^{j+1}|}  \leq \frac{C_{1}}{ \frac{1}{(\mathcal{C}_{1})^{2} C_{2}  } e^{-\mathcal{C} C_{1}} \mathbf{r}_{i} } \lesssim_{\xi} \frac{1}{\mathbf{r}_{i}},
\]
and this complete our claim (\ref{one_group}).

\vspace{8pt}

\noindent\textit{Step 7. Whole intermediate groups}

\vspace{4pt}
Recall $\mathcal{P}$ and $\mathcal{P}^{-1}$ from (\ref{diagonal_matrix}). We claim that, there exists $C_{3}>0$ such that
\begin{equation}\label{whole_group}
\begin{split}
&\prod_{i=1}^{[\frac{|t-s||v|}{L_{\xi}}]}J^{\ell_{i+1}}_{\ell_{i+1}-1} \times \cdots \times J_{\ell_{i}}^{\ell_{i}+1}
 \leq
 (C_{3})^{|t-s||v|}
  \mathcal{P}(\mathbf{r}_{[\frac{|t-s||v|}{L_{\xi}}]})    \mathcal{P}^{-1}( \mathbf{r}_{1}).
\end{split}
\end{equation}

From the one group estimate (\ref{one_group}),
\begin{equation}\notag
\begin{split}
&\prod_{i=1}^{[\frac{|t-s||v|}{L_{\xi}}]}J^{\ell_{i+1}}_{\ell_{i+1}-1} \times \cdots \times J_{\ell_{i}}^{\ell_{i}+1} \\
\lesssim  & \ \mathcal{P}(\mathbf{r}_{ [ \frac{|t-s||v|}{L_{\xi}} ]  })( \Lambda(\mathbf{r}_{ [ \frac{|t-s||v|}{L_{\xi}} ]}))^{\frac{C_{\xi}}{\mathbf{r}_{ [ \frac{|t-s||v|}{L_{\xi}} ]}}}\underbrace{ \mathcal{P}^{-1}(\mathbf{r}_{ [ \frac{|t-s||v|}{L_{\xi}} ]}) \times  \mathcal{P}(\mathbf{r}_{ [ \frac{|t-s||v|}{L_{\xi}} ]-1}) }\\
&\times   ( \Lambda(\mathbf{r}_{ [ \frac{|t-s||v|}{L_{\xi}} ]-1}))^{\frac{C_{\xi}}{\mathbf{r}_{ [ \frac{|t-s||v|}{L_{\xi}} ]-1}}} \underbrace{\mathcal{P}^{-1}(\mathbf{r}_{ [ \frac{|t-s||v|}{L_{\xi}} ]-1}) \times \ \  \ \ }\cdots \\
&\times \cdots\underbrace{\ \ \ \ \times  \mathcal{P}(\mathbf{r}_{i+1})}( \Lambda(\mathbf{r}_{i+1}))^{\frac{C_{\xi}}{\mathbf{r}_{i+1}}}  \underbrace{\mathcal{P}^{-1}(\mathbf{r}_{i+1})  \times\mathcal{P}(\mathbf{r}_{i})}( \Lambda(\mathbf{r}_{i}))^{\frac{C_{\xi}}{\mathbf{r}_{i}}} \underbrace{\mathcal{P}^{-1}(\mathbf{r}_{i}) \times
\mathcal{P}(\mathbf{r}_{i-1})}\\
&\times ( \Lambda(\mathbf{r}_{i-1}))^{\frac{C_{\xi}}{\mathbf{r}_{i-1}}} \underbrace{ \mathcal{P}^{-1}(\mathbf{r}_{i-1}) \times \ \ \ \ } \cdots\\
&\times \cdots \underbrace{ \ \ \ \ \times \mathcal{P}(\mathbf{r}_{2})}( \Lambda(\mathbf{r}_{2}))^{\frac{C_{\xi}}{\mathbf{r}_{2}}} \underbrace{\mathcal{P}^{-1}(\mathbf{r}_{2}) \times \mathcal{P}(\mathbf{r}_{1})}( \Lambda(\mathbf{r}_{1}))^{\frac{C_{\xi}}{\mathbf{r}_{1}}} \mathcal{P}^{-1}(\mathbf{r}_{1}).
\end{split}
\end{equation}

Now we focus on the underbraced matrix multiplication. Directly
\begin{equation}\notag
\begin{split}
 \mathcal{P}^{-1}(\mathbf{r}_{i+1})  \mathcal{P}(\mathbf{r}_{i})
=  \left[\begin{array}{ccccccc}
1 & 0 & 0 & 0 & 0 & 0 & 0 \\
0 & 1 & 0 &  \frac{-1+ \frac{\mathbf{r}_{i}}{\mathbf{r}_{i+1}}}{5|v|}  & 0 & 0 & \frac{1-\frac{\mathbf{r}_{i}}{\mathbf{r}_{i+1}}  }{5}\\
0 & 0 & 1 & \frac{1-\frac{\mathbf{r}_{i}}{\mathbf{r}_{i+1}}}{5} & 0 & 0 & |v|\frac{-1 + \frac{\mathbf{r}_{i}}{\mathbf{r}_{i+1}}}{5}\\
0 & 0 & 0 & \frac{1+ 4\frac{\mathbf{r}_{i}}{\mathbf{r}_{i+1}}}{5} & 0 & 0 & 4|v|\frac{1- \frac{\mathbf{r}_{i}}{\mathbf{r}_{i+1}}}{5}\\
0 & 0 & 0 & \frac{1- \frac{\mathbf{r}_{i}}{\mathbf{r}_{i+1}}}{5} & 1 & 0 & |v| \frac{-1+ \frac{\mathbf{r}_{i}}{\mathbf{r}_{i+1}}}{5} \\
0 &  0 & 0 & \frac{1-\frac{\mathbf{r}_{i}}{\mathbf{r}_{i+1}}}{5} & 0 & 1 & |v| \frac{-1 + \frac{\mathbf{r}_{i}}{\mathbf{r}_{i+1}}}{5} \\
0 &  0 & 0 & \frac{1-\frac{\mathbf{r}_{i}}{\mathbf{r}_{i+1}}}{5|v|} & 0 & 0 & \frac{4+ \frac{\mathbf{r}_{i}}{\mathbf{r}_{i+1}}}{5}
\end{array}\right]. 
\end{split}
\end{equation}
Due to the choice of $\mathbf{r}_{i}\equiv \mathcal{C}_{1} e^{\frac{\mathcal{C}}{2} C_{1}} \mathbf{r}^{\ell_{i}}$ in (\ref{bound_ri}) we have $\left|\frac{\mathbf{r}_{i}}{ \mathbf{r}_{i+1}}\right| = \left|\frac{\mathbf{r}^{\ell_{i}}}{ \mathbf{r}^{\ell_{i+1}}}\right| \leq C_{\xi},$
where we have used the Velocity lemma and (\ref{40}) and (\ref{41}): $\frac{1}{\mathcal{C}_{1} }e^{-\frac{\mathcal{C}}{2}C_{1}} \mathbf{r}^{\ell_{i+1}}\leq\frac{1}{\mathcal{C}_{1} }e^{-\frac{\mathcal{C}}{2}|t^{\ell_{i}}-t^{\ell_{i+1}}|} \mathbf{r}^{\ell_{i+1}}\leq \mathbf{r}^{\ell_{i}} \leq \mathcal{C}_{1} e^{\frac{\mathcal{C}}{2}|t^{\ell_{i}}-t^{\ell_{i+1}}|} \mathbf{r}^{\ell_{i+1}} \leq \mathcal{C}_{1} e^{\frac{\mathcal{C}}{2}C_{1}} \mathbf{r}^{\ell_{i+1}}.$

Therefore for sufficiently large $C_{\xi}>0$, for all $i$
\begin{equation}\label{Q}
\widetilde{ \mathcal{P}^{-1}(\mathbf{r}_{i+1}) \mathcal{P}(\mathbf{r}_{i})}  \ \leq  \ \mathcal{Q} := \left[\begin{array}{ccccccc}
1 & 0 & 0 & 0 & 0 & 0 & 0 \\
0 & 1  & 0 & \frac{C_{\xi}}{|v|} & 0 & 0 & C_{\xi} \\
0 & 0 & 1 & C_{\xi} & 0 & 0 & C_{\xi}|v| \\
0 & 0 & 0 & C_{\xi} & 0 & 0 & C_{\xi}|v|\\
0 & 0 & 0 & C_{\xi} & 1 & 0 & C_{\xi}|v| \\
0 & 0 & 0 & C_{\xi} & 0 & 1 & C_{\xi}|v|\\
0 & 0 & 0 & \frac{C_{\xi}}{|v|} & 0 & 0 & C_{\xi}
\end{array}\right],
\end{equation}
where we use a notation: For a matrix $A$,  the entries of a matrix $\widetilde{A}$ is an absolute value of the entries of $A$, i.e.
$(\widetilde{A} )_{ij}= |(A)_{ij}|.$

Again we diagonalize $\mathcal{Q}$ as
\begin{equation}\notag
\begin{split}
&\mathcal{Q} = \mathcal{F} \mathcal{A} \mathcal{F}^{-1}\\
&\tiny{ :=\left[\begin{array}{ccccccc} 1 & 0 & 0 & 0 & 0 & 0 & 0 \\
0 & 1 & 0 & 0 & 0 & 0 & \frac{2C_{\xi}}{2C_{\xi}-1} \\
0 & 0 & 1 & 0 & 0 & 0 & \frac{2C_{\xi}|v|}{2C_{\xi}-1}\\
0 & 0 & 0 & 0 & 0 & -|v| & |v|\\
0 & 0 & 0 & 1 & 0 & 0 & \frac{2C_{\xi}|v|}{2C_{\xi}-1}\\
0 & 0 & 0 & 0 & 1 & 0 & \frac{2C_{\xi}|v|}{2C_{\xi}-1}\\
0 & 0 & 0 & 0 & 0 & 1 & 1
 \end{array}\right] 
 \left[\begin{array}{ccccccc}1 &&&&&& \\ &1&&&&0& \\ &&1&&&& \\ &&&1&&& \\ &&&&1&& \\ &0&&&&0& \\&&&&&&2C_{\xi} \end{array}\right]
 \left[\begin{array}{ccccccc}
 1&0&0&0&0&0&0 \\
 0 &1&0& \frac{-C_{\xi}}{2C_{\xi}-1} \frac{1}{|v|}& 0 & 0 & \frac{-C_{\xi}}{2C_{\xi}-1} \\
 0&0&1 & \frac{-C_{\xi}}{2C_{\xi}-1}&0&0& \frac{-C_{\xi}|v|}{2C_{\xi}-1} \\
 0&0&0&\frac{-C_{\xi}}{2C_{\xi}-1}& 1&0 & \frac{-C_{\xi}|v|}{2 C_{\xi}-1} \\
0&0&0&\frac{-C_{\xi}}{2C_{\xi}-1}& 0&1 & \frac{-C_{\xi}|v|}{2 C_{\xi}-1} \\
 0&0&0& \frac{-1}{2|v|}& 0& 0& \frac{1}{2} \\
0 & 0 & 0& \frac{1}{2|v|}& 0& 0& \frac{1}{2} \\
 \end{array} \right],}
 \end{split}
 \end{equation}
 and directly
 \begin{equation}\label{Q_diag}
 \begin{split}
& \mathcal{Q}^{ [ \frac{|t-s||v|}{L_{\xi}} ]} = \mathcal{F} \mathcal{A}^{ [ \frac{|t-s||v|}{L_{\xi}} ]} \mathcal{F}^{-1}\\
& = 
  \mathcal{F} \
  \text{diag}\Big[  1, 1 , 1, 1, 1 , 0 , (2C_{\xi})^{ [ \frac{|t-s||v|}{L_{\xi}} ]}\Big] \
 \mathcal{F}^{-1}
  \\
 &= \left[\begin{array}{ccccccc}
   1 & 0    &0      &0      &0      &0      & 0       \\
   0 &   1  &  0&   \frac{1}{|v|} \frac{C_{\xi}}{2C_{\xi}-1} ((2C_{\xi})^{ [ \frac{|t-s||v|}{L_{\xi}} ]}  -1)   &     0 & 0     &    \frac{C_{\xi}}{2C_{\xi}-1} ( (2C_{\xi})^{ [ \frac{|t-s||v|}{L_{\xi}} ]}  -1)        \\
   0 &  0   &  1    &   \frac{C_{\xi}}{2C_{\xi}-1} ((2C_{\xi})^{[ \frac{|t-s||v|}{L_{\xi}} ]}  -1)    &  0    &0      &  |v|\frac{C_{\xi}}{2C_{\xi}-1} ((2C_{\xi})^{[ \frac{|t-s||v|}{L_{\xi}} ]}  -1)      \\
   0 &  0   &  0    &  \frac{(2C_{\xi})^{[ \frac{|t-s||v|}{L_{\xi}} ]}}{2}    & 0     & 0     & |v|    \frac{(2C_{\xi})^{ [ \frac{|t-s||v|}{L_{\xi}} ]}}{2}     \\
   0  &  0   & 0     & \frac{C_{\xi}}{2C_{\xi}-1} ((2C_{\xi})^{ [ \frac{|t-s||v|}{L_{\xi}} ]}-1)     &1      & 0     & |v|   \frac{C_{\xi}}{2C_{\xi}-1} ((2C_{\xi})^{ [ \frac{|t-s||v|}{L_{\xi}} ]}-1)     \\
   0  &  0   & 0     & \frac{C_{\xi}}{2C_{\xi}-1} ((2C_{\xi})^{ [ \frac{|t-s||v|}{L_{\xi}} ]}-1)     &0      & 1     & |v|   \frac{C_{\xi}}{2C_{\xi}-1} ((2C_{\xi})^{ [ \frac{|t-s||v|}{L_{\xi}} ]}-1)     \\
    0   &  0   &  0    &  \frac{1}{|v|} \frac{(2C_{\xi})^{ [ \frac{|t-s||v|}{L_{\xi}} ]}}{2}    & 0      &0      & \frac{(2C_{\xi})^{ [ \frac{|t-s||v|}{L_{\xi}} ]}}{2}
 \end{array} \right]
 .
\end{split}
\end{equation}

Notice that from (\ref{eigenvalue}) 
\begin{equation}\notag
\begin{split}
  \Big[ \Lambda(\mathbf{r}_{i})\Big]^{\frac{C_{\xi}}{\mathbf{r}_{i}}} \leq (1+ 5M\mathbf{r}_{i})^{\frac{C_{\xi}}{\mathbf{r}_{i}}} \ \mathbf{Id}_{7,7}
  \leq C_{\xi}^{\prime}  \ \mathbf{Id}_{7,7}.
\end{split}
\end{equation}
%
Now we use (\ref{one_group}) and take the absolute value of the entries and then use (\ref{Q}) and (\ref{Q_diag}), for $\tilde{t}:=t-s,$
\begin{equation}\notag
\begin{split}
&\prod_{i=1}^{ [ \frac{\tilde{t}|v|}{L_{\xi}} ]}J^{\ell_{i+1}}_{\ell_{i+1}-1} \times \cdots \times J_{\ell_{i}}^{\ell_{i}+1}\\
\leq &  \ \widetilde{\mathcal{P}(\mathbf{r}_{ [\frac{\tilde{t}|v|}{L_{\xi}}]})}( 1+5M\mathbf{r}_{ [\frac{\tilde{t}|v|}{L_{\xi}}]})^{\frac{C_{\xi}}{\mathbf{r}_{ [\frac{\tilde{t}|v|}{L_{\xi}}]}}}   \mathcal{Q}   ( 1+5M\mathbf{r}_{ [\frac{\tilde{t}|v|}{L_{\xi}}]-1})^{\frac{C_{\xi}}{\mathbf{r}_{ [\frac{\tilde{t}|v|}{L_{\xi}}]-1}}}  \mathcal{Q}  \times\cdots \\
& \ \times\cdots\times  \mathcal{Q} ( 1+ 5M \mathbf{r}_{i+1})^{\frac{C_{\xi}}{\mathbf{r}_{i+1}}}     \mathcal{Q}  ( 1+ 5M\mathbf{r}_{i})^{\frac{C_{\xi}}{\mathbf{r}_{i}}}  
 \mathcal{Q}  ( 1+ 5M \mathbf{r}_{i-1}) ^{\frac{C_{\xi}}{\mathbf{r}_{i-1}}}     \mathcal{Q}    \times
\cdots\\
& \ \times \cdots\times    \mathcal{Q} (1+5M\mathbf{r}_{2})^{\frac{C_{\xi}}{\mathbf{r}_{2}}}   
\mathcal{Q} 
 (1+5M\mathbf{r}_{1} )^{\frac{C_{\xi}}{\mathbf{r}_{1}}}     \widetilde{ \mathcal{P}^{-1}(\mathbf{r}_{1})}\\
 \leq & \    ( C_{\xi}^{\prime} ) ^{ { [\frac{\tilde{t}|v|}{L_{\xi}}]  }   }   \times  \widetilde{ \mathcal{P}(\mathbf{r}_{ [\frac{\tilde{t}|v|}{L_{\xi}}]})}    \mathcal{Q}^{ [\frac{\tilde{t}|v|}{L_{\xi}}]-1}   \widetilde{ \mathcal{P}^{-1}(\mathbf{r}_{1})}\\
 \leq & \  (C_{\xi}^{\prime})^{ [\frac{\tilde{t}|v|}{L_{\xi}}]} \times\widetilde{ \mathcal{P}(\mathbf{r}_{ [\frac{\tilde{t}|v|}{L_{\xi}}]}) }\mathcal{F} \mathcal{A}^{ [\frac{\tilde{t}|v|}{L_{\xi}}]} \mathcal{F}^{-1} \widetilde{ \mathcal{P}^{-1}(\mathbf{r}_{1})}.
 \end{split}
 \end{equation}

 Now we use the explicit form of (\ref{Q_diag}) to bound
\begin{equation}
\begin{split}
&     {C^{C\tilde{t}|v|} \tiny{ \left[\begin{array}{cc|cc|c|cc}
1 & 0 & \frac{ (C_{\xi})^{\tilde{t}|v|}}{|v|} &  \frac{ (C_{\xi})^{\tilde{t}|v|}}{|v|} & \frac{ (C_{\xi})^{\tilde{t}|v|}}{|v|^{2}} \frac{1}{|\mathbf{r}_{1}|} & \frac{ (C_{\xi})^{\tilde{t}|v|}}{|v|^{2}} & \frac{ (C_{\xi})^{\tilde{t}|v|}}{|v|^{2}}\\
0 & 1 & 0 & 0 & 0 & 0 & 0 \\ \hline
0 & 0 & (C_{\xi})^{\tilde{t}|v|} & (C_{\xi})^{\tilde{t}|v|} &  \frac{ (C_{\xi})^{\tilde{t}|v|}}{|v|} \frac{1}{|\mathbf{r}_{1}|} & \frac{ (C_{\xi})^{\tilde{t}|v|}}{|v|} & \frac{ (C_{\xi})^{\tilde{t}|v|}}{|v|}\\
0 & 0 & (C_{\xi})^{\tilde{t}|v|} & (C_{\xi})^{\tilde{t}|v|} &  \frac{ (C_{\xi})^{\tilde{t}|v|}}{|v|} \frac{1}{|\mathbf{r}_{1}|} & \frac{ (C_{\xi})^{\tilde{t}|v|}}{|v|} & \frac{ (C_{\xi})^{\tilde{t}|v|}}{|v|}\\ \hline
0 & 0 & |v|(C_{\xi})^{\tilde{t}|v|}  \Big|\mathbf{r}_{[\frac{\tilde{t}|v|}{L_{\xi}}]}\Big| & |v|  (C_{\xi})^{\tilde{t}|v|} \Big|\mathbf{r}_{[ \frac{\tilde{t}|v| }{L_{\xi}}]}\Big|&   (C_{\xi})^{\tilde{t}|v|} \frac{\Big|\mathbf{r}_{[ \frac{\tilde{t}|v|}{L_{\xi}}]}\Big|}{|\mathbf{r}_{1}|} & (C_{\xi})^{\tilde{t}|v|}  \Big|\mathbf{r}_{[ \frac{\tilde{t}|v|}{L_{\xi}}]}\Big| &(C_{\xi})^{\tilde{t}|v|}  \Big|\mathbf{r}_{[\frac{\tilde{t}|v|}{L_{\xi}}]}\Big| \\ \hline
0 & 0 & |v| (C_{\xi})^{\tilde{t}|v|}  &  |v| (C_{\xi})^{\tilde{t}|v|}  & (C_{\xi})^{\tilde{t}|v|}\frac{1}{|\mathbf{r}_{1}|} & (C_{\xi})^{\tilde{t}|v|} & (C_{\xi})^{\tilde{t}|v|}\\
0 & 0 & |v|(C_{\xi})^{\tilde{t}|v|} & |v|(C_{\xi})^{\tilde{t}|v|} & (C_{\xi})^{\tilde{t}|v|} \frac{1}{|\mathbf{r}_{1}|} & (C_{\xi})^{\tilde{t}|v|} & (C_{\xi})^{\tilde{t}|v|}
 \end{array} \right]} } \\
 & \lesssim  C^{C|t-s||v|}
 \left[\begin{array}{cc|c|c|c}
1 &0  & \frac{1}{|v|} & \frac{1}{|v| |\mathbf{v}_{\perp}^{1}|} & \frac{1}{|v|^{2}}  \\
 0 & 1 & \mathbf{0}_{1,2} & 0 & \mathbf{0}_{1,2}\\ \hline
 \mathbf{0}_{2,1} & \mathbf{0}_{2,1} &  O_{\xi}(1) & \frac{1}{|\mathbf{v}_{\perp}^{1}|} & \frac{1}{|v|} \\ \hline
 0 & 0 & |\mathbf{v}_{\perp}^{1}| & O_{\xi}(1) & \frac{|\mathbf{v}_{\perp}^{1}|}{|v|}\\ \hline
 \mathbf{0}_{2,1} & \mathbf{0}_{2,1} & |v| & \frac{|v|}{|\mathbf{v}_{\perp}^{1}|} & O_{\xi}(1)
 \end{array} \right]_{7\times7}\label{whole_inter}
,
\end{split}
\end{equation}
where we have used (\ref{time_angle}) and the Velocity lemma (Lemma \ref{velocity_lemma}) and (\ref{40}), (\ref{41}) and 
\[
\mathbf{r}_{i} = \mathcal{C}_{1} e^{\frac{\mathcal{C}}{2}C_{1}} \mathbf{r}^{i} \lesssim e^{C|t-s||v|} \frac{|\mathbf{v}_{\perp}^{1}|}{|v|}, \ \ \text{and} \ \frac{\mathbf{r}_{[ \frac{|t-s||v| }{L_{\xi}}]}}{\mathbf{r}_{1}}  = \frac{\mathbf{r}^{[  \frac{|t-s||v|}{L_{\xi}}]}}{\mathbf{r}^{1}}
= \frac{\Big|\mathbf{v}_{\perp}^{[ \frac{   |t-s||v| }{L_{\xi}}  ]}\Big|}{|\mathbf{v}_{\perp}^{1}|} \leq
\mathcal{C}_{1} e^{\frac{\mathcal{C}}{2} |v||t-s|}.
\]

\vspace{8pt}

\noindent\textit{Step 8. Intermediate summary for the matrix method and the final estimate for \textit{Type II}}

\vspace{4pt}

Recall from (\ref{chain}) and (\ref{s1tstar}), (\ref{whole_inter}), (\ref{t1s1}),
\begin{equation}\notag
\begin{split}
&\frac{\partial (s^{\ell_{*}},\mathbf{X}_{ \ell_{*}}(s^{\ell_{*}}),\mathbf{V}_{  \ell_{*}}(s^{\ell_{*}}))}{\partial (s^{1},\mathbf{X}_{1}(s^{1}),\mathbf{V}_{1}(s^{1}) )}\equiv
\frac{\partial (s^{\ell_{*}}, \mathbf{x}_{\perp_{\ell_{*}}}(s^{\ell_{*}}),\mathbf{x}_{\parallel_{\ell_{*}}}(s^{\ell_{*}}), \mathbf{v}_{\perp_{\ell_{*}}}(s^{\ell_{*}}), \mathbf{v}_{\parallel_{\ell_{*}}}(s^{\ell_{*}})   )}{\partial ( s^{1}, \mathbf{x}_{\perp_{1}}(s^{1}), \mathbf{x}_{\parallel_{1}}(s^{1}), \mathbf{v}_{\perp_{1}}(s^{1}), \mathbf{v}_{\parallel_{1}}(s^{1})   )}
\\
 = & 
\ \frac{ \partial (s^{\ell_{*}}, \mathbf{x}_{\perp_{\ell_{*}}}(s^{\ell_{*}}),\mathbf{x}_{\parallel_{\ell_{*}}}(s^{\ell_{*}}), \mathbf{v}_{\perp_{\ell_{*}}}(s^{\ell_{*}}), \mathbf{v}_{\parallel_{\ell_{*}}}(s^{\ell_{*}})   )}{\partial (t^{\ell_{*}},0, \mathbf{x}_{\parallel_{\ell_{*}}}^{\ell_{*}}, \mathbf{v}_{\perp_{\ell_{*}}}^{\ell_{*}}, \mathbf{v}_{\parallel_{\ell_{*}}}^{\ell_{*}})}    \\
&  \times\prod_{i=1}^{ [\frac{|t-s||v|}{L_{\xi}}]}  \frac{\partial (t^{\ell_{i+1}},0, \mathbf{x}_{\parallel_{ \ell_{i+1}}}^{\ell_{i+1}}, \mathbf{v}_{\perp_{ \ell_{i+1}}}^{\ell_{i+1}}, \mathbf{v}_{\parallel_{ \ell_{i+1}}}^{\ell_{i+1}})}
{\partial (t^{\ell_{i+1}-1},0, \mathbf{x}_{\parallel_{{\ell_{i+1}-1}}}^{\ell_{i+1}-1}, \mathbf{v}_{\perp_{{\ell_{i+1}-1}}}^{\ell_{i+1}-1}, \mathbf{v}_{\parallel_{{\ell_{i+1}-1}}}^{\ell_{i+1}-1})}
 \times \cdots \times
 \frac{\partial (t^{\ell_{i}+1},0, \mathbf{x}_{\parallel_{\ell_{i} +1}}^{\ell_{i} +1}, \mathbf{v}_{\perp_{\ell_{i} +1}}^{\ell_{i}+1}, \mathbf{v}_{\parallel_{\ell_{i} +1}}^{\ell_{i}+1})}{\partial (t^{\ell_{i} },0, \mathbf{x}_{\parallel_{\ell_{i} }}^{\ell_{i}  }, \mathbf{v}_{\perp_{\ell_{i} }}^{\ell_{i} }, \mathbf{v}_{\parallel_{\ell_{i} }}^{\ell_{i} })}   \\
& \times    \frac{\partial (t^{1}, 0 , \mathbf{x}_{\parallel_{1}}^{1}, \mathbf{v}_{\perp_{1}}^{1}, \mathbf{v}_{\parallel_{1}}^{1})}{\partial (s^{1}, \mathbf{x}_{\perp_{1}}(s^{1}), \mathbf{x}_{\parallel_{1}}(s^{1}), \mathbf{v}_{\perp_{1}}(s^{1}), \mathbf{v}_{\parallel_{1}}(s^{1}) )}
 \\
 \leq&  \ (\ref{s1tstar}) \times (\ref{whole_inter}) \times (\ref{t1s1}) .
 \end{split}
 \end{equation}
 Then directly we bound
 \begin{equation}\label{inter_1}
 \begin{split}
 \leq & \ (\ref{s1tstar}) \times C^{C|t-s||v|}\\
 &\times { \left[\begin{array}{c|cc|cc}
 1 & \frac{1}{|\mathbf{v}_{\perp}^{1}|} + \frac{|v|}{|\mathbf{v}_{\perp}^{1}|^{2}} + |t^{1}-s^{1}| & \frac{1}{|v|} + \frac{|v|}{|\mathbf{v}_{\perp}^{1}|^{2}}   + |s^{1}-t^{1}| & \frac{1}{|v||\mathbf{v}_{\perp}^{1}|}   + |s^{1}-t^{1}|^{2} & \frac{1}{|v|^{2}} + \frac{|s^{1}-t^{1}|}{|v|} \\ \hline
0 & 0 & \mathbf{0}_{1,2} & 0 & \mathbf{0}_{1,2} \\
 \mathbf{0}_{2,1} & \frac{|v|^{2}}{|\mathbf{v}_{\perp}^{1}|^{2}} + \frac{|v|}{|\mathbf{v}_{\perp}^{1}|} + |v||s^{1}-t^{1}| & 1+ \frac{|v|^{2}}{|\mathbf{v}_{\perp}^{1}|^{2}}   & \frac{1}{|\mathbf{v}_{\perp}^{1}|}   + |s^{1}-t^{1}| & \frac{1}{|v|}  \\ \hline
0 & \frac{|v|^{2}}{|\mathbf{v}_{\perp}^{1}|} + |v| & |\mathbf{v}_{\perp}^{1}| + \frac{|v|^{2}}{|\mathbf{v}_{\perp}^{1}|} & O_{\xi}(1)  & \frac{|\mathbf{v}_{\perp}^{1}|}{|v|}\\
\mathbf{0}_{2,1} & \frac{|v|^{3}}{|\mathbf{v}_{\perp}^{1}|^{2}} + \frac{|v|^{2}}{|\mathbf{v}_{\perp}^{1}|} + |v|^{2} |s_{1}-t^{1}| & |v| + \frac{|v|^{3}}{|\mathbf{v}_{\perp}^{1}|^{2}} & \frac{|v|}{|\mathbf{v}_{\perp}^{1}|} + |v||s_{1}-t^{1}| & O_{\xi}(1)
 \end{array}
 \right] },
\end{split}
\end{equation}
 where we have used the Velocity lemma (Lemma \ref{velocity_lemma}) and (\ref{time_angle}), (\ref{40}), (\ref{41}) and
 \[
 |v||t^{1}-s^{1}| \leq \min \{  |v|(t_{\mathbf{b}}(x,v) + t_{\mathbf{b}}(x,-v)   ) , (t-s)|v| \} \lesssim_{\Omega} \min \{\frac{|\mathbf{v}_{\perp}^{1}|}{|v|}, (t-s)|v|  \}
 \lesssim_{\Omega} C^{C|t-s||v|} \min \{\frac{|\mathbf{v}_{\perp}^{1}|}{|v|}, 1 \}.
 \]
 Again we use the Velocity lemma (Lemma \ref{velocity_lemma}) and (\ref{time_angle}), (\ref{40}), (\ref{41}) and
\begin{eqnarray*}%
 |v||t^{\ell_{*}}-s^{\ell_{*}} | \leq \min\{ |v||t^{\ell_{*}} - t^{\ell_{*}+1}| ,| t-s| |v|\}\lesssim_{\Omega}  \min \{\frac{|\mathbf{v}_{\perp}^{\ell_{*}}|}{|v|}, |t-s||v|\}
 \lesssim_{\Omega} C^{C|t-s||v|}  \min \{\frac{|\mathbf{v}_{\perp}^{1}|}{|v|}, 1\}
 ,\\
\end{eqnarray*}
and
$|\mathbf{v}_{\perp}(s^{\ell_{*}})|   \lesssim_{\Omega} C^{C|v|(t-s)} |\mathbf{v}_{\perp}^{1}|$ to have, from (\ref{inter_1})
\begin{equation}\label{middle}
\begin{split}
\frac{\partial (s^{\ell_{*}},\mathbf{X}_{ \ell_{*}}(s^{\ell_{*}}),\mathbf{V}_{  \ell_{*}}(s^{\ell_{*}}))}{\partial (s^{1},\mathbf{X}_{1}(s^{1}),\mathbf{V}_{1}(s^{1}) )}
\lesssim C^{C|t-s||v|}  \left[\begin{array}{c|c|cc}
 0 &  \mathbf{0}_{1,3} & 0 & \mathbf{0}_{1,2}  \\ \hline
  |\mathbf{v}_{\perp}^{1}| &     \frac{|v|}{|\mathbf{v}_{\perp}^{1}|}  & \frac{1}{|v|}& \frac{1}{|v|} \\
 |v| &     \frac{|v|^{2}}{|\mathbf{v}_{\perp}^{1}|^{2}} 
 &  \frac{1}{|\mathbf{v}_{\perp}^{1}|} &   \frac{1}{|v|}\\ \hline
&&&   \\
|v|^{2}&      \frac{|v|^{3}}{|\mathbf{v}_{\perp}^{1}|^{2}} 
&   \frac{|v|}{|\mathbf{v}_{\perp}^{1}|} & O_{\xi}(1) \\
&&&
   \end{array}\right]_{7\times 7}.
\end{split}
\end{equation}

We consider the following case:
\begin{equation}\label{typeII}
\text{There exists } \ell \in [\ell_{*}(s;t,x,v), 0]  \text{ such that } \mathbf{r}^{\ell} \geq \sqrt{\delta}.
\end{equation}
Therefore $\ell$ is \textit{Type II} in (\ref{type_r}). Equivalently $\tau \in [t^{\ell+1}, t^{\ell}]$ for some $\ell_{*}\leq \ell \leq 0$ and $|\xi(X_{\mathbf{cl}}(\tau;t,x,v))|\geq C\delta$. By the Velocity lemma (Lemma \ref{velocity_lemma}), for all $1 \leq i \leq  \ell_{*}(s;t,x,v) ,$
\[
|\mathbf{r}^{i}| \ =  \ \frac{|\mathbf{v}_{\perp}^{i}|}{|v|} \ \gtrsim_{\xi} \ e^{-C_{\xi}|v||t^{i}-t^{\ell}|} |\mathbf{r}^{\ell}| \ \gtrsim_{\xi} \ e^{-C_{\xi}|v|(t-s)} \sqrt{\delta}.
\]
Especially, for all $1 \leq i \leq  \ell_{*}(s;t,x,v) ,$
\[
|\mathbf{r}^{1}| \gtrsim_{\xi} e^{-C_{\xi}|v|(t-s)} \sqrt{\delta}, \ \ \ \frac{1}{|\mathbf{r}^{i}|} = \frac{|v|}{|\mathbf{v}_{\perp}^{i}|} \lesssim_{\xi} \frac{e^{C_{\xi}|v|(t-s)}}{\sqrt{\delta}}.
\]
Note that $\ell_{*}(s;t,x,v) \lesssim \max_{i}\frac{|v||t-s|}{ \mathbf{r}^{i}} \lesssim_{\delta} C^{C|v||t-s|}.$

Therefore in the case of (\ref{typeII}), from (\ref{middle}),
\begin{equation}\notag
 \begin{split}
  \frac{\partial (s^{\ell_{*}}, \mathbf{X}_{\ell_{*}} (s^{\ell_{*}}), \mathbf{V}_{\ell_{*} } (s^{\ell_{*}})   )}{\partial (s^{1}, \mathbf{X}_{1} (s^{1}), \mathbf{V}_{1} (s^{1})   )} 
 &\lesssim C^{C(t-s)|v|} \left[\begin{array}{c|cc|ccc}
 0 & 0 & \mathbf{0}_{1,2} & 0 & \mathbf{0}_{1,2} &\\ \hline
  |\mathbf{v}_{\perp}^{1}| &  \frac{1}{\sqrt{\delta}}&     \frac{1}{\sqrt{\delta}}& \frac{1}{|v|}& \frac{1}{|v|}&\\
 |v| &  \frac{1}{ {\delta}}    &  \frac{1}{ {\delta}}    &  \frac{1}{|v|} \frac{1}{\sqrt{\delta}}&   \frac{1}{|v|}\\ \hline
&&&&& \\
|v|^{2}& |v|  \frac{1}{ {\delta}}   &     |v|     \frac{1}{ {\delta}}   & \frac{1}{\sqrt{\delta}} & 1 \\
&&&&&
   \end{array}\right]\\
   & \lesssim_{\delta} C^{C|v|(t-s)} \left[\begin{array}{c|c|c} 0 & \mathbf{0}_{1,3} & \mathbf{0}_{1,3} \\ \hline |v| & 1 & \frac{1}{|v|} \\ \hline |v|^{2} & |v|  & 1 \end{array}\right].
 \end{split}
 \end{equation}
Using (\ref{ss1}) and (\ref{s1t}) we conclude
\begin{equation}\label{final_Dxv_nongrazing}
\begin{split}
 &\frac{\partial (s,X_{\mathbf{cl}}(s;t,x,v),V_{\mathbf{cl}}(s;t,x,v))}{\partial (t,x,v)} \\
 &  \lesssim_{\delta,\xi} \ C^{C|v|(t-s)}  \ 
 \frac{\partial (s,X_{\mathbf{cl}}(s), V_{\mathbf{cl}}(s))}{\partial (s^{\ell_{*}}, \mathbf{X}_{\ell_{*} }(s^{\ell_{*}}), \mathbf{V}_{ \ell_{*}}(s^{\ell_{*}}))}
\left[\begin{array}{c|c|c} 0 & \mathbf{0}_{1,3} & \mathbf{0}_{1,3} \\ \hline |v| & 1 & \frac{1}{|v|} \\ \hline |v|^{2} & |v|  & 1 \end{array}\right]  \frac{   {\partial (s^{1}, \mathbf{x}_{\perp_{1}}(s^{1}), \mathbf{x}_{\parallel_{1}}(s^{1}), \mathbf{v}_{\perp_{1}}(s^{1}), \mathbf{v}_{\parallel_{1}}(s^{1}) )}       }{\partial (t,x,v)}\\
&  \lesssim_{\delta,\xi}\ C^{C|v|(t-s)} \left[\begin{array}{ccc} 0 & \mathbf{0}_{1,3} & \mathbf{0}_{1,3} \\ |v| & 1 & |s^{\ell_{*}}-s| \\ \mathbf{0}_{3,1} & |v| & 1 \end{array}\right]
 \left[\begin{array}{ccc} 0 & \mathbf{0}_{1,3} & \mathbf{0}_{1,3} \\   |v| & 1 & \frac{1}{|v|} \\  |v|^{2} & |v|  & 1 \end{array}\right]
 \left[\begin{array}{ccc} 1 & \mathbf{0}_{1,3} & \mathbf{0}_{1,3} \\ \mathbf{0}_{3,1} & 1 & |t-s^{1}| \\ \mathbf{0}_{3,1} & |v| &  1  \end{array}\right]\\
 & \lesssim_{\delta,\xi} \ C^{C|v|(t-s)}  \left[\begin{array}{ccc} 0 & \mathbf{0}_{1,3} & \mathbf{0}_{1,3} \\ |v| & 1 & \frac{1}{|v|} \\ |v|^{2} & |v| & 1 \end{array} \right].
\end{split}
\end{equation}

Now we only need to consider the remainder case of (\ref{typeII}), i.e.
\begin{equation}\label{typeI}
\text{ For all } \ell \in [\ell_{*}(s;t,x,v),0] , \text{ we have }  \mathbf{r}_{\ell} \leq \sqrt{\delta}.
\end{equation}
Note that in this case the moving frame($\mathbf{p}^{\ell}-$spherical coordinate) is well-defined for all $\tau \in [s,t]$. In next two step we use the ODE method to refine the submatrix  of (\ref{middle}): 
\begin{equation}\notag
\begin{split}
 \frac{\partial(\mathbf{x}_{\parallel_{\ell_{*}}}(s^{\ell_{*}}),   \mathbf{v}_{\parallel_{\ell_{*}}}(s^{\ell_{*}})   )}{\partial (\mathbf{x}_{\perp_{1}}(s^{1}), \mathbf{x}_{\parallel_{1}}(s^{1}), \mathbf{v}_{\perp_{1}}(s^{1}), \mathbf{v}_{\parallel_{1}}(s^{1})   )} 
 =
\left[\begin{array}{cccccccc}
 \frac{\partial\mathbf{x}_{\parallel_{\ell_{*}} }(s^{\ell_{*}})}{\partial \mathbf{x}_{\perp_{1}}(s^{1})  } & 
  \frac{\partial\mathbf{x}_{\parallel_{\ell_{*}} }(s^{\ell_{*}})}{\partial \mathbf{x}_{\parallel_{1} }(s^{1})  } & 
    \frac{\partial\mathbf{x}_{\parallel_{\ell_{*}} }(s^{\ell_{*}})}{\partial \mathbf{v}_{\perp_{1}}(s^{1})  } &  
    \frac{\partial\mathbf{x}_{\parallel_{\ell_{*}} }(s^{\ell_{*}})}{\partial \mathbf{v}_{\parallel_{1} }(s^{1})  } \\ 
     \frac{\partial\mathbf{v}_{\parallel_{\ell_{*}} }(s^{\ell_{*}})}{\partial \mathbf{x}_{\perp_{1}}(s^{1})  } &  
    \frac{\partial\mathbf{v}_{\parallel_{\ell_{*}} }(s^{\ell_{*}})}{\partial \mathbf{x}_{\parallel_{1} }(s^{1})  } &
    \frac{\partial\mathbf{v}_{\parallel_{\ell_{*}} }(s^{\ell_{*}})}{\partial \mathbf{v}_{\perp_{1}}(s^{1})  } &  
    \frac{\partial\mathbf{v}_{\parallel_{\ell_{*}} }(s^{\ell_{*}})}{\partial \mathbf{v}_{\parallel_{1} }(s^{1})  }  
\end{array}
\right]_{4\times 6}
.
\end{split}
\end{equation}

\vspace{8pt}

\noindent{\textit{Step 9. ODE method within the time scale $|t-s||v|\sim L_{\xi}$}}

\vspace{4pt}
 
Recall the end points (time) of intermediate groups from (\ref{group}):  
\begin{equation}\notag
\begin{split}
s < \underbrace{ t^{\ell_{*}} < t^{\ell_{[ \frac{|t-s||v|}{L_{\xi}}]} +1 } }_{{[ \frac{|t-s||v|}{L_{\xi}}]  +1 }}  <  \underbrace{t^{\ell_{[ \frac{|t-s||v|}{L_{\xi}}]   }} < t^{\ell_{[ \frac{|t-s||v|}{L_{\xi}}]-1}+1}   }_{{[ \frac{|t-s||v|}{L_{\xi}}]   }}
<  \cdots <  \underbrace{t^{\ell_{i }} < t^{\ell_{i-1}+1  }}_{i }  < \cdots
<  \underbrace{t^{\ell_{2}}<t^{\ell_{1} +1} }_{2}< \underbrace{t^{\ell_{1}}< t^{1} }_{1}< t,
\end{split}
\end{equation}
where the underbraced numbering indicates the index of the intermediate group. We further choose points independently on $(t,x,v)$ for all $i=1,2, \cdots, [\frac{|t-s||v|}{L_{\xi}}]:$ 
\begin{eqnarray*} 
& t^{\ell_{1} + 1} < s^{2} < t^{\ell_{1}},&
\\
& t^{\ell_{2}+1} < s^{3}< t^{\ell_{2}},&\\
& \vdots&\\
&t^{\ell_{i}+1}<s^{i+1}  <\underbrace{ t^{\ell_{i }}< \cdots \cdots< t^{\ell_{i-1}+1} }_{i -\text{intermediate group}}< s^{i }<   t^{\ell_{i-1} }, & \\ 
&\vdots &\\ 
&t^{\ell_{[ \frac{|t-s||v|}{L_{\xi}}] }+1} <  s^{\ell_{[ \frac{|t-s||v|}{L_{\xi}}] }+1}  < t^{\ell_{[ \frac{|t-s||v|}{L_{\xi}}] } }
. & \end{eqnarray*} 

%
%

We claim the following estimate at $s^{i+1 }$ via $s^{i}$: 
\begin{equation}\label{ODE_onegroup}
\begin{split}
&\left[\begin{array}{cc} |
  \frac{\partial \mathbf{x}_{\parallel_{\ell_{i}}} (s^{i+1})}{\partial {\mathbf{x}_{\perp_{1}}}(s^{1})  }|  &  |  \frac{\partial \mathbf{x}_{\parallel_{\ell_{i}}} (s^{i+1})}{\partial {\mathbf{x}_{\parallel_{1}}} (s^{1})  }|  \\
| \frac{\partial   \mathbf{v}_{\parallel _{\ell_{i}}}   (s^{i+1})}{  \partial {\mathbf{x}_{\perp _{1}}}(s^{1}) } |
& | \frac{\partial   \mathbf{v}_{\parallel _{\ell_{i}}}   ( s^{i+1})}{  \partial {\mathbf{x}_{\parallel _{1}}}(s^{1}) } |
\end{array} \right] \\
&\lesssim_{\delta,\xi}
\left[\begin{array}{cc} 1 & \frac{1}{|v|} \\ |v| & 1   \end{array} \right]
\left[\begin{array}{cc} |   \frac{\partial \mathbf{x}_{\parallel_{\ell_{i}}} ( s^{i })}{\partial {\mathbf{x}_{\perp_{1}}} (s^{1}) }|
& |   \frac{\partial \mathbf{x}_{\parallel_{\ell_{i}}}  ( s^{i })}{\partial {\mathbf{x}_{\parallel_{1}}} (s^{1})  }|
   \\
| \frac{\partial   \mathbf{v}_{\parallel _{\ell_{i}}}    ( s^{i })}{  \partial {\mathbf{x}_{\perp _{1}}} (s^{1})} |
&
| \frac{\partial   \mathbf{v}_{\parallel _{\ell_{i}}}    ( s^{i })}{  \partial {\mathbf{x}_{\parallel _{1}}} (s^{1})} |
 \end{array} \right]
+ e^{C|v||t-s^{i}|}
\left[\begin{array}{cc} 1 & \frac{1}{|v|} \\ |v| & 1   \end{array} \right]
\left[\begin{array}{cc}  0 & 0\\ |v| \Big(1+ \frac{|v|}{|\mathbf{v}_{\perp }^{1}|} \Big) &   |v| \Big(1+ \frac{|v|}{|\mathbf{v}_{\perp }^{1}|} \Big) \end{array} \right]  ,\\
&\left[\begin{array}{cc} |
  \frac{\partial \mathbf{x}_{\parallel_{\ell_{i}}}  (s^{i+1})}{\partial {\mathbf{v}_{\perp_{1}}}(s^{1})  }|  &  |  \frac{\partial \mathbf{x}_{\parallel_{\ell_{i}}}(s^{i+1})}{\partial {\mathbf{v}_{\parallel_{1}}}  (s^{1})}|  \\
| \frac{\partial   \mathbf{v}_{\parallel _{\ell_{i}}}   (s^{i+1})}{  \partial {\mathbf{v}_{\perp _{1}}} (s^{1})} |
& | \frac{\partial   \mathbf{v}_{\parallel _{\ell_{i}}}   (s^{i+1})}{  \partial {\mathbf{v}_{\parallel _{1}}} (s^{1})} |
\end{array} \right] \\
&\lesssim_{\delta,\xi}
\left[\begin{array}{cc} 1 & \frac{1}{|v|} \\ |v| & 1   \end{array} \right]
\left[\begin{array}{cc} |   \frac{\partial \mathbf{x}_{\parallel_{\ell_{i}}} ( s^{i})}{\partial {\mathbf{v}_{\perp_{1}}} (s^{1}) }|
& |   \frac{\partial \mathbf{x}_{\parallel_{\ell_{i}}}  ( s^{i})}{\partial {\mathbf{v}_{\parallel_{1}}}   (s^{1})}|
   \\
| \frac{\partial   \mathbf{v}_{\parallel _{\ell_{i}}}    ( s^{i})}{  \partial {\mathbf{v}_{\perp _{1}}}  (s^{1})} |
&
| \frac{\partial   \mathbf{v}_{\parallel _{\ell_{i}}}    ( s^{i})}{  \partial {\mathbf{v}_{\parallel _{1}}}  (s^{1})} |
 \end{array} \right]
+ e^{C|v||t-s^{i}|}
\left[\begin{array}{cc} 1 & \frac{1}{|v|} \\ |v| & 1   \end{array} \right]
\left[\begin{array}{cc}   0 &  0
\\ 1 & 1 \end{array} \right] .
\end{split}
\end{equation}

Within the $i-$th intermediate group, we fix $\mathbf{p}^{\ell_{i}}-$spherical coordinate in \textit{Step 9}. For the sake of simplicity we drop the index $\ell_{i}$.

Denote, from (\ref{F||}),
\begin{equation}\label{F||_decompose}
\begin{split}
F_{\parallel}(\mathbf{x}_{\perp}, \mathbf{x}_{\parallel}, \mathbf{v}_{\perp}, \mathbf{v}_{\parallel}) := D(\mathbf{x}_{\perp}, \mathbf{x}_{\parallel} ,\mathbf{v}_{\parallel})+E(\mathbf{x}_{\perp}, \mathbf{x}_{\parallel}, \mathbf{v}_{\parallel})\mathbf{v}_{\perp} ,
\end{split}
\end{equation}
where $D$ is a $\mathbf{r}^{3}$-vector-valued function and $E$ is a $3\times3$ matrix-valued function:
\begin{equation}\notag
\begin{split}
 D(\mathbf{x}_{\perp}, \mathbf{x}_{\parallel}, \mathbf{v}_{\parallel}) 
 =&\sum_{i} G_{ij} (\mathbf{x}_{\perp}, \mathbf{x}_{\parallel})\frac{(-1)^{i+1}}{-\mathbf{n}(\mathbf{x}_{\parallel}) \cdot (\partial_{1} \mathbf{\eta}(\mathbf{x}_{\parallel}) \times \partial_{2} \mathbf{\eta}(\mathbf{x}_{\parallel}))}\\
& \ \ \ \times  \Big\{  \mathbf{v}_{\parallel} \cdot  \nabla^{2}\mathbf{ \eta}(\mathbf{x}_{\parallel}) \cdot \mathbf{v}_{\parallel} - \mathbf{x}_{\perp} \mathbf{v}_{\parallel} \cdot \nabla^{2} \mathbf{n}(\mathbf{x}_{\parallel}) \cdot \mathbf{v}_{\parallel}
\Big\} \cdot (-\mathbf{n} (\mathbf{x}_{\parallel}) \times \partial_{i+1} \mathbf{\eta}(\mathbf{x}_{\parallel})),
\end{split}
\end{equation}
and 
\begin{equation}\notag
\begin{split}
 E(\mathbf{x}_{\perp}, \mathbf{x}_{\parallel}, \mathbf{v}_{\parallel}) 
 = \sum_{i} G_{ij} (\mathbf{x}_{\perp}, \mathbf{x}_{\parallel})\frac{(-1)^{i+1}}{-\mathbf{n}(\mathbf{x}_{\parallel}) \cdot (\partial_{1} \mathbf{\eta}(\mathbf{x}_{\parallel}) \times \partial_{2} \mathbf{\eta}(\mathbf{x}_{\parallel}))}   2 \mathbf{v}_{\perp} \mathbf{v}_{\parallel}\cdot \nabla \mathbf{n}(\mathbf{x}_{\parallel}) \cdot (-\mathbf{n} (\mathbf{x}_{\parallel}) \times \partial_{i+1} \mathbf{\eta}(\mathbf{x}_{\parallel})).
\end{split}
\end{equation}
Here $G_{ij}(\cdot,\cdot)$ is a smooth bounded function defined in (\ref{G}) and we used the notational convention $i\equiv i \ \mathrm{mod } \ 2$.

From Lemma \ref{chart_lemma} we take the time integration of (\ref{ODE_ell}) along the characteristics to have
\begin{equation}\notag
\begin{split}
\mathbf{x}_{\parallel}( s^{i+1}) & = \mathbf{x}_{\parallel}(s^{i}) - \int^{s^{i}}_{s^{i+1}} \mathbf{v}_{\parallel}(\tau ) \mathrm{d}\tau ,\\
\mathbf{v}_{\parallel}(s^{i+1}) & =  \mathbf{v}_{\parallel}(s^{i}) - \int^{s^{i}}_{s^{i+1}} \big\{  E(\mathbf{x}_{\perp}(\tau )  , \mathbf{x}_{\parallel}(\tau ), \mathbf{v}_{\parallel}(\tau ))\mathbf{v}_{\perp}(\tau ) +  D(\mathbf{x}_{\perp}(\tau ), \mathbf{x}_{\parallel}(\tau ) ,\mathbf{v}_{\parallel}(\tau ))\big\} \ \mathrm{d}\tau .
\end{split}
\end{equation}
Note that $\mathbf{v}_{\perp}(\tau )$ is not continuous with respect to the time $\tau$. Using (\ref{ODE_ell}) we rewrite this time integration as
\begin{equation}\notag
\begin{split}
\int^{s^{i}}_{s^{i+1}}  E(\mathbf{x}_{\perp}(\tau ), \mathbf{x}_{\parallel}(\tau ), \mathbf{v}_{\parallel}(\tau ))   \mathbf{v}_{\perp}(\tau ) \mathrm{d}\tau 
 =  
 \int^{s^{i}}_{t^{\ell_{i-1}+1}  } + \sum_{\ell= \ell_{i}-1}^{\ell_{i-1}+1} \int^{t^{\ell}}_{t^{\ell+1}} + \int^{t^{\ell_{i}  }}_{ s^{i+1}} ,\\
\end{split}
\end{equation}
then we use $\mathbf{v}_{\perp}(\tau ) = \dot{\mathbf{x}}_{\perp}(\tau )$ and the integration by parts to have
\begin{equation}\notag
\begin{split}
 &  \int_{t^{\ell_{i-1} +1} }^{s^{i}}   E(\mathbf{x}_{\perp}(\tau ), \mathbf{x}_{\parallel}(\tau ), \mathbf{v}_{\parallel}(\tau ) ) \dot{\mathbf{x}}_{\perp}(\tau ) \mathrm{d}\tau 
  +  \sum_{\ell=\ell_{i}-1}^{\ell_{i-1}+1} \int^{t^{\ell}}_{t^{\ell+1}}  E(\mathbf{x}_{\perp}(\tau ), \mathbf{x}_{\parallel}(\tau ),\mathbf{v}_{\parallel}(\tau ) )\dot{\mathbf{x}}_{\perp}(\tau ) \mathrm{d}\tau  \\
 &+ \int^{t^{\ell_{i}   }}_{s^{i+1}}  E(\mathbf{x}_{\perp}(\tau ), \mathbf{x}_{\parallel}(\tau ), \mathbf{v}_{\parallel}(\tau )) \dot{\mathbf{x}}_{\perp}(\tau ) \mathrm{d}\tau \\
 =& \  E(s^{i})    {\mathbf{x}}_{\perp}(s^{i}) -E( t^{\ell_{i-1}+1} )  \underbrace{\mathbf{x}_{\perp}(  t^{\ell_{i-1} +1})}_{=0} - \int_{t^{\ell_{i-1}+1}}^{s^{i}}
\big[{\mathbf{v}}_{\perp}(\tau ),  {\mathbf{v}}_{\parallel}(\tau ), {F}_{\parallel}(\tau )  \big]\cdot \nabla E(\tau )
   {\mathbf{x}}_{\perp}(\tau ) \mathrm{d}\tau \\
 & + \sum_{\ell=\ell_{i}-1}^{\ell_{i-1}+1 } \Big\{
 E( t^{\ell} )  \underbrace{\mathbf{x}_{\perp}(t^{\ell})}_{=0} -  E(t^{\ell+1})   \underbrace{{\mathbf{x}}_{\perp}(t^{\ell+1})}_{=0}  
  - \int_{t^{\ell+1}}^{t^{\ell}}
\big[{\mathbf{v}}_{\perp}(\tau ),  {\mathbf{v}}_{\parallel}(\tau ), {F}_{\parallel}(\tau )  \big]  \cdot \nabla E(\tau )
   {\mathbf{x}}_{\perp}(\tau ) \mathrm{d}\tau 
 \Big\}\\
 & +
 E( t^{\ell_{i} } )  \underbrace{\mathbf{x}_{\perp}( t^{\ell_{i}  }) }_{=0}- E(s^{i+1})  { {\mathbf{x}}_{\perp}(s^{i+1}) }
    - \int^{ t^{\ell_{i} } }_{s^{i+1}}
\big[{\mathbf{v}}_{\perp}(\tau ),  {\mathbf{v}}_{\parallel}(\tau ), {F}_{\parallel}(\tau )  \big]   \cdot \nabla E(\tau )
  {\mathbf{x}}_{\perp}(\tau ) \mathrm{d}\tau \\
 =& \  E( \mathbf{x}_{\perp}, \mathbf{x}_{\parallel}, \mathbf{v}_{\parallel} ) (s^{i}) \mathbf{x}_{\perp}(s^{i})  -  E(s^{i+1}) {\mathbf{x}}_{\perp}(s^{i+1})
 - \int_{s^{i}}^{s^{i+1}}   \big[{\mathbf{v}}_{\perp}(\tau ),  {\mathbf{v}}_{\parallel}(\tau ), {F}_{\parallel}(\tau )  \big]    \cdot \nabla E(\tau )  {\mathbf{x}}_{\perp}(\tau ) \mathrm{d}\tau ,
\end{split}
\end{equation}
where we have used the fact $X_{\mathbf{cl}}(t^{\ell}) \in\partial\Omega$(therefore $\mathbf{x}_{\perp}(t^{\ell})=0$) and the notation $E(\tau ) = E(\mathbf{x}_{\perp}(\tau ), \mathbf{x}_{\parallel}(\tau ), \mathbf{v}_{\parallel}(\tau )), \ D(\tau ) = D(\mathbf{x}_{\perp}(\tau ), \mathbf{x}_{\parallel}(\tau ),\mathbf{v}_{\parallel}(\tau )), \ F_{\parallel}(\tau  ) = F_{\parallel}( \mathbf{x}_{\perp}(\tau ), \mathbf{x}_{\parallel}(\tau ), \mathbf{v}_{\perp}(\tau ), \mathbf{v}_{\parallel}(\tau ) ).$

Overall we have
\begin{equation}\label{xv_mildform}
\begin{split}
\mathbf{x}_{\parallel}(s^{i+1}) & = \mathbf{x}_{\parallel}(s^{i}) - \int_{s^{i+1}}^{s^{i}} \mathbf{v}_{\parallel}(\tau ) \mathrm{d}\tau ,\\
\mathbf{v}_{\parallel}(s^{i+1}) & =  \mathbf{v}_{\parallel}( s^{i }) - E( s^{i }) \mathbf{x}_{\perp} ( s^{i }) +
 E ( s^{i+1})  \mathbf{x}_{\perp}( s^{i+1}) \\
 & \ \  + \int^{ s^{i }}_{ s^{i+1}}     \big[{\mathbf{v}}_{\perp}(\tau ),  {\mathbf{v}}_{\parallel}(\tau ), {F}_{\parallel}(\tau )  \big]    \cdot \nabla E(\tau )  {\mathbf{x}}_{\perp}(\tau ) \mathrm{d}\tau   - \int^{ s^{i }}_{ s^{i+1}} D(  \tau )  \mathrm{d}\tau .
\end{split}
\end{equation}
Denote $$\partial = [\partial_{\mathbf{x}_{\perp}(s^{1})}, \partial_{\mathbf{x}_{\parallel}(s^{1})}, \partial_{\mathbf{v}_{\perp}(s^{1})}, \partial_{\mathbf{v}_{\parallel} (s^{1})}]= [\frac{\partial}{\partial \mathbf{x}_{\perp} (s^{1})}, \frac{\partial}{\partial \mathbf{x}_{\parallel}  ( s^{1})}, \frac{\partial}{\partial \mathbf{v}_{\perp}  ( s^{1})}, \frac{\partial}{\partial \mathbf{v}_{\parallel}  ( s^{1})}].$$ 
 We claim that, in a sense of distribution on $(s^{1}, \mathbf{x}_{\perp}( s^{1}), \mathbf{x}_{\parallel}(  s^{1}), \mathbf{v}_{\perp}(  s^{1}), \mathbf{v}_{\parallel}(  s^{1})) \in [0,\infty) \times (0,C_{\xi}) \times  (0,2\pi] \times (\delta, \pi-\delta) \times \mathbb{R}\times \mathbb{R}^{2}$,
 \begin{equation}\label{piece_d}
 \begin{split}
& \big[\partial \mathbf{x}_{\perp}( s^{i+1};s^{1},\mathbf{x}(s^{1}),\mathbf{v}(s^{1})), 
\partial \mathbf{x}_{\parallel} ( s^{i+1};s^{1},\mathbf{x}(s^{1}),\mathbf{v}(s^{1})),  \  \partial \mathbf{v}_{\parallel} ( s^{i+1};s^{1},\mathbf{x}(s^{1}),\mathbf{v}(s^{1}))\big]\\
 &= \sum_{\ell  } \mathbf{1}_{[t^{\ell+1},t^{\ell})} (s^{i+1})[\partial \mathbf{x}_{\perp} , \partial \mathbf{x}_{\parallel} ,     \partial \mathbf{v}_{\parallel}  \big],\\
& \partial\big[ \mathbf{v}_{\perp} ( s^{i+1};s^{1},\mathbf{x}(s^{1}),\mathbf{v}(s^{1}))\mathbf{x}_{\perp} ( s^{i+1};s^{1},\mathbf{x}(s^{1}),\mathbf{v}(s^{1}))\big] = \sum_{\ell  } \mathbf{1}_{[t^{\ell+1},t^{\ell})}(s^{i+1}) \big\{   \partial \mathbf{v}_{\perp}  \mathbf{x}_{\perp} + \mathbf{v}_{\perp}   \partial \mathbf{x}_{\perp} \big\},
 \end{split}
 \end{equation}
i.e. the distributional derivatives of $[ \mathbf{x}_{\perp}, \mathbf{x}_{\parallel} , \mathbf{v}_{\parallel}]$ and   $\mathbf{v}_{\perp} \mathbf{x}_{\perp}$ equal the piecewise derivatives. Let $\phi(\tau^{\prime}, \mathbf{x}_{\perp}, \mathbf{x}_{\parallel}, \mathbf{v}_{\perp}, \mathbf{v}_{\parallel}) \in C^{\infty}_{c} ( [0,\infty) \times   (0,C_{\xi})\times \mathbb{S}^{2}\times\mathbb{R}\times\mathbb{R}^{2})$. Therefore $\phi\equiv 0$ when $\mathbf{x}_{\perp} < \delta$. For $\mathbf{x}_{\perp } \geq \delta$ we use the proof of Lemma \ref{chart_lemma}: For $x= \mathbf{\eta}(\mathbf{x}_{\parallel}) + \mathbf{x}_{\perp} [-\mathbf{n}(\mathbf{x}_{\parallel})],$
\[
|\mathbf{x}_{\perp}|  \ \lesssim_{\xi} \ \xi(x) = \xi(  \mathbf{\eta}(\mathbf{x}_{\parallel}) + \mathbf{x}_{\perp}[- \mathbf{n}(\mathbf{x}_{\parallel})] ) \ \lesssim_{\xi} \ |\mathbf{x}_{\perp}|,
\]
and therefore $\xi(x)\gtrsim_{\xi}\delta$ and $\alpha(x,v) \gtrsim_{\xi} |\xi(x)||v|^{2} \gtrsim_{\xi} |v|^{2 } \delta.$ Since we are considering the case $t-s> t_{\mathbf{b}}(x,v)$, from $|v|t_{\mathbf{b}}(x,v) \gtrsim \mathbf{x}_{\perp} \geq \delta$ we have $|v| \gtrsim_{\xi} \frac{\delta}{t-s}$ and finally we obtain the lower bound $
\alpha(x,v) \gtrsim_{\xi} \frac{\delta^{3}}{|t-s|^{2}}>0.$ By the Velocity lemma, for $(x,v)\in \text{supp}(\phi)$
\[
\alpha(x^{\ell},v^{\ell}) \gtrsim_{\xi} e^{-C|v||t^{1}-t^{\ell}|} \alpha(x,v) \gtrsim_{\xi} e^{-C|v|(t-s)} \frac{\delta^{3}}{|t-s|^{2}} \gtrsim_{\xi, |t-s|,\delta, \phi}1>0,
\]
where we used the fact that $\phi$ vanishes away from a compact subset $\text{supp}(\phi)$. Therefore $t^{\ell}(t,x,v)= t^{\ell}(t, \mathbf{x}_{\perp}, \mathbf{x}_{\parallel}, \mathbf{v}_{\perp}, \mathbf{v}_{\parallel})$ is smooth with respect to $\mathbf{x}_{\perp}, \mathbf{x}_{\parallel}, \mathbf{v}_{\perp}, \mathbf{v}_{\parallel}$ locally on $\text{supp}(\phi)$ and therefore $\mathcal{M}=\{(\tau^{\prime}, \mathbf{x}, \mathbf{v}) \in  \text{supp}(\phi): \tau^{\prime} = t^{\ell}(t, \mathbf{x}, \mathbf{v})  \}$ is a smooth manifold.

It suffices to consider the case $\tau^{\prime} \sim t^{\ell}(t,x,v)$. Denote $\partial_{\mathbf{e}}= [\partial_{\mathbf{x_{\perp}}}, \partial_{\mathbf{x}_{\parallel,1}}, \partial_{\mathbf{x}_{\parallel,2}}, \partial_{\mathbf{v_{\perp}}}, \partial_{\mathbf{v}_{\parallel,1}}, \partial_{\mathbf{v}_{\parallel,2}}]$ and $n_{\mathcal{M}}= \mathbf{e}_{1}$ to have
\begin{equation}\notag
\begin{split}
&\int_{\{(\tau^{\prime},\mathbf{x}, \mathbf{v}) \in \text{supp}(\phi)\}}[\partial_{\mathbf{e}} \mathbf{x}_{\perp}(\tau^{\prime};t,\mathbf{x}, \mathbf{v}), \partial_{\mathbf{e}} \mathbf{x}_{\parallel}(\tau^{\prime};t,\mathbf{x}, \mathbf{v}), \partial_{\mathbf{e}}\mathbf{v}_{\parallel}(\tau^{\prime};t,\mathbf{x}, \mathbf{v}) ]
\phi(\tau^{\prime},\mathbf{x}, \mathbf{v}) \mathrm{d} \mathbf{x} \mathrm{d}\mathbf{v} \mathrm{d}\tau^{\prime}\\
=& \int_{\tau^{\prime} < t^{\ell}} + \int_{\tau^{\prime} \geq t^{\ell}}\\
=& \int_\mathcal{M} \Big(\lim_{\tau^{\prime} \uparrow t^{\ell} } [\mathbf{x}_{\perp}(\tau^{\prime}), \mathbf{x}_{\parallel}(\tau^{\prime}), \mathbf{v}_{\parallel}(\tau^{\prime})]  - \lim_{\tau^{\prime} \downarrow t^{\ell}} [\mathbf{x}_{\perp}(\tau^{\prime}), \mathbf{x}_{\parallel}(\tau^{\prime}), \mathbf{v}_{\parallel} (\tau^{\prime}) ] \Big) \phi(\tau^{\prime}, \mathbf{x}, \mathbf{v})\{ \mathbf{e}\cdot n_{\mathcal{M}}\}  \mathrm{d}\mathbf{x} \mathrm{d}\mathbf{v}\\
&- \int_{\{\tau^{\prime} \neq t^{\ell}(t,\mathbf{x}, \mathbf{v})\}} [\mathbf{x}_{\perp}(\tau^{\prime}), \mathbf{x}_{\parallel}(\tau^{\prime}), \mathbf{v}_{\parallel}(\tau^{\prime})] \partial_{\mathbf{e}} \phi(\tau^{\prime}, \mathbf{x}, \mathbf{v}) \mathrm{d}\tau^{\prime} \mathrm{d}\mathbf{v} \mathrm{d}\mathbf{x} \\
=& - \int_{\{\tau^{\prime} \neq t^{\ell}(t,\mathbf{x}, \mathbf{v})\}} [\mathbf{x}_{\perp}(\tau^{\prime}), \mathbf{x}_{\parallel}(\tau^{\prime}), \mathbf{v}_{\parallel}(\tau^{\prime})] \partial_{\mathbf{e}} \phi(\tau^{\prime}, \mathbf{x}, \mathbf{v}) \mathrm{d}\tau^{\prime} \mathrm{d}\mathbf{v} \mathrm{d}\mathbf{x},
\end{split}
\end{equation}
where we used the continuity of $[ \mathbf{x}_{\perp}(\tau^{\prime};t,\mathbf{x}, \mathbf{v}),  \mathbf{x}_{\parallel}(\tau^{\prime};t,\mathbf{x}, \mathbf{v}),  \mathbf{v}_{\parallel}(\tau^{\prime};t,\mathbf{x}, \mathbf{v}) ]$ in terms of $\tau^{\prime}$ near $t^{\ell}(t,\mathbf{x}, \mathbf{v})$.

Note that $\mathbf{v}_{\perp}(\tau^{\prime};t, \mathbf{x}, \mathbf{v})$ is discontinuous around $\tau^{\prime} \sim t^{\ell}. (\lim_{\tau^{\prime} \downarrow t^{\ell}} \mathbf{v}_{\perp}(\tau^{\prime}) = - \lim_{\tau^{\prime} \uparrow t^{\ell}} \mathbf{v}_{\perp}(\tau^{\prime}))$ However with crucial $\mathbf{x}_{\perp}(\tau^{\prime})-$multiplication we have $\mathbf{x}_{\perp}(t^{\ell}) \mathbf{v}_{\perp}(t^{\ell})=0$ and therefore
\begin{equation}\notag
\begin{split}
&\int_{\{(\tau^{\prime},\mathbf{x}, \mathbf{v}) \in \text{supp}(\phi)\}} \partial_{\mathbf{e}} [ \mathbf{x}_{\perp}(\tau^{\prime};t,\mathbf{x},\mathbf{v}) \mathbf{v}_{\perp}(\tau^{\prime};t,\mathbf{x},\mathbf{v}) ]
\phi(\tau^{\prime},\mathbf{x}, \mathbf{v}) \mathrm{d} \mathbf{x} \mathrm{d}\mathbf{v} \mathrm{d}\tau^{\prime}\\
=& \int_{\tau^{\prime} < t^{\ell}} + \int_{\tau^{\prime} \geq t^{\ell}}\\
=& \int_\mathcal{M} \Big(\lim_{\tau^{\prime} \uparrow t^{\ell} } [\mathbf{x}_{\perp}(\tau^{\prime})  \mathbf{v}_{\perp}(\tau^{\prime})]  - \lim_{\tau^{\prime} \downarrow t^{\ell}} [\mathbf{x}_{\perp}(\tau^{\prime}) \mathbf{v}_{\perp} (\tau^{\prime}) ] \Big) \phi(\tau^{\prime}, \mathbf{x}, \mathbf{v})\{ \mathbf{e}\cdot n_{\mathcal{M}}\}  \mathrm{d}\mathbf{x} \mathrm{d}\mathbf{v}\\
&- \int_{\{\tau^{\prime} \neq t^{\ell}(t,\mathbf{x}, \mathbf{v})\}} [\mathbf{x}_{\perp}(\tau^{\prime})  \mathbf{v}_{\perp}(\tau^{\prime})] \partial_{\mathbf{e}} \phi(\tau^{\prime}, \mathbf{x}, \mathbf{v}) \mathrm{d}\tau^{\prime} \mathrm{d}\mathbf{v} \mathrm{d}\mathbf{x} \\
=& - \int_{\{\tau^{\prime} \neq t^{\ell}(t,\mathbf{x}, \mathbf{v})\}} [\mathbf{x}_{\perp}(\tau^{\prime};t,\mathbf{x},\mathbf{v})  \mathbf{v}_{\perp}(\tau^{\prime};t,\mathbf{x},\mathbf{v})] \partial_{\mathbf{e}} \phi(\tau^{\prime}, \mathbf{x}, \mathbf{v}) \mathrm{d}\tau^{\prime} \mathrm{d}\mathbf{v} \mathrm{d}\mathbf{x}.
\end{split}
\end{equation}

We apply (\ref{piece_d}) to (\ref{xv_mildform})
\begin{equation}\label{Dode}
\begin{split}
&\partial\mathbf{x}_{\parallel}(s^{i+1}) = \partial \mathbf{x}_{\parallel}(s^{i}) - \int^{s^{i}}_{s^{i+1}} \partial \mathbf{v}_{\parallel}(\tau ) \mathrm{d}\tau ,\\
&\partial \mathbf{v}_{\parallel}(s^{i+1}) = \partial E(s^{i+1}) \mathbf{x}_{\perp}(s^{i+1}) + E(s^{i+1}) \partial \mathbf{x}_{\perp}(s^{i+1}) + \partial \mathbf{v}_{\parallel}(s^{i }) - \partial [E(\mathbf{x}_{\perp}, \mathbf{x}_{\parallel}, \mathbf{v}_{\parallel}) \mathbf{x}_{\perp}](s^{i+1})\\
& \ \ + \int^{s^{i}}_{s^{i+1}}  \partial \mathbf{v}_{\perp}(\tau) \partial_{\mathbf{x}_{\perp}}E(\tau)  \mathbf{x}_{\perp}(\tau)+ \partial \mathbf{v}_{\parallel}(\tau)\cdot \nabla_{\mathbf{x}_{\parallel}}E(\tau)   \mathbf{x}_{\perp}(\tau) \mathrm{d}\tau  \\
& \ \ + \int^{ s^{i}}_{s^{i+1}}
\Big\{  \Big[ \partial \mathbf{x}_{\perp}(\tau ) \partial_{\mathbf{x}_{\perp}}E(\tau) + \partial \mathbf{x}_{\parallel}(\tau ) \cdot \nabla_{\mathbf{x}_{\parallel}}E(\tau) + \partial \mathbf{v}_{\parallel}(\tau) \cdot \nabla_{\mathbf{v}_{\parallel}} E(\tau) \Big]\mathbf{v}_{\perp}(\tau) \\
& \ \ \ \ \  + E(\tau) \partial \mathbf{v}_{\perp}(\tau)  + \partial \mathbf{x}_{\perp}(\tau) \partial_{\mathbf{x}_{\perp}} D(\tau) + \partial \mathbf{x}_{\parallel}(\tau) \cdot \nabla_{\mathbf{x}_{\parallel}} D(\tau)  + \partial \mathbf{v}_{\parallel}(\tau) \nabla_{\mathbf{v}_{\parallel}} D(\tau)    \Big\}
\cdot \nabla_{\mathbf{v}_{\parallel}} E(\tau) \mathbf{x}_{\perp}(\tau) \mathrm{d}\tau  \\
& \ \ + \int^{ s^{i}}_{s^{i+1}}
\Big\{  \mathbf{v}_{\perp}(\tau) [\partial \mathbf{x}_{\perp}(\tau), \partial \mathbf{x}_{\parallel}(\tau), \partial \mathbf{v}_{\parallel}(\tau)] \cdot \nabla \partial_{\mathbf{x}_{\perp}} E(\tau)
+ \mathbf{v}_{\parallel}(\tau) \cdot  [\partial \mathbf{x}_{\perp}(\tau), \partial \mathbf{x}_{\parallel}(\tau), \partial \mathbf{v}_{\parallel}(\tau)] \cdot \nabla \nabla_{\mathbf{x}_{\parallel}} E(\tau) \\
& \ \ \ \ \ \ \ \ \ \ \ \
+ F_{\parallel}(\tau) \cdot  [\partial \mathbf{x}_{\perp}(\tau), \partial \mathbf{x}_{\parallel}(\tau), \partial \mathbf{v}_{\parallel}(\tau)] \cdot \nabla \nabla_{\mathbf{v}_{\parallel}} E(\tau)
\Big\}\mathbf{x}_{\perp}(\tau)  \mathrm{d}\tau  \\
& \ \ + \int^{ s^{i}}_{s^{i+1}} \big\{
\mathbf{v}_{\perp}(\tau) \partial_{\mathbf{x}_{\perp}} E(\tau) + \mathbf{v}_{\parallel}(\tau) \cdot \nabla_{\mathbf{x}_{\parallel}} E(\tau) + F_{\parallel}(\tau) \cdot \nabla_{\mathbf{v}_{\parallel}} E(\tau)
\big\} \partial \mathbf{x}_{\perp}(\tau) \mathrm{d}\tau\\
& \ \ - \int^{ s^{i}}_{s^{i+1}}  \big[  \partial \mathbf{x}_{\perp}(\tau), \partial \mathbf{x}_{\parallel}(\tau), \partial \mathbf{v}_{\parallel}(\tau) \big] \cdot \nabla D(\tau) \mathrm{d}\tau.
\end{split}
\end{equation}
Now we use (\ref{middle}) to control $[\partial \mathbf{x}_{\perp}, \partial \mathbf{v}_{\perp}]$. Notice that we cannot directly use (\ref{middle}) since now we fix the chart for whole $i-$th intermediate group but the estimate (\ref{middle}) is for the moving frame. (For clarity, we write the index for the chart for this part.) Note the time of bounces within the $i-$th intermediate group ($|t^{\ell_{i-1}}-t^{\ell_{i}}||v| \sim L_{\xi}$) are
\[
t^{\ell_{i}+1} < s^{i+1} < t^{\ell_{i}} <  t^{\ell_{i}-1} < \cdots\cdots <t^{\ell_{i-1}+2}  < t^{\ell_{i-1}+1} < s^{i} < t^{\ell_{i-1}}.
\]

Now we apply (\ref{chart_changing}) and (\ref{middle}) to bound, for $\tau \in (s^{i+1},s^{i})$ and $\ell \in \{ \ell_{i}, \ell_{i}-1, \cdots, \ell_{i-1}+2, \ell_{i-1} +1 , \ell_{i-1}\}$
\begin{equation}\label{ell_1}
\begin{split}
&\frac{\partial (  \mathbf{x}_{\perp_{\ell}}(\tau ), \mathbf{x}_{\parallel_{\ell}}( \tau), \mathbf{v}_{\perp_{\ell}} ( \tau ), \mathbf{v}_{\parallel_{\ell}} ( \tau )   )}{\partial    (    \mathbf{x}_{\perp_{1}}(s^{1}),   \mathbf{x}_{\parallel_{1}}(s^{1})  , \mathbf{v}_{\perp_{1}}(s^{1}),   \mathbf{v}_{\parallel_{1}}(s^{1})   )}\\
=&\frac{\partial (  \mathbf{x}_{\perp_{\ell}}(\tau), \mathbf{x}_{\parallel_{\ell}}( \tau ), \mathbf{v}_{\perp_{\ell}} ( \tau ), \mathbf{v}_{\parallel_{\ell}} ( \tau )   )}{\partial (  \mathbf{x}_{\perp_{\ell_{i}}}(\tau ), \mathbf{x}_{\parallel_{\ell_{i}}}( \tau ), \mathbf{v}_{\perp_{\ell_{i}}} ( \tau ), \mathbf{v}_{\parallel_{\ell_{i}}} ( \tau )   )}
\frac{  \partial (   \mathbf{x}_{\perp_{\ell_{i}}}(\tau ), \mathbf{x}_{\parallel_{\ell_{i}}}( \tau ), \mathbf{v}_{\perp_{\ell_{i}}} ( \tau ), \mathbf{v}_{\parallel_{\ell_{i}}} ( \tau )   )  }{\partial    (    \mathbf{x}_{\perp_{1}}(s^{1}),   \mathbf{x}_{\parallel_{1}}(s^{1})  , \mathbf{v}_{\perp_{1}}(s^{1}),   \mathbf{v}_{\parallel_{1}}(s^{1})   )}\\
\lesssim&  \tiny{ e^{C|t-s||v|} \left\{ \mathbf{Id}_{6,6} + O_{\xi}(|\mathbf{p}^{\ell} - \mathbf{p}^{\ell_{i}}|)
\left[\begin{array}{ccc|ccc} 
0 & 0 & 0 & & & \\
0 & 1 & 1 & & \mathbf{0}_{3,3} & \\
0 & 1 & 1 &  & & \\ \hline
0 & 0 & 0 & 0 & 0 & 0 \\
0 & |v| & |v| & 0 & 1 & 1 \\
0 & |v| & |v| & 0 & 1 & 1
\end{array}
\right]
 \right\}  
  \left[\begin{array}{cccccc}
 \frac{|v|}{|\mathbf{v}_{\perp}^{1}|}  &  \frac{|v|}{|\mathbf{v}_{\perp}^{1}|}  &  \frac{|v|}{|\mathbf{v}_{\perp}^{1}|} & \frac{1}{|v|} & \frac{1}{|v|} & \frac{1}{|v|}\\
  \frac{|v|^{2}}{|\mathbf{v}_{\perp}^{1}|^{2}} &  \frac{|v|^{2}}{|\mathbf{v}_{\perp}^{1}|^{2}} &  \frac{|v|^{2}}{|\mathbf{v}_{\perp}^{1}|^{2}} & 
  \frac{1}{|\mathbf{v}_{\perp}^{1}|} & \frac{1}{|v|} & \frac{1}{|v|}\\
   \frac{|v|^{2}}{|\mathbf{v}_{\perp}^{1}|^{2}} &  \frac{|v|^{2}}{|\mathbf{v}_{\perp}^{1}|^{2}} &  \frac{|v|^{2}}{|\mathbf{v}_{\perp}^{1}|^{2}} & 
    \frac{1}{|\mathbf{v}_{\perp}^{1}|} & \frac{1}{|v|} & \frac{1}{|v|}\\
    \frac{|v|^{3}}{|\mathbf{v}_{\perp}^{1}|^{2}} &  \frac{|v|^{3}}{|\mathbf{v}_{\perp}^{1}|^{2}} &  \frac{|v|^{3}}{|\mathbf{v}_{\perp}^{1}|^{2}}
    & \frac{|v|}{|\mathbf{v}_{\perp}^{1}|} & O_{\xi}(1) & O_{\xi}(1)\\
     \frac{|v|^{3}}{|\mathbf{v}_{\perp}^{1}|^{2}}&  \frac{|v|^{3}}{|\mathbf{v}_{\perp}^{1}|^{2}} &  \frac{|v|^{3}}{|\mathbf{v}_{\perp}^{1}|^{2}}
   & \frac{|v|}{|\mathbf{v}_{\perp}^{1}|} & O_{\xi}(1) & O_{\xi}(1)\\    
      \frac{|v|^{3}}{|\mathbf{v}_{\perp}^{1}|^{2}}&  \frac{|v|^{3}}{|\mathbf{v}_{\perp}^{1}|^{2}} & \frac{|v|^{3}}{|\mathbf{v}_{\perp}^{1}|^{2}}
        & \frac{|v|}{|\mathbf{v}_{\perp}^{1}|} & O_{\xi}(1) & O_{\xi}(1)\\
\end{array}\right] }  \\
\lesssim &e^{C|t-s||v|}  
 \left[\begin{array}{cccccc}
 \frac{|v|}{|\mathbf{v}_{\perp}^{1}|}  &  \frac{|v|}{|\mathbf{v}_{\perp}^{1}|}  &  \frac{|v|}{|\mathbf{v}_{\perp}^{1}|} & \frac{1}{|v|} & \frac{1}{|v|} & \frac{1}{|v|}\\
  \frac{|v|^{2}}{|\mathbf{v}_{\perp}^{1}|^{2}} &  \frac{|v|^{2}}{|\mathbf{v}_{\perp}^{1}|^{2}} &  \frac{|v|^{2}}{|\mathbf{v}_{\perp}^{1}|^{2}} & 
  \frac{1}{|\mathbf{v}_{\perp}^{1}|} & \frac{1}{|v|} & \frac{1}{|v|}\\
   \frac{|v|^{2}}{|\mathbf{v}_{\perp}^{1}|^{2}} &  \frac{|v|^{2}}{|\mathbf{v}_{\perp}^{1}|^{2}} &  \frac{|v|^{2}}{|\mathbf{v}_{\perp}^{1}|^{2}} & 
    \frac{1}{|\mathbf{v}_{\perp}^{1}|} & \frac{1}{|v|} & \frac{1}{|v|}\\
    \frac{|v|^{3}}{|\mathbf{v}_{\perp}^{1}|^{2}} &  \frac{|v|^{3}}{|\mathbf{v}_{\perp}^{1}|^{2}} &  \frac{|v|^{3}}{|\mathbf{v}_{\perp}^{1}|^{2}}
    & \frac{|v|}{|\mathbf{v}_{\perp}^{1}|} & O_{\xi}(1) & O_{\xi}(1)\\
     \frac{|v|^{3}}{|\mathbf{v}_{\perp}^{1}|^{2}}&  \frac{|v|^{3}}{|\mathbf{v}_{\perp}^{1}|^{2}} &  \frac{|v|^{3}}{|\mathbf{v}_{\perp}^{1}|^{2}}
   & \frac{|v|}{|\mathbf{v}_{\perp}^{1}|} & O_{\xi}(1) & O_{\xi}(1)\\    
      \frac{|v|^{3}}{|\mathbf{v}_{\perp}^{1}|^{2}}&  \frac{|v|^{3}}{|\mathbf{v}_{\perp}^{1}|^{2}} & \frac{|v|^{3}}{|\mathbf{v}_{\perp}^{1}|^{2}}
        & \frac{|v|}{|\mathbf{v}_{\perp}^{1}|} & O_{\xi}(1) & O_{\xi}(1)\\
\end{array}\right]
,
\end{split}
\end{equation}
where we have used $|\mathbf{p}^{\ell} - \mathbf{p}^{\ell_{i}}| \lesssim 1.$


Together with (\ref{middle}), we have (for clarity, we write estimates for each derivative $\partial= [\partial_{\mathbf{x}_{\perp}}, \partial_{\mathbf{x}_{\parallel}}, \partial_{\mathbf{v}_{\perp}}, \partial_{\mathbf{v}_{\parallel}}]$):
\begin{equation}\notag
\begin{split}
| \partial_{\mathbf{x}_{\perp}  } {  \mathbf{x}_{\parallel}}  (s^{i+1})|& \lesssim_{\xi} \int^{s^{i}}_{s^{i+1}} |\partial_{\mathbf{x}_{\perp}} \mathbf{v}_{\parallel}(\tau )| \mathrm{d}\tau ,\\
|\partial_{\mathbf{x}_{\perp} } \mathbf{v}_{\parallel}(s^{i+1})| 
& \lesssim_{\xi}
|v||\mathbf{x}_{\perp}(\tau_{i})| |\partial_{\mathbf{x}_{\perp}} \mathbf{x}_{\parallel}(s^{i+1})| +  |\mathbf{x}_{\perp}(s^{i+1})| |\partial_{\mathbf{x}_{\perp}} \mathbf{v}_{\parallel}(s^{i+1})|
+   e^{C|v||t-s|} \Big[ \frac{|v|^{2}}{|\mathbf{v}_{\perp}^{\ell_{i-1}}|} 
+   |v|\Big]    \\
& \ \ \ + \int^{s^{i}}_{s^{i+1}}\Big\{     |v|^{2} |\partial_{\mathbf{x}_{\perp}} \mathbf{x}_{\parallel}(\tau )| 
  + |v| |\partial_{\mathbf{x}_{\perp}} \mathbf{v}_{\parallel}(\tau )|
  + e^{C|v||t-s|}\Big[  \frac{|v|^{4} |\mathbf{x}_{\perp}(\tau )|    }{|\mathbf{v}_{\perp}^{\ell_{i-1}}|^{2}}   +  \frac{|v|^{3}}{|\mathbf{v}_{\perp}^{\ell_{i-1}}| }\Big]
  \Big\}
\mathrm{d}\tau .
\end{split}
\end{equation}

We use (\ref{40}), (\ref{41}) and (\ref{xi_xperp}) and the condition $|\xi(X_{\mathbf{cl}}(\tau ))|<\delta$ for all $\tau  \in [s,t]$ to have, $\mathbf{x}_{\perp}(\tau ;t,\mathbf{x}, \mathbf{v})\lesssim_{\xi} |\xi(X_{\mathbf{cl}}(\tau;t,x, v) )|$ for all $\tau \in [s,t],$ and therefore
\begin{equation}\notag
\begin{split}
|v|^{2}|\mathbf{x}_{\perp}(\tau ;t,\mathbf{x}, \mathbf{v})| \ &\lesssim_{\xi} 2 \xi(X_{\mathbf{cl}}(\tau ;t,x,v)) \{ V_{\mathbf{cl}}(\tau ;t,x,v) \cdot \nabla^{2}\xi(X_{\mathbf{cl}}(\tau ;t,x,v))\cdot V_{\mathbf{cl}}(\tau ;t,x,v)  \}  \\ &\lesssim_{\xi} \ \alpha(\tau ;t,\mathbf{x}, \mathbf{v}) \ \lesssim_{\xi} \
e^{C|v||t-\tau |} |\mathbf{v}_{\perp}^{1}|^{2},
\end{split}
\end{equation}
where we used the convexity of $\xi$ in (\ref{convex}) and the Velocity lemma(Lemma \ref{velocity_lemma}).

Hence we rewrite as, for $0< \delta\ll 1,$
 \begin{equation}\label{ODE_group}
\begin{split}
|\frac{\partial  \mathbf{x}_{\parallel}(s^{i+1})}{ \partial {\mathbf{x}_{\perp}} }| & \lesssim_{\xi} |\frac{\partial  \mathbf{x}_{\parallel}(s^{i})}{ \partial {\mathbf{x}_{\perp}} }| +\int^{s^{i}}_{s^{i+1}} | \frac{ \partial \mathbf{v}_{\parallel}(\tau^{\prime})}{\partial{\mathbf{x}_{\perp}}}  | \mathrm{d}\tau^{\prime}
  ,\\
 |  \frac{ \partial   \mathbf{v}_{\parallel}(s^{i+1})}{  \partial {\mathbf{x}_{\perp}}} | - \delta   |v|
 | \frac{  \partial  \mathbf{x}_{\parallel}(s^{i+1}) }{   \partial {\mathbf{x}_{\perp}}  }|
& \lesssim_{\xi,\delta}  |  \frac{ \partial   \mathbf{v}_{\parallel}(s^{i})}{  \partial {\mathbf{x}_{\perp}}} | +
 \int^{s^{i}}_{s^{i+1}} \Big\{ |v|^{2} |   \frac{ \partial \mathbf{x}_{\parallel}(\tau^{\prime}) }{   \partial {\mathbf{x}_{\perp}}}   |+  |v|
  | \frac{ \partial  \mathbf{v}_{\parallel}(\tau^{\prime})}{  \partial {\mathbf{x}_{\perp}}}|  \Big\}
 \mathrm{d}\tau^{\prime} \\
 & \ \ \ +    |v|  e^{C|v||t-s |} (1+ \frac{|v|}{|\mathbf{v}_{\perp}^{1}|} ).\\
 \end{split}
 \end{equation}
 Similarly, from (\ref{Dode}) and (\ref{ell_1})
 \begin{equation}\notag
 \begin{split}
|   \frac{\partial  \mathbf{x}_{\parallel}(s^{i+1})}{    \partial {\mathbf{x}_{\parallel}}}  |
& \lesssim_{\xi } \  |   \frac{\partial  \mathbf{x}_{\parallel}(s^{i})}{    \partial {\mathbf{x}_{\parallel}}}  |
+ \int^{s^{i}}_{s^{i+1}} | \frac{\partial  \mathbf{v}_{\parallel}(\tau^{\prime})}{\partial {\mathbf{x}_{\parallel}} } | \mathrm{d}\tau^{\prime},\\
| \frac{\partial \mathbf{v}_{\parallel}(s^{i+1})}{  \partial {\mathbf{x}_{\parallel}} } |
 -\delta|v| |   \frac{\partial  \mathbf{x}_{\parallel}(s^{i+1})}{  \partial {\mathbf{x}_{\parallel}}} |
 & \lesssim_{\xi,\delta} |\frac{\partial \mathbf{v}_{\parallel}(s^{i}) }{\partial \mathbf{x}_{\parallel}}| +
|v| \Big(1+ \frac{|v|}{|\mathbf{v}_{\perp}^{1}|}\Big) e^{C|v||t-s |}  + \int^{s^{i}}_{s^{i+1}} \Big\{ |v|^{2} | \frac{\partial    \mathbf{x}_{\parallel}(\tau^{\prime})}{\partial {\mathbf{x}_{\parallel}}}  |
+ |v|  | \frac{\partial    \mathbf{v}_{\parallel}(\tau^{\prime})}{     \partial {\mathbf{x}_{\parallel}}} | \Big\}\mathrm{d}\tau^{\prime},\\
| \frac{\partial  \mathbf{x}_{\parallel}(s^{i+1})}{  \partial {\mathbf{v}_{\perp}}   }|
& \lesssim_{\xi } \
 | \frac{\partial   \mathbf{x}_{\parallel}(s^{i})   }{  \partial \mathbf{v}_{\perp}  } |+
 \int^{s^{i}}_{s^{i+1}}
| \frac{\partial  \mathbf{v}_{\parallel}(\tau^{\prime}) }{     \partial {\mathbf{v}_{\perp}} }| \mathrm{d}\tau^{\prime},\\
| \frac{\partial  \mathbf{v}_{\parallel}(s^{i+1})}{ \partial {\mathbf{v}_{\perp}} } | -\delta|v| |  \frac{\partial   \mathbf{x}_{\parallel}(s^{i+1})}{  \partial {\mathbf{v}_{\perp}}}|&
\lesssim_{\xi,\delta}
|\frac{\partial \mathbf{v}_{\parallel}(s^{i})}{\partial \mathbf{v}_{\perp}}|
+ e^{C|v||t-s |}
+  \int^{s^{i}}_{s^{i+1}}  \Big\{ |v|^{2}| \frac{\partial \mathbf{x}_{\parallel}(\tau^{\prime})}{ \partial {\mathbf{v}_{\perp}} } |  + |v|  |\frac{\partial \mathbf{v}_{\parallel}(\tau^{\prime})}{\partial {\mathbf{v}_{\perp}}} | \Big\} \mathrm{d}\tau^{\prime}
 ,\\
 | \frac{\partial  \mathbf{x}_{\parallel}(s^{i+1})}{  \partial  {\mathbf{v}_{\parallel}}}  | & \lesssim_{\xi } \
 |\frac{\partial \mathbf{x}_{\parallel}(s^{i})}{\partial \mathbf{v}_{\parallel}}|
+ \int^{s^{i}}_{s^{i+1}} |  \frac{\partial  \mathbf{v}_{\parallel}(\tau^{\prime})}{\partial {\mathbf{v}_{\parallel} } } | \mathrm{d}\tau^{\prime},\\
|\frac{\partial  \mathbf{v}_{\parallel}(s^{i+1})}{\partial {\mathbf{v}_{\parallel} }}|
-\delta|v| | \frac{\partial  \mathbf{x}_{\parallel}(s^{i+1})}{ \partial {\mathbf{v}_ {\parallel}}}  |& \lesssim_{\xi,\delta}
|\frac{\partial \mathbf{v}_{\parallel}(s^{i})  }{\partial  \mathbf{v}_{\parallel}}|
+ e^{C|v||t-s |} + \int^{s^{i}}_{s^{i+1}}\Big\{ |v|^{2} | \frac{\partial \mathbf{x}_{\parallel}(\tau^{\prime})}{ \partial {\mathbf{v}_{\parallel} }}  |   + |v|
| \frac{\partial  \mathbf{v}_{\parallel}(\tau^{\prime})}{\partial {\mathbf{v}_{\parallel} } } | \Big\}\mathrm{d}\tau^{\prime}.
\end{split}
\end{equation}

Now we claim a version of Gronwall's estimate: If $a(\tau),b(\tau),f(\tau),g(\tau)\geq 0$ for all $0 \leq \tau \leq t$, and satisfy, for $0<\delta\ll 1$
\begin{equation}\notag
\begin{split}
\left[\begin{array}{cc} 1 & 0 \\ -\delta |v| & 1\end{array} \right]
\left[\begin{array}{c} a( {\tau})  \\ b( {\tau})   \end{array} \right] \lesssim_{\xi}
\left[\begin{array}{cc} 0 & 1 \\ |v|^{2} & |v| \end{array}\right]
 \left[\begin{array}{c} \int^{t}_{\tau } a(\tau^{\prime}) \mathrm{d}\tau^{\prime} \\ \int^{t}_{\tau} b(\tau^{\prime}) \mathrm{d}\tau^{\prime} \end{array}\right] + \left[\begin{array}{c} g(t-\tau) \\  h(t-\tau) \end{array} \right]
\end{split}
\end{equation}
then
\begin{equation}\label{matrix_gronwall}
\begin{split}
\left[\begin{array}{c} a(\tau) \\ b(\tau) \end{array}\right] \lesssim_{\delta, \xi} \int^{t}_{\tau} e^{|v|(\tau^{\prime}-\tau)} \big\{g(\tau^{\prime}) + \frac{h(\tau^{\prime})}{|v|}\big\} \mathrm{d}\tau^{\prime} \left[ \begin{array}{c} |v| \\ |v|^{2} \end{array} \right] +  \left[\begin{array}{c} g(t-\tau) \\ \delta |v| g(t-\tau) + h(t-\tau) \end{array} \right].
\end{split}
\end{equation}

Define $\tilde{a}(\tau):= a(t-\tau), \tilde{b}(\tau):= b(t-\tau)$ and $A( {\tau}):= \int^{\tau}_{0} \tilde{a}(\tau^{\prime}) \mathrm{d} {\tau^{\prime}},
 B( {\tau}):= \int^{\tau}_{0} \tilde{b}(\tau^{\prime}) \mathrm{d} {\tau^{\prime}}.$
 Then
\begin{equation}\notag
\begin{split}
\frac{d}{d {\tau}} \left[\begin{array}{cc} 1 & 0 \\ -\delta |v| & 1 \end{array} \right] \left[\begin{array}{c} A( {\tau}) \\ B( {\tau}) \end{array} \right]& =
 \left[\begin{array}{cc} 1 & 0 \\ -\delta |v| & 1 \end{array} \right] \left[\begin{array}{c} \tilde{a}({\tau}) \\ b(\tilde{\tau}) \end{array} \right]
\lesssim_{\xi}\left[\begin{array}{cc} 0 & 1 \\ |v|^{2} & |v|\end{array}\right]
\left[\begin{array}{c} A(\tau) \\ B(\tau)\end{array}\right]
+ \left[\begin{array}{c} \tilde{g}({\tau}) \\    \tilde{h}({\tau}) \end{array}\right]\\
&\lesssim_{\xi } \left[\begin{array}{cc} \delta |v| & 1 \\ (1+\delta)  |v|^{2} & |v|\end{array}\right]   \left[\begin{array}{cc} 1 & 0 \\ -\delta |v| & 1 \end{array} \right]
\left[\begin{array}{c} A(\tau) \\ B(\tau)\end{array}\right]
+  \left[\begin{array}{c} \tilde{g}({\tau}) \\    \tilde{h}({\tau}) \end{array}\right].\\
\end{split}
\end{equation}
Using $\left[\begin{array}{cc} 0 & 1 \\ |v|^{2} & |v|\end{array} \right] \left[\begin{array}{cc} 1 & 0 \\ \delta|v| & 1 \end{array} \right] = \left[\begin{array}{cc} \delta|v| & 1 \\  (1+\delta)|v|^{2} & |v|   \end{array} \right]$ and the notation $\left[\begin{array}{c}  \tilde{{A}} ({\tau}) \\ \tilde{ {B}}({\tau})\end{array} \right] := \left[\begin{array}{cc} 1 & 0 \\ -\delta |v| & 1 \end{array} \right] \left[\begin{array}{c}A(\tau)\\ B(\tau)\end{array} \right],$
\begin{equation}\notag
\begin{split}
\frac{d}{d {\tau}} \left[\begin{array}{c} \tilde{A}(\tau) \\ \tilde{B}(\tau)\end{array} \right] & \lesssim_{\xi}
\left[\begin{array}{cc}  \delta|v| & 1 \\(1+\delta) |v|^{2} & |v| \end{array} \right]\left[\begin{array}{c} \tilde{A}(\tau) \\ \tilde{B}(\tau) \end{array} \right] +  \left[\begin{array}{c} \tilde{g}(\tau) \\    \tilde{h}(\tau)\end{array} \right],
\end{split}
\end{equation}
 We diagonalize $\left[\begin{array}{cc}  \delta|v| & 1 \\ (1+\delta) |v|^{2} & |v| \end{array} \right]$ as
\begin{equation*}
\begin{split}
 =& \left[\begin{array}{cc} 1 & 1 \\ \frac{(1-\delta)+ \sqrt{(1+\delta)^{2} + 4}}{2}|v| & \frac{(1-\delta)- \sqrt{(1+\delta)^{2} + 4}}{2}|v| \end{array} \right]
 \left[\begin{array}{cc}
 \frac{(1+\delta) + \sqrt{(1+\delta)^{2} +4}}{2}|v| & 0 \\ 0 &  \frac{(1+\delta) -\sqrt{(1+\delta)^{2} +4}}{2}|v| \end{array}\right]\\
&\times  \left[\begin{array}{cc}
 \frac{-(1-\delta) + \sqrt{(1+\delta)^{2} +4}}{2 \sqrt{(1+\delta)^{2} +4}} & \frac{1}{ \sqrt{(1+\delta)^{2} +4}}\frac{1}{|v|} \\
 \frac{(1-\delta) + \sqrt{(1+\delta)^{2} +4}}{2\sqrt{(1+\delta)^{2} +4}} & \frac{-1}{\sqrt{(1+\delta)^{2} +4}}\frac{1}{|v|}
  \end{array} \right].
\end{split}
\end{equation*}
Denote $ \left[\begin{array}{c}  { {\mathcal{A}}}( {\tau}) \\  {\mathcal{B}}( {\tau}) \end{array} \right]:=\left[\begin{array}{cc}
 \frac{-(1-\delta) + \sqrt{(1+\delta)^{2} +4}}{2 \sqrt{(1+\delta)^{2} +4}} & \frac{1}{ \sqrt{(1+\delta)^{2} +4}}\frac{1}{|v|} \\
 \frac{(1-\delta) + \sqrt{(1+\delta)^{2} +4}}{2\sqrt{(1+\delta)^{2} +4}} & \frac{-1}{\sqrt{(1+\delta)^{2} +4}}\frac{1}{|v|}
  \end{array} \right]\left[\begin{array}{c} \tilde{ {A}}( {\tau}) \\ \tilde{ {B}}( {\tau})\end{array} \right] $ to rewrite
\begin{equation*}
\begin{split}
\frac{d}{d {\tau}}
\left[\begin{array}{c}  { {\mathcal{A}}}( {\tau}) \\  {\mathcal{B}} ({\tau}) \end{array} \right]
  & \lesssim_{\xi}
\left[\begin{array}{cc}
 \frac{(1+\delta) + \sqrt{(1+\delta)^{2} +4}}{2}|v| & 0 \\ 0 &  \frac{(1+\delta) -\sqrt{(1+\delta)^{2} +4}}{2}|v| \end{array}\right]
 \left[\begin{array}{c}  { {\mathcal{A}}}( {\tau}) \\  {\mathcal{B}}( {\tau}) \end{array} \right]  \\
 & \ \ +
 \left[\begin{array}{cc}
 \frac{-(1-\delta) + \sqrt{(1+\delta)^{2} +4}}{2 \sqrt{(1+\delta)^{2} +4}} & \frac{1}{ \sqrt{(1+\delta)^{2} +4}}\frac{1}{|v|} \\
 \frac{(1-\delta) + \sqrt{(1+\delta)^{2} +4}}{2\sqrt{(1+\delta)^{2} +4}} & \frac{-1}{\sqrt{(1+\delta)^{2} +4}}\frac{1}{|v|}
  \end{array} \right]
  \left[\begin{array}{c} \tilde{g}(\tau) \\    \tilde{h}(\tau) \end{array} \right].
\end{split}
\end{equation*}
Therefore
\begin{equation}\notag
\begin{split}
&\left[\begin{array}{c}  { {\mathcal{A}}}( {\tau}) \\  {\mathcal{B}}( {\tau}) \end{array} \right]  \leq \left[\begin{array}{c}
 e^{C_{\xi,\delta}   \frac{(1+\delta) + \sqrt{(1+\delta)^{2} +4}}{2}|v| {\tau}}
  { {\mathcal{A}}}(0) \\   e^{C_{\xi,\delta}   \frac{(1+\delta) -\sqrt{(1+\delta)^{2} +4}}{2}|v|  {\tau} }  {\mathcal{B}}(0) \end{array} \right]\\
  & \ \ \  +  \int ^{\tau}_{0}   \left[\begin{array}{cc}
 e^{\frac{(1+\delta) + \sqrt{(1+\delta)^{2} +4}}{2}|v| (\tau-\tau^{\prime})}& 0 \\ 0 &  e^{\frac{(1+\delta) -\sqrt{(1+\delta)^{2} +4}}{2}|v| (\tau-\tau^{\prime})}\end{array}\right]  \left[\begin{array}{cc}
 \frac{-(1-\delta) + \sqrt{(1+\delta)^{2} +4}}{2 \sqrt{(1+\delta)^{2} +4}} & \frac{1}{ \sqrt{(1+\delta)^{2} +4}}\frac{1}{|v|} \\
 \frac{(1-\delta) + \sqrt{(1+\delta)^{2} +4}}{2\sqrt{(1+\delta)^{2} +4}} & \frac{-1}{\sqrt{(1+\delta)^{2} +4}}\frac{1}{|v|}
  \end{array} \right]
  \left[\begin{array}{c} \tilde{g}(\tau^{\prime}) \\    \tilde{h}(\tau^{\prime}) \end{array} \right] \mathrm{d}\tau^{\prime},
  \end{split}
  \end{equation}
  and then
  \begin{equation}\notag
  \begin{split}
   \left[\begin{array}{c}
A(\tau)\\ B(\tau)
\end{array} \right] &= \left[\begin{array}{cc} 1& 0 \\ \delta |v| & 1 \end{array} \right]  \left[\begin{array}{cc} 1 & 1 \\ \frac{(1-\delta)+ \sqrt{(1+\delta)^{2} + 4}}{2}|v| & \frac{(1-\delta)- \sqrt{(1+\delta)^{2} + 4}}{2}|v| \end{array} \right] \left[\begin{array}{c}  { {\mathcal{A}}}( {\tau}) \\  {\mathcal{B}}({\tau}) \end{array} \right]\\
&   \lesssim_{\xi,\delta}\int_{0}^{\tau} e^{C_{\xi,\delta}|v|(\tau-\tau^{\prime}) } \big\{ \tilde{g}(\tau^{\prime}) + \frac{\tilde{h}(\tau^{\prime})}{|v|}\big\} \mathrm{d}\tau^{\prime}  \left[\begin{array}{c} 1 \\ |v| \end{array} \right] .
   \end{split}
\end{equation}
Together with the first inequality(the condition of the claim)
\begin{equation}\notag
\begin{split}
\left[\begin{array}{c} a(\tau) \\ b(\tau) \end{array} \right]
& \lesssim_{\xi} \left[\begin{array}{cc} 0 & 1 \\ |v|^{2} & (1+\delta)|v| \end{array} \right]
\left[\begin{array}{c} A(t-\tau) \\ B(t-\tau) \end{array} \right] + \left[\begin{array}{c} g(t-\tau) \\ \delta |v| g(t-\tau) + h(t-\tau) \end{array} \right]\\
& \lesssim_{\xi,\delta} \int^{t}_{\tau} e^{|v|(\tau^{\prime}-\tau)} \big\{g(\tau^{\prime}) + \frac{h(\tau^{\prime})}{|v|}\big\} \mathrm{d}\tau^{\prime} \left[ \begin{array}{c} |v| \\ |v|^{2} \end{array} \right] +  \left[\begin{array}{c} g(t-\tau) \\ \delta |v| g(t-\tau) + h(t-\tau) \end{array} \right]\\
& \lesssim_{\xi,\delta} e^{C|v||t-\tau|} \left[\begin{array}{cc} 1 & \frac{1}{|v|} \\ |v| & 1 \end{array} \right]
 \left[\begin{array}{c} \sup |g| \\ \sup |h| \end{array} \right],
\end{split}
\end{equation}
and this proves the claim (\ref{matrix_gronwall}). We apply (\ref{matrix_gronwall}) to (\ref{ODE_group}) and we prove the claim (\ref{ODE_onegroup}).

\vspace{8pt}

\textit{Step 10. ODE method within the time scale $|t-s|\sim 1$: Refinement of the estimate (\ref{middle}) }

We claim that
\begin{equation}\label{ODE_whole}
\left[\begin{array}{cccc}
| \frac{\partial  \mathbf{x}_{\parallel_{\tilde{\ell}} }(s)}{\partial \mathbf{x}_{\perp_{1}}}  | & | \frac{\partial  \mathbf{x}_{\parallel_{\tilde{\ell}} }(s)}{\partial \mathbf{x}_{\parallel_{1}}}  |  &  | \frac{\partial  \mathbf{x}_{\parallel_{\tilde{\ell}} }(s)}{\partial \mathbf{v}_{\perp_{1}}}  |  &  | \frac{\partial  \mathbf{x}_{\parallel_{\tilde{\ell}} }(s)}{\partial \mathbf{v}_{\parallel_{1}}}  | \\
| \frac{\partial  \mathbf{v}_{\parallel_{\tilde{\ell}} }(s)}{\partial \mathbf{x}_{\perp_{1}}}  | & | \frac{\partial  \mathbf{v}_{\parallel_{\tilde{\ell}} }(s)}{\partial \mathbf{x}_{\parallel_{1}}}  |  &  | \frac{\partial  \mathbf{v}_{\parallel_{\tilde{\ell}} }(s)}{\partial \mathbf{v}_{\perp_{1}}}  |  &  | \frac{\partial  \mathbf{v}_{\parallel_{\tilde{\ell}} }(s)}{\partial \mathbf{v}_{\parallel_{1}}}  |
\end{array} \right]
\lesssim C^{C|v||t-s|}
\left[\begin{array}{cccc}
  \frac{|v|}{|\mathbf{v}_{\perp}^{1}|}  &   \frac{|v|}{|\mathbf{v}_{\perp}^{1}|}  & \frac{1}{|v|} & \frac{1}{|v|} \\
  \frac{|v|^{2}}{|\mathbf{v}_{\perp}^{1}|} &  \frac{|v|^{2}}{|\mathbf{v}_{\perp}^{1}| } & 1 & 1
\end{array}\right],
\end{equation}
where $\tilde{\ell}= [\frac{|t-s||v|}{L_{\xi}}].$

\noindent\textit{Proof of the claim (\ref{ODE_whole}).}
By the chain rule
\[
 \left[\begin{array}{cc} D_{\mathbf{x}} \mathbf{x}_{\parallel_{i}} & D_{\mathbf{v}} \mathbf{x}_{\parallel_{i}} \\
D_{\mathbf{x}} \mathbf{v}_{\parallel_{i}} & D_{\mathbf{v}} \mathbf{v}_{\parallel_{i}}
 \end{array}\right]
=
\frac{\partial( \mathbf{x}_{\parallel_{i}}, \mathbf{v}_{\parallel_{i}} )}{\partial(  \mathbf{x}_{\parallel_{i-1}}, \mathbf{v}_{\parallel_{i-1}}  )}
 \left[\begin{array}{cc} D_{\mathbf{x}} \mathbf{x}_{\parallel_{i-1}} & D_{\mathbf{v}} \mathbf{x}_{\parallel_{i-1}} \\
D_{\mathbf{x}} \mathbf{v}_{\parallel_{i-1}} & D_{\mathbf{v}} \mathbf{v}_{\parallel_{i-1}}
 \end{array}\right].
 \]
Note, from (\ref{chart_changing})
\[ 
  \frac{\partial( \mathbf{x}_{\parallel_{i}}, \mathbf{v}_{\parallel_{i}} )}{\partial(  \mathbf{x}_{\parallel_{i-1}}, \mathbf{v}_{\parallel_{i-1}}  )}
\leq C  \left[\begin{array}{c|c}1 & \mathbf{0} \\ \hline |v|  & 1 \end{array} \right]
 \leq C \left[\begin{array}{c|c}1 & \frac{1}{|v|} \\ \hline |v|  & 1 \end{array} \right]:= C\mathbf{B}.
\]
Denote
\[
\mathbf{D}_{i}(s) = \left[\begin{array}{cc} |D_{\mathbf{x}} \mathbf{x}_{\parallel_{i}  }(s) |& |D_{\mathbf{v}} \mathbf{x}_{\parallel_{i}  }(s) |  \\
|D_{\mathbf{x}} \mathbf{v}_{\parallel_{i}  }(s) |& |D_{\mathbf{v}} \mathbf{v}_{\parallel_{i}  }(s)   |
 \end{array} \right], \ \ \ \mathbf{G}:= \left[\begin{array}{cc} \mathbf{0} & \mathbf{0} \\  \frac{|v|^{2}}{|\mathbf{v}_{\perp}^{1}|}  & 1  \end{array} \right].
\]
Note that from (\ref{ODE_onegroup})
\[
\mathbf{D}_{i}(s^{i+1}) \leq C \mathbf{B} \mathbf{D}_{i}(s^{i}) + C \mathbf{B}\mathbf{G}.
\]
Therefore, by induction,
\begin{equation}\notag
\begin{split}
\mathbf{D}_{  [\frac{|t-s||v|}{L_{\xi}}]} ( s ) &\leq C \mathbf{D}_{  [\frac{|t-s||v|}{L_{\xi}}]} ( \tau_{  [\frac{|t-s||v|}{L_{\xi}}]} ) + C  \mathbf{B}\mathbf{G}\\
& \leq C^{2} \mathbf{B} \mathbf{D}_{  [\frac{|t-s||v|}{L_{\xi}}]-1} ( \tau_{  [\frac{|t-s||v|}{L_{\xi}}]} )+ C  \mathbf{B}\mathbf{G}\\
&\leq  C^{2} \mathbf{B} \mathbf{D}_{  [\frac{|t-s||v|}{L_{\xi}}]-1} ( \tau_{  [\frac{|t-s||v|}{L_{\xi}}]-1} ) + C^{3} \mathbf{B}  \mathbf{G}+ C  \mathbf{B}\mathbf{G}\\
&\leq C^{3} \mathbf{B}^{2} \mathbf{D}_{  [\frac{|t-s||v|}{L_{\xi}}]-1} ( \tau_{  [\frac{|t-s||v|}{L_{\xi}}]-2} ) +  \{ C^{2} \mathbf{B} + \mathbf{Id}  \}C  \mathbf{B}\mathbf{G}\\
&\leq C^{4} \mathbf{B}^{3}  \mathbf{D}_{  [\frac{|t-s||v|}{L_{\xi}}]-2} ( \tau_{  [\frac{|t-s||v|}{L_{\xi}}]-2} )
+     \{ C^{3} \mathbf{B}^{2}+C^{2} \mathbf{B} + \mathbf{Id}  \}C  \mathbf{B}\mathbf{G}\\
& \ \ \ \vdots\\
&\lesssim  C^{C|t-s||v|} \mathbf{B}^{C[|t-s||v|]} \mathbf{D}_{1}(\tau_{1}) + \sum_{i=0}^{C[|t-s||v|]} C^{i+1} \mathbf{B}^{i}  \mathbf{B} \mathbf{G}.
\end{split}
\end{equation}
But direct computation yields $\mathbf{B}^{j} \leq 2^{j} \mathbf{B}$. Therefore
\begin{equation}\notag
\begin{split}
\mathbf{D}_{[\frac{|t-s||v|}{L_{\xi}}]}(s) \lesssim C^{C|t-s||v|}  \mathbf{B}\big\{\mathbf{D}_{1}(\tau_{1}) +  \mathbf{B}\mathbf{G}\big\}.
\end{split}
\end{equation}
From (\ref{Dxv_free}) we have $\mathbf{D}_{1}(\tau_{1})\lesssim \left[\begin{array}{c|c} 1 & \frac{1}{|v|} \\ \hline |v| & 1   \end{array} \right]$ and we conclude our claim (\ref{ODE_whole}).

With these estimates, we refine (\ref{middle}) to give a final estimate for the case that $|\xi(X_{\mathbf{cl}}(\tau;t,x,v))|<\delta$ for all $\tau \in [s,t]$:
  \begin{equation}\label{middle_refined}
 \begin{split}
  \frac{\partial (s^{\ell_{*}}, \mathbf{x}_{\perp}(s^{\ell_{*}}), \mathbf{x}_{\parallel}(s^{\ell_{*}}), \mathbf{v}_{\perp}(s^{\ell_{*}}), \mathbf{v}_{\parallel}(s^{\ell_{*}})  )}{\partial (s^{1}, \mathbf{x}_{\perp}(s^{1}), \mathbf{x}_{\parallel}(s^{1}), \mathbf{v}_{\perp}(s^{1}), \mathbf{v}_{\parallel}(s^{1}) )}
 \lesssim C^{C|v|(t-s)} \left[\begin{array}{c|cc|ccc}
 0 & 0 & \mathbf{0}_{1,2} & 0 & \mathbf{0}_{1,2} &\\ \hline
  |\mathbf{v}_{\perp}^{1}| &   \frac{|v|}{|\mathbf{v}_{\perp}^{1}|} &   \frac{|v|}{|\mathbf{v}_{\perp}^{1}|}  & \frac{1}{|v|}& \frac{1}{|v|}&\\
 |v| &    \frac{|v|}{|\mathbf{v}_{\perp}^{1}|}   &   \frac{|v|}{|\mathbf{v}_{\perp}^{1}|}   & |t-s|&  |t-s|\\ \hline
|v|^{2}& \frac{|v|^{3}}{|\mathbf{v}_{\perp}^{1}|^{2}}   &    \frac{|v|^{3}}{|\mathbf{v}_{\perp}^{1}|^{2}}   &  \frac{|v|}{|\mathbf{v}_{\perp}^{1}|} & O_{\xi}(1) \\
|v|^{2}&   \frac{|v|^{2}}{|\mathbf{v}_{\perp}^{1}|}&  \frac{|v|^{2}}{|\mathbf{v}_{\perp}^{1}|}& O_{\xi}(1)&O_{\xi}(1) &
   \end{array}\right],
 \end{split}
 \end{equation}
and from (\ref{ss1}) and (\ref{s1t})
\begin{equation}\label{final_Dxv_grazing}
\begin{split}
 &\frac{\partial (s,X_{\mathbf{cl}}(s;t,x,v),V_{\mathbf{cl}}(s;t,x,v))}{\partial (t,x,v)} \\
 &  \lesssim   C^{C|v|(t-s)}  \ \frac{\partial (s,X_{\mathbf{cl}}(s), V_{\mathbf{cl}}(s))}{\partial (s^{\ell_{*}}, \mathbf{X}_{\mathbf{cl}}(s^{\ell_{*}}), \mathbf{V}_{\mathbf{cl}}(s^{\ell_{*}}))}
\left[\begin{array}{c|c|c} 0 & \mathbf{0}_{1,3} & \mathbf{0}_{1,3} \\ \hline
 |v| &   \frac{|v|}{|\mathbf{v}_{\perp}^{1}|} & \frac{1}{|v|} \\ \hline
 |v|^{2} &   \frac{|v|^{3}}{|\mathbf{v}_{\perp}^{1}|^{2}} &  \frac{|v|}{|\mathbf{v}_{\perp}^{1}|} \end{array}\right]  \frac{   {\partial (s^{1}, \mathbf{x}_{\perp}(s^{1}), \mathbf{x}_{\parallel}(s^{1}), \mathbf{v}_{\perp}(s^{1}), \mathbf{v}_{\parallel}(s^{1}) )}       }{\partial (t,x,v)}\\
&  \lesssim C^{C|v|(t-s)} \left[\begin{array}{ccc} 0 & \mathbf{0}_{1,3} & \mathbf{0}_{1,3} \\ |v| & 1 & |s^{\ell_{*}}-s| \\ \mathbf{0}_{3,1} & |v| & 1 \end{array}\right]
\left[\begin{array}{c|c|c} 0 & \mathbf{0}_{1,3} & \mathbf{0}_{1,3} \\ \hline
 |v| &   \frac{|v|}{|\mathbf{v}_{\perp}^{1}|} & \frac{1}{|v|} \\ \hline
 |v|^{2} &   \frac{|v|^{3}}{|\mathbf{v}_{\perp}^{1}|^{2}} &   \frac{|v|}{|\mathbf{v}_{\perp}^{1}|} \end{array}\right]   
 \left[\begin{array}{ccc} 1 & \mathbf{0}_{1,3} & \mathbf{0}_{1,3} \\ \mathbf{0}_{3,1} & 1 & |t-s^{1}| \\ \mathbf{0}_{3,1} & |v| &  1  \end{array}\right]\\
 & \lesssim  C^{C|v|(t-s)}
  \left[\begin{array}{c|c|c} 0 & \mathbf{0}_{1,3} & \mathbf{0}_{1,3} \\ \hline
  |v| &   \frac{|v|}{|\mathbf{v}_{\perp}^{1}|} & \frac{1}{|v|} \\ \hline
  |v|^{2} &  \frac{|v|^{3}}{|\mathbf{v}_{\perp}^{1}|^{2}} &   \frac{|v|}{|\mathbf{v}_{\perp}^{1}|}\end{array} \right].
\end{split}
\end{equation}
Finally from (\ref{final_Dxv_nongrazing}) and (\ref{final_Dxv_grazing}) we conclude, for all $\tau \in [s,t]$
\begin{equation}\notag
\begin{split}
\frac{\partial(  X_{\mathbf{cl}}(s;t,x,v), V_{\mathbf{cl}}(s;t,x,v))}{\partial (t,x,v)}
\leq
Ce^{C|v|(t-s)}    \left[\begin{array}{c|c|c}
  |v| &   \frac{|v|}{|\mathbf{v}_{\perp}^{1}|} & \frac{1}{|v|} \\ \hline
  |v|^{2} &   \frac{|v|^{3}}{|\mathbf{v}_{\perp}^{1}|^{2}} &   \frac{|v|}{|\mathbf{v}_{\perp}^{1}|}\end{array} \right]_{6\times 7}
\end{split}
\end{equation}
From the Velocity lemma(Lemma \ref{velocity_lemma}),
\begin{equation}\notag
\begin{split}
|\mathbf{v}_{\perp}^{1}| &= |v^{1}\cdot [-n(x^{1})] |= |V_{\mathbf{cl}}(t^{1};t,x,v) \cdot n(X_{\mathbf{cl}}(t^{1};t,x,v) ) |\\
&= \sqrt{\alpha(X_{\mathbf{cl}}(t^{1} ),V_{\mathbf{cl}}(t^{1} ))} \geq e^{\mathcal{C}|v| |t-t^{1}|}\alpha(x,v) \gtrsim \alpha(x,v),
\end{split}
\end{equation}
and this completes the proof for the case (\ref{typeII}).

\end{proof}

\begin{proof}[\textbf{Proof of Theorem \ref{main_specular}}]
We use the approximation sequence (\ref{fm}) with (\ref{specular_BC_m}). Due to Lemma \ref{local_existence} we have $\sup_{m} \sup_{0 \leq  t \leq T}||  e^{\theta|v|^{2}}f^{m}(t) ||_{\infty}\lesssim_{\xi, T}  P(   || 
e^{\theta^{\prime}|v|^{2}}  f_{0}||_{\infty})$.

Now we claim that the distributional derivatives coincide with the piecewise derivatives. This is due to Proposition \ref{theo:trace} and Proposition \ref{inflowW1p} together with an invariant property of $\Gamma(f,f)= \Gamma_{\mathrm{gain}}(f,f) - \nu(\sqrt{\mu}f)f:$ \textit{Assume $f^{m}(v)=f^{m-1}(\mathcal{O}v)$ holds for some orthonormal matrix. Then }
\begin{equation}\label{invariant}
\Gamma(f^{m}, f^{m})(v) = \Gamma(f^{m-1}, f^{m-1})(\mathcal{O}v).
\end{equation}
We apply Proposition 1 to have 
\begin{equation}\notag
\begin{split}
&f^{m}(t,x,v) \\
&= e^{- \int_{0}^{t} \sum_{\ell=0}^{\ell_{*} (0)} \mathbf{1}_{[t^{\ell+1} , t^{\ell} ) } (s) \nu(\sqrt{\mu} f^{m-\ell}) (s) \mathrm{d}s  } f_{0}(X_{\mathbf{cl}}(0), V_{\mathbf{cl}}(0))\\
& \ \ +  \int_{0}^{t}\sum_{\ell=0 }^{\ell_{*}(0)}  \mathbf{1}_{[t^{\ell+1},t^{\ell})}(s) e^{-\int^{t}_{s}
\sum_{j=0}^{\ell_{*}(s)} \mathbf{1}_{[{t^{j+1}},{t^{j}})}(\tau) \nu (F^{m-j})(\tau )%
\mathrm{d}\tau }     \Gamma_{\mathrm{gain}} (f^{m-\ell}, f^{m-\ell}) (s,X_{\mathbf{cl}}(s),V_{\mathbf{cl}}(s)) \mathrm{d}%
s.
\end{split}
\end{equation}

 Now we consider the spatial and velocity derivatives. In the sense of distributions, we have for $\partial_{\mathbf{e}} = [\partial_{x}, \partial_{v}]$ with $\mathbf{e} \in \{x,v\},$
\begin{equation}\label{deriv_spec}
\begin{split}
  \partial_{\mathbf{e}} f^{m}(t,x,v) = \text{I} _{\mathbf{e}}+ \text{II}_{\mathbf{e}} + \text{III}_{\mathbf{e}}. \end{split}
\end{equation}
Here
\[
\text{I}_{\mathbf{e}} \ = \ e^{- \int^{t}_{0}   \sum_{\ell=0}^{\ell_{*}(0)} \mathbf{1}_{[{t^{\ell+1}} ,{t^{\ell}})}(s) \nu(F^{m-\ell} )(s) \mathrm{d}s} \
 \partial_{\mathbf{e}}[ X_{\mathbf{cl}}(0),  V_{\mathbf{cl}}(0) ] \cdot \nabla_{x,v}f_{0}(X_{\mathbf{cl}}(0),V_{\mathbf{cl}}(0)),
\]
and
   \begin{equation}\notag
   \begin{split}
\text{II}_{\mathbf{e}}&= \int_{0}^{t}\sum_{\ell=0 }^{\ell_{*}(0)}  \mathbf{1}_{[t^{\ell+1},t^{\ell})}(s) e^{-\int^{t}_{s}
\sum_{j=0}^{\ell_{*}(s)} \mathbf{1}_{[{t^{j+1}},{t^{j}})}(\tau) \nu (F^{m-j})(\tau )%
\mathrm{d}\tau }     \partial_{\mathbf{e}} \big[    \Gamma_{\mathrm{gain}} (f^{m-\ell}, f^{m-\ell}) (s,X_{\mathbf{cl}}(s),V_{\mathbf{cl}}(s))\big] \mathrm{d}%
s\\
& -\int_{0}^{t}\sum_{\ell =0}^{\ell_{*}(0)}  \mathbf{1}_{[t^{\ell+1},t^{\ell})}(s) e^{-\int^{t}_{s}
\sum_{j} \mathbf{1}_{[{t^{j+1}},{t^{j}})}(\tau) \nu (F^{m-j})(\tau )%
\mathrm{d}\tau }  \int^{t}_{s}
\sum_{j=0}^{\ell_{*}(s)} \mathbf{1}_{[{t^{j+1}},{t^{j}})}(\tau) \partial_{\mathbf{e}} [ \nu (F^{m-j})(\tau, X_{\mathbf{cl}}(\tau), V_{\mathbf{cl}}(\tau) )]%
\mathrm{d}\tau   \\  
& \ \ \ \ \  \times  \Gamma_{\text{gain}}(f^{m-\ell}, f^{m-\ell}) (s,X_{\mathbf{cl}}(s),V_{\mathbf{cl}}(s)) \mathrm{d}%
s  \\
&   - e^{- \int^{t}_{0}   \sum_{\ell=0 }^{\ell_{*}(0)} \mathbf{1}_{[{t^{\ell+1}} ,{t^{\ell}})}(s) \nu(F^{m-\ell} )(s) \mathrm{d}s} \
   f_{0}(X_{\mathbf{cl}}(0),V_{\mathbf{cl}}(0)){  \int^{t}_{0}   \sum_{\ell=0 }^{\ell_{*}(0)} \mathbf{1}_{[{t^{\ell+1}} ,{t^{\ell}})}(s)  \partial_{\mathbf{e}} \big[  \nu(F^{m-\ell} )(s, X_{\mathbf{cl}}(s), V_{\mathbf{cl}}(s))  \big] \mathrm{d}s},
\end{split}
\end{equation}
and
\begin{equation}\notag
\begin{split}
\mathbf{III}_{\mathbf{e}}=& \sum_{\ell=0}^{\ell_{*}(0)} \big[ -\partial_{\mathbf{e}} t^{\ell} 
\lim_{s \uparrow t^{\ell} } \nu(\sqrt{\mu} f^{m-\ell})(s,X_{\mathbf{cl}}(s), V_{\mathbf{cl}}(s)) +  \partial_{\mathbf{e}} t^{\ell+1} \lim_{s\downarrow t^{\ell+1}} \mu(\sqrt{\mu} f^{m-\ell})(s,X_{\mathbf{cl}}(s), V_{\mathbf{cl}}(s))  \big]\\
& \ \ \ \ \ \  \times e^{-\int_{0}^{t}  \sum_{\ell=0}^{\ell_{*} (0)} \mathbf{1}_{[t^{\ell+1}, t^{\ell} )}(s) \nu(\sqrt{\mu} f^{m-\ell})(s)   }  \\
+& \sum_{\ell=0}^{\ell_{*}(0)}  \Big[\lim_{s\uparrow t^{\ell}} e^{- \int^{t}_{s} \sum_{j} \mathbf{1}_{[t^{j+1}, t^{j} ) } (\tau) \nu(F^{m-j})(\tau) \mathrm{d}\tau} \Gamma_{\mathrm{gain}}(f^{m-\ell}, f^{m-\ell}) (s,X_{\mathbf{cl}}(s), V_{\mathbf{cl}}(s))\\
& \ \ \ \ \ \ \ \ - \lim_{s\downarrow t^{\ell+1}} e^{- \int^{t}_{s} \sum_{j} \mathbf{1}_{[t^{j+1}, t^{j} ) } (\tau) \nu(F^{m-j})(\tau) \mathrm{d}\tau} \Gamma_{\mathrm{gain}}(f^{m-\ell}, f^{m-\ell}) (s,X_{\mathbf{cl}}(s), V_{\mathbf{cl}}(s))\\
+& \int_{0}^{t} \sum_{\ell} \mathbf{1}_{[t^{\ell+1}, t^{\ell} )}(s) \sum_{j=0}^{\ell_{*}(s)}\big[ -\lim_{\tau \downarrow t^{j}} \nu(F^{m-j})(\tau, X_{\mathbf{cl}}(\tau), V_{\mathbf{cl}}(\tau))  + \lim_{\tau \uparrow t^{j+1}} \nu(F^{m-j})(\tau, X_{\mathbf{cl}}(\tau), V_{\mathbf{cl}}(\tau))\big] \\
&  \ \ \ \ \  \ \ \ \times e^{-\int^{t}_{s} \sum_{j=0}^{\ell_{*} (s)}  \mathbf{1}_{[t^{j+1} , t^{j} ) }(\tau) \nu(F^{m-j} ) (\tau) \mathrm{d}\tau }\Gamma_{\mathrm{gain}}(f^{m-\ell} , f^{m-\ell})(s, X_{\mathbf{cl}}(s), V_{\mathbf{cl}}(s)).
\end{split}
\end{equation}
For $\mathbf{III}_{\mathbf{e}}$ we rearrange the summation and use (\ref{v_perp}) and apply (\ref{invariant})
\begin{equation}\notag
\begin{split}
\mathbf{III}_{\mathbf{e}} = & \sum_{\ell=0}^{\ell_{*}(0)} 
\Big[  - \nu(\sqrt{\mu} f^{m-\ell})(t^{\ell}, x^{\ell}, v^{\ell} ) + \nu(\sqrt{\mu} f^{m-\ell+1})( t^{\ell}, x^{\ell}, R_{x^{\ell}}v^{\ell})  \Big] \partial_{\mathbf{e}} t^{\ell}e^{-\int_{0}^{t}  \sum_{\ell=0}^{\ell_{*} (0)} \mathbf{1}_{[t^{\ell+1}, t^{\ell} )}(s) \nu(\sqrt{\mu} f^{m-\ell})(s)   } \\
+ &\sum_{\ell=0}^{\ell_{*}(0)} e^{- \int^{t}_{t^{\ell}} \sum_{j} \mathbf{1}_{[t^{j+1},t^{j}) }(\tau) \nu(\sqrt{\mu} f^{m-j}) (\tau) \mathrm{d}\tau  } 
\Big[ \Gamma_{\mathrm{gain}}( f^{m-\ell}, f^{m-\ell} )(t^{\ell}, x^{\ell}, v^{\ell}) -\Gamma_{\mathrm{gain}}( f^{m-\ell+1}, f^{m-\ell+1} )(t^{\ell}, x^{\ell}, R_{x^{\ell}}v^{\ell}) \Big]\\
+& \int_{0}^{t} \sum_{\ell} \mathbf{1}_{[t^{\ell+1}, t^{\ell} )}(s) ^{-\int^{t}_{s} \sum_{j=0}^{\ell_{*} (s)}  \mathbf{1}_{[t^{j+1} , t^{j} ) }(\tau) \nu(F^{m-j} ) (\tau) \mathrm{d}\tau }\Gamma_{\mathrm{gain}}(f^{m-\ell} , f^{m-\ell})(s, X_{\mathbf{cl}}(s), V_{\mathbf{cl}}(s))\\
&\times\sum_{\ell=0}^{\ell_{*}(s)} 
\Big[  - \nu(\sqrt{\mu} f^{m-\ell})(t^{\ell}, x^{\ell}, v^{\ell} ) + \nu(\sqrt{\mu} f^{m-\ell+1})( t^{\ell}, x^{\ell}, R_{x^{\ell}}v^{\ell})  \Big]\\
=&0.
\end{split}
\end{equation}

\noindent\textit{Proof of (\ref{invariant}).} The proof is due to the change of variables 
\[
\tilde{u} = \mathcal{O} u, \ \ \tilde{\omega} = \mathcal{O} \omega, \ \ \ 
\mathrm{d}\tilde{u} = \mathrm{d}u , \ \ \mathrm{d}\tilde{\omega} = \mathrm{d}\omega.
\]
Note
\begin{equation}\notag
\begin{split}
&\Gamma(f^{m}, f^{m})(v)\\
=&\int_{\mathbb{R}^{3}} \int_{\mathbb{S}^{2}} |v-u|^{\kappa}q_{0}(\frac{v-u}{|v-u|}\cdot \omega) \sqrt{\mu(u)} \Big\{ f^{m}(u-[(u-v)\cdot \omega]\omega) f^{m}(v+ [(u-v)\cdot \omega]\omega) -f^{m} (u)f^{m} (v) \Big\}
\mathrm{d}\omega\mathrm{d}u\\
=&\int_{\mathbb{R}^{3}} \int_{\mathbb{S}^{2}}
 |\mathcal{O}v-\mathcal{O}u|^{\kappa}q_{0}(\frac{\mathcal{O}v-\mathcal{O}u}{|\mathcal{O}v-\mathcal{O}u|}\cdot \mathcal{O}\omega) \sqrt{\mu(\mathcal{O}u)} \\
&  \ \ \ \times \Big\{ f^{m-1}
(\mathcal{O} u -[( \mathcal{O}u- \mathcal{O}v)\cdot \mathcal{O} \omega] \mathcal{O}{\omega}) f^{m-1}(\mathcal{O}v+ [(\mathcal{O}u-\mathcal{O}v)\cdot \mathcal{O}\omega]\mathcal{O}\omega) -f^{m-1} ( \mathcal{O}u)f^{m-1} (\mathcal{O}v) \Big\}
\mathrm{d}\omega\mathrm{d}u\\
=& \int_{\mathbb{R}^{3}} \int_{\mathbb{S}^{2}} | \mathcal{O}v- \tilde{u}|^{\kappa}q_{0}(\frac{ \mathcal{O}v-  \tilde{u}}{|\mathcal{O}v-  \tilde{u}|}\cdot  \tilde{\omega}) \sqrt{\mu( \tilde{u})}\\
& \ \ \ \times \Big\{
f^{m-1} (  \tilde{u} - [ (   \tilde{u} - \mathcal{O}  {v}  ) \cdot  \tilde{\omega}  ]  \tilde{\omega} )
f^{m-1} ( \mathcal{O}v + [(\tilde{u} - \mathcal{O}v)]\cdot \tilde{\omega}  ) \tilde{\omega}
-f^{m-1} (\tilde{u}) f^{m-1}(\mathcal{O}v)
\Big\} \mathrm{d}\tilde{\omega} \mathrm{d}\tilde{u}\\
= & \ \Gamma(f^{m-1}, f^{m-1})(  \mathcal{O} v ).
\end{split}
\end{equation}

This proves (\ref{invariant}). Especially we can apply (\ref{invariant}) for the specular reflection BC (\ref{specular_BC_m}) with $\mathcal{O}v= R_{x}v$ as well as the bounce-back reflection BC (\ref{bb_BC_m}) with $\mathcal{O}v=-v.$ 

Using Lemma \ref{lemma_operator} and (\ref{nabla_Gamma}),
\begin{equation}\notag
\begin{split}
\text{II}_{\mathbf{e}}& \lesssim \   P(|| e^{\theta|v|^{2}} f_{0} ||_{\infty})\int_{0}^{t}\sum_{\ell=0 }^{\ell_{*}(0)}  \mathbf{1}_{[t^{\ell+1},t^{\ell})}(s)    |\partial_{\mathbf{e}} X_{\mathbf{cl}}(s)| 
\int_{\mathbb{R}^{3}}  \frac{e^{-C_{\theta}  |V_{\mathbf{cl}}(s)-u|^{2}}}{|V_{\mathbf{cl}}(s)-u|^{2-\kappa}} | \nabla_{x} f^{m-\ell} (s,X_{\mathbf{cl}}(s),u  )| \mathrm{d}u \mathrm{d}s
\\
& \ \ + P(|| e^{\theta|v|^{2}} f_{0} ||_{\infty})\int_{0}^{t}\sum_{\ell=0 }^{\ell_{*}(0)}  \mathbf{1}_{[t^{\ell+1},t^{\ell})}(s)    |\partial_{\mathbf{e}} V_{\mathbf{cl}}(s)| 
\int_{\mathbb{R}^{3}}  \frac{e^{-C_{\theta}  |V_{\mathbf{cl}}(s)-u|^{2}}}{|V_{\mathbf{cl}}(s)-u|^{2-\kappa}} | \nabla_{v} f^{m-\ell} (s,X_{\mathbf{cl}}(s),u  )| \mathrm{d}u \mathrm{d}s
\\
& \ \ + tP(|| e^{\theta|v|^{2}} f_{0} ||_{\infty})  \langle v\rangle^{\kappa} e^{-\theta |v|^{2}} \sup_{0 \leq s\leq t} |\partial_{\mathbf{e}}V(s;t,x,v)|.
\end{split}
\end{equation}

We shall estimate the followings:
\[
e^{-\varpi \langle v\rangle t} \frac{ [\alpha(x,v)]^{\beta} }{\langle v\rangle^{b+1}} |\partial_{x} f(t,x,v)|, \ \ \ 
e^{-\varpi \langle v\rangle t} \frac{ |v|[\alpha(x,v)]^{\beta-\frac{1}{2}} }{\langle v\rangle^{b}} |\partial_{v} f(t,x,v)|.
\]
 
 From (\ref{lemma_Dxv}), the Velocity lemma (Lemma \ref{velocity_lemma}), Lemma \ref{local_existence}, and $F^{m}\geq 0$ for all $m,$ with $\varpi \gg 1$
\begin{equation}\notag
\begin{split}
&e^{-\varpi \langle v\rangle t}  \frac{1}{\langle v\rangle^{b+1}}
 [\alpha(x,v)]^{\beta} \ \text{I}_{\mathbf{x}}\\
& \lesssim_{\xi,t}   e^{-\varpi \langle v \rangle t } \frac{1}{\langle v\rangle^{b+1 }}
 [\alpha(X_{\mathbf{cl}}(0),V_{\mathbf{cl}}(0)) ]^{\beta}e^{ 2\mathcal{C}|v|t} \\
 &  \ \ \ \times \Big\{   \frac{|v|}{\sqrt{\alpha(x,v)}} |\partial_{x} f_{0}(X_{\mathbf{cl}}(0), V_{\mathbf{cl}}(0))| +  \frac{|v|^{3}}{\alpha(x,v)}  |\partial_{v} f_{0}(X_{\mathbf{cl}}(0), V_{\mathbf{cl}}(0))|  \Big\} \\
&\lesssim_{\xi,t}   
\big|\big| \  \frac{|v| }{\langle v\rangle^{b+1}}
 {\alpha}^{\beta-\frac{1}{2}} \partial_{x} f_{0} \  \big|\big|_{\infty} + \big|\big| \   \frac{   |v|^{3}}{\langle v\rangle^{b+1}}
 \alpha^{\beta-1} \partial_{v} f_{0}  \ \big|\big|_{\infty}\\
& \lesssim_{\xi,t} 
 \big|\big| \  \frac{ {\alpha}^{\beta-\frac{1}{2}}   }{\langle v\rangle^{b }}
\partial_{x} f_{0} \  \big|\big|_{\infty} + \big|\big| \   \frac{   |v|^{2}   \alpha^{\beta-1} }{\langle v\rangle^{b }}
 \partial_{v} f_{0}  \ \big|\big|_{\infty}
 ,
\end{split}
\end{equation}
and
\begin{equation}\notag
\begin{split}
& e^{-\varpi \langle v\rangle t}  \frac{|v|}{\langle v\rangle^{b}}
[\alpha(x,v)]^{\beta-\frac{1}{2}} \ \text{I}_{\mathbf{v}}\\
& \lesssim_{\xi,t}  \ e^{-\varpi \langle v \rangle t } \frac{ |v|}{\langle v\rangle^{b}}
 [\alpha(X_{\mathbf{cl}}(0),V_{\mathbf{cl}}(0))]^{\beta -\frac{1}{2} } e^{2 \mathcal{C}|v|t}  \\
 &  \ \ \ \times  \Big\{ \frac{1}{|v|} |\partial_{x} f_{0}(X_{\mathbf{cl}}(0), V_{\mathbf{cl}}(0))| +   \frac{|v| }{\sqrt{\alpha(x,v)}}  |\partial_{v} f_{0}(X_{\mathbf{cl}}(0), V_{\mathbf{cl}}(0))|  \Big\} \\
 & \lesssim_{\xi,t}\
\big|\big| \    \frac{\alpha^{\beta-\frac{1}{2} } }{\langle v\rangle^{b}}
 \partial_{x} f_{0} \  \big|\big|_{\infty} + \big|\big| \  \frac{|v|^{2} }{\langle v\rangle^{b}}
{\alpha}^{ \beta-1}  \partial_{v} f_{0}  \ \big|\big|_{\infty},
\end{split}
\end{equation}
where we have used $\alpha(x,v) \lesssim_{\xi} |v|^{2}$ and the choice of $\varpi \gg 1.$

From Lemma \ref{lemma_K}, Lemma \ref{lemma_operator}, and Lemma \ref{local_existence},  
\begin{equation}\notag
\begin{split}
 \text{II}_{\mathbf{e}} & \lesssim_{t}  P(|| e^{\theta|v|^{2}} f_{0} ||_{\infty} ) \int_{0}^{t} \mathrm{d}s \sum_{\ell=0}^{\ell_{*}(0)} \mathbf{1}_{[t^{\ell+1}, t^{\ell})}(s)    \int_{\mathbb{R}^{3}} \mathrm{d}u 
 \frac{e^{-C_{\theta} |u-V_{\mathbf{cl}}(s)|^{2} }}{|u-V_{\mathbf{cl}}(s)|^{2-\kappa}}
   \\
 &  \ \   \times
  \Big\{   |\partial_{\mathbf{e}} X_{\mathbf{cl}}(s)| |\partial_{x} f^{m-j}(s, X_{\mathbf{cl}}(s), u)| 
  + |\partial_{\mathbf{e}} V_{\mathbf{cl}}(s) |   \big(  1+  |\partial_{v} f^{m-j} (s, X_{\mathbf{cl}}(s), u)| \big)
  \Big\} .
\end{split}
\end{equation}
Now we use (\ref{lemma_Dxv}) to have 
\begin{equation}\notag
\begin{split}
& e^{-\varpi \langle v \rangle t}   \frac{ [\alpha(x,v)]^{\beta}}{\langle v\rangle^{b+1 }}
 \ \text{II}_{\mathbf{x}}  \lesssim_{t,\xi}   
P(||  e^{\theta|v|^{2}} f_{0} ||_{\infty} )  \\
& \ \ \ \times 
  \bigg\{ 
  \int_{0}^{t}  \int_{\mathbb{R}^{3}}  
  \frac{  e^{-C_{\theta}|V_{\mathbf{cl}}(s) -u|^{2}}}{ |u-V_{\mathbf{cl}}(s)|^{2-\kappa}}
 e^{-\varpi \langle v\rangle t }e^{\varpi \langle u\rangle s} e^{C|v||t-s|} \frac{|v| [\alpha(x,v)]^{\beta- \frac{1}{2}}}{ [\alpha(X_{\mathbf{cl}}(s),u)]^{\beta}} \frac{\langle u\rangle^{b+1} }{ \langle v\rangle^{b+1}}  \mathrm{d}u\mathrm{d}s
 \\
& \ \ \ \ \ \ \ \ \ \ \ \ \ \ \ \times  \sup_{m} \sup_{0 \leq s\leq t} \big|\big| e^{-\varpi \langle u \rangle s} \frac{[\alpha(X_{\mathbf{cl}} (s), u  )]^\beta}{ \langle u \rangle^{b+1}} \partial_{x} f^{m-j}(s,X_{\mathbf{cl}}(s),u)  \big|\big|_{\infty}  \\ 
& \ \ \ \ \ \ + \int_{0}^{t} \int_{\mathbb{R}^{3}}
  \frac{  e^{-C_{\theta}|V_{\mathbf{cl}}(s) -u|^{2}}}{ |u-V_{\mathbf{cl}}(s)|^{2-\kappa}} e^{- \varpi \langle v\rangle t} e^{\varpi \langle u\rangle s} e^{C|v||t-s|} 
\frac{ \langle u\rangle^{b}}{ \langle v\rangle^{b }}  \frac{|v|^{2} [\alpha(x,v)]^{\beta-1}}{ |u|[\alpha(X_{\mathbf{cl}}(s),u)]^{\beta- \frac{1}{2}}}\\
& \ \ \ \ \ \ \ \ \ \ \ \ \ \ \ \times  \sup_{m} \sup_{0 \leq s\leq t} \big|\big| e^{-\varpi \langle u \rangle s} \frac{|u|[\alpha(X_{\mathbf{cl}} (s), u  )]^{\beta-\frac{1}{2}}}{ \langle u \rangle^{b}} \partial_{v} f^{m-j}(s,X_{\mathbf{cl}}(s),u)  \big|\big|_{\infty}\bigg\}
.
\end{split}
\end{equation}
We first claim that 
\begin{equation}\label{exponent}
e^{-\varpi \langle v \rangle t}e^{\varpi \langle u \rangle s} e^{C|v|(t-s)} e^{-C^{\prime}|v-u|^{2} }\lesssim  e^{-\frac{\varpi  \langle v\rangle}{2} (t-s)}e^{C^{\prime\prime}  (s+s^{2})}   e^{- {C}^{\prime\prime}  |v-u|^{2}}.
\end{equation}
Using $\langle u \rangle   \leq  1+ |u| \leq 1+ |v|  + |u-v|  \leq 1+ \langle v\rangle  + |v-u|,$ we bound the first three exponents as
\begin{equation}\notag
\begin{split}
  -(\varpi-C) \langle v \rangle (t-s) - \varpi (\langle v\rangle - \langle u \rangle )s \leq -(\varpi-C)\langle v\rangle (t-s)+   \varpi |v-u|s + \varpi s.
\end{split}
\end{equation} 
Then we use a complete square trick, for $0 < \sigma \ll 1$
\begin{equation}\notag
\begin{split}
\varpi |v-u|s  
  = \frac{\sigma \varpi^{2}}{2} |v-u|^{2} + \frac{s^{2}}{2\sigma} - \frac{1}{2\sigma} \big[s- \sigma \varpi |v-u| \big]^{2} \leq \frac{\sigma \varpi^{2}}{2} |v-u|^{2} + \frac{s^{2}}{2\sigma} ,
\end{split}
\end{equation}
to bound the whole exponents of (\ref{exponent}) by
\begin{equation}\notag
\begin{split}
&-(\varpi -C) \langle v \rangle (t-s) + \varpi |v-u|s  -C^{\prime}|v-u|^{2} + \varpi s \\ 
&\leq -(\varpi-C) \langle v \rangle (t-s) -(C - \frac{\sigma \varpi^{2}}{2}) |v-u|^{2} + \frac{s^{2}}{2\sigma} +\varpi s  \\
& \leq - (\varpi-C)\langle  v\rangle (t-s) - C_{\sigma, \varpi}  |v-u|^{2} + C_{\sigma,\varpi}^{\prime} \big\{ s^{2} +s\big\}.
\end{split}
\end{equation}
Hence we prove the claim (\ref{exponent}) for $\varpi \gg 1.$

Now we use (\ref{exponent}) to bound 
\begin{equation}\label{AB}
\begin{split}
&e^{-\varpi \langle v \rangle t}  \frac{1}{\langle v\rangle^{ b+1}}
[\alpha(x,v)]^{\beta} \ \text{II}_{\mathbf{x}}\\
& \lesssim_{t,\xi} P( ||   e^{\theta|v|^{2}} f_{0} ||_{\infty} ) \times\\
& \ \times \bigg\{  \underbrace{  \int_{0}^{t} \int_{ \mathbb{R}^{3}}
 e^{-  \frac{\varpi \langle v \rangle}{2} (t-s) }  
\frac{ e^{-C_{\theta}^{\prime}|v-u|^{2}} }{|v-u|^{2-\kappa}}
 \frac{\langle u\rangle^{b+1}}{\langle v\rangle^{b+1}} \frac{ \langle v\rangle   [\alpha(x,v)]^{\beta-\frac{1}{2}}}{ [\alpha(X_{\mathbf{cl}}(s),u)]^{\beta}}
 \mathrm{d}u \mathrm{d}s }_{ \mathbf{(A)} }  \ 
\sup_{m}
 \sup_{0\leq s \leq t}  \big|\big| e^{-\varpi \langle v \rangle s}    \frac{\alpha^{\beta}}{\langle v\rangle^{b+1}}
 \partial_{x} f^{m}(s) \big|\big|_{\infty}\\
& \ \  + 
  \underbrace{  \int_{0}^{t} \int_{ \mathbb{R}^{3}}
 e^{-  \frac{\varpi \langle v \rangle}{2} (t-s) }   \frac{ e^{-C_{\theta}^{\prime}|v-u|^{2}} }{|v-u|^{2-\kappa}}\frac{\langle u\rangle^{b }}{\langle v\rangle^{b }} \frac{ |v|^{2}  [\alpha(x,v)]^{\beta-1}}{ |u|[\alpha(X_{\mathbf{cl}}(s),u)]^{\beta-\frac{1}{2}}}
 \mathrm{d}u \mathrm{d}s }_{ \mathbf{(B)} }    \ 
\sup_{m}
 \sup_{0\leq s \leq t} \big|\big| e^{-\varpi \langle v \rangle s}    \frac{|v|\alpha^{\beta-\frac{1}{2}}}{\langle v\rangle^{b }}
 \partial_{v} f^{m}(s) \big|\big|_{\infty} \bigg\}
 .\end{split}
\end{equation}
For $\mathbf{(A)}$ we use (\ref{specular_nonlocal}) with $Z=\langle v\rangle\big[ \alpha(x,v) \big]^{\beta-\frac{1}{2}}$ and $l= \frac{\varpi}{2}$ and $r=b+1$. For $\mathbf{(B)}$ we use (\ref{specular_nonlocal_u}) with $\beta \mapsto\beta-\frac{1}{2}$ and $Z=\langle v\rangle\big[ \alpha(x,v) \big]^{\beta- {1}}$ and $l= \frac{\varpi}{2}$ and $r=b$. Then 
\[
\mathbf{(A)} , \  \mathbf{(B)} \   \ll \ 1.
\] 

Similarly, but with different weight $e^{-\varpi \langle v\rangle t}  \frac{|v|}{\langle v\rangle^{b}}
 [\alpha(x,v)]^{\beta-\frac{1}{2}}$, we use (\ref{lemma_Dxv}) to have
\begin{equation}\notag
\begin{split}
 &e^{-\varpi \langle v \rangle t}   \frac{|v|}{\langle v\rangle^{b}}
[\alpha(x,v)]^{\beta-\frac{1}{2}} \ \text{II}_{\mathbf{v}} \\
&\lesssim_{t,\xi}   
P(|| e^{\theta|v|^{2}} f_{0} ||_{\infty} )  \times\\
& \ \  \times 
  \bigg\{ 
  \int_{0}^{t}  \int_{\mathbb{R}^{3}} 
  \frac{e^{-C|V_{\mathbf{cl}}(s) -u|^{2}} }{ |u-V_{\mathbf{cl}}(s)|^{2-\kappa}}e^{-\varpi \langle v\rangle t }e^{\varpi \langle u\rangle s} e^{C|v||t-s|} \frac{ \langle v\rangle  [\alpha(x,v)]^{\beta-\frac{1}{2}} }{   [\alpha(X_{\mathbf{cl}}(s),u)]^{\beta}} \frac{\langle u\rangle^{b+1} }{ \langle v\rangle^{b+1}}  \mathrm{d}u\mathrm{d}s
 \\
& \ \ \ \ \ \ \ \ \ \ \ \ \ \ \ \times  \sup_{m} \sup_{0 \leq s\leq t} \big|\big| e^{-\varpi \langle u \rangle s} \frac{[\alpha(X_{\mathbf{cl}} (s), u  )]^\beta}{ \langle u \rangle^{b+1}} \partial_{x} f^{m }(s,X_{\mathbf{cl}}(s),u)  \big|\big|_{\infty}  \\
& \ \ \ \ \ \ 
+ \int_{0}^{t} \int_{\mathbb{R}^{3}}
 \frac{e^{-C|V_{\mathbf{cl}} (s)-u|^{2}} }{|u-V_{\mathbf{cl}}(s)|^{2-\kappa} }  e^{- \varpi \langle v\rangle t} e^{\varpi \langle u\rangle s} e^{C|v||t-s|} 
\frac{ \langle u\rangle^{b}}{ \langle v\rangle^{b }}  \frac{|v|^{2} [\alpha(x,v)]^{\beta-1}}{ |u|[\alpha(X_{\mathbf{cl}}(s),u)]^{\beta- \frac{1}{2}}}\\
& \ \ \ \ \ \ \ \ \ \ \ \ \ \ \ \times  \sup_{m} \sup_{0 \leq s\leq t} \big|\big| e^{-\varpi \langle u \rangle s} \frac{|u|[\alpha(X_{\mathbf{cl}} (s), u  )]^{\beta-\frac{1}{2}}}{ \langle u \rangle^{b}} \partial_{v} f^{m }(s,X_{\mathbf{cl}}(s),u)  \big|\big|_{\infty}\bigg\}
.
\end{split}
\end{equation}
Again we use (\ref{exponent}) and (\ref{specular_nonlocal}) and (\ref{specular_nonlocal_u}) exactly as (\ref{AB}). Therefore for $0<\delta = \delta( || e^{\theta|v|^{2}} f_{0}||_{\infty}  ) \ll 1$
\begin{equation*} 
\begin{split}
&e^{-\varpi \langle v\rangle t} \frac{1}{\langle v\rangle^{b+1}} [\alpha(x,v)]^{\beta} \text{II}_{\mathbf{x}}
+e^{-\varpi \langle v\rangle t} \frac{|v|}{ \langle v\rangle^{b} } [\alpha(x,v)]^{\beta-\frac{1}{2}} \text{II}_{\mathbf{v}}\\
\lesssim & \ \delta \ \big\{ \sup_{m} \sup_{0 \leq s\leq t} \big|\big| e^{-\varpi \langle v \rangle s} \frac{ \alpha  ^\beta}{ \langle v \rangle^{b+1}} \partial_{x} f^{m }(s )  \big|\big|_{\infty}+   \sup_{m} \sup_{0 \leq s\leq t} \big|\big| e^{-\varpi \langle v \rangle s} \frac{|v| \alpha ^{\beta-\frac{1}{2}}}{ \langle v \rangle^{b}} \partial_{v} f^{m }(s)  \big|\big|_{\infty}\big\}.
\end{split}
\end{equation*}

Collecting all the terms, for $1< \beta < \frac{3}{2}$ and $b \in\mathbb{R}$ with $\varpi \gg1$ and $0 < \delta \ll1 $
\begin{equation}\notag
\begin{split}
&\sup_{m} \sup_{0 \leq s\leq t}||e^{-\varpi \langle v\rangle t}   \frac{\alpha ^{\beta} }{\langle v\rangle^{b+1 }}
\partial_{x}f^{m}(t )||_{\infty}
+ \sup_{m}  \sup_{0 \leq s\leq t}||e^{-\varpi \langle v\rangle t}  \frac{|v|  \alpha ^{\beta-\frac{1}{2}} }{\langle v\rangle^{b}}
\partial_{v}f^{m}(t )||_{\infty}
\\
 & \lesssim 
 ||   \frac{ \alpha^{\beta-\frac{1}{2}}}{\langle v\rangle^{b}}  \partial_{x}f_{0} ||_{\infty}
   + ||  \frac{ |v|^{2} \alpha^{\beta-1}}{\langle v\rangle^{b }}\partial_{v }f_{0 }||_{\infty} 
     +  {P}(||   e^{\theta |v|^{2}} f_{0} ||_{\infty}).
\end{split}
\end{equation} 
 
We remark that this sequence $f^{m}$ is Cauchy in $L^{\infty}([0,T]\times \bar{\Omega}\times\mathbb{R}^{3})$ for $0< T\ll 1$. Therefore the limit function $f$ is a solution of the Boltzmann equation satisfying the specular reflection BC. On the other hand, due to the weak lower semi-continuity of $L^{p}$, $p>1$, we pass a limit $\partial f^{m} \rightharpoonup \partial f$ weakly in the weighted $L^{\infty}-$norm.

Now we consider the continuity of $e^{-\varpi \langle v\rangle t} \langle v\rangle^{-1}\alpha^{\beta} \partial_{x} f$ and $e^{-\varpi \langle v\rangle t} |v| \alpha^{\beta-\frac{1}{2}} \partial_{v}f$. Remark that $e^{-\varpi \langle v\rangle t} \langle v\rangle^{-1}\alpha^{\beta} \partial_{x} f^{m}$ and $e^{-\varpi \langle v\rangle t} |v| \alpha^{\beta-\frac{1}{2}} \partial_{v}f^{m}$ satisfy all the conditions of Proposition \ref{inflowW1p}. Therefore we conclude 
$$e^{-\varpi \langle v\rangle t} \langle v\rangle^{-1}\alpha^{\beta} \partial_{x} f^{m} \in C^{0}([0,T^{*}] \times \bar{\Omega} \times\mathbb{R}^{3}), \ e^{-\varpi \langle v\rangle t} |v| \alpha^{\beta-\frac{1}{2}} \partial_{v}f^{m}\in C^{0}([0,T^{*}] \times \bar{\Omega} \times\mathbb{R}^{3}).$$
Now we follow $W^{1,\infty}$ estimate proof for $e^{-\varpi \langle v\rangle t} \langle v\rangle^{-1}\alpha^{\beta} [\partial_{x} f^{m+1}- \partial_{x}f^{m}]$ and $e^{-\varpi \langle v\rangle t} |v| \alpha^{\beta-\frac{1}{2}} [\partial_{v}f^{m+1}-\partial f^{m}]$ to show that $e^{-\varpi \langle v\rangle t} \langle v\rangle^{-1}\alpha^{\beta} \partial_{x} f^{m}$ and $e^{-\varpi \langle v\rangle t} |v| \alpha^{\beta-\frac{1}{2}} \partial_{v}f^{m}$ are Cauchy in $L^{\infty}$. Then we pass a limit $e^{-\varpi \langle v\rangle t} \langle v\rangle^{-1}\alpha^{\beta} \partial_{x} f^{m}\rightarrow e^{-\varpi \langle v\rangle t} \langle v\rangle^{-1}\alpha^{\beta} \partial_{x} f$ and $e^{-\varpi \langle v\rangle t} |v| \alpha^{\beta-\frac{1}{2}} \partial_{v}f^{m}\rightarrow e^{-\varpi \langle v\rangle t} |v| \alpha^{\beta-\frac{1}{2}} \partial_{v}f$ strongly in $L^{\infty}$ so that 
$e^{-\varpi \langle v\rangle t} \langle v\rangle^{-1}\alpha^{\beta} \partial_{x} f  \in C^{0}([0,T^{*}] \times \bar{\Omega} \times\mathbb{R}^{3})$ and $e^{-\varpi \langle v\rangle t} |v| \alpha^{\beta-\frac{1}{2}} \partial_{v}f \in C^{0}([0,T^{*}] \times \bar{\Omega} \times\mathbb{R}^{3}).$

\end{proof}
 
 \vspace{8pt}
 
\section{\large{Bounce-Back Reflection BC}}

\vspace{4pt}

We recall the bounce-back cycles from \textit{(iv)} of Definition \ref{cycles}: $(t^{0},x^{0},v^{0}) = (t,x,v)$ and
for $\ell\geq 1,$
\begin{eqnarray*}
t^{\ell}= t^{1} -(\ell-1)t_{\mathbf{b}}(x^{1},v^{1}) , \ \ x^{\ell} = \frac{%
1-(-1)^{\ell}}{2} x^{1} + \frac{1+(-1)^{\ell}}{2} x^{2}, \ \ v^{\ell+1} =
(-1)^{\ell+1} v,
\end{eqnarray*}
where $t_{\mathbf{b}}(x,v)$ is defined in (\ref{exit}).

 \begin{lemma}
\label{estimate_bb}  For all $0 \leq s \leq t,$
\begin{equation*}
\min \{ \alpha( x^{1},v^{1}),\alpha( x^{2},v^{2}) \}\lesssim_{\Omega}
\alpha( X_{\mathbf{cl}}(s;t,x,v), V_{\mathbf{cl}}(s;t,x,v)) \lesssim_{\Omega}
\max\{\alpha(x^{1},v^{1}),\alpha(x^{2},v^{2}) \}.
\end{equation*}
For $\ell_{*}(s;t,x,v)\in\mathbb{N}$ (therefore $t^{\ell_{*}+1}(t,x,v) \leq s \leq t^{\ell_{*}}(t,x,v)$)
\begin{equation*}
\ell_{*}(s;t,x,v) \ \leq \ \frac{|t-s|}{t_{\mathbf{b}}(x_{1},v_{1}) }  \ \lesssim_{\Omega} \ \frac{%
|t-s||v|^{2}}{\sqrt{\alpha( x,v) }}.
\end{equation*}
For all $0 \leq s \leq t$ uniformly
\begin{equation*}
\begin{split}
| \partial_{x_{i}} t^{\ell}(t,x,v)|&= \ \Big|-\ell \frac{\partial_{x_{i}}%
\xi(x^1) }{v\cdot \nabla \xi(x^{1})} - (\ell-1) \frac{\partial_{x_{i}}%
\xi(x^{2}) }{-v\cdot \nabla \xi(x^{2})}\Big| \ \lesssim_{\Omega} \frac{t|v|^{2}%
}{\alpha( x,v)} , \\
|\partial_{v_{i}} t^{\ell} (t,x,v)|&= \ \Big| \ell t_{\mathbf{b}}(x,v) \frac{%
\partial_{x_{i}}\xi(x^{1}) }{v\cdot \nabla \xi(x^{1})} + (\ell-1) t_{\mathbf{b}%
}(x,-v)\frac{ \partial_{x_{i}} \xi(x^{2})}{-v\cdot \nabla \xi(x^{2})} \Big| %
 \ \lesssim_{\Omega} \frac{t}{\sqrt{\alpha( x,v)}} , \\
|\partial_{x_{i}} x^{\ell}_{j}(x,v)| &= \ \Big| \frac{1-(-1)^{\ell}}{2} \Big\{%
\delta_{ij} - \frac{v_{j} \partial_{x_{i}} \xi(x_{1}) }{v\cdot \nabla
\xi(x_{1})}\Big\} + \frac{1+(-1)^{\ell}}{2}\Big\{ \delta_{ij} -\frac{v_{j}
\partial_{x_{i} }\xi(x_{2})}{v\cdot \nabla \xi(x_{2})} \Big\} \Big| %
 \ \lesssim_{\Omega} 1+ \frac{|v|}{\sqrt{\alpha( x,v)}} , \\
|\partial_{v_{i}} x^{\ell}_{j} (x,v)|&= \ \Big|\frac{1-(-1)^{\ell}}{2} (-t_{%
\mathbf{b}}(x,v)) \Big\{\delta_{ij} - \frac{v_{j} \partial_{x_{i}}
\xi(x_{1}) }{v\cdot \nabla \xi(x_{1})}\Big\} + \frac{1+(-1)^{\ell}}{2}(-t_{%
\mathbf{b}}(x,-v))\Big\{ \delta_{ij} -\frac{v_{j} \partial_{x_{i} }\xi(x_{2})%
}{v\cdot \nabla \xi(x_{2})} \Big\}\Big|, \\
&\lesssim_{\Omega} \frac{1}{|v|} , \\
\partial_{x_{i}} v^{\ell} &= \ 0, \ \ \ |\partial_{v_{i}} v^{\ell}_{j}|
=|(-1)^{\ell} \delta_{ij}| \ \lesssim_{\Omega} 1, \\
|\partial_{x_{i}}(t^{\ell}-t^{\ell+1})| &= \ \Big| \frac{\partial_{x_{i}}
\xi(x_{1}) }{v\cdot \nabla \xi(x_{1})} + \frac{\partial_{x_{i}} \xi(x_{2})}{%
-v\cdot \nabla \xi(x_{2})} \Big|  \ \lesssim_{\Omega} \frac{1}{\sqrt{%
\alpha( x,v)}}, \\
|\partial_{v_{i}}(t^{\ell}-t^{\ell+1})| &= \ \Big| t_{\mathbf{b}}(x,v) \frac{%
-\partial_{x_{i}} \xi(x_{1}) }{v\cdot \nabla \xi(x_{1})} + t_{\mathbf{b}%
}(x,-v)\frac{\partial_{x_{i}} \xi(x_{2})}{v\cdot \nabla \xi(x_{2})} \Big|  \ 
\lesssim_{\Omega} \frac{1}{|v|^{2}}.
\end{split}
\end{equation*}
\end{lemma}
\begin{proof}
These are direct consequence of (\ref{xb}) and the Velocity lemma (Lemma \ref{velocity_lemma}).
\end{proof}

Now we state the key ingredient in the case of the bounce-back BC which is the general version of Lemma \ref{change_time_bb}: In the sense of distribution, 
\begin{equation}\label{change_time_bb_ell}
\begin{split}
&\partial_{\mathbf{e}} \Big[ \sum_{\ell=0}^{\ell_{*}(s)} \int^{t^{j}}_{ \max\{  s, t^{j+1} \}} A^{m-j}(\tau, x^{j}-(t^{j}- {\tau})  v^{j}, v^{j}) \mathrm{d} {\tau}  \Big]\\
=&\sum_{j=0}^{\ell_{*} (s)} \int^{t^{j}}_{ \max\{s, t^{j+1}\}}  \big[ \partial_{\mathbf{e}} t^{j}, \partial_{\mathbf{e}} x^{j} + \tau \partial_{\mathbf{e}} v^{j}, \partial_{\mathbf{e}}v^{j}  \big] \cdot \nabla_{t,x,v} A^{m-j} (\tau, x^{j} -(t^{j} - \tau) v^{j}, v^{j}) \mathrm{d}\tau\\
& \ +  \sum_{j=0}^{\ell_{*} (s)-1 }  \partial_{\mathbf{e}}[t^{j}-t^{j+1}] 
\lim_{\tau \downarrow -(t^{j} -t^{j+1})} A^{m-j} (\tau+t^{j}, x^{j} + \tau v^{j}, v^{j})  \\
& \  + \partial_{\mathbf{e}} t^{\ell_{*}(s)} 
\lim_{\tau \downarrow -(t^{\ell_{*} (s)}-s)} A^{m-\ell_{*}(s)} ( \tau+ t^{\ell_{*}(s)}, x^{\ell_{*}(s)} + \tau v^{\ell_{*}(s)}, v^{\ell_{*}(s)}  ).
\end{split}
\end{equation}

Note that (\ref{change_time_bb_ell}) is more general than Lemma \ref{change_time_bb}.

\begin{proof}[Proof of (\ref{change_time_bb_ell}) and Lemma \ref{change_time_bb}] 
Once we prove (\ref{change_time_bb_ell}) then Lemma \ref{change_time_bb} holds clearly. Now we prove (\ref{change_time_bb_ell}): For each time intervals $[t^{j+1},t^{j}]$, we
apply the change of variables
\begin{equation}\label{COV_time}
\begin{split} 
x^{j} -(t^{j}-\tau)v^{j}, \ \tau\in[t^{j+1},t^{j}] \ \ &\mapsto \ \
x^{j}+\tau v^{j}, \ \tau  \in [-(t^{j}-t^{j+1}),0], \\
& \ \ \ \ \ \ \ \ \ \ \ \ \ \ \ \ \ \ \ \ \ \ \ \  \text{for} \ j=0,1, \cdots, \ell_{*}(s)-1,\\
x^{\ell_{*}(s)} - ( t^{\ell_{*}(s)} -\tau ) v^{\ell_{*}(s)}, \  \tau \in [s, t^{\ell_{*}(s)}]
 \ \ &\mapsto \ \ x^{\ell_{*}(s)} +\tau  v^{\ell_{*}(s)}, \ \tau  \in [-(t^{ \ell_{*}(s)}-s),0].
\end{split}
\end{equation}
  
From (\ref{piecewise}) the piecewise derivatives equal distributional derivatives almost everywhere. Therefore we prove Lemma \ref{change_time_bb}.  Moreover
\begin{equation}\notag
\begin{split}
&\partial_{\mathbf{e}} \Big[ \sum_{ j=0}^{\ell_{*}(s)} \int^{t^{j}}_{ \max\{  s, t^{j+1} \}} A^{m-j}(\tau, x^{j}-(t^{j}- {\tau})  v^{j}, v^{j}) \mathrm{d} {\tau}  \Big]\\
=& \ 
\partial_{\mathbf{e}} \Big[ \sum_{j=0}^{\ell_{*}(s)-1} \int^{t^{j}}_{t^{j+1}} \cdots \Big] + \partial_{\mathbf{e}} \Big[ \int^{t^{\ell_{*}(s) }  }_{s} \cdots \Big]
\\
=& \ \partial_{\mathbf{e}} \Big[  { \sum_{j=0}^{\ell_{*} (s) -1} \int^{0}_{   -(t^{j} -t^{j+1}) }    A^{m-j} (\tau + t^{j}, x^{j} + \tau v^{j} ,v^{j}) \mathrm{d}\tau  }\Big]\\
&+ \partial_{\mathbf{e}} \Big[ \int^{0}_{-( t^{\ell_{*} (s)}-s )} 
A^{m-\ell_{*}(s)} (\tau + t^{\ell_{*}(s)}, x^{\ell_{*}(s)} + \tau v^{\ell_{*}(s)}, v^{\ell_{*}(s)} )
\Big]
 \\ 
=&   \sum_{j=0}^{\ell_{*} (s) -1 }  \int^{0}_{    -(t^{j} -t^{j+1})  } \partial_{\mathbf{e}} \big[  A^{m-j} (\tau + t^{j}, x^{j} + \tau v^{j} ,v^{j})\big] \mathrm{d}\tau\\ 
&+  \sum_{j=0}^{\ell_{*} (s)-1 } 
 \partial_{\mathbf{e}}[t^{j}-t^{j+1}] 
 \lim_{\tau \downarrow -(t^{j} -t^{j+1})}A^{m-j} (\tau+t^{j}, x^{j} + \tau v^{j}, v^{j})  \\
&    +   \int^{0}_{-( t^{\ell_{*} (s)}-s )}  \partial_{\mathbf{e}} \big[
A^{m-\ell_{*}(s)} (\tau + t^{\ell_{*}(s)}, x^{\ell_{*}(s)} + \tau v^{\ell_{*}(s)}, v^{\ell_{*}(s)} ) \big]\\
& + \partial_{\mathbf{e}} t^{\ell_{*}(s)} \lim_{\tau \downarrow - (t^{\ell_{*}(s)}-s)}A^{m-\ell_{*}(s)}
(\tau  + t^{\ell_{*}(s)}, x^{\ell_{*}(s)} +\tau v^{\ell_{*}(s)}, v^{\ell_{*}(s)} )
.
\end{split}
\end{equation}
Directly we have
\[
\partial_{\mathbf{e}} \big[  A^{m-j} (\tau + t^{j}, x^{j} + \tau v^{j} ,v^{j})\big] 
= \big[ \partial_{\mathbf{e}} t^{j}, \partial_{\mathbf{e}} x^{j} + \tau \partial_{\mathbf{e}} v^{j}, \partial_{\mathbf{e}} v^{j} \big] \cdot
\nabla_{t,x,v} A^{m-j} (\tau + t^{j}, x^{j} + \tau v^{j} ,v^{j}).
\]

Then we apply the inverse of the change of variables in (\ref{COV_time}) to the time integration terms:
\begin{equation}\notag
\begin{split}
\sum_{j=0}^{\ell_{*} (s)} \int^{t^{j}}_{ \max\{s, t^{j+1}\}}  \big[ \partial_{\mathbf{e}} t^{j}, \partial_{\mathbf{e}} x^{j} + \tau \partial_{\mathbf{e}} v^{j}, \partial_{\mathbf{e}}v^{j}  \big] \cdot \nabla_{t,x,v} A^{m-j} (\tau, x^{j} -(t^{j} - \tau) v^{j}, v^{j}) \mathrm{d}\tau.
\end{split}
\end{equation}
We collect the terms and conclude (\ref{change_time_bb_ell}).
\end{proof}

 Now we are ready to proof the main theorem:

\begin{proof}[\textbf{Proof of Theorem \ref{main_bb}}]
  %
%
We use the approximation sequence (\ref{fm}) with (\ref{bb_BC_m}). Due to Lemma \ref{local_existence} we have (\ref{bounded}) and (\ref{bounded_t}).

Now we consider the spatial and velocity derivatives. From the iteration (\ref{positive_iteration}) and (\ref{bb_BC_m}), for $\ell_{*}(0;t,x,v)=\ell_{*}$ with $t^{\ell_{*} +1} \leq 0 < t^{\ell_{*}},$
\begin{equation}
\notag
\begin{split}
& \ f^{m+1}(t,x,v)\\
 &= \ e^{-\sum_{j=0}^{\ell_{*}(0) }\int_{ \max\{0,t^{j+1} \} }^{t^{j}} \nu(F^{m-j})(\tau) \mathrm{d}\tau  } \
f_{0}(x^{\ell_{*} (0)}-t^{\ell_{*} (0)}v^{\ell_{*} (0)},v^{\ell_{*} (0)}) \\
&+  \sum_{\ell=0}^{\ell_{*}(0)}\int_{ \max\{ 0, t^{\ell+1}\}}^{t^{\ell}} e^{-
\sum_{j=0}^{\ell_{*}(s)}\int_{\max \{0, t^{j+1}\}}^{t^{j}}\nu (F^{m-j})(\tau )%
\mathrm{d}\tau } \ \Gamma _{\text{gain } }(f^{m-\ell},f^{m-\ell})
(s,x^{\ell}-(t^{\ell}-s)v^{\ell},v^{\ell})\mathrm{d}s,
\end{split}
\end{equation}
where $\nu(F^{m-j})(\tau)= \mu(\sqrt{\mu} f^{m-j})(\tau)=\nu (\sqrt{\mu}f^{m-j})(\tau, x^{j}-(t^{j}-\tau)v^{j},v^{j})$.

From Lemma \ref{change_time_bb} and (\ref{change_time_bb_ell}) and (\ref{invariant}), in the sense of distribution, for $\partial_{\mathbf{e}} = [\partial_{x}, \partial_{v}]$ with $\mathbf{e} \in \{x,v\},$
\begin{equation}\label{D_duhamel_bb}
\begin{split}
&\partial_{\mathbf{e}} f^{m}(t,x,v)  \\
=& \ \mathbf{I}_{\mathbf{e}} + \mathbf{II}_{\mathbf{e}}  \\
=& \ e^{- \int^{t}_{0} \sum_{j}\mathbf{1}_{[t^{j+1},t^{j})}(\tau)\nu(F^{m-j})(\tau, X_{\mathbf{cl}}(\tau), V_{\mathbf{cl}}(\tau)) \mathrm{d}\tau } f_{0} (X_{\mathbf{cl}}(0), V_{\mathbf{cl}}(0))\\
&\times \Big\{
-\sum_{j=0}^{\ell_{*}(0)} \int^{t^{j}}_{\max\{  0,t^{j+1}\}} \underline{ \big[ \partial_{\mathbf{e}} t^{j}, \partial_{\mathbf{e}} x^{j} + \tau \partial_{\mathbf{e}} v^{j}, \partial_{\mathbf{e}}v^{j}\big]\cdot \nabla_{t,x,v} \nu(F^{m-j})(\tau, x^{j}-(t^{j}-\tau) v^{j}, v^{j}) \mathrm{d}\tau }_{\mathbf{II}_{\mathbf{e}}}\\
& \ \ \ \ \ \  -\sum_{j=0}^{\ell_{*}(0) -1}    \underline {\partial_{\mathbf{e}}[t^{j} -t^{j+1}] \nu(F^{m-j})(t^{j+1}, x^{j+1}, v^{j})}_{\mathbf{I}_{\mathbf{e}}} - \underline{ \partial_{\mathbf{e}} t^{\ell_{*}(0)} \nu(F^{m-\ell_{*}(0) }) (0, x^{j}-t^{j} v^{j} ,v^{j})}_{\mathbf{I}_{\mathbf{e}}}
\Big\} \\
+& \ \underline{  e^{- \int^{t}_{0} \sum_{j}\mathbf{1}_{[t^{j+1},t^{j})}(\tau)\nu(F^{m-j})(\tau, X_{\mathbf{cl}}(\tau), V_{\mathbf{cl}}(\tau)) \mathrm{d}\tau } \partial_{\mathbf{e}} \big[  x^{\ell_{*} (0)} - t^{\ell_{*} (0)} v^{\ell_{*}(0)}, v^{\ell_{*} (0)} \big] \cdot \nabla_{x,v} f_{0} (X_{\mathbf{cl}}(0), V_{\mathbf{cl}}(0))}_{\mathbf{I}_{\mathbf{e}}} \\
+&\sum_{\ell=0}^{\ell_{*} (0)-1}  { \partial_{\mathbf{e}}[t^{\ell} -t^{\ell+1}] e^{ - \sum_{j=0}^{\ell_{*}(t^{\ell} -t^{\ell+1})} \int^{0}_{\max\{ t^{\ell}-t^{\ell+1} -t^{j}, -(t^{j}-t^{j+1})  \}} \nu(F^{m-j}) (\tau+ t^{j}, x^{j}+ \tau v^{j}, v^{j}) \mathrm{d}\tau
} }\\
& \times \underline {  \Gamma_{\mathrm{gain}}(f^{m-\ell} , f^{m-\ell}) (t^{\ell+1}, x^{\ell+1}, v^{\ell})}_{\mathbf{I}_{\mathbf{e}}}\\
+& \  \underline{\partial_{\mathbf{e}} t^{\ell_{*} (0)}  e^{- \int^{t}_{0} \mathbf{1}_{[t^{j+1}, t^{j})}(s)  \nu(F^{m-j})(\tau) \mathrm{d}\tau  }
  \Gamma_{\mathrm{gain}}(f^{m-\ell_{*} (0)}, f^{m-\ell_{*} (0)}) (0, x^{\ell_{*}(0)  }  -t^{\ell_{*} (0)} v^{\ell_{*} (0)}, v^{\ell_{*} (0)} )}_{\mathbf{I}_{\mathbf{e}}}\\
+&\int_{0}^{t} \mathbf{1}_{  [t^{\ell +1}, t^{\ell})  } (s) e^{-\int^{t}_{s}  \sum_{j=0}^{\ell_{*} (s)} \mathbf{1}_{[ t^{j+1}, t^{j} )}(s)    \nu(F^{m-j})(\tau) \mathrm{d}\tau }\\
&\times \underline{[\partial_{\mathbf{e}} t^{\ell}, \partial_{\mathbf{e}} x^{\ell} +s \partial_{\mathbf{e}} v^{\ell} , \partial_{\mathbf{e}} v^{\ell}] \cdot
\nabla_{t,x,v} \Gamma_{\mathrm{gain}}(f^{m-\ell} ,f^{m-\ell})](s, x^{\ell} - (t^{\ell} -s) v^{\ell} , v^{\ell}) \mathrm{d}s}_{\mathbf{II}_{\mathbf{e}}}\\
+& \int^{t}_{0} \mathbf{1}_{[ t^{\ell+1}, t^{\ell})}(s)
  \Gamma_{\mathrm{gain}} (f^{m-\ell}, f^{m-\ell}) (s, X_{\mathbf{cl}}(s), V_{\mathbf{cl}}(s))
 \mathrm{d}s\\
 & \times \Big\{
 - \sum_{j=0}^{\ell_{*}(s)-1} \underline{ \partial_{\mathbf{e}} [t^{j} -t^{j+1}] \nu(F^{m-j}) (t^{j+1}, x^{j+1}, v^{j}) }_{\mathbf{I}_{\mathbf{e}}}
 -\underline{\partial_{\mathbf{e}} t^{\ell_{*} (s)} \nu(F^{m-\ell_{*} (s)})(s, X_{\mathbf{cl}}(s), V_{\mathbf{cl}}(s))}_{\mathbf{I}_{\mathbf{e}}}\\
  & \ \ \ \ \ -\sum_{j=0}^{\ell_{*}(s)} \underline{\int^{t^{j}}_{\max\{ s, t^{j+1} \}} [ \partial_{\mathbf{e}} t^{j} , \partial_{\mathbf{e}} x^{j } + \tau \partial_{\mathbf{e}} v^{j} , \partial_{\mathbf{e}} v^{j}]  \cdot \nabla_{t,x,v} \nu(F^{m-j}) (\tau, x^{j} -(t^{j} -\tau) v^{j} ,v^{j}) \mathrm{d}\tau}_{\mathbf{II}_{\mathbf{e}}}
  \Big\}.
\end{split}
\end{equation} 
We shall estimate the followings:
\[
e^{-\varpi \langle v\rangle t}   \frac{\alpha(x,v)}{\langle v\rangle^{2}} \partial_{x} f(t,x,v), \ \ \ \ e^{-\varpi \langle v\rangle t}   \frac{|v|\alpha(x,v)^{1/2}}{\langle v\rangle^{2}} \partial_{v} f(t,x,v).
\] 

Firstly, we estimate $\mathbf{I}_{\mathbf{e}}.$ Using Lemma \ref{estimate_bb} and Lemma \ref{lemma_operator} and $F^{m} \geq 0$ from (\ref{positive_iteration}) and Lemma \ref{local_existence}, for some polynomial $P$,
\begin{equation}\notag
\begin{split}
& e^{-\varpi \langle v\rangle t} \langle v\rangle^{-2} \alpha(x,v)  \mathbf{I}_{\mathbf{x}} \\
   \lesssim  & \ e^{-\varpi \langle v\rangle t}   \langle v\rangle^{-2} \alpha(x,v) 
P( ||  e^{\theta  |v|^{2}}f  ||_{\infty})
\\ &  \ \times    \Big\{    e^{-\theta |v|^{2}} \frac{t|v|^{2}  }{ \alpha(x,v)} \langle v\rangle^{\kappa}  + \big[(1+ \frac{|v| }{\alpha(x,v)}) + \frac{t|v|^{3}  }{ \alpha(x,v)}    \big] |\partial_{x} f_{0}| +\frac{t|v|^{2}  }{\alpha(x,v)}  e^{- \frac{\theta}{2}|v|^{2}}    + t   e^{- \frac{\theta}{2}|v|^{2}}    \langle v\rangle^{\kappa} \frac{t|v|^{2}  }{\alpha(x,v)}  \Big\}\\
  \lesssim &  \ ||    \langle v\rangle^{-2} \alpha (1+ \frac{|v|+ |v|^{3}}{\alpha(x,v)} ) \partial_{x}f_{0} ||_{\infty} +      \langle v\rangle^{-2}  e^{-C_{\theta} |v|^{2}} P(||  e^{\theta |v|^{2}} f_{0} ||_{\infty}  )\\
 \lesssim  & \ (1+ || \langle v\rangle \partial_{x} f_{0} ||_{\infty} ) \times P(||  e^{\theta |v|^{2}} f_{0} ||_{\infty}  ) 
.
\end{split}
\end{equation}
%

Similarly
\begin{equation}\notag
\begin{split}
& e^{-\varpi \langle v\rangle t} |v| \langle v\rangle^{-2} \alpha^{{1}/{2}} \mathbf{I}_{\mathbf{v}} \\
  \lesssim & \ e^{-\varpi \langle v\rangle t}  |v| \langle v\rangle^{-2}[\alpha(x,v)]^{1/2} P(|| e^{\theta |v|^{2}} f||_{\infty})
\\ &  \ \times    \Big\{
e^{-\frac{\theta}{2} |v|^{2}}
\frac{t  }{ \alpha(x,v)^{1/2}} \langle v\rangle^{\kappa}  + \big[(\frac{1}{|v|}+ \frac{  \alpha(x,v)^{1/2}}{|v|^{2}}) + \frac{t|v|  }{ \alpha(x,v)^{1/2}}  + t    \big] |\partial_{x} f_{0}| + |\nabla_{v} f_{0}|\\
& \ \ \ \ +\frac{t    }{\alpha(x,v)^{1/2}}   e^{-\frac{\theta}{2}|v|^{2}}   + t  e^{-\frac{\theta}{2}|v|^{2}}  \langle v\rangle^{\kappa}
 \frac{t   }{\alpha(x,v)^{1/2}}  \Big\}\\
 \lesssim & \   (1+ || \langle v\rangle \partial_{x} f_{0}||_{\infty} + || \partial_{v} f_{0} ||_{\infty}) P (|| e^{\theta |v|^{2}} f_{0} ||_{\infty})
 .
\end{split}
\end{equation}

Secondly, we estimate $\mathbf{II}_{\mathbf{e}}.$ Let $\phi_{\mathbf{e}} \in \{ \phi_{\mathbf{x}}, \phi_{\mathbf{v}}\}$ with $\phi_{\mathbf{x}} = e^{-\varpi\langle v\rangle t } \frac{\alpha(x,v)}{\langle v\rangle^{2}}$ and $\phi_{\mathbf{v}}= e^{-\varpi \langle v\rangle t} \frac{|v| \alpha(x,v)^{1/2}}{\langle v\rangle^{2}}$. We have
\begin{eqnarray} 
&& e^{-\varpi \langle v\rangle t}  \phi_{ {\mathbf{e}}}(v)  [\alpha(x,v)]^{\beta_{ {\mathbf{e}}}}  \mathbf{II}_{\mathbf{e}} \notag \\ 
&  \lesssim& e^{-\varpi \langle v\rangle t}  \phi_{ {\mathbf{e}}}(v)  [\alpha(x,v)]^{\beta_{ {\mathbf{e}}}}
\Big\{ 1+ (1+t) 
  e^{-\frac{\theta}{2} |v|^{2}} ||  e^{\theta |v|^{2}} f  ||_{\infty}
  \Big\}
   \notag \\
&&   \times \bigg\{
\int_{0}^{t} \sum_{j=0}^{\ell_{*}(0)} \mathbf{1}_{[t^{j+1}, t^{j})}(s)
 |\partial_{  {\mathbf{e}}} t^{j}|  \langle v\rangle^{\kappa} \mathrm{d}s
 \times  || e^{\theta|v|^{2}} \partial_{t}f  ||_{\infty} \label{II_1} \\
 && \   + \int_{0}^{t} \sum_{j=0}^{\ell_{*}(0)} \mathbf{1}_{[t^{j+1}, t^{j})}(s)  \{ |\partial_{  {\mathbf{e}}} x^{\ell}| + t|\partial_{  {\mathbf{e}}} v^{\ell}| \} \nu(\sqrt{\mu}\partial_{x}f^{m-\ell}) (s,X_{\mathbf{cl}}(s), V_{\mathbf{cl}}(s)) \mathrm{d}s \label{II_2}
\\
&& \  +  \int_{0}^{t} \sum_{j=0}^{\ell_{*}(0)} \mathbf{1}_{[t^{j+1}, t^{j})}(s)    |\partial_{  {\mathbf{e}}} v^{\ell}|  \int_{\mathbb{R}^{3}} |V_{\mathbf{cl}}(s)-u|^{\kappa-1} \sqrt{\mu(u)}    f^{m-\ell}  (s,X_{\mathbf{cl}}(s), u) \mathrm{d}u   \mathrm{d}s \label{II_3} \\
&& \ +    \int_{0}^{t} \sum_{\ell=0}^{\ell_{*}(0)} \mathbf{1}_{[t^{\ell+1}, t^{\ell})}(s)  |\partial_{  {\mathbf{e}}} t^{\ell}|
\big[   | \Gamma_{\mathrm{gain}}(\partial_{t }f^{m-\ell}, f^{m-\ell})| + | \Gamma_{\mathrm{gain}}(f^{m-\ell},\partial_{ t} f^{m-\ell}  )| \big] \mathrm{d}s
\label{II_4}
\\
&& \   + \int_{0}^{t} \sum_{\ell=0}^{\ell_{*}(0)} \mathbf{1}_{[t^{\ell+1}, t^{\ell})}(s) \big\{|\partial_{  {\mathbf{e}}} x^{\ell}| + t|\partial_{    {\mathbf{e}}} v^{\ell}|\big\}
\big[      | \Gamma_{\mathrm{gain}}(\partial_{  {x}}f^{m-\ell}, f^{m-\ell})| + | \Gamma_{\mathrm{gain}}(f^{m-\ell},\partial_{  {x}} f^{m-\ell}  )| \big] \mathrm{d}s\nonumber\\
\label{II_5}\\
&& \   + \int_{0}^{t} \sum_{\ell=0}^{\ell_{*}(0)} \mathbf{1}_{[t^{\ell+1}, t^{\ell})}(s) |\partial_{  {\mathbf{e}}} v^{\ell}|  \big[  |\Gamma_{\mathrm{gain},  v}(f^{m-\ell}, f^{m-\ell}) | \nonumber   \\
&&   \ \ \ \ \  \ \ \ \ \  \ \ \ \ \  \ \ \ \ \   \ \ \ \ \  \ \ \ \ \  \ \ \ \ \  \ \ \ \ \ + |\Gamma_{\mathrm{gain} }(f^{m-\ell},  \partial_{v}f^{m-\ell}) |+   |\Gamma_{\mathrm{gain} }( \partial_{v}f^{m-\ell}, f^{m-\ell}) | \big] \mathrm{d}s\bigg\}\label{II_6}
 . 
\end{eqnarray}

Firstly, we consider $\partial_{\mathbf{e}} t^{j}-$contribution. Then from Lemma \ref{estimate_bb} and \textit{(2)} of Lemma \ref{lemma_operator}  
\begin{equation}\notag
\begin{split}
 &e^{-\varpi \langle v\rangle t}  
 \frac{\alpha(x,v)}{\langle v\rangle^{2}} \{  (\ref{II_1})_{\mathbf{x}}  + (\ref{II_4})_{\mathbf{x}} \}\\
 &\lesssim e^{-\varpi \langle v\rangle t} \frac{\alpha(x,v)}{ \langle v\rangle^{2}}
  e^{- \frac{\theta}{2} |v|^{2}}
 t \frac{t|v|^{2}}{\alpha(x,v)}  \langle v\rangle 
  ||  e^{\theta|v|^{2}} f ||_{\infty}
  ||   e^{\theta|v|^{2}} \partial_{t} f  ||_{\infty} \\
 & \ + e^{-\varpi \langle v\rangle t}  \frac{\alpha(x,v)}{\langle v\rangle^{2}} (1+ t) ||  e^{\theta|v|^{2}} f ||_{\infty }   t \frac{t|v|^{2}}{\alpha(x,v)}    e^{- \frac{\theta}{2}|v|^{2}} ||  e^{\theta|v|^{2}} \partial_{t}f ||_{\infty}   \\
 & \lesssim_{t} 1+   P(||   e^{\theta |v|^{2}} \partial_{t}f_{0}||_{\infty}) + P(||   e^{\theta |v|^{2}} f_{0}||_{\infty}).
\end{split}
\end{equation}
Similarly, 
\begin{equation}\notag
\begin{split}
&e^{-\varpi \langle v\rangle t}  \frac{|v| \alpha(x,v)^{1/2}}{\langle v\rangle^{2}} \{ (\ref{II_1})_{\mathbf{v}}+ (\ref{II_4})_{\mathbf{x}} \}\\
&\lesssim_{t}  \frac{|v|}{ \langle v\rangle^{2}}  e^{-C_{\theta} |v|^{2}}  \big[ P(||  e^{\theta |v|^{2}} \partial_{t}f_{0}||_{\infty}) + P(||  e^{\theta |v|^{2}} f_{0}||_{\infty} )\big] \\
& \lesssim_{t} 1+ P( || e^{\theta |v|^{2}} \partial_{t}f_{0}||_{\infty} )+ P(||  e^{\theta |v|^{2}} f_{0}||_{\infty}) .
\end{split}
\end{equation}

Secondly, we consider the terms $(\ref{II_2}),  (\ref{II_4})$, which include $|\partial_{\mathbf{e}} x^{\ell} | + t |\partial_{\mathbf{e}} v^{\ell}|$. We use \textit{(2)} of Lemma \ref{lemma_operator} and Lemma \ref{estimate_bb}, $|\partial_{x} x^{\ell}| + t |\partial_{x} v^{\ell}|  \lesssim  \frac{|v|}{\sqrt{\alpha(x,v)}}$, and (\ref{exponent})
\begin{equation}\notag
\begin{split}
&e^{-\varpi \langle v\rangle t}
 \frac{\alpha(x,v)}{\langle v\rangle^{2}}
 \{ (\ref{II_2})_{\mathbf{x}} + (\ref{II_4})_{\mathbf{x}}  \}\\
&\lesssim_{t} \big[1+ P( ||  e^{\theta |v|^{2}} \partial_{t}f_{0}||_{\infty}) +P( ||  e^{\theta|v|^{2}} f_{0}||_{\infty} ) \big]\\
& \ \ \ \times  \sum_{\ell=0}^{\ell_{*}(0;t,x,v)} \int^{t^{\ell}}_{t^{\ell+1}} \int_{\mathbb{R}^{3}} e^{-\varpi \langle v\rangle t}   \frac{|v|}{\langle v\rangle^{2}}\alpha(x,v)^{  {1}/{2}  } 
\frac{e^{-C_{\theta} |u-v^{\ell}|^{2} }}{|u-v^{\ell}|^{2-\kappa}}
  |\partial_{x} f^{m-\ell}(s,X_{\mathbf{cl}}(s), u)|  \mathrm{d}u \mathrm{d}s,\\
&\lesssim_{t} 
\big[1+ P(||   e^{\theta |v|^{2}} \partial_{t}f_{0}||_{\infty} )+ P(||   e^{\theta |v|^{2}} f_{0}||_{\infty}) \big]
\max_{0 \leq \ell \leq m}\sup_{ 0 \leq s \leq t} || e^{-\varpi \langle v\rangle s}  \frac{\alpha}{\langle v\rangle^{2}} \partial_{x} f^{m-\ell}(s) ||_{\infty}\\
& \ \ \ \times   \int^{t}_{0} \sum_{\ell=0}^{\ell_{*}(0;t,x,v)} \mathbf{1}_{[t^{\ell+1}, t^{\ell})}(s)\int_{\mathbb{R}^{3}} e^{-\frac{\varpi}{2} \langle v\rangle ( t-s)} \frac{\langle u\rangle^{2}}{\langle v\rangle^{2}} \frac{ |v|\alpha(x,v)^{ \frac{1}{2}  } }{ |V_{\mathbf{cl}}(s)-u|^{2-\kappa} \alpha(X_{\mathbf{cl}}(s) , u )   } e^{-C_{\theta} |v-u|^{2}}  \mathrm{d}u \mathrm{d}s
.
\end{split}
\end{equation}
We use $ \frac{\langle u\rangle^{2}}{\langle v\rangle^{2}}\lesssim   \langle v-u\rangle^{2}$ and (\ref{specular_nonlocal_u}) to have
\begin{equation}\notag
\begin{split}
   &\lesssim_{t} \big[1+ P(||  e^{\theta |v|^{2}} \partial_{t}f_{0}||_{\infty}) + P(||   e^{\theta|v|^{2}} f_{0}||_{\infty})\big]   \max_{0 \leq \ell \leq m}\sup_{ 0 \leq s \leq t} || e^{-\varpi \langle v\rangle s} 
\frac{\alpha }{\langle v\rangle^{2} }
   \partial_{x} f^{m-\ell}(s) ||_{\infty}\\
   & \ \ \  \times \frac{ O(\delta)  }{ \langle v\rangle \alpha(x,v)^{ 1/2  }} |v| \alpha(x,v)^{1/2}\\
   & \lesssim   O(\delta) \big[1+ P(|| e^{\theta |v|^{2}} \partial_{t}f_{0}||_{\infty}) + P(||  e^{\theta |v|^{2}} f_{0}||_{\infty}) \big]   \max_{0 \leq \ell \leq m}\sup_{ 0 \leq s \leq t} || e^{-\varpi \langle v\rangle s}  \frac{ \alpha}{\langle v\rangle^{2} } \partial_{x} f^{m-\ell}(s) ||_{\infty}.
\end{split}
\end{equation}
Similarly we further use $|\partial_{v} x^{\ell}| + t |\partial_{v} v^{\ell}|   \lesssim  \frac{1}{|v|}$ from Lemma \ref{estimate_bb}
\begin{equation}\notag
\begin{split}
&e^{-\varpi \langle v\rangle t}   
\frac{|v| \alpha(x,v)^{1/2}}{\langle v\rangle^{2}}
\{  (\ref{II_2})_{\mathbf{v}}+(\ref{II_4})_{\mathbf{v}}  \}\\
&\lesssim_{t} \big[1+  P(||   e^{\theta |v|^{2}} \partial_{t}f_{0}||_{\infty}) + P(||   e^{\theta |v|^{2}} f_{0}||_{\infty})  \big]\\
& \ \ \ \times  \sum_{\ell=0}^{\ell_{*}(0;t,x,v)} \int^{t^{\ell}}_{t^{\ell+1}} \int_{\mathbb{R}^{3}} e^{-\varpi \langle v\rangle t} 
\frac{|v|}{\langle v\rangle^{2}}
   ( \frac{1}{|v|}+1)\alpha(x,v)^{1/2}
\frac{e^{-C_{\theta} |u-v^{\ell}|^{2}}}{|u-v^{\ell}|^{2-\kappa}}
  |\partial_{x} f^{m-\ell}(s,X_{\mathbf{cl}}(s), u)|  \mathrm{d}u \mathrm{d}s,\\
&\lesssim_{t} \big[1+ P(||   e^{\theta |v|^{2}} \partial_{t}f_{0}||_{\infty}) + P(||   e^{\theta |v|^{2}} f_{0}||_{\infty} ) \big]  \max_{0 \leq \ell \leq m} \sup_{0 \leq s\leq t}  || e^{-\varpi \langle v\rangle t} 
\frac{|u| \alpha^{1/2}}{\langle u\rangle^{2}}
  \partial_{x} f^{m-\ell}(s,X_{\mathbf{cl}}(s),u)||_{\infty}  \\
& \ \ \ \times  \sum_{\ell=0}^{\ell_{*}(0;t,x,v)} \int^{t^{\ell}}_{t^{\ell+1}} \int_{\mathbb{R}^{3}} e^{- \frac{\varpi}{2} \langle v\rangle (t-s)} \frac{|v|\langle u\rangle^{2}}{  \langle v\rangle^{2} } 
 \frac{( \frac{1}{|v|}+1)\alpha(x,v)^{1/2} }{|V_{\mathbf{cl}}(s)-u | ^{2-\kappa}\alpha(x,u) }
   \mathrm{d}u \mathrm{d}s.
\end{split}
\end{equation}
From $\frac{\langle u\rangle^{2}}{\langle v\rangle^{2}}\lesssim \langle v-u\rangle^{2}$, the last integration is bounded by
\[
 \sum_{\ell=0}^{\ell_{*}(0;t,x,v)} \int^{t^{\ell}}_{t^{\ell+1}} \int_{\mathbb{R}^{3}} e^{-\frac{\varpi }{2} \langle v\rangle(t-s)} \langle v-u \rangle^{2} \frac{\langle v\rangle \alpha(x,v)^{1/2} }{ \alpha(x,u)} \frac{e^{-C|V_{\mathbf{cl}}(s)-u|^{2}}}{|V_{\mathbf{cl}}(s)-u|^{2-\kappa}} \mathrm{d}u \mathrm{d}s.
\]
By the dynamical non-local to local estimate (\ref{specular_nonlocal}), this is bounded by
\begin{equation}\notag
\begin{split}
  O(\delta) \big[1+ P( ||   e^{\theta |v|^{2}} \partial_{t}f ||_{\infty}) + P(||  e^{\theta |v|^{2}} f ||_{\infty}) \big] \max_{0 \leq \ell \leq m}\sup_{ 0 \leq s \leq t} || e^{-\varpi \langle v\rangle s}\frac{\alpha(x,v)}{\langle v\rangle^{2}} \partial_{x} f^{m-\ell}(s) ||_{\infty}.
\end{split}
\end{equation}

Thirdly, we consider $\partial_{\mathbf{e}} v^{\ell}-$contribution, $(\ref{II_3})$ and $(\ref{II_6})$. Note that $(\ref{II_3})_{\mathbf{x}}=0= (\ref{II_6})_{\mathbf{x}}$ since $\partial_{x} v^{j}\equiv0$. From Lemma \ref{estimate_bb} and \textit{(3)} of Lemma \ref{lemma_operator}
\begin{equation}\notag
\begin{split}
 &e^{-\varpi \langle v\rangle t}  \frac{ |v| \alpha(x,v)^{1/2} }{\langle v\rangle^{2}} \{  (\ref{II_3})_{\mathbf{v}}+ (\ref{II_6})_{\mathbf{v}}  \}\\
\lesssim & \ \big[1+ P(|| e^{\theta |v|^{2}} \partial_{t} f_{0} ||_{\infty})  + P( || e^{\theta |v|^{2}} f_{0} ||_{\infty})\big]\\
 & \times \int_{0}^{t} \sum_{\ell=0}^{\ell_{*}(0)} \mathbf{1}_{[t^{\ell+1}, t^{\ell})} (s)  e^{-\varpi \langle v\rangle t} 
  \frac{ |v| \alpha(x,v)^{1/2} }{\langle v\rangle^{2}} e^{-C |v|^{2}}
    \int_{\mathbb{R}^{3}}  \frac{e^{-C_{\theta }|V_{\mathbf{cl}}(s) - u |^{2} }}{|V_{\mathbf{cl}}(s) - u |^{2-\kappa}} 
 |\partial_{v} f^{m-\ell}(s,X_{\mathbf{cl}} (s), u)|
 \mathrm{d}u  \mathrm{d}s \\
\lesssim  & \   \big[1+ P(|| e^{\theta |v|^{2}} \partial_{t} f_{0} ||_{\infty})  + P( || e^{\theta |v|^{2}} f_{0} ||_{\infty})\big]\\
 & \times \int_{0}^{t} \sum_{\ell=0}^{\ell_{*}(0)} \mathbf{1}_{[t^{\ell+1}, t^{\ell})} (s)   \int_{\mathbb{R}^{3}}  e^{-\varpi \langle v\rangle t} e^{-\varpi \langle u\rangle s}   e^{-C |v|^{2}}
 \frac{  |v| \langle u\rangle^{2} \alpha(x,v)^{1/2}   }{    |u| \langle v\rangle^{2} \alpha(X_{\mathbf{cl}}(s),u)^{1/2} } 
  \frac{e^{-C_{\theta}|V_{\mathbf{cl}}(s)-u|^{2} }}{|V_{\mathbf{cl}}(s)-u|^{2-\kappa}}   \mathrm{d}u  \mathrm{d}s\\
&\times
\sup_{0 \leq s\leq t} \max_{0 \leq \ell \leq m} || e^{-\varpi \langle u\rangle s}   \frac{|u| \alpha(x,u)^{1/2}}{\langle u\rangle^{2}}    \partial_{v} f^{m-\ell}(s,x, u)
||_{\infty}
 .
\end{split}
\end{equation}\notag
Now we choose $ \beta^{\prime} \in( \frac{1}{2},1)$ and use $\alpha(x,u) \lesssim |u|^{2}$ to have
\begin{equation}\notag
\frac{1}{[\alpha(X_{\mathbf{cl}}(s), u)]^{1/2}  } \lesssim \frac{|u|^{2(\beta^{\prime} -\frac{1}{2}   )}}{[\alpha(X_{\mathbf{cl}}(s), u)]^{\beta^{\prime} }}.
\end{equation}
Now we use (\ref{exponent}) to bound the integration by
\[
  \int_{0}^{t} \sum_{\ell=0}^{\ell_{*}(0)} \mathbf{1}_{[t^{\ell+1}, t^{\ell})} (s)   \int_{\mathbb{R}^{3}}  e^{-\varpi \langle v\rangle (t-s)}
 \frac{|v|}{|u|}
 \frac{|u|^{2 \beta^{\prime} -1} \alpha(x,v)^{ 1/2}  }{  |V_{\mathbf{cl}}(s) -u|^{2-\kappa}\alpha(X_{\mathbf{cl}}(s),u)^{\beta^{\prime}} }
 e^{ -C |v|^{2} }
 e^{-C_{\theta} |V_{\mathbf{cl}}(s)-u|^{2}}   \mathrm{d}u  \mathrm{d}s
\]
Now we use $|u|^{2\beta^{\prime}-1} \leq \langle v\rangle ^{2\beta^{\prime}-1} \langle u-v\rangle ^{2\beta^{\prime}-1} $ and we apply (\ref{specular_nonlocal_u}) to bound this integration by
\[
 O( {\delta})   \langle v\rangle^{-2 + 2\beta^{\prime}} \alpha(x,v)^{1-\beta^{\prime}} \ \lesssim \ O(\delta),
\]
Hence
\begin{equation}\notag
\begin{split}
&e^{-\varpi \langle v\rangle t} \frac{|v| \alpha(x,v)^{1/2}}{\langle v\rangle^{2}} \{ (\ref{II_3})_{\mathbf{v}} + (\ref{II_6})_{\mathbf{v}}  \}\\
  \lesssim  &\big[1+ P( || e^{\theta|v|^{2}} f_{0} ||_{\infty})   + P( || e^{\theta |v|^{2}}\partial_{t} f_{0} ||_{\infty}) \big]
 \sup_{0 \leq s\leq t} \max_{0 \leq \ell \leq m} || e^{-\varpi \langle v\rangle s}
\frac{|u| \alpha(x,u)^{1/2}}{\langle u\rangle^{2}}
  \partial_{v} f^{m-\ell}(s,x, u)
||_{\infty}
 \\
 & \times \Big\{
 \frac{ O({\delta})    }{\langle v\rangle [\alpha(x,v)]^{\beta^{\prime}-1/2}} \alpha(x,v)^{1/2} e^{-C_{\theta}|v|^{2}}+O(\delta) \Big\}\\
  \lesssim  & \big[1+ P( || e^{\theta |v|^{2}} f_{0} ||_{\infty})   + P( || e^{\theta |v|^{2}}\partial_{t} f_{0} ||_{\infty}) \big]
\\
  & +   O( {\delta})  \big[ P(  || e^{\theta |v|^{2}} f_{0} ||_{\infty}) + P( || e^{\theta |v|^{2}} \partial_{t}f_{0} ||_{\infty}) \big]
 \sup_{0 \leq s\leq t} \max_{0 \leq \ell \leq m} || e^{-\varpi \langle v\rangle s} 
\frac{|u| \alpha(x,u)^{1/2}}{\langle u\rangle^{2}}
 \partial_{v} f^{m-\ell}(s,x, u)
||_{\infty}
.
\end{split}
\end{equation}
Now we gather all the estimates with small $0<\delta\ll 1$ to close the estimate. Then we follow the exactly same argument as the specular case and this complete the proof of Theorem \ref{main_bb}.

\end{proof}

\section*{{{Appendix.} Non-Existence of Second Derivatives}}

In the previous theorem, we consider the \textit{first-order derivative}
of the Boltzmann solution with {several boundary conditions}. Now we show that some second order spatial derivative does not exist up to
the boundary in general so that our result is quite optimal.

Assume that all the second order spatial derivatives exist away from the grazing set $\gamma_0=\{(x,v) \in \partial\Omega\times\mathbb{R}^3: n(x)\cdot v=0\}$ but up to some boundary $\partial\Omega\times\mathbb{R}^3$.
Taking the normal derivative $\partial _{n}= n(x)\cdot \nabla_{x}=\frac{\nabla_{x} \xi(x)}{|\nabla_{x} \xi(x)|}\cdot \nabla_{x}$ to the Boltzmann equation
directly yields
\[
v\cdot \partial_{n} \nabla_{x} f = - \partial_{n} \partial_{t} f - \nu(\sqrt{\mu}f) \partial_{n}f 
+  \underbrace{\partial_{n} \Gamma_{\mathrm{gain}}(f,f) - \partial_{n} \nu(\sqrt{\mu}f) f}. 
\]

From previous Theorem we know that $\partial_{n} \partial_{t} f, \ \nu(\sqrt{\mu}f) \partial_{n}f \sim \frac{1}{\alpha^{a}}$ with some $a>0$. In this section we show that the underbraced term blows up at the boundary with any velocity for symmetric domains. 

Assume $f_{0} \sim (\sqrt{\mu})^{1-\delta}$ for some $0< \delta \ll1.$ Then there exists $\mathbf{k}_{f_{0}}(v,u)$ such that
\[
\Gamma_{\mathrm{gain}}(f,f_{0}) + \Gamma_{\mathrm{gain}}(f_{0}, f) - \nu(\sqrt{\mu} f) f_{0} := \int_{\mathbb{R}^{3}} \mathbf{k}_{f_{0}}(v,u) f(u) \mathrm{d}u.
\]
 
%
%
%

First consider the diffuse reflection boundary condition. Theorem 2
plays an important role in our proof.

\begin{proposition}[Diffuse BC] Assume $\Omega= \{ x\in \mathbb{R}^{3} : |x| <1 \}$ and $\xi(x) = |x|^{2}-1$. Assume the initial datum $f_{0}$ satisfies, for some $x_{0} \in\partial\Omega,$
\begin{equation}\label{condition_diffuse}
 \Big[\int_{ n(x_{0})\cdot u_{\tau}=0}\mathbf{k}_{f_{0}}(v,u)   u\cdot n(x_{0}) \partial_{n} f_{0} (x_{0}, u) \mathrm{d}u_{\tau}\Big]_{u\cdot n(x_{0})=0} >C >0.
\end{equation} 
Then there exist $t>0 $ such that for all $v\in\mathbb{R}^{3},$
\begin{equation}\label{blow_up}
\partial_{n}  \Gamma_{\mathrm{gain}} (f,f)(t,x_{0},v) - \partial_{n}\nu(\sqrt{\mu}f) f(t,x_{0},v)=\infty.
\end{equation}
\end{proposition}

We remark that for $0<\theta <\frac{1}{4}$ we have $%
\sup_{t}||e^{\theta |v|^{2}}f(t)||_{\infty }\lesssim ||e^{\theta
|v|^{2}}f_{0}||_{\infty }$ due to Lemma \ref{local_existence} or \cite{Guo10,EGKM} and $|| \alpha^{1/2}%
\partial f(t)||_{\infty }\lesssim 1$ due to Theorem 2. We also remark that the condition (\ref{condition_diffuse}) is very natural for the diffuse BC.
\begin{proof}
We denote the different quotient
\begin{equation}\notag
\begin{split}
\bigtriangleup_{\varepsilon} f(t,x , v) := \frac{f(t, x  + \varepsilon[- n(x )],v   )   - f(t,x , v)}{\varepsilon}.
\end{split}
\end{equation}
Then 
\begin{equation}\notag
\begin{split}
\bigtriangleup_{\varepsilon} \{  \Gamma_{\mathrm{gain}} (f,f)    \}  - \nu(\sqrt{\mu}  \bigtriangleup_{\varepsilon}f ) f& = \Gamma_{\mathrm{gain}}  (\bigtriangleup_{\varepsilon}  f,f ) + 
\Gamma_{\mathrm{gain}}  ( f, \bigtriangleup_{\varepsilon} f ) - \nu(\sqrt{\mu}  \bigtriangleup_{\varepsilon}f ) f.
\end{split}
\end{equation}
Assuming $f\sim  f_{0 } \sim (\sqrt{\mu})^{1- \delta}$ for $0<\delta\ll 1,$ we have
\begin{equation}\label{diff_q_gamma}
\begin{split}
& \Gamma_{\mathrm{gain}}  (\bigtriangleup_{\varepsilon}  f,f ) + 
\Gamma_{\mathrm{gain}}  ( f, \bigtriangleup_{\varepsilon} f ) - \nu(\sqrt{\mu}  \bigtriangleup_{\varepsilon}f ) f\\ 
&\sim   \int_{\mathbb{R}^{3}} {\mathbf{k}_{f_{0}}}(v,u)  \bigtriangleup_{\varepsilon}f   ( x,u) \mathrm{d}u \sim \int_{ \mathbb{R}^{3}} \mathbf{k}_{f_{0}} (v,u)  
\frac{ f( x-\varepsilon n(x), u) - f( x,u) }{\varepsilon} \mathrm{d}u,
\end{split}
\end{equation}
where $\mathbf{k}_{f_{0}}(v,u) \sim \mathbf{k}(v,u)$ in (\ref{k_estimate}) with slightly different exponents. For simplicity let us assume $\mathbf{k}_{f_{0}}(v,u)$ is bounded. We split as
\begin{eqnarray}  
&& \int_{ \mathbb{R}^{3}} \mathbf{k}_{f_{0}} (v,u)  
\frac{ f(t,x-\varepsilon n(x), u) - f(t,x,u) }{\varepsilon} \mathrm{d}u \nonumber\\
&=&  \underbrace{\int_{|n(x)\cdot u| \leq \varepsilon}}_{{\mathbf{I}}} +  \underbrace{\int_{\varepsilon \leq |n(x)\cdot u| \leq\sigma } }_{\mathbf{II}}+   \underbrace{\int_{\sigma \leq |n(x)\cdot u|   } }_{\mathbf{III}}.\label{split_nonexistence}
\end{eqnarray}
The first term is bounded as $\mathbf{I} \lesssim O(1) || e^{\theta |v|^{2}} f ||_{\infty}.$ The last term is bounded due to Theorem \ref{weigh_W1p}. Since $\xi(x)=|x|^{2}-1$, for all $0< r< \varepsilon \ll 1,$
\begin{equation}\label{ball_n}
\begin{split}
\nabla \xi(x- r n(x))\cdot u &= \nabla \xi(x) \cdot u - \int^{r}_{0} \big\{ \nabla \xi(x) \cdot \nabla^{2} \xi(x- r^{\prime} n(x)) \cdot u \big\}\mathrm{d} r^{\prime}\\
&= \nabla \xi(x)\cdot u - 2\int^{r}_{0} \nabla\xi(x) \cdot u \mathrm{d}r^{\prime}\\
&= \nabla \xi(x) \cdot u + O(\varepsilon) |\nabla \xi(x) \cdot u| \\
& \sim  \nabla \xi(x) \cdot u.
\end{split}
\end{equation}
Therefore $\sigma \leq |n(x)\cdot u|$ implies $\sigma \lesssim \sqrt{\alpha(x,u)}$ and 
\begin{equation}\notag
\begin{split}
\mathbf{III} & \ \lesssim \  || e^{- \varpi \langle v\rangle t} \sqrt{\alpha} \nabla_{x} f (t) ||_{\infty} \int_{\sigma \lesssim \sqrt{\alpha}} \frac{e^{\varpi \langle u \rangle t}}{\sqrt{\alpha}}  \mathbf{k}_{f_{0}}(v,u)\mathrm{d}u\\
& \ = \ \int_{\sigma \leq \sqrt{\alpha}, |u| \leq N}   + \int_{\sigma \leq \sqrt{\alpha}, |u| \geq N}   \  \lesssim \ \frac{ O(1) + e^{C Nt}}{\sigma}. 
\end{split}
\end{equation}
For the second term of (\ref{split_nonexistence}) we use (\ref{ball_n}) to conclude, for $0 \leq r \leq \varepsilon$, 
\[
\varepsilon \ \lesssim \ |n(x-rn(x)) \cdot u| \  \lesssim \sigma. 
\]
Therefore $f(t,x-\varepsilon n(x), u)$ is differentiable so that 
\begin{equation}\label{n_epsilon}
\frac{f(t, x-\varepsilon n (x), u) - f(t,x,u)}{\varepsilon} = \int_{0}^{1} 
\partial_{n} f(t,x-\varepsilon r n(x), u)\mathrm{d}r.
\end{equation}
We further split $\mathbf{II}$ as
\[
\mathbf{II} \ =  \ \underbrace{\int_{\substack{ \varepsilon \leq |n(x)\cdot u| \leq \sigma \\  \frac{1}{N}\leq |  u  |  \leq N}}}_{\mathbf{II}_{a}} \ + \   \underbrace{\int_{\substack{ \varepsilon \leq |n(x)\cdot u| \leq \sigma \\  |u| \leq \frac{1}{N},| u  |  \geq N}} }_{\mathbf{II}_{b}}. 
\]
For the second term we use Theorem \ref{weigh_W1p} to have
\begin{equation}\label{II_b}
\begin{split}
\mathbf{II}_{b} & \lesssim e^{-N} \int_{0}^{1} \mathrm{d}r \int_{\varepsilon \lesssim |u_{n} | \lesssim \sigma} \mathrm{d} u_{n}\int_{|u_{\tau}| \gtrsim N } \mathrm{d}u_{\tau} \ 
\mathbf{k}_{f_{0}} (v,u) \partial_{n} f(t, x- \varepsilon r n(x), u)\\
& \lesssim e^{-N} \int_{0}^{1} \mathrm{d}r \int_{\varepsilon \lesssim |u_{n} | \lesssim \sigma} \mathrm{d} u_{n}\int_{|u_{\tau}| \gtrsim N } \mathrm{d}u_{\tau} \ 
\frac{\mathbf{k}_{f_{0}} (v,u) }{ \sqrt{ |u_{n}|^{2} + C r \varepsilon N^{2}  }  },
\end{split}
\end{equation}
where we used
\[
\xi( x- \varepsilon r n(x)) = {\xi(x)}  + C \varepsilon r = C \varepsilon r .
\]

The main term is $\mathbf{II}_{a}:$
\begin{equation}\notag
\begin{split}
 \mathbf{II}_{a} & = \int_{0}^{1} \mathrm{d}r \iint_{ \substack{\varepsilon \lesssim |u_{n}| \lesssim \sigma  \\ \frac{1}{N} \leq |u| \leq N  }}  \mathrm{d}u_{\tau} \mathrm{d}u_{n}  \ \mathbf{k}_{f_{0}} (v,u) \partial_{n} f(t,x-\varepsilon r n(x),u).
\end{split}
\end{equation}
From (\ref{41}), for $\varepsilon \lesssim |u_{n}| \lesssim \sigma$ and $\frac{1}{N} \leq |u| \leq N,$ 
\[
t_{\mathbf{b}}(x-\varepsilon r n(x), u)\lesssim \frac{\sqrt{\alpha (x-\varepsilon r n(x), u)}}{|u|^{2}} \lesssim \frac{\sqrt{\sigma^{2}+ \varepsilon r N^{2}}}{\frac{1}{N^{2}}} \lesssim N^{2} \sqrt{\sigma^{2} + \varepsilon  N^{2}}.
\]
Let $x(r)= x-\varepsilon r n(x).$ For $\varepsilon \lesssim |u_{n}| \lesssim \sigma$ and $\frac{1}{N} \leq |u| \leq N$ and $t \gtrsim N^{2} \sqrt{\sigma^{2} + \varepsilon N ^{2}},$
\begin{equation}\notag
\begin{split}
&\partial_{n}f(t, x(r) ,u)\\
= & \ n(x  (r))\cdot  \nabla_{x } \Big\{ f(t-t_{\mathbf{b}}  , x_{\mathbf{b}},u) + \int^{t_{\mathbf{b}}}_{0} [ \Gamma_{\mathrm{gain}}(f,f) - \nu(F)f ](t-s, x(r)-su,u) \mathrm{d}s  \Big\}\\
=& \ \sum_{i=1}^{2} n(x (r))\cdot \tau_{i} (x_{\mathbf{b}}) \partial_{\tau_{i}} f(t-t_{\mathbf{b}}, x_{\mathbf{b}}, u) 
+ \frac{n(x (r))\cdot n(x_{\mathbf{b}})}{ n(x_{\mathbf{b}}) \cdot u }
\underline{u\cdot n(x_{\mathbf{b}}) \partial_{n} f(t-t_{\mathbf{b}}, x_{\mathbf{b}},u)}
\\
& + \int^{t_{\mathbf{b}}}_{0} n(x(r) )\cdot  \{ \Gamma_{\mathrm{gain}}(\nabla_{x} f,f) + \Gamma_{\mathrm{gain}}(f,\nabla_{x}f) - \nu(\sqrt{\mu}\nabla_{x} f) f - \nu(\sqrt{\mu}f) \nabla_{x}f  \}(t-s, x(r)-su,u) \mathrm{d}s.
\end{split}
\end{equation}
Now we expand in time for the underlined term and  choose $0< t \ll 1$ ($N^{2} \sqrt{\sigma^{2} + \varepsilon N^{2}}\ll 1$) so that
\begin{equation}\notag
\begin{split}
&u\cdot n(x_{\mathbf{b}}) \partial_{n} f(t-t_{\mathbf{b}}, x_{\mathbf{b}}, u)\\
&=  u\cdot n(x_{\mathbf{b}}) \partial_{n} f_{0}( x_{\mathbf{b}}, u)
+\int^{t-t_{\mathbf{b}}}_{0} \{u\cdot n(x_{\mathbf{b}})\} \partial_{t} \partial_{n} f(s, x_{\mathbf{b}}, u)
 \mathrm{d}s\\
 & = u\cdot n(x_{\mathbf{b}}) \partial_{n} f_{0}( x_{\mathbf{b}}, u)
 + O(1)t e^{\varpi N t} || e^{-\varpi \langle v\rangle t} \sqrt{\alpha} \partial_{t} \partial_{n} f(t)||_{\infty} .
 \end{split}
\end{equation}
The tangential derivative term is bounded by
\begin{equation}\notag
\begin{split}
&|n(x  (r)) \cdot \tau_{i} (x_{\mathbf{b}} (x (r),u)) ||\partial_{\tau_{i}} f(t-t_{\mathbf{b}}, x_{\mathbf{b}}, u)|\\
 \lesssim   & \ |n(x_{\mathbf{b}})\cdot \tau_{i} (x_{\mathbf{b}})  + 
 O(t_{\mathbf{b}} (x(r), u) ) u\cdot \nabla_{x} n ( x(r))
||\partial_{\tau_{i}} f(t-t_{\mathbf{b}}, x_{\mathbf{b}}, u)|\\
 \lesssim & \ \frac{\sqrt{\alpha(x_{\mathbf{b}},u)}}{|u|} | \nabla_{x} f(t-t_{\mathbf{b}},x_{\mathbf{b}},u)|\\
  \lesssim& \  Ne^{\varpi Nt} || e^{-\varpi \langle  v\rangle t} \sqrt{\alpha} \nabla_{x} f(t,x,v) ||_{\infty},
\end{split}
\end{equation}
and the time integration terms are bounded by
\begin{equation}\notag
\begin{split}
&|| e^{\theta|v|^{2}} f||_{\infty}   \int^{t_{\mathbf{b}}}_{0} \int_{\mathbb{R}^{3}}\frac{e^{-C|u-u^{\prime}|^{2}}}{|u-u^{\prime}|^{2-\kappa}} 
|\partial_{x} f (t-s,  x(r)-su, u^{\prime})|
\mathrm{d}u^{\prime}\mathrm{d}s\\
& +  N e^{\varpi  N t} || e^{\theta|v|^{2}} f||_{\infty}  || e^{-\varpi \langle v\rangle t} \sqrt{\alpha} \nabla_{x} f(t) ||_{\infty}\\
\lesssim & \ || e^{\theta|v|^{2}} f||_{\infty}  || e^{-\varpi \langle v\rangle t} \sqrt{\alpha} \nabla_{x} f(t) ||_{\infty}  \times e^{\varpi Nt}
\Big\{  \int^{t_{\mathbf{b}}}_{0} \int_{\mathbb{R}^{3}}  e^{-\varpi \langle u\rangle (t-s)}
\frac{e^{-C|u-u^{\prime}|^{2}}}{|u-u^{\prime}|^{2-\kappa}} 
\frac{|u^{\prime}|^{\delta}}{ \alpha( x(r) -su, u^{\prime})^{\frac{1+\delta}{2}}}
\mathrm{d}u^{\prime} \mathrm{d}s
 \Big\}\\
 \lesssim & \ || e^{\theta|v|^{2}} f||_{\infty}  || e^{- \varpi \langle v\rangle t} \sqrt{\alpha} \nabla_{x} f(t) ||_{\infty}  \frac{ e^{\varpi Nt} C_{N}}{[\alpha( x(r), u)]^{ {\delta}/{2}}}.
\end{split}
\end{equation}

Now we plug these estimates into $\mathbf{II}_{a}$ to have
\begin{equation}\notag
\begin{split}
\mathbf{II}_{a} + \mathbf{II}_{b}& \gtrsim \int^{1}_{0}  
 \int_{ \substack{ \varepsilon \lesssim |u_{n}| \lesssim \sigma \\  \frac{1}{N} \lesssim |u| \lesssim N }} \frac{ 1 }{\sqrt{\alpha(x_{0}- \varepsilon r n(x_{0}),u)}}   \Big[\int_{\cdot n(x_{0})\cdot u_{\tau}=0}\mathbf{k}_{f_{0}}(v,u)   u\cdot n(x_{0}) \partial_{n} f_{0} (x_{0}, u) \mathrm{d}u_{\tau}\Big]_{u\cdot n(x_{0})=0} \\
& - \Big\{O(t) e^{-\varpi N t} || e^{-\varpi \langle v\rangle t} \sqrt{\alpha} \partial_{t} \partial_{n} f(t) ||_{\infty} + e^{-N}\Big\} \int^{1}_{0}  
 \iint_{ \substack{ \varepsilon \lesssim |u_{n}| \lesssim \sigma \\  \frac{1}{N} \lesssim |u| \lesssim N }}  \frac{ 1 }{\sqrt{\alpha( x_{0}- \varepsilon r n(x_{0}),u)}}   \\
 & - O(1) N e^{ \varpi Nt } || e^{-\varpi \langle v\rangle t} \sqrt{\alpha} \nabla_{x} f(t) ||_{\infty}\\
 &- O_{N}(1) e^{\varpi N t}  || e^{\theta |v|^{2}} f_{0}||_{\infty} || e^{-\varpi \langle v\rangle t} \sqrt{\alpha} \nabla_{x} f(t) ||_{\infty}   \int^{1}_{0}  
 \iint_{ \substack{ \varepsilon \lesssim |u_{n}| \lesssim \sigma \\  \frac{1}{N} \lesssim |u| \lesssim N }}  \frac{ 1 }{[{\alpha( x_{0}- \varepsilon r n(x_{0}),u)}]^{\delta/2}}. 
 \end{split}
\end{equation}
Due to (\ref{condition_diffuse}), for $N\gg1$ and $t\ll 1$ with $N^{2} \sqrt{\sigma^{2} + \varepsilon N^{2}}\ll1$
\begin{equation}\notag
\begin{split}
\mathbf{II}& \  \gtrsim  \  \int^{1}_{0}  
 \int_{ \substack{ \varepsilon \lesssim |u_{n}|  \lesssim \sigma \\  \frac{1}{N} \lesssim |u| \lesssim N }} \frac{ 1 }{ \sqrt{ |u_{n}|^{2} + C \varepsilon r |u_{\tau}|^{2}  }   }  \mathrm{d}u_{n} \mathrm{d}u_{\tau} \mathrm{d}r\\
& \ \gtrsim \ \int_{\varepsilon \leq |u_{n}| \leq \sigma}   \frac{N^{2}}{|u_{n}| + \sqrt{\varepsilon}N}  \mathrm{d}u_{\tau} \mathrm{d}u_{n}\\
& \ \gtrsim \  N^{2}\ln \frac{1}{\varepsilon + \sqrt{\varepsilon} N}   - O_{N, \sigma}(1)\\
& \  \gtrsim \ \frac{N^{2}}{2} \ln \frac{1}{\varepsilon} - o(1) \ln\frac{1}{\varepsilon} - O_{N, \sigma}(1) \rightarrow \infty.
\end{split}
\end{equation}
\end{proof}

For the bounce-back case, we identify the condition for non-existence of $\nabla^{2}f$ up to the boundary:
\begin{proposition}[Bounce-Back BC]\label{bb_nonexistence}
Assume $\Omega= \{ x\in \mathbb{R}^{3} : |x| <1 \}$ and $\xi(x) = |x|^{2}-1$. Assume the initial datum $f_{0}$ satisfies, for some $x_{0} \in\partial\Omega$ and some $v_{0} \in\mathbb{R}^{3}$ with $|v_{0}|\sim 1$ with $n(x_{0})\cdot v_{0} =0,$ 
\begin{equation}\label{bb_condition}
 \int_{n(x_{0})\cdot u_{\tau} =0 } \mathbf{k}_{f_{0}} (v,u) v_{0} \cdot \nabla_{x} f_{0} (x_{0}, v_{0}) \mathrm{d}u_{\tau}> C>0,
\end{equation} 
where $u_{\tau} = u- [u\cdot n(x_{0})] n(x_{0}).$ Then there exists $t>0$ we have (\ref{blow_up}).
\end{proposition}
We remark that $v_{0}\cdot \nabla_{x} f_{0}(x_{0},v_{0})$ is rather arbitrary for $v_{0}\cdot n(x_{0})=0.$
\begin{proof} 
We choose $(x,v) \in \bar{\Omega} \times\mathbb{R}^{3}$ so that $x^{\ell}\sim x_{0}$ and $v^{\ell}\sim \pm v_{0}$ for all $\ell\in\mathbb{N}$. Then 
\begin{eqnarray*}  
&& \int_{ \mathbb{R}^{3}} \mathbf{k}_{f_{0}} (v,u)  
\frac{ f(t,x-\varepsilon n(x), u) - f(t,x,u) }{\varepsilon} \mathrm{d}u \nonumber\\
&=&  \int_{|n(x)\cdot u| \leq \varepsilon} +  \int_{\varepsilon \leq |n(x)\cdot u|  }.
\end{eqnarray*}
The first terms is bounded. Due to (\ref{ball_n}) we have $|n(x)\cdot u| \geq \varepsilon$ implies $|n(x-r n(x))\cdot u| \gtrsim \varepsilon$ for all $0\leq r \leq \varepsilon.$ Then by Theorem \ref{main_bb}, the function $f$ is differentiable and the second term equals
\begin{eqnarray}
&&\int_{ |n(x)\cdot u| \geq \varepsilon}  \mathbf{k}_{f_{0}} (v,u)   \frac{ f(t,x-\varepsilon n(x), u) - f(t,x,u) }{\varepsilon} \mathrm{d}u  \nonumber
\\
&=  & \int_{0}^{1} \int_{ |n(x)\cdot u| \geq \varepsilon}  \mathbf{k}_{f_{0}} (v,u)  
\partial_{n} f(t, x- \varepsilon r n(x),u) 
 \mathrm{d}u
\mathrm{d}r\nonumber  \\
&=&  \int_{0}^{1} \int_{ \varepsilon \leq  |n(x)\cdot u| \leq 1,  |\tau(x)\cdot u| \leq N} + \int_{0}^{1} \int_{\varepsilon \leq  |n(x)\cdot u| \leq 1,  |\tau(x)\cdot u| \geq N}  +  \int_{0}^{1} \int_{ |n(x)\cdot u| \geq 1} .  \label{k_bb}
\end{eqnarray}

For the third term of (\ref{k_bb}) we use Theorem \ref{main_bb} to have 
\[
|\partial_{n} f(t,x-\varepsilon r n(x), u)| \lesssim \frac{\langle u\rangle^{2} e^{\varpi \langle u\rangle t}}{\alpha(x-\varepsilon r n(x), u)} \lesssim 
 {\langle u\rangle^{2} e^{\varpi \langle u\rangle t}},
\]
and therefore the third term of (\ref{k_bb}) is bounded. For the second term of (\ref{k_bb}) we use Theorem \ref{main_bb} to bound
\begin{equation}\label{second_bb}
\begin{split}
&|| e^{-\varpi \langle v\rangle t} \frac{\alpha}{\langle v\rangle^{2} } \nabla_{x} f(t)||_{\infty} 
\times  \int_{0}^{1} \mathrm{d}r \int_{\varepsilon}^{1} \mathrm{d} u_{n} \frac{e^{\varpi \delta} N^{2}}{ |u_{n}|^{2} + C r\varepsilon N^{2}} \int_{\mathbb{R}^{2}} \mathrm{d} u_{\tau} \mathbf{k}_{f_{0}}(v,u)\\
&\lesssim e^{-N}
\times  \int_{0}^{1} \mathrm{d}r \int_{\varepsilon}^{1} \mathrm{d} u_{n} \frac{1}{ |u_{n}|^{2} + C r\varepsilon N^{2}}.
\end{split}
\end{equation}

Now we focus on the first and the second terms of (\ref{k_bb}). Set $y= x-\varepsilon r n(x)$ for $|\partial_{n} f(t,x-\varepsilon^{\prime} n(x), u)|.$ We use (\ref{D_duhamel_bb}), Theorem \ref{main_bb}, and Lemma \ref{estimate_bb}, we have
\begin{equation}\notag
\begin{split}
&\partial_{n} f(t,y,u)\\
& = e^{-\int^{t}_{0} \nu(\sqrt{\mu}f)(\tau) \mathrm{d}\tau} \Big\{[\partial_{n} x^{\ell_{*}(0)} -\partial_{n}t^{\ell_{*} (0) } v^{\ell_{*}(0)}] \cdot \nabla_{x} f_{0}(X_{\mathbf{cl}}(0), V_{\mathbf{cl}}(0)) + v^{\ell_{*}(0)}\cdot \nabla_{v} f_{0}(X_{\mathbf{cl}}(0), V_{\mathbf{cl}}(0)) \Big\}\\
& +    O( ||   e^{\theta |v|^{2}} f_{0} ||_{\infty} )
\frac{t^{2}|u|^{2}}{\alpha(y,u)} e^{-\frac{\theta}{4}|u|^{2}} || e^{\theta|v|^{2}} \partial_{t }f_{0 }||_{\infty}
 + O(||  e^{\theta |v|^{2}} f_{0} ||_{\infty}  )  \frac{t(1+t) |u|^{2}}{\alpha(y,u)} e^{-\frac{\theta}{4}|u|^{2}}  
\\
& + O( ||   e^{\theta|v|^{2}} f_{0} ||_{\infty}  )\sup_{0\leq s \leq t} ||e^{-\varpi \langle v\rangle t} \alpha  \partial_{x}f(t) ||_{\infty}\\
&  \ \ \ \times      {\int_{0}^{t} \int_{\mathbb{R}^{3}} \sum_{\ell} \mathbf{1}_{[t^{\ell+1}, t^{\ell})}(s) \frac{e^{-\frac{\theta}{8} |v^{\ell}-u^{\prime}|^{2}}  }{|v^{\ell}-u^{\prime}|^{2-\kappa}}   \frac{|\partial_{n} x^{\ell}|}{  e^{-\varpi \langle u^{\prime} \rangle s} \alpha(X_{\mathbf{cl}}(s), u^{\prime}) } \mathrm{d}u^{\prime} \mathrm{d}s } .
\end{split}
\end{equation} 
Using Lemma \ref{specular_nonlocal}, Lemma \ref{estimate_bb}, and (\ref{exponent}), for $0<\varepsilon \ll 1$ we bound the last integration by
\begin{equation}\notag
\begin{split}
 & e^{\varpi \langle u\rangle t} \frac{|u|}{\sqrt{\alpha(y,u)}}\int_{0}^{t} \int_{\mathbb{R}^{3}} e^{- \varpi \langle u\rangle (t-s)}    \frac{e^{-\frac{\theta}{8} |V_{\mathbf{cl}}(s)-u^{\prime}|^{2}}  }{| V_{\mathbf{cl}}(s)-u^{\prime}|^{2-\kappa}}\frac{1}{\alpha(X_{\mathbf{cl}}(s),u^{\prime})}
 \mathrm{d}u^{\prime} \mathrm{d}s\\
 &  \lesssim O(\varepsilon) e^{\varpi \langle u\rangle t} \frac{|u|}{\langle u\rangle  } \frac{1}{\alpha(y,u)}.
 \end{split}
\end{equation}
Now by the explicit computations in Lemma \ref{estimate_bb}
\begin{equation}\notag
\begin{split}
&e^{-\int^{t}_{0} \nu(\sqrt{\mu}f)(\tau) \mathrm{d}\tau} \Big\{[\partial_{n} x^{\ell_{*}(0)} -\partial_{n}t^{\ell_{*} (0) } v^{\ell_{*}(0)}] \cdot \nabla_{x} f_{0}(X_{\mathbf{cl}}(0), V_{\mathbf{cl}}(0)) + v^{\ell_{*}(0)}\cdot \nabla_{v} f_{0}(X_{\mathbf{cl}}(0), V_{\mathbf{cl}}(0)) \Big\}\\
\geq & \ e^{-t\langle u\rangle ||  e^{\theta|v|^{2}} f_{0 }||_{\infty} } 
 \underbrace{\Big\{ \ell_{*}(0)  \frac{n(y)\cdot \nabla \xi(x^{1})}{ v\cdot \nabla \xi(x^{1})}   + (\ell_{*}(0)-1) \frac{n(y)\cdot \nabla \xi(x^{2})}{-v\cdot \nabla \xi(x^{2})} \Big\} }
 V_{\mathbf{cl}}(0) \cdot \nabla_{x} f_{0}(X_{\mathbf{cl}}(0), V_{\mathbf{cl}}(0))\\
& - C_{\xi}\frac{|u|}{\sqrt{\alpha(y,u)}} | \nabla_{x} f_{0} (X_{\mathbf{cl}}(0), V_{\mathbf{cl}}(0))|   -C_{\xi} |u| |\nabla_{v} f_{0} (X_{\mathbf{cl}}(0), V_{\mathbf{cl}}(0)) |\\
\geq &  \ e^{-t\langle u\rangle ||  e^{\theta|v|^{2}} f_{0 }||_{\infty} }  O_{\xi}(1) \frac{t|u|^{2}}{\alpha(y,u)}  V_{\mathbf{cl}}(0) \cdot \nabla_{x} f_{0}(X_{\mathbf{cl}}(0), V_{\mathbf{cl}}(0))\\
&- \frac{O_{\xi}(1+ t|u|)}{\sqrt{\alpha(y,u)}} |u||\nabla_{x}f_{0}(X_{\mathbf{cl}} (0) , V_{\mathbf{cl}}(0))|  -C_{\xi} |u| |\nabla_{v} f_{0} (X_{\mathbf{cl}}(0), V_{\mathbf{cl}}(0)) |,
\end{split}
\end{equation}
where for the underbraced term we used 
\begin{equation}\label{diff_nv}
\begin{split}
& n(y)\cdot \left\{ \frac{n(x_{\mathbf{b}})}{n(x_{\mathbf{b}})\cdot v}-\frac{%
n(y)}{n(y)\cdot v}\right\}  \\
 =& \ \frac{n(y)\cdot \nabla \xi (y-t_{\mathbf{b}}v)\big(\nabla \xi (y)\cdot v%
\big)-n(y)\cdot \nabla \xi (y)\big(\nabla \xi (y-t_{\mathbf{b}}v)\cdot v\big)%
}{\big(\nabla \xi (y)\cdot v\big)\big(\nabla \xi (y-t_{\mathbf{b}}v)\cdot v%
\big)} \\
 =& \ \frac{t_{\mathbf{b}}\big\{\big(n(y)\cdot \lbrack v\cdot \nabla ]\nabla
\xi (y-\tilde{\tau}v)\big)(\nabla \xi (y)\cdot v)-n(y)\cdot \nabla \xi (y)%
\big(v\cdot \nabla ^{2}\xi (y-\tilde{\tau}v)\cdot v\big)\big\}}{\big(\nabla
\xi (y)\cdot v\big)\big(-\nabla \xi (y-t_{\mathbf{b}}v)\cdot v\big)} \\
 =& \  
- \frac{t_{\mathbf{b}}}{(-\nabla \xi (x_{\mathbf{b}})\cdot v)}  
\frac{\big(v\cdot \nabla ^{2}\xi (y-\tilde{\tau}v)\cdot v\big)}{n(y)\cdot v}+
 \frac{t_{\mathbf{b}}}{(-\nabla \xi (x_{\mathbf{b}})\cdot v)}
\big(n(y)\cdot \lbrack v\cdot \nabla ]\nabla \xi (y-\tilde{\tau}v)\big)%
   \\
 := & \ -\frac{A(y,v)}{n(y)\cdot v}+B(y,v),
 \end{split}
\end{equation}%
where for some $\tilde{\tau}\in \lbrack 0,t_{\mathbf{b}}],$ and from (\ref{40}), (\ref{41}), and the Velocity
lemma (Lemma \ref{velocity_lemma}) we have $A\geq 0$ and 
\begin{equation*}
A(y,v)\geq C_{\xi }\frac{v}{|v|}\cdot \nabla ^{2}\xi (y-\tilde{\tau}v)\cdot
\frac{v}{|v|}\gtrsim _{\Omega }1,\ \ \ B(y,v)\sim _{\Omega }\frac{1}{|v|},
\end{equation*}%
and therefore finally the underbraced term has the following explicit lower bound:
\begin{equation}\notag
\begin{split}
&\ell_{*}(0) \Big[  \frac{n(y)\cdot \nabla \xi(x^{1})}{v\cdot \nabla \xi(x^{1})} - \frac{n(y)\cdot \nabla \xi(x^{2})}{v\cdot \nabla \xi(x^{2})} \Big] + \frac{n(y)\cdot \nabla \xi(x^{2})}{v\cdot \nabla \xi(x^{2})}\\
&= \ell_{*}(0) \frac{A(y,v)}{n(x^{1})\cdot v} + \ell_{*}(0) B(y,v)  + O\Big(\frac{1}{n(x^{1})\cdot v}\Big)\\
&= O_{\xi}(1)  \frac{t |v|^{2}}{\alpha(y,v)}+ \frac{O_{\xi}(1+t|v|) }{ \sqrt{\alpha(y,v)}}
.
\end{split}
\end{equation} 
Therefore
\begin{equation}\label{fn_bb}
\begin{split}
\partial_{n} f(t,y,u) &\geq e^{- t\langle u\rangle ||  e^{\theta |v|^{2}}  f_{0}||_{\infty}  } O_{\xi}(1) \frac{t|u|^{2}}{\alpha(y,u)} V_{\mathbf{cl}}(0)\cdot \nabla_{x} f_{0}(X_{\mathbf{cl}}(0), V_{\mathbf{cl}}(0)) \\
& - \frac{O_{\xi}(1+ t|u|)}{\sqrt{\alpha(y,u)}} |u||\nabla_{x}f_{0}(X_{\mathbf{cl}} (0) , V_{\mathbf{cl}}(0))|  -C_{\xi} |u| |\nabla_{v} f_{0} (X_{\mathbf{cl}}(0), V_{\mathbf{cl}}(0)) |\\
&   -  {O}( ||   e^{\theta |v|^{2}} f_{0} ||_{\infty})
\frac{t^{2}|u|^{2}}{\alpha(y,u)} e^{-\frac{\theta}{4}|u|^{2}} || e^{\theta|v|^{2}} \partial_{t }f_{0 }||_{\infty} \\
& -    {O}( ||  e^{\theta |v|^{2}} f_{0} ||_{\infty})
\frac{t(1+t) |u|^{2}}{\alpha(y,u)} e^{-\frac{\theta}{4}|v|^{2}}   \\
& -  {O}(||   e^{\theta|v|^{2}}  f_{0} ||_{\infty}   ) \sup_{0 \leq s\leq t} || e^{-\varpi \langle v\rangle t} \alpha \partial_{x}  f (t)||_{\infty} \times O( \varepsilon) e^{\varpi \langle u\rangle t} \frac{|u|}{\langle u\rangle } \frac{1}{\alpha(y,u)} .
\end{split}
\end{equation}
Choose $y=x-\varepsilon r n(x).$
First consider the first contribution of (\ref{fn_bb}). It has following lower bound as
\begin{eqnarray}  
&&\int^{1}_{0} \int_{\varepsilon \leq |n(x)\cdot u| \leq 1} \int_{|\tau(y)\cdot u| \leq N}   \nonumber\\
&\gtrsim &  \int_{0}^{1} \mathrm{d}r \int_{\varepsilon \leq |u_{n}| \leq 1} \mathrm{d}u_{n} \int_{|u_{\tau}| \leq N}   \mathrm{d}u_\tau
\frac{1}{|u_{n}|^{2} + C r \varepsilon N^{2}} \mathbf{k}_{f_{0} }(v,u) V_{\mathbf{cl}}(0) \cdot \nabla_{x} f_{0} (X_{\mathbf{cl}}(0), V_{\mathbf{cl}}(0))\nonumber  \\
&&-
 \notag \\
& \sim & \int_{0}^{1} \mathrm{d}r \int_{\varepsilon \leq |u_{n}| \leq 1}  \frac{ \mathrm{d}u_{n}}{|u_{n}|^{2} + C r \varepsilon N^{2}} \int_{|u_{\tau}| \leq N}   \mathrm{d}u_\tau
 \mathbf{k}_{f_{0} }(v,u)  v_{0} \cdot \nabla_{x } f_{0} (x_{0}, v_{0}) 
 .\label{bb_lower}
\end{eqnarray}
Now we use the condition (\ref{bb_condition}) for $\varepsilon \sim 0$ and $u_{n} \sim 0$ 
\begin{equation}\notag
\begin{split}
&\int_{|u_{\tau} |\geq N}   \mathbf{k}_{f_{0}} (v,u)V_{\mathbf{cl}}(0)\cdot \nabla_{x} f_{0} (X_{\mathbf{cl}}(0), V_{\mathbf{cl}}(0)) \mathrm{d}u_{\tau}\\
\sim &\int_{n(x_{0})\cdot u_{\tau} =0 } \mathbf{k}_{f_{0}} (v_{0},u) v_{0} \cdot \nabla_{x} f_{0} (x_{0}, v_{0}) \mathrm{d}u_{\tau} = C \neq 0.
\end{split}
\end{equation}
We combine this term with the second term of (\ref{k_bb}) to conclude
\begin{equation} \label{bb_dominant}
\begin{split}
(\ref{bb_lower}) - (\ref{second_bb})&\gtrsim \{C- e^{-N} \}
\int_{\varepsilon}^{1}   \mathrm{d}  u_{n}  \int_{0}^{1}   \frac{ \mathrm{d}r}{ |u_{n}|^{2} + C r\varepsilon N^{2}}\\
& \gtrsim 
\int_{\varepsilon}^{1} \frac{1}{C \varepsilon N^{2}} \ln \Big(1 + \frac{C\varepsilon N^{2}}{ |u_{n}|^{2}}  \Big)\mathrm{d}u_{n} \ \gtrsim \ \frac{1}{N^{2}} \frac{1}{\varepsilon} .
 \end{split}
\end{equation}

Now the all the other terms of (\ref{fn_bb}) except the first term are bounded by
\begin{equation}\notag
\begin{split}
&\int_{\varepsilon}^{1} \mathrm{d}u_{n} \frac{1}{|u_{n}|} + O( || e^{\theta |v|^{2}} f_{0} ||_{\infty}) \int_{\varepsilon}^{1} \mathrm{d}u_{n} \frac{1}{|u_{n}|^{2}}\\
& \lesssim |\ln \varepsilon|+  O( || e^{\theta |v|^{2}} f_{0} ||_{\infty}) \frac{1}{\varepsilon}.
\end{split}
\end{equation}
 Finally we choose large $N>0$ and small $\delta>0$ and small $|| e^{\theta |v|^{2}} f_{0} ||_{\infty}$ to conclude for small $\varepsilon>0$
\[
(\ref{k_bb}) \gtrsim  \frac{1}{\varepsilon}.
\] 
 Therefore we conclude (\ref{blow_up}).
\end{proof}

In order to show the non-existence of $\nabla^{2}f$ up to the boundary for the specular reflection BC (Proposition \ref{sing_specular}) we first obtain the explicit lower bound of (\ref{lemma_Dxv}) with a a lower dimensional symmetric domain, 2D disk. 

\begin{example}
Let $\Omega= \{\bar{x}=(x_{1},x_{2}) \in\mathbb{R}^{2} : |x_{1}|^{2} + |x_{2}|^{2} <1\}$.  Define
\begin{equation}\notag
\begin{split}
r&:=\sqrt{x_{1}^{2} + x_{2}^{2}}\in [0,1], \ \
\theta  \in [0,2\pi) \ \ \text{such that}  \ \ (\cos\theta,\sin\theta) = \frac{1}{\sqrt{x_{1}^{2}  + x_{2}^{2}}} (x_{1},x_{2}) ,\\
\bar{v}_{n}  &:=  v_{1} \cos\theta + v_{2} \sin\theta, \ \
\bar{v}_{\theta}  :=  -v_{1} \sin\theta + v_{2} \cos\theta.
\end{split}
\end{equation}
We claim that as $\alpha \rightarrow 0$ $($therefore $r\sim 1, \bar{v}_{n}\sim 0$$)$ asymptotically 
\begin{equation}\label{lower_Dxv}
\begin{split}
|\partial_{n} \bar{X}_{\mathbf{cl}}(s;t,x,v)\cdot \bar{n}^{\perp} (\bar{X}_{\mathbf{cl}}(s;t,x,v)) | &\sim \frac{|t-s||\bar{v}_{\theta}|^{2}}{\sqrt{\bar{v}_{n}^{2} + (1-r^{2}) \bar{v}_{\theta}^{2}}}
\sim  \frac{|t-s||\bar{v}|^{2}}{ \sqrt{\alpha(x,v)}} ,\\
|\partial_{n} \bar{V}_{\mathbf{cl}}(s;t,x,v)\cdot \bar{n}(\bar{X}_{\mathbf{cl}}(s;t,x,v)) |& \sim \frac{  |t-s||\bar{v}|^{4}}{ \bar{v}_{n}^{2} + (1-r^{2}) \bar{v}_{\theta}^{2}}
\sim \frac{|t-s||\bar{v}|^{4}}{\alpha(x,v)}
,
\end{split}
\end{equation}
where $\bar{n}^{\perp} = \left[\begin{array}{cc} 0 & - 1 \\ 1 & 0 \end{array}\right]\bar{n}.$ \end{example}
\begin{proof}
 
Explicitly $x = (r\cos\theta, r\sin\theta,x_{3}),$ and $v= (\bar{v}_{n} \cos\theta- \bar{v}_{\theta} \sin\theta, \bar{v}_{n} \sin\theta + \bar{v}_{\theta} \cos\theta, v_{3}),$ and
\begin{equation}\notag
\begin{split}
x^{\ell}& = (\cos\theta^{\ell}, \sin\theta^{\ell}, x_{3}-(t-t^{\ell}) v_{3}),\\
v^{\ell} & = ( \sqrt{\bar{v}_{n}^{2} + \bar{v}_{\theta}^{2}} \cos\psi^{\ell} , \sqrt{\bar{v}_{n}^{2} + \bar{v}_{\theta}^{2}} \sin\psi^{\ell}
,v_{3}),\\
t^{1} &= t- \frac{r | \bar{v}_{n}| + \sqrt{(1-r^{2}) \bar{v}_{\theta}^{2}  + \bar{v}_{n}^{2} }}{ \bar{v}_{n}^{2} + \bar{v}_{\theta}^{2}},\\
t^{\ell} & =  t- \frac{r | \bar{v}_{n}| +(2\ell -1) \sqrt{(1-r^{2}) \bar{v}_{\theta}^{2}  + \bar{v}_{n}^{2} }}{ \bar{v}_{n}^{2} + \bar{v}_{\theta}^{2}}, 
\end{split}
\end{equation}
and
\[
\ell_{*}(s;t,x,v) \leq \frac{(t-s)|\bar{v}|^{2}}{2\sqrt{(1-r^{2}) \bar{v}_{\theta}^{2} + \bar{v}_{n}^{2}}} - \frac{r| \bar{v}_{n}|}{2\sqrt{(1-r^{2}) \bar{v}_{\theta}^{2} + \bar{v}_{n}^{2}}} + \frac{1}{2} < \ell_{*}(s;t,x,v) + 1,
\]
where, for $\ell\geq 1$
\begin{equation}\notag
\begin{split}
\theta^{0}&=\theta, \ \ \ \theta^{\ell}=  \theta- \cos^{-1}  \Big(   \frac{\bar{v}_{\theta}}{\sqrt{\bar{v}_{\theta}^{2} + \bar{v}_{n}^{2} }}\Big) -(2\ell-1)\cos^{-1}  \Big( \frac{r \bar{v}_{\theta}}{\sqrt{ \bar{v}_{n}^{2} + \bar{v}_{\theta}^{2}}}\Big),\\
\psi^{0} &=\cos^{-1} \Big( \frac{ \bar{v}_{n} \cos\theta - \bar{v}_{\theta} \sin\theta }{\sqrt{ \bar{v}_{n}^{2} + \bar{v}_{\theta}^{2}}}\Big), \ \ \  \psi^{\ell}  = \psi^{0} -2\ell  \cos^{-1} \Big( \frac{r \bar{v}_{\theta}}{\sqrt{ \bar{v}_{n}^{2} + \bar{v}_{\theta}^{2} }}\Big).
\end{split}
\end{equation}
Therefore, if $t^{\ell+1}<s<t^{\ell},$
\begin{equation}\notag
\begin{split}
X_{\mathbf{cl}}(s)= x^{\ell}-(t^{\ell}-s) v^{\ell}, \ \ \ V_{\mathbf{cl}}(s) = v^{\ell},
\end{split}
\end{equation}
and
\begin{equation}\notag
\begin{split}
r(s)&=|\bar{X}_{\mathbf{cl}}(s)| = |\bar{x}^{\ell}-(t^{\ell}-s)\bar{v}^{\ell}|,\\
\bar{v}_{n}(s)&= \bar{V}_{\mathbf{cl}}(s) \cdot \frac{\bar{X}_{\mathbf{cl}}(s)}{|\bar{X}_{\mathbf{cl}}(s)|}, \ \
\bar{v}_{\theta}(s)   = \bar{V}_{\mathbf{cl}}(s)\cdot \left(\begin{array}{cc} 0 & -1 \\ 1 & 0\end{array} \right)  \frac{\bar{X}_{\mathbf{cl}}(s)}{|\bar{X}_{\mathbf{cl}}(s)|},\\
 v_{3}(s)&=v_{3}.
\end{split}
\end{equation}
Directly
\begin{equation}\notag
\begin{split}
& \partial_{\theta} \bar{v}_{n} = v_{\theta}, \ \ \ \partial_{\theta} \bar{v}_{\theta} = -\bar{v}_{n},
\\
&\partial_{n} \cos^{-1} \Big( \frac{r \bar{v}_{\theta}}{\sqrt{ \bar{v}_{n}^{2} + \bar{v}_{\theta}^{2}}}\Big) = \frac{-\bar{v}_{\theta}}{\sqrt{ \bar{v}_{n}^{2} + (1-r^{2}) \bar{v}_{\theta}^{2}}},
\\
& \partial_{\theta} \cos^{-1} \Big( \frac{ \bar{v}_{\theta}}{\sqrt{ \bar{v}_{n}^{2} + \bar{v}_{\theta}^{2}}}\Big) = 1  , \ \ \
\partial_{\theta} \cos^{-1} \Big( \frac{r \bar{v}_{\theta}}{\sqrt{ \bar{v}_{n}^{2} + \bar{v}_{\theta}^{2}}}  \Big) = \frac{r \bar{v}_{n}}{\sqrt{ \bar{v}_{n}^{2} + (1-r^{2}) \bar{v}_{\theta}^{2}}}  ,
\\
&\partial_{ \bar{v}_{n}} \cos^{-1} \Big( \frac{ \bar{v}_{\theta}}{\sqrt{ \bar{v}_{n}^{2}+ \bar{v}_{\theta} ^{2} }}\Big)= \frac{ \bar{v}_{\theta}}{ \bar{v}_{n}^{2} + \bar{v}_{\theta}^{2}} , \ \ \
\partial_{ \bar{v}_{n}} \cos^{-1} \Big( \frac{r \bar{v}_{\theta}}{\sqrt{ \bar{v}_{n}^{2}+ \bar{v}_{\theta} ^{2} }}\Big)= \frac{r \bar{v}_{\theta}}{ \bar{v}_{n}^{2} + \bar{v}_{\theta}^{2}},
\\
&\partial_{ \bar{v}_{\theta}} \cos^{-1} \Big( \frac{\bar{v}_{\theta}}{\sqrt{\bar{v}_{n}^{2}+ \bar{v}_{\theta} ^{2} }}\Big)= \frac{- \bar{v}_{n}}{ \bar{v}_{n}^{2} + \bar{v}_{\theta}^{2}}, \ \ \
\partial_{ \bar{v}_{\theta}} \cos^{-1} \Big( \frac{ r \bar{v}_{\theta}}{\sqrt{ \bar{v}_{n}^{2}+ \bar{v}_{\theta} ^{2} }}\Big)= \frac{-r \bar{v}_{n}}{ \bar{v}_{n}^{2} + \bar{v}_{\theta}^{2}},\\
& \partial_{ \bar{v}_{n}} \cos^{-1} \Big( \frac{ \bar{v}_{n} \cos\theta - \bar{v}_{\theta} \sin\theta}{\sqrt{ \bar{v}_{n}^{2} + \bar{v}_{\theta}^{2}}}\Big)=\frac{ \bar{v}_{\theta}}{ { \bar{v}_{n}^{2} + \bar{v}_{\theta}^{2}}}, \ \ \ \partial_{ \bar{v}_{\theta}} \cos^{-1} \Big( \frac{ \bar{v}_{n} \cos\theta - \bar{v}_{\theta} \sin\theta}{\sqrt{ \bar{v}_{n}^{2} + \bar{v}_{\theta}^{2}}}\Big)=\frac{ \bar{v}_{n}}{ { \bar{v}_{n}^{2} +\bar{v}_{\theta}^{2}}},\\
& \partial_{\theta} \cos^{-1} \Big( \frac{ \bar{v}_{n} \cos\theta - \bar{v}_{\theta} \sin\theta}{\sqrt{\bar{v}_{n}^{2} + \bar{v}_{\theta}^{2}}}\Big)=0= \partial_{n} \cos^{-1} \Big( \frac{\bar{v}_{n} \cos\theta - \bar{v}_{\theta} \sin\theta}{\sqrt{ \bar{v}_{n}^{2} + \bar{v}_{\theta}^{2}}}\Big),
\end{split}
\end{equation}
and
\begin{equation}\notag
\begin{split}
&\partial_{\bar{v}_{\theta}} \theta^{\ell}= \frac{|\bar{v}_{n}|}{|\bar{v}|^{2}} + |t-s|, \ \ \ \partial_{\bar{v}_{r}} \theta^{\ell} = -\frac{\bar{v}_{\theta}}{|\bar{v}|^{2}} -(2\ell-1) \frac{r \bar{v}_{\theta}}{|\bar{v}|^{2}}, \\
& \partial_{\theta} \theta^{\ell} \lesssim \frac{|t-s||\bar{v}|^{2} |\bar{v}_{n}|}{\bar{v}_{n}^{2} + (1-r^{2}) \bar{v}_{\theta}^{2}}
, \  \partial_{n} \theta^{\ell} = \frac{(2\ell-1) \bar{v}_{\theta}}{\sqrt{\bar{v}_{n}^{2} + (1-r^{2}) \bar{v}_{\theta}^{2}}},
\end{split}
\end{equation}
and
\begin{equation}\notag
\begin{split}
&\partial_{\theta} \psi^{\ell} \lesssim
 \frac{|t-s||\bar{v}|^{2} |\bar{v}_{n}|}{\bar{v}_{n}^{2} + (1-r^{2}) \bar{v}_{\theta}^{2}}
, \ \ \partial_{n} \psi^{\ell} = \frac{2\ell \bar{v}_{\theta}}{ \sqrt{\bar{v}_{n}^{2} + (1-r^{2}) \bar{v}_{\theta}^{2}}},\\
& \partial_{\bar{v}_{\theta}} \psi^{\ell} \lesssim \frac{|t-s| |\bar{v}_{n}|}{\sqrt{\bar{v}_{n}^{2}+ (1-r^{2}) \bar{v}_{\theta}^{2}}}\lesssim |t-s| , \ \  \partial_{\bar{v}_{n}} \psi^{\ell} = -2\ell \frac{r \bar{v}_{\theta}}{ \bar{v}_{n}^{2} + \bar{v}_{\theta}^{2}}+O_{\xi}(1)\frac{1}{|\bar{v}|},
\end{split}
\end{equation}
and
\begin{equation}\notag
\begin{split}
t^{\ell}-t^{\ell+1} & \leq \frac{2 \sqrt{(1-r^{2}) \bar{v}_{\theta}^{2} + \bar{v}_{n}^{2}}}{ \bar{v}_{n}^{2} + \bar{v}_{\theta}^{2}},\\
\ell_{*}(s) & \leq \frac{|t-s| |\bar{v}|^{2}}{2 \sqrt{\bar{v}_{n}^{2} + (1-r^{2}) \bar{v}_{\theta}^{2}}} - \frac{r|\bar{v}_{n}|}{2\sqrt{ \bar{v}_{n}^{2} + (1-r^{2}) \bar{v}_{\theta}^{2}}} + \frac{1}{2} \leq \ell_{*}(s) +1,
\end{split}
\end{equation}
and
\begin{equation}\notag
\begin{split}
\partial_{r} t^{\ell} & = \frac{-| \bar{v}_{n}|}{ \bar{v}_{n}^{2} + \bar{v}_{\theta}^{2}} + (2\ell-1) \frac{r \bar{v}_{\theta}^{2}}{|\bar{v}|^{2} \sqrt{\bar{v}_{n}^{2} + (1-r^{2}) \bar{v}_{\theta}^{2}}}
,\\
\partial_{\theta} t^{\ell} & =\frac{-(2\ell -1) \bar{v}_{n} \bar{v}_{\theta} r^{2}}{|\bar{v}|^{2} \sqrt{\bar{v}_{n}^{2} + (1-r^{2}) \bar{v}_{\theta}^{2}}} \lesssim \frac{|t-s||\bar{v}_{\theta} | r^{2}}{\sqrt{\bar{v}_{n}^{2}  + (1-r^{2}) \bar{v}_{\theta}^{2}}}, 
\end{split}
\end{equation}
\begin{equation}\notag
\begin{split}
\partial_{\bar{v}_{n}} t^{\ell} & =- (2\ell-1) \frac{\bar{v}_{n}}{|\bar{v}|^{2} \sqrt{\bar{v}_{n}^{2} + (1-r^{2})\bar{v}_{\theta}^{2}}} + O_{\xi}(1) \frac{1+ |\bar{v}||t-s|}{|\bar{v}|^{2}}
,\\
\partial_{\bar{v}_{\theta}} t^{\ell} & \leq (2\ell-1) \frac{(1-r^{2})|\bar{v}_{\theta}|}{|\bar{v}|^{2} \sqrt{\bar{v}_{n}^{2} + (1-r^{2}) \bar{v}_{\theta}^{2}}} + 2(2\ell-1) \frac{| \bar{v}_{\theta}| \sqrt{ \bar{v}_{n}^{2} +(1-r^{2}) \bar{v}_{\theta}^{2}}}{|\bar{v}|^{4}}  \\
& \lesssim |t-s| \frac{(1-r^{2}) | \bar{v}_{\theta}|}{ \bar{v}_{n}^{2} + (1-r^{2}) \bar{v}_{\theta}^{2}} + |t-s| \frac{| \bar{v}_{\theta}|}{|\bar{v}|^{2}}.
\end{split}
\end{equation}
If $r<\frac{1}{2}$ then $(1-r^{2}) \bar{v}_{\theta}^{2} + \bar{v}_{n}^{2} \geq \frac{3}{4} |\bar{v}|^{2}$ and $\partial_{\bar{v}_{\theta}} t^{\ell} \lesssim |t-s| \frac{4|\bar{v}_{\theta}|}{3|\bar{v}|^{2}} + |t-s|\frac{| \bar{v}_{\theta}|}{|\bar{v}|^{2}} \lesssim |t-s| \frac{| \bar{v}_{\theta}|}{|\bar{v}|^{2}}.$ If $r\geq \frac{1}{2}$ and $| \bar{v}_{\theta}| \leq | \bar{v}_{n}|$ then $\partial_{ \bar{v}_{\theta}} t^{\ell} \lesssim \frac{|t-s| | \bar{v}_{\theta}|}{\frac{ \bar{v}_{n}^{2}}{2} + \frac{ \bar{v}_{n}^{2}}{2}} + |t-s| \frac{| \bar{v}_{\theta}|}{|\bar{v}|^{2}}\lesssim |t-s| \frac{| \bar{v}_{\theta}|}{|\bar{v}|^{2}}.$ If $r\geq \frac{1}{2}$ and $|\bar{v}_{\theta}|\geq | \bar{v}_{n}|$ then
$\partial_{ \bar{v}_{\theta}} t^{\ell} \lesssim |t-s| \frac{(1-r^{2})| \bar{v}_{\theta}|}{(1-r^{2})| \bar{v}_{\theta}| (\frac{| \bar{v}_{\theta}|}{2} + \frac{| \bar{v}_{\theta}|}{2})} + |t-s| \frac{| \bar{v}_{\theta}|}{|\bar{v}|^{2}} \lesssim \frac{|t-s|}{|\bar{v}|}.$

Therefore $${\partial_{ \bar{v}_{\theta}}} t^{\ell}\lesssim \frac{|t-s|}{|\bar{v}|}.$$

Directly
\begin{equation}\notag
\begin{split}
\partial_{n} \bar{X}_{\mathbf{cl}}(s) & =  \partial_{n} \theta^{\ell} \left(\begin{array}{c} -\sin\theta^{\ell} \\ \cos\theta^{\ell} \end{array}\right)- \frac{\partial t^{\ell}}{\partial n} |\bar{v}| \left(\begin{array}{c}\cos\psi^{\ell} \\ \sin\psi^{\ell} \end{array}\right) - (t^{\ell}-s) |\bar{v}| \frac{\partial \psi^{\ell}}{\partial n} \left( \begin{array}{c} -\sin\psi^{\ell} \\ \cos\psi^{\ell} \end{array}
\right)\\
& = \frac{(2\ell-1) \bar{v}_{\theta}^{2}}{|\bar{v}|\sqrt{\bar{v}_{n}^{2} + (1-r^{2}) \bar{v}_{\theta}^{2}}} \left(\begin{array}{c}  -\sin \theta^{\ell} -\cos\psi^{\ell} \\ \cos\theta^{\ell} - \sin\psi^{\ell} \end{array}\right)\\
& +O_{\xi}(1)\Big\{ \frac{(2\ell-1) |\bar{v}_{\theta}||\bar{v}_{n}|}{ |\bar{v}|\sqrt{ \bar{v}_{n}^{2} + (1-r^{2}) \bar{v}_{\theta}^{2}}} +  \frac{(2\ell-1) (1-r)| \bar{v}_{\theta}|^{2} }{ |\bar{v}|\sqrt{ \bar{v}_{n}^{2} + (1-r^{2}) \bar{v}_{\theta}^{2}}} + \frac{| \bar{v}_{n}| }{|\bar{v}|} +  {\ell |t^{\ell}-t^{\ell+1}|}
 \Big\},
\end{split}
\end{equation}
where $O_{\xi}(1)-$remainder is bounded by
\begin{equation}\notag
\begin{split}
\lesssim& \frac{|t-s||\bar{v}|^{2}}{ {\bar{v}_{n}^{2} + (1-r^{2})\bar{v}_{\theta}^{2}}} \Big\{ \frac{| \bar{v}_{\theta}|| \bar{v}_{n}|}{|\bar{v}|} + \frac{|1-r| | \bar{v}_{\theta}^{2}|}{|\bar{v}|} + \frac{\sqrt{ \bar{v}_{n}^{2} + (1-r^{2}) \bar{v}_{\theta}^{2}}}{|\bar{v}|}\Big\} + \frac{|\bar{v}_{n}|}{|\bar{v}|}\\
\lesssim& \frac{|t-s||\bar{v}| (1+ |\bar{v}|)}{\sqrt{\bar{v}_{n}^{2} + (1-r^{2})\bar{v}_{\theta}^{2}}} + |t-s| |\bar{v}| + \frac{|\bar{v}_{n}|}{|\bar{v}|}.
\end{split}
\end{equation}
Now we use some trigonometric identities to have
\begin{equation}\notag
\begin{split}
&\sin \theta^{\ell} + \cos\psi^{\ell} \\
=& \sin\theta \cos \big( \cos^{-1} (\frac{\bar{v}_{\theta}}{|\bar{v}|}) + (2\ell-1) \cos^{-1} (\frac{r \bar{v}_{\theta}}{|\bar{v}|})  \big)
-\cos\theta \sin \big( \cos^{-1} (\frac{ \bar{v}_{\theta}}{|\bar{v}|})  + (2\ell-1) \cos^{-1} (\frac{r \bar{v}_{\theta}}{|\bar{v}|})  \big)\\
&+ \frac{\bar{v}_{n} \cos\theta - \bar{v}_{\theta} \sin\theta}{|\bar{v}|} \cos(2\ell \cos^{-1}(\frac{r \bar{v}_{\theta}}{|\bar{v}|})) + \sin \big(\cos^{-1} (\frac{ \bar{v}_{n}\cos\theta - \bar{v}_{\theta} \sin\theta}{|\bar{v}|})\big) \sin \big(-2\ell \cos^{-1} (\frac{r \bar{v}_{\theta}}{|\bar{v}|})\big).
\end{split}
\end{equation}
Here
\begin{equation}\notag
\begin{split}
&\cos\Big( \cos^{-1} (\frac{\bar{v}_{\theta}}{|\bar{v}|}) - \cos^{-1} (\frac{r \bar{v}_{\theta}}{|\bar{v}|}) + 2\ell \cos^{-1} (\frac{r \bar{v}_{\theta}}{|\bar{v}|})\Big)\\
=& \cos \Big( \cos^{-1} (\frac{ \bar{v}_{\theta}}{|\bar{v}|}) - \cos^{-1} (\frac{r \bar{v}_{\theta}}{|\bar{v}|}) \Big) \cos \Big(   2\ell \cos^{-1} (\frac{r\bar{v}_{\theta}}{|\bar{v}|})\Big) - \sin \Big( \cos^{-1} (\frac{\bar{v}_{\theta}}{|\bar{v}|}) - \cos^{-1} (\frac{r \bar{v}_{\theta}}{|\bar{v}|}) \Big) \sin  \Big(   2\ell \cos^{-1} (\frac{r \bar{v}_{\theta}}{|\bar{v}|})\Big)\\
=& \cos \Big(   2\ell \cos^{-1} (\frac{r \bar{v}_{\theta}}{|\bar{v}|})\Big) + O_{\xi}(1)\Big|1-  \cos \Big( \cos^{-1} (\frac{\bar{v}_{\theta}}{|\bar{v}|}) - \cos^{-1} (\frac{r \bar{v}_{\theta}}{|\bar{v}|}) \Big) \Big| \\
&+O_{\xi}(1) \Big|   \sin \Big( \cos^{-1} (\frac{ \bar{v}_{\theta}}{|\bar{v}|}) - \cos^{-1} (\frac{r \bar{v}_{\theta}}{|\bar{v}|}) \Big) \sin  \Big(   2\ell \cos^{-1} (\frac{r \bar{v}_{\theta}}{|\bar{v}|})\Big)\Big|\\
=&  \cos \Big(   2\ell \cos^{-1} (\frac{r \bar{v}_{\theta}}{|\bar{v}|})\Big) + O_{\xi}(1) \Big\{\Big| 1- \frac{r \bar{v}_{\theta}^{2} }{|\bar{v}|^{2}}  - \frac{ \bar{v}_{n} \sqrt{ \bar{v}_{n}^{2} + (1-r^{2}) \bar{v}_{\theta}^{2}}}{|\bar{v}|^{2}} \Big| + \Big| \frac{r \bar{v}_{n}  \bar{v}_{\theta}}{|\bar{v}|^{2}} - \frac{ \bar{v}_{\theta} \sqrt{ \bar{v}_{n}^{2} + (1-r^{2}) \bar{v}_{\theta}^{2} }}{|\bar{v}|^{2}}  \Big|\Big\}\\
=&  \cos \Big(   2\ell \cos^{-1} (\frac{r \bar{v}_{\theta}}{|\bar{v}|})\Big) + O_{\xi}(1)\big\{ \frac{(1-r) \bar{v}_{\theta}^{2}}{|\bar{v}|^{2}} + \frac{ \sqrt{ \bar{v}_{n}^{2} + (1-r^{2}) \bar{v}_{\theta}^{2} }}{|\bar{v}| }  \big\},
\end{split}
\end{equation}
 and
\begin{equation}\notag
\begin{split}
& \sin \Big( \cos^{-1} (\frac{ \bar{v}_{\theta}}{|\bar{v}|}) -\cos^{-1} (\frac{r \bar{v}_{\theta}}{|\bar{v}|}) + 2\ell \cos^{-1} (\frac{r \bar{v}_{\theta}}{|\bar{v}|})\Big)\\
=& \sin \Big(  \cos^{-1} (\frac{\bar{v}_{\theta}}{|\bar{v}|}) -\cos^{-1} (\frac{r \bar{v}_{\theta}}{|\bar{v}|}) \Big) \cos \Big(  2\ell \cos^{-1} (\frac{r \bar{v}_{\theta}}{|\bar{v}|})\Big)  \Big) + \cos \Big(  \cos^{-1} (\frac{ \bar{v}_{\theta}}{|\bar{v}|}) -\cos^{-1} (\frac{r \bar{v}_{\theta}}{|\bar{v}|}) \Big) \sin \Big(  2\ell \cos^{-1} (\frac{r \bar{v}_{\theta}}{|\bar{v}|})\Big)  \Big) \\
=& \sin \Big(  2\ell \cos^{-1} (\frac{r \bar{v}_{\theta}}{|\bar{v}|}) \Big) + O_{\xi}(1) \big\{ \frac{(1-r)\bar{v}_{\theta}^{2}}{|\bar{v}|^{2}} + \frac{ \sqrt{\bar{v}_{n}^{2} + (1-r^{2}) \bar{v}_{\theta}^{2} }}{|\bar{v}| }  \big\}.
\end{split}
\end{equation}
Therefore
\begin{equation}\notag
\begin{split}
\sin \theta^{\ell} + \cos\psi^{\ell}& = (1- \frac{|\bar{v}_{\theta}|}{|\bar{v}|} )\sin\theta \cos\Big( (2\ell-1) \cos^{-1} (\frac{r \bar{v}_{\theta}}{|\bar{v}|})  \Big)
-  (1- \frac{| \bar{v}_{\theta}|}{|\bar{v}|})\cos\theta   \sin \Big( 2\ell \cos^{-1} (\frac{r \bar{v}_{\theta}}{|\bar{v}|})\Big)\\
&+ O_{\xi}(1)  \big\{ \frac{(1-r) \bar{v}_{\theta}^{2}}{|\bar{v}|^{2}} + \frac{ \sqrt{ \bar{v}_{n}^{2} + (1-r^{2}) \bar{v}_{\theta}^{2} }}{|\bar{v}| }  \big\}\\
&\sim   \frac{ \sqrt{ \bar{v}_{n}^{2} + (1-r^{2}) \bar{v}_{\theta}^{2} }}{|\bar{v}| } .
  \end{split}
\end{equation}
Since $\cos\theta^{\ell} - \sin\psi^{\ell} = \sin(\theta^{\ell} + \frac{\pi}{2}) + \cos(\psi^{\ell} + \frac{\pi}{2})$,
\begin{equation}\notag
\begin{split}
 \cos\theta^{\ell} - \sin\psi^{\ell} \sim  \frac{ \sqrt{ \bar{v}_{n}^{2} + (1-r^{2}) \bar{v}_{\theta}^{2} }}{|\bar{v}| } .
  \end{split}
\end{equation}
Therefore we conclude our claim for $\partial_{n} \bar{X}_{\mathbf{cl}}$.

Using the same estimates
\begin{equation}\notag
\begin{split}
\partial_{\bar{v}_{n}} \bar{X}_{\mathbf{cl}}(s) & =  (2\ell-1)\Big\{ \frac{-r \bar{v}_{\theta}}{|\bar{v}|^{2}} \left(\begin{array}{c} -\sin\theta^{\ell} \\ \cos\theta^{\ell} \end{array}\right)
+   \frac{ \bar{v}_{n}}{|\bar{v}| \sqrt{ \bar{v}_{n}^{2} + (1-r^{2}) \bar{v}_{\theta}^{2}}} \left(\begin{array}{c}  \cos\psi^{\ell} \\ \sin\psi^{\ell}\end{array}\right)  \Big\}+ O_{\xi}(1)   \frac{1 + |\bar{v}||t-s|}{|\bar{v}|}  \\
&= \frac{2\ell-1}{|\bar{v}|} \left(\begin{array}{c}  \sin\theta^{\ell} + \cos\psi^{\ell} \\ -\cos\theta^{\ell} + \sin\psi^{\ell}  \end{array} \right) +  O_{\xi}(1)   \frac{1 + |\bar{v}||t-s|}{|\bar{v}|}  \lesssim \frac{1}{|\bar{v}|}.
\end{split}
\end{equation}

Since $t^{\ell_{*}+1}<0<t^{\ell_{*}}$,
\begin{equation}\notag\label{p_nV}
\begin{split}
\partial_{n} v^{\ell}  =  \partial_{n} \psi^{\ell}( -|\bar{v}|  \sin \psi^{\ell}, |\bar{v}| \cos\psi^{\ell},0  )
  = \frac{|t-s||\bar{v}|^{2}|\bar{v}_{\theta}|  }{ \bar{v}_{n}^{2} + (1-r^{2}) \bar{v}_{\theta}^{2} } ( -|\bar{v}|  \sin \psi^{\ell}, |\bar{v}| \cos\psi^{\ell},0  ).
\end{split}
\end{equation}
Therefore we conclude our claim for $\partial_{n} \bar{V}_{\mathbf{cl}}(0)$. Moreover
 
\begin{equation}\label{2D_specular}
\begin{split}
|\partial_{\theta} \bar{X}_{\mathbf{cl}}(s)| & \lesssim \frac{ |\bar{v}| }{\sqrt{\bar{v}_{n}^{2} + (1-r^{2}) \bar{v}_{\theta}^{2}}}  |\bar{v}||t-s| , \ \ \ |\partial_{\theta} \bar{V}_{\mathbf{cl}}(s)|  \lesssim \frac{ |\bar{v}|^{2}}{\sqrt{\bar{v}_{n}^{2} + (1-r^{2}) \bar{v}_{\theta}^{2}}}|\bar{v}||t-s|,\\
|\partial_{\bar{v}_{n}} \bar{V}_{\mathbf{cl}}(s)| & \lesssim 1 + \frac{ |\bar{v}|^{2}|t-s|}{\sqrt{ \bar{v}_{n}^{2} + (1-r^{2}) \bar{v}_{\theta}^{2}}} , \  \ 
|\partial_{\bar{v}_{\theta}} \bar{X}_{\mathbf{cl}}(s)|   \lesssim \frac{1 }{|\bar{v}|}  , \  \ | {\partial_{{  \bar{v}_{\theta}}} \bar{V}_{\mathbf{cl}}(s)} | \lesssim 1+ |\bar{v}||t-s|.
\end{split}
\end{equation}
 \end{proof}
   
Based on Example 1, we naturally consider the 2D specular problem. We consider the 2D specular problem for $f(t,x_{1},x_{2},v_{1},v_{2},v_{3})$ solving
\begin{equation}
  \partial _{t} f+v_{1}\partial _{x_{1}}f+ v_{2}\partial
_{x_{2}}f=  \Gamma _{\text{gain}}(f,f) - \nu (\sqrt{\mu}f)f,  \label{boltzmann2d}
\end{equation}%
where $v_{3}$ is a paramter. Here $(x_{1},x_{2})\in \Omega =\{x\in \mathbb{R}%
^{2}:$ $\xi (x)>0\}$ and the convexity (\ref{convex}) is valid for all $\zeta \in \mathbb{R%
}^{2}.$ We study (\ref{boltzmann2d}) with specular boundary condition (\ref
{specularBC}). Denote $v:=(\bar{v},v_{3}) = (v_{1},v_{2};v_{3})\in \mathbb{R}^{3}$. We define $$%
\alpha (x,\bar{v})=|\bar{v}\cdot \nabla \xi (x)|^{2}-2\{\bar{v}\cdot \nabla
^{2}\xi (x)\cdot \bar{v}\}\xi (x).$$ Note that $\nabla \xi(x)= (\partial_{x_1}\xi(x), \partial_{x_{2}} \xi(x),0).$

\vspace{8pt}

The following estimate is crucial to establish the weighted $C^{1}$ estimate (Theorem \ref{theo2D}) and non-existence of $\nabla^{2}f$ up to the boundary (Proposition \ref{sing_specular}). 

\begin{lemma}\label{2D_Gamma}  
For $\theta>0$ and for $i=1,2,$
\begin{equation}
\begin{split}
 &e^{-\varpi \langle v\rangle s}|\partial _{v_{i}}\Gamma _{\mathrm{gain}}(f,f)|\\ &\lesssim
 ||  e^{\theta |v|^{2}}f ||_{\infty } \Big\{  ||  e^{\theta |v|^{2}}f ||_{\infty }+  ||  \partial _{v_3}f ||_{\infty }+ \Big|\Big|
 e^{-\varpi\langle v \rangle s}
  \frac{|\bar{v}|}{\langle \bar{v} \rangle} \alpha^{1/2} \nabla_{\bar{v}} f  \Big|\Big|_{\infty} \Big\},
  \end{split}
\label{2dv}
\end{equation}
where $v=(\bar{v}, v_{3})= (v_{1}, v_{2}, v_{3}).$
\end{lemma}

\begin{proof}
The key is to use the splitting $u_{||,3}$ with respect to $|\bar{%
v}+\bar{u}_{\perp }|\sqrt{\alpha (\bar{v}+\bar{u}_{\perp })}.$ 

Recall from \cite{Guo03} that the gain term of the nonlinear Boltzmann operator in (\ref{gain_Gamma}) equals
\begin{equation}\label{carleman}
\begin{split}
 &\Gamma_{\mathrm{gain}}(g_{1},g_{2})(v)\\
=&C
\int_{\mathbb{R}^{3}} \mathrm{d}u  \int_{u \cdot w =0} \mathrm{d}w \ g_{1}(v+ w) g_{2} (v+ u )
q_{0}^{*} \big( \frac{|u |}{|u  + w|} \big) \frac{|u  + w|^{\kappa-1}}{|u | } e^{- \frac{|u  + v + w|^{2}}{4}}
,\\
=&
C\int_{\mathbb{R}^{3}} \mathrm{d}u  \int_{u \cdot w =0} \mathrm{d}w \ g_{2}(v+w) g_{1} (v+ u )
q_{0}^{*} \big( \frac{|u |}{|u  + w|} \big) \frac{|u  + w|^{\kappa-1}}{|u | } e^{- \frac{|u  + v + w|^{2}}{4}}
,\\
=& C\int_{\mathbb{R}^{3}} \mathrm{d}u  \int_{ (u  - v)\cdot w=0}  \mathrm{d}w \ g_{1}(v+w) g_{2}(u )
q_{0}^{*}\big(  \frac{|u -v|}{ { |u -v+w|    }}\big)  \frac{\big| u -v + w \big|^{{\kappa-1}}}{|u -v|}
e^{-\frac{|u +w |^{2}}{4}} ,\\
=&C \int_{\mathbb{R}^{3}} \mathrm{d}u  \int_{ (u - v)\cdot w=0}  \mathrm{d}w \  g_{2}(v+w) g_{1}( u )
q_{0}^{*}\big(  \frac{|u -v|}{ { |u -v+w|    }}\big)  \frac{\big| u -v + w \big|^{ {\kappa-1} }}{|u -v|}
e^{-\frac{|u +w |^{2}}{4}}
,
\end{split}
\end{equation}
where $q_{0}^{*}(\cos \theta )=  \frac{q_{0}(\cos\theta)}{|\cos\theta|}.$  This is due to two change of variables (37),(38) and page 316 of \cite{Guo03}.
Then 
\begin{eqnarray*}
&&\partial _{v_{i}}\Gamma _{\text{gain}}(f,f) \\
&=&2\Gamma _{\text{gain}}(\partial _{v_{i}}f,f) \\
&&+C\int_{\mathbb{R}^{3}}\mathrm{d}u_{\parallel}\int_{u_{\parallel}\cdot u_{\perp}=0}\mathrm{d}wf(v+ u_{\perp})f(v+u_{\parallel})\\
&& \ \ \ \ \ \ \ \ \ \ \ \ \ \ \ \ \ \ \  \times q_{0}^{\ast }(%
\frac{|u_{\parallel}|}{|u_{\parallel}+ u_{\perp}|})\frac{|u_{\parallel}+ u_{\perp}|^{\kappa -1}}{|u_{\parallel}|}e^{-\frac{|u_{\parallel}+v+ u_{\perp}|^{2}}{4}}(-%
\frac{\mathbf{e}_{i}}{2})\cdot  {(u_{\parallel}+v+ u_{\perp})}  \\
&=&2\Gamma _{\text{gain}}(\partial _{v_{i}}f,f)+O_{\xi }(1) e^{-C|v|^{2}}|| e^{\theta |v|^{2}}f||_{\infty }^{2}.
\end{eqnarray*}

Denote the standard cutoff function $\chi\geq 0 $: $\chi\equiv 1$ on $[0,1]$ and $\chi\equiv 0$ for $[2,\infty)$. We have 
\begin{equation*}
\Gamma _{\text{gain}%
}(\partial _{v_{i}}f,f)=\int_{\mathbb{R}^{3}}\mathrm{d}u_{||}f(v+u_{||})\int_{\mathbb{R}^{2}}\mathrm{d}u_{\perp
}\partial _{v_{i}}f(v+u_{\perp })e^{-\frac{|u_{\parallel}+u_{\perp} +v|^{2}}{4}} q_{0}^{\ast
}(\frac{|u_{\parallel}|}{|u_{\parallel}+u_{\perp}|}) \frac{ |u_{\parallel} +u_{\perp}|^{\kappa -1}}{|u_{||}|}.
\end{equation*}%
We further split it into, for $0< \varepsilon \ll1 ,$
\begin{equation}\notag
\begin{split}
& \int_{\mathbb{R}^{3}}\mathrm{d}u_{||}f(v+u_{||})\int_{\mathbb{R}^{2}}\chi \left( 
\frac{|\bar{v}+\bar{u}_{\perp }|^{1-  \varepsilon}\alpha ^{1/2-  \varepsilon}}{u_{||,3}}\right)\\
& \ \ \ \ \ \times \partial _{v_{i}}f(v+u_{\perp })e^{-\frac{|u_{\parallel}+v+u_{\perp}|^{2}}{4}} q_{0}^{\ast
}(\frac{|u_{\parallel}|}{|u_{\parallel}+u_{\perp}|}) \frac{ |u_{\parallel} +u_{\perp}|^{\kappa -1}}{|u_{||}|}\mathrm{d}u_{\perp } \\
& +\int_{\mathbb{R}^{3}}du_{||}f(v+u_{||})\int_{\mathbb{R}^{2}}\left\{
1-\chi \left( \frac{|\bar{v}+\bar{u}_{\perp }|^{1- \varepsilon}\alpha ^{1/2- \varepsilon}}{u_{||,3}}%
\right ) \right\}  \\
&  \ \ \ \ \ \  \times \partial _{v_{i}}f(v+u_{\perp })e^{-\frac{%
|u_{\parallel}+v+u_{\perp}|^{2}}{4}} q_{0}^{\ast
}(\frac{|u_{\parallel}|}{|u_{\parallel}+u_{\perp}|}) \frac{ |u_{\parallel} +u_{\perp}|^{\kappa -1}}{|u_{||}|}\mathrm{d}u_{\perp }.
\end{split}
\end{equation}%
For the first part, $|u_{||,3}|\geq |\bar{v}+\bar{u}_{\perp }|^{1- \varepsilon}\alpha
^{1/2- \varepsilon},$ and we parametrize $u_{\perp }$ as $u_{\perp ,3}=-\frac{\bar{%
u}_{||}\cdot \bar{u}_{\perp }}{u_{||,3}}$ so that 
\begin{equation}\label{3_perp}
\mathrm{d}u_{\perp }=\frac{| {u}_{||}|}{|u_{||,3}|}\mathrm{d}\bar{u}_{\perp } : = \frac{| {u}_{||}|}{|u_{||,3}|}\mathrm{d} {u}_{\perp,1 } \mathrm{d} {u}_{\perp,2} ,
\end{equation}%
and the first part equals
\begin{eqnarray*}
& \int_{\mathbb{R}^{3}} \mathrm{d}u_{||}f(v+u_{||})\int_{\mathbb{R}^{2}} \mathrm{d}\bar{u}%
_{\perp }\partial _{v_{i}}f(v_{1}+u_{\perp ,1},v_{2}+u_{\perp ,2},v_{3}-%
\frac{\bar{u}_{||}\cdot \bar{u}_{\perp }}{u_{||,3}}) \\
&\times  \chi \left( \frac{|\bar{v}+\bar{u}_{\perp }|^{1-\varepsilon}\alpha ^{1/2-\varepsilon}}{%
u_{||,3}}\right) e^{-\frac{|u_{\parallel}+v+u_{\perp}|^{2}}{4}}\frac{q_{0}^{\ast }(\frac{|u_{\parallel}|}{%
|u_{\parallel}+u_{\perp}|})}{|u_{||,3}||u_{\parallel}+u_{\perp}|^{ 1-\kappa}}.
\end{eqnarray*}%
We now integrate by part in $u_{\perp ,i}$ for $i=1,2$ to get 
\begin{equation}\notag
\begin{split}
&- \int_{\mathbb{R}^{3}}\mathrm{d}u_{||}f(v+u_{||})\int_{\mathbb{R}^{2}}\partial
_{v_{3}}f(\bar{v}+ \bar{u}_{\perp},v_{3}-\frac{\bar{u}%
_{||}\cdot \bar{u}_{\perp }}{u_{||,3}})\\
 & \ \ \ \ \ \ \ \ \times  \chi \left( \frac{|\bar{v}+\bar{u}_{\perp }|^{1-}\alpha ^{1/2-}}{%
u_{||,3}}\right) e^{-\frac{|u_{\parallel}+v+u_{\perp}|^{2}}{4}}\frac{q_{0}^{\ast }(\frac{|u_{\parallel}|}{%
|u_{\parallel}+u_{\perp}|})u_{||,i}\mathrm{d}\bar{u}_{\perp }}{|u_{||,3}|^{2}|u_{||}+u_{\perp }|^{1-\kappa
 }} \\
&- \int_{\mathbb{R}^{3}}\mathrm{d}u_{||}f(v+u_{||})\int_{\mathbb{R}%
^{2}}f(\bar{v} +\bar{u}_{\perp } ,v_{3}-\frac{\bar{u}_{||}\cdot 
\bar{u}_{\perp }}{u_{||,3}})\\
& \ \  \ \ \ \ \ \ \times  \partial _{u_{\perp },i}\left\{  \chi \left( \frac{|\bar{v}+
\bar{u}_{\perp }|^{1- \varepsilon}\alpha ^{1/2- \varepsilon}}{u_{||,3}}\right) e^{-\frac{|u_{\parallel}+v+u_{\perp}|^{2}%
}{4}}\frac{    q_{0}^{\ast }(\frac{|u_{\parallel}|}{|u_{\parallel}+u_{\perp}|})\bar{u}_{||,i}}{%
|u_{||,3}||u_{||}+u_{\perp }|^{1-\kappa }}\right\} \mathrm{d}\bar{u}_{\perp }.
\end{split}
\end{equation}
Directly we have $|\partial_{u_{\perp,i}} \alpha( \bar{v} + \bar{u}_{\perp})| \lesssim \alpha( \bar{v} + \bar{u}_{\perp})^{1/2}$ and $|\frac{d u_{\perp,3}}{d u_{\perp,i}}| \leq \frac{|\bar{u}_{\parallel}|}{|u_{\parallel,3}|} $ to conclude 
\begin{equation}\notag
\begin{split}
 |\partial _{u_{\perp },i}\Big\{ \ \ \Big\}|\sim \  & \ \chi^{\prime} \frac{|\bar{v} + \bar{u}_{\perp}|^{-\varepsilon} \alpha^{1/2-\varepsilon} + |\bar{v} + \bar{u}_{\perp}|^{1-\varepsilon} \alpha^{ -\varepsilon} }{u_{\parallel,3}} e^{-\frac{|u_{\parallel}+v+u_{\perp}|^{2}%
}{4}}\frac{    ||q_{0}^{\ast }||_{\infty}|\bar{u}_{||}|}{%
|u_{||,3}||u_{||}+u_{\perp }|^{1-\kappa }} \\
&  + \chi  e^{-C|u_{\parallel} + u_{\perp} + v|^{2}} \Big\{ \frac{ || q_{0}^{*}||_{\infty}|\bar{u}_{\parallel}|^{2}}{|u_{\parallel,3}|^{2} |u_{\parallel} + u_{\perp}|^{1-\kappa}}
+ \frac{|| q_{0}^{*}||_{C^{1}}|\bar{u}_{\parallel}|^{2}}{|u_{\parallel,3}|^{2}|u_{\parallel} + u_{\perp}|^{3-\kappa}}
+ \frac{   || q_{0}^{*}||_{\infty}|\bar{u}_{\parallel}| (1+ \frac{|\bar{u}_{\parallel}|}{|u_{\parallel,3}|})}{|u_{\parallel,3}||u_{\parallel} + u_{\perp}|^{2-\kappa}}
\Big\}\\
\lesssim_{q_{0}^{*}} & \ \mathbf{1}_{\{u_{\parallel,3} \sim |\bar{v}+ \bar{u}_{\perp}|^{1-\varepsilon} \alpha^{1/2-\varepsilon}\}}
\Big\{ 
\frac{1}{|\bar{v} + \bar{u}_{\perp}|} 
+ \frac{ |\bar{v} + \bar{u}_{\perp}|^{1-\varepsilon}  \alpha^{-\varepsilon}}{|u_{\parallel,3}| }
\Big\} 
\frac{ e^{-\frac{|u_{\parallel}+v+u_{\perp}|^{2}%
}{4}}   |\bar{u}_{||}|}{%
|u_{||,3}||u_{||}+u_{\perp }|^{1-\kappa }}\\
&+\mathbf{1}_{\{ |\bar{v} + \bar{u}_{\perp}|^{1-\varepsilon} \alpha^{1/2-\varepsilon} \leq u_{\parallel,3} \}} \frac{|\bar{u}_{\parallel}|(1+ |\bar{u}_{\parallel}|)  }{|u_{\parallel,3}|^{2}|u_{\parallel} + u_{\perp}|^{1-\kappa}} (1+ \frac{1}{|u_{\parallel} + u_{\perp}|^{2}})e^{-C|u_{\parallel}+v+u_{\perp}|^{2}%
}     .
\end{split}
\end{equation}
Note that $|f(v+u_{\parallel})| \lesssim e^{-C|v+u_{\parallel}|^{2}} ||  e^{\theta|v|^{2}} f||_{\infty}$ and 
\begin{equation}\label{parallel3}
 |f(\bar{v} + \bar{u}_{\perp}, v_{3} - \frac{\bar{u}_{\parallel} \cdot \bar{u}_{\perp}}{u_{\parallel,3}})| \lesssim e^{-C|\bar{v}+ \bar{u}_{\perp}|^{2} -C |v_{3} + u_{\perp,3}(u_{\parallel}, \bar{u}_{\perp})|^{2}}  ||  e^{\theta|v|^{2}} f||_{\infty}, 
\end{equation}
and 
\[
e^{-|u_{\parallel} + u_{\perp} + v|^{2}} e^{-C|v+u_{\parallel}|^{2}}e^{-C|\bar{v}+ \bar{u}_{\perp}|^{2} -C |v_{3} + u_{\perp,3}(u_{\parallel}, \bar{u}_{\perp})|^{2}} \lesssim e^{-C^{\prime}|v|^{2}} e^{-C^{\prime} |u_{\perp}|^{2}} e^{-C|v+u_{\parallel}|^{2}},
\]
where $v:=v_{\parallel} + v_{\perp}$ with $v_{\parallel} :=  v\cdot \frac{u_{\parallel}}{|u_{\parallel}|}$ and $$|v+u_{\parallel}|^{2} + |v+u_{\perp}|^{2} = |v_{\parallel} + u_{\parallel}|^{2} + |v_{\perp}|^{2} + |v_{\perp} + u_{\perp}|^{2} + |v_{\parallel}|^{2} \geq |v|^{2}.$$ 
The $\partial_{u_{\perp,i}}\big\{ \ \big\}-$contribution are bounded by following three estimates: For the first term
\begin{equation}\notag
\begin{split}
e^{-|v|^{2}}  || e^{\theta|v|^{2}} f||_{\infty}^{2}\int_{\mathbb{R}^{2}} \frac{  e^{-|u_{\perp}|^{2}}\mathrm{d}\bar{u}_{\perp}}{|\bar{v}+ \bar{u}_{\perp}|}
 \int_{\mathbb{R}^{2}}  |\bar{u}_{\parallel}|^{\kappa} e^{-|\bar{v} + \bar{u}_{\parallel}|^{2}}  \mathrm{d}\bar{u}_{\parallel} 
\int_{u_{\parallel,3} \sim |\bar{v} + \bar{u}_{\perp}|^{1-\varepsilon} \alpha^{\frac{1}{2}-\varepsilon}} \frac{\mathrm{d}u_{\parallel,3}}{|u_{\parallel,3}|} \lesssim e^{-C^\prime |v|^{2}}.
\end{split}
\end{equation}
For the second term we use $f(\bar{v}+ \bar{u}_{\perp}, v_{3} -\frac{\bar{u}_{\parallel} \cdot \bar{u}_{\perp}}{u_{\parallel,3}})\lesssim e^{-C|v+u_{\perp}|^{2}} || e^{\theta|v|^{2}}f||_{\infty} \frac{1}{1+ (v_{3} - \frac{\bar{u}_{\parallel} \cdot \bar{u}_{\perp}}{|u_{\parallel,3}|}  )^{\varepsilon}}$ such that 
\begin{equation}\notag
\begin{split}
&f(v+u_{\parallel})  \frac{|\bar{v} + \bar{u}_{\perp}|^{1-\varepsilon} \alpha^{ -\varepsilon}}{|u_{\parallel,3}|^{2}} f(v_{3} - \frac{\bar{u}_{\parallel} \cdot \bar{u}_{\perp}}{u_{\parallel,3}})
\frac{e^{-\frac{|u_{\parallel} + u_{\perp} + v|^{2}}{4}} |\bar{u}_{\parallel}|}{ |u_{\parallel} + u_{\perp}|^{1-\kappa}}
\\
\lesssim&e^{-|v|^{2}} e^{-|v+u_{\parallel}|^{2}} e^{-|u_{\perp}|^{2}} || e^{\theta|v|^{2}} f||_{\infty}^{2}   \frac{|\bar{v} + \bar{u}_{\perp}|^{1-\varepsilon} \alpha^{ -\varepsilon}|\bar{u}_{\parallel}|  }{ |u_{\parallel} + u_{\perp}|^{1-\kappa} }  \frac{1}{|u_{\parallel,3}|^{2-\varepsilon}} \frac{1}{|u_{\parallel,3}|^{\varepsilon}  + [  v_{3} u_{\parallel,3} - \bar{u}_{\parallel} \cdot \bar{u}_{\perp} ]^{\varepsilon}}\\
\lesssim&e^{-|v|^{2}} e^{-|v+u_{\parallel}|^{2}} e^{-|u_{\perp}|^{2}} || e^{\theta|v|^{2}} f||_{\infty}^{2}   \frac{|\bar{v} + \bar{u}_{\perp}|^{1-\varepsilon} \alpha^{ -\varepsilon}
\langle v\rangle }{ |u_{\parallel} + u_{\perp}|^{1-\kappa} }  \frac{1}{|u_{\parallel,3}|^{2-\varepsilon} |\bar{u}_{\perp}|^{\varepsilon} } \frac{1}{  \big[  \frac{v_{3} u_{\parallel,3}}{|\bar{u}_{\perp}|} - \bar{u}_{\parallel} \cdot \frac{\bar{u}_{\perp}}{|\bar{u}_{\perp}|} \big]^{\varepsilon}},
\end{split}
\end{equation}
where we have used $e^{-|v+u_{\parallel}|^{2}}|u_{\parallel}|\lesssim\{ |u_{\parallel}+ v| + |v|\}e^{-|v+u_{\parallel}|^{2}} \lesssim(1+ |v|)e^{-C|v+u_{\parallel}|^{2}}.$ 

Now we decompose $\bar{u}_{\parallel} = \bar{u}_{\parallel,a} + \bar{u}_{\parallel,b} :=  \bar{u}_{\parallel} \cdot \frac{\bar{u}_{\perp}}{|\bar{u}_{\perp}|} + (\bar{u}_{\parallel} -  \bar{u}_{\parallel} \cdot \frac{\bar{u}_{\perp}}{|\bar{u}_{\perp}|})$ and bound $e^{-|v|^{2}} || e^{\theta|v|^{2}} f||_{\infty}^{2}\times $
\begin{equation}\notag
\begin{split}
& \int_{ u_{\parallel,3} \sim |\bar{v}+ \bar{u}_{\perp}|^{1-\varepsilon} \alpha^{\frac{1}{2} -\varepsilon}} \mathrm{d}u_{\parallel,3}  \int_{\mathbb{R}^{2}} \mathrm{d}\bar{u}_{\perp} e^{-|\bar{u}_{\perp}|^{2}} \frac{|\bar{v}+ \bar{u}_{\perp}|^{1-\varepsilon} \alpha^{-\varepsilon}}{|\bar{u}_{\perp}|^{\varepsilon}|u_{\parallel,3}|^{2-\varepsilon}}\\
& \ \ \times \int_{\mathbb{R}} \frac{ e^{-| \bar{u}_{\parallel,b}  + (v-v\cdot \frac{\bar{u}_{\perp}}{|\bar{u}_{\perp}|})|^{2}}\mathrm{d}\bar{u}_{\parallel,b}}{ |\bar{u}_{\parallel,b}|^{1-\kappa} }
  \int_{\mathbb{R}}  \frac{ \mathrm{d} \bar{u}_{\parallel,a}e^{-|\bar{u}_{\parallel,a}|^{2}}}{[\bar{u}_{\parallel,a} -\frac{v_{3} u_{\parallel,3}}{|\bar{u}_{\perp}|}    ]^{\varepsilon}}\\
  \lesssim &  \int_{\mathbb{R}^{2}} \mathrm{d}\bar{u}_{\perp} e^{-|\bar{u}_{\perp}|^{2}} \int_{ u_{\parallel,3} \sim |\bar{v}+ \bar{u}_{\perp}|^{1-\varepsilon} \alpha^{\frac{1}{2} -\varepsilon}} \mathrm{d}u_{\parallel,3}   \frac{|\bar{v}+ \bar{u}_{\perp}|^{1- \varepsilon} \alpha(\bar{v}+ \bar{u}_{\perp})^{ -\varepsilon}}{|\bar{u}_{\perp}|^{\varepsilon}|u_{\parallel,3}|^{2-\varepsilon}},
\end{split}
\end{equation}
where we have used $\bar{u}_{\parallel,a} \mapsto \bar{u}_{\parallel,a} - v\cdot \frac{\bar{u}_{\perp}}{|\bar{u}_{\perp}|}.$ The $u_{\parallel,3}-$integration yields
\begin{equation}\notag
\begin{split}
\lesssim&  \int_{\mathbb{R}^{2}} \mathrm{d}\bar{u}_{\perp} e^{-|\bar{u}_{\perp}|^{2}} 
\frac{|\bar{v}+ \bar{u}_{\perp}|^{1-\varepsilon} \alpha(\bar{v}+ \bar{u}_{\perp})^{ -\varepsilon}}{|\bar{u}_{\perp}|^{\varepsilon} }
\int_{ u_{\parallel,3} \sim |\bar{v}+ \bar{u}_{\perp}|^{1-\varepsilon} \alpha^{\frac{1}{2} -\varepsilon}} \frac{\mathrm{d}u_{\parallel,3} }{|u_{\parallel,3}|^{2-\varepsilon}} \\
\lesssim&
  \int_{\mathbb{R}^{2}} \mathrm{d}\bar{u}_{\perp} e^{-|\bar{u}_{\perp}|^{2}} 
\frac{|\bar{v}+ \bar{u}_{\perp}|^{1- \varepsilon} \alpha(\bar{v}+ \bar{u}_{\perp})^{ - \varepsilon}}{|\bar{u}_{\perp}|^{\varepsilon} }
\frac{1}{| \bar{v} + \bar{u}_{\perp}|^{(1-\varepsilon)(1- \varepsilon)}  } \frac{1}{\alpha(\bar{v} + \bar{u}_{\perp})^{(\frac{1}{2}- \varepsilon)(1-\varepsilon)}}\\
\lesssim & \int_{\mathbb{R}^{2}} \mathrm{d}\bar{u}_{\perp} e^{-|\bar{u}_{\perp}|^{2}}
 \frac{ |\bar{v}+ \bar{u}_{\perp}|^{\varepsilon (1- \varepsilon)}  }{|\bar{u}_{\perp}|^{\varepsilon}} \frac{1}{\alpha(\bar{v} + \bar{u}_{\perp})^{\frac{1}{2}-(1- \varepsilon)\varepsilon  }}.
  \end{split}
\end{equation}
Note that $\alpha(\bar{v}+ \bar{u}_{\perp})^{\frac{1}{2}-\varepsilon(1-\varepsilon)} \gtrsim [n(x) \cdot ( \bar{v} + \bar{u}_{\perp})]^{1-\varepsilon(1-\varepsilon)}$ and $|\bar{u}_{\perp}|^{\varepsilon} \gtrsim [n^{\perp} \cdot \bar{u}_{\perp}]^{\varepsilon}$ to bound
\begin{equation}\notag
\lesssim \langle v\rangle \int_{\mathbb{R}} \frac{e^{-|n\cdot \bar{u}_{\perp}|^{2}} }{ [n\cdot \bar{u}_{\perp} + n\cdot \bar{v}]^{1-2\varepsilon(1-\varepsilon)} }\mathrm{d}[n\cdot \bar{u}_{\perp}] \int_{\mathbb{R}} \frac{e^{-|n^{\perp} \cdot \bar{u}_{\perp}|^{2}}}{|n^{\perp} \cdot \bar{u}_{\perp}|^{\varepsilon}} \mathrm{d} [n^{\perp} \cdot \bar{u}_{\perp}]\lesssim \langle v\rangle .
\end{equation}
For the third term is bounded by $|| e^{\theta |v|^{2}}f||_{\infty }^{2}\times $  
\begin{eqnarray*}
&&\int_{%
\mathbb{R}^{3}}d\bar{u}_{||}\int_{\mathbb{R}^{2}}d\bar{u}_{\perp }\left\{
\int_{|u_{||,3}|\geq |\bar{v}+\bar{u}_{\perp }|^{1- \varepsilon}\alpha ^{1/2 -\varepsilon}}\frac{%
e^{-| u_{||,3}-v_{3}|^{2}}}{|u_{||,3}|^{2}}du_{||,3}\right\} \\
&& \ \ \ \  \ \ \ \  \ \ \   \ \ \ \  \ \ \   \ \ \ \times \frac{%
\langle u_{||}\rangle^{2}e^{-C\{|u_{||}+v|^{2}-|u_{\perp }|^{2}\}}}{|u+w|^{1-\kappa }}(1+\frac{%
1}{|u_{||}+u_{\perp }|^{2}}) \\
&\lesssim & \int_{\mathbb{R}^{3}}d\bar{u}_{||}\left\{ \int_{\mathbb{R}^{2}}\frac{%
e^{-C|u_{\perp }|^{2}}d\bar{u}_{\perp }}{|\bar{v}+\bar{u}_{\perp
}|^{1-\varepsilon}\alpha ^{1/2-\varepsilon}}\right\} \frac{\langle \bar{u}_{||}\rangle^{2}e^{-C\{|u_{||}+v|^{2}\}}}{%
|u_{\parallel}+u_{\perp}|^{1-\kappa }}(1+\frac{1}{|u_{||}+u_{\perp }|^{2}}),
\end{eqnarray*}%
where we have used
\begin{equation}\notag
\begin{split}
&\int_{|u_{\parallel,3}| \geq |\bar{v} + \bar{u}_{\perp}|^{1-\varepsilon} \alpha^{\frac{1}{2} - \varepsilon}  } 
\frac{e^{-|u_{\parallel,3} -v_{3}|^{2}}}{|u_{\parallel,3}|^{1-\varepsilon}} \frac{1}{|u_{\parallel,3}|^{1+\varepsilon}}
\mathrm{d}u_{\parallel,3}\\
&\lesssim
\frac{1}{|\bar{v}+ \bar{u}_{\perp}|^{(1-\varepsilon)(1+ \varepsilon^{2})}  \alpha^{(\frac{1}{2} -\varepsilon)(1+  \varepsilon^{2})} }
 \int_{\mathbb{R}}  \frac{e^{-|u_{\parallel,3} -v_{3}|^{2}}}{|u_{\parallel,3}|^{1-\varepsilon^{2}}}\mathrm{d}u_{\parallel,3}\\
& \lesssim  \frac{1}{|\bar{v}+ \bar{u}_{\perp}|^{1-(1-\varepsilon) (1+ \varepsilon^{2})}  \alpha^{ \frac{1}{2}-(1-\varepsilon) (1+ \varepsilon^{2})}}.
 \end{split}
\end{equation}

We note that, by separating $|\xi| \geq \delta $ or $|\xi| \leq \delta ,$ we can
write $\alpha ^{1/2- (1-\varepsilon) (1+ \varepsilon^{2})}\geq \{n\cdot \lbrack \bar{v}+\bar{u}_{\perp }]\}^{1- (1-\varepsilon) (1+ \varepsilon^{2})}$
and $|\bar{v}+\bar{u}_{\perp }|^{1- (1-\varepsilon) (1+ \varepsilon^{2})}\geq \{n^{\perp }\cdot \lbrack \bar{v}+%
\bar{u}_{\perp }]\}^{1-(1-\varepsilon) (1+ \varepsilon^{2})},$ where $n^{\perp} = \left[\begin{array}{cc} 0 & -1 \\ 1 & 0\end{array} \right] n,x$ so that the inner 2D integral are two convergent
1D one 
\begin{eqnarray*}
&&\int_{|(\bar{v}+\bar{u}_{\perp }) \cdot n^{\perp}|\geq 1}\frac{e^{-C|u_{\perp }|^{2}}d\bar{u}%
_{\perp }}{\alpha ^{1/2- \varepsilon }}+\int_{|\bar{v}+\bar{u}_{\perp }|\leq 1,|n^{\perp
}\{\bar{v}+\bar{u}_{\perp }\}|\leq 1}\frac{d\bar{u}_{\perp }}{|\bar{v}+\bar{u%
}_{\perp }|^{1- \varepsilon }\alpha ^{1/2- \varepsilon}} 
 \\
&\leq &1
+\int_{|n^{\perp }\{\bar{v}+\bar{u}_{\perp }\}|\leq 1}\frac{%
e^{-C|u_{\perp }|^{2}}d\bar{u}_{\perp }}{\alpha ^{1/2-  \varepsilon}}
+\int_{ |n^{\perp }\{\bar{v}+\bar{u}_{\perp }\}|\leq 1}\frac{d%
\bar{u}_{\perp }}{|\bar{v}+\bar{u}_{\perp }|^{1- \varepsilon}\alpha ^{1/2- \varepsilon}}\\
&&+\int_{|n^{\perp }\{\bar{v}+\bar{u}_{\perp }\}|\leq 1}\frac{e^{-C|u_{\perp
}|^{2}}d\bar{u}_{\perp }}{|\bar{v}+\bar{u}_{\perp }|^{1- \varepsilon}} \\ 
&<&+\infty .
\end{eqnarray*}

Similarly, the first term is bounded by 
\begin{equation*}
||\langle v\rangle^{\zeta}e^{\theta|v|^{2}}f||_{\infty }||\partial _{v_{3}}f||\int_{\mathbb{R}%
^{3}}du_{||}\left\{ \int_{|u_{||,3}|\geq |\bar{v}+\bar{u}_{\perp
}|^{1- \varepsilon}\alpha ^{1/2- \varepsilon}}\frac{e^{-|v_{3}-u_{||,3}|^{2}}}{|u_{||,3}|^{2}}%
\right\} \frac{q_{0}^{\ast }(\frac{|u|}{|v+w|})u_{||,i} \mathrm{d}\bar{u}_{\perp }}{%
|u_{||}+u_{\perp }|^{\kappa -1}}e^{-\frac{|u+v+w|^{2}}{4}},
\end{equation*}%
and the same argument yields the same bound.

We now turn to 
\begin{equation*}
e^{-\varpi \langle v\rangle s}\int_{\mathbb{R}^{3}}\mathrm{d}u_{||}f(v+u_{||})\int_{\mathbb{R}^{2}}\left\{ 1-\chi
\left( \frac{|\bar{v}+\bar{u}_{\perp }|^{1-  \varepsilon}\alpha ^{1/2- \varepsilon}}{u_{||,3}}\right)
\right\} \mathrm{d}u_{\perp }\partial _{v_{i}}f(v+u_{\perp })e^{-\frac{|u+v+w|^{2}}{4}%
}\frac{q_{0}^{\ast }(\frac{|u|}{|v+w|})}{|u_{||}||u+w|^{1-\kappa  }}.
\end{equation*}%
In this case, 
\begin{equation*}
|\bar{v}+\bar{u}_{\perp }|^{1- \varepsilon}\alpha ^{1/2- \varepsilon}\geq |u_{||,3}|.
\end{equation*}%
We now parametrize $\mathrm{d}u_{\perp }$ in two different ways. 

We decompose
\begin{equation}\label{decomp_app}
\bar{u}_{\parallel} = \bar{u}_{\parallel,n} + \bar{u}_{\parallel,n^{\perp}} := \bar{u}_{\parallel} \cdot n + \bar{u}_{\parallel} \cdot n^{\perp}.
\end{equation}

If $|u_{||,3}|\geq|\bar{u}_{\parallel,n^{\perp}}|,$ then we use
the same parametrization to get%
\begin{eqnarray*}
&&e^{-\varpi \langle v\rangle s} e^{\varpi\langle v+ u_{\perp} \rangle s}\int_{\mathbb{R}^{3}} \mathrm{d}u_{||}f(v+u_{||})\int_{\mathbb{R}^{2}} \mathrm{d}\bar{u}%
_{\perp } e^{-\varpi\langle v+ u_{\perp} \rangle s}\partial _{v_{i}}f(\bar{v}+\bar{u}_{\perp} ,v_{3}-%
\frac{\bar{u}_{||}\cdot \bar{u}_{\perp }}{u_{||,3}}) \\
&&\times \lbrack 1-\chi \left( \frac{|\bar{v}+\bar{u}_{\perp }|^{1- \varepsilon}\alpha
^{1/2- \varepsilon}}{u_{||,3}}\right) ]e^{-\frac{|u_{\parallel}+v+u_{\perp}|^{2}}{4}}\frac{q_{0}^{\ast }(%
\frac{|u_{\parallel}|}{|u_{\parallel}+u_{\perp}|})}{|u_{||,3}||u_{\parallel}+u_{\perp}|^{1-\kappa  }} \\
&\lesssim_{s} &||e^{\theta|v|^{2}}f||_{\infty } ||e^{-\varpi \langle v\rangle s}\frac{\bar{v}\alpha ^{1/2}}{%
\langle \bar{v}\rangle }  \partial _{v_{i}}f (s)||_{\infty }\\
&& \ \ \ \ \ \times\int_{\mathbb{R}^{3}}\mathrm{d}u_{||}\int_{|\bar{v%
}+\bar{u}_{\perp }|^{1- \varepsilon}\alpha ^{1/2- \varepsilon}\geq u_{||,3}}\frac{\mathrm{d}\bar{u}_{\perp }}{%
|\bar{v}+\bar{u}_{\perp }|\alpha ^{1/2}}\frac{e^{-C|u_{\perp}|^{2}} e^{-C|v+u_{\parallel}|^{2}} }{|u_{||,3}| |u_{\parallel} + u_{\perp}|^{1-\kappa}}.
\end{eqnarray*}%
First we integrate $u_{\parallel,n}$ to drop $\frac{1}{| u_{\parallel} + u_{\perp} |^{1-\kappa}}$ singular term for $0 < \kappa\leq 1$
\[
 \int  \frac{{e^{-|v_{n} + u_{\parallel,n}|^{2}}}    }{ | u_{\parallel } - u_{\perp } |^{1-\kappa}}   \mathrm{d}u_{\parallel,n} \leq \int  \frac{e^{-|v_{n} + u_{\parallel,n}|^{2}}}{ | u_{\parallel,n} - u_{\perp,n} |^{1-\kappa}}\mathrm{d}u_{\parallel,n}<\infty,
\]
so that we only need to bound
\begin{eqnarray*}
&&\int \mathrm{d}\bar{u}_{||,n^{\perp}}\int \frac{e^{-|\bar{u}_{\perp }|^{2}-|v+u_{||}|^{2}}\mathrm{d}\bar{u}%
_{\perp }}{|\bar{v}+\bar{u}_{\perp }|^{2\varepsilon }\alpha ^{2\varepsilon }|%
\bar{v}+\bar{u}_{\perp }|^{1-2\varepsilon }\alpha ^{1/2-2\varepsilon }}\frac{1%
}{|u_{||,3}| } \\
&\leq &\int \mathrm{d}\bar{u}_{||,n^{\perp}}\left\{ \int \frac{e^{-|\bar{u}_{\perp
}|^{2}}\mathrm{d}\bar{u}_{\perp }}{|\bar{v}+\bar{u}_{\perp
}|^{\varepsilon }\alpha ^{\varepsilon }}\right\} \frac{e^{|v+u_{\parallel} + u_{\perp}|^{2}}e^{-|v+u_{||}|^{2}}}{|u_{||,3}|^{2-2\varepsilon} }.
\end{eqnarray*}%

The inner integral is finite, since $\alpha \geq |n\cdot \{\bar{v}+\bar{u}%
_{\perp }\}|=|\bar{v}\cdot n + \bar{u}_{\perp,n}|,$ and the integral is a 1D integral:%
\begin{equation*}
 \int_{\mathbb{R}}\frac{e^{-|\bar{u}_{\perp,n }|^{2}} \mathrm{d}\bar{u}_{\perp,n }%
}{|n\cdot \bar{v}+ \bar{u}_{\perp,n } |^{3 \varepsilon }}<+\infty .
\end{equation*}%
Moreover, from $|u_{||,3}| \geq |\bar{u}_{||,n^{\perp}}|,$ the outer integral takes
the form 
\begin{equation*}
\int_{ \mathbb{R}}\frac{ e^{-|u_{\parallel,3}+ v_{3}|^{2}}e^{-|v_{\parallel,n^{\perp}} + u_{\parallel,n^{\perp}} |^{2}}  \mathrm{d}u_{||,3} \mathrm{d}u_{\parallel,n^{\perp}}}{|u_{||,3}|^{2-\varepsilon}}%
\leq \int \frac{ e^{-|u_{\parallel,3}+ v_{3}|^{2}}e^{-|v_{\parallel,n^{\perp}} + u_{\parallel,n^{\perp}} |^{2}}  \mathrm{d}u_{||},_{n^{\perp}} \mathrm{d}u_{||,3}}{%
\{|u_{||,n^{\perp}}|+|u_{||,3}|\}^{2-\varepsilon}} <\infty .
\end{equation*}%
We are done in this case.

We now consider the case $|u_{||,3}|\leq |u_{\parallel,n^{\perp}}|.$ We now choose a different parametrization. We define
\begin{equation}\notag
u_{\perp,n} := u_{\perp} \cdot n, \ \ \  u_{\perp,n^{\perp}}:= u_{\perp} \cdot n^{\perp}.
\end{equation}
Now we
choose $_{{}} {u}_{\perp ,n }$ and $ {u}_{\perp ,3}$ as parameters so
that $ {u}_{\perp ,n^{\perp}}=-\frac{ {u}_{\perp ,n}u_{||,n}+u_{\perp ,3}u_{||,3}}{%
u_{||,n^{\perp}}}$ and 
\begin{equation*}
\mathrm{d}u_{\perp }=\frac{|u_{||}|}{|u_{||,n^{\perp}}|}\mathrm{d}u_{\perp ,n}\mathrm{d}u_{\perp ,3},
\end{equation*}%
so that we need to bound
\begin{eqnarray*}
&&e^{-\varpi \langle v\rangle s}\int_{\mathbb{R}^{3}}\mathrm{d}u_{||}\\
&&\times f(v+u_{||})\int_{\mathbb{R}^{2}}\mathrm{d} {u}%
_{\perp ,n} \mathrm{d}u_{\perp,3}\partial _{v_{i}}
f(v_{n}+u_{\perp ,n},v_{n^{\perp}}-\frac{u_{\perp,n}u_{||,n}+u_{\perp ,3}u_{||,3}}{u_{||,n^{\perp}}},v_{3}+u_{\perp ,3})\frac{|u_{||}|
}{|u_{||,n^{\perp}}|}  \\
&&\times \Big\{ 1-\chi \left( \frac{|\bar{v}+\bar{u}_{\perp }|^{1-}\alpha
^{1/2-}}{u_{||,3}}\right) \Big\}e^{-\frac{|u_{\parallel}+v+u_{\perp}|^{2}}{4}}\frac{q_{0}^{\ast }(%
\frac{|u_{\parallel}|}{|u_{\parallel}+u_{\perp}|})}{|u_{|| }||u_{\parallel}+u_{\perp}|^{1-\kappa  }}.
\end{eqnarray*}
Directly this is bounded by $|| e^{\theta|v|^{2}}f||_{\infty } || e^{-\varpi \langle v\rangle s}\frac{ |\bar{v}|\alpha ^{1/2}\partial _{v_{i}}f}{%
\langle \bar{v}\rangle }||_{\infty }\times $
\begin{eqnarray*}
&&\int_{\mathbb{R}^{3}}\mathrm{d}u_{||}\int_{|\bar{v%
}+\bar{u}_{\perp }|^{1- \varepsilon}\alpha ^{1/2-  \varepsilon}\geq |u_{||,3}|}
\frac{ \langle \bar{v} + \bar{u}_{\perp}\rangle
e^{-|\bar{u}%
_{\perp }|^{2}-|v+u_{||}|^{2}}\mathrm{d}u_{\perp ,n}\mathrm{d}u_{\perp ,3}}{|\bar{v}+\bar{u}%
_{\perp }|\alpha ^{1/2}}\frac{\mathrm{d}u_{\parallel}}{|u_{||,n^{\perp}}||u_{\parallel}+u_{\perp}|^{1-\kappa  }} \\
&\lesssim_{s} & \int_{\mathbb{R}^{3}}\mathrm{d}u_{||}\int_{|\bar{v%
}+\bar{u}_{\perp }|^{1- \varepsilon}\alpha ^{1/2- \varepsilon}\geq | u_{||,3}|}\frac{   \langle \bar{v} + \bar{u}_{\perp} \rangle e^{-|\bar{u}%
_{\perp,n }|^{2}-|v+u_{||}|^{2}}     } {|\bar{v}+\bar{u}%
_{\perp }|^{ 2\varepsilon }\alpha ^{2 \varepsilon }|\bar{v}+\bar{u}_{\perp
}|^{1-2\varepsilon }\alpha ^{1/2-2\varepsilon }}\frac{1}{|u_{||,n^{\perp}}|}  \mathrm{d}\bar{u}_{\perp ,n} .
\end{eqnarray*}
where we integrate $ u_{\perp,3}$ first to drop $\int_{\mathbb{R}}\frac{ e^{-C|u_{\perp,3}|^{2}}  \mathrm{d}u_{\perp,3}}{|u_{\parallel} + u_{\perp}|^{1-\kappa}} \lesssim \int_{\mathbb{R}}\frac{ e^{-C|u_{\perp,3}|^{2}}  \mathrm{d}u_{\perp,3}}{|u_{\parallel,3} + u_{\perp,3}|^{1-\kappa}}<+\infty.$

In the case of $|\bar{v} + \bar{u}_{\perp}|\leq 1,$ this is bounded by
\begin{eqnarray*}
&\lesssim_{s} & \int_{ \mathbb{R}^{3} }\mathrm{d}u_{||}\int \frac{e^{-| {u}_{\perp }|^{2}-|v+u_{||}|^{2}}\mathrm{d}u_{\perp
,n} }{|\bar{v}+\bar{u}_{\perp }|^{2\varepsilon }\alpha
^{2\varepsilon }}\frac{1}{|u_{||,n^{\perp}}||u_{||,3}|^{1- \varepsilon}  } \\
&\lesssim_{s} & \int_{  \mathbb{R}^{3} }du_{||}\int \frac{e^{-| {u}_{\perp }|^{2}-|v+u_{||}|^{2}}\mathrm{d}u_{\perp
,n} }{ |\bar{u}_{\perp,n} + \bar{v}_{n}|^{4\varepsilon}    }\frac{1}{%
|u_{||,n^{\perp}}||u_{||,3}|^{1- \varepsilon}    }
 \\
&\lesssim_{s} & \int_{  \mathbb{R}^{3} }\mathrm{d}u_{||} \Big\{\int \frac{e^{-| {\bar{u}}_{\perp,n }|^{2}}\mathrm{d}u_{\perp ,n}}
{ 
 |\bar{u}_{\perp,n} + \bar{v}_{n}|^{4\varepsilon} 
}\Big\}
\frac{e^{-|v+u_{||}|^{2}}  }{%
|u_{||,n^{\perp}}||u_{||,3}|^{1- \varepsilon} },
\end{eqnarray*}%
where the inner integral is 1D which is finite and bounded. On the other hand, from the assumption 
$|u_{||,3}|\leq |\bar{u}_{||,n^{\perp}}|,$ the outer integral is%
\begin{eqnarray*}
&&\int \frac{e^{-|v+u_{||}|^{2}} \mathrm{d}\bar{u}_{||,n^{\perp}} \mathrm{d}\bar{u}_{\parallel,n} \mathrm{d}u_{||,3}  }{%
|\bar{u}_{||,n^{\perp}}||u_{||,3}|^{1- \varepsilon}   } \\
&\leq &\int \left\{ \int_{0}^{|\bar{u}_{||,n^{\perp}}|}\frac{\mathrm{d}u_{||,3}}{|u_{||,3}|^{1- \varepsilon}}%
\right\}   \frac{e^{-|\bar{v}_{n}+\bar{u}_{||,n}|^{2}} e^{-|\bar{v}_{n^{\perp}}+\bar{u}_{||,n^{\perp}}|^{2}}  }{|\bar{u}_{||,n^{\perp}}|} \mathrm{d}\bar{u}_{||,n^{\perp}}\mathrm{d}\bar{u}_{||,n }\\
&\leq &\iint_{\mathbb{R}^{2}}  \frac{e^{-|\bar{v}_{n}+\bar{u}_{||,n}|^{2}} e^{-|\bar{v}_{n^{\perp}}+\bar{u}_{||,n^{\perp}}|^{2}}  }{|\bar{u}_{||,n^{\perp}}|} \mathrm{d}\bar{u}_{||,n^{\perp}}\mathrm{d}\bar{u}_{||,n }  <\infty .
\end{eqnarray*}

In the case of $|\bar{v} + \bar{u}_{\perp}|\geq 1$ we bound the integration as
\begin{equation}\notag
\begin{split}
&\int_{\mathbb{R}^{3}} \int_{\mathbb{R}}
\frac{   \langle \bar{v} + \bar{u}_{\perp} \rangle^{\varepsilon} \langle \bar{v} + \bar{u}_{\perp} \rangle^{1-\varepsilon}
}{   |u_{\parallel,3}|^{\frac{2 \varepsilon}{1-2  \varepsilon}} |\bar{v}+ \bar{u}_{\perp}|^{1-\varepsilon} \alpha^{\frac{1}{2}-\varepsilon} } \frac{e^{-|\bar{u}_{\perp,n}|^{2}}e^{-|v+ u_{\parallel}|^{2}} }{|u_{\parallel,n^{\perp}}|  |u_{\parallel} + u_{\perp}|^{1-\kappa}}  \mathrm{d}\bar{u}_{\perp,n}  \mathrm{d}u_{\parallel}\\
\lesssim&   
\int_{\mathbb{R}^{3}} \mathrm{d}u_{\parallel} \int 
\frac{
 \langle \bar{v}_{n} + \bar{u}_{\perp,n} \rangle^{\varepsilon}  
 \langle \bar{v}_{n^{\perp}}  \rangle^{\varepsilon} 
 e^{-|\bar{u}_{\perp,n}|^{2} }  e^{-|v+u_{\parallel}|^{2}}
 }{ |u_{\parallel,3}|^{ {\frac{2 \varepsilon}{1-2  \varepsilon}}}  
 [
 \bar{u}_{\perp,n} + \bar{v}_{n} 
 ]^{2(\frac{1}{2}-\varepsilon)} |\bar{u}_{\parallel,n^{\perp}} | } \mathrm{d}\bar{u}_{\perp,n}.
\end{split}
\end{equation}

Again $\int_{0}^{|\bar{u}_{\parallel, n^{\perp}}|} \frac{\mathrm{d}u_{\parallel,3}}{|u_{\parallel,3}|^{
\frac{2\varepsilon}{1-2 \varepsilon}
}
}\lesssim |\bar{u}_{\parallel, n^{\perp}}|^{1-  \frac{2\varepsilon}{1-2 \varepsilon}}$ and hence the integration is bounded by
\begin{equation}\notag
\begin{split}
\langle \bar{v}_{n} \rangle^{\varepsilon} \langle \bar{v}_{n^{\perp}} \rangle^{\varepsilon} \iint \frac{e^{-|\bar{u}_{\perp,n}|^{2}}  e^{-|\bar{v} + \bar{u}_{\parallel}|^{2}}  }{ |\bar{u}_{\parallel,n^{\perp}}|^{ \frac{2\varepsilon}{1-2 \varepsilon}}   | \bar{u}_{\perp,n} + \bar{v}_{n}  |^{1-2\varepsilon}  } \leq \langle \bar{v}_{n} \rangle^{\varepsilon} \langle \bar{v}_{n^{\perp}} \rangle^{\varepsilon} \langle \bar{v}_{  n^{\perp}} \rangle^{- \frac{2\varepsilon}{1-2 \varepsilon}} \langle \bar{v}_{n}\rangle^{-(1-2\varepsilon)} \lesssim1.
\end{split}
\end{equation} 
\end{proof}

 Our main result for 2D specular case is the following.

\begin{theorem}\label{theo2D} 
Assume a stronger cut-off assumption on $q_{0}(\theta)$ of (\ref{collision_Q})
\begin{equation}\label{strong_cutoff}
\Big|\nabla_{v} q_{0} (\frac{v-u}{|v-u|}\cdot \omega)\Big| \Big{/} \Big|  \frac{v-u}{|v-u|}\cdot \omega \Big| \lesssim1.
\end{equation}
Assume $f_{0}\in W^{1,\infty }$ with (\ref{specularBC}).
Assume that
\begin{equation*}
\sup_{0<t\leq T}\{||  e^{\theta |v|^{2}}f(t)||_{\infty }+||  \partial _{v_{3}}f(t)||_{\infty } \}\leq c_{T,\zeta
,f_{0}}<+\infty ,
\end{equation*}%
then
\begin{equation}\notag
\begin{split}
& \sup_{0\leq t\leq T}\{||e^{-\varpi\langle \bar{v} \rangle t}\frac{\alpha }{1+|\bar{v}|^{2}}%
\nabla _{x}f(t)||_{\infty }+||e^{-\varpi\langle \bar{v}\rangle t} \frac{ |\bar{v}| }{\langle \bar{v} \rangle}\sqrt{\alpha }\nabla
_{v}f(t)||_{\infty }\} \\ 
&\lesssim  _{T,\Omega ,L}|| \frac{\alpha^{1/2}}{\langle \bar{v} \rangle}  \nabla _{\bar{x} }f_{0}||_{\infty } + || \frac{|\bar{v}|^{2}}{\langle v\rangle} \nabla_{\bar{v}} f_{0} ||_{\infty} + || \partial_{v_{3}}f||_{\infty}+P(||  e^{\theta
|v|^{2}}f_{0}||_{\infty }),
\end{split}
\end{equation}%
where $P\,$\ is some polynomial. If $f_{0}\in C^{1}$ and the compatibility conditions (\ref{compatibility_specular}) and (\ref{compatibility_specular_t}) are satisfied,
then $f\in C^{1}$ away from the grazing set $\gamma _{0}.$ Furthermore, if $%
|| e^{\theta |v|^{2}}f_{0}||_{\infty }\ll1,$ and $\partial \Omega $ (therefore $\xi$) is
real analytic, then $T$ can be arbitrarily large.
\end{theorem}

We remark that powers\ of singularity\ $\alpha $ and $\sqrt{\alpha }$ are
barely missed in 3D case (borderline case).

\begin{proof}
We repeat our program in 3D for the simpler 2D case, and we only point out
the differences. Lemma \ref{local_existence} is valid with easy adaptations.
The new $  \partial _{v_{3}}f(t)$
estimate follows from taking the $v_{3}$ derivative
\begin{equation*}
\begin{split}
&\{\partial _{t}+v_{1}\partial _{x_{1}}+v_{2}\partial _{x_{2}}\} \partial _{v_{3}}f   +    \nu(F) \partial_{v_{3}} f \\
&=   \Gamma_{\text{gain},v_{3}} (f,f) + \Gamma_{\text{gain}}(\partial_{v_{3}}f,f) + \Gamma_{\text{gain}}(f,\partial_{v_{3}}f)  - 
\nu_{v_{3}}(\sqrt{\mu}f) f  - \nu(\sqrt{\mu} \partial_{v_{3}} f)f .
\end{split}
\end{equation*}%
Since
\begin{eqnarray*}
 | \nu_{v_{3}} (\sqrt{\mu} f)f |+ | \Gamma _{\text{gain},v_{3}} (f,f)|&\lesssim &   P(||  e^{\theta
|v|^{2}}f||_{\infty }), \\
| \nu(\sqrt{\mu} \partial_{v_{3}} f)f|+|\Gamma_{\mathrm{gain}}(\partial_{v_{3}} f,f)| &\lesssim &  P(||  e^{\theta
|v|^{2}}f||_{\infty }) \int \frac{e^{-C|v-u|^{2}}}{|v-u|^{2-\kappa}} |\partial_{v_{3}}f (u)| \mathrm{d}u ,
\end{eqnarray*}%
and for $(x,v)\in\gamma_{-}$
\[
\partial_{v_{3}}f(t,x,v) = \partial_{v_{3}}f(t,x,R_{x}v),
\]
then we follow the proof of Lemma \ref{local_existence} (similar to $\partial_{t}f$ proof) to conclude
\begin{equation*}
|| \partial _{v_{3}}f(t)||_{\infty }\lesssim
|| \partial _{v_{3}}f_{0}||_{\infty }+ P(|| e^{\theta
|v|^{2}}f||_{\infty }).
\end{equation*}%

The Velocity lemma (Lemma \ref{velocity_lemma}) is valid with changing $v$ to $\bar{v}.$ The non-local to local estimates (\ref{nonlocal}) and (\ref%
{specular_nonlocal}) are valid for $0<\kappa \leq 1$ for $\bar{v}%
=(v_{1},v_{2})$: In the proof of (\ref{nonlocal}) in Lemma \ref{lemma_nonlocal}, \text{Step 1}, the claim (\ref{int_xi}) is valid. \textit{Step 2}, (\ref{sigma}) and (\ref{COV_xi_s})is valid with $\alpha(x,\bar{v})$. In \text{Step 3} we define $\tilde{\sigma}_{1}$ and $\tilde{\sigma}_{2}$ with changing $v$ to $\bar{v}$. Then (\ref{max_xi}), and (\ref{tz}) hold with changing $v$ to $\bar{v}$. We follow the same proof of \textit{Step 4} to bound $\int_{0}^{t_{\mathbf{b}}(x, \bar{v})  } \frac{ e^{-l \langle \bar{v} \rangle (t-s)}  }{|v|^{2\beta-1} |\xi|^{\beta-\frac{1}{2}}} Z(s,v) \mathrm{d}s.$ We use $\frac{1}{|v|} \leq \frac{1}{|\bar{v}|}$ to conclude (\ref{nonlocal}). For the proof of (\ref{specular_nonlocal}) in Lemma \ref{lemma_nonlocal}, we use the same time splitting of (\ref{time_splitting}) with changing $|v|$ to $|\bar{v}|.$ Then all the proofs are followed and we conclude the proof using $\frac{1}{|v|} \leq \frac{1}{|\bar{v}|}$.

The fundamental Theorem \ref{theorem_Dxv} is valid with simpler proof with changing all $v$ to $\bar{v}$. In
fact, due to topological advantage, we can use a global chart $x_{||}=\theta
$ in $\mathbb{R}^{1}$ (such as the polar co-ordinates) for the boundary as
\begin{equation*}
\eta (x_{||})=[R(x_{||})\cos x_{||},R(x_{||})\sin x_{||}],
\end{equation*}%
(vector-valued function) with a global ODE for in the polar co-ordinate
system near the boundary$!$ The proof of Theorem \ref{theorem_Dxv} follows
step by step of the 3D case but with simpler argument without changes of
charts. The estimate of $e^{-\varpi\langle v\rangle t}\frac{\alpha }{1+|v|^{2}}%
\nabla _{x}f(t)$ exactly as in 3D case, valid for $\alpha .$ The most
delicate part is to estimate $\partial _{v_{3}}\Gamma _{\text{gain}}(f,f),$
where a weight stronger than $\sqrt{\alpha }$, due to $\beta >1/2$ in (\ref%
{specular_nonlocal}). It is important to know, that we are unable to
establish (\ref{specular_nonlocal}) in the 2D case with $\beta=1/2$. However, we are able to
close the estimate by using additional bounds on $\partial _{v_{3}}f$. 

Basically we follow the {\textit{Proof of Theorem \ref{main_specular}}} but we use Lemma \ref{2D_Gamma} when derivatives act on $\bar{V}_{\mathbf{cl}}(s)$ argument of $\Gamma_{\mathrm{gain}}(f^{m-\ell}, f^{m-\ell})(s, X_{\mathbf{cl}} (s), \bar{V}_{\mathbf{cl}}(s),v_{3}).$

More precisely we use Lemma \ref{lemma_operator} for $\mathbf{e}\in \{x_{1}, x_{2}, v_{1}, v_{2} \}$
\begin{equation}\notag
\begin{split}
&\text{II}_{\mathbf{e}} \text{ of } (\ref{deriv_spec})\\
&= \int_{0}^{t} \mathrm{d}%
s \sum_{\ell=0 }^{\ell_{*}(0)}  \mathbf{1}_{[t^{\ell+1},t^{\ell})}(s) e^{-\int^{t}_{s}
\sum_{j=0}^{\ell_{*}(s)} \mathbf{1}_{[{t^{j+1}},{t^{j}})}(\tau) \nu (F^{m-j})(\tau )%
\mathrm{d}\tau }   \\
&  \ \ \ \ \  \times \bigg\{  \partial_{\mathbf{e}} \bar{X}_{\mathbf{cl}}(s)\cdot\big[    \Gamma_{\mathrm{gain}}  (\nabla_{\bar{x}}f^{m-\ell}, f^{m-\ell})
+
  \Gamma_{\mathrm{gain}}  (f^{m-\ell}, \nabla_{\bar{x}}f^{m-\ell})
  \big]  (s,\bar{X}_{\mathbf{cl}}(s),V_{\mathbf{cl}}(s))\\
& \ \ \ \ \ \ \ \ \ \  + \partial_{\mathbf{e}} \bar{V}_{\mathbf{cl}}(s)\cdot\nabla_{\bar{v}}\big[    \Gamma_{\mathrm{gain}}  ( f^{m-\ell}, f^{m-\ell})
  \big]  (s,\bar{X}_{\mathbf{cl}}(s),V_{\mathbf{cl}}(s)) \bigg\}
\end{split}
\end{equation}
\begin{equation}\notag
\begin{split}
& \ -\int_{0}^{t} \mathrm{d}%
s   \sum_{\ell =0}^{\ell_{*}(0)}  \mathbf{1}_{[t^{\ell+1},t^{\ell})}(s) e^{-\int^{t}_{s}
\sum_{j} \mathbf{1}_{[{t^{j+1}},{t^{j}})}(\tau) \nu (F^{m-j})(\tau )%
\mathrm{d}\tau }    \Gamma_{\text{gain}}(f^{m-\ell}, f^{m-\ell}) (s,\bar{X}_{\mathbf{cl}}(s),V_{\mathbf{cl}}(s))  \\  
& \ \ \ \ \  \times
 \int^{t}_{s}
\mathrm{d}\tau  
\sum_{j=0}^{\ell_{*}(s)} \mathbf{1}_{[{t^{j+1}},{t^{j}})}(\tau)  \bigg\{\partial_{\mathbf{e}} \bar{X}_{\mathbf{cl}}(s) \cdot 
 \nu (\sqrt{\mu} \nabla_{\bar{x}}f^{m-j})(\tau, \bar{X}_{\mathbf{cl}}(\tau), V_{\mathbf{cl}}(\tau) ) \\
 & \ \ \ \ \ \ \ \ \ \ \ \ \ \ \ \  \ \ \ \ \ \ \ \ \ \ \ \ \ \ \ \ \ \ \ \ \  + \int_{\mathbb{R}^{3}} \partial_{\mathbf{e}} \bar{V}_{\mathbf{cl}}(s)\cdot \nabla_{\bar{v}} B(v-u,\omega) \sqrt{\mu(u)} f^{m-j} (\tau, \bar{X}_{\mathbf{cl}}(\tau), u)  \mathrm{d}u \bigg\}
 \\
&  \ - e^{- \int^{t}_{0}   \sum_{\ell=0 }^{\ell_{*}(0)} \mathbf{1}_{[{t^{\ell+1}} ,{t^{\ell}})}(s) \nu(F^{m-\ell} )(s) \mathrm{d}s} \
   f_{0}(X_{\mathbf{cl}}(0),V_{\mathbf{cl}}(0))  \int^{t}_{0}   \sum_{\ell=0 }^{\ell_{*}(0)} \mathbf{1}_{[{t^{\ell+1}} ,{t^{\ell}})}(s) \\
&  \ \ \ \ \ \  
     \times  \bigg\{\partial_{\mathbf{e}} \bar{X}_{\mathbf{cl}}(s) \cdot 
 \nu (\sqrt{\mu} \nabla_{\bar{x}}f^{m-j})(\tau, \bar{X}_{\mathbf{cl}}(\tau), V_{\mathbf{cl}}(\tau) ) \\
 & \ \ \ \ \ \ \ \    + \int_{\mathbb{R}^{3}} \partial_{\mathbf{e}} \bar{V}_{\mathbf{cl}}(s)\cdot \nabla_{\bar{v}} B(v-u,\omega) \sqrt{\mu(u)} f^{m-j} (\tau, \bar{X}_{\mathbf{cl}}(\tau), u)  \mathrm{d}u \bigg\}.
\end{split}
\end{equation}

We use the crucial lemma (\ref{lemma_Dxv}) for the terms containing $\partial_{\bar{v}}\Gamma_{\mathrm{gain}}$ as
\begin{equation}\notag
\begin{split} 
& \int_{0}^{t}  e^{-\varpi \langle v\rangle t} \frac{\alpha(x,\bar{v})}{\langle \bar{v} \rangle^{2}} |\partial_{\bar{x}} \bar{V}_{\mathbf{cl}}(s)| |
\partial_{\bar{v}} \Gamma_{\mathrm{gain}}(f^{m-\ell}, f^{m-\ell})|\mathrm{d}s\\
&\lesssim \int_{0}^{t}  e^{-\varpi \langle v\rangle(t-s)} \frac{|\bar{v}|^{3} e^{C|\bar{v}||t-s|}}{\langle \bar{v} \rangle^{2}}  | e^{\varpi \langle v \rangle s} \partial_{\bar{v}}\Gamma_{\mathrm{gain}}(f^{m-\ell}, f^{m-\ell})  | \mathrm{d}s\\
&\lesssim \int_{0}^{t}  e^{-\varpi \langle \bar{v} \rangle (t-s)} |\bar{v}| \mathrm{d}s \times \{ \text{ RHS of } (\ref{2dv})\}\\ 
& \lesssim \frac{1}{\varpi} P (|| e^{\theta|v|^{2}} f_{0}||_{\infty} ) \{1+ || \partial_{v_{3}} f_{0}||_{\infty}
+ \Big|\Big| e^{-\varpi \langle v\rangle s} \frac{|\bar{v}|}{\langle \bar{v} \rangle } \alpha^{1/2} \partial_{\bar{v}}f \Big|\Big|_{\infty}
\}.
\end{split}
\end{equation} 

Similarly
\begin{equation}\notag
\begin{split} 
&  \int_{0}^{t}  e^{-\varpi \langle v\rangle t} \frac{\alpha(x,\bar{v})}{\langle \bar{v} \rangle^{2}} |\partial_{\bar{v}} \bar{V}_{\mathbf{cl}}(s)| |
\partial_{\bar{v}} \Gamma_{\mathrm{gain}}(f^{m-\ell}, f^{m-\ell})|\mathrm{d}s\\
&\lesssim \int_{0}^{t}  e^{-\varpi \langle v\rangle(t-s)} \frac{|\bar{v}|^{2} e^{C|\bar{v}||t-s|}}{\langle \bar{v} \rangle }  | e^{\varpi \langle v \rangle s} \partial_{\bar{v}}\Gamma_{\mathrm{gain}}(f^{m-\ell}, f^{m-\ell})  | \mathrm{d}s\\
&\lesssim \int_{0}^{t}  e^{-\varpi \langle \bar{v} \rangle (t-s)} |\bar{v}| \mathrm{d}s \times \{ \text{ RHS of } (\ref{2dv})\}\\ 
& \lesssim \frac{1}{\varpi} P (||   e^{\theta|v|^{2}} f_{0}||_{\infty} ) \{1+ || \partial_{v_{3}} f_{0}||_{\infty}
+ \Big|\Big| e^{-\varpi \langle v\rangle s} \frac{|\bar{v}|}{\langle \bar{v} \rangle } \alpha^{1/2} \partial_{\bar{v}}f \Big|\Big|_{\infty}
\}.
\end{split}
\end{equation}
 
For the term containing $\partial_{x} \bar{V}_{\mathbf{cl}}(s) \cdot \nabla_{\bar{v}}B(v-u,\omega)$ we use (\ref{strong_cutoff}). The estimate for the other terms are same as the proof of Theorem \ref{main_specular}.
\end{proof}

\begin{proposition}[Specular BC]\label{sing_specular}
Assume $\Omega: = \{ \bar{x} \in\mathbb{R}^{2} : |\bar{x}|<1\}$ be 2D disk and $\xi(\bar{x})= |\bar{x}|^{2}-1$. For any $1 \leq k$ assume the compatibility conditions for $0 \leq i \leq k-1$ 
\[
\partial^{i}_{t} f_{0}(x,v) = \partial^{i}_{t} f_{0}(x,R_{x}v)  \ \ \text{on} \ \gamma_{-},
\]
and for some $x_{0} \in \partial\Omega$ and some $u_{0} \in \mathbb{R}^{3}$ with $|u_{0}| \sim 1$ and $n(x_{0})\cdot u_{0} =0,$
\begin{equation}\label{specular_condition}
 \int_{  \substack{n(x_{0})\cdot u_{\tau}=0\\ u_{\tau } \sim u_{0}}} \mathbf{k}_{f_{0}}(v,u) \partial_{v_{n}} \partial_{t}^{k } f_{0} (x_{0}, u) \mathrm{d}u_{\tau}  > C >0,
\end{equation}
where $\mathbf{k}_{f_{0}}$ is defined in (\ref{diff_q_gamma}). Then there exists $  t >0$ such that if $X_{\mathbf{cl}}(0;t,x,v)\sim x_{0}$ then for all $v\in\mathbb{R}^{3}$ we have a blow-up (\ref{blow_up}). 
\end{proposition}

\begin{proof}The crucial ingredients of the proof is a 2D borderline estimate of Theorem \ref{theo2D} (due to Lemma \ref{2D_Gamma}) and the explicit lower bound of (\ref{lower_Dxv}) in Example 1.   

For simplicity we only consider the case of $k=1$. In order to show (\ref{blow_up}) it suffices to show
\begin{equation}\label{blowup_specular}
\partial_{n} \partial_{t} \Gamma_{\mathrm{gain}}(f,f) (t,x_{0},v) - \partial_{n} \partial_{t} \nu(\sqrt{\mu} f) f(t,x_{0},v) = +\infty.
\end{equation}
This is due to the fundamental theory of calculus
\begin{equation}\notag
\begin{split}
&\partial_{n}  \Gamma_{\mathrm{gain}}(f,f) (t,x_{0},v) - \partial_{n}   \nu(\sqrt{\mu} f) f(t,x_{0},v)\\
= &\partial_{n}  \Gamma_{\mathrm{gain}}(f_{0},f_{0}) (x_{0},v) - \partial_{n}   \nu(\sqrt{\mu} f_{0}) f_{0}(x_{0},v) + \int^{t}_0 \partial_{n} \partial_{s} \Gamma_{\mathrm{gain}}(f,f) (s,x_{0},v) - \partial_{n} \partial_{s} \nu(\sqrt{\mu} f) f(s,x_{0},v) \mathrm{d}s,
\end{split}
\end{equation}
where we can choose the initial datum as good as possible.

  We decompose
\begin{equation}\notag
\begin{split}
 \int_{\mathbb{R}^{3}} \mathbf{k}_{f_{0}} (v,u) \frac{\partial_{t}f(t, x-\varepsilon n(x) ,u) - \partial_{t}f(t,x,u)}{\varepsilon} 
=  \int_{|n(x)\cdot u|\leq \varepsilon} +  \underbrace{\int_{\varepsilon \leq |n(x)\cdot u| \leq 1} }_{\mathbf{II}}+ \int_{1\leq |n(x)\cdot u|}.
\end{split}
\end{equation} 
By Lemma \ref{local_existence}, the first term is bounded by 
\[
\int_{|n(x)\cdot u|\leq \varepsilon}  \lesssim O(1) || \partial_{t} f||_{\infty}.
\]
Due to (\ref{ball_n}), $1 \leq |n(x)\cdot u|$ implies $1 \lesssim | n(x-\varepsilon n(x))\cdot u|$ for $0\leq \varepsilon \ll1.$ Therefore we use Theorem \ref{theo2D} to bound the third term as
\begin{equation}\notag
\begin{split}
\int_{1 \leq |n(x)\cdot u|} &\lesssim \int e^{\varpi \langle \bar{u} \rangle t}\frac{1+ |\bar{u}|^{2}}{1+ \varepsilon|\bar{u}|^{2}} \mathbf{k}_{f_{0}}(v,u) \mathrm{d} u\times ||  e^{-\varpi \langle \bar{v} \rangle t}\frac{\alpha}{1 + |\bar{v}|^{2}} \partial_{t} \nabla_{\bar{x}} f(t) ||_{\infty} \\
& \lesssim O_{N,t}(1) ||  e^{-\varpi \langle \bar{v} \rangle t}\frac{\alpha}{1 + |\bar{v}|^{2}} \partial_{t} \nabla_{\bar{x}} f(t) ||_{\infty}.
\end{split}
\end{equation}
 
 Now we focus on the second term $\mathbf{II}.$ Due to (\ref{ball_n}), $\partial_{t} f(t, x-\varepsilon r n(x),u)$ is differentiable for all $0 \leq r \leq 1$ and we have (\ref{n_epsilon}). We further decompose
 \[
\mathbf{II}=  \int_{\varepsilon \leq |n(x)\cdot u| \leq 1}  \int_{0}^{1}  \mathbf{k}_{f_{0}} (v,u) 
 {\partial_{n}\partial_{t}f(t, x-\varepsilon r n(x) ,u)  } \mathrm{d}r \mathrm{d}u
=   \int_{ \substack{\varepsilon \leq |u_{n}| \leq 1 \\  |u_{\tau}| \leq N}} +  \int_{ \substack{\varepsilon \leq |u_{n}| \leq 1 \\  |u_{\tau}| \geq N}}  .
 \]
Set $x(r) : = x- \varepsilon r n(x).$ Now we use (\ref{deriv_spec}) for the first term and apply Theorem \ref{theo2D} to the second term ($|u_{\tau}| \geq N$) to have
\begin{equation}\notag
\begin{split}
 \mathbf{II}  
&\gtrsim  \int_{ \substack{\varepsilon \leq |u_{n}| \leq 1 \\  |u_{\tau}| \leq N}} \mathrm{d}u \int^{1}_{0} \mathrm{d}r \ \mathbf{k}_{f_{0}} (v,u)\big\{ \partial_{n} \bar{X}_{\mathbf{cl}}(0) \cdot \nabla_{\bar{x}}\partial_{t} f_{0} (X_{\mathbf{cl}}(0), V_{\mathbf{cl}}(0))
+ \underbrace{\partial_{n} \bar{V}_{\mathbf{cl}}(0) \cdot \nabla_{\bar{v}}\partial_{t} f_{0} (X_{\mathbf{cl}}(0), V_{\mathbf{cl}}(0))}
\big\}\\
&- P(|| e^{\theta |v|^{2}} f_{0} ||_{\infty}) \int_{ \substack{\varepsilon \leq |u_{n}| \leq 1 \\  |u_{\tau}| \leq N}} \mathrm{d}u \int^{1}_{0} \mathrm{d}r \ \mathbf{k}_{f_{0}} (v,u) \\
&  \ \   \times\bigg\{ \int_{0}^{t} |\partial_{n} \bar{X}_{\mathbf{cl}}(s)| \int_{\mathbb{R}^{3}} \frac{e^{-C_{\theta} |V_{\mathbf{cl} }(s) - u|^{2}}}{|V_{\mathbf{cl} }(s) -u|^{2-\kappa}} |\nabla_{\bar{x}} f(s)|\mathrm{d}u \mathrm{d}s  
 +\int_{0}^{t} |\partial_{n} \bar{V}_{\mathbf{cl}}(s)| \int_{\mathbb{R}^{3}} \frac{e^{-C_{\theta} |V_{\mathbf{cl}}(s) - u |^{2}}}{|V_{\mathbf{cl} }(s) -u|^{2-\kappa}} |\nabla_{\bar{v}} f(s)|\mathrm{d}u \mathrm{d}s \\ 
 &  \ \ \ \ \ \ \  
 + \langle u \rangle^{\kappa} e^{-\theta |u|^{2}} \sup _{ 0 \leq s\leq t} |\partial_{n} \bar{V}_{\mathbf{cl}}(s;t,x,v)| 
\bigg\}\\
&- O(1)  \int_{0}^{1}\int_{\varepsilon \leq |u_{n}| \leq 1} \frac{ e^{-N}}{|u_{n}|^{2} + C \varepsilon r N^{2} }  \mathrm{d}u_{n} \mathrm{d}r. 
\end{split}
\end{equation}
We use (\ref{lower_Dxv}) and (\ref{specular_condition}) and (\ref{bb_dominant}) to bound (lower) the underbraced term
\begin{equation}\notag
\begin{split}
&\gtrsim\int_{0}^{1} \mathrm{d}r\int_{\varepsilon \leq |u_{n}| \leq 1} \mathrm{d}u_{n}
 \int_{|u_{\tau}| \leq N} \mathrm{d}u_{\tau}  
\frac{t|\bar{u}_{\tau}|^{4}}{ |u_{n}|^{2} + C \varepsilon r N^{2}} \mathbf{k}_{f_{0}}(v,u)  \partial_{v_{n}}\partial_{t} f_{0} (X_{\mathbf{cl}}(0) , {V}_{\mathbf{cl}}(0))\\
&\sim \int_{0}^{1} \mathrm{d}r\int_{\varepsilon \leq |u_{n}| \leq 1} \mathrm{d}u_{n}  \frac{1}{|u_{n}|^{2} + C \varepsilon r N^{2}},
\end{split}
\end{equation}
where $X_{\mathbf{cl}}(0) \sim x_{0}$ and $V_{\mathbf{cl}}(0) \sim u_{0}.$

Except the underbraced term, all the other terms are bounded, by Theorem \ref{theo2D} and (\ref{lower_Dxv}) and (\ref{specular_nonlocal}),
\begin{equation}\notag
\begin{split}
\gtrsim&- \int_{\varepsilon \leq |u_{n}| \leq 1} \mathrm{d}u_{n} \int_{|u_{\tau}| \leq N} \mathrm{d}u_{\tau} \mathbf{k}_{f_{0}} (v,u)\frac{t |\bar{u}|^{2}}{ \sqrt{ |u_{n}|^{2 } + C\varepsilon r N^{2}}}  || \nabla_{\bar{x}}\partial_{t} f_{0} ||_{\infty} \\
&- P( || e^{\theta |v|^{2}} f_{0} ||_{\infty} ) \int_{ \substack{\varepsilon \leq |u_{n}| \leq 1 \\  |u_{\tau}| \leq N}} \mathrm{d}u \int_{0}^{1} \mathrm{d}r \ \mathbf{k}_{f_{0}}(v,u)  \int_{0}^{t} \int_{\mathbb{R}^{3}} \frac{|t-s| |\bar{u}^{\prime}|^{2}(1+ |\bar{u}^{\prime}|^{2}) e^{\varpi \langle \bar{u}^{\prime}  \rangle s}  }{\alpha(X_{\mathbf{cl}}(s) ,u^{\prime})^{3/2}} \mathrm{d}u^{\prime} \mathrm{d}s\\
&- P( || e^{\theta |v|^{2}} f_{0} ||_{\infty} ) \int_{ \substack{\varepsilon \leq |u_{n}| \leq 1 \\  |u_{\tau}| \leq N}} \mathrm{d}u \int_{0}^{1} \mathrm{d}r \ \mathbf{k}_{f_{0}}(v,u)  \langle u\rangle^{\kappa} e^{-\theta |u|^{2}} \frac{|t | |\bar{u}|^{4}}{  |u_{n}|^{2} + C\varepsilon r N^{2}}\\
&-O(1) \int_{0}^{1}\int_{\varepsilon \leq |u_{n}| \leq 1} \frac{ e^{-N}}{|u_{n}|^{2} + C \varepsilon r N^{2} }  \mathrm{d}u_{n} \mathrm{d}r\\
\gtrsim & - O_{N}(1)|| \nabla_{x} \partial_{t} f_{0} ||_{\infty} \ln (\frac{1} {\varepsilon})  - o(1) \int_{0}^{1}\int_{\varepsilon \leq |u_{n}| \leq 1} \frac{ e^{-N}}{|u_{n}|^{2} + C \varepsilon r N^{2} }  \mathrm{d}u_{n} \mathrm{d}r.
\end{split}
\end{equation}

Collecting the terms and using (\ref{bb_dominant})
\begin{equation}\notag
\begin{split}
\mathbf{II}   \gtrsim \int_{0}^{1} \mathrm{d}r \int_{\varepsilon \leq |u_{n}| \leq 1} \mathrm{d}u_{n} \frac{1}{|u_{n}|^{2} + C\varepsilon r N^{2}} - \ln (\frac{1}{\varepsilon})  \gtrsim_{N} \frac{1}{\varepsilon} \rightarrow +\infty,
\end{split}
\end{equation}
 and $\varepsilon \rightarrow 0$ and this proves (\ref{blowup_specular}).

\end{proof}

\noindent \textbf{Acknowledgements}: This project was initiated on the study of diffusive reflection boundary condition during the Kinetic Program at ICERM, 2011. We thank a referee for the constructive remarks on this earlier work, leading to our study of specular and bounce-back cases 
included in the paper now (carried out by Y. Guo and C. Kim). Y.Guo's research is
supported in part by a NSF grant, FRG grant, a Chinese NSF grant as well as Beijing International Center for
Mathematical Research. He thanks Nader Masmoudi for earlier discussions on the same subject. C. Kim's research
is supported in part by the Herchel Smith fund. He thanks Division of Applied Mathematics, Brown University and
KAIST for the kind hospitality and support and also he thanks Cl\'ement Mouhot for
helpful discussions. A. Trescases thanks the Division of Applied Mathematics, Brown University for its hospitality
during her visit during 2011-2012. The authors would like to dedicate this work to the memory of Seiji Ukai.

\vspace{20pt}
 
\noindent{Division of Applied Mathematics, Brown University,
Providence, RI 02812, U.S.A., \ \ Yan$\underline{~}$Guo@brown.edu }

\noindent{Department of Pure Mathematics and Mathematical Statistics ,
University of Cambridge, Cambridge, CB3 0WB,UK, \ \ C.W.Kim@dpmms.cam.ac.uk}\newline

\noindent {Institut de math\'{e}matiques de Jussieu, case 472, 4 Place de
Jussieu, 75252, Paris, France, \ \ tonond@math.jussieu.fr} \newline

\noindent{CMLA, ENS Cachan, 61, avenue du Pr\'{e}sident Wilson, 94235
Cachan, France, \ \ ariane.trescases@ens-cachan.fr}


\bigskip






    \begin{thebibliography}{99}
\bibitem{CIP} Cercignani, C.; Illner, R.; Pulvirenti, M.: \textit{The
mathematical theory of dilute gases.} Applied Mathematical Sciences, 106.
Springer-Verlag, New York, 1994

\bibitem{EGKM} Esposito, R.; Guo, Y.; Kim, C. ; Marra, R: Non-Isothermal
Boundary in the Boltzmann Theory and Fourier Law. \textit{ Comm. Math. Phys.} 323(2013) 177--239.

\bibitem{DV} Desvillettes, L.; Villani, C.: On the trend to global
equilibrium for spatial inhomogeneous kinetic systems: the Boltzmann
equation. \textit{Invent. Math.} 159(2005) 245--316.

\bibitem{Gui} Guiraud, J. P.: An H-theorem for a gas of rigid spheres in a
bounded domain. CNRS, Paris(1975), 29--58.

\bibitem{Guo94} Guo, Y.: Regularity of the Vlasov equations in a half space.
\textit{Indiana. Math. J.} 43(1994) 255--320.

\bibitem{Guo95} Guo, Y.: Singular Solutions of the Vlasov-Maxwell System on
a Half Line. \textit{Arch. Rational Mech. Anal.} 131(1995) 241--304.

\bibitem{Guo03} Guo, Y.: Classical Solutions to the Boltzmann Equation for Molecules with an Angular Cutoff. \textit{Arch. Rational Mech. Anal.} 169(2003) 305-353.

\bibitem{Guo10} Guo, Y.: Decay and Continuity of Boltzmann Equation in
Bounded Domains. \textit{Arch. Rational Mech. Anal.} 197(2010) 713--809.


\bibitem{gl} Glassey, R. T.: \emph{The Cauchy problem in kinetic theory.}
Society for Industrial and Applied Mathematics (SIAM), Philadelphia, 1996



\bibitem{HV} Hwang, H-J, Velazquez, J.: Global Existence for the
Vlasov-Poisson System in Bounded Domains. \textit{Arch. Rational Mech. Anal.}
195(2010) 763--796.

\bibitem{Kim11} Kim, C.: Formation and propagation of discontinuity for
Boltzmann equation in non-convex domains. \textit{\ Comm. Math. Phys.} 308(2011) 641--701.
\end{thebibliography}
\end{document}